\documentclass[11pt,reqno]{amsart}
\usepackage{amsmath,amssymb,amsthm,amsfonts,latexsym,times,color,bm}
\usepackage{amsmath,amsfonts}
\usepackage{amsthm}

\usepackage{inputenc}

\usepackage[linktocpage=true]{hyperref}

\usepackage{graphicx}
\usepackage{color}
\usepackage{multicol}

\usepackage{xcolor}

\usepackage[width=14cm, left=2cm, right=2cm,marginpar=2cm]{geometry}

\newcounter{mnotecount}[section]

\newtheorem{thm}{Theorem}[section]
\newtheorem{definition}{Definition}[section]

\newtheorem{Lemma}[thm]{Lemma}
\newtheorem{remark}{Remark}[section]
\newtheorem{theorem}[thm]{Theorem}
\newtheorem{proposition}[thm]{Proposition}

\newtheorem{corollary}[thm]{Corollary}

\numberwithin{equation}{section}

\newcommand{\beq}{\begin{equation}}
\newcommand{\eeq}{\end{equation}}
\newcommand{\ben}{\begin{eqnarray}}
\newcommand{\een}{\end{eqnarray}}
\newcommand{\beno}{\begin{eqnarray*}}
\newcommand{\eeno}{\end{eqnarray*}}

\newcommand{\bv}{\mathbf{v}}

\newcommand{\bw}{\mathbf{w}}
\newcommand{\by}{\mathbf{y}}
\newcommand{\bx}{\mathbf{x}}

\newcommand{\bB}{\mathbf{B}}
\newcommand{\bF}{\mathbf{F}}

\newcommand{\bH}{\mathbf{H}}

\newcommand{\bW}{\mathbf{W}}
\newcommand{\bU}{\mathbf{U}}
\newcommand{\bQ}{\mathbf{Q}}
\newcommand{\bK}{\mathbf{K}}
\newcommand{\bJ}{\mathbf{J}}
\newcommand{\bY}{\mathbf{Y}}
\newcommand{\bV}{\mathbf{V}}

\newcommand{\sstar}{s}
\numberwithin{equation}{section}

\begin{document}
\title[Three-dimensional compressible Euler equations]{Well-posedness for rough solutions of the 3D compressible Euler equations}


\author{Lars Andersson}
%    Address of record for the research reported here
\address{Department of Mathematics, University of Potsdam, Karl-Liebknecht-Str. 24-25, 14476 Potsdam, Germany}
%    Current address
\email{lars.andersson@uni-potsdam.de}

\author{Huali  Zhang}
%    Address of record for the research reported here
\address{School of Mathematics, Hunan University, Lushan South road, 410082 Changsha, China.}
%    Current address
\email{hualizhang@hnu.edu.cn}

\date{\today}


\begin{abstract}
This paper considers solutions of the compressible 3D Euler equation for low-regularity initial data. An important inspiration for this work is the paper by Qian Wang \cite{WQEuler}, where she proved, using the vector-field approach, local existence and uniqueness of solutions for 3D isentropic compressible Euler equations for initial velocity $\bv_0$, logarithmic density $\rho_0$ and specific vorticity $\bw_0$ satisfy $(\bv_0, \rho_0, \bw_0) \in H^s\times H^s \times H^{s_0} (2<s_0<s)$, and Wang proposed a conjecture that the local solutions of 3D isentropic, compressible Euler equations is well posed if $(\bv_0, \rho_0,\bw_0) \in H^{2+}\times H^{2+}\times H^{\frac{3}{2}+}$.

The main new result in this paper is local well-posedness of solutions of 3D isentropic compressible Euler equations if $(\bv_0, \rho_0, \bw_0) \in H^{2+} \times H^{2+} \times H^{2}$, where in particular the regularity condition on the initial vorticity has been reduced to $\bw_0 \in H^2$. The fact that the compressible Euler equation implies a coupled system of quasi-linear wave equations, and transport equations for the vorticity \cite{LS2} plays an essential role in the work of Wang and the present paper. In particular the wave nature of the system allows one to prove Strichartz estimates. Instead of the vector-field approach used by Wang, we here use the approach of Smith and Tataru \cite{ST}, which is well adapted to the structure of the system, in order to prove Strichartz estimates.

As a first step, we prove the local well-posedness of solutions for initial velocity $\bv_0$, logarithmic density $\rho_0$, entropy $h_0$, and specific vorticity $\bw_0$ satisfying $(\bv_0, \rho_0,h_0,\bw_0) \in H^s \times H^s \times H^{s_0+1}\times H^{s_0} (2<s_0<s)$, giving an alternate proof of the result of Wang. The proof has three distinct features: \emph{(i)} a Strichartz estimate of linear wave equations endowed with a acoustic metric, \emph{(ii)} a continuous dependence on initial data with low regularity, \emph{(iii)} a type of Strichartz estimate for solutions with a regularity of the velocity $2+$, density $2+$, entropy $\frac{5}{2}+$ and specific vorticity $\frac{3}{2}+$. For $(\bv_0, \rho_0, \bw_0) \in H^{2+} \times H^{2+} \times H^{2}$, then a stronger Strichartz estimates on a short time-interval hold. By summing up these short time intervals to a regular time-interval, a Strichartz estimate with loss of derivatives can be obtained. This argument is inspired by the work of \cite{AIT}. This allows us to prove the local well-posedness of solutions of 3D isentropic compressible Euler equations for initial data $(\bv_0, \rho_0, \bw_0) \in H^{2+} \times H^{2+} \times H^{2}$. Using a similar idea, we are also able to prove the local well-posedness of solutions of 3D compressible Euler equations with the initial data $(\bv_0, \rho_0, h_0, \bw_0) \in H^{\frac52} \times H^{\frac52} \times H^{\frac52+}\times H^{\frac32+}$.
\end{abstract}

\maketitle
\newpage

\tableofcontents

\newpage

\section{Introduction and Results}
%\subsection{Background}
We consider the Cauchy problem for the
%low regularity problem of
compressible Euler equations (cf. \cite{M}) in three dimensions, %which is described as%\footnote{The model \eqref{CEE}-\eqref{py1} is from Majda's book \cite{M}.}
\begin{equation}\label{CEE}
	\begin{cases}
	\varrho_t+\text{div}\left(\varrho \bv \right)=0,  \quad t>0,\ x\in \mathbb{R}^3,
	\\
	\bv_t + \left(\bv\cdot \nabla \right)\bv+\frac{1}{\varrho}\nabla p=0,
\\
\partial_t h + \left(\bv\cdot \nabla \right)h=0,
\end{cases}
\end{equation}
with a state function given by
\begin{equation}\label{py1}
  p=p(\varrho,h)=\varrho^\gamma \mathrm{e}^{h}\quad (\gamma>1 \ \mathrm{or} \ \gamma=-1),
\end{equation}
%\textcolor{red}{explain here difference from isentropic case}
and initial data
\begin{equation}\label{id}
	(\varrho, \bv, h)|_{t=0}=(\varrho_0, \bv_0, h_0),
\end{equation}
where $\bv=(v^1,v^2,v^3)^{\text{T}}, \varrho$,  $h$, and $p$ respectively, denote the fluid velocity, density, entropy and pressure, $v^1,v^2, v^3, \varrho, h$ and $p$ are scalar functions denoted from $(t,\bx)\in \mathbb{R}^+ \times \mathbb{R}^3$ to $\mathbb{R}^3$. The compressible Euler equation describes the motion of an ideal fluid. A fluid with constant entropy $h$ is called isentropic. The phenomena displayed in the interior of a fluid fall into two broad classes, the phenomena of waves,  and of vortex motion.

In order to display the wave and vortex structure in the Euler equation, and to fix notations,
%For the purposes of seeing what's the form of the wave and vortex structure and fixing notations,
we state the following definitions and lemmas, which can also be found in \cite{LS2,S2}.
\subsection{Reductions and definitions}
Let us first introduce some variables---logarithmic density, specific vorticity and the acoustic metric.
\begin{definition}\label{pw}
Let $\bar{\rho}> 0$ be a constant background density. The logarithmic density ${\rho}$ and specific vorticity $\bw$ are defined by
\begin{equation}\label{pw1}
\rho :=\ln \left( \frac{\varrho}{{\bar{\rho}}}\right),
\end{equation}
\begin{equation}\label{pw11}
  \bw :={\mathrm{e}^{-\rho}}{\mathrm{curl}\bv},
\end{equation}
and we write $\bw=(w^1,w^2,w^3)^{\mathrm{T}}$.
\end{definition}

\begin{definition}\label{pwd}
Let ${\rho}$ be the logarithmic density. We define the vector $\bH$ as
\begin{equation}\label{pwh}
  \bH:=\mathrm{e}^{-\rho}\nabla h.
\end{equation}
\end{definition}
\begin{definition}\label{shengsu}
We denote the speed of sound\footnote{If $h\neq \mathrm{constant}$, then $c_s$ is determined by $\rho, \bv$ and $h$. If $h= \mathrm{constant}$, $c_s$ only depends on $\rho$ and $\bv$.}
\begin{equation}\label{ss}
c_s:=\sqrt{\frac{\text{d}p}{\text{d}\varrho}}.
\end{equation}
\end{definition}
Based on above definitions, let us introduce a reduction of \eqref{CEE} under the above variables.
\begin{Lemma}\label{wte} \cite{LS2}
Let $(\bv,\rho,h)$ be a solution of \eqref{CEE}. Let $\bw$ and $\rho$ are defined in \eqref{pw1}-\eqref{pw11}. The 3D compressible Euler equations \eqref{CEE} can be written in the form
\begin{equation}\label{fc0}
\begin{cases}
\mathbf{T} v^i=-c^2_s \partial^i \rho-\frac{1}{\gamma}c^2_s \partial^i h,
\\
\mathbf{T} \rho=-\mathrm{div} \bv,
\\
\mathbf{T} h=0,
\end{cases}
\end{equation}
where
\begin{equation} \label{eq:Tdef}
\mathbf{T}:=\partial_t + \bv \cdot \nabla .
\end{equation}
\end{Lemma}
We now introduce the acoustic metric, which will be used to reduce \eqref{fc0} to a wave-transport system.
\begin{definition}\label{metricd}
We define the acoustic metric $g$ and the inverse acoustic metric $g^{-1}$ relative to the Cartesian coordinates as follows:
\begin{equation}\label{Met}
	\begin{split}
&g:=-dt\otimes dt+c_s^{-2}\sum_{a=1}^{3}\left( dx^a-v^adt\right)\otimes\left( dx^a-v^adt\right),
\\
&g^{-1}:=-(\partial_t+v^a \partial_a)\otimes (\partial_t+v^b \partial_b)+c_s^{2}\sum_{i=1}^{3}\partial_i \otimes \partial_i.
\end{split}
\end{equation}	
\end{definition}
Following \cite{WQEuler}, we introduce a decomposition for the velocity
\begin{equation}\label{dvc}
  v^i=v_{+}^i+ v_{-}^i,
\end{equation}
where the vector $\bv_{-}=( v_{-}^1,  v_{-}^2,  v_{-}^3)^{\mathrm{T}}$ is determined by
\begin{equation}\label{etad}
  -\Delta v_{-}^{i}:=\mathrm{e}^{\rho}\mathrm{curl} \bw^i .
\end{equation}
\begin{Lemma}\label{Wte} \cite{S2}
Let $(\bv,\rho,h)$ be a solution of \eqref{fc0}. Let ${\rho}$ and $\bw$ be given by \eqref{pw1}-\eqref{pw11}. Then the system \eqref{fc0} can be written as
\begin{equation}\label{fc1}
\begin{cases}
\square_g v^i=-\mathrm{e}^{\rho}c_s^2 \mathrm{curl} \bw^i+Q^i,
\\
\square_g {\rho}=-\frac{1}{\gamma}c_s^2\Delta h+ D,
\\
\square_g h=c_s^2\Delta h + E,
\\
\mathbf{T} h =0,
\\
 \mathbf{T}w^i=w^a \partial_a v^i+\bar{\rho}^{\gamma-1} \mathrm{e}^{h+(\gamma-2)\rho}\epsilon^{iab}\partial_a \rho \partial_b h,
\end{cases}
\end{equation}
where $Q^i$, $D$ and $E$ are quadratic forms, given by
\begin{equation}\label{DDi}
\begin{split}
 {Q^i}:=& -\frac{2}{\gamma}e^{\rho} c^2_s \epsilon^i_{ab}  w^b  \partial^a h-( 1+c_s^{-1}\frac{\partial c_s}{\partial \rho})g^{\alpha \beta} \partial_\alpha {\rho} \partial_\beta v^i
 -\frac{1}{\gamma}c^2_s \epsilon^{iab}w_a \partial_b h
 \\
 & - \frac12 c^2_s \partial^a h \partial_a v^i+ \frac{2c_s}{\gamma} \frac{\partial c_s}{\partial \rho} \mathrm{div}\bv \partial^i h+c^3_s \partial_\beta v^i \partial_{\alpha} (c^{-3}_s g^{\alpha \beta}),
\\
{D:}=& -3c_s^{-1}\frac{\partial c_s}{\partial \rho} g^{\alpha \beta} \partial_\alpha \rho \partial_\beta \rho+2 \textstyle{\sum_{1 \leq a < b \leq 3} }\big\{ \partial_a v^a \partial_b v^b-\partial_a v^b \partial_b v^a \big\}
\\
& - \frac{1}{\gamma}c^2_s \partial^a h \partial_a \rho- \frac{1}{\gamma}c^2_s \partial^a h \partial_a h+c^3_s \partial_\beta \rho \partial_{\alpha} (c^{-3}_s g^{\alpha \beta}),
\\
E:=&c^3_s \partial_\beta h \partial_{\alpha} (c^{-3}_s g^{\alpha \beta}),
\end{split}
\end{equation}
and where the operator $\square_g$ is given by
\begin{equation}\label{Box}
  \square_g := g^{\alpha\beta}\partial^2_{\alpha\beta}=-\mathbf{T}\mathbf{T}+c^2_s \Delta, \quad \Delta=\partial^2_1+ \partial^2_2+\partial^2_3,
\end{equation}
with $\mathbf{T}$ as in \eqref{eq:Tdef}.
\end{Lemma}
For brevity, we set $\bQ=(Q^1,Q^2,Q^3)^\mathrm{T}$.
\begin{remark}
We can also write $Q^i=Q^{i\alpha\beta} \partial_\alpha \rho \partial_\beta h+Q^{i\alpha\beta }_{1j} \partial_\alpha \rho \partial_\beta v^j+Q^{i\alpha\beta}_{2j} \partial_\alpha h \partial_\beta v^j +Q^{\alpha\beta}_{3j} \partial_\alpha v^i \partial_\beta v^j$, $D=D_1^{\alpha\beta} \partial_\alpha \rho \partial_\beta h + D^{\alpha\beta}_{2} \partial_\alpha h \partial_\beta h+D^{\alpha\beta}_{3} \partial_\alpha \rho \partial_\beta \rho+D^{\alpha\beta }_{1j} \partial_\alpha \rho \partial_\beta v^j+D^{\alpha\beta}_{2j} \partial_\alpha h \partial_\beta v^j +D^{i\alpha\beta}_{3j} \partial_\alpha v_i \partial_\beta v^j$, and $E=E_1^{\alpha\beta} \partial_\alpha h \partial_\beta \rho+E^{\alpha\beta }_{2} \partial_\alpha h \partial_\beta h+E^{\alpha\beta}_{3j} \partial_\alpha h \partial_\beta v^j $. Here $Q^{i\alpha\beta}$, $Q^{i\alpha\beta }_{1j}$, $Q^{i\alpha\beta}_{2j}$,  $Q^{\alpha\beta}_{3j}$, $D_1^{\alpha\beta}$, $D_2^{\alpha\beta}$, $D_3^{\alpha\beta}$, $D^{\alpha\beta }_{1j}$, $D^{\alpha\beta}_{2j}$, $D^{i\alpha\beta}_{3j}$, $E_1^{\alpha\beta}$, $E^{\alpha\beta }_{2}$, and $E^{\alpha\beta}_{3j}$ are smooth functions of $\bv$, $\rho$ and $h$, and they are uniquely determined by \eqref{DDi}.
\end{remark}
\begin{remark}
The system \eqref{fc1} is derived from \eqref{CEE}, with \eqref{pw1}-\eqref{pw11}, and using this convention, the  systems \eqref{CEE}, \eqref{fc0}, and \eqref{fc1}, as well as \eqref{sq} below, are equivalent.
\end{remark}
\begin{remark}
Using \eqref{fc1}, \eqref{dvc} and \eqref{etad}, we have
\begin{equation}\label{wrtu}
\begin{split}
&\square_g v^i_{+}=\mathbf{T}\mathbf{T} v_{-}^i+Q^i.
\end{split}
\end{equation}
This will be proved in Lemma \ref{wte1} below.
\end{remark}

\begin{Lemma}
Taking $h=$constant in \eqref{fc0}, we then have
\begin{equation}\label{fc0s}
	\begin{cases}
		&\mathbf{T} v^i=-c^2_s \partial^i \rho,
		\\
		&\mathbf{T} \rho=-\mathrm{div} \bv,
	\end{cases}
\end{equation}
Taking $h=$constant in System \eqref{fc1}, then the system \eqref{fc1} is reduced to
\begin{equation}\label{fc1s}
\begin{cases}
\square_g v^i=-\mathrm{e}^{\rho}c_s^2 \mathrm{curl} \bw^i+\mathcal{Q}^i,
\\
\square_g {\rho}= \mathcal{D},
\\
 \mathbf{T}w^i=w^a \partial_a v^i.
\end{cases}
\end{equation}
Above, $\mathcal{Q}^i$ and $\mathcal{D}$ are quadratic forms, which are defined by
\begin{equation}\label{DDI}
\begin{split}
 {\mathcal{Q}^i}:=& -( 1+c_s^{-1}c'_s)g^{\alpha \beta} \partial_\alpha {\rho} \partial_\beta v^i+c^3_s \partial_\beta v^i \partial_{\alpha} (c^{-3}_s g^{\alpha \beta}),
\\
{\mathcal{D}:}=& -3c_s^{-1}c'_s g^{\alpha \beta} \partial_\alpha \rho \partial_\beta \rho+2 \textstyle{\sum_{1 \leq a < b \leq 3} }\big\{ \partial_a v^a \partial_b v^b-\partial_a v^b \partial_b v^a \big\}+c^3_s \partial_\beta \rho \partial_{\alpha} (c^{-3}_s g^{\alpha \beta}).
\end{split}
\end{equation}
\end{Lemma}
\begin{remark}
We can also write $\mathcal{Q}^i=\mathcal{Q}^{i\alpha\beta }_{j} \partial_\alpha \rho \partial_\beta v^j+ \mathcal{Q}^{\alpha\beta}_{2j} \partial_\alpha v^i \partial_\beta v^j$, $\mathcal{D}=\mathcal{D}^{\alpha\beta }_{j} \partial_\alpha \rho \partial_\beta v^j+ \mathcal{D}^{\alpha\beta}_{2} \partial_\alpha \rho \partial_\beta \rho+\mathcal{D}^{i\alpha\beta}_{3j} \partial_\alpha v_i \partial_\beta v^j$. Here $\mathcal{Q}^{i\alpha\beta }_{j}$, $\mathcal{Q}^{\alpha\beta}_{2j}$, $\mathcal{D}^{\alpha\beta }_{j}$, $\mathcal{D}^{\alpha\beta}_{2}$, and $\mathcal{D}^{i\alpha\beta}_{3j}$ are smooth functions of $\bv$, $\rho$, and they are uniquely determined by \eqref{DDI}.
\end{remark}
\begin{remark}
The system \eqref{fc1s} can also be derived from \eqref{fc0s} under \eqref{pw1}.
\end{remark}
\subsection{Previous results} Let us recall the historical results starting from the incompressible fluid. For the Cauchy problem of $n$-dimensional incompressible Euler equations:
\begin{equation}\label{IEE}
	\begin{cases}
	\bv_t + \left(\bv\cdot \nabla \right)\bv+\nabla p=0,   \quad t>0,\ x\in \mathbb{R}^3,
	\\
	\mathrm{div} \bv=0,
\\
\bv|_{t=0}=\bv_0,
\end{cases}
\end{equation}
Kato and Ponce in \cite{KP} proved the local well-posedness of \eqref{IEE} if $\bv_0 \in W^{s,p}(\mathbb{R}^n), s>1+\frac{n}{p}$. Chae in \cite{Chae} proved the local existence of solutions by setting $\bv_0$ in Triebel-Lizorkin spaces. In the opposite direction, Bourgain and Li \cite{BL, BL2} proved that the Cauchy problem is ill-posed for $\bv_0 \in W^{1+\frac{n}{p}}(\mathbb{R}^n), 1\leq p< \infty, n=2,3$.  Very recently, Andersson and Kapitanski \cite{AK} proved the well-posedness of low regularity solutions of incompressible Neo-Hookean materials for $s>\frac74 (n=2)$ or $s>2 (n=3)$ in Lagrangian coordinates, with some additional regularity conditions on the vorticity.

In the irrotational and isentropic case, the compressible Euler equations can be reduced to a system of quasilinear wave equations (see Lemma \ref{Wte} and taking $\mathrm{curl}\bv=0$). Consider the Cauchy problem for a quasilinear wave equation
\begin{equation}\label{qwe}
\begin{cases}
  &\square_{h(\phi)} \phi=q(d \phi, d \phi),  \quad t>0,\ x\in \mathbb{R}^3,
  \\
  & \phi|_{t=0}=\phi_0, \partial_t \phi|_{t=0}=\phi_1,
  \end{cases}
\end{equation}
where $\phi$ is a scalar function, $h(\phi)$ is a Lorentzian metric depending on $\phi$, $d=(\partial_t, \partial_1, \partial_2, \cdots, \partial_n)$, and $q$ is quadratic in $d \phi$. Set the initial data $(\phi_0, \phi_1) \in H^s(\mathbb{R}^n) \times H^{s-1}(\mathbb{R}^n)$. By using classical energy methods, Hughes-Kato-Marsden in \cite{HKM} proved the local well-posedness of the problem \eqref{qwe} for $ s>\frac{n}{2}+1$. On the other hand, Lindblad in \cite{L} constructed some counterexamples for \eqref{qwe} when $s=2,n=3$. For $s={\frac{7}{4}}, n=2$, a 2D counterexample was constructed by Ohlman \cite{Oman}. There is a gap between the results \cite{HKM} and \cite{L,Oman}. To lower the regularity required for well-posedness, %get improvement on local solutions,
one may use estimates of Strichartz type, that  provide integrated space-time estimates for $d \phi$. Work on Strichartz estimates has evolved from linear analysis, to the study of bilinear wave intersections, and next to the study of nonlinear wave intersections. The first stage is to study the Strichartz estimates of a wave equation with variable coefficients
\begin{equation}\label{qw0}
  \square_{h(t,x)}\phi=0,
\end{equation}
and then exploit it to obtain the low regularity solutions of \eqref{qwe}.

Kapitanskii \cite{Kap} and Mockenhaupt-Seeger-Sogge \cite{MSS} proved Strichartz estimates of \eqref{qw0} with smooth coefficients $h$. For a rough metric $h \in C^2$, the study of Strichartz estimates of \eqref{qw0} in two or three dimensions began with Smith's result \cite{Sm}. This was extended to all dimensions by Tataru \cite{T2}. In the opposite direction, Smith and Sogge \cite{SS} showed that the Strichartz estimates fail for $h \in C^\alpha, \alpha<2$. The first improvement in study of nonlinear wave equations was independently achieved by Bahouri-Chemin \cite{BC2} and Tataru \cite{T1}, who proved the local well-posedness of \eqref{qwe} with $s > \frac{n}{2} + \frac{7}{8}, n=2$ or $s > \frac{n}{2} + \frac{3}{4}, n\geq3$. This corresponding 3D result was also obtained by Klainerman in \cite{K} using the vector-field approach. Shortly afterward, Tataru \cite{T3} relaxed the Sobolev indices to $s>\frac{n+1}{2}+\frac{1}{6}, n \geq 3$. Also, Smith-Tataru \cite{ST0} showed that the $\frac{1}{6}$ loss is sharp for variable coefficients. Thus, to improve the above results of \eqref{qwe}, one needs to exploit a new way or structure of \eqref{qwe}.

By exploiting additional geometrical structure of the metric and introducing a decomposition of curvature, the 3D result was improved by Klainerman and Rodnianski \cite{KR2}, with regularity exponent satisfying $s>2+\frac{2-\sqrt{3}}{2}$. Inspired by \cite{KR2}, Geba in \cite{Geba} studied the local existence and uniqueness of 2D quasilinear wave equations for $s > \frac{7}{4} + \frac{5-\sqrt{22}}{4}$. In 2005, %Klainerman and Rodnianski \cite{KR} proved a sharp result for Einstein vacuum equations in wave guage. Independently,
a sharp result was established by Smith and Tataru in \cite{ST}, and they proved the local existence and uniqueness of solution of \eqref{qwe} if the regularity of initial data satisfies $s>\frac{7}{4}, n=2$ or $s>2, n=3$ or $s>\frac{n+1}{2}, 4 \leq n \leq 5$. An alternative proof of the 3D result in \cite{ST} was also obtained by Q. Wang \cite{WQSharp} using the vector-field approach .

We should also mention the important progress in Einstein vacuum equations. In Yang-Mills gauge, the $L^2$ conjecture is solved by Klainerman-Rodnianski-Szeftel \cite{KR1}. In wave gauge, the system is well-posed in $H^{2+}$ and ill-posed in $H^{2}$, cf. Klainerman-Rodnianski \cite{KR} and Ettinger-Lindblad \cite{EL}. For the time-like minimal surface equations, which satisfy null conditions, Ai-Ifrim-Tataru \cite{AIT} obtained a substantial improvement, namely by $\frac18$ derivatives in two space dimensions and by $\frac14$ derivatives in higher dimensions. For other local well-posedness results of \eqref{qwe} in various settings, we refer the readers to \cite{AAR,AM,IT1,Sog,WCB,WQRough,WQ1, WQ2,ZH,Zh,ZL}.

For the compressible Euler system, local well-posedness of \eqref{CEE}-\eqref{id} was initially obtained by Majda \cite{M}, with $(\bv_0, \varrho_0, h_0) \in H^{s}, s>1+\frac{n}{2}$ and the density bounded away from vacuum. Following this result, there is a huge literature that deals with blow up problems, the incompressible limit, and free boundary problems, cf.  \cite{C,CLS,JM,IT,KM,KM2,LS1,LS2,MR,MR1,S,S1} and reference therein. In the last several years, much progress has been made on low regularity problems of \eqref{CEE}-\eqref{id}. The first improvement of rough solutions was independently  obtained by Q. Wang \cite{WQEuler} and Disconzi-Luo-Mazzone-Speck \cite{DLS}. Q. Wang \cite{WQEuler} established the existence and uniqueness of solutions of \eqref{fc1s} when $(\bv_0,\rho_0,\bw_0) \in H^{s}(\mathbb{R}^3) \times H^{s}(\mathbb{R}^3) \times H^{s_0}(\mathbb{R}^3), \ 2<s_0<s$. Disconzi-Luo-Mazzone-Speck \cite{DLS} proved the existence and uniqueness of solutions of \eqref{fc1} if $(\bv_0,\rho_0,\bw_0,h_0)\in H^{2+}(\mathbb{R}^3) \times H^{2+}(\mathbb{R}^3) \times H^{2+}(\mathbb{R}^3) \times H^{3+}(\mathbb{R}^3)$ with an additional H\"older condition on the initial data for $\mathrm{curl} \bw$. In the reverse direction, An-Chen-Yin \cite{ACY} proved ill-posedness of \eqref{CEE}-\eqref{id} for initial data $(\bv_0, \rho_0, h_0) \in \dot{H}^{2}(\mathbb{R}^3)$ with smooth $\bw_0$.  In two dimensions, Zhang \cite{Z1,Z2} proved the existence, uniqueness and continuous dependence of solutions of 2D isentropic compressible Euler equations if $(\bv_0, \rho_0, \bw_0) \in H^{\frac74+}(\mathbb{R}^2), \partial \bw_0 \in L^\infty $ or $(\bv_0, \rho_0, \bw_0) \in H^{\frac74+}(\mathbb{R}^2) \times H^{\frac74+}(\mathbb{R}^2) \times H^2(\mathbb{R}^2)$. Recently, An-Chen-Yin \cite{ACY1} showed that the Cauchy problem for 2D compressible Euler equations is ill-posed if $(\bv_0, \rho_0, h_0) \in \dot{H}^{\frac74}(\mathbb{R}^2)$.

In this paper, we improve on the above mentioned results on the Cauchy problem for the compressible Euler system in several different ways. In particular, we prove local well-posedness including continuous dependence on initial data, for data  $(\bv_0, \rho_0) \in H^{2+}(\mathbb{R}^3) \times H^{2+}(\mathbb{R}^3) $ with  $\bw_0 \in H^2(\mathbb{R}^3)$, improving on the result of Q. Wang. We also prove local well-posedness for data $(\bv_0, \rho_0, h_0, \bw_0) \in H^{5/2} (\mathbb{R}^3) \times H^{5/2} (\mathbb{R}^3) \times H^{5/2+} (\mathbb{R}^3)\times H^{3/2+} (\mathbb{R}^3)$.

\subsection{Motivation}
Let us go back to the 3D results \cite{WQEuler} and \cite{DLS}. In \cite{WQEuler, DLS}, to establish Strichartz estimates of the velocity and density, the regularity of the vorticity is assumed to be strictly greater than $2$. This turns out to be a very crucial condition, due to the role played in the Strichartz estimates of the regularity of the acoustic metric, and therefore of the vorticity.
As discussed in \cite{WQEuler,DLS}, both  the vector field approach and the method of Smith and Tataru  fail to establish Strichartz estimates when the Sobolev regularity of the vorticity and the corresponding foliations of space-time is less than or equal to $2$.
It is known that the classical Strichartz estimates fail to hold for spacetime metrics with regularity $C^\alpha$, $\alpha < 2$, cf. \cite{SS}. See also \cite{T1,T2,T3} for estimates of Strichartz type under weaker regularity assumptions.

In view of the just mentioned facts, well-posedness for \eqref{CEE}-\eqref{id} with $(\bv_0,\rho_0,h_0, \bw_0) \in H^{2+}(\mathbb{R}^3) \times H^{2+}(\mathbb{R}^3) \times H^{3+}(\mathbb{R}^3) \times H^{s'}(\mathbb{R}^3),\ s' \leq 2 $,  remains a challenging open problem, since the regularity of foliations will be less than or equal to $2$. However, further considering the role of the term $\mathrm{curl}\bw$ in the Strichartz estimates leads to new insights.

Let us consider a simplified model of \eqref{fc1s}. We also make note of the better variable $\bv_{+}$ in \eqref{wrtu}. Thus, let $\bV$ be a vector function satisfying the Cauchy problem
\begin{equation}\label{model}
\begin{cases}
& \square_{\mathbf{m}} \bV= -\text{curl} \bW,  \quad t>0,\ x\in \mathbb{R}^3,
\\
&\bV=\bV_{+}+\bV_{-}, \quad -\Delta \bV_{-}=\text{curl} \bW, \quad \bW=\text{curl} \bV,
\\
& \square_{\mathbf{m}} \bV_{+}= \partial_t(\partial_t \bV_{-}),
\\
&\partial_t \bW=\bW \cdot \partial \bV,
\\
&\partial_t \text{curl}\text{curl} \bW=\text{curl}\text{curl} \bW \cdot \partial \bV,
\end{cases}
\end{equation}
with initial data
\begin{equation*}
	(\bV, \partial_t \bV, \bW)_{t=0}=(\bV_0,\bV_1,\bW_0)  \in H_x^{2+} \times H_x^{1+} \times H_x^{2}.
\end{equation*}
Above, $\mathbf{m}$ is a standard Minkovski metric. If $\bW_0 \in H_x^{2}$, one would at most hope that $\text{curl} \bW \in H_x^1$. So we cannot obtain the Strichartz estimates of $\|d\bV\|_{L^2_t L^\infty_x}$ directly via the first equation $\square_{\mathbf{m}} \bV= -\text{curl} \bW$ in \eqref{model}. Fortunately, by using Duhamel's principle, Strichartz estimates for $\square_{\mathbf{m}} f=0$ (cf. \cite{Tao2}), and Sobolev's imbeddings for $d \bV_{-}$, it is possible to prove that there is a Strichartz estimate for $d \bV$ and that \eqref{model} is locally well-posed. Moreover, there exists a finite time-interval $[0,t]$ such that
\begin{equation}\label{ms}
  \| d \bV \|_{L^2_{[0,t]} L^\infty_x} \leq \| \bV_0 \|_{H^{2+}_x}+\| \bV_1 \|_{H^{1+}_x}+\| \bW_0 \|_{H^{2}_x}.
\end{equation}
The wave-transport system \eqref{fc1s} coupled with \eqref{wrtu} is very similar to \eqref{model} if we replace $\mathbf{m}$ by a Lorentz metric depending on $\bV$. Using finite speed of propagation and restricting to a small time interval, we shall reduce \eqref{fc1s} to small data problem. Then the acoustic metric $g$ can be viewed as a perturbation of $\mathbf{m}$. So we may expect there exists a local solution of \eqref{fc1s} for $(\bv_0,\rho_0,\bw_0)\in H^{2+}\times H^{2+} \times H^{2}$ or $(\bv_0,\rho_0,\bw_0)\in H^{\frac52}\times H^{\frac52} \times H^{\frac32+}$. Our approach differs from that of \cite{WQEuler} and \cite{DLS}, for they prove the Strichartz estimates of solutions of \eqref{fc1s} by reducing the problem to conformal energy estimates of linear wave equations, whereas our method is to reduce to a small data problem, and then prove the Strichartz estimates of linear wave equations endowed with the acoustic metric $g$ by constructing approximate solutions.

On the other hand, for the time-like minimal surface equations, a quasilinear model, the work of Ai-Ifrim-Tataru \cite{AIT} shows that a significant improvement can be achieved by exploiting the structure of Strichartz estimates and
of equations. Ai-Ifrim-Tataru were able to lower the regularity requirements for well-posedness by $\frac14$ derivatives in three spatial dimensions compared to the sharp results of Smith-Tataru for the generic case \eqref{qwe}. Indeed, in our former preprint \cite{ZA}, in the reduction of small solutions, if $\bw_0 \in H^{2+}$, we found that there is a Strichartz estimate of solutions of \eqref{CEE}-\eqref{id} only depending on the quantities $\|\bv\|_{H_x^{2+}}$, $\|\rho\|_{H_x^{2+}}$, $\|\bw\|_{H_x^{3/2+}}$, and $\|h\|_{H_x^{5/2+}}$. This allows us to study Strichartz estimates and low regularity solutions of \eqref{CEE}-\eqref{id} when $\bw_0 \in H^{s'} (s' \leq 2)$ based on the pioneering works \cite{WQEuler,DLS}, and our former preprint \cite{ZA}.

\subsection{Statement of  results} \label{sec:statement}
Before stating our main theorem, we introduce some notation and function spaces.
Let $\partial=(\partial_{1}, \partial_{2}, \partial_{3})^\mathrm{T}$, $d=(\partial_t, \partial_{1}, \partial_{2}, \partial_{3})^\mathrm{T}$, $\left< \xi \right>=(1+|\xi|^2)^{\frac{1}{2}}, \ \xi \in \mathbb{R}^3$. Denote by $\left< \partial \right>$ the corresponding Bessel potential multiplier. Denote the fractional operator $\Lambda_x=(-\Delta)^{\frac{1}{2}}$.  The symbol $\epsilon^{ijk}(i, j, k=1, 2, 3)$ denotes the standard volume form on $\mathbb{R}^3$.
Let ${\Delta}_j$ be the homogeneous frequency localized operator with frequency $2^j, j \in \mathbb{Z}$,  cf. \cite{BCD}.

For $f \in H^s(\mathbb{R}^3)$, we let
\begin{align*}
\|f\|_{H^s}:= \|f\|_{L^2(\mathbb{R}^3)}+\|f\|_{\dot{H}^s(\mathbb{R}^3)},
\end{align*}
with the homogeneous norm $\|f\|^2_{\dot{H}^s} := {\sum_{j \in \mathbb{Z}}} 2^{2js}\|{\Delta}_j f\|^2_{L^2(\mathbb{R}^3)} $. We shall also make use of the homogenous Besov norm
\begin{align*}
\|f\|^r_{\dot{B}^s_{p,r}} := {\sum_{j \in \mathbb{Z} }}2^{jsr}\|{\Delta}_j f\|^r_{L^p(\mathbb{R}^3)}.
\end{align*}
To avoid confusion, when the function $f$ is related to both time and space variables, we use the notation $\|f\|_{H_x^s}:=\|f(t,\cdot)\|_{H^s(\mathbb{R}^3)}$.

Assume there holds
\begin{equation}\label{HEw}
	|\bv_0, \rho_0, h_0| \leq C_0, \qquad c_s|_{t=0}>c_0>0,
\end{equation}
where $C_0, c_0 > 0 $ are constants. In the following, constants $C$ depending only on $C_0, c_0$ shall be called universal. Unless otherwise stated, all constants that appear are universal in this sense.
The notation $X \lesssim Y$ means $X \leq CY$, where $C$ is a universal constant, possibly depending on $C_0, c_0$. Similarly, we write  $X \simeq Y$ when $C_1 Y \leq X \leq C_2Y$, with $C_1$ and $C_2$ universal constants, and $X \ll Y$ when $X \leq CY$ for a sufficiently large constant $C$.

Let $s_0, s$ satisfy $2< s_0 < s\leq \frac{5}{2}$.
%and $ 2< \sstar < \frac{5}{2}$.
We also set
\begin{equation}\label{a1}
	\begin{split}
		& \delta_0=s_0-2, \quad \delta\in (0, \delta_0),
		\\
		& \delta_1=\frac{\sstar-2}{10}.
	\end{split}
\end{equation}
We use four small parameters
\begin{equation}\label{a0}
0 < 	\epsilon_3 \ll \epsilon_2 \ll \epsilon_1 \ll \epsilon_0 \ll 1.
\end{equation}

We are now ready to state the main result of this paper.
%The main result of this paper is as follows:
\begin{theorem}\label{dingli}
	Consider the Cauchy problem \eqref{CEE}-\eqref{id}. Let $\rho$ and $\bw$ be defined in \eqref{pw1}-\eqref{pw11}.
Let $2< s_0 < s\leq \frac{5}{2}$. For given initial data $(\bv_0,\rho_0,h_0, \bw_0)$ satisfying \eqref{HEw} and for any $M_0>0$ such that
	\begin{equation}\label{chuzhi1}
	\| \bv_0\|_{H^{s}} +
	\| \rho_0\|_{H^{s}} + \| \bw_0\|_{H^{s_0}}+\| h_0\|_{H^{1+s_0}}
 \leq M_0,
	\end{equation}
there exist positive constants $T>0$
%($T_0$ only depending on $\|\bv_0\|_{H^s},\|\boldsymbol{\rho}\|_{H^s},\|\bw_0\|_{H^{s_0}}$, and $\| h_0\|_{H^{1+s_0}}$)
and $M_1>0$ ($T$ and ${M}_1$ depends on $C_0, c_0, s, s_0, {M}_0$) such that \eqref{CEE}-\eqref{id} has a unique solution $(\bv,\rho,\bw, h)$ satisfying
\begin{align*}
(\bv,\rho) \in C([0,T],H_x^s)\cap C^1([0,T],H_x^{s-1}), \\
\bw\in C([0,T],H_x^{s_0})\cap C^1([0,T],H_x^{s_0-1}), \\
h\in C([0,T],H_x^{1+s_0})\cap C^1([0,T],H_x^{s_0}) .
\end{align*}
%To be precise,
The following statements hold.
%the solution subject to the condition $dv, d\boldsymbol{\rho} \in L^2_tL_x^\infty$.
\begin{enumerate}
\item \label{point:1}
%	 $\mathrm{(1)}$
The solutions  $\bv, \rho, h$ and $\bw$ satisfy the energy estimate
\begin{equation*}
\begin{split}
  &\|(\bv, \rho)\|_{L^{\infty}_{[0,T]}H_x^s}+ \|\bw\|_{L^\infty_{[0,T]}H_x^{s_0}} + \|h\|_{L^\infty_{[0,T]}H_x^{1+s_0}}\leq M_1,
\\
& \|(\partial_t\bv, \partial_t\rho)\|_{L^\infty_{[0,T]}H_x^{s-1}}+\|\partial_t\bw\|_{L^\infty_{[0,T]}H_x^{s_0-1}} + \|\partial_t h\|_{L^\infty_{[0,T]}H_x^{s_0}} \leq M_1,
\end{split}
\end{equation*}
and
\begin{equation*}
\|\bv, \rho,h\|_{L^\infty_{ [0,{T}] \times \mathbb{R}^3}} \leq 1+C_0.
\end{equation*}
\item \label{point:2}
% $\mathrm{(2)}$
The solution $\bv, \rho,h$ and $\bv_+$ satisfy the Strichartz estimate
\begin{equation}\label{SSr}
  \|(d\bv, d\rho, d\bv_{+})\|_{L^2_{[0,T]}L_x^\infty}+ \|d\bv, d\rho,dh\|_{L^2_{[0,T]} \dot{B}^{s_0-2}_{\infty,2}}+ \|\partial \bv_{+}\|_{L^2_{[0,T]} \dot{B}^{s_0-2}_{\infty,2}} \leq M_1.
\end{equation}
\item \label{point:3}
% $\mathrm{(3)}$
For any $1 \leq r \leq s+1$, and for each $t_0 \in [0,T)$, the linear equation
	\begin{equation}\label{linear}
	\begin{cases}
	& \square_g f=0, \qquad (t,x) \in (t_0,T]\times \mathbb{R}^3,
	\\
	&f(t_0,\cdot)=f_0 \in H^r(\mathbb{R}^3), \quad \partial_t f(t_0,\cdot)=f_1 \in H^{r-1}(\mathbb{R}^3),
	\end{cases}
	\end{equation}
admits a solution $f \in C([0,T],H_x^r) \times C^1([0,T],H_x^{r-1})$ and the following estimates hold:
\begin{equation*}%\label{E0}
\| f\|_{L_{[0,T]}^\infty H_x^r}+ \|\partial_t f\|_{L_{[0,T]}t^\infty H_x^{r-1}} \leq  C_{M_0}( \|f_0\|_{H_x^r}+ \|f_1\|_{H_x^{r-1}} ).
\end{equation*}
Additionally, the following estimates hold, provided $k<r-1$,
\begin{equation}\label{SE1}
\| \left<\partial \right>^k f\|_{L^2_{{[0,T]}}L^\infty_x} \leq  C( \|f_0\|_{H_x^r}+ \|f_1\|_{H_x^{r-1}} ),
\end{equation}
and the same estimates hold with $\left< \partial \right>^k$ replaced by $\left< \partial \right>^{k-1}d$. Here, $C_{M_0}$ is a constant depending on $C_0, c_0, s, s_0$ and $M_0$.
\item \label{point:4}
%\quad  $\mathrm{(4)}$
the map from $(\bv_0,\rho_0,h_0,\bw_0) \in H^s \times H^s \times H^{s_0+1} \times H^{s_0}$  to $(\bv,\rho,h,\bw)(t,\cdot) \in C([0,T];H_x^s \times H_x^s \times H_x^{s_0+1} \times H_x^{s_0})$ is continuous.
\end{enumerate}
\end{theorem}
\begin{remark}
\begin{enumerate}
\item
Theorem \ref{dingli}  provides the first proof of strong well-posedness for the compressible Euler system with rough data.
\item
Points \eqref{point:1}, \eqref{point:2} of	Theorem \ref{dingli} recovers the result of Q. Wang \cite[Theorem 1.1]{WQEuler}.
%, (Theorem 1.1 on Page 513).
In addition, Point \eqref{point:3} of Theorem \ref{dingli} above
%Compared with Wang's result, our Theorem \ref{dingli}
provides a Strichartz estimate \eqref{SE1} for a linear wave equation \eqref{linear} endowed with the acoustic metric, and Point \eqref{point:4} gives the continuous dependence of the solution on the initial data, which is lacking in \cite{WQEuler}.
\item
We emphasize that the estimate \eqref{SE1} is much stronger than \eqref{SSr}, and that the estimate \eqref{SE1} is a key observation to prove the continuous dependence of solutions, as well as  Theorem \ref{dingli2} and Theorem \ref{dingli3}.
\end{enumerate}
\end{remark}
\begin{remark}
\begin{enumerate}
\item
To ensure that \eqref{CEE} is hyperbolic in time, we assume that \eqref{HEw} holds. Equation \eqref{linear} is devised for adapting to the structure of the acoustic metric.
\item
The regularity exponent of the vorticity $s_0>2$  is essential to constructing wave-packets, which implies a sufficient regularity of characteristic hypersurfaces. See Corollary \ref{Rfenjie} and Section \ref{Ap} for details.
\end{enumerate}
\end{remark}

Our second result is as follows.
\begin{theorem}\label{dingli2}
	Let $0<\epsilon<\frac19$.	Consider the Cauchy problem \eqref{CEE}-\eqref{id}. Let $\rho$ and $\bw$ be defined in \eqref{pw1}-\eqref{pw11}. For given initial data $(\bv_0,\rho_0,h_0, \bw_0)$ satisfying \eqref{HEw} and for any $\bar{M}_0>0$ such that
	\begin{equation}\label{chuzhi2}
		\| \bv_0\|_{H^{\frac52}} +
		\| \rho_0\|_{H^{\frac52}} + \| \bw_0\|_{H^{\frac32+\epsilon}}+ \|h_0\|_{H^{\frac52+\epsilon}}
		\leq \bar{M}_0,
	\end{equation}
	there exist positive constants $\bar{T}>0$ and $\bar{M}_1>0$ ($\bar{T}$ and $\bar{M}_1$ depends on $C_0, c_0, \epsilon, \bar{M}_0$) such that \eqref{CEE}-\eqref{id} has a unique solution $(\bv,\rho) \in C([0,\bar{T}],H_x^{\frac52})$, $\bw \in C([0,\bar{T}],H_x^{\frac32+\epsilon})$, $h \in C([0,\bar{T}],H_x^{\frac52+\epsilon})$. To be precise,
	
\begin{enumerate}
	\item \label{point:d2:1}	
the solution $\bv, \rho, h$ and $\bw$ satisfy the energy estimates
	\begin{equation*}
		\begin{split}
& \|\bv, \rho\|_{L^\infty_{[0,\bar{T}]}H_x^{\frac52}}+ \|\bw\|_{L^\infty_{[0,\bar{T}]}H_x^{\frac32+\epsilon}}+ \|h\|_{L^\infty_{[0,\bar{T}]}H_x^{\frac52+\epsilon}}+ \|\bv, \rho,h\|_{L^\infty_{ [0,\bar{T}] \times \mathbb{R}^3}} \leq \bar{M}_1,
\\
& \|\partial_t \bv, \partial_t\rho\|_{L^\infty_{[0,\bar{T}]}H_x^{\frac32}}+ \|\partial_t \bw\|_{L^\infty_{[0,\bar{T}]}H_x^{\frac12+\epsilon}}+ \|\partial_t h\|_{L^\infty_{[0,\bar{T}]}H_x^{\frac32+\epsilon}} \leq \bar{M}_1,
\end{split}
\end{equation*}
	\item \label{point:d2:2}	
	the solution $\bv, \rho$ and $h$ satisfy the Strichartz estimate
	\begin{equation*}
		\|d\bv, d\rho, dh\|_{L^2_{[0,\bar{T}]}L_x^\infty} \leq \bar{M}_1,
	\end{equation*}
\item \label{point:d2:3}
for any $\frac{7}{6} \leq r < \frac72$, the linear wave equation
	\begin{equation}\label{linearj}
	\begin{cases}
		& \square_g f=0,
		\\
		&(f, \partial_t f)|_{t=0}=(f_0, f_1) \in H^r(\mathbb{R}^3) \times H^{r-1}(\mathbb{R}^3),
	\end{cases}
\end{equation}
has a unique solution on $[0,\bar{T}]$. Moreover, for $a\leq r-\frac{7}{6}$, the solution satisfies
\begin{equation}\label{xu02}
	\begin{split}
		&\|\left< \partial \right>^{a-1} d{f}\|_{L^2_{[0,\bar{T}]} L^\infty_x}
		\leq  {\bar{M}_3}(\|{f}_0\|_{{H}_x^r}+ \|{f}_1 \|_{{H}_x^{r-1}}),
		\\
		&\|{f}\|_{L^\infty_{[0,\bar{T}]} H^{r}_x}+ \|\partial_t {f}\|_{L^\infty_{[0,\bar{T}]} H^{r-1}_x} \leq  {\bar{M}_3}(\| {f}_0\|_{H_x^r}+ \| {f}_1\|_{H_x^{r-1}}).
	\end{split}
\end{equation}
Here ${\bar{M}_3}>0$ depends on $C_0, c_0, \epsilon, \bar{M}_0$.
\item \label{point:d2:4}
%\quad  $\mathrm{(4)}$
the map from $(\bv_0,\rho_0,h_0,\bw_0) \in H^{\frac52} \times H^{\frac52} \times H^{\frac32+\epsilon} \times H^{\frac52+\epsilon}$  to $(\bv,\rho,h,\bw)(t,\cdot) \in C([0,\bar{T}];H_x^{\frac52} \times H_x^{\frac52} \times H_x^{\frac32+\epsilon} \times H_x^{\frac52+\epsilon})$ is continuous.
\end{enumerate}
\end{theorem}
\begin{remark}
By the work of Majda \cite{M}, the Cauchy problem \eqref{CEE}-\eqref{id} is well-posed if $(\bv_0, \rho_0, h_0) \in H^{\frac52+}$(which implies $\bw_0 \in H^{\frac32+}$). Essentially, compared to the classical result \cite{M}, Theorem \ref{dingli2}
lowers the regularity requirement for the velocity and density from $H^{\frac{5}{2}+}$ to $H^\frac{5}{2}$.
\end{remark}

\begin{remark} \label{rem:difficult}
The main difficulty in proving Theorem \ref{dingli2}
is that neither the vector-field approach in \cite{WQEuler} nor the method of  Smith-Tataru can be used directly if the  vorticity has Sobolev's regularity $s_0 \leq 2$.
\end{remark}

\begin{remark}
Theorems \ref{dingli} and \ref{dingli2} hold if we replace $(\bv_0, \rho_0, h_0)$ with $(\bv_0, \rho_0, h_0-c_1)$, with $c_1$ is a constant in $\mathbb{R}$.
\end{remark}

\begin{remark}
For the isentropic and irrotational compressible Euler equations, the best known well-posedness result holds for $(\bv_0,\rho_0)\in H^{2+}(\mathbb{R}^3)$, c.f. Smith and Tataru \cite{ST}. For the incompressible Euler equations, the corresponding sharp
well-posedness result holds if $\bw_0 \in H^{\frac32+}(\mathbb{R}^3)$, as shown by Kato-Ponce \cite{KP} and Bourgain-Li \cite{BL}; here, the absence of dispersion estimate for vorticity is the key factor. By analogy, entropy in the full compressible Euler equations also lacks dispersive estimate--it is transported by the flow much like vorticity. Therefore, we expect that the optimal local well-posedness to hold for initial data $(\bv_0,\rho_0,h_0,\bw_0)\in H^{2+}(\mathbb{R}^3)\times H^{2+}(\mathbb{R}^3)\times H^{\frac52+}(\mathbb{R}^3) \times H^{\frac32+}(\mathbb{R}^3)$. 
\end{remark}
Our third result is about the isentropic  compressible Euler equations.
\begin{theorem}\label{dingli3}
Let $2<\sstar<\frac52$. Let $\rho$ and $\bw$ be defined in \eqref{pw1}-\eqref{pw11}. Consider the Cauchy problem \eqref{fc0s} with the initial data
\begin{equation}\label{ids}
	(\bv,\rho,\bw)|_{t=0}=(\bv_0,\rho_0,\bw_0).
\end{equation}
Assume that \eqref{HEw} holds. For any $M_*>0$ and initial data $(\bv_0, \rho_0,\bw_0)$ satisfying
	\begin{equation}\label{chuzhi3}
		\|\bv_0\|_{H_x^{\sstar}} + \|\rho_0\|_{H_x^{\sstar}}+ \|\bw_0\|_{H_x^{2}}
		\leq M_*,
	\end{equation}
	there exist positive constants $T^*>0$ and $\bar{M}_{2}>0$ ($T^{*}$ and $\bar{M}_{2}$ depends on $C_0, c_0, s, M_{*}$) such that \eqref{fc0s}-\eqref{ids} has a unique solution $(\bv,\rho)$ satisfying $(\bv,\rho) \in C([0,T^*],H_x^s)\cap C^1([0,T^*],H_x^{s-1})$, $\bw\in C([0,T^*],H_x^{2})\cap C^1([0,T^*],H_x^{1})$.
	To be precise, 
\begin{enumerate}
\item \label{point:d3:1}		
%	$\mathrm{(1)}$
the solution $(\bv, \rho,\bw)$ satisfies the energy estimate
	\begin{equation*}
		\begin{split}
			&\|(\bv, \rho)\|_{L^\infty_tH_x^s}+ \|\bw\|_{L^\infty_tH_x^{2}}  \leq \bar{M}_{2},
\\
& \|(\partial_t\bv, \partial_t\rho)\|_{L^\infty_tH_x^{s-1}}+\|\partial_t\bw\|_{L^\infty_tH_x^{1}} \leq \bar{M}_{2},
		\end{split}
	\end{equation*}
and
\begin{equation*}
\|\bv, \rho \|_{L^\infty_{ [0,{T}^*] \times \mathbb{R}^3}} \leq 2+C_0.
\end{equation*}
\item \label{point:d3:2}	
%	$\mathrm{(2)}$
the solution $(\bv, \rho)$
%$\bv, \rho,h$ and $\bv_+$
satisfies the Strichartz estimate
	\begin{equation*}
		\|(d\bv, d\rho)\|_{L^2_{ [0,{T}^*]} L_x^\infty} \leq \bar{M}_{2}.
	\end{equation*}	
\item \label{point:d3:3}
for any $\frac{s}{2} \leq r \leq 3$, and for each $t_0 \in [0,T^*]$, the linear wave equation \eqref{linearj} has a unique solution on $[0,T^*]$. Moreover, for $a\leq r-\frac{s}{2}$, the solution satisfies
\begin{equation}\label{tu02}
	\begin{split}
		&\|\left< \partial \right>^{a-1} d{f}\|_{L^2_{[0,T^*]} L^\infty_x}
		\leq  \bar{M}_4 (\|{f}_0\|_{{H}_x^r}+ \|{f}_1 \|_{{H}_x^{r-1}}),
		\\
		&\|{f}\|_{L^\infty_{[0,T^*]} H^{r}_x}+ \|\partial_t {f}\|_{L^\infty_{[0,T^*]} H^{r-1}_x} \leq \bar{M}_4 (\| {f}_0\|_{H_x^r}+ \| {f}_1\|_{H_x^{r-1}}).
	\end{split}
\end{equation}
Here ${\bar{M}_4}>0$ depends on $C_0, c_0, s, {M}_*$.
\item \label{point:d3:4}
%\quad  $\mathrm{(4)}$
the map from $(\bv_0,\rho_0,\bw_0) \in H^s \times H^s \times H^{2}$  to $(\bv,\rho,\bw)(t,\cdot) \in C([0,T^*];H_x^s \times H_x^s \times H_x^{2})$ is continuous.
\end{enumerate}
\end{theorem}
\begin{remark}
	Controlling the optimal regularity of velocity and density, some good structure of \eqref{CEE} may allow us lowering the regularity of the vorticity. In Theorem \ref{dingli3}, the regularity of vorticity is $2$. So it provides us $(s_0-2)$-order ($s_0>2$) regularity improvement of the vorticity compared with \cite{WQEuler}.
\end{remark}

\begin{remark} \label{rem:difficult-2}
Similar to Theorem \ref{dingli2}, the main difficulty in proving Theorem \ref{dingli3}
is that neither the vector-field approach in \cite{WQEuler} nor the method of  Smith-Tataru can be used directly if the  vorticity has regularity $s_0 \leq 2$.
\end{remark}
\begin{remark}
Comparing \eqref{SE1} with \eqref{tu02}, \eqref{xu02}, \eqref{SE1} is better than \eqref{tu02} and \eqref{xu02}. There is a Strichartz estimate with a loss of derivatives  in \eqref{SE1}, which is proved by extending the short (semi-classical) time-interval to a regular time-interval by summing up estimates on these short time-intervals.
\end{remark}

\subsection{Overview of the proofs}

\subsubsection{A sketch of proof for Theorem \ref{dingli}}
The main idea is as follows.

$\bullet \ \textit{Step 1: energy estimates}$. The energy estimate relies on the combination of the hyperbolic structure and a transport structure of \eqref{fc0}. Using the hyperbolic system, we have
\begin{equation*}%\label{EE}
	\|(\bv, \rho, h)\|_{H_x^a} \leq \|(\bv_0, \rho_0, h_0)\|_{H_x^a}\exp(C\int^t_0 {\| (d\bv, d\rho, dh)\|_{L^\infty_x}}d\tau), \quad a \geq 0.
\end{equation*}
This implies that the energy estimates are independent of
%the
vorticity. In fact, the key role of the vorticty is
%included
in the Ricci curvature $R_{ll}$ (here $l$ is a null vector for the characteristic surfaces), not in the energy estimates, cf. Corollary \ref{Rfenjie} and Lemma \ref{chi}. Thus we also need to study the vorticity, especially with regard to its role in Strichartz estimates for velocity and density. Motivated by \cite{WQEuler,LS2}, we derive some modified transport equations:
\begin{equation}\label{pre0}
		 \mathbf{T}\bw=(\bw \cdot \nabla)\bv+ \partial \rho \partial h,
\end{equation}
and
\begin{equation}\label{pre1}
	\begin{split}
	 &\mathbf{T} ( \mathrm{curl} \mathrm{curl} \bw^i-\mathrm{e}^{-\rho} \epsilon^{ijk} \partial_k h \Delta v_j + \partial \rho \partial \bw+\partial \bv \partial^2 h )
	 \\
	=& \partial^i \big( 2 \partial_n v_a \partial^n w^a \big) + (\partial \bv, \partial \rho, \partial h) \cdot (\partial^2 \bw, \partial^3 h)+ (\partial \bv, \partial \rho, \partial h)^2 \cdot \partial^2 h 
	\\
	& + \partial^2 h \cdot \partial^2 h+ \partial \bw \cdot \partial \bw + (\partial \bv,\partial \rho,\partial h)\cdot (\partial^2 \bv,\partial^2 \rho,\partial^2 h) \cdot \partial h
	\\
	&   +   (\partial \bv, \partial \rho, \partial h)^2 \cdot \partial \bw +    (\partial \bv, \partial \rho, \partial h)^3 \cdot \partial h.	
\end{split}
\end{equation}
We use the equation \eqref{pre0} to obtain an $L^2_x$ estimate. To the estimate in the norm $\dot{H}^{s_0}_x$, we use the equation \eqref{pre1} to handle it. So the most difficult term is
\begin{equation}\label{eq01}
  \int^t_0 \int_{\mathbb{R}^3} \Lambda_x^{s_0-2} \left\{ \mathrm{curl} \mathrm{curl} \bw^i-\mathrm{e}^{-\rho} \epsilon^{ijk} \partial_k h \Delta v_j + \partial \rho \partial \bw+\partial \bv \partial^2 h \right\} \cdot \partial^i \big( 2 \partial_n v_a \partial^n w^a \big) dxd\tau,
\end{equation}
for the other terms can be estimated by H\"older equalities and commutator estimates. Integrating \eqref{eq01} by parts and using $\partial^i \mathrm{curl} \mathrm{curl} \bw^i=0$, so we have
\begin{equation}\label{eq02}
 \eqref{eq01}= \int^t_0 \int_{\mathbb{R}^3} \Lambda_x^{s_0-2} \left\{ \partial \rho \partial \bw+\partial \bv \partial^2 h -\mathrm{e}^{-\rho} \epsilon^{ijk} \partial_k h \Delta v_j  \right\} \cdot \partial^i \big( 2 \partial_n v_a \partial^n w^a \big) dxd\tau,
\end{equation}
In contrast to \cite{WQEuler}, we utilize the Plancherel formula to handle the following term
\begin{equation*}
	\int^t_0 \int_{\mathbb{R}^3}\Lambda_x^{s_0-2}\partial_i \big(    \partial_n v^a \partial^n w_a \big)\cdot \Lambda_x^{s_0-2}\big(   \partial {\rho}  \partial\bw \big)dx d\tau.
\end{equation*}
which after a partial integration becomes
\begin{equation*}
	\int^t_0 \int_{\mathbb{R}^3} \Lambda_x^{s_0-\frac{5}{2}}\partial_i \big(   \partial_n v^a \partial^n \bw_a \big)\cdot \Lambda_x^{s_0-\frac{3}{2}}\big(  \partial {\rho}  \partial \bw \big)dx d\tau.
\end{equation*}
Similarly, we can handle the second difficult term
\begin{equation*}
	\int^t_0 \int_{\mathbb{R}^3}\Lambda_x^{s_0-2}\partial_i \big(    \partial_n v^a \partial^n \bw_a \big)\cdot \Lambda_x^{s_0-2}\big(   \partial {\bv}  \partial^2 h\big)dx d\tau.
\end{equation*}
in the form
\begin{equation*}
	\int^t_0 \int_{\mathbb{R}^3} \Lambda_x^{s_0-\frac{5}{2}}\partial_i \big(   \partial_n v^a \partial^n \bw_a \big)\cdot \Lambda_x^{s_0-\frac{3}{2}}\big(  \partial {\bv}  \partial^2 h \big)dx d\tau.
\end{equation*}
We also integrate the third difficult term by parts
\begin{equation}\label{pre2}
	\int^t_0 \int_{\mathbb{R}^3}\Lambda_x^{s_0-2}\partial_i \big(    \partial_n v^a \partial^n w_a \big)\cdot \Lambda_x^{s_0-2}\big( \mathrm{e}^{-\rho} \epsilon^{ijk} \partial_k h \Delta v_j \big)dx d\tau.
\end{equation}
Using $\epsilon^{ijk} \partial_i \Delta v_j= \Delta \bw^k $ and $\epsilon^{ijk} \partial_i \partial_k h=0 $, we can transform \eqref{pre2} to
 \begin{equation*}
 	\begin{split}
 			& \int^t_0 \int_{\mathbb{R}^3}\Lambda_x^{s_0-2} \big(    \partial_n v^a \partial^n w_a \big)\cdot \{ \Lambda_x^{s_0-2} \{ \partial_i ( \mathrm{e}^{-\rho} \epsilon^{ijk} \partial_k h) \Delta v_j \} dx+ \Lambda_x^{s_0-2} [\mathrm{e}^{-\rho}  \partial_k h \cdot \Delta \bw^k ] \} dxd\tau
 			\\
 		=	& \int^t_0 \int_{\mathbb{R}^3}\Lambda_x^{s_0-2} \big(    \partial_n v^a \partial^n w_a \big)\cdot \{ \Lambda_x^{s_0-2} \{ \epsilon^{ijk} \partial_i  (\mathrm{e}^{-\rho})  \partial_k h \Delta v_j \} dx+ \Lambda_x^{s_0-2} [\mathrm{e}^{-\rho}  \partial_k h \cdot \Delta \bw^k ] \} dxd\tau.
 	\end{split}
 \end{equation*}
For the entropy, we
%also
calculate
\begin{equation*}
	\begin{split}
	& \mathrm{curl} \bH= \partial \rho \partial h,
\\
	& \mathbf{T}( \mathrm{e}^{\rho} \partial_i H^i)=\partial \bv   \partial^2 h   + \bw \partial \rho \partial h,
\end{split}
\end{equation*}
and
\begin{equation}\label{pref}
	\begin{split}
		\mathbf{T} \left\{ \partial^k( \mathrm{e}^{\rho} \partial_i H^i) \right\}=&-2\partial^k(\partial_i v^j) \partial_j  \partial^i h+\partial \bv \partial^3 h
		+ \partial \bw \partial \rho \partial h 
		\\
		& + \bw \partial \rho \partial^2h
		 +\bw \partial \rho \partial \rho \partial h + \bw \partial h \partial^2 h.
	\end{split}
\end{equation}
We can obtain an $L^2_x$ estimate of $h$ by using \eqref{fc0}. To  bound  $\|h\|_{\dot{H}_x^{s_0+1}}$, we operate on \eqref{pref} with $\Lambda^{s_0-2}_x$, multiply it with $\Lambda^{s_0-2}_x \partial^k( \mathrm{e}^{\rho} \partial_i H^i)$, and integrate over $[0,t]\times \mathbb{R}^3$. The most difficult term is
\begin{equation}\label{ppp}
	\int^t_0 \int_{\mathbb{R}^3} \Lambda^{s_0-2}_x[\partial^k(\partial_m v^j) \partial_j  \partial^m h]\cdot \Lambda_x^{s_0-2} \partial_k ( \mathrm{e}^\rho \partial_i H^i) dx d\tau.
\end{equation}
Integrating \eqref{ppp} by parts yields
\begin{equation}\label{ppp0}
	\begin{split}
		&\int^t_0 \int_{\mathbb{R}^3} \Lambda_x^{s_0-2}[\partial^k v^j \partial_m(\partial_j  \partial^m h)]\cdot \Lambda_x^{s_0-2} \partial_k ( \mathrm{e}^\rho \partial_i H^i) dx d\tau
		\\
		& +  \int^t_0 \int_{\mathbb{R}^3} \Lambda_x^{s_0-2}(\partial^k v^j \partial_j  \partial^m h )\cdot \Lambda_x^{s_0-2} \partial_m \partial_k ( \mathrm{e}^\rho \partial_i H^i) dx d\tau .
	\end{split}
\end{equation}
The first term in \eqref{ppp0} can be estimated by using product estimates. We rewrite the second term in \eqref{ppp0} as
\begin{align}
		& \int^t_0 \int_{\mathbb{R}^3} \Lambda_x^{s_0-2}(\partial^k v^j \partial_j  \partial^m h )\cdot \Lambda_x^{s_0-2} \partial_m \partial_k ( \mathrm{e}^\rho \partial_i H^i) dx d\tau \nonumber
		\\
		= & \int^t_0 \int_{\mathbb{R}^3} \Lambda_x^{s_0-2}\{ ( \partial^k v^j- \partial^j v^k) \partial_j  \partial^m h \} \cdot \Lambda_x^{s_0-2} \partial_m \partial_k ( \mathrm{e}^\rho \partial_i H^i) dx  d\tau \nonumber
		\\
		& - \int^t_0  \int_{\mathbb{R}^3} \Lambda_x^{s_0-2}(  \partial^j v^k \partial_j  \partial^m h ) \cdot \Lambda_x^{s_0-2} \partial_m \partial_k ( \mathrm{e}^\rho \partial_i H^i) dx  d\tau\nonumber
		\\\label{eq06}
		=& \int^t_0  \int_{\mathbb{R}^3} \Lambda_x^{s_0-2}\{ \epsilon^{lkj}\mathrm{e}^{\rho} w_l \partial_j  \partial^m h \} \cdot \Lambda_x^{s_0-2} \partial_m \partial_k ( \mathrm{e}^\rho \partial_i H^i) dx  d\tau 	
		\\\label{eq07}
		& - \int^t_0  \int_{\mathbb{R}^3} \Lambda_x^{s_0-2}(  \partial^j v^k \partial_j  \partial^m h ) \cdot \Lambda_x^{s_0-2} \partial_m \partial_k ( \mathrm{e}^\rho \partial_i H^i) dx  d\tau.
		\end{align}
We then use the Plancherel formula and the H\"older  inequality to bound \eqref{eq06}. Integrating \eqref{eq07} by parts yields
\begin{align}\label{eq08}
\eqref{eq07}= & -2 \int^t_0 \int_{\mathbb{R}^3} \Lambda_x^{s_0-2} (  \partial^j v^k \partial_k \partial_j  \partial^m h ) \cdot \Lambda_x^{s_0-2} \partial_m  ( \mathrm{e}^\rho \partial_i H^i) dx d\tau
\\\label{eq09}
		& -2 \int^t_0 \int_{\mathbb{R}^3} \Lambda_x^{s_0-2}( \partial_k  \partial^j v^k \partial_j  \partial^m h ) \cdot \Lambda_x^{s_0-2} \partial_m  ( \mathrm{e}^\rho \partial_i H^i) dx d\tau.
\end{align}
By using Lemma \ref{lpe}, we can bound \eqref{eq08}. For \eqref{eq09}, we use $\mathbf{T}\rho$ to replace $\partial_k v^k$ and then integrate it by parts. Gathering these above estimates together, and using Young's inequality, Gronwall's inequality, we can get the desired energy estimates:
\begin{equation}\label{EEE0}
	\begin{split}
		& \| (\bv, \rho)\|_{H_x^s}+\|h\|_{H_x^{1+s_0}}+\|\bw\|_{H_x^{s_0}}
		\leq E_0  \exp \{ {\int^t_0} (\|(d\bv, d\boldsymbol{\rho}, dh)\|_{L^\infty_x}+ \|\partial \bv, \partial h, \partial \rho\|_{\dot{B}^{s_0-2}_{\infty,2}})d\tau  \},
	\end{split}
\end{equation}
where $E_0$ are determined by the Sobolev norms $\|\bv_0\|_{H_x^s}, \|\rho_0\|_{H_x^s}$, $\|h_0\|_{H_x^{s_0+1}}$, and $\|\bw_0\|_{H_x^{s_0}}$. Hence, a Strichartz estimate for
\begin{equation}\label{SO}
	{\int^t_0} \|(d\bv, d\rho, dh)\|_{L^\infty_x}d\tau+{\int^t_0}\|\partial \bv, \partial\rho, \partial h\|_{\dot{B}^{s_0-2}_{\infty,2}}d\tau.
\end{equation}
is needed.
See Section \ref{ES1} for details.

$\bullet \ \textit{Step 2: Reduction to a result with small initial data}$.
Since the system \eqref{fc0} has finite propagation speed, by a standard compactness method, scaling and localization technique, we can reduce Theorem \ref{dingli} to the following statement. For any smooth, supported data $(\bv_0, {\rho}_0, h_0, \bw_0)$ satisfying
	\begin{equation*}
		\begin{split}
			&\|\bv_0\|_{H^s} + \|\rho_0 \|_{H^s} + \|\bw_0\|_{H^{s_0}}+ \|h_0\|_{H^{s_0+1}} \leq \epsilon_3,
		\end{split}
	\end{equation*}
then the Cauchy problem \eqref{fc0} admits a smooth solution $(\bv,\rho,h,\bw)$ on $[-1,1] \times \mathbb{R}^3$, which has the following properties:
	
	$\mathrm{(i)}$ energy estimates
	\begin{equation}\label{pre3}
		\begin{split}
			&\|\bv\|_{L^\infty_t H_x^{s}}+\| \rho \|_{L^\infty_t H_x^{s}} + \| \bw\|_{L^\infty_t H_x^{s_0}} + \| h \|_{L^\infty_t H_x^{s_0+1}} \leq \epsilon_2,
		\end{split}
	\end{equation}

	$\mathrm{(ii)}$ dispersive estimate for $\bv$ and $\rho$
	\begin{equation}\label{pre4}
		\|d \bv, d \rho, dh\|_{L^2_t C^\delta_x}+\| d \rho, \partial \bv_{+}, d \bv, dh\|_{L^2_t \dot{B}^{s_0-2}_{\infty,2}} \leq \epsilon_2,
	\end{equation}

	$\mathrm{(iii)}$ dispersive estimate for the linear equation

	Let $f$ satisfy
	the equation
		\begin{equation*}
		\begin{cases}
			& \square_{\mathbf{g}} f=0, \qquad (t,x) \in [-1,1]\times \mathbb{R}^3,
			\\
			&f(t_0,\cdot)=f_0 \in H_x^r(\mathbb{R}^3), \quad \partial_t f(t_0,\cdot)=f_1 \in H_x^{r-1}(\mathbb{R}^3).
		\end{cases}
	\end{equation*}
Here $\mathbf{g}$ is defined in \eqref{boldg}, which is a localization of $g$. For each $1 \leq r \leq s+1$, the Cauchy problem \eqref{linear} is well-posed in $H_x^r \times H_x^{r-1}$, and the following estimate holds:
\begin{equation}\label{pre5}
	\|\left< \partial \right>^k f\|_{L^2_{[-1,1]} L^\infty_x} \lesssim  \| f_0\|_{H^r}+ \| f_1\|_{H^{r-1}},\quad k<r-1.
\end{equation}
See Section \ref{Sec4} for the complete proof.

$\bullet \ \textit{Step 3: Reduction to a continuity argument}$. Due to the fact that the initial data $(\bv_0, {\rho}_0, h_0, \bw_0)$ is small, we can consider a perturbation problem around Minkowski space. We further
reduce the proof of \eqref{pre3}-\eqref{pre5} to: there is a continuous functional $G: \mathcal{H} \rightarrow \mathbb{R}^{+}$, satisfying $G(0)=0$, so that for each $(\bv, \rho, h, \bw) \in \mathcal{H}$ satisfying $G(\bv, \rho, h) \leq 2 \epsilon_1$ the following hold: $\mathrm{(i)}$ the function $\bv, \rho, h$, and $\bw$ satisfies $G(\bv, \rho, h) \leq \epsilon_1$; $\mathrm{(ii)}$ the estimates \eqref{pre3}-\eqref{pre4} and \eqref{pre5} both hold. By using energy estimates in Theorem \ref{ve}, the estimate \eqref{pre3} follows from \eqref{pre4}. Hence, the proof of $G(\bv, \rho, h) \leq \epsilon_1$ and \eqref{pre4}-\eqref{pre5} are the core. See Section \ref{ABA} for details.

$\bullet \ \textit{Step 4: Definition of $G$ and characteristic energy estimates}$.
%Let $\Gamma_{\theta}$ be the flowout of \red{this section \mnote{LA: which section?}}  under the Hamiltonian flow of $ \mathbf{g}$.

Following \cite{ST}, for $\theta \in \mathbb{S}^2$, $\bx \in \mathbb{R}^3$, and let $\Sigma_{\theta, r}$ be the flowout of the set $\theta \cdot \bx = r-2$ along the null geodesic flow with respect of $\mathbf{g}$ in the direction $\theta$ at $t=-2$.
Let $\bx'_{\theta}$ be given orthonormal coordinates on the hyperplane in $\mathbb{R}^3$ perpendicular to $\theta$. Then, $\Sigma_{\theta,r}$ is of the form
\begin{equation*}
	\Sigma_{\theta,r}=\left\{ (t,\bx): \theta\cdot \bx-\phi_{\theta, r}=0  \right\}
\end{equation*}
for a smooth function $\phi_{\theta, r}(t,\bx'_{\theta})$.
For a given $\theta$, the family $\{\Sigma_{\theta,r}, \ r \in \mathbb{R}\}$ defines a foliation of $[-2,2]\times \mathbb{R}^3$, by characteristic hypersurfaces with respect to $\mathbf{g}$.

We define (see \eqref{500} for details)
\begin{equation}\label{dG}
	G= \sup_{\theta, r} \vert\kern-0.25ex\vert\kern-0.25ex\vert d \phi_{\theta,r}-dt\vert\kern-0.25ex\vert\kern-0.25ex\vert_{s_0,2,{\Sigma_{\theta,r}}}.
\end{equation}
For simplicity, let us suppose $\theta=(0,0,1)$ and $r=0$. In this case, we set $\Sigma=\Sigma_{\theta,r}$ and $\phi=\phi_{\theta,r}$. To prove $G \leq \epsilon_1$ and catch the geometry properties of $\Sigma$, let us introduce a null frame along $\Sigma$ as follows. Let
\begin{equation*}
	V=(dr)^*,
\end{equation*}
where $r$ is the foliation of $\Sigma$, and where $*$ denotes the identification of covectors and vectors induced by $\mathbf{g}$. Define a null frame with respect to $\mathbf{g}$ by setting
\begin{equation*}
	l=\left< dt,dx_3-d\phi\right>^{-1}_{\mathbf{g}} \left( dx_3-d \phi \right)^*, \quad \underline{l}=l+2\partial_t,
\end{equation*}
and choosing orthonormalized vector fields tangent to tangent to the fixed-time slice $\Sigma^t$ of $\Sigma$.
%Then $\{l, \underline{l}, e_1, e_2\}$ is a null frame of $\Sigma$.
By further calculation (see Corollary \ref{vte}, Lemma \ref{te20}, and Lemma \ref{te21}), the bound of $G$ relies on
$\vert\kern-0.25ex\vert\kern-0.25ex\vert (\bv, \rho, h)\vert\kern-0.25ex\vert\kern-0.25ex\vert_{s_0,2,{\Sigma}}$ and $\vert\kern-0.25ex\vert\kern-0.25ex\vert ( \mathrm{curl}\bw, \Delta h ) \vert\kern-0.25ex\vert\kern-0.25ex\vert_{s_0-1,2,{\Sigma}}$. Hence, we need go back to the hyperbolic structure and the wave-transport structure. By a change of coordinates, we could use the hyperbolic structure to get a bound of $\vert\kern-0.25ex\vert\kern-0.25ex\vert \bv, \rho, h\vert\kern-0.25ex\vert\kern-0.25ex\vert_{s_0,2,{\Sigma_{\theta,r}}}$. As for $\mathrm{curl}\bw$ and $\Delta h$, the characteristic energy estimate is more difficult. For example, on the Cauchy slice $\{t=\tau\}\times \mathbb{R}^3$, we can use elliptic estimates to get the energy estimate of all derivatives of $\bw$ by using $\mathrm{div}\bw$ and $\mathrm{curl}\bw$. However, on the characteristic hypersurface, this type of elliptic energy estimates does not work. To overcome this difficulty, we use the Hodge decomposition to recover some transport equations for derivatives of $\bw$ and $\Delta h$, which is related to Riesz operator. To do that, some commutator estimates for the Riesz operator and $\bv \cdot \nabla$ concerning the compressible fluid are required. See Lemma \ref{ceR}, \ref{LPE}, \ref{ce}, \ref{te3}, and \ref{te20} for details. For the whole proof of the characteristic energy estimates, cf. Section \ref{sec6}.

$\bullet \ \textit{Step 5: Strichartz estimates}$. We should note that \eqref{pre5} of linear wave is not sufficient for us to establish the desired Strichartz estimate \eqref{pre4}. This leads us to introduce Lemma \ref{LD}. Using Lemma \ref{LD}, we are able to establish the following Strichartz estimate for a non-homogeneous linear wave equation. For any $1 \leq r \leq s+1$, for each $t_0 \in [-2,2)$, then the Cauchy problem
\begin{equation}\label{pre7}
	\begin{cases}
		& \square_{\mathbf{g}} f=\mathbf{T}F, \qquad (t,x) \in (t_0,2]\times \mathbb{R}^3,
		\\
		&f(t,x)|_{t=t_0}=f_0 \in H_x^r(\mathbb{R}^3),
		\\
		&\partial_t f(t,x)|_{t=t_0}=f_1-F(t_0,\cdot) \in H_x^{r-1}(\mathbb{R}^3),
	\end{cases}
\end{equation}
admits a solution $f \in C([-2,2],H_x^r) \times C^1([-2,2],H_x^{r-1})$ and the following estimates holds, provided $k<r-1$,
\begin{equation}\label{pre6}
	\| \left<\partial \right>^{k} f, \left<\partial \right>^{k-1} df\|_{L^2_{[-2,2]}L^\infty_x} \lesssim \|f_0\|_{H_x^r}+ \|f_1\|_{H_x^{r-1}}+\| F\|_{L^1_{[-2,2]}H_x^r}.
\end{equation}
So \eqref{pre6} is stronger than that in wave equations \cite{ST}. Please see Proposition \ref{r3} and Proposition \ref{r5} for details. As for \eqref{pre4}, it is essential for us to use a better wave equations for $\bv_{+}$ and $\rho+
\frac{1}{\gamma} h$ (see Lemma \ref{crh} and Lemma \ref{wte1}),
\begin{equation}\label{pre8}
\begin{cases}
	&{\square}_g (\rho +\frac{1}{\gamma} h)=D+ \frac{1}{\gamma}E,
	\\
	&{\square}_g \bv_{+}=\mathbf{T} \mathbf{T} {\bv}_{-}+\bQ.
\end{cases}
\end{equation}
We shall mention that this good structure of $\bv_{+}$ is firstly proposed by Wang \cite{WQEuler}.

Operating with $\Delta_j$ on \eqref{pre8} and using \eqref{pre5}-\eqref{pre7} (taking $r=s-s_0+1, k=0$), it follows that
\begin{equation}\label{pre9}
\begin{split}
	& \| (\Delta_j (d\rho +\frac{1}{\gamma} dh), \Delta_j d\bv_{+})\|_{L^2_{[-2,2]} L^\infty_x}
	\\
	\lesssim \ & \|\Delta_j D\|_{L^1_{[-2,2]} H_x^{s-s_0+1} } + \| \Delta_j E \|_{L^1_{[-2,2]} H_x^{s-s_0+1} }  + \| \Delta_j \bQ \|_{L^1_{[-2,2]} H_x^{s-s_0+1} }
	\\
	& +  \| \Delta_j (\rho, \bv, h) \|_{{L^\infty_{[-2,2]} H_x^{s-s_0+2} }} +  \| \Delta_j (\mathbf{T} {\bv}_{-}) \|_{{L^1_{[-2,2]} H_x^{s-s_0+1} }}
	\\
	&  + \| [\mathbf{T}, \Delta_j](\mathbf{T} {\bv}_{-}) \|_{{L^1_{[-2,2]} H_x^{s-s_0+1} }}+ \| [\square_{{g}}, \Delta_j] {\bv}_{+} \|_{{L^1_{[-2,2]} H_x^{s-s_0+1} }}.
\end{split}
\end{equation}
Multiplying \eqref{pre9} by $2^{(s_0-2)j}$, squaring and summing it over $j\geq 1$, we get
\begin{equation}\label{prea}
\| (d\rho +\frac{1}{\gamma} dh), d{\bv}_{+}\|_{L^2_{[-2,2]} \dot{B}^{s_0-2}_{\infty, 2}}
\lesssim \| \bv, \rho, h\|_{L^\infty_{[-2,2]} H_x^s}+ \| h\|_{L^2_{[-2,2]} H_x^{\frac52+}}+ \| \bw\|_{ L^2_{[-2,2]}H_x^{\frac{3}{2}+}}.
\end{equation}
Using \eqref{prea} and Sobolev imbedding for $h$ and $\bv_{-}$, allows us to prove \eqref{pre4}. Moreover, when $ 2< s \leq\frac{5}{2}$ and $a<s-1$, a general Strichartz estimate yields
\begin{equation}\label{preb}
	\begin{split}
		 \|\left< \partial \right>^{a-1} (d\rho+\frac{1}{\gamma}dh), \left< \partial \right>^{a-1} d\bv_{+}\|_{L^2_{[-2,2]} \dot{B}^{0}_{\infty, 2}}
		\lesssim  & \| \bv, \rho, h\|_{L^\infty_{[-2,2]} H_x^s}+ \| h\|_{L^2_{[-2,2]} H_x^{\frac52+}}
		\\
		\qquad & + \| \bw\|_{ L^2_{[-2,2]}H_x^{\frac{3}{2}+}}.
	\end{split}
\end{equation}
The entire proof can be found in Proposition \ref{r6} and Proposition \ref{r4}.

$\bullet \ \textit{Step 6: Proof of the estimates \eqref{pre6}-\eqref{pre7}}$. Through the above steps, we have obtained sufficient regularity of the characteristic hypersurfaces to control
%, which ensure the geometric conditions of
the characteristic hypersurfaces. This allows us to use the idea of constructing approximate solutions by wave-packets \cite{ST,Sm,Wo}. The whole proof is presented in Section \ref{Ap}.

$\bullet \ \textit{Step 7: Continuous dependence.}$ If we prove the continuous dependence directly, we will find that the structure of the system is destroyed and there is loss of derivatives. In section \ref{Sub}, we overcome the difficulty by using \eqref{pre5} and a frequency decomposition technique. We first prove convergence in a weaker space. Secondly, we construct a sequence of smooth solutions  $\{(\bv^l,\rho^l,h^l,\bw^l)\}_{l \in \mathbb{Z}^+}$  satisfying \eqref{dbsh}-\eqref{dbs1}, which crucially relies on the Strichartz estimates \eqref{pre5} for Equation \eqref{Ss4}. Finally, we use we prove the convergence in $H^s_x \times H^s_x \times H^{s_0+1}_x\times H^{s_0}_x$ by making use of $(\bv^l,\rho^l,h^l,\bw^l)$. See section \ref{Sub} for details.

From the wave structure in equations \eqref{fc1} and \eqref{fc1s}, the regularity of characteristic hypersurfaces
is governed by the vorticity. When the vorticity has Sobolev regularity $s'\leq 2$, neither the vector-field approach by Wang \cite{WQEuler} nor Smith-Tataru method \cite{ST} can be applied directly to establish Strichartz estimate of velocity and logarithmic density, because they require higher(at least $2+$ order) regularity of the characteristic surfaces. To overcome the difficulty, a crucial insight relies on a new type of Strichartz estimate for solutions and Strichartz estimates for linear wave equations endowed with an acoustic metric, which allows for loss of derivatives but can be summed over short time intervals to obtain overall estimates.

The proofs of Theorems \ref{dingli2} and \ref{dingli3} have a common framework, adapted to the different regularity
regimes. The main steps are outlined below.
\subsubsection{A sketch of proof of Theorem \ref{dingli2}.}
The main idea is in the following.

$\bullet \ \textit{Step 1: Energy estimates}$. By using the structure of \eqref{CEE} and a fractional Leibniz rule, the energy estimate $\|  \bv\|_{L^\infty_{[0,t]} H_x^{\frac52}}+\| \rho\|_{L^\infty_{[0,t]} H_x^{\frac52}}+\|h\|_{L^\infty_{[0,t]} H_x^{\frac52+\epsilon}}+\|\bw\|_{L^\infty_{[0,t]} H_x^{\frac32+\epsilon}} $ is established in Theorem \ref{dingli2}. As a result, if we expect a bound of the energy
\begin{equation*}
	\|  \bv\|_{L^\infty_{[0,t]} H_x^{\frac52}}+\| \rho\|_{L^\infty_{[0,t]} H_x^{\frac52}}+\|h\|_{L^\infty_{[0,t]} H_x^{\frac52+\epsilon}}+\|\bw\|_{L^\infty_{[0,t]} H_x^{\frac32+\epsilon}}
\end{equation*}
then the goal is to estimate
\begin{equation}\label{pqqqq}
	\|d\bv, d\rho,dh\|_{L_{[0,t]}^2 L_x^\infty}.
\end{equation}
See Section \ref{ES3} for its proof.

$\bullet \ \textit{Step 2: Constructing a strong solution as a limit of smooth solutions}$.
Since the initial data is not so smooth, we apply a mollifier to $(\bv_{0},\rho_0,h_0, \bw_0)$. Thus we consider a sequence
%$(\bv_{0j},\rho_{0j},\bw_{0j})$ such that
\begin{align*}
	(\bv_{0j},\rho_{0j},h_{0j})=(P_{\leq j}\bv_{0},P_{\leq j}\rho_{0},P_{\leq j}h_{0}), \quad \bw_{0j}= \textrm{e}^{-\rho_{0j}}\textrm{curl} \bv_{0j},
\end{align*}
where $P_j=\sum_{j' \leq j}P_{j'}$ and $P_{j'}$ is a LP operator with frequency supports $\{ 2^{j'-3} \leq |\xi| \leq 2^{j'+3} \}$.

Considering above initial datum $(\bv_{0j},\rho_{0j}, h_{0j}, \bw_{0j})$, there exists $\bar{T}>0$, $\bar{M}_1, \bar{M}_3>0$ ($\bar{T}$, $\bar{M}_1$ and $\bar{M}_3$ only depends on $\bar{M}_0$ and $C_0, c_0, \epsilon$) such that a sequence solutions of \eqref{fc1} satisfy $(\bv_{j},\rho_{j})\in C([0,\bar{T}];H_x^\frac52)\cap C^1([0,\bar{T}];H_x^{\frac32})$, $h_{j}\in C([0,\bar{T}];H_x^{\frac52+\epsilon}) \cap C^1([0,\bar{T}];H_x^{\frac32+\epsilon})$, $\bw_{j}\in C([0,\bar{T}];H_x^{\frac32+\epsilon}) \cap C^1([0,\bar{T}];H_x^{\frac12+\epsilon})$ and
\begin{equation}\label{prec01}
	\|d\bv_j, d\rho_j, dh_j\|_{L^2_{[0,\bar{T}]}L_x^\infty} \leq \bar{M}_{1}.
\end{equation}
Moreover, for $\frac{7}{6} \leq r < \frac72$, consider the following linear wave equation
\begin{equation}\label{prec02}
	\begin{cases}
		\square_{{g}_j} f_j=0, \qquad [0,\bar{T}]\times \mathbb{R}^3,
		\\
		(f_j,\partial_t f_j)|_{t=0}=(f_{0j},f_{1j}),
	\end{cases}
\end{equation}
where $(f_{0j},f_{1j})=(P_{\leq j}f_0,P_{\leq j}f_1)$ and $(f_0,f_1)\in H_x^r \times H^{r-1}_x$. Then there is a unique solution $f_j$ on $[0,\bar{T}]\times \mathbb{R}^3$ such that for $a \leq r-\frac76$,
\begin{equation}\label{prec03}
	\begin{split}
		&\|\left< \partial \right>^{a-1} d{f}_j\|_{L^2_{[0,\bar{T}]} L^\infty_x}
		\leq  {\bar{M}_3}(\|{f}_0\|_{{H}_x^r}+ \|{f}_1 \|_{{H}_x^{r-1}}),
		\\
		&\|{f}_j\|_{L^\infty_{[0,\bar{T}]} H^{r}_x}+ \|\partial_t {f}_j\|_{L^\infty_{[0,\bar{T}]} H^{r-1}_x} \leq  {\bar{M}_3} (\| {f}_0\|_{H_x^r}+ \| {f}_1\|_{H_x^{r-1}}).
	\end{split}
\end{equation}
Based on \eqref{prec01}-\eqref{prec03}, then we are able to obtain a strong solution of Theorem \ref{dingli2} as a limit of the sequence $(\bv_{j},\rho_{j}, h_j, \bw_j)$. Please see Subsection \ref{keyc} for details.

It remains for us to prove \eqref{prec01}-\eqref{prec03}.

$\bullet \ \textit{Step 3: Obtaining a sequence of solutions on small time-intervals}$. We note $\|{\bw}_{0j}\|_{H^{2+\epsilon}}\lesssim 2^{\frac{j}{2}}$, and also $\|{h}_{0j}\|_{H^{\frac52+\epsilon}} \lesssim 2^{\frac{j}{2}}$. By using compactness, scaling, and physical localization, we shall get a sequence of small initial datum. By using Proposition \ref{DL3}, we can prove the existence of small solutions with these small datum. On one hand, Proposition \ref{DL3} can be derived from \eqref{preb}. On the other hand, returning back from small solutions to large solutions $(\bv_{j},\rho_{j},h_j, \bw_{j})$, then the time of existence of solutions $(\bv_{j},\rho_{j},h_j,\bw_{j})$ depends on $j\in \mathbb{Z}^+$. Because the norm $\|{\bw}_{0j}\|_{H^{2+\epsilon}}$ depending on $j$. Moreover, we precisely calculate $T_j=2^{-\frac{j}{3+\epsilon}} [\mathbb{E}(0)]^{-1}$, where $\mathbb{E}(0)=\|\bv_{0j}\|_{H^\frac52}+ \|\rho_{0j}\|_{H^\frac52} + \|h_{0j}\|_{H^\frac52+\epsilon}+ \|\bw_{0j}\|_{H^\frac32+\epsilon} $ and please see \eqref{qu01} and \eqref{qu03} for its formulations. On time-interval $[0, T_j]$, the solutions $(\bv_j, \rho_j, h_j)$ yields a higher-order Strichartz estimates,
\begin{equation}\label{prec04}
	\|d\rho_j, d\bv_{j}\|_{L^2_{[0,T_j]} C^a_x} \leq C \{ T_j \}^{-(\frac12+a)}  (1+\mathbb{E}(0)), \quad a<\sstar-1.
\end{equation}
Therefore, for high frequency $k\geq j$, by choosing a suitable number $a=\frac12-\frac{\epsilon}{40}$ in \eqref{prec04}, so we obtain
\begin{equation}\label{prec05}
	\begin{split}
		\| P_{k} d \bv_j, P_{k} d \rho_j, P_{k} d h_j \|_{L^2_{[0,T_j]}L^\infty_x}
		\leq & C  (1+\mathbb{E}^3(0)) \cdot 2^{-\frac{\epsilon}{40}k}  2^{-\frac{\epsilon}{10}j} 2^{-\frac{1}{2(3+\epsilon)}j},
	\end{split}
\end{equation}
Decompose
\begin{equation*}
	\begin{split}
		d \rho_j= P_{\geq j}d \rho_j+ \textstyle{\sum}^{j-1}_{k=1} \textstyle{\sum}_{m=k}^{j-1} P_k  (d\rho_{m+1}-d\rho_m)+ \textstyle{\sum}^{j-1}_{k=1}P_k d \rho_k .
	\end{split}
\end{equation*}
Then we need to see if there is a similar estimates for $P_k  (d\bv_{m+1}-d\bv_m)$. By using a Strichartz estimates for linear wave, for $k<j$, we shall also obtain
\begin{equation}\label{prec06}
	\begin{split}
		& \|P_k (d\rho_{m+1}-d\rho_m), P_k (d\bv_{m+1}-d\bv_m), P_k (dh_{m+1}-dh_m)\|_{L^2_{[0,T_{m+1}]} L^\infty_x}
		\\
		\leq & C(1+\mathbb{E}^2(0)) 2^{-\frac{1}{3+\epsilon}m}2^{-\frac{\epsilon}{10} k} 2^{-\frac{7\epsilon}{10} m}.
	\end{split}	
\end{equation}
We prove \eqref{prec04}-\eqref{prec06} in Subsection \ref{keyd}.

$\bullet \ \textit{Step 4: Extending a sequence of solutions to a regular time-intervals}$. Compared \eqref{prec} with \eqref{prec05}-\eqref{prec06}, there is a loss of derivatives in \eqref{prec05}-\eqref{prec06}. Fortunately, there are some decay estimates \eqref{prec05}-\eqref{prec06} on frequency number $j$ and $k$. If we can always extend it with a length $\approx T_j$ of time intervals, and with a number $\approx (T_j)^{-1}$, then we shall sum these estimates up and obtain a uniform Strichartz estimates $\|d \rho_j, d\bv_{j}, dh_j\|_{L^1_{[0,\bar{T}]} L^\infty_x}$ and $\|d \rho_j, d\bv_{j}, dh_{j}\|_{L^2_{[0,\bar{T}]} L^\infty_x}$ on a regular time-interval $[0,\bar{T}]$. In fact, we can extend it in this way from $[0,T_j]$ to $[0,T_{N_1}]$. We first extend it from $[0,T_j]$ to $[0,T_{j-1}]$. In this process, the growth in Strichartz estimates and energy estimates can be calculate precisely. Then we conclude it by induction method. Please refer Subsection \ref{keye} for details. Finally, we get
\begin{equation*}
	\bar{T}=2^{-\frac{N_1}{3+\epsilon}} [\mathbb{E}(0)]^{-1},
\end{equation*}
where $N_1$ is an fixed integer depending on $\bar{M}_0$ and $C_0, c_0, \epsilon$(please see \eqref{qu04}). Furthermore,
\begin{equation*}
	\| \bv_j \|_{L^\infty_{[0,\bar{T}]}H_x^{\frac52}}+ \| \rho_j \|_{L^\infty_{[0,\bar{T}]}H_x^{\frac52}} +  \| h_j \|_{L^\infty_{[0,\bar{T}]}H_x^{\frac52+\epsilon}}+ \| \bw_j \|_{L^\infty_{[0,\bar{T}]}H_x^{\frac32+\epsilon}},
\end{equation*}
and
\begin{equation*}
	\|d \rho_j\|_{L^2_{[0,\bar{T}]} L^\infty_x} + \|d\bv_{j}\|_{L^2_{[0,\bar{T}]} L^\infty_x}+ \|dh_{j}\|_{L^2_{[0,\bar{T}]} L^\infty_x}
\end{equation*}
is uniformly bounded (See \eqref{qu100}-\eqref{qu101}). In a similar idea, \eqref{prec02}-\eqref{prec03} also holds. For the whole proof, see Section \ref{Sec9}.

\subsubsection{A sketch of proof of Theorem \ref{dingli3}} Our main steps are as follows.

$\bullet \ \textit{Step 1: Energy estimates}$. The energy estimate $\|  \bv\|_{L^\infty_{[0,t]}H_x^{\sstar}}+\| \rho\|_{L^\infty_{[0,t]}H_x^{\sstar}}+\|\bw\|_{L^\infty_{[0,t]} H_x^{2}}$ is established in Theorem \ref{dingli3}. Therefore, if we expect a bound of the energy
\begin{equation*}
	\begin{split}
		& \|  \bv\|_{L^\infty_{[0,t]}H_x^{\sstar}}+\| \rho\|_{L^\infty_{[0,t]}H_x^{\sstar}}+\|\bw\|_{L^\infty_{[0,t]} H_x^{2}},
	\end{split}
\end{equation*}
then the goal is to estimate
\begin{equation}\label{pqqq}
	\|d\bv, d\rho\|_{L_{[0,t]}^2 L_x^\infty}.
\end{equation}
We refer the reader to Section \ref{ES2}.

$\bullet \ \textit{Step 2: Constructing a strong solution as a limit of smooth solutions}$.
Since the initial data is not so smooth, we apply a mollifier to $(\bv_{0},\rho_0,\bw_0)$. Thus we consider a sequence
%$(\bv_{0j},\rho_{0j},\bw_{0j})$ such that
\begin{align*}
	(\bv_{0j},\rho_{0j})=(P_{\leq j}\bv_{0},P_{\leq j}\rho_{0}), \quad \bw_{0j}= \textrm{e}^{-\rho_{0j}}\textrm{curl} \bv_{0j},
\end{align*}
where $P_j=\sum_{j' \leq j}P_{j'}$ and $P_{j'}$ is a LP operator with frequency supports $\{\xi\in\mathbb{R}^3: 2^{j'-3} \leq |\xi| \leq 2^{j'+3} \}$.

Considering above initial datum $(\bv_{0j},\rho_{0j}, \bw_{0j})$, there exists $T^*>0$, $\bar{M}_2, \bar{M}_4>0$ ($T^*$, $\bar{M}_2$ and $\bar{M}_4$ only depends on $C_0, c_0, s$ and $M_*$) such that a sequence solutions of \eqref{fc1s} satisfy $(\bv_{j},\rho_{j})\in C([0,T^*];H_x^s)\cap C^1([0,T^*];H_x^{s-1})$, $\bw_{j}\in C([0,T^*];H_x^2) \cap C^1([0,T^*];H_x^1)$. Moreover, for all $j\geq 1$, $\bv_j$ and $\rho_j$ satisfy a uniform Strichartz estimate
\begin{equation}\label{prec1}
	\|d\bv_j, d\rho_j\|_{L^2_{[0,T^*]}L_x^\infty} \leq \bar{M}_{2}.
\end{equation}
Additionally, for $\frac{s}{2} \leq r \leq 3$, considering
\begin{equation}\label{prec2}
	\begin{cases}
		\square_{{g}_j} f_j=0, \qquad [0,T^*]\times \mathbb{R}^3,
		\\
		(f_j,\partial_t f_j)|_{t=0}=(f_{0j},f_{1j}),
	\end{cases}
\end{equation}
where $(f_{0j},f_{1j})=(P_{\leq j}f_0,P_{\leq j}f_1)$ and $(f_0,f_1)\in H_x^r \times H^{r-1}_x$, there is a unique solution $f_j$ on $[0,T^*]\times \mathbb{R}^3$ such that for $a\leq r-\frac{s}{2}$,
\begin{equation}\label{prec4}
	\begin{split}
		&\|\left< \partial \right>^{a-1} d{f}_j\|_{L^2_{[0,T^*]} L^\infty_x}
		\leq  {\bar{M}_4}(\|{f}_0\|_{{H}_x^r}+ \|{f}_1 \|_{{H}_x^{r-1}}),
		\\
		&\|{f}_j\|_{L^\infty_{[0,T^*]} H^{r}_x}+ \|\partial_t {f}_j\|_{L^\infty_{[0,T^*]} H^{r-1}_x} \leq  {\bar{M}_4} (\| {f}_0\|_{H_x^r}+ \| {f}_1\|_{H_x^{r-1}}).
	\end{split}
\end{equation}
Based on \eqref{prec1}-\eqref{prec4}, then we are able to obtain a strong solution of Theorem \ref{dingli3} as a limit of the sequence $(\bv_{j},\rho_{j}, \bw_j)$. Please refer Subsection \ref{keypra} for details.

It remains for us to prove \eqref{prec1}-\eqref{prec4}.

$\bullet \ \textit{Step 3: Obtaining a sequence of solutions on small time-intervals}$. Note $\|{\bw}_{0j}\|_{H^{2+r}}\lesssim 2^{jr}$($r \geq 0$). By using compactness, scaling, and physical localization, we shall get a sequence of small initial datum. By using Proposition \ref{DDL3}, we can prove the existence of small solutions with these small datum. On one hand, Proposition \ref{DDL3} can be derived from \eqref{preb}. On the other hand, returning back from small solutions to large solutions $(\bv_{j},\rho_{j},\bw_{j})$, then the time of existence of solutions $(\bv_{j},\rho_{j},\bw_{j})$ depends on $j\in \mathbb{Z}^+$. Because the norm $\|{\bw}_{0j}\|_{H^{2+}}$ depending on $j$. Moreover, we precisely calculate $T^*_j=2^{-\delta_1j} [E(0)]^{-1}$($\delta_1=\frac{\sstar-2}{10}$ and $E(0)=\|\bv_{0j}\|_{H^s}+ \|\rho_{0j}\|_{H^s} + \|\bw_{0j}\|_{H^2} $, c.f \eqref{pu0} and \eqref{DTJ}). On $[0, T^*_j]$, the solutions $\bv_j$ and $\rho_j$ yields a higher-order Strichartz estimates,
\begin{equation}\label{prec}
	\|\left< \partial \right>^{a-1} d\rho_j, \left< \partial \right>^{a-1} d\bv_{j}\|_{L^2_{[0,T^*_j]} L^\infty_x} \leq C(1+E(0)), \quad a<\sstar-1.
\end{equation}
Therefore, by choosing a suitable number $a=9\delta_1$, for high frequency $k\geq j$, then we get
\begin{equation}\label{prec0}
	\|P_k d \rho_j, P_k d\bv_{j}\|_{L^2_{[0,T^*_j]} L^\infty_x} \leq C(1+E^3(0))2^{-\delta_{1} k} 2^{-7\delta_{1} j}.
\end{equation}
Decompose
\begin{equation*}
	\begin{split}
		d \bv_j= P_{\geq j}d \bv_j+ \textstyle{\sum}^{j-1}_{k=1} \textstyle{\sum}_{m=k}^{j-1} P_k  (d\bv_{m+1}-d\bv_m)+ \textstyle{\sum}^{j-1}_{k=1}P_k d \bv_k .
	\end{split}
\end{equation*}
Then we need to see if there is a similar estimates for $P_k  (d\bv_{m+1}-d\bv_m)$. By using a Strichartz estimates for linear wave, for $k<j$, we shall also obtain
\begin{equation}\label{prec00}
	\|P_k (d\rho_{m+1}-d\rho_m), P_k (d\bv_{m+1}-d\bv_m)\|_{L^2_{[0,T^*_{m+1}]} L^\infty_x} \leq C(1+E^3(0))2^{-\delta_{1} k} 2^{-6\delta_{1} m}.
\end{equation}
We prove \eqref{prec0}-\eqref{prec00} in Subsection \ref{esest}.

$\bullet \ \textit{Step 4: Extending a sequence of solutions to a regular time-intervals}$. Compared \eqref{prec} with \eqref{prec0}-\eqref{prec00}, there is a loss of derivatives in \eqref{prec0}-\eqref{prec00}. The benefit relies on some decay estimates \eqref{prec0}-\eqref{prec00} on frequency number $j$ and $k$. Inspired by Bahouri-Chemin \cite{BC2}, Tataru \cite{T1} and Ai-Ifrim-Tataru \cite{AIT}, if we can always extend it with a length $\approx T_j^*$ of time intervals, and with a number $\approx (T_j^*)^{-1}$, then we shall sum these estimates up and obtain a uniform Strichartz estimates $\|d \rho_j, d\bv_{j}\|_{L^1_{[0,T^*]} L^\infty_x}$ and $\|d \rho_j, d\bv_{j}\|_{L^2_{[0,T^*]} L^\infty_x}$ on a regular time-interval $[0,T^*]$. In fact, we can extend it in this way from $[0,T^*_j]$ to $[0,T^*_{N_0}]$. We first extend it from $[0,T^*_j]$ to $[0,T^*_{j-1}]$. In this process, the growth in Strichartz estimates and energy estimates can be calculate precisely. Then we conclude it by induction method. Please refer Subsection \ref{finalk} for details. Finally, we get
\begin{equation*}
	T^*=2^{-\delta_{1} N_0} [E(0)]^{-1},
\end{equation*}
where $N_0$ is an fixed integer depending on $M_*$ and $C_0, c_0, s$(please see \eqref{pp8}). Furthermore,
\begin{equation*}
	\| \bv_j \|_{L^\infty_{[0,T^*]}H_x^s}+ \| \rho_j \|_{L^\infty_{[0,T^*]}H_x^s} + \| \bw_j \|_{L^\infty_{[0,T^*]}H_x^2},
\end{equation*}
and
\begin{equation*}
	\|d \rho_j\|_{L^2_{[0,T^*]} L^\infty_x} + \|d\bv_{j}\|_{L^2_{[0,T^*]} L^\infty_x},
\end{equation*}
are uniformly bounded (See \eqref{kz65}-\eqref{kz66}). In a similar idea, \eqref{prec2}-\eqref{prec4} also holds. For the whole proof, see Section \ref{Sec8}.

\subsection{Organization of the paper} \label{sec:org-alt} Section \ref{sec:preliminaries} is concerned with deriving the structure of equations and some commutator estimates and some useful lemmas. This includes the highest-order derivatives of vorticity and entropy, i.e. Lemma \ref{PW}, Lemma \ref{PWh}, and a wave equation for $\rho+\frac1\gamma h,\bv_{+}$ in Lemma \ref{crh} and Lemma \ref{wte1}. In Section \ref{sec:energyest}, we prove some energy theorems Theorem \ref{ve}, Theorem \ref{vve}, and Theorem \ref{TT2} for Theorem \ref{dingli}, Theorem \ref{dingli2} and Theorem \ref{dingli3} respectively. We also prove the uniqueness of solutions based on these energy estimates, i.e. Corollary \ref{cor}, Corollary \ref{cor2}, and Corollary \ref{cor3}. In Section \ref{Sec4}-\ref{Ap}, we give a proof of existence of solutions for Theorem \ref{dingli}. Section \ref{Sub} presents the proof of continuous dependence of solutions in Theorem \ref{dingli}. In Sections \ref{SEs} and \ref{Ap}, we also establish some Strichartz estimates for the acoustic metric. Sections \ref{Sec8} contain the proofs of Theorem \ref{dingli3}. Finally, in Section \ref{Sec9}, we prove
Theorem \ref{dingli2} by using a similar idea as in the proof of Theorem \ref{dingli3}.

%%%%%%%%%%%%%%%%%%%%%

\section{Preliminaries} \label{sec:preliminaries}
\subsection{Reduction to a hyperbolic system}
Let us introduce a symmetric hyperbolic system to 3D compressible Euler equations \eqref{fc0}.
\begin{Lemma}\cite{LQ} \label{sh}
Let $(\bv, {\rho},h)$ be a solution of \eqref{fc0}. Then it also satisfies the following hyperbolic system
\begin{equation}\label{sq}
  {A}^0 (\bU) \partial_t \bU +\sum_{i=1}^3 {A}^i(\bU) \partial_{x_i}\bU=0,
\end{equation}
where $\bU=(v^1,v^2,v^3,p(\rho),h)^\mathrm{T}$ and
\begin{equation*}
A_0=
\left(
\begin{array}{ccccc}
\bar{\rho}\mathrm{e}^\rho & 0 & 0 & 0 & 0 \\
0 & \bar{\rho}\mathrm{e}^\rho & 0 & 0 & 0 \\
0 & 0 & \bar{\rho}\mathrm{e}^\rho & 0 & 0
\\ 0 & 0 & 0 & \bar{\rho}\mathrm{e}^\rho & 0
\\
0 & 0 & 0 & 0 & \bar{\rho}\mathrm{e}^\rho
\end{array}
\right ), \quad
A_1=\left(
\begin{array}{ccccc}
\bar{\rho}\mathrm{e}^\rho v_1 & 0 & 0 & 1 & 0 \\
0 &  \bar{\rho}\mathrm{e}^\rho v_1 & 0 &0 & 0 \\
0 & 0 &  \bar{\rho}\mathrm{e}^\rho v_1 &0 & 0  \\
1 & 0 & 0 & \bar{\rho}^{-1}\mathrm{e}^{-\rho}c^{-2}_s v_1 & 0
\\
0 & 0 & 0 & 0 & v_1
\end{array}
\right ),
\end{equation*}
\begin{equation*}
A_2=\left(
\begin{array}{ccccc}
\bar{\rho}\mathrm{e}^\rho v_2 & 0 & 0 & 0 & 0 \\
0 &  \bar{\rho}\mathrm{e}^\rho v_2 & 0 &1 & 0 \\
0 & 0 &  \bar{\rho}\mathrm{e}^\rho v_2 &0 & 0  \\
0 & 1 & 0 & \bar{\rho}^{-1}\mathrm{e}^{-\rho}c^{-2}_s v_2 & 0
\\
0 & 0 & 0 & 0 & v_2
\end{array}
\right ), \quad
A_3=\left(
\begin{array}{ccccc}
\bar{\rho}\mathrm{e}^\rho v_3 & 0 & 0 & 0 & 0 \\
0 &  \bar{\rho}\mathrm{e}^\rho v_3 & 0 &0 & 0 \\
0 & 0 &  \bar{\rho}\mathrm{e}^\rho v_3 &1 & 0  \\
0 & 0 & 1 & \bar{\rho}^{-1}\mathrm{e}^{-\rho}c^{-2}_s v_3 & 0
\\
0 & 0 & 0 & 0 & v_3
\end{array}
\right ).
\end{equation*}
%For simplicity, we set $B_i(U)=A_0^{-1}(U)A_i(U), i=1, 2, 3$.
\end{Lemma}
%We are ready to introduce a wave-transport reduction.

\begin{Lemma}[The wave equation for $\rho+\frac{1}{\gamma}h$]\label{crh}
Let $(\bv,\rho,h)$ be a solution of \eqref{fc0}. Let ${\rho}$ and $\bw$ are defined in \eqref{pw1} and \eqref{pw11}. Then $\rho+\frac{1}{\gamma}h$ satisfies
\begin{equation}\label{wav2}
  \square_g (\rho+\frac{1}{\gamma}h)= D+ \frac{1}{\gamma}E.
\end{equation}
Here $D$ and $E$ are defined in \eqref{DDi}.
\end{Lemma}
\begin{proof}
In \eqref{fc1}, we note
\begin{equation*}
\begin{cases}
  & \square_g {\rho}=-\frac{1}{\gamma}c_s^2\Delta h+ D,
\\
& \square_g h =c_s^2\Delta h + E.
\end{cases}
\end{equation*}
Therefore we have
\begin{equation*}
  \square_g (\rho+\frac{1}{\gamma}h)= D+ \frac{1}{\gamma}E.
\end{equation*}
\end{proof}
\begin{Lemma}[a wave equation for $\bv_{-}$]\label{wte1}
Let $(\bv,\rho,h)$ be a solution of \eqref{fc0}. Let $\bv_{+}$ and $\bv_{-}$ be described in \eqref{dvc}-\eqref{etad}. Then $\bv_{+}$ satisfies
\begin{equation}\label{fc}
\begin{split}
&\square_g v^i_{+}=\mathbf{T}\mathbf{T} v_{-}^i+Q^i. %\quad \mathcal{R}=(\mathcal{R}^1, \mathcal{R}^2, \mathcal{R}^3)^{\mathrm{T}},
\end{split}
\end{equation}
\end{Lemma}
\begin{proof}
By Lemma \ref{wte}, we find that
\begin{equation*}
  \square_g v^i=-\mathrm{e}^{{\rho}}c_s^2 \mathrm{curl} \bw^i+Q^i.
\end{equation*}
Substituting \eqref{dvc} and \eqref{etad} in the above equation, we have
\begin{equation*}
  \begin{split}
  \square_g v_{+}^i=&-\square_g v_{-}^i-\mathrm{e}^{{\rho}}c_s^2 \mathrm{curl} \bw^i+Q^i
  \\
  =& \mathbf{T}\mathbf{T}v_{-}^i-c_s^2 \Delta v_{-}^i-\mathrm{e}^{{\rho}}c_s^2 \mathrm{curl} \bw^i+Q^i
  \\
  =& \mathbf{T}\mathbf{T}v_{-}^i+c_s^2 \left( - \Delta v_{-}^i- \mathrm{e}^{{\rho}}\mathrm{curl} \bw^i \right) +Q^i
  \\
  =& \mathbf{T}\mathbf{T}v_{-}^i+Q^i.
  \end{split}
\end{equation*}
\end{proof}
\begin{Lemma}[Transport equations for $\bw$]\label{PW}
Let $(\bv,\rho,h)$ be a solution of \eqref{fc0}. Let ${\rho}$ and $\bw$ are defined in \eqref{pw1} and \eqref{pw11}. Then $\bw=(w^1,w^2,w^3)^{\mathrm{T}}$ satisfies
\begin{equation}\label{W0}
\mathbf{T} w^i =  (\bw \cdot \nabla )v^i+ \bar{\rho}^{\gamma-1} \mathrm{e}^{h+(\gamma-2)\rho}\epsilon^{iab}\partial_a \rho \partial_bh.
\end{equation}
Furthermore, we have
\begin{equation}\label{Wd1}
	\mathrm{div} \bw= -(\bw \cdot \nabla) {\rho},
\end{equation}
and
\begin{equation}\label{W1}
\begin{split}
\mathbf{T}  (\mathrm{curl} \bw^i) = &(\mathrm{curl}\bw \cdot \nabla)v^i-  \mathrm{curl}\bw^i \mathrm{div}\bv-2\epsilon^{imn}\partial_m v^j \partial_n w_j
+  \bar{\rho}^{\gamma-1}  \partial^l ( \mathrm{e}^{h+(\gamma-2)\rho}) \partial_l \rho \partial^i h
\\
&- \bar{\rho}^{\gamma-2}  \partial^l ( \mathrm{e}^{h+(\gamma-2)\rho}) \partial^i \rho \partial_l h
+\bar{\rho}^{\gamma-1} \mathrm{e}^{h+(\gamma-2)\rho}    \Delta \rho \partial^i h + \bar{\rho}^{\gamma-1} \mathrm{e}^{h+(\gamma-2)\rho}  \partial^m \rho \partial_m\partial^i h
 \\
 &- \bar{\rho}^{\gamma-1} \mathrm{e}^{h+(\gamma-2)\rho}    \partial_m \partial^i \rho  \partial^m h - \bar{\rho}^{\gamma-1} \mathrm{e}^{h+(\gamma-2)\rho}  \partial^i \rho \Delta h.
\end{split}
\end{equation}
and
\begin{equation}\label{W2}
\begin{split}
& \mathbf{T} ( \mathrm{curl} \mathrm{curl} \bw^i+F^i  )
=   \partial^i \big( 2 \partial_n v_a \partial^n w^a \big) + K^i,
\end{split}
\end{equation}
where
\begin{equation}\label{Fb}
  F^i= -\epsilon^{ijk} \partial_j \rho \cdot \mathrm{curl}\bw_k- 2\partial^a {\rho} \partial^i w_a+2 \mathrm{e}^{-\rho} \epsilon^{ijk} \partial_j v^m  \partial_m\partial_k h-\mathrm{e}^{-\rho} \epsilon^{ijk} \partial_k h \Delta v_j,
\end{equation}
and
\begin{equation}\label{rF}
\begin{split}
K^i=\sum^5_{a=1}  K^i_a,
\end{split}
\end{equation}
and
\begin{equation}\label{R1ib}
\begin{split}
 K^i_1=
&-  \bar{\rho}^{\gamma-1} \epsilon^{ijk}\partial_j \rho  \partial^l ( \mathrm{e}^{h+(\gamma-2)\rho}) \partial_l \rho \partial_k h
-2\epsilon_{kmn}\epsilon^{ijk}\partial^m v^a \partial_{j}(\partial^n w_a)
 \\
 & + 2\epsilon_{kmn}\epsilon^{ijk}\partial_m v^a \partial_{n}w_a \partial_j {\rho}
 +\bar{\rho}^{\gamma-1} \epsilon^{ijk}   \partial_l \rho \partial_k h \partial_j \partial^l ( \mathrm{e}^{h+(\gamma-2)\rho})
  \\
  & + \bar{\rho}^{\gamma-1} \epsilon^{ijk}   \partial^l ( \mathrm{e}^{h+(\gamma-2)\rho})\partial_k h \partial_j \partial_l \rho
   + \bar{\rho}^{\gamma-1} \epsilon^{ijk}  \partial^l ( \mathrm{e}^{h+(\gamma-2)\rho})  \partial_l \rho \partial_j\partial_k h
 \\
 & + \bar{\rho}^{\gamma-1}\epsilon^{ijk}  \partial_j \rho  \partial^l ( \mathrm{e}^{h+(\gamma-2)\rho}) \partial_k \rho \partial_l h
   - \bar{\rho}^{\gamma-1}\epsilon^{ijk}   \partial_k \rho \partial_l h \partial_j  \partial^l ( \mathrm{e}^{h+(\gamma-2)\rho})
 \\
 & - \bar{\rho}^{\gamma-1}\epsilon^{ijk}    \partial^l ( \mathrm{e}^{h+(\gamma-2)\rho})\partial_l h \partial_j \partial_k \rho
  - \bar{\rho}^{\gamma-1}\epsilon^{ijk}   \partial^l ( \mathrm{e}^{h+(\gamma-2)\rho})  \partial_k \rho \partial_j \partial_l h
\\
& +\bar{\rho}^{\gamma-1}\epsilon^{ijk} \mathrm{e}^{\rho}\partial_j( \mathrm{e}^{h+(\gamma-3)\rho})    \Delta \rho \partial_k h
+ \bar{\rho}^{\gamma-1} \epsilon^{ijk} \mathrm{e}^{\rho} \partial_j(\mathrm{e}^{h+(\gamma-3)\rho})  \partial^m \rho \partial_m\partial_k h
 \\
 & - \bar{\rho}^{\gamma-1} \epsilon^{ijk} \mathrm{e}^{\rho} \partial_j(\mathrm{e}^{h+(\gamma-3)\rho}) \partial^m h    \partial_m \partial_k \rho
- \bar{\rho}^{\gamma-1} \epsilon^{ijk} \mathrm{e}^{\rho} \partial_j(\mathrm{e}^{h+(\gamma-3)\rho})  \partial_k \rho \Delta h
\\
&- \bar{\rho}^{\gamma-1} \epsilon^{ijk} \mathrm{e}^{h+(\gamma-2)\rho}  \partial_k \rho \partial_j\Delta h,
\end{split}
\end{equation}
\begin{equation}\label{R2ib}
\begin{split}
  K^i_2=& -\frac{2}{\gamma} \bar{\rho}^{\gamma-1} \epsilon^{ijk} \mathrm{e}^{h+(\gamma-2)\rho}   \partial_j \partial^m h  \partial_m\partial_k h 
  \\
  & + \frac{2}{\gamma}\bar{\rho}^{\gamma-1} \epsilon^{ijk} \mathrm{e}^{h+(\gamma-2)\rho}  \partial_j (c^{-2}_s) (c^2_s \partial^m \rho-\frac{1}{\gamma}c^2_s \partial^m h) \partial_m\partial_k h
  \\
  &
  -\frac{2}{\gamma} \mathrm{e}^{-\rho} \epsilon^{ijk} \partial_j v^n  \partial_n v^m  \partial_m \partial_k h
   +\frac{4}{\gamma} \mathrm{e}^{-\rho} \partial_a v^a \epsilon^{ijk} \partial_n v^m  \partial_m \partial_k h
   \\
   &- \frac{2}{\gamma} \mathrm{e}^{-\rho} \epsilon^{ijk} \partial_j v^m ( \partial_m v^a \partial_a \partial_k h - \partial_a h \partial_m \partial_k v^a-\partial_k v^a \partial_a \partial_m h ), \ \ \ \ \ \ \ \ \ \ \ 
\end{split}
\end{equation}
\begin{equation}\label{R3ib}
\begin{split}
  K^i_3=& \frac{1}{\gamma} \mathrm{e}^{-\rho}\epsilon^{ijk}  \partial_a v^a \partial_k h \Delta v_j+ \frac{1}{\gamma}\mathrm{e}^{-\rho}\epsilon^{ijk}  \partial_k v^a  \partial_a h  \Delta v_j \ \ \ \ \ \ \ \ \ \ \ \ \ \ \ \ \ \ \ \ \ \ \ \ \ \ \ \ \ \ \ \ \ 
  \\
  &+ \frac{1}{\gamma} \mathrm{e}^{-\rho}\epsilon^{ijk} \partial_k h ( \partial^m v^a \partial_m \partial_a v_j + \Delta v^a \partial_a v_j )
   \\
   & +\frac{1}{\gamma} \bar{\rho}^{\gamma-1}\epsilon^{ijk} \mathrm{e}^{h+(\gamma-2)\rho} \partial_k h  \partial^m(c^{-2}_s) \partial_m (c^2_s\partial_j \rho- \frac{1}{\gamma}c^2_s \partial_j h)
  \\
  & +\frac{1}{\gamma} \bar{\rho}^{\gamma-1}\epsilon^{ijk} \mathrm{e}^{h+(\gamma-2)\rho} \partial_k h   \Delta(-c^{-2}_s)  (c^2_s\partial_j \rho- \frac{1}{\gamma}c^2_s \partial_j h)
  \\
  & -\frac{1}{\gamma}\bar{\rho}^{\gamma-1}\epsilon^{ijk} \mathrm{e}^{h+(\gamma-2)\rho} \partial_k h   \partial_j \Delta h, 
\end{split}
\end{equation}
\begin{equation}\label{R4ib}
  \begin{split}
  K^i_4= 
  &- 2 \partial^a {\rho} \big( \partial^i w^m \partial_m v_a + w^m \partial^i(\partial_{m}v_a) - \partial^i v^k \partial_k w_a \big) \ \ \ \ \ \ \ \ \ \ \ \ \ \ \ \ \ \ \ \ \ \ \ \ 
  \\
  &-2  \partial^a v^k \partial_k {\rho}  \partial^i w_a+2 \mathrm{e}^{{\rho}} \mathrm{curl}\bw^a  \partial^i w_a
  + 2 \epsilon^{ajk} \mathrm{e}^{{\rho}} w_k  \partial_j {\rho}  \partial^iw_a
  \\
  & + 2\mathrm{e}^{-{\rho}}\mathrm{div}\bv \partial^a {\rho}  \partial^i w_a- 2\epsilon_{amn} \bar{\rho}^{\gamma-1}  \mathrm{e}^{h+(\gamma-3)\rho} \partial^a {\rho} \partial^m h \partial^i \partial^n \rho
  \\
  &- 2\epsilon_{amn} \bar{\rho}^{\gamma-1}  \partial^a {\rho}  \partial^i(\mathrm{e}^{h+(\gamma-2)\rho}) \partial^n \rho \partial^m h+ 2\mathrm{e}^{\rho}w^i  \mathrm{curl}\bw^m \partial_m \rho
 \\
 &- 2\epsilon_{amn} \bar{\rho}^{\gamma-1}  \mathrm{e}^{h+(\gamma-2)\rho} \partial^a {\rho}   \partial^n \rho \partial^i \partial^m h+2\mathrm{e}^{\rho}\mathrm{curl}\bw^m \partial_m w^i  , 
  \end{split}
\end{equation}
\begin{equation}\label{R5ib}
  K^i_5=2\partial^i {\rho} \partial_{j}v^a \partial^j w_a
  -2  \partial_{j} v^a \partial^j\partial^i  w_a-\epsilon^{ijk} \mathrm{e}^{\rho} \partial_j v^m \partial_m (\mathrm{e}^{-{\rho}} \mathrm{curl}\bw_k),  \ \ \ \ \
\end{equation}
\begin{equation}\label{R6ib}
\begin{split}
  K^i_6=&2 \partial^i \rho \partial_n v_a \partial^n w^a- \mathrm{e}^{\rho}\mathrm{div}\bv \cdot \epsilon^{ijk}\partial_j (\mathrm{e}^{-\rho}  \mathrm{curl}\bw_k)+2\mathrm{div}\bv\cdot\partial^a {\rho} \partial^i w_a
  \\
  &-2 \mathrm{e}^{-\rho} \epsilon^{ijk} \mathrm{div}\bv \cdot \partial_j v^m  \partial_m\partial_k h+ \mathrm{e}^{-\rho} \epsilon^{ijk} \mathrm{div}\bv \cdot \partial_k h \Delta v_j.
\end{split}
\end{equation}
\end{Lemma}
\begin{proof}
From \eqref{py1}, we can calculate
\begin{equation}\label{ssf}
  c^2_s=\gamma \bar{\rho}^{\gamma-1}\mathrm{e}^{h+(\gamma-1)\rho}.
\end{equation}
Applying the curl-operator to the first equation in \eqref{fc0}, we find that
\begin{equation*}
\begin{split}
  \mathbf{T}( \epsilon^{ijk} \partial_j v_k)= &- \epsilon^{ijk} \partial_j v^m \partial_m v_k+ \bar{\rho}^{\gamma-1}  \mathrm{e}^{h+(\gamma-1)\rho}\epsilon^{ijk} \partial_k \rho \partial_j h
  \\
  = &- \epsilon^{ijk} \partial_j v^m (\partial_m v_k-\partial_k v_m) -\epsilon^{ijk} \partial_j v^m \partial_k v_m + \bar{\rho}^{\gamma-1}  \mathrm{e}^{h+(\gamma-1)\rho}\epsilon^{ijk} \partial_k \rho \partial_j h
  \\
  = &- \mathrm{e}^\rho \epsilon^{ijk}\epsilon_{mk}^{\ \ l} w_l \partial_j v^m + \bar{\rho}^{\gamma-1}  \mathrm{e}^{h+(\gamma-1)\rho}\epsilon^{ijk} \partial_k \rho \partial_j h
  \\
   = & \mathrm{e}^\rho (\bw \cdot \nabla) v^i- \mathrm{e}^\rho w^i \mathrm{div}\bv + \bar{\rho}^{\gamma-1}  \mathrm{e}^{h+(\gamma-1)\rho}\epsilon^{ijk} \partial_k \rho \partial_j h.
\end{split}
\end{equation*}
As a result, we get
\begin{equation*}
  \mathbf{T}w^i= (\bw \cdot \nabla) v^i+ \bar{\rho}^{\gamma-1}  \mathrm{e}^{h+(\gamma-2)\rho}\epsilon^{ijk} \partial_k \rho \partial_j h.
\end{equation*}
Then Equation \eqref{W0} is derived. On the other hand, it's direct for us to calculate
\begin{equation*}
	\mathrm{div}\bw= \mathrm{div}(\mathrm{e}^{-\rho}\mathrm{curl}\bv^i )=-(\bw \cdot \nabla)\rho.
\end{equation*}
So we have proved \eqref{Wd1}. Applying the curl-operator to \eqref{W0}, we have
\begin{equation}\label{W00}
  \begin{split}
  \mathbf{T}( \mathrm{curl}\bw^i )=& - \epsilon^{ijk}\partial_j v^m \partial_m w_k+ \epsilon^{ijk}\partial_m v_k \partial_j w^m+ \epsilon^{ijk}  w^m \partial_m \partial_j v_k
  \\
  & +  \bar{\rho}^{\gamma-1}\epsilon^{ijk}\epsilon^{kml}  \partial_j ( \mathrm{e}^{h+(\gamma-2)\rho} \partial_l \rho \partial_m h)
  \\
  =& - \epsilon^{ijk}\partial_j v^m \partial_m w_k+ \epsilon^{ijk}\partial_m v_k \partial_j w^m+  w^m \partial_m ( \mathrm{e}^\rho w^i)
  \\
  & +  \bar{\rho}^{\gamma-1}\epsilon^{ijk}\epsilon^{kml}  \partial_j ( \mathrm{e}^{h+(\gamma-2)\rho} \partial_l \rho \partial_m h)
  \end{split}
\end{equation}
Calculate
\begin{equation}\label{W01}
  \begin{split}
  - \epsilon^{ijk}\partial_j v^m \partial_m w_k=& - \epsilon^{ijk} \partial_j v^m ( \partial_m w_k - \partial_k w_m)- \epsilon^{ijk}\partial_j v^m\partial_k w_m
  \\
  =& (\delta^{i}_m \delta^j_l - \delta^{i}_l \delta^j_m ) \partial_j v^m \mathrm{curl}\bw^l- \epsilon^{ijk}\partial_j v^m\partial_k w_m
  \\
  =& \partial_j v^i \mathrm{curl}\bw^j-  \mathrm{curl}\bw^i \mathrm{div}\bv- \epsilon^{ijk}\partial_j v^m\partial_k w_m,
  \end{split}
\end{equation}
and
\begin{equation}\label{W02}
\begin{split}
  \epsilon^{ijk}\partial_m v_k \partial_j w^m =&  \epsilon^{ijk}  \partial_j w^m (\partial_m v_k - \partial_k v_m)+ \epsilon^{ijk}  \partial_j w^m \partial_k v_m
  \\
  =& -\mathrm{e}^\rho w^j \partial_j w^i + \mathrm{e}^\rho w^i \partial_j w^j + \epsilon^{ijk}  \partial_j w^m \partial_k v_m
  \\
  =&-\mathrm{e}^\rho w^j \partial_j w^i - \mathrm{e}^\rho w^i  w^j \partial_j \rho + \epsilon^{ijk}  \partial_j w^m \partial_k v_m
  \\
  =&-\mathrm{e}^\rho w^j \partial_j w^i - \mathrm{e}^\rho w^i  w^j \partial_j \rho - \epsilon^{ijk} \partial_j v_m \partial_k w^m .
\end{split}
\end{equation}
Inserting \eqref{W01} and \eqref{W02} to \eqref{W00}, we get
\begin{equation}\label{lpg}
\begin{split}
\mathbf{T}  (\mathrm{curl} \bw^i) = &(\mathrm{curl}\bw \cdot \nabla)v^i-  \mathrm{curl}\bw^i \mathrm{div}\bv-2\epsilon^{imn}\partial_m v^j \partial_n w_j
+  \bar{\rho}^{\gamma-1}  \partial^l ( \mathrm{e}^{h+(\gamma-2)\rho}) \partial_l \rho \partial^i h
\\
&- \bar{\rho}^{\gamma-2}  \partial^l ( \mathrm{e}^{h+(\gamma-2)\rho}) \partial^i \rho \partial_l h
+\bar{\rho}^{\gamma-1} \mathrm{e}^{h+(\gamma-2)\rho}    \Delta \rho \partial^i h + \bar{\rho}^{\gamma-1} \mathrm{e}^{h+(\gamma-2)\rho}  \partial^m \rho \partial_m\partial^i h
 \\
 &- \bar{\rho}^{\gamma-1} \mathrm{e}^{h+(\gamma-2)\rho}    \partial_m \partial^i \rho  \partial^m h - \bar{\rho}^{\gamma-1} \mathrm{e}^{h+(\gamma-2)\rho}  \partial^i \rho \Delta h.
\end{split}
\end{equation}
So we have proved \eqref{W1}. Multiplying with $\mathrm{e^{-\rho}}$, we can update \eqref{lpg} by
\begin{equation}\label{Lpg}
\begin{split}
\mathbf{T}  (\mathrm{e}^{-\rho} \mathrm{curl}\bw^i) = &(\mathrm{e}^{-\rho} \mathrm{curl}\bw \cdot \nabla)v^i-2\epsilon^{imn}\mathrm{e}^{-{\rho}}\partial_m v^j \partial_n w_j
+  \bar{\rho}^{\gamma-1} \mathrm{e}^{-\rho}   \partial^l ( \mathrm{e}^{h+(\gamma-2)\rho}) \partial_l \rho \partial^i h
\\
&- \bar{\rho}^{\gamma-2} \mathrm{e}^{-\rho}   \partial^l ( \mathrm{e}^{h+(\gamma-2)\rho}) \partial^i \rho \partial_l h
+\bar{\rho}^{\gamma-1} \mathrm{e}^{h+(\gamma-3)\rho}    \Delta \rho \partial^i h + \bar{\rho}^{\gamma-1} \mathrm{e}^{h+(\gamma-3)\rho}  \partial^m \rho \partial_m\partial^i h
 \\
 &- \bar{\rho}^{\gamma-1} \mathrm{e}^{h+(\gamma-3)\rho}    \partial_m \partial^i \rho  \partial^m h - \bar{\rho}^{\gamma-1} \mathrm{e}^{h+(\gamma-3)\rho}  \partial^i \rho \Delta h.
\end{split}
\end{equation}
Applying the operator $\mathrm{curl}$ to \eqref{Lpg} yields
\begin{equation}\label{ct}
\begin{split}
  \mathrm{curl} \mathbf{T}  (\mathrm{e}^{-\rho} \mathrm{curl}\bw^i)
  %=&
  %\epsilon^{ijk} \partial_j \mathbf{T}  \bW_k
  =& \epsilon^{ijk} \partial_j\left( -2\epsilon_{kmn}\mathrm{e}^{-{\rho}}\partial^m v^a \partial^n w_a \right) +  \epsilon^{ijk}\partial_j ( \mathrm{e}^{-\rho} \mathrm{curl}\bw^m \partial_m v_k )
  \\
  &-  \bar{\rho}^{\gamma-1} \epsilon^{ijk}\mathrm{e}^{-\rho}\partial_j \rho  \partial^l ( \mathrm{e}^{h+(\gamma-2)\rho}) \partial_l \rho \partial_k h
   +\bar{\rho}^{\gamma-1} \epsilon^{ijk}\mathrm{e}^{-\rho}   \partial_l \rho \partial_k h \partial_j \partial^l ( \mathrm{e}^{h+(\gamma-2)\rho})
  \\
  & + \bar{\rho}^{\gamma-1} \epsilon^{ijk}\mathrm{e}^{-\rho}   \partial^l ( \mathrm{e}^{h+(\gamma-2)\rho})\partial_k h \partial_j \partial_l \rho
   + \bar{\rho}^{\gamma-1} \epsilon^{ijk}\mathrm{e}^{-\rho}   \partial^l ( \mathrm{e}^{h+(\gamma-2)\rho})  \partial_l \rho \partial_j\partial_k h
  \\
  &+ \bar{\rho}^{\gamma-1}\epsilon^{ijk} \mathrm{e}^{-\rho} \partial_j \rho  \partial^l ( \mathrm{e}^{h+(\gamma-2)\rho}) \partial_k \rho \partial_l h
   - \bar{\rho}^{\gamma-1}\epsilon^{ijk} \mathrm{e}^{-\rho}  \partial_k \rho \partial_l h \partial_j  \partial^l ( \mathrm{e}^{h+(\gamma-2)\rho})
  \\
  &- \bar{\rho}^{\gamma-1}\epsilon^{ijk} \mathrm{e}^{-\rho}   \partial^l ( \mathrm{e}^{h+(\gamma-2)\rho})\partial_l h \partial_j \partial_k \rho
  - \bar{\rho}^{\gamma-1}\epsilon^{ijk} \mathrm{e}^{-\rho}   \partial^l ( \mathrm{e}^{h+(\gamma-2)\rho})  \partial_k \rho \partial_j \partial_l h
\\
&+\bar{\rho}^{\gamma-1}\epsilon^{ijk} \partial_j( \mathrm{e}^{h+(\gamma-3)\rho})    \Delta \rho \partial_k h
+ \bar{\rho}^{\gamma-1} \epsilon^{ijk} \partial_j(\mathrm{e}^{h+(\gamma-3)\rho})  \partial^m \rho \partial_m\partial_k h
\\
& - \bar{\rho}^{\gamma-1} \epsilon^{ijk} \partial_j(\mathrm{e}^{h+(\gamma-3)\rho}) \partial^m h    \partial_m \partial_k \rho
- \bar{\rho}^{\gamma-1} \epsilon^{ijk} \partial_j(\mathrm{e}^{h+(\gamma-3)\rho})  \partial_k \rho \Delta h
\\
&- \bar{\rho}^{\gamma-1} \epsilon^{ijk} \mathrm{e}^{h+(\gamma-3)\rho}  \partial_k \rho \partial_j\Delta h
 +\bar{\rho}^{\gamma-1}\epsilon^{ijk} \mathrm{e}^{h+(\gamma-3)\rho} \partial_k h   \partial_j\Delta \rho
\\
&+ 2\bar{\rho}^{\gamma-1} \epsilon^{ijk} \mathrm{e}^{h+(\gamma-3)\rho}  \partial_j\partial^m \rho \partial_m\partial_k h .
\end{split}
\end{equation}
Calculate
\begin{equation}\label{W10}
  \begin{split}
  \epsilon^{ijk}\partial_j ( \mathrm{e}^{-\rho} \mathrm{curl}\bw^m \partial_m v_k ) =&  \epsilon^{ijk}\partial_j  (\mathrm{e}^{-\rho} \mathrm{curl}\bw^m) \partial_m v_k +  \epsilon^{ijk}  (\mathrm{e}^{-\rho} \mathrm{curl}\bw^m) \partial_j \partial_m v_k
  \\
  =& \epsilon^{ijk}\partial_j  (\mathrm{e}^{-\rho} \mathrm{curl}\bw^m) \partial_m v_k + w^i  \mathrm{curl}\bw^m \partial_m \rho+\mathrm{curl}\bw^m \partial_m w^i
  \\
  =& \epsilon^{ijk}\mathrm{e}^{-\rho} \partial_m v_k \partial_j  (\mathrm{curl} \bw^m) -\epsilon^{ijk}  \mathrm{e}^{-\rho}  \partial_j \rho \partial_m v_k  \mathrm{curl} \bw^m
  \\
  & +\epsilon^{ijk} \mathrm{e}^{-\rho}  \partial^2_{jm}v_k \mathrm{curl} \bw^m + w^i  \mathrm{curl}\bw^m \partial_m \rho+\mathrm{curl}\bw^m \partial_m w^i
  \\
  =& \epsilon^{ijk}\mathrm{e}^{-\rho} \partial_m v_k \partial_j  (\mathrm{curl} \bw^m) -\epsilon^{ijk}  \mathrm{e}^{-\rho}  \partial_j \rho \partial_m v_k  \mathrm{curl} \bw^m
  \\
  &  + 2w^i  \mathrm{curl}\bw^m \partial_m \rho+2\mathrm{curl}\bw^m \partial_m w^i.
  \end{split}
\end{equation}
Substituting \eqref{W10} in \eqref{ct}, we find that
\begin{equation*}
\begin{split}
 \mathrm{curl} \mathbf{T}  (\mathrm{e}^{-\rho} \mathrm{curl}\bw^i)=&   -2\mathrm{e}^{-{\rho}} \epsilon^{ijk}\epsilon_{kmn} \partial_j(\partial^m v^a) \partial^n w_a
 -2\mathrm{e}^{-{\rho}}\epsilon_{kmn}\epsilon^{ijk}\partial^m v^a \partial_{j}(\partial^n w_a)
  \\
 & +\epsilon^{ijk}\mathrm{e}^{-\rho} \partial_m v_k \partial_j  (\mathrm{curl} \bw^m) -\epsilon^{ijk}  \mathrm{e}^{-\rho}  \partial_j \rho \partial_m v_k  \mathrm{curl} \bw^m
 \\
 &-  \bar{\rho}^{\gamma-1} \epsilon^{ijk}\mathrm{e}^{-\rho}\partial_j \rho  \partial^l ( \mathrm{e}^{h+(\gamma-2)\rho}) \partial_l \rho \partial_k h
    +\bar{\rho}^{\gamma-1} \epsilon^{ijk}\mathrm{e}^{-\rho}   \partial_l \rho \partial_k h \partial_j \partial^l ( \mathrm{e}^{h+(\gamma-2)\rho})
  \\
  & + \bar{\rho}^{\gamma-1} \epsilon^{ijk}\mathrm{e}^{-\rho}   \partial^l ( \mathrm{e}^{h+(\gamma-2)\rho})\partial_k h \partial_j \partial_l \rho
   + \bar{\rho}^{\gamma-1} \epsilon^{ijk}\mathrm{e}^{-\rho}   \partial^l ( \mathrm{e}^{h+(\gamma-2)\rho})  \partial_l \rho \partial_j\partial_k h
 \\
 & + \bar{\rho}^{\gamma-1}\epsilon^{ijk} \mathrm{e}^{-\rho} \partial_j \rho  \partial^l ( \mathrm{e}^{h+(\gamma-2)\rho}) \partial_k \rho \partial_l h
   - \bar{\rho}^{\gamma-1}\epsilon^{ijk} \mathrm{e}^{-\rho}  \partial_k \rho \partial_l h \partial_j  \partial^l ( \mathrm{e}^{h+(\gamma-2)\rho})
 \\
 & - \bar{\rho}^{\gamma-1}\epsilon^{ijk} \mathrm{e}^{-\rho}   \partial^l ( \mathrm{e}^{h+(\gamma-2)\rho})\partial_l h \partial_j \partial_k \rho
  - \bar{\rho}^{\gamma-1}\epsilon^{ijk} \mathrm{e}^{-\rho}   \partial^l ( \mathrm{e}^{h+(\gamma-2)\rho})  \partial_k \rho \partial_j \partial_l h
\\
& +\bar{\rho}^{\gamma-1}\epsilon^{ijk} \partial_j( \mathrm{e}^{h+(\gamma-3)\rho})    \Delta \rho \partial_k h
+ \bar{\rho}^{\gamma-1} \epsilon^{ijk} \partial_j(\mathrm{e}^{h+(\gamma-3)\rho})  \partial^m \rho \partial_m\partial_k h
 \\
 & - \bar{\rho}^{\gamma-1} \epsilon^{ijk} \partial_j(\mathrm{e}^{h+(\gamma-3)\rho}) \partial^m h    \partial_m \partial_k \rho
- \bar{\rho}^{\gamma-1} \epsilon^{ijk} \partial_j(\mathrm{e}^{h+(\gamma-3)\rho})  \partial_k \rho \Delta h
\\
&- \bar{\rho}^{\gamma-1} \epsilon^{ijk} \mathrm{e}^{h+(\gamma-3)\rho}  \partial_k \rho \partial_j\Delta h
+\bar{\rho}^{\gamma-1}\epsilon^{ijk} \mathrm{e}^{h+(\gamma-3)\rho} \partial_k h   \partial_j\Delta \rho
\\
& + 2\bar{\rho}^{\gamma-1} \epsilon^{ijk} \mathrm{e}^{h+(\gamma-3)\rho}  \partial_j\partial^m \rho \partial_m\partial_k h
+ 2\mathrm{e}^{-{\rho}}\epsilon_{kmn}\epsilon^{ijk}\partial_m v^a \partial_{n}w_a \partial_j {\rho}.
\end{split}
\end{equation*}
For simplicity, we rewrite
\begin{equation}\label{W11}
\begin{split}
 \mathrm{curl} \mathbf{T}  (\mathrm{e}^{-\rho} \mathrm{curl}\bw^i)=& -2\mathrm{e}^{-{\rho}} \epsilon^{ijk}\epsilon_{kmn} \partial_j(\partial^m v^a) \partial^n w_a
 +\bar{\rho}^{\gamma-1}\epsilon^{ijk} \mathrm{e}^{h+(\gamma-3)\rho} \partial_k h   \partial_j\Delta \rho
\\
& + 2\bar{\rho}^{\gamma-1} \epsilon^{ijk} \mathrm{e}^{h+(\gamma-3)\rho}  \partial_j\partial^m \rho \partial_m\partial_k h + Z^i_1 ,
\end{split}
\end{equation}
where
\begin{equation}\label{R1i}
\begin{split}
 Z^i_1= & 2w^i  \mathrm{curl}\bw^m \partial_m \rho+2\mathrm{curl}\bw^m \partial_m w^i - \bar{\rho}^{\gamma-1} \epsilon^{ijk} \mathrm{e}^{h+(\gamma-3)\rho}  \partial_k \rho \partial_j\Delta h
 \\
&-  \bar{\rho}^{\gamma-1} \epsilon^{ijk}\mathrm{e}^{-\rho}\partial_j \rho  \partial^l ( \mathrm{e}^{h+(\gamma-2)\rho}) \partial_l \rho \partial_k h
-2\mathrm{e}^{-{\rho}}\epsilon_{kmn}\epsilon^{ijk}\partial^m v^a \partial_{j}(\partial^n w_a)
 \\
 & + 2\mathrm{e}^{-{\rho}}\epsilon_{kmn}\epsilon^{ijk}\partial_m v^a \partial_{n}w_a \partial_j {\rho}
 +\bar{\rho}^{\gamma-1} \epsilon^{ijk}\mathrm{e}^{-\rho}   \partial_l \rho \partial_k h \partial_j \partial^l ( \mathrm{e}^{h+(\gamma-2)\rho})
  \\
  & + \bar{\rho}^{\gamma-1} \epsilon^{ijk}\mathrm{e}^{-\rho}   \partial^l ( \mathrm{e}^{h+(\gamma-2)\rho})\partial_k h \partial_j \partial_l \rho
   + \bar{\rho}^{\gamma-1} \epsilon^{ijk}\mathrm{e}^{-\rho}   \partial^l ( \mathrm{e}^{h+(\gamma-2)\rho})  \partial_l \rho \partial_j\partial_k h
 \\
 & + \bar{\rho}^{\gamma-1}\epsilon^{ijk} \mathrm{e}^{-\rho} \partial_j \rho  \partial^l ( \mathrm{e}^{h+(\gamma-2)\rho}) \partial_k \rho \partial_l h
   - \bar{\rho}^{\gamma-1}\epsilon^{ijk} \mathrm{e}^{-\rho}  \partial_k \rho \partial_l h \partial_j  \partial^l ( \mathrm{e}^{h+(\gamma-2)\rho})
 \\
 & - \bar{\rho}^{\gamma-1}\epsilon^{ijk} \mathrm{e}^{-\rho}   \partial^l ( \mathrm{e}^{h+(\gamma-2)\rho})\partial_l h \partial_j \partial_k \rho
  - \bar{\rho}^{\gamma-1}\epsilon^{ijk} \mathrm{e}^{-\rho}   \partial^l ( \mathrm{e}^{h+(\gamma-2)\rho})  \partial_k \rho \partial_j \partial_l h
\\
& +\bar{\rho}^{\gamma-1}\epsilon^{ijk} \partial_j( \mathrm{e}^{h+(\gamma-3)\rho})    \Delta \rho \partial_k h
+ \bar{\rho}^{\gamma-1} \epsilon^{ijk} \partial_j(\mathrm{e}^{h+(\gamma-3)\rho})  \partial^m \rho \partial_m\partial_k h
 \\
 & - \bar{\rho}^{\gamma-1} \epsilon^{ijk} \partial_j(\mathrm{e}^{h+(\gamma-3)\rho}) \partial^m h    \partial_m \partial_k \rho
- \bar{\rho}^{\gamma-1} \epsilon^{ijk} \partial_j(\mathrm{e}^{h+(\gamma-3)\rho})  \partial_k \rho \Delta h.
\end{split}
\end{equation}
can be viewed as a lower order term.
The third term of the right hand of \eqref{W11} causes some difficulty. We use $\partial^m \rho= -c^{-2}_s \mathbf{T}v^m-\frac{1}{\gamma} \partial^m h$ to calculate
\begin{equation}\label{W13}
\begin{split}
   2\bar{\rho}^{\gamma-1} \epsilon^{ijk} \mathrm{e}^{h+(\gamma-3)\rho}  \partial_j\partial^m \rho \partial_m\partial_k h
  =&  \mathbf{T}(-\frac{2}{\gamma} \mathrm{e}^{-2\rho} \epsilon^{ijk} \partial_j v^m  \partial_m\partial_k h )+Z^i_2,
\end{split}
\end{equation}
where
\begin{equation}\label{R2i}
\begin{split}
  Z^i_2=& -\frac{2}{\gamma} \bar{\rho}^{\gamma-1} \epsilon^{ijk} \mathrm{e}^{h+(\gamma-3)\rho}   \partial_j \partial^m h  \partial_m\partial_k h
  \\
  & + \frac{2}{\gamma}\bar{\rho}^{\gamma-1} \epsilon^{ijk} \mathrm{e}^{h+(\gamma-3)\rho}  \partial_j (c^{-2}_s) (c^2_s \partial^m \rho-\frac{1}{\gamma}c^2_s \partial^m h) \partial_m\partial_k h
  \\
  &
  -\frac{2}{\gamma} \mathrm{e}^{-2\rho} \epsilon^{ijk} \partial_j v^n  \partial_n v^m  \partial_m \partial_k h
   +\frac{4}{\gamma} \mathrm{e}^{-2\rho} \partial_a v^a \epsilon^{ijk} \partial_n v^m  \partial_m \partial_k h
   \\
   &- \frac{2}{\gamma} \mathrm{e}^{-2\rho} \epsilon^{ijk} \partial_j v^m ( \partial_m v^a \partial_a \partial_k h - \partial_a h \partial_m \partial_k v^a-\partial_k v^a \partial_a \partial_m h ).
\end{split}
\end{equation}
Similarly, we have
%In a similar way, we could derive that
\begin{equation}\label{W14}
\begin{split}
  &\bar{\rho}^{\gamma-1}\epsilon^{ijk} \mathrm{e}^{h+(\gamma-3)\rho} \partial_k h   \partial_j\Delta \rho
  \\
  =&-\bar{\rho}^{\gamma-1}\epsilon^{ijk} \mathrm{e}^{h+(\gamma-3)\rho} \partial_k h   \Delta(-c^{-2}_s \mathbf{T}v_j)-\frac{1}{\gamma}\bar{\rho}^{\gamma-1}\epsilon^{ijk} \mathrm{e}^{h+(\gamma-2)\rho} \partial_k h   \partial_j \Delta h
  \\
  =& \mathbf{T}(\frac{1}{\gamma} \mathrm{e}^{-2\rho} \epsilon^{ijk} \partial_k h \Delta v_j)+Z^i_3,
\end{split}
\end{equation}
where
\begin{equation}\label{R3i}
\begin{split}
  Z^i_3=& \frac{1}{\gamma} \mathrm{e}^{-2\rho}\epsilon^{ijk}  \partial_a v^a \partial_k h \Delta v_j+ \frac{1}{\gamma}\mathrm{e}^{-2\rho}\epsilon^{ijk}  \partial_k v^a  \partial_a h  \Delta v_j
  \\
  &+ \frac{1}{\gamma} \mathrm{e}^{-2\rho}\epsilon^{ijk} \partial_k h ( \partial^m v^a \partial_m \partial_a v_j + \Delta v^a \partial_a v_j )-\frac{1}{\gamma}\bar{\rho}^{\gamma-1}\epsilon^{ijk} \mathrm{e}^{h+(\gamma-3)\rho} \partial_k h   \partial_j \Delta h
   \\
   & +\frac{1}{\gamma} \bar{\rho}^{\gamma-1}\epsilon^{ijk} \mathrm{e}^{h+(\gamma-3)\rho} \partial_k h  \partial^m(c^{-2}_s) \partial_m (c^2_s\partial_j \rho- \frac{1}{\gamma}c^2_s \partial_j h)
  \\
  & +\frac{1}{\gamma} \bar{\rho}^{\gamma-1}\epsilon^{ijk} \mathrm{e}^{h+(\gamma-3)\rho} \partial_k h   \Delta(-c^{-2}_s)  (c^2_s\partial_j \rho- \frac{1}{\gamma}c^2_s \partial_j h).
\end{split}
\end{equation}
By using $\epsilon^{ijk} \epsilon_{kmn}=\delta^j_m \delta^i_n- \delta^i_m \delta^j_n$, we can obtain
\begin{equation}\label{W15}
\begin{split}
 -2 \epsilon^{ijk} \epsilon_{kmn} \mathrm{e}^{-{\rho}} \partial_{j} (\partial^m v^a) \partial^n w_a
= & -2\mathrm{e}^{-{\rho}} (\delta^j_m \delta^i_n- \delta^i_m \delta^j_n) \partial_{j}(\partial^m v^a) \partial_n w_a
  \\
  = & \Xi_1+\Xi_2,
\end{split}
\end{equation}
where
\begin{equation}\label{A12}
\begin{split}
  \Xi_1:= -2\mathrm{e}^{-{\rho}} \Delta v^a \partial^i w_a, \quad \Xi_2:= 2\mathrm{e}^{-{\rho}} \partial^i( \partial_j v^a) \partial^j w_a.
  \end{split}
\end{equation}
The Hodge decomposition yields
\begin{equation}\label{Hod}
  \begin{split}
  \Delta \bv = \partial \mathrm{div} \bv - \mathrm{curl}^2 \bv
  & =- \partial \mathbf{T} {\rho}+ \mathrm{curl} \big( \mathrm{e}^{{\rho}}\bw \big)
  \\
  & = -\mathbf{T} (\partial {\rho})+ \partial v^k \partial_k {\rho}++ \mathrm{curl} \big( \mathrm{e}^{{\rho}}\bw \big).
  \end{split}
\end{equation}
Substituting \eqref{Hod} in $\Xi_1$, we can rewrite $\Xi_1$ as
\begin{equation}\label{A1F}
\begin{split}
  \Xi_1 =&-2\mathrm{e}^{-{\rho}} \big[ -\mathbf{T} (\partial^a {\rho})+ \partial^a v^k \partial_k {\rho}-\mathrm{e}^{{\rho}} \epsilon^{ajk} \partial_j {\rho} w_k- \mathrm{e}^{{\rho}}\mathrm{curl}\bw^a \big] \partial^i w_a
  \\
   = &\mathbf{T} \big( 2\mathrm{e}^{-{\rho}}  \partial^a {\rho} \partial^i w_a  \big)- 2\mathrm{e}^{-{\rho}} \partial^a {\rho} \mathbf{T} ( \partial^i w_a)
   +  2\mathrm{e}^{-{\rho}}\mathrm{div}\bv \partial^a {\rho}  \partial^i w_a
   \\
   &-2\mathrm{e}^{-{\rho}}   \partial^a v^k \partial_k {\rho}  \partial^i w_a
   +2 \mathrm{curl}\bw^a  \partial^i w_a+2 \epsilon^{ajk} \partial_j {\rho} w_k  \partial^iw_a
  \\
  =& \mathbf{T} \big( 2\mathrm{e}^{-{\rho}} \partial^a {\rho}  \partial^i w_a  \big)- 2\mathrm{e}^{-{\rho}} \partial^a {\rho}  \mathbf{T} (\partial^i w_a)
  \\
  &-2\mathrm{e}^{-{\rho}}   \partial^a v^k \partial_k {\rho}  \partial^i w_a+2 \mathrm{curl}\bw^a  \partial^i w_a
  \\
  &
  + 2 \epsilon^{ajk} \partial_j {\rho} w_k  \partial^iw_a+ 2\mathrm{e}^{-{\rho}} \partial^a {\rho}  \partial^i w_a.
\end{split}
\end{equation}
In \eqref{A1F}, it remains for us to write $- 2\mathrm{e}^{-{\rho}} \partial^a {\rho}  \mathbf{T} (\partial^i w_a)$ in a suitable way. Note $\mathbf{T} w_a= w^m \partial_m v_a+\bar{\rho}^{\gamma-1}  \mathrm{e}^{h+(\gamma-2)\rho}\epsilon_{amn} \partial^n \rho \partial^m h $. Then we can calculate that
\begin{equation*}
\begin{split}
   \mathbf{T} (\partial^i w_a)  = & \partial^i (\mathbf{T} w_a)- [\partial^i,  \mathbf{T}]w_a
 \\
 =&  \partial^i w^m \partial_m v_a + w^m \partial^i(\partial_{m}v_a) - \partial^i v^k \partial_k w_a
 \\
 &+ \epsilon_{amn}\bar{\rho}^{\gamma-1}  \partial^i(\mathrm{e}^{h+(\gamma-2)\rho}) \partial^n \rho \partial^m h
 \\
 &  + \epsilon_{amn}\bar{\rho}^{\gamma-1}  \mathrm{e}^{h+(\gamma-2)\rho}\partial^i \partial^n \rho \partial^m h
 \\
 &  +\epsilon_{amn} \bar{\rho}^{\gamma-1}  \mathrm{e}^{h+(\gamma-2)\rho}  \partial^n \rho \partial^i \partial^m h .
\end{split}
\end{equation*}
We therefore derive that
\begin{equation}\label{A1R}
\begin{split}
 & - 2\mathrm{e}^{-{\rho}}\partial^a {\rho}  \mathbf{T}( \partial^iw_a ) \\
  =& - 2\mathrm{e}^{-{\rho}} \partial^a {\rho} \big( \partial^i w^m \partial_m v_a + w^m \partial^i(\partial_{m}v_a) - \partial^i v^k \partial_k w_a \big)
  \\
  &- 2\epsilon_{amn} \bar{\rho}^{\gamma-1} \mathrm{e}^{-{\rho}} \partial^a {\rho}  \partial^i(\mathrm{e}^{h+(\gamma-2)\rho}) \partial^n \rho \partial^m h
 \\
 & - 2\epsilon_{amn} \bar{\rho}^{\gamma-1}  \mathrm{e}^{h+(\gamma-3)\rho} \partial^a {\rho} \partial^i \partial^n \rho \partial^m h )
 \\
 &- 2\epsilon_{amn} \bar{\rho}^{\gamma-1}  \mathrm{e}^{h+(\gamma-3)\rho} \partial^a {\rho}   \mathrm{e}^{h+(\gamma-2)\rho}  \partial^n \rho \partial^i \partial^m h  .
 \end{split}
\end{equation}
Substituting \eqref{A1R} in \eqref{A1F}, we have
\begin{equation}\label{A1v}
\begin{split}
  \Xi_1
  = & \mathbf{T} \big( 2\mathrm{e}^{-{\rho}}   \partial^a {\rho} \partial^i w_a  \big)+Z_4^{i},
\end{split}
\end{equation}
where
\begin{equation}\label{R4i}
  \begin{split}
  Z^i_4=&-2\mathrm{e}^{-{\rho}}   \partial^a v^k \partial_k {\rho}  \partial^i w_a+2 \mathrm{e}^{{\rho}} W^a  \partial^i w_a
  \\
  &
  + 2 \epsilon^{ajk} \partial_j {\rho} w_k  \partial^iw_a+2\mathrm{e}^{-{\rho}}\mathrm{div}\bv \partial^a {\rho}  \partial^i w_a
  \\
  & - 2\mathrm{e}^{-{\rho}} \partial^a {\rho} \big( \partial^i w^m \partial_m v_a + w^m \partial^i(\partial_{m}v_a) - \partial^i v^k \partial_k w_a \big)
  \\
  &- 2\epsilon_{amn} \bar{\rho}^{\gamma-1} \mathrm{e}^{-{\rho}} \partial^a {\rho}  \partial^i(\mathrm{e}^{h+(\gamma-2)\rho}) \partial^n \rho \partial^m h
 \\
 & - 2\epsilon_{amn} \bar{\rho}^{\gamma-1}  \mathrm{e}^{h+(\gamma-3)\rho} \partial^a {\rho}  \partial^m h \partial^i \partial^n \rho
 \\
 &- 2\epsilon_{amn} \bar{\rho}^{\gamma-1}  \mathrm{e}^{h+(\gamma-3)\rho} \partial^a {\rho}   \partial^n \rho \partial^i \partial^m h  .
  \end{split}
\end{equation}
Finally, let us consider $\Xi_2$. By using chain rule, we obtain
\begin{equation}\label{A2m}
\begin{split}
  \Xi_2 = \ &  2\mathrm{e}^{-{\rho}} \partial^i ( \partial_j v^a) \partial^j  w_a
  \\
  = \ &  \partial^i (2\mathrm{e}^{-{\rho}} \partial_{j} v^a \partial^j w_a)+2\mathrm{e}^{-{\rho}} \partial^i {\rho} \partial_{j}v^a \partial^j w_a
  -2\mathrm{e}^{-{\rho}}  \partial_{j} v^a \partial^j(\partial^i  w_a)
  \\
  = \ & \partial^i (2\mathrm{e}^{-{\rho}} \partial_{j} v^a \partial^j w_a)+\tilde{Z}^i_5,
\end{split}
\end{equation}
where
\begin{equation}\label{R5I}
  \tilde{Z}^i_5=2\mathrm{e}^{-{\rho}} \partial^i {\rho} \partial_{j}v^a \partial^j w_a
  -2\mathrm{e}^{-{\rho}}  \partial_{j} v^a \partial^j(\partial^i  w_a).
\end{equation}
Inserting \eqref{A1v}, \eqref{R4i}, \eqref{A2m}, and \eqref{R5I} to \eqref{W15}, we can update \eqref{W15} by
\begin{equation}\label{W17}
\begin{split}
 -2 \epsilon^{ijk} \epsilon_{kmn} \mathrm{e}^{-{\rho}} \partial_{j} (\partial^m v^a) \partial^n w_a
= \mathbf{T} \big( 2\mathrm{e}^{-{\rho}}   \partial^a {\rho} \partial^i w_a  \big)+ \partial^i (2\mathrm{e}^{-{\rho}} \partial_{j} v^a \partial^j w_a)+K_4^{i}+Z^i_5.
\end{split}
\end{equation}
Substituting \eqref{W13},\eqref{W14}, and \eqref{W17} in \eqref{W11}, we find
\begin{equation}\label{CW}
\begin{split}
   \mathrm{curl} \mathbf{T}  (\mathrm{e}^{-\rho} \mathrm{curl}\bw^i)=& \partial^i \big( 2 \mathrm{e}^{-\rho}  \partial_n v_a \partial^n w^b \big) + Z^i_1+Z^i_2+Z^i_3+Z^i_4+Z^i_5
  \\
  &+ \mathbf{T} \big( - 2\mathrm{e}^{- {\rho}}\partial^a {\rho} \partial^i w_a-\frac2\gamma \mathrm{e}^{-2\rho} \epsilon^{ijk} \partial_j v^m  \partial_m\partial_k h-\frac{1}{\gamma}\mathrm{e}^{-2\rho} \epsilon^{ijk} \partial_k h \Delta v_j \big).
\end{split}
\end{equation}
By chain rule, we then get
\begin{equation}\label{TW}
  \begin{split}
  \mathbf{T} \left\{ \mathrm{curl} (\mathrm{e}^{-\rho} \mathrm{curl}\bw^i)\right\}  =& \mathbf{T} \left\{ \epsilon^{ijk}\partial_j (\mathrm{e}^{-\rho} \mathrm{curl}\bw_k) \right\}
  \\
 &=\epsilon^{ijk} \partial_j \mathbf{T}  (\mathrm{e}^{-\rho} \mathrm{curl}\bw_k) +\epsilon^{ijk} [\mathbf{T}, \partial_j](\mathrm{e}^{-\rho} \mathrm{curl}\bw_k)
 \\
 & = \mathrm{curl} \mathbf{T}  (\mathrm{e}^{-\rho} \mathrm{curl}\bw^i) -\epsilon^{ijk} \partial_j v^m \partial_m (\mathrm{e}^{-\rho} \mathrm{curl}\bw_k).
  \end{split}
\end{equation}
Substituting \eqref{CW} in \eqref{TW} yields
\begin{equation}\label{Re}
\begin{split}
 \mathbf{T} \left\{ \mathrm{curl} (\mathrm{e}^{-\rho} \mathrm{curl}\bw^i) \right\}
 =& \mathbf{T}\big( 2\mathrm{e}^{- {\rho}}\partial^a {\rho} \partial^i w_a-\frac2\gamma \mathrm{e}^{-2\rho} \epsilon^{ijk} \partial_j v^m  \partial_m\partial_k h+ \mathrm{e}^{-\rho} \epsilon^{ijk} \partial_k h \Delta v_j \big)
 \\
 &+ \partial^i \big( 2 \mathrm{e}^{-{\rho}} \partial_n v^a \partial^n w_a \big)
+ Z^i_1+Z^i_2+Z^i_3+Z^i_4+Z^i_5.
\end{split}
\end{equation}
Here
\begin{equation}\label{R5ie}
	\begin{split}
	Z^i_5=&\tilde{Z}^i_5-\epsilon^{ijk} \partial_j v^m \partial_m (\mathrm{e}^{-\rho} \mathrm{curl}\bw_k)
	\\
	=&\tilde{Z}^i_5-\epsilon^{ijk} \partial_j v^m \partial_m (\mathrm{e}^{-{\rho}} \mathrm{curl}\bw_k).
	\end{split}
\end{equation}
%At this stage, seeing from \eqref{R1i}, \eqref{R2i}, \eqref{R3i}, \eqref{R4i}, \eqref{R5ie}, \eqref{Ri} and \eqref{Re}, and
Multiplying $\mathrm{e}^{ {\rho}}$ on \eqref{Re}, we conclude
\begin{equation}\label{rrr}
\begin{split}
& \mathbf{T} \left\{ \mathrm{e}^{ {\rho}} \mathrm{curl}  (\mathrm{e}^{-\rho} \mathrm{curl}\bw^i)- 2\partial^a {\rho} \partial^i w_a+2 \mathrm{e}^{-\rho} \epsilon^{ijk} \partial_j v^m  \partial_m\partial_k h-\mathrm{e}^{-\rho} \epsilon^{ijk} \partial_k h \Delta v_j  \right\}
\\
=&    \partial^i \big( 2 \partial_n v_a \partial^n w^a \big) + K^i,
\end{split}
\end{equation}
where $K^i=\sum_{a=1}^6 K^i_a$, $K_a^i=\mathrm{e}^{\rho} Z^i_a, (a=1,2,\cdots,5)$ and $K_6^i$ is as stated in \eqref{R6ib}. We also note $\mathrm{e}^{ {\rho}} \mathrm{curl}  (\mathrm{e}^{-\rho} \mathrm{curl}\bw^i)= \mathrm{curl}\mathrm{curl}\bw^i -\epsilon^{ijk} \partial_j \rho \cdot \mathrm{curl}\bw_k$, so we can obtain \eqref{W2} if we use \eqref{rrr}.
\end{proof}

\begin{Lemma}[transport equations for $\bH$]\label{PWh}
Let $\bH=(H^1,H^2,H^3)^{\mathrm{T}}$ be defined in \eqref{pwh}. Then $\bH$ satisfies
\begin{equation}\label{Hh0}
  \begin{split}
  \mathbf{T}H^i= \mathrm{e}^{-\rho} \partial_j v^j \partial^i h -\mathrm{e}^{-\rho} \partial^j v^i \partial_j h- \epsilon^{ijk}w_k\partial_j h,
  \end{split}
\end{equation}
Moreover, we have
\begin{equation}\label{curH}
  \mathrm{curl} \bH^i=-\epsilon^{ijk}\mathrm{e}^{-\rho}\partial_j\rho \partial_k h,
\end{equation}
and
\begin{equation}\label{Hh1}
  \mathbf{T}( \mathrm{e}^{\rho} \partial_i H^i)=-2\partial_i v^j \partial_j  \partial^i h   + \epsilon^{ijm}\mathrm{e}^{\rho} w_m \partial^i \rho \partial_j h,
\end{equation}
and
\begin{equation}\label{Hh2}
\begin{split}
  \mathbf{T} \left\{ \partial^k( \mathrm{e}^{\rho} \partial_i H^i) \right\}=&-2\partial^k(\partial_i v^j) \partial_j  \partial^i h+Y^k.
\end{split}
\end{equation}
Above, we have denoted
\begin{equation}\label{Hh3}
 \begin{split}
   Y^k=& -2\partial_i v^j \partial^k  \partial^i (\partial_j h) + \epsilon^{ijm}\mathrm{e}^{\rho} w_m \partial^i \rho \partial^k\rho \partial_j h + \epsilon^{ijm}\mathrm{e}^{\rho} \partial^k w_m \partial^i \rho \partial_j h
  \\
  & + \epsilon^{ijm}\mathrm{e}^{\rho} w_m \partial_j h \partial_k \partial^i \rho   + \epsilon^{ijm}\mathrm{e}^{\rho} w_m \partial^i \rho \partial^k \partial_j h .
 \end{split}
\end{equation}
\end{Lemma}
\begin{proof}
Firstly, taking derivatives of $\mathbf{T}h=0$, we have
\begin{equation}\label{H01}
\begin{split}
  \partial_t (\partial^i h)+(\bv \cdot \nabla) (\partial^i h)=&- \partial^i v^j \partial_j h
  \\
  =&- (\partial^i v^j-\partial^j v^i) \partial_j h-\partial^j v^i \partial_j h
  \\
  =& -\partial^j v^i \partial_j h-\mathrm{e}^{\rho}\epsilon^{ijk}w_k\partial_j h.
\end{split}
\end{equation}
By using \eqref{H01} and \eqref{fc0}, we derive that
\begin{equation}\label{H02}
  \begin{split}
  \mathbf{T}(\mathrm{e}^{-\rho}\partial^i h)=& -\partial^i h \mathrm{e}^{-\rho} \mathbf{T}h+ \mathrm{e}^{-\rho} \mathbf{T}(\partial^i h)
  \\
  =& \mathrm{e}^{-\rho} \partial_j v^j \partial^i h -\mathrm{e}^{-\rho} \partial^j v^i \partial_j h- \epsilon^{ijk}w_k\partial_j h.
  \end{split}
\end{equation}
Taking derivative $\partial_i$ of \eqref{H02}, we get
\begin{equation}\label{H03}
  \begin{split}
  \mathbf{T}\left\{ \partial_i (\mathrm{e}^{-\rho}\partial^i h) \right\}=& -\partial_i v^j \partial_j  (\mathrm{e}^{-\rho}\partial^i h)- \mathrm{e}^{-\rho} \partial_i \rho  (\partial^j v^i \partial_j h+\partial_j v^j \partial^i h )- \epsilon^{ijk}\partial_i w_k\partial_j h- \epsilon^{ijk} w_k\partial_{ij} h
  \\
  &+ \mathrm{e}^{-\rho}\Delta h \partial_jv^j + \mathrm{e}^{-\rho}\partial^i h \partial_i (\partial_j v^j) -\mathrm{e}^{-\rho}\partial_j h \partial^j (\partial_i v^i)- \mathrm{e}^{-\rho}   \partial^j v^i \partial_i(\partial_j h)
  \\
  =& -2\mathrm{e}^{-\rho}\partial_i v^j \partial_j  (\partial^i h)+ \mathrm{e}^{-\rho}\Delta h\cdot \mathrm{div}\bv- \mathrm{e}^{-\rho} \partial_i \rho  (\partial^j v^i \partial_j h+\partial_j v^j \partial^i h )
  \\
  &- \epsilon^{ijk}\partial_i w_k\partial_j h- \epsilon^{ijk} w_k\partial_{ij} h  + \mathrm{e}^{-\rho} \partial_i \rho  \partial_i v^j \partial_j h.
  \end{split}
\end{equation}
Therefore, we can deduce that
\begin{equation}\label{H04}
  \begin{split}
  \mathbf{T}\left\{ \mathrm{e}^{\rho} \partial_i (\mathrm{e}^{-\rho}\partial^i h) \right\}=& \mathrm{e}^{\rho}  \mathbf{T}\left\{ \partial_i (\mathrm{e}^{-\rho}\partial^i h) \right\}+ \mathrm{e}^{\rho} \partial_i (\mathrm{e}^{-\rho}\partial^i h )\mathbf{T}\rho
  \\
  =& \mathrm{e}^{\rho}  \mathbf{T}\left\{ \partial_i (\mathrm{e}^{-\rho}\partial^i h) \right\}+ \partial_i \rho\ \partial^i h \mathrm{div}\bv-\mathrm{div}\bv \Delta h
  \\
  =&-2\partial_i v^j \partial_j  (\partial^i h)- \partial_i \rho  (\partial^j v^i \partial_j h+\partial_j v^j \partial^i h )
  \\
  &- \epsilon^{ijk}\mathrm{e}^{\rho}\partial_i w_k\partial_j h- \epsilon^{ijk} \mathrm{e}^{\rho}w_k\partial_{ij} h  +  \partial_i \rho  \partial_i v^j \partial_j h+\partial_i \rho\ \partial^i h \mathrm{div}\bv
  \\
  =&-2\partial_i v^j \partial_j  \partial^i h- \partial_i \rho  \partial_j h (\partial_i v^j-\partial^j v^i)
  - \epsilon^{ijk}\mathrm{e}^{\rho}\partial_j h \partial_i w_k
  \\
  =&-2\partial_i v^j \partial_j  \partial^i h   + \epsilon^{ijm}\mathrm{e}^{\rho} w_m \partial^i \rho \partial_j h.
  \end{split}
\end{equation}
Taking derivatives $\partial^k$ of \eqref{H04} yields
\begin{equation}\label{H06}
\begin{split}
  \mathbf{T} \left\{ \partial^k( \mathrm{e}^{\rho} \partial_i H^i) \right\}=&\partial^k\left\{ -2\partial_i v^j \partial_j  \partial^i h   + \epsilon^{ijm}\mathrm{e}^{\rho} w_m \partial^i \rho \partial_j h \right\}
  \\
  =&-2\partial^k(\partial_i v^j) \partial_j  \partial^i h+Y^k,
\end{split}
\end{equation}
where
\begin{equation}\label{H07}
\begin{split}
  Y^k=& -2\partial_i v^j \partial^k\partial_j  \partial^i h + \epsilon^{ijm}\mathrm{e}^{\rho} w_m \partial^i \rho \partial^k\rho \partial_j h + \epsilon^{ijm}\mathrm{e}^{\rho} \partial^k w_m \partial^i \rho \partial_j h
  \\
  & + \epsilon^{ijm}\mathrm{e}^{\rho} w_m \partial_j h \partial_k \partial^i \rho   + \epsilon^{ijm}\mathrm{e}^{\rho} w_m \partial^i \rho \partial^k \partial_j h .
\end{split}
\end{equation}
\end{proof}
\begin{remark}
	This lemma is different from \cite[Lemma 5.3]{S2}, since we take third-order derivatives on the transport equations. We are not able to prove the lemma directly by taking spatial derivatives as in  \cite[Lemma 5.3]{S2} or  \cite[Proposition 2.1]{DLS}.
\end{remark}
\subsection{Commutator estimates and product estimates}
We start by recalling some commutator estimates and product estimates, which have been proved in references \cite{KP,LD,ML,ST,WQRough}.
\begin{Lemma}\label{jh}\cite{KP}
	Let $ a \geq 0$. Then for any scalar function $f_1, f_2$, we have
	\begin{equation*}%\label{200}
		\|\Lambda_x^a(f_1f_2)-(\Lambda_x^a f_1)f_2\|_{L_x^2} \lesssim \|\Lambda_x^{a-1}f_1\|_{L^{2}_x}\|\partial f_2\|_{L_x^\infty}+ \|f_1\|_{L^p_x}\|\Lambda_x^a f_2\|_{L_x^{q}},
	\end{equation*}
	where $\frac{1}{p}+\frac{1}{q}=\frac{1}{2}$.
\end{Lemma}
\begin{Lemma}\label{cj}\cite{KP}
	Let $a \geq 0$. For any scalar function $f_1$ and $f_2$, we have
	\begin{equation*}%\label{200}
		\|f_1 f_2\|_{H_x^a} \lesssim \|f_1\|_{L_x^{\infty}}\| f_2\|_{H_x^a}+ \|f_2\|_{L_x^{\infty}}\| f_1 \|_{H_x^a}.
	\end{equation*}
\end{Lemma}
\begin{Lemma}\label{jh0}\cite{AK}
Let $F(u)$ be a smooth function of $u$, $F(0)=0$ and $u \in L^\infty_x$. For any $s \geq 0$, we have
	\begin{equation*}%\label{200}
		\|F(u)\|_{H^s} \lesssim  \|u\|_{H^{s}}(1+ \|u\|_{L^\infty_x}).
	\end{equation*}
	%where $\frac{1}{p}+\frac{1}{q}=\frac{1}{2}$.
\end{Lemma}

\begin{Lemma}\label{ps}\cite{ST}
	Suppose that $0 \leq r, r' < \frac{3}{2}$ and $r+r' > \frac{3}{2}$. Then for any scalar function $f_1$ and $f_2$, the following estimate holds:
	\begin{equation*}
		\|f_1f_2\|_{H^{r+r'-\frac{3}{2}}(\mathbb{R}^3)} \leq C_{r,r'} \|f_1\|_{H^{r}(\mathbb{R}^3)}\|f_2\|_{H^{r'}(\mathbb{R}^3)}.
	\end{equation*}
	If $-r \leq r' \leq r$ and $r>\frac{3}{2}$ then
\begin{equation*}
		\|f_1 f_2\|_{H^{r'}(\mathbb{R}^3)} \leq C_{r,r'} \|f_1\|_{H^{r}(\mathbb{R}^3)}\|f_2\|_{H^{r'}(\mathbb{R}^3)}.
	\end{equation*}
\end{Lemma}

\begin{Lemma}\label{lpe}\cite{WQRough}
	Let $0 \leq \alpha <1 $. Then
	\begin{equation*}
		\|\Lambda_x^\alpha(f_1f_2)\|_{L^{2}_x(\mathbb{R}^3)} \lesssim \|f_1\|_{\dot{B}^{\alpha}_{\infty,2}(\mathbb{R}^3)}\|f_2\|_{L^2_x(\mathbb{R}^3)}+ \|f_1\|_{L^\infty_x(\mathbb{R}^3)}\|f_2\|_{\dot{H}^\alpha_x(\mathbb{R}^3)}.
	\end{equation*}
\end{Lemma}
\begin{Lemma}\label{wql}\cite{WQRough}
	Let $0 < \alpha <1 $. Then
	\begin{equation*}
		\|\Lambda_x^\alpha(f_1f_2f_3)\|_{L^{2}_x(\mathbb{R}^3)} \lesssim \|f_i\|_{H^{1+\alpha}_x(\mathbb{R}^3)}\textstyle{\prod_{j\neq i}}\|f_j\|_{H^1_x(\mathbb{R}^3)}.
	\end{equation*}
\end{Lemma}
\begin{Lemma}\cite{ML}\label{Miao}
Given $p \in [1,\infty]$
such that $p \geq m'$ with $m'$ the conjugate exponent
of $m$. Let $f, g$ and $\Phi$ belong to the suitable functional spaces. Then
\begin{equation*}
  \| \Phi * (hf)-h *(\Phi f) \|_{L^p_x} \lesssim \| x \Phi \|_{L^1} \| \nabla f\|_{L^\infty_x} \| h \|_{L^p}.
\end{equation*}
\end{Lemma}
\begin{Lemma}\label{yinli}\cite{LD}
	Let $0<a <1$. Then, for any functions $f_1$ and $f_2$, the following estimates
	\begin{equation}\label{YL0}
		\| \Lambda_x^{a}(f_1f_2)- f_1 \Lambda_x^{a} f_2- f_2 \Lambda_x^{a} f_1 \|_{L^2_x} \lesssim \| \Lambda_x^{a} f_1 \|_{L^2_x} \| f_2 \|_{L^\infty_x}.
	\end{equation}
\end{Lemma}
\begin{remark}
	Let $0<\epsilon<\frac16$. By Lemma \ref{LD}, we have
	\begin{equation}\label{YL2}
		\| \Lambda_x^{\frac12+\epsilon}(f_1f_2)- f_1 \Lambda_x^{\frac12+\epsilon} f_2- f_2 \Lambda_x^{\frac12+\epsilon} f_1 \|_{L^2_x} \lesssim \| \Lambda_x^{\frac12+\epsilon} f_1 \|_{L^2_x} \| f_2 \|_{L^\infty_x}.
	\end{equation}
	So we get
	\begin{equation*}
		\begin{split}
			\| \Lambda_x^{\frac12+\epsilon}(f_1f_2) \|_{L^2_x} \lesssim & \| f_1 \Lambda_x^{\frac12+\epsilon} f_2\|_{L^2_x} + \| f_2 \Lambda_x^{\frac12+\epsilon} f_1 \|_{L^2_x} + \| \Lambda_x^{\frac12+\epsilon} f_1 \|_{L^2_x} \| f_2 \|_{L^\infty_x}
			\\
			\lesssim & \| f_1 \|_{L^p_x} \| \Lambda_x^{\frac12+\epsilon} f_2\|_{L^q_x} + \| f_2 \|_{L^\infty_x} \|\Lambda_x^{\frac12+\epsilon} f_1 \|_{L^2_x}
			\\
			\lesssim & \| f_1 \|_{H^{\frac12+\epsilon}_x} \| \Lambda_x^{\frac12+\epsilon} f_2\|_{H^{\frac12-\epsilon}_x} + \| f_2 \|_{L^\infty_x} \|\Lambda_x^{\frac12+\epsilon} f_1 \|_{L^2_x}
			\\
			\lesssim & \| f_1 \|_{H^{\frac12+\epsilon}_x} \|  f_2\|_{H^{1}_x} + \| f_2 \|_{L^\infty_x} \| f_1 \|_{H^{\frac12+\epsilon}_x}.
		\end{split}
	\end{equation*}
	Here we take $p=\frac{3}{1-\epsilon}$ and $q=\frac{6}{1+2\epsilon}$.

	Take $f_1= \partial \bw, f_2=\partial \bv$,
	\begin{equation}\label{YL3}
		\| \Lambda_x^{\frac12+\epsilon}(\partial \bw \partial \bv) \lesssim \| \partial \bw \|_{H^{\frac12+\epsilon}_x} \| \partial \bv \|_{H^{\frac32}_x}+ \| \partial \bv \|_{L^\infty_x} \| \partial \bw \|_{H^{\frac12+\epsilon}_x} .
	\end{equation}
	Similarly, if we take $f_1= \partial \bw, f_2=\partial \bv$ or $\partial \rho$, then
	\begin{equation}\label{YL4}
		\begin{split}
			& \| \Lambda_x^{\frac12+\epsilon}(\partial^2 h \partial \bv) \lesssim \| \partial^2 h \|_{H^{\frac12+\epsilon}_x} \| \partial \bv \|_{H^{\frac32}_x}+ \| \partial \bv \|_{L^\infty_x} \| \partial^2 h \|_{H^{\frac12+\epsilon}_x} ,
			\\
			& \| \Lambda_x^{\frac12+\epsilon}(\partial^2 h \partial \rho) \lesssim \| \partial^2 h \|_{H^{\frac12+\epsilon}_x} \| \partial \rho \|_{H^{\frac32}_x} + \| \partial \bv \|_{L^\infty_x} \| \partial^2 h \|_{H^{\frac12+\epsilon}_x} .
		\end{split}
	\end{equation}
\end{remark}
\subsection{Useful lemmas}
Let us introduce some useful product estimates and commutator estimates, which play a crucial role in the paper.
\begin{Lemma}\label{LD}%(modified Duhamel's principle)
	Let the metric $g$ is defined in \eqref{metricd}. If $f(t,x;\tau)$ is the solution of
\begin{equation*}
\begin{cases}
  &\square_{g}f=0, \quad t> \tau,
  \\
  &(f,\mathbf{T} f)|_{t=\tau}=-(F,B)(\tau,x),
\end{cases}
\end{equation*}
then
\begin{equation*}
  V(t,x)=\int^t_0 f(t,x;\tau)d\tau
\end{equation*}
is the solution of the linear wave equation
\begin{equation*}
\begin{cases}
  &\square_{g}V=\mathbf{T}F+B,
  \\
  &(V, \mathbf{T}V)|_{t=0}=(0,-F(0,x)).
  \end{cases}
\end{equation*}
\end{Lemma}
\begin{proof}
We first note $\square_g= -\mathbf{T} \mathbf{T}+c^2_s \Delta$.  By calculating, we get
\begin{equation*}
\begin{split}
\mathbf{T} V &= \int^t_0 \mathbf{T} f(t,x;\tau)d\tau+f(t,x;t)
\\
&=\int^t_0 \mathbf{T} f(t,x;\tau)d\tau-F(t,x).
\end{split}
\end{equation*}
Applying the operator $\mathbf{T}$ again, we have
\begin{equation}\label{d1}
\begin{split}
-\mathbf{T}\mathbf{T} V &= -\int^t_0\mathbf{T} \mathbf{T} f(t,x;\tau)d\tau-\mathbf{T}f(t,x;t)+\mathbf{T}F
\\
&= -\int^t_0\mathbf{T} \mathbf{T} f(t,x;\tau)d\tau+B(t,x)+\mathbf{T}F.
\end{split}
\end{equation}
On the other hand,
\begin{equation}\label{d2}
  c^2_s \Delta V= \int^t_0 c^2_s \Delta f(t,x;\tau)d\tau.
\end{equation}
Adding \eqref{d1} and \eqref{d2}, we can derive that
\begin{equation*}
  \square_{g}V=\mathbf{T}F+B.
\end{equation*}
Furthermore, the function $V$ satisfies
\begin{equation*}
  V|_{t=0}=0, \quad \mathbf{T}V|_{t=0}=-F(0,x).
\end{equation*}
The proof of the lemma is now complete.
\end{proof}

\begin{Lemma}\label{LPE}
	{Let $0 < \alpha <1 $. Let $\beta> \alpha$ and sufficiently close to $\alpha$. Then}
	\begin{equation}\label{HF}
		\| f_1f_2\|_{\dot{B}^{\alpha}_{\infty, 2}(\mathbb{R}^3)} \lesssim \|f_2\|_{L^{\infty}_x} \|f_1\|_{\dot{B}^{\alpha}_{\infty, 2}(\mathbb{R}^3)}+\|f_2\|_{C^{\beta}(\mathbb{R}^3)} \|f_1\|_{L^\infty}.
	\end{equation}
\end{Lemma}
\begin{proof}
Firstly, we have
\begin{equation*}
  \| f_1 f_2 \|^2_{\dot{B}^{\alpha}_{\infty, 2}(\mathbb{R}^3)}= \textstyle{\sum}_{j\geq -1} 2^{2j\alpha}\| \Delta_j (f_1 f_2) \|^2_{L^\infty}.
\end{equation*}
By using Bony decomposition, we get
\begin{equation*}
  \Delta_j (f_1 f_2)= \textstyle{\sum}_{|k-j|\leq 2 }\Delta_j (\Delta_k f_2 S_{k-1}f_1)+\textstyle{\sum}_{|k-j|\leq 2 }\Delta_j (\Delta_k f_1 S_{k-1}f_2)+\textstyle{\sum}_{k \geq j-1 }\Delta_j (\Delta_k f_2 \Delta_{k}f_1).
\end{equation*}
By H\"older's inequality, we derive that
\begin{equation}\label{HF1}
  \begin{split}
  & \textstyle{\sum}_{j\geq -1} 2^{2j\alpha}\| \Delta_j (f_1 f_2) \|^2_{L^\infty}
  \\
   \lesssim & \ \textstyle{\sum}_{j\geq -1} 2^{2j\alpha} (\textstyle{\sum}_{|k-j|\leq 2 } \| \Delta_j \Delta_k f_2 \|_{L^\infty} \| S_{k-1} f_1  \|_{L^\infty})^2
   \\
   & \ + \textstyle{\sum}_{j\geq -1} 2^{2j\alpha} (\textstyle{\sum}_{|k-j|\leq 2 } \| \Delta_j \Delta_k f_1 \|_{L^\infty} \| S_{k-1} f_2 \|_{L^\infty})^2
  \\
    & \ + \textstyle{\sum}_{j\geq -1} 2^{2j\alpha} (\textstyle{\sum}_{k \geq j-1} \| \Delta_j \Delta_k f_1 \|_{L^\infty} \| \Delta_{k} f_2 \|_{L^\infty})^2
    \\
   \lesssim & \   \textstyle{\sum}_{j\geq -1} 2^{2j\alpha} \| \Delta_j f_2 \|^2_{L^\infty} \| f_1 \|^2_{L^\infty}+\|h\|^2_{L^{\infty}_x} \|f_1\|^2_{\dot{B}^{\alpha}_{\infty, 2}(\mathbb{R}^3)}
     \\
    & + \textstyle{\sum}_{j\geq -1}\| \Delta_j f_1 \|^2_{L^\infty}  (\textstyle{\sum}_{k \geq j-1}2^{2(-\beta+\alpha)}2^{k\beta}  \| \Delta_{k} f_2 \|_{L^\infty})^2 .
  \end{split}
\end{equation}
It is sufficient to verify the bounds of the last right terms on \eqref{HF1}. Note
\begin{equation}\label{HF2}
\begin{split}
  \textstyle{\sum}_{j\geq -1} 2^{2j\alpha} \| \Delta_j f_2 \|^2_{L^\infty} \| f_1 \|^2_{L^\infty} \lesssim & \textstyle{\sum}_{j\geq -1} 2^{2j (-\beta+\alpha) } 2^{2j \beta }  \| \Delta_j f_1 \|^2_{L^\infty} \| f_1 \|^2_{L^\infty}
  \\
 \lesssim & \| f_1 \|^2_{L^\infty} \{ 2^{2j (-\beta+\alpha) } \}_{l^1_{j}} (\{2^{j \beta }  \| \Delta_j f_2 \|_{L^\infty}\}_{l^\infty_{j}} )^2
  \\
  \lesssim & \| f_1 \|^2_{L^\infty} \| f_2 \|^2_{\dot{B}^{\beta}_{\infty,\infty}}
  \lesssim  \| f_1 \|^2_{L^\infty} \| f_2 \|^2_{C^{\beta}}.
\end{split}
\end{equation}
We also note
\begin{equation}\label{HF3}
\begin{split}
& \textstyle{\sum}_{j\geq -1}\| \Delta_j f_1 \|^2_{L^\infty}  (\textstyle{\sum}_{k \geq j-1}2^{2(-k\beta+j\alpha)}2^{k\beta}  \| \Delta_{k} f_2 \|_{L^\infty})^2
\\
\lesssim & \ \textstyle{\sum}_{j\geq -1}\| \Delta_j f_1 \|^2_{L^\infty}  (\textstyle{\sum}_{k \geq j-1}2^{2(-k\beta+j\alpha)}2^{k\beta}  \| \Delta_{k} f_2 \|_{L^\infty})^2
\\
\lesssim & \ \| f_2 \|^2_{\dot{B}^\beta_{\infty,\infty}}\textstyle{\sum}_{j\geq -1} 2^{2j(\alpha-\beta)}\| \Delta_j f_1 \|^2_{L^\infty}
\\
\lesssim & \ \| f_1 \|^2_{L^\infty}  \| f_2 \|^2_{C^\beta}.
\end{split}
\end{equation}
Substituting \eqref{HF2} and \eqref{HF3} in \eqref{HF1}, we arrive at \eqref{HF}.
\end{proof}

The next commutator estimate is concerns Riesz operators.
\begin{Lemma}\label{ceR}
	{Let $0 \leq \alpha <1$. Denote the Riesz operator $\mathbf{R}:=\partial^2(-\Delta)^{-1}$. Then}
	\begin{equation*}%\label{wq20}
		\| [\mathbf{R}, \bv \cdot \nabla]f\|_{\dot{H}^{\alpha}_x(\mathbb{R}^3)} \lesssim \| \bv\|_{\dot{B}^{1+\alpha}_{\infty, \infty}} \|f\|_{L^2_{x}(\mathbb{R}^3)}+ \| \bv\|_{\dot{B}^1_{\infty, \infty}} \|f\|_{\dot{H}^\alpha_x(\mathbb{R}^3)}.
	\end{equation*}
\end{Lemma}
\begin{proof}
By paraproduct decomposition (c.f. \cite{BCD}), we have
\begin{equation*}
\begin{split}
 \Delta_j [\mathbf{R}, \bv \cdot \nabla]f &= \textstyle{\sum}_{|k-j|\leq 2} \Delta_j \left[\mathbf{R} (\Delta_k \bv \cdot \nabla S_{k-1}f)-\Delta_k \bv \cdot \nabla \mathbf{R} S_{k-1}f \right]
 \\
 & \quad + \textstyle{\sum}_{|k-j|\leq 2} \Delta_j \left[\mathbf{R} (S_{k-1} \bv \cdot \nabla \Delta_{k}f)-S_{k-1} \bv \cdot \nabla \mathbf{R} \Delta_{k}f \right]
 \\
 & \quad + \textstyle{\sum}_{k\geq j-1} \Delta_j \left[\mathbf{R} (\Delta_{k} \bv \cdot \nabla \Delta_{k}f)-\Delta_{k} \bv \cdot \nabla \mathbf{R} \Delta_{k}f \right]
 \\
 & = B_1+B_2+B_3,
 \end{split}
\end{equation*}
where
\begin{equation*}
\begin{split}
  B_1&= \textstyle{\sum}_{|k-j|\leq 2} \Delta_j \left\{ \mathbf{R} (\Delta_k \bv \cdot \nabla S_{k-1}f)-\Delta_k \bv \cdot \nabla \mathbf{R} S_{k-1}f \right\},
  \\
  B_2&= \textstyle{\sum}_{|k-j|\leq 2} \Delta_j \left\{\mathbf{R} (S_{k-1} \bv \cdot \nabla \Delta_{k}f)-S_{k-1} \bv \cdot \nabla \mathbf{R} \Delta_{k}f \right\}
  %\\
%  & = \textstyle{\sum}_{|k-j|\leq 2} \Delta_j [\mathbf{R} , S_{k-1} v \cdot]\Delta_{k} \nabla f,
  \\
  B_3&=\textstyle{\sum}_{k\geq j-1} \Delta_j \left\{ \mathbf{R} (\Delta_{k} \bv \cdot \nabla \Delta_{k}f)-\Delta_{k} \bv \cdot \nabla \mathbf{R} \Delta_{k}f \right\}.
\end{split}
\end{equation*}

For $0\leq \alpha <1$, by H\"older's inequality and Bernstein's inequality, we find that
\begin{equation}\label{B1}
\begin{split}
  \{2^{j\alpha} \| B_1 \|_{L^2_x} \}_{l^2_j} & \lesssim  \left\{ \textstyle{\sum}_{|k-j|\leq 2} (\| \nabla \mathbf{R} S_{k-1} f\|_{L^\infty_x}+\| \nabla S_{k-1} f\|_{L^\infty_x}) 2^{j\alpha} \|\Delta_j \Delta_k \bv\|_{L^2} \right\}_{l^2_j}
  \\
  & \lesssim \left\{ \textstyle{\sum}_{|k-j|\leq 2}  2^{j\alpha} 2^k \|\Delta_j \Delta_k \bv\|_{L_x^\infty} (\|\mathbf{R} S_{k-1} f \|_{L^2}+\| S_{k-1} f \|_{L^2}) \right\}_{l^2_j}
  \\
  & \lesssim \| \bv\|_{\dot{B}^{1+\alpha}_{\infty,\infty}} (\|\mathbf{R} f \|_{L^2_x}+\| f \|_{L^2_x}) \lesssim \| \bv\|_{\dot{B}^{1+\alpha}_{\infty,\infty}} \| f \|_{L^2_x}.
\end{split}
\end{equation}
Due to H\"older's inequality, we have
\begin{equation}\label{B3}
\begin{split}
  \{2^{j\alpha} \| B_3 \|_{L^2_x} \}_{l^2_j} &\lesssim \{2^{j\alpha}  \textstyle{\sum}_{k\geq j-1} 2^k  \| \Delta_k \bv\|_{L^\infty}\cdot 2^{-k}(\| \nabla \Delta_k f\|_{L^2}+\| \nabla \mathbf{R} \Delta_k f\|_{L^2})  \}_{l^2_j}
  \\
  &\lesssim \|\bv\|_{\dot{B}^{1}_{\infty,\infty}} \| f \|_{\dot{H}^\alpha}.
\end{split}
\end{equation}
Note
\begin{equation*}
  B_2= \textstyle{\sum}_{|k-j|\leq 2} \Delta_j [\mathbf{R} , S_{k-1} \bv \cdot]\Delta_{k} \nabla f.
\end{equation*}
By using Lemma \ref{Miao} and H\"older's inequality, we get
\begin{equation}\label{B2}
\begin{split}
 \{2^{j\alpha} \| B_2 \|_{L_x^2} \}_{l^2_j}\lesssim & \left\{ \textstyle{\sum}_{|k-j|\leq 2} \|x\Phi\|_{L^1} \| \nabla S_{k-1} \bv \|_{L^\infty_x} 2^{j\alpha} \| \Delta_j \Delta_{k}f \|_{L^2_x} \right\}_{l^2_j}
 \\
 \lesssim & \left\{ \textstyle{\sum}_{|k-j|\leq 2} \| \nabla S_{k-1} \bv \|_{L^\infty_x} 2^{j\alpha} \| \Delta_j \Delta_{k}f \|_{L^2_x} \right\}_{l^2_j}
\\
 \lesssim & \|\bv\|_{\dot{B}^{1}_{\infty,\infty}}\| f \|_{\dot{H}^\alpha}.% \lesssim  \|\nabla v \|_{L_x^\infty} \| f \|_{L^2_x}.
\end{split}
\end{equation}
Here, we use the fact that ${\Phi}=\frac{x_ix_j}{|x|^2}2^{3j}\Psi(2^jx)$ and $\Psi$ is in Schwartz space. Gathering \eqref{B1}, \eqref{B3} and \eqref{B2} together, we have finished the proof of Lemma \ref{ceR}.
\end{proof}
\begin{Lemma}\label{YR}
	Let $2<s_1 \leq s_2$. Then
	\begin{equation*}%\label{wq20}
		\{ 2^{(s_1-1)j}\| [\Delta_j, \mathbf{T}]f\|_{\dot{H}^{s_2-s_1}_x(\mathbb{R}^3)} \}_{l^2_j} \lesssim  \| \partial \bv\|_{L^\infty_x }\|f\|_{\dot{H}^{s_2-1}_x(\mathbb{R}^3)}+\| \bv \|_{H^{s_2}}  \| f\|_{{H}^{1}}.
	\end{equation*}
\end{Lemma}
\begin{proof}
We first have
\begin{equation*}
  [\Delta_j, \mathbf{T}]f= [\Delta_j, \bv \cdot \nabla]f.
\end{equation*}
By using paraproduct decomposition, we have
\begin{equation*}
\begin{split}
  [\Delta_j, \bv \cdot \nabla]f &= \textstyle{\sum}_{|k-j|\leq 2}\big( \Delta_j  (\Delta_k \bv \cdot \nabla S_{k-1}f)-\Delta_k \bv \cdot \nabla  S_{k-1}\Delta_jf \big)
 \\
 & \quad + \textstyle{\sum}_{|k-j|\leq 2}\big( \Delta_j (S_{k-1} \bv \cdot \nabla \Delta_{k}f)-S_{k-1} \bv \cdot \nabla  \Delta_{k}\Delta_jf \big)
 \\
 & \quad + \textstyle{\sum}_{k\geq j-1} \big( \Delta_j  (\Delta_{k}\bv \cdot \nabla \Delta_{k}f)-\Delta_{k}\bv \cdot \nabla  \Delta_{k}\Delta_jf \big)
 \\
 & = A_1+A_2+A_3,
 \end{split}
\end{equation*}
where
\begin{equation*}
\begin{split}
  A_1&= \textstyle{\sum}_{|k-j|\leq 2}  \left\{  \Delta_j(\Delta_k \bv \cdot \nabla S_{k-1}f)-\Delta_k \bv \cdot \nabla  S_{k-1}\Delta_jf \right\},
  \\
  A_2&= \textstyle{\sum}_{|k-j|\leq 2}  \left\{\Delta_j (S_{k-1} \bv \cdot \nabla \Delta_{k}f)-S_{k-1} \bv \cdot \nabla \Delta_{k}\Delta_jf \right\}
  %\\
%  & = \textstyle{\sum}_{|k-j|\leq 2} \Delta_j [\mathbf{R} , S_{k-1} v \cdot]\Delta_{k} \nabla f,
  \\
  A_3&=\textstyle{\sum}_{k\geq j-1}  \left\{  \Delta_j(\Delta_{k}\bv \cdot \nabla \Delta_{k}f)-\Delta_{k}\bv \cdot \nabla  \Delta_{k}\Delta_jf \right\}.
\end{split}
\end{equation*}
Note $\mathrm{supp} \widehat{A}_1, \mathrm{supp} \widehat{A}_2 \subseteq \{\xi\in \mathbb{R}^3: 2^{j-5} \leq |\xi|\leq 2^{j+5} \}$. By H\"older's inequality and commutator estimate, we find that
\begin{equation}\label{A2}
\begin{split}
  \{2^{j(s_1-1)} \| A_2 \|_{\dot{H}^{s_2-s_1}_x} \}_{l^2_j} & \lesssim  \left\{ 2^{j(s_2-1)} \| A_2 \|_{L^2} \right\}_{l^2_j}
  \\
  & = \left\{ 2^{j(s_2-1)} \| [\Delta_j, S_{j-1} \bv \cdot \nabla] \Delta_j f\|_{L^2} \right\}_{l^2_j}
  \\
  & \lesssim \left\{ 2^{j(s_2-1)}  \|\nabla \bv\|_{L^\infty}  \|\Delta_j f\|_{L^2} \right\}_{l^2_j}\\
  & \lesssim  \|\nabla \bv\|_{L^\infty}  \| f\|_{\dot{H}^{s_2-1}}.
\end{split}
\end{equation}
Due to H\"older's inequality, we have
\begin{equation}\label{A1}
\begin{split}
  \{2^{j(s_1-1)} \| A_1 \|_{\dot{H}^{s_2-s_1}_x} \}_{l^2_j} & \lesssim  \left\{ 2^{j(s_2-1)} \| A_2 \|_{L^2} \right\}_{l^2_j}
  \\
  & = \left\{ 2^{j(s_2-1)} \|\Delta_j \bv \|_{L^3}  \| \nabla f\|_{L^6} \right\}_{l^2_j}
  \\
  & \lesssim  \| \bv \|_{\dot{B}^{s_2-1}_{3,2}}  \| \nabla f\|_{L^6}
  \\
  & \lesssim  \| \bv \|_{H^{s_2-\frac{1}{2}}}  \| f\|_{H^1} \lesssim  \| \bv \|_{H^{s_2}}  \| f\|_{H^1}.
\end{split}
\end{equation}
We rewrite $A_3$ by the form
\begin{equation*}
  A_3= \textstyle{\sum}_{k \geq j-1}  [\Delta_j , \Delta_k \bv \cdot \nabla]\Delta_{k}  f.
\end{equation*}
By using H\"older's inequality, we get
\begin{equation}\label{A3}
\begin{split}
  \{2^{j(s_1-1)} \| A_3 \|_{\dot{H}^{s_2-s_1}_x} \}_{l^2_j} & \lesssim  \|\nabla \bv\|_{L_x^\infty}  \| f\|_{\dot{H}_x^{s_2-1}}.
\end{split}
\end{equation}
Thus, we complete the proof of the lemma.
\end{proof}
\begin{Lemma}\label{ce}
	{Let $0 < \alpha <1 $. Then}
	\begin{equation*}
		\| [\Lambda^\alpha_x, \bv \cdot \nabla]f\|_{L^2_x(\mathbb{R}^3)} \lesssim \|\partial \bv\|_{L^\infty_x} \|f\|_{\dot{H}^{\alpha}_{x}(\mathbb{R}^3)}.
	\end{equation*}
\end{Lemma}
\begin{proof}
By using paraproduct decomposition, we have
\begin{equation*}
\begin{split}
 \Delta_j [\Lambda^\alpha_x, \bv \cdot \nabla]f &= \textstyle{\sum}_{|k-j|\leq 2} \Delta_j \left[\Lambda^\alpha_x (\Delta_k \bv \cdot \nabla S_{k-1}f)-\Delta_k \bv \cdot \nabla \Lambda^\alpha_x S_{k-1}f \right]
 \\
 & \quad + \textstyle{\sum}_{|k-j|\leq 2} \Delta_j \left[\Lambda^\alpha_x (S_{k-1} \bv \cdot \nabla \Delta_{k}f)-S_{k-1} \bv \cdot \nabla \Lambda^\alpha_x \Delta_{k}f \right]
 \\
 & \quad + \textstyle{\sum}_{k\geq j-1} \Delta_j \left[\Lambda^\alpha_x (\Delta_{k}\bv \cdot \nabla \Delta_{k}f)-\Delta_{k}\bv \cdot \nabla \Lambda^\alpha_x \Delta_{k}f \right]
 \\
 & = V_1+V_2+V_3,
 \end{split}
\end{equation*}
where
\begin{equation*}
\begin{split}
  V_1&= \textstyle{\sum}_{|k-j|\leq 2} \Delta_j \left[\Lambda^\alpha_x (\Delta_k \bv \cdot \nabla S_{k-1}f)-\Delta_k \bv \cdot \nabla \Lambda^\alpha_x S_{k-1}f \right],
  \\
  V_2&= \textstyle{\sum}_{|k-j|\leq 2} \Delta_j \left[\Lambda^\alpha_x (S_{k-1} \bv \cdot \nabla \Delta_{k}f)-S_{k-1} \bv \cdot \nabla \Lambda^\alpha_x \Delta_{k}f \right]
  \\
  & = \textstyle{\sum}_{|k-j|\leq 2} \Delta_j [\Lambda^\alpha_x , S_{k-1} \bv \cdot \nabla]\Delta_{k}  f,
  \\
  V_3&=\textstyle{\sum}_{k\geq j-1} \Delta_j \left[\Lambda^\alpha_x (\Delta_{k}\bv \cdot \nabla \Delta_{k}f)-\Delta_{k}\bv \cdot \nabla \Lambda^\alpha_x \Delta_{k}f \right].
\end{split}
\end{equation*}

For $0<\alpha <1$, by H\"older's inequality and Bernstein's inequality, we can derive
\begin{equation}\label{V1}
\begin{split}
  \{ \| V_1 \|_{L^2_x} \}_{l^2_j} & \lesssim  \left\{ \textstyle{\sum}_{|k-j|\leq 2} (\| \nabla \Lambda_x^\alpha S_{k-1} f\|_{L^\infty_x}+2^{j\alpha}\| \nabla S_{k-1} f\|_{L^\infty_x}) \|\Delta_j \Delta_k \bv\|_{L^2_x} \right\}_{l^2_j}
  \\
  & \lesssim \left\{ \textstyle{\sum}_{|k-j|\leq 2}  2^{\alpha j} \|\Delta_j \Delta_k \bv\|_{L_x^2} \| S_{k-1}\Lambda_x f \|_{L^\infty_x} \right\}_{l^2_j}
  \\
  & \lesssim \|\partial \bv\|_{L^\infty_x} \| f \|_{\dot{H}^\alpha}. %\lesssim \| \nabla v\|_{L_x^\infty} \| f \|_{\dot{H}^\alpha}.
\end{split}
\end{equation}
Using H\"older's inequality again, we have
\begin{equation}\label{V3}
  \{ \| V_3 \|_{L^2_x} \}_{l^2_j} \lesssim \| \nabla \bv\|_{L_x^\infty} \| f \|_{\dot{H}^\alpha}.
\end{equation}
By using Lemma \ref{Miao} (by setting ${\Phi}=2^{j(\alpha+3)}\Psi(2^3x)$, $\Psi$ is in Schwartz space) and H\"older inequality, we find that
\begin{equation}\label{V2}
\begin{split}
 \{ \| V_2 \|_{L_x^2} \}_{l^2_j}\lesssim & \left\{ \textstyle{\sum}_{|k-j|\leq 2} \|x\Phi\|_{L^1} \| \nabla S_{k-1} \bv \|_{L^\infty_x} \| \Delta_j \Delta_{k}f \|_{L^2_x} \right\}_{l^2_j}
 \\
 \lesssim & \left\{ \textstyle{\sum}_{|k-j|\leq 2} 2^{j\alpha} \| \nabla S_{k-1} \bv \|_{L^\infty_x} \| \Delta_j \Delta_{k}f \|_{L^2_x} \right\}_{l^2_j}
\\
 \lesssim & \|\nabla \bv\|_{\dot{B}^{0}_{\infty,\infty}}\| f \|_{\dot{H}^\alpha} \lesssim  \|\nabla \bv \|_{L_x^\infty} \| f \|_{\dot{H}^\alpha}.
\end{split}
\end{equation}
Gathering \eqref{V1}, \eqref{V3}, and \eqref{V2} together, we have proved this lemma.
\end{proof}
\begin{Lemma}\label{yx}
Let $2<s_1\leq s_2$. Let $\mathbf{g}$ be a Lorentz metric and $\mathbf{g}^{00}=-1$. We have
\begin{equation}\label{YX}
\begin{split}
 \left\{ 2^{(s_1-1)j}\|[\square_{\mathbf{g}}, \Delta_j]f\|_{\dot{H}^{s_2-s_1}} \right\}_{l^2_j} \lesssim & \ \| d f\|_{L_x^\infty} \|d \mathbf{g}\|_{\dot{H}^{s_2-1}}+\| d \mathbf{g}\|_{L_x^\infty}\|d f\|_{\dot{H}^{s_2-1}}.
	% +\|d f\|_{\dot{H}^{s}} .
\end{split}
\end{equation}
\end{Lemma}
	\begin{proof}
Let ${R}_j=[\square_{\mathbf{g}}, \Delta_j]f$ and $S_j=\textstyle\sum_{j'\leq j} \Delta_{j'}$. Note $\mathbf{g}^{00}=-1$. Then we can decompose ${R}_j$ as
\begin{equation*}
	{R}_j={E}_j+{A}_j+{G}_j,
\end{equation*}
where
\begin{equation*}
\begin{split}
E_j=&\Delta_j(\mathbf{g}^{\alpha i}\partial_{\alpha i }f)-S_{j}(\mathbf{g}^{\alpha i})\Delta_j(\partial_{\alpha i }f),
\\
A_j=&\left( S_{j}(\mathbf{g}^{\alpha i})-\mathbf{g}^{\alpha i}\right)\Delta_j(\partial_{\alpha i}f),
\\
G_j=&\Delta_j\big( c^{-1}_s \partial_\alpha(c_s\mathbf{g}^{\alpha\beta})\big)\partial_{\beta}f.
\end{split}
\end{equation*}
We note
\begin{equation*}
  \Delta_{k} \Delta_j f=0, \quad \mathrm{if} \ |k-j|\geq 8.
\end{equation*}
Note again the support of $\widehat{E}_j$ being in the set $\{\xi: 2^{j-5} \leq |\xi|\leq 2^{j+10}\}$, then
\begin{equation*}
\begin{split}
  &\| E_j \|_{\dot{H}^{s_2-s_1}} \lesssim 2^{j(s_2-s_1)}\| E_j\|_{L^2_x}.
\end{split}
\end{equation*}
For $A_j$, we first have
\begin{equation*}
   A_j = \sum_{j'>j} \Delta_{j'} (\mathbf{g}^{\alpha i})\Delta_j(\partial_{\alpha i}f).
\end{equation*}
By classical product estimates, we derive that
\begin{equation}\label{AJ}
\begin{split}
  \| A_j \|_{\dot{H}^{s_2-s_1}}
  & \lesssim  \sum_{j'>j} ( \| \Delta_{j'} (\mathbf{g}^{\alpha i}) \|_{L^\infty} \| \Delta_j(\partial_{\alpha i}f) \|_{\dot{H}^{s_2-s_1}}+ \| \Delta_{j'} (\mathbf{g}^{\alpha i}) \|_{\dot{H}^{s_2-s_1}} \| \Delta_j(\partial_{\alpha i}f) \|_{L^\infty}).
  %\\
%  & \lesssim \textstyle{\sum_{m\geq j}} 2^{-m}\|\Delta_m \partial g\|_{L^\infty_x}  \|\Delta_j(\partial_{\alpha i}f)\|_{L^2_x}
%  \\
%  & \lesssim  \| \partial v, \partial \boldsymbol{\rho}\|_{L^\infty_x} \|\Delta_j(d f)\|_{L^2_x}.
\end{split}
\end{equation}
Using Bernstein's inequality, we can update \eqref{AJ} as
\begin{equation*}
\begin{split}
 & \{  2^{j(s_1-1)}\| A_j \|_{\dot{H}^{s_2-s_1}}  \}_{l^2_j}
 \\
  \lesssim  & \{ \sum_{j'>j} ( \| \Delta_{j'} (\nabla \mathbf{g}^{\alpha i}) \|_{L^\infty} 2^{-j'+j}\| \Delta_j(\partial f) \|_{\dot{H}^{s_2-1}}
  \\
  & + 2^{j(s_1-1)}\| \Delta_{j'} (\mathbf{g}^{\alpha i}) \|_{\dot{H}^{s_2-s_1}}\cdot 2^j\| \Delta_j(\partial f) \|_{L^\infty} ) \}_{l^2_j}.
\end{split}
\end{equation*}
We set
\begin{equation*}
  \begin{split}
  N_1&=\{ \sum_{j'>j}  \| \Delta_{j'} (\nabla \mathbf{g}^{\alpha i}) \|_{L^\infty} 2^{-j'+j}\| \Delta_j(\partial f) \|_{\dot{H}^{s_2-1}} \}_{l^2_j},
  \\
  N_2&=\{ \sum_{j'>j} 2^{j(s_1-1)}\| \Delta_{j'} (\mathbf{g}^{\alpha i}) \|_{\dot{H}^{s_2-s_1}}\cdot 2^j\| \Delta_j(\partial f) \|_{L^\infty} ) \}_{l^2_j}.
  \end{split}
\end{equation*}
A direct calculation tells us
\begin{equation*}
  N_1 \lesssim \| d \mathbf{g} \|_{L^\infty_x} \|d f\|_{\dot{H}^{s_2-1}_x}, \quad N_2 \lesssim \| d f \|_{L^\infty_x} \|d \mathbf{g}\|_{\dot{H}^{s_2-1}_x}.
\end{equation*}
In a result, we get
\begin{equation}\label{A9}
 \left\{   2^{(s_1-1)j} \| A_j \|_{\dot{H}^{s_2-s_1}} \right\}_{l^2_j} \lesssim \| d \mathbf{g} \|_{L^\infty_x} \|d f\|_{\dot{H}^{s_2-1}_x}+ \| d f \|_{L^\infty_x} \|d \mathbf{g}\|_{\dot{H}^{s_2-1}_x}.
\end{equation}
For $G_j$, by phase decomposition, we obtain
\begin{equation*}
\begin{split}
  \Delta_k G_j=& \textstyle{\sum_{|k-m|\leq 2}} \Delta_k \left( \Delta_m\big(\Delta_j( c^{-1}_s \partial_\alpha(c_s\mathbf{g}^{\alpha\beta}))\big) S_{m-1}\partial_{\beta}f \right)
  \\
  &+\textstyle{\sum_{|k-m|\leq 2}} \Delta_k \left( S_{m-1}\big(\Delta_j( c^{-1}_s \partial_\alpha(c_s\mathbf{g}^{\alpha\beta}))\big) \Delta_m\partial_{\beta}f \right)
  \\
  &+\Delta_k \left( \tilde{\Delta}_k\big(\Delta_j( c^{-1}_s \partial_\alpha(c_s\mathbf{g}^{\alpha\beta}))\big)\tilde{\Delta}_k\partial_{\beta}f \right),
\end{split}
\end{equation*}
where $\tilde{\Delta}_k=\Delta_{k-1}+\Delta_k+\Delta_{k+1}$. Due to H\"older's inequality, we have
\begin{equation*}
\begin{split}
  \|\Delta_k G_j \|_{L^2_x} \lesssim & \ \|\Delta_k \Delta_j\big( c^{-1}_s \partial_\alpha(c_s\mathbf{g}^{\alpha\beta})\big)\|_{L^2_x} (\| S_k\partial_{\beta}f \|_{L^\infty_x}+\| \Delta_k \partial_{\beta}f \|_{L^\infty_x})
  \\
  & +\| S_k\big(\Delta_j( c^{-1}_s \partial_\alpha(c_s\mathbf{g}^{\alpha\beta})) \|_{L^\infty_x} \|\Delta_k\partial_{\beta}f \|_{L^2_x}.
\end{split}
\end{equation*}
By H\"older's inequality again, we prove
\begin{equation*}
\begin{split}
  \| G_j \|_{L^2_x} \lesssim  \| \Delta_j\big( c^{-1}_s \partial_\alpha(c_s\mathbf{g}^{\alpha\beta})\big)\|_{L^2_x} \| d f \|_{L^\infty_x}.
\end{split}
\end{equation*}
In a result, we can get
\begin{equation*}
  \| G_j \|_{\dot{H}_x^{s_2-s_1}} \lesssim \| \Delta_j\big( c^{-1}_s \partial_\alpha(c_s\mathbf{g}^{\alpha\beta})\big)\|_{L^2_x} \| d f \|_{L^\infty_x}+ \|  c^{-1}_s \partial_\alpha(c_s\mathbf{g}^{\alpha\beta}) \|_{L^\infty_x} \|\Delta_j \partial_{\beta}f \|_{L^2_x}.
\end{equation*}
\begin{equation}\label{G9}
 \left\{   2^{(s_1-1)j} \| G_j \|_{\dot{H}^{s_2-s_1}} \right\}_{l^2_j} \lesssim \| d f\|_{L^\infty_x} \|d\mathbf{g} \|_{\dot{H}^{s_2-1}_x}+ \| d\mathbf{g}\|_{L^\infty_x} \|df \|_{\dot{H}^{s_2-1}_x}.
\end{equation}
It remains for us to give a bound for $E_j$. By phase decomposition, we have
\begin{equation*}
  \begin{split}
  E_j= & \Delta_j(\mathbf{g}^{\alpha i}\partial_{\alpha i }f)-S_{j}(\mathbf{g}^{\alpha i})\Delta_j(\partial_{\alpha i }f)
  \\
   =& \Delta_j\left(S_j(\mathbf{g}^{\alpha i})\partial_{\alpha i }f \right)-S_{j}(\mathbf{g}^{\alpha i})\Delta_j(\partial_{\alpha i }f)
  \\
  &+ \Delta_j\left(\Delta_j(\mathbf{g}^{\alpha i})S_j(\partial_{\alpha i }f) \right) + \Delta_j \left(\tilde{\Delta}_j(\mathbf{g}^{\alpha i})\tilde{\Delta}_j(\partial_{\alpha i }f) \right),
  \end{split}
\end{equation*}
where $\tilde{\Delta}_j={\Delta}_{j-1}+{\Delta}_j+{\Delta}_{j+1}$.
By using commutator estimates, we have
\begin{equation*}
\begin{split}
  \| E_j \|_{L^2_x} &= \| [{\Delta}_j, {S}_j \mathbf{g}^{\alpha i} ]\Delta_j(\partial_{\alpha i })f \|_{L^2_x}+\| \Delta_j\left(\Delta_j(\mathbf{g}^{\alpha i})S_j(\partial_{\alpha i }f) \right)\|_{L^2_x} + \| \Delta_j \left(\tilde{\Delta}_j(\mathbf{g}^{\alpha i})\tilde{\Delta}_j(\partial_{\alpha i }f) \right)\|_{L^2_x}
  \\
  & \lesssim  \|\partial \mathbf{g}\|_{L^\infty_x} \|\Delta_j(d f) \|_{L^2_x}+\|\partial f\|_{L^\infty_x} \|\Delta_j(d \mathbf{g}) \|_{L^2_x}.
\end{split}
\end{equation*}
Therefore, we obtain
\begin{equation}\label{E9}
 \left\{   2^{(s_1-1)j} \| E_j \|_{\dot{H}^{s_2-s_1}} \right\}_{l^2_j} \lesssim \| \partial \mathbf{g}\|_{L^\infty_x} \|d f\|_{\dot{H}^{s_2-1}_x}+\| \partial f\|_{L^\infty_x} \|d \mathbf{g}\|_{\dot{H}^{s_2-1}_x}.
\end{equation}
By combining \eqref{A9}, \eqref{G9}, and \eqref{E9}, we find that \eqref{YX} holds.
\end{proof}
\begin{Lemma}\label{yux}
Let $(\bv,\rho,h)$ be a solution of \eqref{fc0}. Let $\bv_{-}$ be defined in \eqref{etad}. Let $\bQ, D$ and $E$ be stated in \eqref{DDi}. Let $s \in (2,\frac52]$ and $2<s_0<s$. Then the following estimates
\begin{equation}\label{YYE}
  \| \bQ, D, E\|_{ H_x^{s-1}} \lesssim \| d\rho, d\bv, dh \|_{L_x^\infty} (\| \rho, \bv \|_{H_x^{s}}+ \| h \|_{H^{\frac52+}}+ \| \bw \|_{H^{\frac32+}}),
\end{equation}
and
\begin{equation}\label{eta}
  \| \mathbf{T} \bv_{-} \|_{H_x^s} \lesssim  \|h\|^3_{H_x^{s}} + \| \rho \|^3_{H_x^{s}}+  ( \|\bv\|_{H_x^{s}} + \| \rho \|_{H_x^{s}})(\| \bw \|_{H^{\frac{3}{2}+}} + \| h \|_{H_x^{\frac{5}{2}+}}) ,
\end{equation}
hold. Moreover, the function $\bv_{-}$ satisfies
\begin{equation}\label{eee}
  \| \bv_{-} \|_{H^{s_0+1}}  \lesssim (1+\| \rho \|_{H^{s_0}_x}) \| \bw \|_{H_x^{s_0}}.
\end{equation}
\end{Lemma}
\begin{proof}
We recall the expression of $D$
\begin{equation*}
\begin{split}
  D=& -3c_s^{-1}\frac{\partial c_s}{\partial \rho} g^{\alpha \beta} \partial_\alpha \rho \partial_\beta \rho+2 \textstyle{\sum_{1 \leq a < b \leq 3} }\big\{ \partial_a v^a \partial_b v^b-\partial_a v^b \partial_b v^a \big\}
\\
&  - \frac{1}{\gamma}c^2_s \partial^a h \partial_a \rho- \frac{1}{\gamma}c^2_s \partial^a h \partial_a h+c^3_s \partial_\beta h \partial_{\alpha} (c^{-3}_s g^{\alpha \beta})
\end{split}
\end{equation*}
which is described in \eqref{DDi}. Using Lemma \ref{cj}, we can get
\begin{equation}\label{Ds}
  \| D \|_{H_x^{s-1}} \lesssim  \| d\rho \|_{L_x^\infty} \| d\rho \|_{H_x^{s-1}}+ \| \partial \bv \|_{L_x^\infty} \| \partial v \|_{H_x^{s-1}}.
\end{equation}
Note
\begin{equation*}
\begin{split}
  {Q^i}:=& -\frac{2}{\gamma}e^{\rho} c^2_s \epsilon^i_{ab}  w^b  \partial^a h-( 1+c_s^{-1}\frac{\partial c_s}{\partial \rho})g^{\alpha \beta} \partial_\alpha {\rho} \partial_\beta v^i
 -\frac{1}{\gamma}c^2_s \epsilon^{iab}w_a \partial_b h
 \\
 & - \frac12 c^2_s \partial^a h \partial_a v^i+ \frac{2c_s}{\gamma} \frac{\partial c_s}{\partial \rho} \mathrm{div}\bv \partial^i h+c^3_s \partial_\beta h \partial_{\alpha} (c^{-3}_s g^{\alpha \beta}).
\end{split}
\end{equation*}
By Lemma \ref{cj}, we can deduce that
\begin{equation}\label{Qi}
\begin{split}
  \|\bQ \|_{H_x^{s-1}}&  \lesssim (\|d\rho \|_{L_x^\infty} + \|d \bv \|_{L_x^\infty}+ \|dh \|_{L_x^\infty}) ( \| d \rho \|_{H_x^{s-1}}+ \| d \bv \|_{H_x^{s-1}}+ \| d h \|_{H_x^{s-1}}).
\end{split}
\end{equation}
By using \eqref{fc0}, then we have
\begin{equation}\label{fa}
  \| d \rho \|_{H_x^{s-1}}+ \| d \bv \|_{H_x^{s-1}}+ \| d h \|_{H_x^{s-1}} \lesssim \|\bv\|_{H_x^{s}}+\| \rho \|_{H_x^{s}}+\| h \|_{H_x^{s}}.
\end{equation}
Combining \eqref{fa}, \eqref{Qi}, and \eqref{Ds}, we have proved \eqref{YYE}. It remains for us to prove \eqref{eta} and \eqref{eee}. By using \eqref{etad} and Sobolev imbedding, and elliptic estimates, we can derive that
\begin{equation}\label{ETAs}
\begin{split}
  \| \bv_{-} \|_{H_x^{s}} \leq  \| \mathrm{e}^{\rho} \mathrm{curl} \bw \|_{H_x^{s-2}} &\lesssim  (1+\|\rho\|_{H^{\frac{3}{2}+}} )\| \bw \|_{H_x^{s-1}},
\end{split}
\end{equation}
and
\begin{equation}\label{ETAs0}
\begin{split}
  \| \bv_{-}  \|_{H^{s_0+1}} \lesssim (1+\| \rho \|_{H^{2}_x})\| \bw \|_{H_x^{s_0}}.
\end{split}
\end{equation}
Let us give a bound for $\mathbf{T} \bv_{-}$. By using \eqref{etad}, we get
\begin{equation*}
  -\Delta \bv_{-} =\mathrm{e}^{\rho}\mathrm{curl}\bw.
\end{equation*}
Then we have
\begin{equation}\label{bs}
  -\Delta(\mathbf{T} v_{-}^i)  = \mathbf{T} (\mathrm{e}^{\rho}\mathrm{curl}\bw^i)  - \Delta v^m \partial_m v_{-}^i-2 \partial_j v^m \partial^j(\partial_m v_{-}^i).
\end{equation}
By \eqref{bs}, \eqref{W1}, and Lemma \ref{ps}, we have
\begin{equation}\label{bsl}
\begin{split}
 \| \mathbf{T}  \bv_{-} \|_{H_x^s}  \lesssim & \| \partial \bv \cdot \partial \bw \|_{H_x^{s-2}}+ \| \partial^2 \bv \cdot \partial \bv_{-} \|_{H_x^{s-2}}+\| \partial^2 \bv_{-} \cdot \partial \bv \|_{H_x^{s-2}}+ \| \partial^2 \rho \cdot \partial h \|_{H_x^{s-2}}
 \\
 & + \| \partial h \cdot \partial \rho \cdot \partial h \|_{H_x^{s-2}}+ \| \partial \rho \cdot \partial \rho \cdot \partial h \|_{H_x^{s-2}}
  + \| \partial \rho \cdot \partial^2 h \|_{H_x^{s-2}}
 \\
  \lesssim & \| \partial \bv \|_{H_x^{s-1}} \| \partial \bw \|_{H_x^{\frac{1}{2}+}} +\| \partial^2 \bv \|_{H_x^{s-2}} \| \partial \bv_{-} \|_{H_x^{\frac{3}{2}+}}+\| \partial \bv \|_{H_x^{s-1}} \| \partial^2 \bv_{-} \|_{H_x^{\frac{1}{2}+}}
  \\
  & + (\| \partial h \|_{H_x^{s-1}}+\| \partial \rho \|_{H_x^{s-1}}) \| \partial \rho \|_{H_x^{s-1}} \| \partial h \|_{H_x^{s-1}}
  \\
  &+  \| \partial \rho \|_{H_x^{s-1}} \| \partial^2 h \|_{H_x^{\frac{1}{2}+}} + \| \partial^2 \rho \|_{H_x^{s-2}} \| \partial h \|_{H_x^{\frac{3}{2}+}}
 \\
   \lesssim & (\| \bw \|_{H^{\frac{3}{2}+}} + \| h \|_{H_x^{\frac{5}{2}+}})  ( \|\bv\|_{H_x^{s}} + \| \rho \|_{H_x^{s}})+  \|h\|^3_{H_x^{s}} + \| \rho \|^3_{H_x^{s}}.
\end{split}
\end{equation}
This completes the proof of the lemma.
\end{proof}
\section{Energy estimates} \label{sec:energyest}
In this part, we will prove the energy estimates. Let us first use the hyperbolic system to give a classical energy estimate for density, velocity, and entropy.
\begin{theorem}[\cite{M}]\label{dv}{(Classical energy estimates for velocity and density)}
	Let $(\rho,\bv,h)$ be a solution of \eqref{fc0}. Then for any $a\geq 0$, we have
%\begin{equation}\label{E0}
 %\frac{d}{dt}\| ({\rho},\bv,h)\|^2_{H_x^a}  \lesssim    \|(d\bv, d\rho,dh)\|_{L^\infty_x}\| ({\rho},\bv,h)\|^2_{H_x^a},
%\end{equation}
and
\begin{equation}\label{E0}
 \| ({\rho},\bv,h)\|_{H_x^a}  \leq   \|({\rho}_0,\bv_0,h_0)\|_{H_x^a}  \exp(C {\int^t_0} \|(d\bv, d\rho,dh)\|_{L^\infty_x}d\tau), \quad t \in [0,T].
\end{equation}
\end{theorem}

\subsection{Energy estimates of Theorem \ref{dingli}}\label{ES1}
\begin{theorem}\label{ve}{(Energy estimates: type 1)}
Let $(\rho,\bv,h)$ be a solution of \eqref{fc0}. Let $\bw$ be defined in \eqref{pw11} and $\bw$ satisfy the equation \eqref{W0}. For $2<s_0<s\leq\frac52$, the following energy estimates hold:
\begin{equation}\label{eet}
\begin{split}
  E_s(t) \leq & C{E}_0 \exp\left\{ E_0 \int^t_0  \|(\partial \bv, \partial \rho, \partial h) \|_{\dot{B}^{s_0-2}_{\infty,2}}  d\tau
    \cdot \exp ( \int^t_0  \|(d \bv, d \rho, d h) \|_{L^\infty_x}  d\tau ) \right\},
\end{split}
\end{equation}
where $C$ is a universal constant in the sense of section \ref{sec:statement} and
\begin{equation}\label{ee1}
  E_s(t)=\|\bv\|^2_{H_x^{s}} + \|\rho\|^2_{H_x^{s}}+\|h\|^2_{H_x^{s_0+1}}+ \|\bw\|^2_{H_x^{s_0}},
\end{equation}
and ${E}_0=E_s(0)+E^2_s(0)+E^3_s(0)+E^{2s_0-1}_s(0)+E^{1+\frac{s_0-1}{s_0-2}}_s(0)$.
\end{theorem}
\begin{proof}
By using \eqref{E0}, we have
\begin{equation}\label{e000}
 \| ({\rho},\bv,h)\|^2_{H_x^s}(t)\leq C\| (\rho_0,\bv_0,h_0)\|^2_{H_x^s} \exp\{ \int^t_0 \|(d\bv, d\rho,dh)\|_{L^\infty_x} d\tau \} .
\end{equation}
We then discuss $\bw$ and $h$ into several steps.

\textbf{Step 1: The estimate for $\bw$}. Recall that $\bw$ satisfies \eqref{W0}. Multiplying $\bw$ on \eqref{W0} and integrating it on $ [0,t]\times \mathbb{R}^3$, we can get
\begin{equation}\label{e001}
  \| \bw \|^2_{L_x^2}(t)- \| \bw \|^2_{L_x^2}(0)\leq C \int^t_0 \| \partial \bv\|_{L^\infty_x}\|\bw\|^2_{L^2_x}d\tau+C\int^t_0\| \partial \rho\|_{L^\infty_x}\|\partial h\|_{L^2_x} \|\bw\|_{L^2_x} d\tau.
\end{equation}
By elliptic estimates, so we have
\begin{equation}\label{w1}
  \| \bw \|_{\dot{H}^{s_0}_x} \leq \| \mathrm{curl} \mathrm{curl} \bw \|_{\dot{H}^{s_0-2}_x}+\| \mathrm{div} \bw\|_{\dot{H}^{s_0-1}_x}.
\end{equation}
%\textbf{Step 2: higher-order energy estimate of $\mathrm{div} \bw$}.
Let us first estimate $\| \mathrm{div} \bw \|_{\dot{H}^{s_0-1}_x}$. Note \eqref{Wd1}. By applying Lemma \ref{ps} to \eqref{Wd1}, it follows that
\begin{equation}\label{DW}
\begin{split}
  \| \mathrm{div} \bw \|_{{H}^{s_0-1}_x} &= \|  \bw \cdot \partial \rho  \|_{{H}^{s_0-1}_x} \leq C \|\bw\|_{H_x^{\frac32+}}\|\rho\|_{{H}^{s_0}_x}.
\end{split}
\end{equation}
Let us next estimate $\| \mathrm{curl} \mathrm{curl} \bw \|_{\dot{H}^{s_0-2}_x}$. Taking derivatives $\Lambda^{s_0-2}_x$ of \eqref{W2}, we find
\begin{equation}\label{W31}
\begin{split}
 \mathbf{T} \left\{ \Lambda^{s_0-2}_x \left( \mathrm{curl} \mathrm{curl} \bw^i +F^i \right) \right\}=& \Lambda^{s_0-2}_x \partial^i \big( 2 \partial_n v_a \partial^n w^a \big)+ \Lambda^{s_0-2}_xK^i
 \\
 &  + [\mathbf{T},\Lambda^{s_0-2}_x]\left( \mathrm{curl} \mathrm{curl} \bw^i +F^i \right).
\end{split}
\end{equation}
Multiplying $\Lambda^{s_0-2}_x(\mathrm{curl} \mathrm{curl} \bw^i +F^i)$
on \eqref{W31}, and then integrating it over $[0,t] \times \mathbb{R}^3$, which can implies
\begin{equation}\label{W32}
\begin{split}
    \|\mathrm{curl} \mathrm{curl} \bw + \bF\|^2_{\dot{H}^{s_0-2}_x}(t)
  \lesssim & \|\mathrm{curl} \mathrm{curl} \bw + \bF\|^2_{\dot{H}^{s_0-2}_x}(0)+ | \int^t_0 \int_{\mathbb{R}^3}\mathrm{div}\bv \cdot |\Lambda^{s_0-2}_x(\mathrm{curl} \mathrm{curl} \bw + \bF)|^2 dxd\tau |
  \\
  &+  | \int^t_0 \int_{\mathbb{R}^3}[\mathbf{T},\Lambda^{s_0-2}_x]\left( \mathrm{curl} \mathrm{curl} \bw^i +F^i \right)\cdot \Lambda^{s_0-2}_x(\mathrm{curl} \mathrm{curl} \bw_i + F_i)dxd\tau |
  \\
  & +  | \int^t_0 \int_{\mathbb{R}^3}\Lambda^{s_0-2}_x \partial^i \big( 2   \partial_n v_a \partial^n w^a \big) \cdot \Lambda^{s_0-2}_x(\mathrm{curl} \mathrm{curl} \bw_i + F_i)dxd\tau |
  \\
  & + | \int^t_0 \int_{\mathbb{R}^3} \Lambda^{s_0-2}_x K^i \cdot \Lambda^{s_0-2}_x(\mathrm{curl} \mathrm{curl} \bw_i + F_i)dxd\tau | ,
\end{split}
\end{equation}
where $ \bF=(F^1,F^2,F^3)^{\mathrm{T}}$. By H\"older's inequality, it follows that
\begin{equation}\label{e002}
  \|\bF\|_{\dot{H}^{s_0-2}_x}(t) \lesssim  \|\rho\|_{H^{s_0}_x}\|\bw\|_{H^{\frac32+}_x}+\|\bv\|_{H^{s_0}_x}\|h\|_{H^{\frac52+}_x}.
\end{equation}
By Young's inequality, we further derive that
\begin{equation}\label{e002a}
  \|\bF\|^2_{\dot{H}^{s_0-2}_x}(t) \leq  \frac18 (\|\bw\|^2_{H^{s_0}_x}+ \|h\|^2_{H^{s_0+1}_x})+ C (\|\bv\|^4_{H^{s_0}_x}+\|\rho\|^4_{H^{s_0}_x} ).
\end{equation}
For the right hand side of \eqref{W32}, let us set
\begin{equation}\label{W34}
\begin{split}
   I_0
  =& \|\mathrm{curl} \mathrm{curl} \bw + \bF\|^2_{\dot{H}^{s_0-2}_x}(0),
  \\
  I_1=&  \int^t_0 \int_{\mathbb{R}^3}\mathrm{div}\bv \cdot |\Lambda^{s_0-2}_x(\mathrm{curl} \mathrm{curl} \bw + \bF)|^2 dxd\tau,
  \\
  I_2=& \int^t_0 \int_{\mathbb{R}^3} \Lambda^{s_0-2}_xK^i\cdot\Lambda^{s_0-2}_x(\mathrm{curl} \mathrm{curl} \bw_i + F_i)dxd\tau,
  \\
  I_3=& \int^t_0 \int_{\mathbb{R}^3}[\mathbf{T},\Lambda^{s_0-2}_x]\left( \mathrm{curl} \mathrm{curl} \bw^i +F^i \right)\cdot\Lambda^{s_0-2}_x(\mathrm{curl} \mathrm{curl} \bw_i + F_i)dxd\tau
  \\
  I_4=&\int^t_0 \int_{\mathbb{R}^3}\Lambda^{s_0-2}_x \partial^i \big( 2   \partial_n v_a \partial^n w^a \big) \cdot \Lambda^{s_0-2}_x(\mathrm{curl} \mathrm{curl} \bw_i + F_i)dxd\tau.
\end{split}
\end{equation}
Let us estimate $I_0,I_1,\ldots,I_4$ one by one. For $I_0$, we have
\begin{equation}\label{I00}
\begin{split}
  |I_0| & \leq \|\mathrm{curl} \mathrm{curl} \bw_0 \|^2_{\dot{H}^{s_0-2}_x}+ C(\|\rho\|_{H^{s_0}_x}\|\bw\|_{H^{\frac32+}_x}+\|\bv\|_{H^{s_0}_x}\|h\|_{H^{\frac52+}_x})
  \\
  & \leq  \|\bw_0 \|^2_{{H}^{s_0}_x} + C(\|\rho_0\|_{H^{s_0}_x}\|\bw_0\|_{H^{\frac32+}_x}+\|\bv_0\|_{H^{s_0}_x}\|h_0\|_{H^{\frac52+}_x}).
\end{split}
\end{equation}
For $I_1$, using \eqref{e002}, we have
\begin{equation}\label{I001}
\begin{split}
  |I_1| \lesssim & \int^t_0 \|\partial \bv\|_{L^\infty_x}\|\mathrm{curl} \mathrm{curl} \bw + \bF\|^2_{\dot{H}^{s_0-2}_x}d\tau
  \\
  \leq & \int^t_0 \|\partial \bv\|_{L^\infty_x}\left\{ \|\bw\|^2_{{H}^{s_0}_x}+\|\bv\|^2_{H^s_x}+\|\rho\|^2_{H^s_x}+\|h\|^2_{H^{s_0+1}_x} \right\}d\tau
  \\
  & + \int^t_0 \|\partial \bv\|_{L^\infty_x} (\|\rho\|^2_{H^{s_0}_x}\|\bw\|^2_{H^{\frac32+}_x}+\|\bv\|^2_{H^{s_0}_x}\|h\|^2_{H^{\frac52+}_x}) d\tau.
\end{split}
\end{equation}
For $I_2$, recalling \eqref{rF} and using Lemma \ref{lpe}, Lemma \ref{wql}, we find that
\begin{equation*}%\label{I002}
\begin{split}
  |{I_2}|
  \lesssim & \int^t_0 \|\bK\|_{ \dot{H}^{s_0-2}_x}  \|\mathrm{curl} \mathrm{curl} \bw + \bF\|_{\dot{H}^{s_0-2}_x}d\tau
  \\
  \leq &  \int^t_0 ( \|\partial \rho \partial h \partial \rho \partial h \|_{ \dot{H}^{s_0-2}_x}
   + \|\partial \bv \partial^2 \bw \|_{ \dot{H}^{s_0-2}_x} ) \|\mathrm{curl} \mathrm{curl} \bw + \bF\|_{\dot{H}^{s_0-2}_x}d\tau
  \\
   & + \int^t_0 ( \|\partial \rho \partial \bv \partial \bw \|_{ \dot{H}^{s_0-2}_x}
  + \|\partial \bw \partial \bw \|_{ \dot{H}^{s_0-2}_x} ) \|\mathrm{curl} \mathrm{curl} \bw + \bF\|_{\dot{H}^{s_0-2}_x}d\tau
  \\
  & + \int^t_0 ( \|\partial \rho \partial \rho \partial \rho \partial h \|_{ \dot{H}^{s_0-2}_x}
  + \|\partial \rho \partial h \partial^2 h \|_{ \dot{H}^{s_0-2}_x} ) \|\mathrm{curl} \mathrm{curl} \bw + \bF\|_{\dot{H}^{s_0-2}_x}d\tau
  \\
  & + \int^t_0 ( \|\partial \rho \partial h \partial^2\rho  \|_{ \dot{H}^{s_0-2}_x}
  + \|\partial \bv \partial \bv \partial^2 h \|_{ \dot{H}^{s_0-2}_x} ) \|\mathrm{curl} \mathrm{curl} \bw + \bF\|_{\dot{H}^{s_0-2}_x}d\tau
  \\
  & + \int^t_0 ( \|\partial \rho \partial^3 h   \|_{ \dot{H}^{s_0-2}_x}
  + \|\partial \rho \partial \rho \partial^2 h \|_{ \dot{H}^{s_0-2}_x} ) \|\mathrm{curl} \mathrm{curl} \bw + \bF\|_{\dot{H}^{s_0-2}_x}d\tau
  \\
  & + \int^t_0 ( \|\partial^2 h \partial^2 h   \|_{ \dot{H}^{s_0-2}_x}
  + \|\partial h \partial \bv \partial^2 \bv \|_{ \dot{H}^{s_0-2}_x} ) \|\mathrm{curl} \mathrm{curl} \bw + \bF\|_{\dot{H}^{s_0-2}_x}d\tau.
\end{split}
\end{equation*}
Therefore, we get
\begin{equation}\label{I002}
	\begin{split}
		|{I_2}|
		\lesssim &  \int^t_0 \| (\partial \bv,\partial \rho,\partial h )\|_{\dot{B}^{s_0-2}_{\infty,2}}( \|\bw\|^2_{{H}^{s_0}_x}+\|\bv\|^2_{H^s_x}+\|\rho\|^2_{H^s_x}+\|h\|^2_{H^{s_0+1}_x} ) d\tau
		\\
		&  +\int^t_0 (\| \partial \bv \|_{\dot{B}^{s_0-2}_{\infty,2}} + \| \partial \rho \|_{\dot{B}^{s_0-2}_{\infty,2}} + \| \partial h \|_{\dot{B}^{s_0-2}_{\infty,2}})(\| \rho\|_{ H^2_x }+\| \bv\|_{ H^2_x }+\| h\|_{ H^2_x } )
		\\
		& \qquad \quad \cdot (\| \bw \|^2_{{H}^{s_0}_x} + \| \rho\|^2_{ H^s_x } +\|\bv\|^2_{H^s_x }+\| h \|^2_{{H}^{s_0+1}_x}) d\tau
		\\
		&  +\int^t_0 (\| \partial \bv \|_{\dot{B}^{s_0-2}_{\infty,2}} + \| \partial \rho \|_{\dot{B}^{s_0-2}_{\infty,2}} + \| \partial h \|_{\dot{B}^{s_0-2}_{\infty,2}})(\| \rho\|^2_{ H^2_x }+\| \bv\|^2_{ H^2_x }+\| h\|^2_{ H^2_x } )
		\\
		& \qquad \quad \cdot (\| \bw \|^2_{{H}^{s_0}_x} + \| \rho\|^2_{ H^s_x } +\|\bv\|^2_{H^s_x }+\| h \|^2_{{H}^{s_0+1}_x}) d\tau
		\\
		&  +\int^t_0 (\| \partial \bv \|_{\dot{B}^{s_0-2}_{\infty,2}} + \| \partial \rho \|_{\dot{B}^{s_0-2}_{\infty,2}} + \| \partial h \|_{\dot{B}^{s_0-2}_{\infty,2}})(\| \rho\|^3_{ H^2_x }+\| \bv\|^3_{ H^2_x }+\| h\|^3_{ H^2_x } )
		\\
		& \qquad \quad \cdot (\| \bw \|^2_{{H}^{s_0}_x} + \| \rho\|^2_{ H^s_x } +\|\bv\|^2_{H^s_x }+\| h \|^2_{{H}^{s_0+1}_x}) d\tau.
	\end{split}
\end{equation}
Above, we set $ \bK=(K^1,K^2,K^3)^{\mathrm{T}}$. For $I_3$, using Lemma \ref{ce}, we obtain
\begin{equation}\label{I03}
\begin{split}
  |I_3| \lesssim & \int^t_0 \|\partial \bv\|_{L^\infty_x}  \|\mathrm{curl} \mathrm{curl} \bw + \bF\|^2_{\dot{H}^{s_0-2}_x}d\tau
  \\
  \leq & \int^t_0 \|\partial \bv\|_{ L^\infty_x } (\|\bw\|^2_{{H}^{s_0}_x}+\|\bv\|^2_{H^s_x}+\|\rho\|^2_{H^s_x}+\|h\|^2_{H^{s_0+1}_x})d\tau
  \\
  & + \int^t_0 \|\partial \bv\|_{ L^\infty_x  }  (\|\rho\|^2_{H^{s_0}_x}\|\bw\|^2_{H^{\frac32+}_x}+\|\bv\|^2_{H^{s_0}_x}\|h\|^2_{H^{\frac52+}_x}) d\tau.
\end{split}
\end{equation}
For $I_4$, we divide it into $I_4=I_{41}+I_{42}+I_{43}+I_{44}$, where
\begin{equation*}
  I_{41}=2\int^t_0\int_{\mathbb{R}^3} \Lambda_x^{s_0-2}\partial^i \big(   \partial_n v_a \partial^n w^a \big) \cdot \Lambda_x^{s_0-2}( \mathrm{curl} \mathrm{curl} \bw_i) dxd\tau, \ \ \ \ \ \ \ \ \ \ \ \ \ \ \ \ \ \ \ \ \ \ \
\end{equation*}
\begin{equation*}
\begin{split}
  I_{42}=&-4\int^t_0\int_{\mathbb{R}^3} \Lambda_x^{s_0-2}\partial^i \big(  \partial_n v_a \partial^n w^a \big) \cdot
  \Lambda_x^{s_0-2}(\mathrm{e}^{-{\rho}}\partial^a \rho \partial_i w_a ) dxd\tau \ \ \ \ \ \ \ \ \ \ \ \ \ \ \ \ \
  \\
  & -2\int^t_0\int_{\mathbb{R}^3} \Lambda_x^{s_0-2}\partial^i \big(   \partial_n v_a \partial^n w^a \big) \cdot \Lambda_x^{s_0-2}( \epsilon^{ijk}\partial_j \rho \mathrm{curl}\bw_k ) dxd\tau,
\end{split}
\end{equation*}
\begin{equation*}
  I_{43}=4\int^t_0\int_{\mathbb{R}^3} \Lambda_x^{s_0-2}\partial^i \big(    \partial_n v_a \partial^n w^a \big) \cdot
  \Lambda_x^{s_0-2}(\mathrm{e}^{-{\rho}}\epsilon_i^{\ jk}\partial_j v^m \partial_m \partial_k h ) dxd\tau, \ \ \ \ \ \ \ \ \
\end{equation*}
\begin{equation*}
  I_{44}=-2\int^t_0\int_{\mathbb{R}^3} \Lambda_x^{s_0-2}\partial^i \big(  \partial_n v_a \partial^n w^a \big) \cdot
  \Lambda_x^{s_0-2}(\mathrm{e}^{-{\rho}}\epsilon_i^{\ jk}\partial_k h \Delta v_j ) dxd\tau, \ \ \ \ \ \ \ \ \ \ \ \
\end{equation*}
Integrating $I_{41}$ by parts, we find that
\begin{equation}\label{I41}
\begin{split}
  I_{41}=&-2\int^t_0\int_{\mathbb{R}^3} \Lambda_x^{s_0-2} \big(   \partial_n v_a \partial^n w^a \big) \cdot \Lambda_x^{s_0-2}\partial^i( \mathrm{curl} \mathrm{curl} \bw_i) dxd\tau
  =  0,
\end{split}
\end{equation}
where we use the fact $\partial^i( \mathrm{curl} \mathrm{curl} \bw_i)=\mathrm{div}(\mathrm{curl} \mathrm{curl} \bw)=0$. Due to the Plancherel formula, it follows that
\begin{equation*}
\begin{split}
  I_{42}=&-4\int^t_0\int_{\mathbb{R}^3} \Lambda_x^{s_0-2}\Lambda_x^{-\frac12}\partial^i \big(   \partial_n v_a \partial^n w^a \big) \cdot
  \Lambda_x^{s_0-2}\Lambda_x^{\frac12}(\partial^a \rho \partial_i w_a ) dxd\tau
  \\
  & -2\int^t_0\int_{\mathbb{R}^3} \Lambda_x^{s_0-2}\Lambda_x^{-\frac12} \partial^i \big(   \partial_n v_a \partial^n w^a \big) \cdot \Lambda_x^{\frac12}\Lambda_x^{s_0-2}( \epsilon^{ijk}\partial_j \rho \mathrm{curl}\bw_k ) dxd\tau.
\end{split}
\end{equation*}
By H\"older's inequality again, we can bound $I_{42}$ as
\begin{equation}\label{I42}
\begin{split}
  |I_{42}| & \leq C\int^t_0 \| \partial \bv \partial \bw \|_{\dot{H}_x^{s_0-\frac32}}\| \partial \rho \partial \bw \|_{\dot{H}_x^{s_0-\frac32}}d\tau
  \\
  & \leq C t  \| \bv\|_{L_t^\infty{H}_x^{s_0}}\| \rho\|_{L_t^\infty{H}_x^{s_0}} \| \bw\|^2_{L_t^\infty{H}_x^{2}}.
\end{split}
\end{equation}
Integrating $I_{43}$ by parts and using $\epsilon_i^{\ jk}\partial^i\partial_j v^m=0$, gives
\begin{equation*}
\begin{split}
  I_{43}=&-4\int^t_0\int_{\mathbb{R}^3} \Lambda^{s_0-2} \big(    \partial_n v_a \partial^n w^a \big) \cdot
  \Lambda^{s_0-2}\partial^i(\mathrm{e}^{-{\rho}}\epsilon_i^{\ jk}\partial_j v^m \partial_m \partial_k h ) dxd\tau
  \\
  =&-4\int^t_0\int_{\mathbb{R}^3} \Lambda^{s_0-2} \big(    \partial_n v_a \partial^n w^a \big) \cdot
  \Lambda^{s_0-2}(-\mathrm{e}^{-{\rho}}\partial^i \rho \epsilon_i^{\ jk}\partial_j v^m \partial_m \partial_k h ) dxd\tau
  \\
  &-4\int^t_0\int_{\mathbb{R}^3} \Lambda^{s_0-2} \big(    \partial_n v_a \partial^n w^a \big) \cdot
  \Lambda^{s_0-2}(\mathrm{e}^{-{\rho}}\epsilon_i^{\ jk}\partial^i\partial_j v^m \partial_m \partial_k h ) dxd\tau
  \\
  &-4\int^t_0\int_{\mathbb{R}^3} \Lambda^{s_0-2} \big(   \partial_n v_a \partial^n w^a \big) \cdot
  \Lambda^{s_0-2}(\mathrm{e}^{-{\rho}}\epsilon_i^{\ jk}\partial_j v^m \partial^i \partial_m \partial_k h ) dxd\tau
  \\
  =&4\int^t_0\int_{\mathbb{R}^3} \Lambda^{s_0-2} \big(   \partial_n v_a \partial^n w^a \big) \cdot
  \Lambda^{s_0-2}(\mathrm{e}^{-{\rho}}\epsilon_i^{\ jk} \partial^i \rho \partial_j v^m \partial_m \partial_k h ) dxd\tau
  \\
  &-4\int^t_0\int_{\mathbb{R}^3} \Lambda^{s_0-2} \big(    \partial_n v_a \partial^n w^a \big) \cdot
  \Lambda^{s_0-2}(\mathrm{e}^{-{\rho}}\epsilon_i^{\ jk}\partial_j v^m \partial^i \partial_m \partial_k h ) dxd\tau.
\end{split}
\end{equation*}
Hence, by using Lemma \ref{lpe}, we can estimate $I_{43}$ by
\begin{equation}\label{I43}
\begin{split}
  | I_{43} | \lesssim & \int^t_0 \| \partial \bv \partial \bw \|_{\dot{H}^{s_0-2}_x} \left(  \|\partial \rho \partial \bv \partial^2h \|_{\dot{H}^{s_0-2}_x} + \|\partial \bv \partial^3 h \|_{\dot{H}^{s_0-2}_x} \right)d\tau
  \\
  \leq &  \int^t_0 \| \partial \bv\|_{\dot{B}^{s_0-2}_{\infty,2}}(\| \rho\|^2_{ H^2_x }+\| \bv\|^2_{ H^2_x } ) (\| \bw \|^2_{{H}^{s_0}_x} + \| \rho\|^2_{ H^s_x } +\|\bv\|^2_{H^s_x }+\| h \|^2_{{H}^{s_0+1}_x}) d\tau
  \\
  &+ \int^t_0 \| \partial \bv\|_{\dot{B}^{s_0-2}_{\infty,2}}(\| \rho\|_{ H^2_x }+\| \bv\|_{ H^2_x } ) (\| \bw \|^2_{{H}^{s_0}_x} + \| \rho\|^2_{ H^s_x } +\|\bv\|^2_{H^s_x }+\| h \|^2_{{H}^{s_0+1}_x}) d\tau.
\end{split}
\end{equation}
Integrating $I_{44}$ by parts, using $\epsilon_i^{\ jk}\partial^i\partial_k h=0$ and $\epsilon_i^{\ jk}\Delta \partial^i v_j=\Delta \bw^k$, we have
\begin{equation}\label{I44i}
\begin{split}
  I_{44}=& 2\int^t_0\int_{\mathbb{R}^3} \Lambda_x^{s_0-2} \big(  \partial_n v_a \partial^n w^a \big) \cdot
  \Lambda_x^{s_0-2}\partial^i(\mathrm{e}^{-{\rho}}\epsilon_i^{\ jk}\partial_k h \Delta v_j ) dxd\tau
  \\
  = & 2\int^t_0\int_{\mathbb{R}^3} \Lambda_x^{s_0-2} \big(  \partial_n v_a \partial^n w^a \big) \cdot
  \Lambda_x^{s_0-2}(-\mathrm{e}^{-{\rho}}\partial^i \rho \epsilon_i^{\ jk}\partial_k h \Delta v_j ) dxd\tau
  \\
  & + 2\int^t_0\int_{\mathbb{R}^3} \Lambda_x^{s_0-2} \big(  \partial_n v_a \partial^n w^a \big) \cdot
  \Lambda_x^{s_0-2}(\mathrm{e}^{-{\rho}}\epsilon_i^{\ jk}\partial^i \partial_k h \Delta v_j ) dxd\tau
  \\
  & + 2\int^t_0\int_{\mathbb{R}^3} \Lambda_x^{s_0-2} \big(  \partial_n v_a \partial^n w^a \big) \cdot
  \Lambda_x^{s_0-2}(\mathrm{e}^{-{\rho}}\epsilon_i^{\ jk}\partial_k h \Delta \partial^i v_j ) dxd\tau
  \\
  = & -2\int^t_0\int_{\mathbb{R}^3} \Lambda_x^{s_0-2} \big(  \partial_n v_a \partial^n w^a \big) \cdot
  \Lambda_x^{s_0-2}(\mathrm{e}^{-{\rho}}\epsilon_i^{\ jk} \partial^i \rho \partial_k h \Delta v_j ) dxd\tau
  \\
  & + 2\int^t_0\int_{\mathbb{R}^3} \Lambda_x^{s_0-2} \big(  \partial_n v_a \partial^n w^a \big) \cdot
  \Lambda_x^{s_0-2}(\mathrm{e}^{-{\rho}}\partial_k h \Delta \bw_k ) dxd\tau.
\end{split}
\end{equation}
To give a bound for $I_{44}$, let us calculate $\Lambda_x^{s_0-2}(\mathrm{e}^{-{\rho}}\epsilon_i^{\ jk} \partial^i \rho \partial_k h \Delta v_j )$. The Hodge decomposition implies that
\begin{equation}\label{bbb}
  \Delta v_j= \epsilon_{jab}\partial^a \mathrm{curl}\bv^b +\partial_j \mathrm{div} \bv= \epsilon_{jab}\mathrm{e}^{\rho}\partial^a \rho w^b + \epsilon_{jab}\mathrm{e}^{\rho} \partial^a w^b +\partial_j \mathrm{div} \bv.
\end{equation}
Due to \eqref{bbb}, we obtain
\begin{equation}\label{I441}
\begin{split}
   \mathrm{e}^{-{\rho}}\epsilon_i^{\ jk} \partial^i \rho \partial_k h \Delta v_j
  =&  \epsilon_i^{\ jk}\epsilon_{jab} \partial^i \rho \partial_k h \partial^a \rho w^b
   + \epsilon_i^{\ jk}\epsilon_{jab} \partial^i \rho \partial_k h \partial^a w^b
   \\
   & + \mathrm{e}^{-{\rho}}\epsilon_i^{\ jk} \partial^i \rho \partial_k h \partial_j \mathrm{div} \bv
   \\
    =& \epsilon_i^{\ jk}\epsilon_{jab} \partial^i \rho \partial_k h \partial^a \rho w^b
   +\epsilon_i^{\ jk}\epsilon_{jab} \partial^i \rho \partial_k h \partial^a w^b
   \\
   & + \partial_j (\mathrm{e}^{-{\rho}}\epsilon_i^{\ jk} \partial^i \rho \partial_k h  \mathrm{div} \bv )
   -  (-\mathrm{e}^{-{\rho}}\epsilon_i^{\ jk}\partial_j \rho \partial^i \rho \partial_k h  \mathrm{div} \bv )
   \\
   & -(\mathrm{e}^{-{\rho}}\epsilon_i^{\ jk}  \partial_j\partial^i \rho \partial_k h  \mathrm{div} \bv )
   - (\mathrm{e}^{-{\rho}}\epsilon_i^{\ jk}  \partial^i \rho \partial_j \partial_k h  \mathrm{div} \bv )
   \\
    =& \epsilon_i^{\ jk}\epsilon_{jab} \partial^i \rho \partial_k h \partial^a \rho w^b
   + \epsilon_i^{\ jk}\epsilon_{jab} \partial^i \rho \partial_k h \partial^a w^b
   \\
   & + \partial_j (\mathrm{e}^{-{\rho}}\epsilon_i^{\ jk} \partial^i \rho \partial_k h  \mathrm{div} \bv ).
\end{split}
\end{equation}
Using \eqref{I441} in \eqref{I44i}, we have
\begin{equation}\label{I440}
\begin{split}
  I_{44}=
   & -2\int^t_0\int_{\mathbb{R}^3} \Lambda_x^{s_0-2} \big(  \partial_n v_a \partial^n w^a \big) \cdot \Lambda_x^{s_0-2} \big(\epsilon_i^{\ jk}\epsilon_{jab} \partial^i \rho \partial_k h \partial^a \rho w^b  \big) dxd\tau
   \\
   & -2\int^t_0\int_{\mathbb{R}^3} \Lambda_x^{s_0-2} \big(  \partial_n v_a \partial^n w^a \big) \cdot \Lambda_x^{s_0-2} \big(\epsilon_i^{\ jk}\epsilon_{jab} \partial^i \rho \partial_k h \partial^a w^b  \big) dxd\tau
   \\
   & -2\int^t_0\int_{\mathbb{R}^3} \Lambda_x^{s_0-2} \big(  \partial_n v_a \partial^n w^a \big) \cdot \Lambda_x^{s_0-2} \left\{ \partial_j (\mathrm{e}^{-{\rho}}\epsilon_i^{\ jk} \partial^i \rho \partial_k h  \mathrm{div} \bv )  \right\} dxd\tau
   \\
   & + 2\int^t_0\int_{\mathbb{R}^3} \Lambda_x^{s_0-2} \big(  \partial_n v_a \partial^n w^a \big) \cdot
   \Lambda_x^{s_0-2}(\mathrm{e}^{-{\rho}}\partial_k h \Delta \bw_k ) dxd\tau.
\end{split}
\end{equation}
Using the Plancherel formula and the H\"older inequality, we can bound $I_{44}$ by
\begin{equation}\label{I144}
\begin{split}
  | I_{44} | \lesssim & \int^t_0 \| \partial \bv \partial \bw\|_{\dot{H}^{s_0-2}_x}   \| \partial \rho \partial h \partial \rho  \bw\|_{\dot{H}^{s_0-2}_x} d\tau
   + \int^t_0 \| \partial \bv \partial \bw\|_{\dot{H}^{s_0-2}_x}    \| \partial \rho \partial h \partial \bw\|_{\dot{H}^{s_0-2}_x} d\tau
  \\
  &+  \int^t_0 \| \partial \bv \partial \bw\|_{\dot{H}^{s_0-\frac{3}{2}}_x}   \| \partial \rho \partial h \partial \bv \|_{\dot{H}^{s_0-\frac{3}{2}}_x} d\tau
   + \int^t_0 \| \partial \bv \partial \bw\|_{\dot{H}^{s_0-2}_x}   \| \partial h \partial^2 \bw \|_{\dot{H}^{s_0-2}_x} d\tau
  \\
  \leq & \int^t_0 \| \partial \bv,\partial \rho,\partial h\|_{L^\infty_x}\| \bv,\rho,h \|^3_{{H}^{2}_x}
   \cdot (\| \bw \|^2_{{H}^{s_0}_x} +\| \bv \|^2_{{H}^{s}_x} +\| \rho \|^2_{{H}^{s}_x} +\| h \|^2_{{H}^{s_0+1}_x}  )  d\tau
  \\
   & + \int^t_0 \| \partial \bv,\partial \rho,\partial h\|_{L^\infty_x}\| \bv,\rho,h \|^2_{{H}^{2}_x}
   \cdot (\| \bw \|^2_{{H}^{s_0}_x} +\| \bv \|^2_{{H}^{s}_x} +\| \rho \|^2_{{H}^{s}_x} +\| h \|^2_{{H}^{s_0+1}_x}  )  d\tau
   \\
   & + \int^t_0 \| \partial \bv,\partial \rho,\partial h\|_{\dot{B}^{s_0-2}_{\infty,2}}\| \bv,\rho,h \|_{{H}^{2}_x}
  \cdot (\| \bw \|^2_{{H}^{s_0}_x} +\| \bv \|^2_{{H}^{s}_x} +\| \rho \|^2_{{H}^{s}_x} +\| h \|^2_{{H}^{s_0+1}_x}  )  d\tau.
\end{split}
\end{equation}
Combining \eqref{I41} to \eqref{I144}, we get
\begin{equation}\label{I004}
\begin{split}
  |I_4| \leq
  & C\int^t_0 (\|(\partial \bv, \partial \rho,\partial h)\|_{L^\infty_x})\| (\bv ,\rho,h) \|^3_{{H}^{2}_x}
    \cdot (\| \bw \|^2_{{H}^{s_0}_x} +\|( \bv ,\rho )\|^2_{{H}^{s}_x} +\| h \|^2_{{H}^{s_0+1}_x}  )  d\tau
  \\
   & + C\int^t_0 (\|(\partial \bv, \partial \rho,\partial h)\|_{L^\infty_x})\| (\bv ,\rho,h) \|^2_{{H}^{2}_x}
   (\| \bw \|^2_{{H}^{s_0}_x} +\|( \bv ,\rho )\|^2_{{H}^{s}_x} +\| h \|^2_{{H}^{s_0+1}_x}  )  d\tau
   \\
   & + C t  \| \bv\|_{L_t^\infty{H}_x^{s_0}}\| \rho\|_{L_t^\infty{H}_x^{s_0}} \| \bw\|^2_{L_t^\infty{H}_x^{2}} .
\end{split}
\end{equation}
On the other hand, we have
\begin{equation}\label{W33}
\begin{split}
  \|\mathrm{curl} \mathrm{curl} \bw \|^2_{\dot{H}^{s_0-2}_x} \leq & 2\|\mathrm{curl} \mathrm{curl} \bw + \bF\|^2_{\dot{H}^{s_0-2}_x} + 2\|\bF\|^2_{\dot{H}^{s_0-2}_x}.
\end{split}
\end{equation}
Combining \eqref{W33}, \eqref{e002a}, \eqref{I001}-\eqref{I004}, we can conclude that
\begin{small}
\begin{equation}\label{e003}
\begin{split}
   & \|\mathrm{curl} \mathrm{curl} \bw \|^2_{\dot{H}^{s_0-2}_x}(t)
   \\
  \leq &  C(\|\rho\|^2_{L^\infty_tH^{s_0}_x}\|\bw\|^2_{L^\infty_t H^{\frac32+}_x}+\|\bv\|^2_{L^\infty_t H^{s_0}_x}\|h\|^2_{L^\infty_t H^{\frac52+}_x})
  \\
  &+ C\|\bw_0 \|^2_{{H}^{s_0}_x}+  C t  \| \bv\|_{L_t^\infty{H}_x^{s_0}}\| \rho\|_{L_t^\infty{H}_x^{s_0}} \| \bw\|^2_{L_t^\infty{H}_x^{2}}
  \\
  & +C\int^t_0 \|\partial \bv\|_{L^\infty_x}( \|\bw\|^2_{{H}^{s_0}_x}+\|\bv\|^2_{H^s_x}+\|\rho\|^2_{H^s_x}+\|h\|^2_{H^{s_0+1}_x} ) d\tau
  \\
  & + C\int^t_0 \|\partial \bv\|_{L^\infty_x} (\|\rho\|^2_{H^{s_0}_x}\|\bw\|^2_{H^{\frac32+}_x}+\|\bv\|^2_{H^{s_0}_x}\|h\|^2_{H^{\frac52+}_x}) d\tau
  \\
  & + C\int^t_0 \| (\partial \bv , \partial \rho , \partial h) \|_{\dot{B}^{s_0-2}_{\infty,2}}( \|\bw\|^2_{{H}^{s_0}_x}+\|\bv\|^2_{H^s_x}+\|\rho\|^2_{H^s_x}+\|h\|^2_{H^{s_0+1}_x} ) d\tau
  \\
  &  +C\int^t_0 \| (\partial \bv , \partial \rho , \partial h) \|_{\dot{B}^{s_0-2}_{\infty,2}}\| (\rho, \bv, h)\|_{ H^2_x } (\| \bw \|^2_{{H}^{s_0}_x} + \| \rho\|^2_{ H^s_x } +\|\bv\|^2_{H^s_x }+\| h \|^2_{{H}^{s_0+1}_x}) d\tau
  \\
  &  +C\int^t_0 \| (\partial \bv , \partial \rho , \partial h) \|_{\dot{B}^{s_0-2}_{\infty,2}}\| (\rho, \bv, h)\|^2_{ H^2_x }
   (\| \bw \|^2_{{H}^{s_0}_x} + \| \rho\|^2_{ H^s_x } +\|\bv\|^2_{H^s_x }+\| h \|^2_{{H}^{s_0+1}_x}) d\tau
  \\
  &  +C\int^t_0 \| ( \partial \bv , \partial \rho , \partial h ) \|_{\dot{B}^{s_0-2}_{\infty,2}} \| ( \rho, \bv, h) \|^3_{ H^2_x }
  (\| \bw \|^2_{{H}^{s_0}_x} + \| \rho\|^2_{ H^s_x } +\|\bv\|^2_{H^s_x }+\| h \|^2_{{H}^{s_0+1}_x}) d\tau .
\end{split}
\end{equation}
\end{small}
Adding \eqref{e003} with \eqref{DW}, we obtain
\begin{equation}\label{e004}
\begin{split}
   & \|\bw \|^2_{\dot{H}^{s_0}_x}(t)
   \\
   \leq & C\|\bw_0 \|^2_{{H}^{s_0}_x}+ C t  \| \bv\|_{L_t^\infty{H}_x^{s_0}}\| \rho\|_{L_t^\infty{H}_x^{s_0}} \| \bw\|^2_{L_t^\infty{H}_x^{2}}
  \\
  & + C(\|\rho\|^2_{L^\infty_tH^{s_0}_x}\|\bw\|^2_{L^\infty_t H^{\frac32+}_x}+\|\bv\|^2_{L^\infty_t H^{s_0}_x}\|h\|^2_{L^\infty_t H^{\frac52+}_x})
  \\
  & +C\int^t_0 \|\partial \bv\|_{L^\infty_x}( \|\bw\|^2_{{H}^{s_0}_x}+\|\bv\|^2_{H^s_x}+\|\rho\|^2_{H^s_x}+\|h\|^2_{H^{s_0+1}_x} ) d\tau
  \\
  & + C\int^t_0 \|\partial \bv\|_{L^\infty_x} (\|\rho\|^2_{H^{s_0}_x}\|\bw\|^2_{H^{\frac32+}_x}+\|\bv\|^2_{H^{s_0}_x}\|h\|^2_{H^{\frac52+}_x}) d\tau
  \\
  & + C\int^t_0 \| (\partial \bv, \partial \rho, \partial h) \|_{\dot{B}^{s_0-2}_{\infty,2}}( \|\bw\|^2_{{H}^{s_0}_x}+\|(\bv,\rho)\|^2_{H^s_x}+\|h\|^2_{H^{s_0+1}_x} ) d\tau
  \\
  &  +C\int^t_0 \| (\partial \bv , \partial \rho , \partial h) \|_{\dot{B}^{s_0-2}_{\infty,2}}\| (\rho, \bv, h)\|_{ H^2_x } (\| \bw \|^2_{{H}^{s_0}_x} + \| (\rho,\bv)\|^2_{H^s_x }+\| h \|^2_{{H}^{s_0+1}_x}) d\tau
  \\
  &  +C\int^t_0 \| (\partial \bv , \partial \rho , \partial h) \|_{\dot{B}^{s_0-2}_{\infty,2}}\| (\rho, \bv, h)\|^2_{ H^2_x }
   (\| \bw \|^2_{{H}^{s_0}_x} + \| (\rho,\bv)\|^2_{H^s_x }+\| h \|^2_{{H}^{s_0+1}_x}) d\tau
  \\
  &  +C\int^t_0 \| ( \partial \bv , \partial \rho , \partial h ) \|_{\dot{B}^{s_0-2}_{\infty,2}} \| ( \rho, \bv, h) \|^3_{ H^2_x }
  (\| \bw \|^2_{{H}^{s_0}_x} + \| (\rho,\bv)\|^2_{H^s_x }+\| h \|^2_{{H}^{s_0+1}_x}) d\tau .
\end{split}
\end{equation}
By interpolation, we have
\begin{equation}\label{DW1}
\begin{split}
  \|  \bw \|_{H_x^{\frac32+}} \lesssim & \|  \bw \|^{\frac{1}{2(s_0-1)}}_{L_x^2}\| \bw \|^{1-\frac{1}{2(s_0-1)}}_{H^{s_0}_x},
  \\
  \|  \bw \|_{H_x^{2}} \lesssim & \|  \bw \|^{\frac{s_0-2}{s_0-1}}_{H_x^1}\| \bw \|^{\frac{1}{s_0-1}}_{H^{s_0}_x},
  \\
  \|  h \|_{H_x^{\frac52+}} \lesssim & \|  h \|^{\frac{1}{2(s_0-1)}}_{H_x^2}\| h \|^{1-\frac{1}{2(s_0-1)}}_{H^{s_0+1}_x}.
\end{split}
\end{equation}
By using \eqref{DW1} and Young's inequality, we can derive that\footnote{In \eqref{E005}, we get the highest-derivatives norm $\frac{1}{100}\| \bw\|^{2}_{L_t^\infty{H}_x^{s_0}}$ for it will be canceled by energy term $\| \bw\|^{2}_{L_t^\infty{H}_x^{s_0}}$.}
\begin{equation}\label{E005}
	\begin{split}
		C t  \| \bv\|_{L_t^\infty{H}_x^{s_0}}\| \rho\|_{L_t^\infty{H}_x^{s_0}} \| \bw\|^2_{L_t^\infty{H}_x^{2}} \leq & 	C t  \| \bv\|_{L_t^\infty{H}_x^{s_0}}\| \rho\|_{L_t^\infty{H}_x^{s_0}} \| \bw\|^{\frac{2(s_0-2)}{s_0-1}}_{L_t^\infty{H}_x^{1}}\| \bw\|^{\frac{2}{s_0-1}}_{L_t^\infty{H}_x^{s_0}}
		\\
		\leq & C t^{1+\frac{1}{s_0-2}}  \| \bv\|^{2+\frac{s_0-1}{s_0-2}}_{L_t^\infty{H}_x^{s_0}}\| \rho\|^{\frac{s_0-1}{s_0-2}}_{L_t^\infty{H}_x^{s_0}}+ \frac{1}{100} \| \bw\|^{2}_{L_t^\infty{H}_x^{s_0}}.
	\end{split}
\end{equation}
In a similar way, we find that
\begin{equation}\label{e005}
\begin{split}
   \|\bv\|^2_{L^\infty_t H^{s_0}_x}\|h\|^2_{L^\infty_t H^{\frac52+}_x}
   \leq & C  \|  \bv \|^{2}_{H_x^{s_0}} \|  h \|^{\frac{1}{s_0-1}}_{H_x^2}  \| h \|^{2-\frac{1}{s_0-1}}_{H^{s_0+1}_x}
   \\
   \leq & C  \|  \bv \|^{4(s_0-1)}_{H_x^{s_0}} \|  h \|^{2}_{H_x^2} + \frac{1}{100}\| h \|^{2}_{H^{s_0+1}_x},
\end{split}
\end{equation}
and
\begin{equation}\label{ee005}
		 \|\rho\|^2_{L^\infty_tH^{s_0}_x}\|\bw\|^2_{L^\infty_t H^{\frac32+}_x}
		\leq  C  \|  \rho \|^{4(s_0-1)}_{H_x^{s_0}} \|  \bv \|^{2}_{H_x^{s_0}} + \frac{1}{100}\| \bw \|^{2}_{H^{s_0}_x} .
\end{equation}
Using \eqref{E005}-\eqref{ee005},  \eqref{e004} yields, upon canceling the small terms with factor $1/100$ from \eqref{E005}, \eqref{ee005},
\begin{small}
\begin{equation}\label{e0004}
	\begin{split}
		 \|\bw \|^2_{\dot{H}^{s_0}_x}(t)
		\leq & C\|\bw_0 \|^2_{{H}^{s_0}_x}+ C t^{1+\frac{1}{s_0-2}}  \| \bv\|^{2+\frac{s_0-1}{s_0-2}}_{L_t^\infty{H}_x^{s_0}}\| \rho\|^{\frac{s_0-1}{s_0-2}}_{L_t^\infty{H}_x^{s_0}} +  \frac{1}{50}\| \bw \|^{2}_{H^{s_0}_x}
		\\
		& + C  \|  \bv \|^{4(s_0-1)}_{L_t^\infty H_x^{s_0}} \|  h \|^{2}_{L_t^\infty H_x^2}+C  \|  \rho \|^{4(s_0-1)}_{L_t^\infty H_x^{s_0}} \|  \bv \|^{2}_{L_t^\infty H_x^{s_0}}  + \frac{1}{100}\| h \|^{2}_{H^{s_0+1}_x}
		\\
		& +C\int^t_0 \|\partial \bv\|_{L^\infty_x}( \|\bw\|^2_{{H}^{s_0}_x}+\|(\bv,\rho)\|^2_{H^s_x}+\|h\|^2_{H^{s_0+1}_x} ) d\tau
		\\
		& + C\int^t_0 \|\partial \bv\|_{L^\infty_x} (\|\rho\|^2_{H^{s_0}_x}\|\bw\|^2_{H^{\frac32+}_x}+\|\bv\|^2_{H^{s_0}_x}\|h\|^2_{H^{\frac52+}_x}) d\tau
		\\
		& + C\int^t_0 \| (\partial \bv, \partial \rho, \partial h) \|_{\dot{B}^{s_0-2}_{\infty,2}}( \|\bw\|^2_{{H}^{s_0}_x}+\|(\bv,\rho)\|^2_{H^s_x}+\|h\|^2_{H^{s_0+1}_x} ) d\tau
		\\
		&  +C\int^t_0 \| (\partial \bv , \partial \rho , \partial h) \|_{\dot{B}^{s_0-2}_{\infty,2}}\| (\rho, \bv, h)\|_{ H^2_x } (\| \bw \|^2_{{H}^{s_0}_x} + \| (\rho,\bv)\|^2_{H^s_x }+\| h \|^2_{{H}^{s_0+1}_x}) d\tau
		\\
		&  +C\int^t_0 \| (\partial \bv , \partial \rho , \partial h) \|_{\dot{B}^{s_0-2}_{\infty,2}}\| (\rho, \bv, h)\|^2_{ H^2_x }
		(\| \bw \|^2_{{H}^{s_0}_x} + \| (\rho,\bv)\|^2_{H^s_x }+\| h \|^2_{{H}^{s_0+1}_x}) d\tau
		\\
		&  +C\int^t_0 \| ( \partial \bv , \partial \rho , \partial h ) \|_{\dot{B}^{s_0-2}_{\infty,2}} \| ( \rho, \bv, h) \|^3_{ H^2_x }
		(\| \bw \|^2_{{H}^{s_0}_x} + \| (\rho,\bv)\|^2_{H^s_x }+\| h \|^2_{{H}^{s_0+1}_x}) d\tau .
	\end{split}
\end{equation}
\end{small}
\textbf{Step 2: The estimate for $h$}.

By using the equation \eqref{fc0}, we can see that
\begin{equation*}
  \partial_t h+ (\bv \cdot \nabla)h=0.
\end{equation*}
This gives the $L^2$ estimate
\begin{equation}\label{e006}
  \|h(t)\|^2_{L^{2}_x}  \leq  \| h_0 \|^2_{L^{2}}+ C \int^t_0 \|d \bv\|_{L^\infty_x}  \|h\|^2_{L^{2}_x} d\tau.
\end{equation}
The Gronwall's inequality tells us
\begin{equation}\label{He1}
  \|h(t)\|_{L^{2}_x}  \leq  C\| h_0 \|_{L^{2}}\exp \big( {\int^t_0} \|d \bv\|_{L^\infty_x}  d\tau \big).
\end{equation}
For the high-order energy, noticing
\begin{equation*}
  \|\Delta h\|_{\dot{H}^{s_0-1}_x}= \|\mathrm{e}^\rho \partial_i H^i + \partial_i \rho \partial^i h \|_{\dot{H}^{s_0-1}_x},
\end{equation*}
we have
\begin{equation}\label{He2}
  \|\Delta h\|^2_{\dot{H}^{s_0-1}_x} \leq  \|\mathrm{e}^\rho \partial_i H^i\|^2_{\dot{H}^{s_0-1}_x}+ \|\partial_i \rho \partial^i h \|^2_{\dot{H}^{s_0-1}_x}+ 2\|\mathrm{e}^\rho \partial_i H^i\|_{\dot{H}^{s_0-1}_x} \|\partial_i \rho \partial^i h \|_{\dot{H}^{s_0-1}_x}.
\end{equation}
Applying  $\Lambda_x^{s_0-2}$ to \eqref{Hh2}, we get
\begin{equation}\label{He5}
  \mathbf{T} \left \{ \Lambda_x^{s_0-2} \partial^k ( \mathrm{e}^\rho \partial_i H^i) \right\}=-2 \Lambda_x^{s_0-2}[\partial^k(\partial_m v^j) \partial_j  H^m]+\Lambda_x^{s_0-2}Y-[\Lambda_x^{s_0-2},\mathbf{T}]\partial^k ( \mathrm{e}^\rho \partial_i H^i).
\end{equation}
Multiplying \eqref{He5} by $\Lambda^{s_0-2} \partial_k ( \mathrm{e}^\rho \partial_i H^i)$ integrating over $[0,t]\times \mathbb{R}^3$, we have
\begin{equation}\label{He6}
  \begin{split}
   \| ( \mathrm{e}^\rho \partial_i H^i) \|^2_{\dot{H}^{s_0-1}_x}(t)=& \| ( \mathrm{e}^\rho \partial_i H^i) \|^2_{\dot{H}^{s_0-1}_x}(0)+\int^t_0\int_{\mathbb{R}^3} \mathrm{div}\bv \cdot |\Lambda_x^{s_0-2} \nabla ( \mathrm{e}^\rho \partial_i H^i) |^2dxd\tau
   \\
   &+ \int^t_0\int_{\mathbb{R}^3} \Lambda_x^{s_0-2}Y^k \cdot \Lambda_x^{s_0-2} \partial_k ( \mathrm{e}^\rho \partial_i H^i) dxd\tau
  \\
  &-\int^t_0\int_{\mathbb{R}^3} [\Lambda_x^{s_0-2},\mathbf{T}]\partial^k ( \mathrm{e}^\rho \partial_i H^i)\cdot \Lambda_x^{s_0-2} \partial_k ( \mathrm{e}^\rho \partial_i H^i) dxd\tau  \\
  &- 2 \int^t_0\int_{\mathbb{R}^3} \Lambda_x^{s_0-2}[\partial^k(\partial_m v^j) \partial_j  H^m]\cdot \Lambda_x^{s_0-2} \partial_k ( \mathrm{e}^\rho \partial_i H^i) dxd\tau
  \\
  =&\widetilde{I}_0+\widetilde{I}_1+\widetilde{I}_2+\widetilde{I}_3+\widetilde{I}_4,
  \end{split}
\end{equation}
where
\begin{equation*}
\begin{split}
\widetilde{I}_0=&\| ( \mathrm{e}^\rho \partial_i H^i) \|^2_{\dot{H}^{s_0-1}_x}(0),
\\
  \widetilde{I}_1=&\int^t_0\int_{\mathbb{R}^3} \mathrm{div}\bv \cdot |\Lambda_x^{s_0-2} \nabla ( \mathrm{e}^\rho \partial_i H^i) |^2dxd\tau,
  \\
  \widetilde{I}_2= &\int^t_0\int_{\mathbb{R}^3} \Lambda_x^{s_0-2}Y^k \cdot \Lambda_x^{s_0-2} \partial_k ( \mathrm{e}^\rho \partial_i H^i) dxd\tau,
  \\
  \widetilde{I}_3=&-\int^t_0\int_{\mathbb{R}^3} [\Lambda_x^{s_0-2},\mathbf{T}]\partial^k ( \mathrm{e}^\rho \partial_i H^i)\cdot \Lambda_x^{s_0-2} \partial_k ( \mathrm{e}^\rho \partial_i H^i) dxd\tau,
  \\
  \widetilde{I}_4=&- 2 \int^t_0 \int_{\mathbb{R}^3} \Lambda_x^{s_0-2}[\partial^k(\partial_m v^j) \partial_j  H^m]\cdot \Lambda_x^{s_0-2} \partial_k ( \mathrm{e}^\rho \partial_i H^i) dxd\tau.
\end{split}
\end{equation*}
Let us estimate $\widetilde{I}_0, \widetilde{I}_1, \ldots, \widetilde{I}_4$ one by one. Note
\begin{equation}\label{nn}
  \mathrm{e}^\rho \partial_i H^i=\Delta h- \partial_i \rho \partial^i h.
\end{equation}
By using H\"older's inequality and \eqref{nn}, we can estimate the term $I_0$ by
\begin{equation}\label{sI00}
  |\tilde{I}_0| \leq  \| h_0 \|^2_{H^{s_0+1}}+ C\| h_0 \|^2_{H^{\frac52+}}\| \rho_0 \|^2_{H^{s_0}} .
\end{equation}
For the term $\widetilde{I}_1$, we can estimate it by
\begin{equation}\label{I01}
\begin{split}
  |\widetilde{I}_1|
  \leq & C\int^t_0 \|\partial \bv \|_{L^\infty_x}\| \Delta h- \partial \rho \partial h\|^2_{\dot{H}^{s_0-1}_x}d\tau
  \\
  \leq & C\int^t_0 \|\partial \bv \|_{L^\infty_x} \|h\|^2_{{H}^{s_0+1}_x} d\tau+C\int^t_0 \|\partial \bv \|_{L^\infty_x} \| h \|^2_{H^{\frac52+}}\| \rho \|^2_{H^{s_0}} d\tau.
\end{split}
\end{equation}
By using Lemma \ref{lpe}, Lemma \ref{wql} and H\"older's inequality, we have
\begin{equation}\label{I02}
\begin{split}
  |\widetilde{I}_2| \lesssim & \int^t_0 \|\partial \bv \|_{\dot{B}^{s_0-2}_{\infty,2}}( \|h\|^2_{H^{s_0+1}_x}+ \|\bw\|^2_{H^{s_0}_x}+ \|\rho\|^2_{H^{s}_x}+ \|\bv\|^2_{H^{s}_x})d\tau
  \\
  &+\int^t_0 \|(\partial \bv, \partial \rho, \partial h) \|_{\dot{B}^{s_0-2}_{\infty,2}}\|( h,\rho,\bv)\|_{H^{2}_x}( \|h\|^2_{H^{s_0+1}_x}+ \|\bw\|^2_{H^{s_0}_x}+ \| (\rho,\bv) \|^2_{H^{s}_x})d\tau
  \\
  &+\int^t_0 \|(\partial \bv, \partial \rho, \partial h) \|_{\dot{B}^{s_0-2}_{\infty,2}} \|( h,\rho,\bv) \|^2_{H^{s_0}_x}( \|h\|^2_{H^{s_0+1}_x}+ \|\bw\|^2_{H^{s_0}_x}+ \|(\rho, \bv)\|^2_{H^{s}_x})d\tau
  \\
  & + \int^t_0 \|(\partial \bv, \partial \rho, \partial h) \|_{\dot{B}^{s_0-2}_{\infty,2}} \|( h,\rho, \bv)\|^3_{H^{2}_x}( \|h\|^2_{H^{s_0+1}_x}+ \|\bw\|^2_{H^{s_0}_x}+ \|(\rho,\bv)\|^2_{H^{s}_x})d\tau.
\end{split}
\end{equation}
For $\widetilde{I}_3$, using Lemma \ref{ce}, we have
\begin{equation}\label{sI03}
  |\widetilde{I}_3| \lesssim \int^t_0 \|\partial \bv \|_{\dot{B}^{0}_{\infty,2}} \|h\|^2_{{H}^{s_0+1}_x} d\tau + \int^t_0 \|\partial \bv \|_{\dot{B}^{0}_{\infty,2}} \| h \|^2_{H^{\frac52+}}\| \rho \|^2_{H^{s_0}} d\tau.
\end{equation}
It remains for us to handle the difficult term $\widetilde{I}_4$, which is the most difficult term. To estimate it, let us integrate it by parts as follows:
\begin{equation}\label{I08}
  \begin{split}
  \widetilde{I}_4=&- 2 \int^t_0 \int_{\mathbb{R}^3} \Lambda^{s_0-2}[\partial^k(\partial_m v^j) \partial_j  \partial^m h]\cdot \Lambda^{s_0-2} \partial_k ( \mathrm{e}^\rho \partial_i H^i) dx d\tau
  \\
  =&  2 \int^t_0 \int_{\mathbb{R}^3} \Lambda^{s_0-2}[\partial^k v^j \partial_m(\partial_j  \partial^m h)]\cdot \Lambda^{s_0-2} \partial_k ( \mathrm{e}^\rho \partial_i H^i) dx d\tau
  \\
  & + 2 \int^t_0 \int_{\mathbb{R}^3} \Lambda^{s_0-2}(\partial^k v^j \partial_j  \partial^m h )\cdot \Lambda^{s_0-2} \partial_m \partial_k ( \mathrm{e}^\rho \partial_i H^i) dx d\tau
  \\
  =& \widetilde{I}_{40}+\widetilde{I}_{41},
  \end{split}
\end{equation}
where
\begin{equation*}
  \begin{split}
  \widetilde{I}_{40}=& 2 \int^t_0 \int_{\mathbb{R}^3} \Lambda^{s_0-2}[\partial^k v^j \partial_m(\partial_j  \partial^m h)]\cdot \Lambda^{s_0-2} \partial_k ( \mathrm{e}^\rho \partial_i H^i) dx d\tau,
  \\
  \widetilde{I}_{41}=&2 \int^t_0 \int_{\mathbb{R}^3} \Lambda^{s_0-2}(\partial^k v^j \partial_j  \partial^m h )\cdot \Lambda^{s_0-2} \partial_m \partial_k ( \mathrm{e}^\rho \partial_i H^i) dx d\tau.
  \end{split}
\end{equation*}
By Lemma \ref{lpe} and \eqref{nn}, it's easy for us to obtain
\begin{equation}\label{I40}
  |\widetilde{I}_{40}| \lesssim \int^t_0 \|\partial \bv \|_{\dot{B}^{s_0-2}_{\infty,2}}\|  h \|^2_{{H}^{s_0+1}_x} d\tau+ \int^t_0 \|\partial \bv \|_{\dot{B}^{0}_{\infty,2}} \| h \|^2_{H^{\frac52+}}\| \rho \|^2_{H^{s_0}} d\tau.
\end{equation}
However, for $\widetilde{I}_{41}$, we may not estimate it directly. We then insert a term to transfer some derivatives.
\begin{equation}\label{I09}
  \begin{split}
  \widetilde{I}_{41}=&2 \int^t_0 \int_{\mathbb{R}^3} \Lambda^{s_0-2}\{ ( \partial^k v^j- \partial^j v^k) \partial_j  \partial^m h \} \cdot \Lambda^{s_0-2} \partial_m \partial_k ( \mathrm{e}^\rho \partial_i H^i) dx  d\tau
  \\
  & -2 \int^t_0  \int_{\mathbb{R}^3} \Lambda^{s_0-2}(  \partial^j v^k \partial_j  \partial^m h ) \cdot \Lambda^{s_0-2} \partial_m \partial_k ( \mathrm{e}^\rho \partial_i H^i) dx  d\tau.
  \\
  =&2 \int^t_0  \int_{\mathbb{R}^3} \Lambda^{s_0-2}\{ \epsilon^{lkj}\mathrm{e}^{\rho} w_l \partial_j  \partial^m h \} \cdot \Lambda^{s_0-2} \partial_m \partial_k ( \mathrm{e}^\rho \partial_i H^i) dx  d\tau
  \\
  & -2 \int^t_0  \int_{\mathbb{R}^3} \Lambda^{s_0-2}(  \partial^j v^k \partial_j  \partial^m h ) \cdot \Lambda^{s_0-2} \partial_m \partial_k ( \mathrm{e}^\rho \partial_i H^i) dx  d\tau
  \\
  =& \widetilde{I}_{42}+\widetilde{I}_{43}.
  \end{split}
\end{equation}
Above, we set
\begin{equation*}
\begin{split}
  \widetilde{I}_{42}=&2 \int^t_0 \int_{\mathbb{R}^3} \Lambda^{s_0-2}\{ \epsilon^{lkj}\mathrm{e}^{\rho} w_l \partial_j  \partial^m h \} \cdot \Lambda^{s_0-2} \partial_m \partial_k ( \mathrm{e}^\rho \partial_i H^i) dx  d\tau,
  \\
  \widetilde{I}_{43}=&-2 \int^t_0 \int_{\mathbb{R}^3} \Lambda^{s_0-2}(  \partial^j v^k \partial_j  \partial^m h ) \cdot \Lambda^{s_0-2} \partial_m \partial_k ( \mathrm{e}^\rho \partial_i H^i) dx  d\tau.
\end{split}
\end{equation*}
For $\widetilde{I}_{42}$, we can integrate it by parts in the following
\begin{equation*}
\begin{split}
  \widetilde{I}_{42}=&-2 \int^t_0 \int_{\mathbb{R}^3} \Lambda^{s_0-2}\partial_m \{ \epsilon^{lkj}\mathrm{e}^{\rho} w_l \partial_j  \partial^m h \} \cdot \Lambda^{s_0-2}  \partial_k ( \mathrm{e}^\rho \partial_i H^i) dx d\tau
  \\
  =&-2 \int^t_0 \int_{\mathbb{R}^3} \Lambda^{s_0-2}\partial^2_{mj} \{ \epsilon^{lkj}\mathrm{e}^{\rho} w_l   \partial^m h \} \cdot \Lambda^{s_0-2}  \partial_k ( \mathrm{e}^\rho \partial_i H^i) dx  d\tau
  \\
  & +2 \int^t_0 \int_{\mathbb{R}^3} \Lambda^{s_0-2}\partial_m \{ \epsilon^{lkj}\mathrm{e}^{\rho} w_l \partial_j \rho \partial^m h \} \cdot \Lambda^{s_0-2}  \partial_k ( \mathrm{e}^\rho \partial_i H^i) dx  d\tau.
\end{split}
\end{equation*}
By H\"older's inequality, we can bound $\widetilde{I}_{42}$ by
\begin{equation}\label{sI42}
\begin{split}
  |\widetilde{I}_{42}| \lesssim & \int^t_0  ( \|\bw \partial h   \|_{\dot{H}^{s_0}_x} + \|\bw \partial h \partial \rho  \|_{\dot{H}^{s_0-1}_x}) \|\Delta h  - \partial \rho \partial h \|_{\dot{H}^{s_0-1}_x} d\tau
  \\
  \leq & \int^t_0 (\|\partial \bv \|_{L^\infty_x}+ \|\partial h \|_{L^\infty_x}) (\| \bv \|^2_{{H}^{s}_x}+ \|  \rho \|^2_{{H}^{s}_x}+ \| \bw \|^2_{{H}^{s_0}_x}+ \| h \|^2_{{H}^{s_0+1}_x}) d\tau
  \\
  & + \int^t_0 \| (\partial \bv , \partial h, \partial \rho) \|_{L^\infty_x} \| ( \bv , \rho , h) \|_{{H}^{s_0}_x} \{ \| (\bv , \rho) \|^2_{{H}^{s}_x}+ \| \bw \|^2_{{H}^{s_0}_x}+ \| h \|^2_{{H}^{s_0+1}_x}\} d\tau
  \\
  & + \int^t_0 \| (\partial \bv , \partial h, \partial \rho) \|_{L^\infty_x} \| ( \bv , \rho , h) \|^2_{{H}^{s_0}_x} \{ \| (\bv , \rho) \|^2_{{H}^{s}_x}+ \| \bw \|^2_{{H}^{s_0}_x}+ \| h \|^2_{{H}^{s_0+1}_x}\} d\tau.
\end{split}
\end{equation}
If we estimate $\widetilde{I}_{43}$ by Sobolev inequalities, then there is a loss of derivatives. So let us integrate it by parts as follows.
\begin{equation*}
\begin{split}
  \widetilde{I}_{43}=&-2 \int^t_0 \int_{\mathbb{R}^3} \Lambda^{s_0-2} \partial_k(  \partial^j v^k \partial_j  \partial^m h ) \cdot \Lambda^{s_0-2} \partial_m  ( \mathrm{e}^\rho \partial_i H^i) dx d\tau
  \\
  =& -2 \int^t_0 \int_{\mathbb{R}^3} \Lambda^{s_0-2} (  \partial^j v^k \partial_k \partial_j  \partial^m h ) \cdot \Lambda^{s_0-2} \partial_m  ( \mathrm{e}^\rho \partial_i H^i) dx d\tau
  \\
  & -2 \int^t_0 \int_{\mathbb{R}^3} \Lambda^{s_0-2}( \partial_k  \partial^j v^k \partial_j  \partial^m h ) \cdot \Lambda^{s_0-2} \partial_m  ( \mathrm{e}^\rho \partial_i H^i) dx d\tau.
%\\
%=& \widetilde{I}_{44}+\widetilde{I}_{45},
\end{split}
\end{equation*}
To write $\widetilde{I}_{43}$ in a simpler form, we set
\begin{equation}\label{I000}
  \widetilde{I}_{43}=\widetilde{I}_{44}+\widetilde{I}_{45},
\end{equation}
where
\begin{equation*}
\begin{split}
  \widetilde{I}_{44}
  =& -2 \int^t_0 \int_{\mathbb{R}^3} \Lambda^{s_0-2} (  \partial^j v^k \partial_k \partial_j  \partial^m h ) \cdot \Lambda^{s_0-2} \partial_m  ( \mathrm{e}^\rho \partial_i H^i) dx d\tau,
  \\
  \widetilde{I}_{45} =& -2 \int^t_0 \int_{\mathbb{R}^3} \Lambda^{s_0-2}( \partial^j \partial_k   v^k \partial_j  \partial^m h ) \cdot \Lambda^{s_0-2} \partial_m  ( \mathrm{e}^\rho \partial_i H^i) dx d\tau.
\end{split}
\end{equation*}
By H\"older's inequality, we can use Lemma \ref{lpe} to obtain
\begin{equation}\label{I44}
\begin{split}
 |\widetilde{I}_{44}| \leq & \int^t_0 \|\partial \bv \partial^3 h \|_{\dot{B}^{s_0-2}_{\infty,2}} \|\Delta h -\partial \rho \partial h \|_{\dot{H}_x^{s_0-1}} d\tau
 \\
 \lesssim &  \int^t_0 \|\partial \bv \|_{\dot{B}^{s_0-2}_{\infty,2}}\| h \|^2_{{H}^{s_0+1}_x} d\tau+\int^t_0 \|\partial \bv \|_{\dot{B}^{s_0-2}_{\infty,2}}\| \rho \|_{{H}^{s_0}_x}\| h \|^2_{{H}^{s_0+1}_x} d\tau.
\end{split}
\end{equation}
For $\widetilde{I}_{45}$, using the fact $\partial_k v^k = \mathrm{div}\bv=- \mathbf{T}\rho$ yields
\begin{equation*}
\begin{split}
  \widetilde{I}_{45} =& -2 \int^t_0 \int_{\mathbb{R}^3} \Lambda^{s_0-2}\left\{  \partial^j (\mathrm{div}\bv)  \partial_j  \partial^m h \right\} \cdot \Lambda^{s_0-2} \partial_m  ( \mathrm{e}^\rho \partial_i H^i) dx d\tau
  \\
  = & 2 \int^t_0 \int_{\mathbb{R}^3} \Lambda^{s_0-2}\{ \partial^j(\mathbf{T}\rho)  \partial_j  \partial^m h \} \cdot \Lambda^{s_0-2} \partial_m  ( \mathrm{e}^\rho \partial_i H^i) dx d\tau
  \\
  = & 2 \int^t_0 \int_{\mathbb{R}^3} \Lambda^{s_0-2}\{ \mathbf{T}(\partial^j\rho)  \partial_j  \partial^m h \} \cdot \Lambda^{s_0-2} \partial_m  ( \mathrm{e}^\rho \partial_i H^i) dx d\tau
  \\
  & +2 \int^t_0 \int_{\mathbb{R}^3} \Lambda^{s_0-2} ( \partial^j v^l \partial_l \rho   \partial_j  \partial^m h ) \cdot \Lambda^{s_0-2} \partial_m  ( \mathrm{e}^\rho \partial_i H^i) dx d\tau.
\end{split}
\end{equation*}
By commutators rule between the derivatives operators, we further derive that
\begin{equation*}
\begin{split}
  \widetilde{I}_{45}
  = & 2 \int^t_0 \int_{\mathbb{R}^3}  \mathbf{T} \left\{  \Lambda^{s_0-2} (\partial^j\rho  \partial_j  \partial^m h)  \cdot \Lambda^{s_0-2} \partial_m  ( \mathrm{e}^\rho \partial_i H^i) \right\} dx d\tau
  \\
  & +2 \int^t_0 \int_{\mathbb{R}^3} \Lambda^{s_0-2} ( \partial^j v^l \partial_l \rho   \partial_j  \partial^m h ) \cdot \Lambda^{s_0-2} \partial_m  ( \mathrm{e}^\rho \partial_i H^i) dx d\tau
  \\
  & - 2 \int^t_0 \int_{\mathbb{R}^3}  [\mathbf{T}, \Lambda^{s_0-2}] (\partial^j\rho  \partial_j  \partial^m h )  \cdot \Lambda^{s_0-2} \partial_m  ( \mathrm{e}^\rho \partial_i H^i) dx d\tau
  \\
  &-2 \int^t_0 \int_{\mathbb{R}^3}    \Lambda^{s_0-2} \{(\partial^j\rho  \mathbf{T}(\partial_j  \partial^m h ) \}  \cdot \Lambda^{s_0-2} \partial_m  ( \mathrm{e}^\rho \partial_i H^i) dx d\tau
  \\
  &+2 \int^t_0 \int_{\mathbb{R}^3}    \Lambda^{s_0-2} (\partial^j\rho  \partial_j  \partial^m h)  \cdot \Lambda^{s_0-2} \mathbf{T} \partial_m    ( \mathrm{e}^\rho   \partial_i H^i  ) dx d\tau
  \\
  &-2 \int^t_0 \int_{\mathbb{R}^3}    \Lambda^{s_0-2} (\partial^j\rho  \partial_j  \partial^m h)  \cdot [\mathbf{T}, \Lambda^{s_0-2} ]   \partial_m( \mathrm{e}^\rho \partial_i H^i)\} dx d\tau.
\end{split}
\end{equation*}
For simplicity, we also write
\begin{equation}\label{I90}
  \widetilde{I}_{45}= \widetilde{I}_{46}+\widetilde{I}_{47}+\widetilde{I}_{48}+\widetilde{I}_{49}+\widetilde{I}_{410}+\widetilde{I}_{411}.
\end{equation}
Here, we set
\begin{equation}\label{I7e}
\begin{split}
  \widetilde{I}_{46}&=2 \int^t_0 \int_{\mathbb{R}^3}  \mathbf{T} \left\{  \Lambda^{s_0-2} (\partial^j\rho  \partial_j  \partial^m h)  \cdot \Lambda^{s_0-2} \partial_m  ( \mathrm{e}^\rho \partial_i H^i) \right\} dx d\tau,
  \\
  \widetilde{I}_{47}& = 2 \int^t_0 \int_{\mathbb{R}^3} \Lambda^{s_0-2} ( \partial^j v^l \partial_l \rho   \partial_j  \partial^m h ) \cdot \Lambda^{s_0-2} \partial_m  ( \mathrm{e}^\rho \partial_i H^i) dx d\tau,
  \\
  \widetilde{I}_{48}&= - 2 \int^t_0 \int_{\mathbb{R}^3}  [\mathbf{T}, \Lambda^{s_0-2}] (\partial^j\rho  \partial_j  \partial^m h)  \cdot \Lambda^{s_0-2} \partial_m  ( \mathrm{e}^\rho \partial_i H^i) dx d\tau,
  \\
  \widetilde{I}_{49}&=-2 \int^t_0 \int_{\mathbb{R}^3}    \Lambda^{s_0-2} \{(\partial^j\rho  \mathbf{T}(\partial_j  \partial^m h) \}  \cdot \Lambda^{s_0-2} \partial_m  ( \mathrm{e}^\rho \partial_i H^i) dx d\tau,
  \\
  \widetilde{I}_{410}& =2 \int^t_0 \int_{\mathbb{R}^3}    \Lambda^{s_0-2} (\partial^j\rho  \partial_j  \partial^m h)  \cdot \Lambda^{s_0-2}\left\{ \mathbf{T} \partial_m    ( \mathrm{e}^\rho   \partial_i H^i  ) \right\} dx d\tau,
  \\
  \widetilde{I}_{411}& = -2 \int^t_0 \int_{\mathbb{R}^3}    \Lambda^{s_0-2} (\partial^j\rho  \partial_j  \partial^m h)  \cdot [\mathbf{T}, \Lambda^{s_0-2} ]   \partial_m( \mathrm{e}^\rho \partial_i H^i)\} dx d\tau.
\end{split}
\end{equation}
Integrating $\widetilde{I}_{46}$ by parts, we have
\begin{equation}\label{I7r}
\begin{split}
  \widetilde{I}_{46}=&2 \int^t_0 \int_{\mathbb{R}^3}  \mathbf{T} \left\{  \Lambda^{s_0-2} (\partial^j\rho  \partial_j  \partial^m h)  \cdot \Lambda^{s_0-2} \partial_m  ( \mathrm{e}^\rho \partial_i H^i) \right\} dx d\tau
  \\
  =
  & 2 \int^t_0 \frac{d}{d\tau} \int_{\mathbb{R}^3}    \Lambda^{s_0-2} (\partial^j\rho  \partial_j  \partial^m h)  \cdot \Lambda^{s_0-2} \partial_m  ( \mathrm{e}^\rho \partial_i H^i)dx d\tau
  \\
  &- 2 \int^t_0 \int_{\mathbb{R}^3}  \mathrm{div}\bv \cdot  \Lambda^{s_0-2} (\partial^j\rho  \partial_j  \partial^m h)  \cdot \Lambda^{s_0-2} \partial_m  ( \mathrm{e}^\rho \partial_i H^i)  dx d\tau
  \\
  =
  &  M(t)-M(0)+\widetilde{I}_{412},
\end{split}
\end{equation}
where
\begin{equation*}
\begin{split}
M(t)=& 2\int_{\mathbb{R}^3}    \Lambda^{s_0-2} (\partial^j\rho  \partial_j  \partial^m h)  \cdot \Lambda^{s_0-2} \partial_m  ( \mathrm{e}^\rho \partial_i H^i)dx,
\\
  \widetilde{I}_{412}= &- 2 \int^t_0 \int_{\mathbb{R}^3}  \mathrm{div}\bv \cdot  \Lambda^{s_0-2} (\partial^j\rho  \partial_j  \partial^m h)  \cdot \Lambda^{s_0-2} \partial_m  ( \mathrm{e}^\rho \partial_i H^i)  dx d\tau.
\end{split}
\end{equation*}
Using Sobolev's inequality, we find
\begin{equation}\label{I412}
\begin{split}
  | \widetilde{I}_{412} | \lesssim  & \int^t_0 \| \partial \bv \|_{L^\infty_x}  \| \partial\rho \partial^2h \|_{\dot{H}^{s_0-2}_x} \|\Delta h - \partial\rho \partial h \|_{\dot{H}^{s_0-1}_x} d\tau
  \\
  \lesssim  & \int^t_0 \| \partial \bv \|_{L^\infty_x} ( \| \rho \|_{{H}^{s_0}_x} +\| h \|_{{H}^{s_0}_x}  ) ( \| \rho \|^2_{{H}^{s}_x} +\| h \|^2_{{H}^{s_0+1}_x}  ) d\tau
  \\
  & + \int^t_0 \| \partial \bv \|_{L^\infty_x} ( \| \rho \|^2_{{H}^{s_0}_x} +\| h \|^2_{{H}^{s_0}_x}  ) ( \| \rho \|^2_{{H}^{s}_x} +\| h \|^2_{{H}^{s_0+1}_x}  ) d\tau.
\end{split}
\end{equation}
By Lemma \ref{lpe}, we also have
\begin{equation}\label{I47}
\begin{split}
  | \widetilde{I}_{47} |\lesssim & \int^t_0 \|\partial \bv \partial \rho \partial^2 h\|_{\dot{H}^{s_0-2}_x}    \|\Delta h - \partial\rho \partial h \|_{\dot{H}^{s_0-1}_x} d\tau
  \\
   \lesssim & \int^t_0 \| \partial \bv \|_{L^\infty_x} ( \| \rho \|_{{H}^{s_0}_x} +\| h \|_{{H}^{s_0}_x}  ) ( \| \rho \|^2_{{H}^{s}_x} +\| h \|^2_{{H}^{s_0+1}_x}  ) d\tau
  \\
  & + \int^t_0 \| \partial \bv \|_{L^\infty_x} ( \| \rho \|^2_{{H}^{s_0}_x} +\| h \|^2_{{H}^{s_0}_x}  ) ( \| \rho \|^2_{{H}^{s}_x} +\| h \|^2_{{H}^{s_0+1}_x}  ) d\tau.
\end{split}
\end{equation}
By utilizing Lemma \ref{ce} and H\"older's inequality, we obtain
\begin{equation}\label{I48}
\begin{split}
  | \widetilde{I}_{48} |
   \lesssim & \int^t_0  \| \partial \bv \|_{\dot{B}^{0}_{\infty,2}} \| \partial \rho \partial^2 h\|_{\dot{H}^{s_0-2}_x}  \|\Delta h - \partial\rho \partial h \|_{\dot{H}^{s_0-1}_x} d\tau
  \\
  \lesssim & \int^t_0 \| \partial \bv \|_{\dot{B}^{0}_{\infty,2}} ( \| \rho \|_{{H}^{s_0}_x} +\| h \|_{{H}^{s_0}_x}  ) ( \| \rho \|^2_{{H}^{s}_x} +\| h \|^2_{{H}^{s_0+1}_x}  ) d\tau
  \\
  & + \int^t_0 \| \partial \bv \|_{\dot{B}^{0}_{\infty,2}} ( \| \rho \|^2_{{H}^{s_0}_x} +\| h \|^2_{{H}^{s_0}_x}  ) ( \| \rho \|^2_{{H}^{s}_x} +\| h \|^2_{{H}^{s_0+1}_x}  ) d\tau.
\end{split}
\end{equation}
We note that $\mathbf{T}(\partial^2 h)$ satisfies
\begin{equation*}
  \mathbf{T}(\partial^2 h)=- \partial^2 \bv \cdot \partial h+ \partial \bv \cdot \partial^2 h.
\end{equation*}
Hence, by H\"older's inequality again, we can estimate $\tilde{I}_{49}$ by
\begin{equation}\label{I49}
\begin{split}
   | \widetilde{I}_{49} |  \leq & C\int^t_0  (\| \partial \rho \partial h \partial^2 \bv \|_{\dot{H}^{s_0-2}_{x}}+ \| \partial \rho \partial \bv\partial^2 h\|_{\dot{H}^{s_0-2}_x} ) \|\Delta h - \partial\rho \partial h \|_{\dot{H}^{s_0-1}_x} d\tau
   \\
   \lesssim & \int^t_0 (\| \partial \rho \|_{\dot{B}^{s_0-2}_{\infty,2}}+ \| ( \partial \rho,  \partial h) \|_{L^\infty_x}) \| \bv \|_{{H}^{s_0}_{x}} ( \| \bv \|^2_{{H}^{s}_{x}} +\| h \|^2_{{H}^{s_0+1}_{x}}) d\tau
   \\
   & + \int^t_0 (\| \partial \rho \|_{\dot{B}^{s_0-2}_{\infty,2}}+ \| ( \partial \rho,  \partial h) \|_{L^\infty_x})  \| ( \bv , \rho , h) \|^2_{{H}^{s_0}_{x}}  ( \| \bv \|^2_{{H}^{s}_{x}} +\| h \|^2_{{H}^{s_0+1}_{x}}) d\tau.
      %\\
   %& +  \| \bw \|_{\dot{B}^{s_0-2}_{\infty,2}} \| \partial  \rho \|_{\dot{B}^{s_0-2}_{\infty,2}}   \| \partial^2 h \|_{\dot{H}^{s_0-2}_{x}} \|\mathrm{e}^\rho \partial  \bH\|_{\dot{H}^{s_0-1}_x} .
\end{split}
\end{equation}
To estimate $\widetilde{I}_{410}$, we calculate
\begin{equation*}
\begin{split}
   \widetilde{I}_{410} =& 2 \int^t_0 \int_{\mathbb{R}^3}    \Lambda^{s_0-2} (\partial^j\rho  \partial_j  \partial^m h)  \cdot \Lambda^{s_0-2}\left\{ -2\partial_m (\partial_i v^j) \partial_j  \partial^i h+Y_m  \right\} dx d\tau
   \\
   =& -4 \int^t_0 \int_{\mathbb{R}^3}    \Lambda^{s_0-2} (\partial^j\rho  \partial_j  \partial^m h)  \cdot \Lambda^{s_0-2} \left\{ \partial_m( \partial_i v^j \partial_j  \partial^i h)  \right\} dx d\tau
   \\
   & + 4 \int^t_0 \int_{\mathbb{R}^3}    \Lambda^{s_0-2} (\partial^j\rho  \partial_j  \partial^m h)  \cdot \Lambda^{s_0-2} (  \partial_i v^j \partial_m \partial_j  \partial^i h  ) dx d\tau
   \\
   &+ 2 \int^t_0 \int_{\mathbb{R}^3}    \Lambda^{s_0-2} (\partial^j\rho  \partial_j  \partial^m h)  \cdot \Lambda^{s_0-2} Y_m   dx d\tau.
\end{split}
\end{equation*}
By the Plancherel formula, we have
\begin{equation*}
  \begin{split}
    \widetilde{I}_{410} =& 2 \int^t_0 \int_{\mathbb{R}^3}    \Lambda^{s_0-2}_x (\partial^j\rho  \partial_j  \partial^m h)  \cdot \Lambda^{s_0-2}_x \left\{ -2\partial_m (\partial_i v^j) \partial_j  \partial^i h+Y_m  \right\} dx d\tau
   \\
   =& -4 \int^t_0 \int_{\mathbb{R}^3}    \Lambda^{s_0-2} \Lambda^{\frac12}_x(\partial^j\rho  \partial_j  \partial^m h)  \cdot \Lambda_x^{s_0-2}\Lambda^{-\frac12}_x \left\{ \partial_m( \partial_i v^j \partial_j  \partial^i h)  \right\} dx d\tau
   \\
   & + 4 \int^t_0 \int_{\mathbb{R}^3}    \Lambda^{s_0-2}_x (\partial^j\rho  \partial_j  \partial^m h)  \cdot \Lambda^{s_0-2}_x (  \partial_i v^j \partial_m \partial_j  \partial^i h  ) dx d\tau
   \\
   &+ 2 \int^t_0 \int_{\mathbb{R}^3}    \Lambda^{s_0-2}_x (\partial^j\rho  \partial_j  \partial^m h)  \cdot \Lambda^{s_0-2}_x Y_m   dx d\tau.
\end{split}
\end{equation*}
Using \eqref{Hh2} and H\"older's inequality, we deduce that
\begin{equation}\label{I410}
\begin{split}
  | \widetilde{I}_{410} | \lesssim &  \int^t_0   \| \partial\rho \partial^2 h \|_{\dot{H}^{s_0-\frac{3}{2}}_{x}} \| \partial\bv \partial^2 h \|_{\dot{H}^{s_0-\frac{3}{2}}_{x}}d\tau+ C\int^t_0   \| \partial\rho \partial^2 h \|_{\dot{H}^{s_0-2}_{x}} \| \partial\bv \partial^3 h \|_{\dot{H}^{s_0-2}_{x}}d\tau
  \\
  & + \int^t_0   \| \partial\rho \partial^2 h \|_{\dot{H}^{s_0-2}_{x}} \| \bY \|_{\dot{H}^{s_0-2}_{x}}d\tau
  \\
  \lesssim &  \int^t_0  \| \partial \bv \|_{\dot{B}^{0}_{\infty,2}}\| \rho \|_{{H}^{2}_{x}} \| h \|^2_{{H}^{s_0+1}_{x}}  d\tau
   + \int^t_0   ( \| \bv \|^2_{{H}^{2}_{x}} +  \| \rho \|^2_{{H}^{2}_{x}}) \| h \|^2_{{H}^{s_0+1}_{x}}  d\tau
  \\
  & + \int^t_0    \| ( \partial \bv, \partial h, \partial \rho) \|_{L^\infty_x}  \| (\bv, \rho , h) \|^2_{{H}^{s_0}_{x}} \{ \| (\bv , \rho) \|^2_{{H}^{s}_{x}} + \| \bw \|^2_{{H}^{s_0}_{x}}+\| h \|^2_{{H}^{s_0+1}_{x}}\} d\tau
  \\
  & + \int^t_0    \| ( \partial \bv, \partial h, \partial \rho) \|_{L^\infty_x}  \| (\bv, \rho , h) \|^3_{{H}^{s_0}_{x}} \{ \| (\bv , \rho) \|^2_{{H}^{s}_{x}} + \| \bw \|^2_{{H}^{s_0}_{x}}+\| h \|^2_{{H}^{s_0+1}_{x}}\} d\tau.
\end{split}
\end{equation}
By Lemma \ref{ce}, we also have
\begin{equation}\label{I411}
\begin{split}
  | \widetilde{I}_{411} | \lesssim &   \int^t_0   \| \partial\rho \partial^2 h \|_{\dot{H}^{s_0-2}_{x}} \| \partial\bv\|_{\dot{B}^{0}_{\infty,2}} \|\Delta h - \partial \rho \partial h \|_{\dot{H}^{s_0-1}_{x}}d\tau
  \\
  \lesssim
   &  \int^t_0  \| \partial \bv \|_{\dot{B}^{0}_{\infty,2}}  \| (\bv, \rho , h) \|_{{H}^{s_0}_{x}} ( \| \bv \|^2_{{H}^{s}_{x}} +  \| \rho \|^2_{{H}^{s}_{x}} + \| \bw \|^2_{{H}^{s_0}_{x}}+\| h \|^2_{{H}^{s_0+1}_{x}}) d\tau
   \\
   &  + \int^t_0  \| \partial \bv \|_{\dot{B}^{0}_{\infty,2}}  \| (\bv, \rho , h) \|^2_{{H}^{s_0}_{x}} ( \| \bv \|^2_{{H}^{s}_{x}} +  \| \rho \|^2_{{H}^{s}_{x}} + \| \bw \|^2_{{H}^{s_0}_{x}}+\| h \|^2_{{H}^{s_0+1}_{x}}) d\tau.
\end{split}
\end{equation}
Observing form \eqref{He6}, \eqref{I08}, \eqref{I09}, \eqref{I000}, \eqref{I90}, and \eqref{I7e}, we have
\begin{equation}\label{Hep}
  \begin{split}
   \| ( \mathrm{e}^\rho \partial_i H^i) \|^2_{\dot{H}^{s_0-1}_x}(t)
  =&M(t)-M(0)+\widetilde{I}_0+\widetilde{I}_1+\widetilde{I}_2
  +\widetilde{I}_3+\widetilde{I}_{40}
  +\widetilde{I}_{42}
  \\
  &+\widetilde{I}_{44}+\widetilde{I}_{47}+\widetilde{I}_{48}+\widetilde{I}_{49}+\widetilde{I}_{410}+\widetilde{I}_{411}+\widetilde{I}_{412}.
  \end{split}
\end{equation}
By using \eqref{Hep} and then combining \eqref{I00}, \eqref{I01}, \eqref{I02}, \eqref{I03}, \eqref{I40}, \eqref{sI42}, \eqref{I44}, \eqref{I47}, \eqref{I48}, \eqref{I49}, \eqref{I410}, \eqref{I411} and \eqref{I412}, we can conclude that
\begin{equation}\label{fin}
\begin{split}
  & \| ( \mathrm{e}^\rho \partial_i H^i) \|^2_{\dot{H}^{s_0-1}_x}(t)
  \\
  \lesssim & M(t)-M(0)+\| h_0 \|^2_{H^{s_0+1}}+ \| h_0 \|^2_{H^{\frac52+}}\| \rho_0 \|^2_{H^{s_0}}
\\
& + \int^t_0 \|\partial \bv \|_{L^\infty_x} \| h \|^2_{H^{\frac52+}}\| \rho \|^2_{H^{s_0}} d\tau
 + \int^t_0   ( \| \bv \|^2_{{H}^{2}_{x}} +  \| \rho \|^2_{{H}^{2}_{x}}) \| h \|^2_{{H}^{s_0+1}_{x}}  d\tau
\\
&+ \int^t_0 \|\partial \bv \|_{\dot{B}^{s_0-2}_{\infty,2}}( \|h\|^2_{H^{s_0+1}_x}+ \|\bw\|^2_{H^{s_0}_x}+ \|\rho\|^2_{H^{s}_x}+ \|\bv\|^2_{H^{s}_x})d\tau
  \\
  &+\int^t_0 \|(\partial \bv, \partial \rho, \partial h) \|_{\dot{B}^{s_0-2}_{\infty,2}}\|( h,\rho,\bv)\|_{H^{s_0}_x}( \|h\|^2_{H^{s_0+1}_x}+ \|\bw\|^2_{H^{s_0}_x}+ \| (\rho,\bv) \|^2_{H^{s}_x})d\tau
  \\
  &+\int^t_0 \|(\partial \bv, \partial \rho, \partial h) \|_{\dot{B}^{s_0-2}_{\infty,2}} \|( h,\rho,\bv) \|^2_{H^{s_0}_x}( \|h\|^2_{H^{s_0+1}_x}+ \|\bw\|^2_{H^{s_0}_x}+ \|(\rho, \bv)\|^2_{H^{s}_x})d\tau
  \\
  & +\int^t_0 \|(\partial \bv, \partial \rho, \partial h) \|_{\dot{B}^{s_0-2}_{\infty,2}} \|( h,\rho, \bv)\|^3_{H^{s_0}_x}( \|h\|^2_{H^{s_0+1}_x}+ \|\bw\|^2_{H^{s_0}_x}+ \|(\rho,\bv)\|^2_{H^{s}_x})d\tau.
\end{split}
\end{equation}
By H\"older's inequality, we have
\begin{equation}\label{Mt}
  | M(t) | \lesssim  \| \rho \|_{{H}^{2}_{x}} \| h \|_{{H}^{s_0+\frac12}_{x}}\| h \|_{{H}^{s_0+1}_{x}}+\| h \|_{{H}^{\frac52+}_{x}}\| \rho \|_{{H}^{2}_{x}}\| \rho \|_{{H}^{s_0}_{x}} \| h \|_{{H}^{s_0+\frac12}}.
  %\leq C \| \rho \|^4_{{H}^{s_0}_{x}} \| h \|^2_{{H}^{s_0}} + \frac{1}{16} \| h \|^2_{{H}^{s_0+1}_{x}},
\end{equation}
Using Young's inequality, \eqref{Mt} yields
\begin{equation}\label{Mtt}
  | M(t) | \leq \frac{1}{100}\| h \|^2_{L^\infty_t {H}^{s_0+1}_{x}}+ C \| \rho \|^4_{L^\infty_t {H}^{2}_{x}}\| h \|^2_{L^\infty_t {H}^{s_0}_{x}}.
\end{equation}
Combining \eqref{fin} with \eqref{Mtt}, we obtain that
\begin{equation}\label{FF}
\begin{split}
  & \| ( \mathrm{e}^\rho \partial_i H^i) \|^2_{\dot{H}^{s_0-1}_x}(t)
  \\
  \lesssim &  \| \rho \|^4_{L^\infty_t {H}^{2}_{x}}\| h \|^2_{L^\infty_t {H}^{s_0}_{x}}+\| h_0 \|^2_{H^{s_0+1}}+ \| h_0 \|^2_{H^{\frac52+}}\| \rho_0 \|^2_{H^{s_0}}+ \frac{1}{100} \| h \|^2_{L^\infty_t {H}^{s_0+1}_{x}}
\\
& + \int^t_0 \|\partial \bv \|_{L^\infty_x} \| h \|^2_{H^{\frac52+}}\| \rho \|^2_{H^{s_0}} d\tau
 + \int^t_0   ( \| \bv \|^2_{{H}^{2}_{x}} +  \| \rho \|^2_{{H}^{2}_{x}}) \| h \|^2_{{H}^{s_0+1}_{x}}  d\tau
\\
&+ \int^t_0 \|\partial \bv \|_{\dot{B}^{s_0-2}_{\infty,2}}( \|h\|^2_{H^{s_0+1}_x}+ \|\bw\|^2_{H^{s_0}_x}+ \|\rho\|^2_{H^{s}_x}+ \|\bv\|^2_{H^{s}_x})d\tau
  \\
  &+\int^t_0 \|(\partial \bv, \partial \rho, \partial h) \|_{\dot{B}^{s_0-2}_{\infty,2}}\|( h,\rho,\bv)\|_{H^{s_0}_x}( \|h\|^2_{H^{s_0+1}_x}+ \|\bw\|^2_{H^{s_0}_x}+ \| (\rho,\bv) \|^2_{H^{s}_x})d\tau
  \\
  &+\int^t_0 \|(\partial \bv, \partial \rho, \partial h) \|_{\dot{B}^{s_0-2}_{\infty,2}} \|( h,\rho,\bv) \|^2_{H^{s_0}_x}( \|h\|^2_{H^{s_0+1}_x}+ \|\bw\|^2_{H^{s_0}_x}+ \|(\rho, \bv)\|^2_{H^{s}_x})d\tau
  \\
  & +\int^t_0 \|(\partial \bv, \partial \rho, \partial h) \|_{\dot{B}^{s_0-2}_{\infty,2}} \|( h,\rho, \bv)\|^3_{H^{s_0}_x}( \|h\|^2_{H^{s_0+1}_x}+ \|\bw\|^2_{H^{s_0}_x}+ \|(\rho,\bv)\|^2_{H^{s}_x})d\tau.
\end{split}
\end{equation}
\textbf{Step 3: The total energy.} By \eqref{E0}, we have
\begin{equation}\label{e007}
  \|( h,\rho,\bv) \|_{H^{s_0}_x} \leq C \|( h_0,\rho_0,\bv_0) \|_{H^{s_0}_x}\exp\{ \int^t_0 \|(d \bv, d \rho, d h) \|_{L^\infty_x} d\tau \}.
\end{equation}
Gathering \eqref{e000}, \eqref{e001}, \eqref{e0004}, \eqref{e006}, \eqref{He2}, \eqref{FF}, \eqref{e007} and taking $0< t \leq \frac{1}{2C}$, upon cancelling the small terms with factor $1/100$, $1/50$  from \eqref{e0004}, \eqref{FF}, we have
\begin{equation}\label{fee}
\begin{split}
  E_s(t) \leq & CE_0\exp\{ \int^t_0 \|(d \bv, d \rho, d h) \|_{L^\infty_x} d\tau\}
   +C \int^t_0  \|(\partial \bv, \partial \rho, \partial h) \|_{\dot{B}^{s_0-2}_{\infty,2}}(\|( h_0,\rho_0,\bv_0) \|_{H^{s_0}_x}
  \\
  & +\|( h_0,\rho_0,\bv_0) \|^2_{H^{s_0}_x}+\|( h_0,\rho_0,\bv_0) \|^3_{H^{s_0}_x}) \cdot \exp\{ \int^\tau_0 \|(d \bv, d \rho, d h) \|_{L^\infty_x} d\tau \}E_s(\tau)d\tau
  \\
  & +C \int^t_0
  \|( h_0,\rho_0,\bv_0) \|^2_{H^{s_0}_x} \exp\{ \int^\tau_0 \|(d \bv, d \rho, d h) \|_{L^\infty_x} d\tau \}E_s(\tau)d\tau,
\end{split}
\end{equation}
where ${E}_0=E_s(0)+E^2_s(0)+E^3_s(0)+E^{2s_0-1}_s(0)+E^{1+\frac{s_0-1}{s_0-2}}_s(0)$.
By using Gronwall's equality, we conclude that
\begin{equation*}
\begin{split}
  E_s(t) \leq & C{E}_0 \exp\left\{ {E}_0 \int^t_0  \|(\partial \bv, \partial \rho, \partial h) \|_{\dot{B}^{s_0-2}_{\infty,2}}  d\tau
    \cdot \exp ( \int^t_0  \|(d \bv, d \rho, d h) \|_{L^\infty_x}  d\tau ) \right\}.
\end{split}
\end{equation*}
So we have proved Theorem \ref{ve}.
\end{proof}

\subsection{Energy estimates of Theorem \ref{dingli3}}\label{ES2}
\begin{theorem}\label{TT2}
Let $(\rho,\bv)$ be a solution of \eqref{fc0} with $h=0$. Let $2<\sstar<\frac52$. Let $\bw$ be defined in \eqref{pw1}. Then the following energy estimates hold:
\begin{equation}\label{DD1}
\begin{split}
\mathcal{E}(t) \lesssim  \mathcal{E}_0 \exp \{ \mathcal{E}_1 \int^t_0 \|(d \bv, d \rho)\|_{L^\infty_x} d\tau \cdot \exp ( \int^t_0  \|(d \bv, d \rho) \|_{L^\infty_x}  d\tau )\},
\end{split}
\end{equation}
where $\mathcal{E}(t)$, $\mathcal{E}_0$ and $\mathcal{E}_1$ are defined by
\begin{equation}\label{DD2}
  \mathcal{E}(t)=\|\bv\|^2_{H_x^{\sstar}} + \|\rho\|^2_{H_x^{\sstar}}+ \|\bw\|^2_{H_x^{2}},
\end{equation}
and
\begin{equation}\label{TTE}
  \begin{split}
  \mathcal{E}_0 = &   \|\bv_0\|^2_{H_x^{\sstar}} + \|\rho_0\|^2_{H_x^{\sstar}}+ \|\bw_0\|^2_{H_x^{2}}+   \|  \bv_0 \|^{2}_{H^{\sstar}}\|\rho_0\|^4_{H^{\sstar}},
  \end{split}
\end{equation}
and
\begin{equation}\label{TTF}
  \mathcal{E}_1 = 1+
    \| (\rho_0, \bv_0)\|_{ H^2 } + \| (\rho_0, \bv_0 )\|^2_{ H^2 }
   + \| ( \rho_0, \bv_0 ) \|^3_{ H^2 }.
\end{equation}
\end{theorem}
\begin{proof}
By using \eqref{E0}, taking $h=0$ and integrating over $[0,t]$, we have
\begin{equation}\label{DD0}
 \| ({\rho},\bv)\|^2_{H_x^{\sstar}}(t)  \leq   \| (\rho_0,\bv_0)\|^2_{H_x^{\sstar}}+ C\int^t_0 \|(d\bv, d\rho)\|_{L^\infty_x}\| ({\rho},\bv)\|^2_{H_x^{\sstar}}d\tau.
\end{equation}
We then discuss $\bw$ into several steps.

\textbf{Step 1: The estimate for $\bw$}. Recall that $\bw$ satisfies \eqref{W0}. Taking $h=0$, \eqref{W0} and \eqref{W2} reduce to
\begin{equation}\label{DB0}
\mathbf{T} w^i =  (\bw \cdot \nabla )v^i,
\end{equation}
and
\begin{equation}\label{DB1}
\begin{split}
& \mathbf{T} ( \mathrm{curl} \mathrm{curl} \bw^i+\Omega^i  )
=   \partial^i \big( 2 \partial_n v_a \partial^n w^a \big) + \Gamma^i,
\end{split}
\end{equation}
where
\begin{equation}\label{DB2}
  \Omega^i= -\epsilon^{ijk} \partial_j \rho \cdot \mathrm{curl}\bw_k- 2\partial^a {\rho} \partial^i w_a,
\end{equation}
and
\begin{equation}\label{DB3}
\begin{split}
\Gamma^i= \Gamma^i_1+\Gamma^i_2+\Gamma^i_3+\Gamma^i_4 ,
\end{split}
\end{equation}
and
\begin{equation*}
\begin{split}
 \Gamma^i_1=
&
-2\epsilon_{kmn}\epsilon^{ijk}\partial^m v^a \partial_{j}(\partial^n w_a) + 2\epsilon_{kmn}\epsilon^{ijk}\partial_m v^a \partial_{n}w_a \partial_j {\rho}, \ \ \ \ \ \ \ \ \ \ \ \ \ \ \ \ \ \ \ \ \ \ \ \ \ \ \ \ \ \ \ \ \ \ \ \ \ \ \
\end{split}
\end{equation*}
\begin{equation*}
  \begin{split}
  \Gamma^i_2=&- 2 \partial^a {\rho} \big( \partial^i w^m \partial_m v_a + w^m \partial^i(\partial_{m}v_a) - \partial^i v^k \partial_k w_a \big) 
   + 2\mathrm{e}^{\rho} w^i  \mathrm{curl}\bw^m \partial_m \rho
   \\
   & +2\mathrm{e}^{\rho}\mathrm{curl}\bw^m \partial_m w^i 
   -2   \partial^a v^k \partial_k {\rho}  \partial^i w_a+2 \mathrm{e}^{{\rho}} \mathrm{curl}\bw^a \partial^i w_a
  + 2 \epsilon^{ajk} \mathrm{e}^{{\rho}} w_k  \partial_j {\rho}  \partial^iw_a,
  \end{split}
\end{equation*}
\begin{equation*}
  \Gamma^i_3=2 \partial^i {\rho} \partial_{j}v^a \partial^j w_a
  -2  \partial_{j} v^a \partial^j\partial^i  w_a-\epsilon^{ijk} \partial_j v^m \partial_m ( \mathrm{e}^{-\rho}\mathrm{curl}\bw_k), \ \ \ \ \ \ \ \ \ \ \ \ \ \ \ \ \ \ \ \ \ \ \ \ \ \ \ \ \ \ \
\end{equation*}
\begin{equation*}
\begin{split}
  \Gamma^i_4=&2 \partial^i \rho \partial_n v_a \partial^n w^a-  \mathrm{e}^{\rho}\mathrm{div}\bv \cdot \epsilon^{ijk} \partial_j ( \mathrm{e}^{-\rho} \mathrm{curl} \bw_k ) +2\mathrm{div}\bv\cdot\partial^a {\rho} \partial^i w_a . \ \ \ \ \ \ \ \ \ \ \ \ \ \ \ \ \ \ \ \ \
\end{split}
\end{equation*}
Multiplying \eqref{DB0} bt $\bw$ and integrating over $ [0,t]\times \mathbb{R}^3$, we can get
\begin{equation}\label{DB4}
  \| \bw \|^2_{L_x^2}(t)- \| \bw \|^2_{L_x^2}(0)\leq C \int^t_0 \| \partial \bv\|_{L^\infty_x}\|\bw\|^2_{L^2_x}d\tau.
\end{equation}
By elliptic estimates, we have
\begin{equation}\label{DB5}
  \| \bw \|_{\dot{H}^{2}_x} \leq \| \mathrm{curl} \mathrm{curl} \bw \|_{L^2_x}+\| \mathrm{div} \bw\|_{\dot{H}^{1}_x}.
\end{equation}
%\textbf{Step 2: higher-order energy estimate of $\mathrm{div} \bw$}.
Let us first estimate $\| \mathrm{div} \bw \|_{\dot{H}^{s_0-1}_x}$. Note \eqref{W01}. Using product estimates in Lemma \ref{ps}, we can derive that
\begin{equation}\label{DB6}
\begin{split}
  \| \mathrm{div} \bw \|_{{H}^{1}_x} &= \|  \bw \cdot \partial \rho  \|_{{H}^{1}_x} \leq C \|\bw\|_{H_x^{\frac32+}}\|\rho\|_{{H}^{s}_x}.
\end{split}
\end{equation}
Let us next estimate $\| \mathrm{curl} \mathrm{curl} \bw \|_{L^2_x}$. Multiplying \eqref{DB1} by $\mathrm{curl} \mathrm{curl} \bw^i +\Omega^i$ on , and integrating over $[0,t] \times \mathbb{R}^3$, we find
\begin{equation}\label{DB7}
\begin{split}
    \|\mathrm{curl} \mathrm{curl} \bw + \Omega\|^2_{L^2_x}(t)
  \leq & \|\mathrm{curl} \mathrm{curl} \bw + \Omega\|^2_{L^2_x}(0)+ C | \int^t_0 \int_{\mathbb{R}^3}\mathrm{div}\bv \cdot |\mathrm{curl} \mathrm{curl} \bw + \Omega|^2 dxd\tau |
  \\
  & + C | \int^t_0 \int_{\mathbb{R}^3} \partial^i \big( 2   \partial_n v_a \partial^n w^a \big) (\mathrm{curl} \mathrm{curl} \bw_i + \Omega_i)dxd\tau |
  \\
  & +C | \int^t_0 \int_{\mathbb{R}^3}  \Gamma^i (\mathrm{curl} \mathrm{curl} \bw_i + \Omega_i)dxd\tau | ,
\end{split}
\end{equation}
where $ \Omega=(\Omega^1,\Omega^2,\Omega^3)$ and $ \Gamma=(\Gamma^1,\Gamma^2,\Gamma^3)$. By H\"older's inequality, we can estimate
\begin{equation}\label{DB8}
  \|\Omega\|_{L^2_x}(t) \leq  C\|\rho\|_{H^{\sstar}_x}\|\bw\|_{H^{\frac32+}_x}.
\end{equation}
By Young's inequality, we further derive that
\begin{equation}\label{DB9}
  \|\Omega\|^2_{L^2_x}(t) \leq  \frac18 \|\bw\|^2_{H^{2}_x}+  C \|\rho\|^4_{H^{\sstar}_x} .
\end{equation}
For the right hand side of \eqref{DB7}, let us set
\begin{equation}\label{DC0}
\begin{split}
   G_0
  =& \|\mathrm{curl} \mathrm{curl} \bw + \Omega\|^2_{L^2_x}(0),
  \\
  G_1=&  \int^t_0 \int_{\mathbb{R}^3}\mathrm{div}\bv \cdot |\mathrm{curl} \mathrm{curl} \bw + \Omega|^2 dxd\tau,
  \\
  G_2=& \int^t_0 \int_{\mathbb{R}^3} \Gamma^i\cdot(\mathrm{curl} \mathrm{curl} \bw_i + \Omega_i)dxd\tau,
  \\
  G_3=&\int^t_0 \int_{\mathbb{R}^3} \partial^i \big( 2   \partial_n v_a \partial^n w^a \big) \cdot (\mathrm{curl} \mathrm{curl} \bw_i + \Omega_i)dxd\tau.
\end{split}
\end{equation}
Let us estimate $G_0,G_1,\ldots,G_3$ one by one. For $G_0$, we can use H\"older's inequality to bound it by
\begin{equation}\label{DC1}
\begin{split}
  |G_0| & \leq \|\mathrm{curl} \mathrm{curl} \bw_0 \|^2_{L^2_x}+ C\|\rho_0\|_{H^{\sstar}_x}\|\bw_0\|_{H^{\frac32+}_x} \leq  C\|\bw_0 \|^2_{{H}^{2}_x} + C\|\rho_0\|^2_{H^{\sstar}_x}.
\end{split}
\end{equation}
For $G_1$, we have
\begin{equation}\label{DC2}
\begin{split}
  |G_1| \lesssim & \int^t_0 \|\partial \bv\|_{L^\infty_x}\|\mathrm{curl} \mathrm{curl} \bw + \Omega\|^2_{L^2_x}d\tau
  \\
  \lesssim & \int^t_0 \|\partial \bv\|_{L^\infty_x}\left\{ \|\bw\|^2_{{H}^{2}_x}+\|\bv\|^2_{H^{\sstar}_x}+\|\rho\|^2_{H^{\sstar}_x} \right\}d\tau.
\end{split}
\end{equation}
For $G_2$, recalling the expression of \eqref{rF} and using Lemma \ref{lpe}, we can obtain
\begin{equation}\label{DC3}
\begin{split}
  |{G_2}|
  \lesssim & \int^t_0 \|\Gamma\|_{L^2_x}  \|\mathrm{curl} \mathrm{curl} \bw + \Omega\|_{L^2_x}d\tau
  \\
  \lesssim %& C\int^t_0 \| \partial \bv \|_{\dot{B}^{s_0-2}_{\infty,2}} \|\bw\|_{{H}^{s_0}_x}( \|\bw\|_{{H}^{s_0}_x}+\|\bv\|_{H^s_x}+\|\rho\|_{H^s_x}+\|h\|_{H^{s_0+1}_x} ) d\tau
%  \\
  &  \int^t_0 (\| \partial \bv \|_{L^\infty_x} + \| \partial \rho \|_{L^\infty_x} )( \|\bw\|^2_{{H}^{2}_x}+\|\bv\|^2_{H^{\sstar}_x}+\|\rho\|^2_{H^{\sstar}_x} ) d\tau + \int^t_0 \|\bw\|^2_{{H}^{2}_x} d\tau
  \\
  &  +\int^t_0 (\| \partial \bv \|_{L^\infty_x} + \| \partial \rho \|_{L^\infty_x} )(\| \rho\|_{ H^2_x }+\| \bv\|_{ H^2_x } )  (\| \bw \|^2_{{H}^{2}_x} + \| \rho\|^2_{ H^{\sstar}_x } +\|\bv\|^2_{H^{\sstar}_x }) d\tau
  \\
  &  +\int^t_0 (\| \partial \bv \|_{L^\infty_x} + \| \partial \rho \|_{L^\infty_x} )(\| \rho\|^2_{ H^2_x }+\| \bv\|^2_{ H^2_x } )
   (\| \bw \|^2_{{H}^{2}_x} + \| \rho\|^2_{ H^{\sstar}_x } +\|\bv\|^2_{H^{\sstar}_x }) d\tau
  \\
  &  +\int^t_0 (\| \partial \bv \|_{L^\infty_x} + \| \partial \rho \|_{L^\infty_x} )(\| \rho\|^3_{ H^2_x }+\| \bv\|^3_{ H^2_x } ) (\| \bw \|^2_{{H}^{2}_x} + \| \rho\|^2_{ H^{\sstar}_x } +\|\bv\|^2_{H^{\sstar}_x }) d\tau.
\end{split}
\end{equation}
For $G_3$, we divide it into $G_3=G_{31}+G_{32}$. Here
\begin{equation*}
 G_{31}=2\int^t_0\int_{\mathbb{R}^3} \partial^i \big(   \partial_n v_a \partial^n w^a \big) \cdot ( \mathrm{curl} \mathrm{curl} \bw_i) dxd\tau, \ \ \ \ \ \ \ \ \ \ \ \ \ \ \ \ \ \ \ \ \ \ \
\end{equation*}
\begin{equation*}
\begin{split}
  G_{32}=&-4\int^t_0\int_{\mathbb{R}^3} \partial^i \big(  \partial_n v_a \partial^n w^a \big) ( \partial^a \rho \partial_i w_a ) dxd\tau \ \ \ \ \ \ \ \ \ \ \ \ \ \ \ \ \ \ \
  \\
  & -2\int^t_0\int_{\mathbb{R}^3} \partial^i \big(   \partial_n v_a \partial^n w^a \big) ( \epsilon^{ijk}\partial_j \rho \mathrm{curl}\bw_k ) dxd\tau.
\end{split}
\end{equation*}
Integrating $G_{31}$ by parts, we have
\begin{equation}\label{G31}
  G_{31}=-2\int^t_0\int_{\mathbb{R}^3}  \big(   \partial_n v_a \partial^n w^a \big) \cdot \partial^i( \mathrm{curl} \mathrm{curl} \bw_i) dxd\tau
  = 0 .
\end{equation}
Here we use the fact $\partial^i( \mathrm{curl} \mathrm{curl} \bw_i)=\mathrm{div}(\mathrm{curl} \mathrm{curl} \bw)=0$. Due to the Plancherel formula, we have
\begin{equation*}
\begin{split}
  G_{32}=&-4\int^t_0\int_{\mathbb{R}^3} \Lambda_x^{-\frac12}\partial^i \big(   \partial_n v_a \partial^n w^a \big) \cdot
  \Lambda_x^{\frac12}(\partial^a \rho \partial_i w_a ) dxd\tau
  \\
  & -2\int^t_0\int_{\mathbb{R}^3} \Lambda_x^{-\frac12} \partial^i \big(   \partial_n v_a \partial^n w^a \big) \cdot \Lambda_x^{\frac12}( \epsilon^{ijk}\partial_j \rho \mathrm{curl}\bw_k ) dxd\tau.
\end{split}
\end{equation*}
We then use H\"older's inequality to bound $G_{32}$
\begin{equation}\label{G42}
\begin{split}
  |G_{32}| & \leq C\int^t_0 \| \partial \bv \partial \bw \|_{\dot{H}_x^{\frac12}}\| \partial \rho \partial \bw \|_{\dot{H}_x^{\frac12}}d\tau
  \\
  & \leq C t\sup_{\tau\in[0,t]} ( \| \bv\|^2_{{H}_x^{{\sstar}}}(\tau) + \| \bw\|^2_{{H}_x^{2}}(\tau)+\| \rho\|^2_{{H}_x^{{\sstar}}}(\tau) ).
\end{split}
\end{equation}
On the other hand, we have
\begin{equation}\label{DC4}
\begin{split}
  \|\mathrm{curl} \mathrm{curl} \bw \|^2_{L^2_x} \leq 2\|\mathrm{curl} \mathrm{curl} \bw + \Omega\|^2_{L^2_x} + 2\|\Omega\|^2_{L^2_x}.
\end{split}
\end{equation}
Combining \eqref{DC1}, \eqref{DC2}, \eqref{DC3}, \eqref{G31}-\eqref{G42}, we get the following estimate
\begin{equation}\label{DC5}
\begin{split}
  & \|\mathrm{curl} \mathrm{curl} \bw \|^2_{L^2_x}(t)
  \\
  \leq & C\|\bw_0 \|^2_{{H}^{2}}+ C\|\rho_0\|_{H^{{\sstar}}}\|\bw_0\|_{H^{\frac32+}}
  \\
  &+ C t( \| \bv\|^2_{L^\infty_t{H}_x^{{\sstar}}} + \| \bw\|^2_{L^\infty_t{H}_x^{2}}+ \| \rho\|^2_{L^\infty_t{H}_x^{{\sstar}}})
  \\
  & +C\int^t_0 \|\partial \bv\|_{L^\infty_x}( \|\bw\|^2_{{H}^{2}_x}+\|\bv\|^2_{H^{\sstar}_x}+\|\rho\|^2_{H^{\sstar}_x} ) d\tau
  \\
  & + C\int^t_0 \| (\partial \bv , \partial \rho ) \|_{L^\infty_x}( \|\bw\|^2_{{H}^{{\sstar}}_x}+\|\bv\|^2_{H^{\sstar}_x}+\|\rho\|^2_{H^{\sstar}_x} ) d\tau
  \\
  &  +C\int^t_0 \| (\partial \bv , \partial \rho ) \|_{L^\infty_x}\| (\rho, \bv)\|_{ H^2_x } (\| \bw \|^2_{{H}^{2}_x} + \| \rho\|^2_{ H^{\sstar}_x } +\|\bv\|^2_{H^{\sstar}_x }) d\tau
  \\
  &  +C\int^t_0 \| (\partial \bv , \partial \rho ) \|_{L^\infty_x}\| (\rho, \bv)\|^2_{ H^2_x }
   (\| \bw \|^2_{{H}^{2}_x} + \| \rho\|^2_{ H^{\sstar}_x } +\|\bv\|^2_{H^{\sstar}_x} ) d\tau
  \\
  &  +C\int^t_0 \| ( \partial \bv , \partial \rho ) \|_{L^\infty_x} \| ( \rho, \bv ) \|^3_{ H^2_x }
  (\| \bw \|^2_{{H}^{2}_x} + \| \rho\|^2_{ H^{\sstar}_x } +\|\bv\|^2_{H^{\sstar}_x } ) d\tau .
\end{split}
\end{equation}
Adding \eqref{DC5} with \eqref{DW}, we can prove
\begin{equation}\label{DC6}
\begin{split}
   \|\bw \|^2_{\dot{H}^{2}_x}(t)
   \leq & C\|\bw_0 \|^2_{{H}^{2}_x}+ C\|\rho\|^2_{L^\infty_tH^{{\sstar}}_x}\|\bw\|^2_{L^\infty_t H^{\frac32+}_x}
  \\
  &+ C t( \| \bv\|^2_{L^\infty_t{H}_x^{{\sstar}}} + \| \bw\|^2_{L^\infty_t{H}_x^{2}}+ \| \rho\|^2_{L^\infty_t{H}_x^{{\sstar}}})
  \\
  & +C\int^t_0 \|\partial \bv\|_{L^\infty_x}( \|\bw\|^2_{{H}^{2}_x}+\|\bv\|^2_{H^{\sstar}_x}+\|\rho\|^2_{H^{\sstar}_x} ) d\tau
  \\
  & + C\int^t_0 \| (\partial \bv , \partial \rho ) \|_{L^\infty_x}( \|\bw\|^2_{{H}^{2}_x}+\|\bv\|^2_{H^{\sstar}_x}+\|\rho\|^2_{H^{\sstar}_x} ) d\tau
  \\
  &  +C\int^t_0 \| (\partial \bv , \partial \rho ) \|_{L^\infty_x}\| (\rho, \bv)\|_{ H^2_x } (\| \bw \|^2_{{H}^{2}_x} + \| \rho\|^2_{ H^{\sstar}_x } +\|\bv\|^2_{H^{\sstar}_x } ) d\tau
  \\
  &  +C\int^t_0 \| (\partial \bv , \partial \rho ) \|_{L^\infty_x}\| (\rho, \bv )\|^2_{ H^2_x }
   (\| \bw \|^2_{{H}^{2}_x} + \| \rho\|^2_{ H^{\sstar}_x } +\|\bv\|^2_{H^{\sstar}_x } ) d\tau
  \\
  &  +C\int^t_0 \| ( \partial \bv , \partial \rho ) \|_{L^\infty_x} \| ( \rho, \bv ) \|^3_{ H^2_x }
  (\| \bw \|^2_{{H}^{2}_x} + \| \rho\|^2_{ H^{\sstar}_x } +\|\bv\|^2_{H^{\sstar}_x } ) d\tau .
\end{split}
\end{equation}
By interpolation formula and Young's inequality, we can derive that
\begin{equation}\label{DC7}
\begin{split}
   C\|\rho\|^2_{L^\infty_tH^{{\sstar}}_x}\|\bw\|^2_{L^\infty_t H^{\frac32+}_x}
   \leq  & C  \|  \bw \|^{2}_{L_x^2}\|\rho\|^4_{L^\infty_tH^{{\sstar}}_x}+ \frac{1}{16} \| \bw \|^{2}_{H^{2}_x}
   \\
    \leq  & C  \|  \bv_0 \|^{2}_{H^{\sstar}}\|\rho_0\|^4_{H^{{\sstar}}} \exp(6\int^t_0 \|d\bv, d\rho\|_{L^\infty_x} d\tau)+ \frac{1}{16} \| \bw \|^{2}_{H^{2}_x}.
\end{split}
\end{equation}
Gathering \eqref{DD0}, \eqref{DB4}, \eqref{DC6} and \eqref{DC7}, and taking
\begin{equation}\label{Can}
	0< t \leq \frac{1}{10(C+1)},
\end{equation}
we have
\begin{equation*}
  \begin{split}
  \mathcal{E}(t) \leq & C \mathcal{E}(0)+ C  \|  \bv_0 \|^{2}_{H_x^{\sstar}}\|\rho_0\|^4_{L^\infty_tH^{{\sstar}}_x} \exp(6\int^t_0 \|d\bv, d\rho\|_{L^\infty_x} d\tau)
  \\
  &  + C\int^t_0 \| (\partial \bv , \partial \rho ) \|_{L^\infty_x}\mathcal{E}(\tau) d\tau
    +C\int^t_0 \| (\partial \bv , \partial \rho ) \|_{L^\infty_x}\| (\rho, \bv)\|_{ H^2_x } \mathcal{E}(\tau) d\tau
   \\
   &  +C\int^t_0 \| (\partial \bv , \partial \rho ) \|_{L^\infty_x}\| (\rho, \bv )\|^2_{ H^2_x }
   \mathcal{E}(\tau) d\tau
    +C\int^t_0 \| ( \partial \bv , \partial \rho ) \|_{L^\infty_x} \| ( \rho, \bv ) \|^3_{ H^2_x }
  \mathcal{E}(\tau) d\tau .
  \end{split}
\end{equation*}
By Gronwall's inequality, we conclude that
\begin{equation}\label{EE}
  \begin{split}
  \mathcal{E}(t) \leq & \{ C \mathcal{E}(0)+ C  \|  \bv_0 \|^{2}_{H^{\sstar}}\|\rho_0\|^4_{H^{{\sstar}}} \exp(6\int^t_0 \|d\bv, d\rho\|_{L^\infty_x} d\tau) \}F(t),
  \end{split}
\end{equation}
where
\begin{equation*}
  F(t)=\exp\{ C\int^t_0 \| (\partial \bv , \partial \rho ) \|_{L^\infty_x}(1+
    \| (\rho, \bv)\|_{ H^2_x } + \| (\rho, \bv )\|^2_{ H^2_x }
   + \| ( \rho, \bv ) \|^3_{ H^2_x }) d\tau \}.
\end{equation*}
By using \eqref{DD0}, we have
\begin{equation}\label{EEE}
	\begin{split}
  F(t) \leq & \exp\{ C\int^t_0 \| (\partial \bv , \partial \rho ) \|_{L^\infty_x} d\tau  (1+
    \| (\rho_0, \bv_0)\|_{ H^2_x }
    \\
    &\quad+ \| (\rho_0, \bv_0 )\|^2_{ H^2_x }
   + \| ( \rho_0, \bv_0 ) \|^3_{ H^2_x })\exp( C\int^t_0 \| (\partial \bv , \partial \rho ) \|_{L^\infty_x} d\tau )  \}.
\end{split}
\end{equation}
Combining \eqref{EE} and \eqref{EEE}, we can obtain \eqref{DD1}.
\end{proof}
\begin{corollary}\label{hes}
Let $(\rho,\bv)$ be a solution of \eqref{fc0} with $h=0$. Let $2<\sstar<\frac52$. Let $\bw$ be defined in \eqref{pw1}. Then the following energy estimates hold:
\begin{equation}\label{hes0}
\begin{split}
\mathbb{E}(t) \lesssim  \mathbb{E}_0 \exp \{ \mathbb{E}_1 \int^t_0 \|(d \bv, d \rho)\|_{L^\infty_x} d\tau \cdot \exp ( \int^t_0  \|(d \bv, d \rho) \|_{L^\infty_x}  d\tau )\},
\end{split}
\end{equation}
where $\mathbb{E}(t)$, $\mathbb{E}_0$ and $\mathbb{E}_1$ are defined by
\begin{equation*}
\begin{split}
  \mathbb{E}(t)=& \|\bv\|^2_{H_x^{\sstar+1}} + \|\rho\|^2_{H_x^{\sstar+1}}+ \|\bw\|^2_{H_x^{3}},
\\
  \mathbb{E}_0 = &   \|\bv_0\|^2_{H_x^{\sstar+1}} + \|\rho_0\|^2_{H_x^{\sstar+1}}+ \|\bw_0\|^2_{H_x^{3}}+   \|  \bv_0 \|^{2}_{H^{\sstar}}\|\rho_0\|^4_{H^{\sstar}} ,
  \\
  \mathbb{E}_1 = &  (1+
    \| \bv_0 \|_{ H^3 } + \| \rho_0\|_{ H^3 }).
   \end{split}
\end{equation*}
\end{corollary}
\begin{proof}
By using \eqref{E0}, we have
\begin{equation}\label{hes1}
 \| ({\rho},\bv)\|^2_{H_x^{\sstar+1}}(t)\leq C\| (\rho_0,\bv_0)\|^2_{H_x^{\sstar+1}} \exp\{ \int^t_0 \|(d\bv, d\rho)\|_{L^\infty_x} d\tau \} .
\end{equation}
Multiplying \eqref{DB0} by $\bw$ and integrating over $ [0,t]\times \mathbb{R}^3$, we  get
\begin{equation}\label{hes2}
  \| \bw \|^2_{L_x^2} \leq \| \bw_0 \|^2_{L_x^2} + C\int^t_0 \| \partial \bv\|_{L^\infty_x}\|\bw\|^2_{L^2_x}d\tau.
\end{equation}
%\textbf{Step 2: higher-order energy estimate of $\mathrm{div} \bw$}.
Let us first estimate $\| \mathrm{div} \bw \|_{\dot{H}^{2}_x}$. Note \eqref{W01}. Using product estimates in Lemma \ref{ps}, we can derive that
\begin{equation}\label{hes3}
\begin{split}
  \| \mathrm{div} \bw \|_{{H}^{2}_x} &= \|  \bw \cdot \partial \rho  \|_{{H}^{2}_x} \leq C \|\bw\|_{H_x^{\frac52+}}\|\rho\|_{{H}^{\sstar+1}_x}.
\end{split}
\end{equation}
Let us next estimate $\| \mathrm{curl} \bw \|_{\dot{H}^2_x}$. By using \eqref{W00}, we obtain
\begin{equation}\label{hes4}
  \begin{split}
  \mathbf{T}(\Delta \mathrm{curl}\bw^i )
  =& - \epsilon^{ijk} \Delta(\partial_j v^m \partial_m w_k)+ \epsilon^{ijk}\Delta(\partial_m v_k \partial_j w^m)+  \Delta\{ w^m \partial_m ( \mathrm{e}^\rho w^i) \}
  \\
  & - \Delta v^m \partial_m \mathrm{curl}\bw^i- \partial^k v^m \partial_k \partial_m \mathrm{curl}\bw^i
  \end{split}
\end{equation}
Multiplying \eqref{hes4} by $\Delta \mathrm{curl}\bw_i$ and integrating over $ [0,t]\times \mathbb{R}^3$, we have
\begin{equation}\label{hes5-a}
  \begin{split}
  & \| \Delta \mathrm{curl}\bw \|^2_{L^2_x}-\| \Delta \mathrm{curl}\bw_0 \|^2_{L^2_x}
  \\
  \leq &   C\int^t_0 \| \partial \bv \partial \bw\|_{H^2_x}\| \Delta \mathrm{curl}\bw \|_{L^2_x}d\tau  + C\int^t_0  \|\partial \bv \partial^3 \bw\|_{L^2_x} \| \Delta \mathrm{curl}\bw \|_{L^2_x}d\tau
  \\
  & + C\int^t_0 \| \bw  \partial( \mathrm{e}^\rho \bw) \|_{H^2_x} \| \Delta \mathrm{curl}\bw \|_{L^2_x}d\tau
   +C\int^t_0 \| \partial^2 \bv \partial^2 \bw\|_{L^2_x} \| \Delta \mathrm{curl}\bw \|_{L^2_x}d\tau.
  \end{split}
\end{equation}
Using H\"older's inequality, \eqref{hes5-a} yields
\begin{equation}\label{hes5-b}
  \begin{split}
 \| \Delta \mathrm{curl}\bw \|^2_{L^2_x}
  \leq & \|\bw_0 \|^2_{H^3_x}+  C\int^t_0 \|\bw\|_{H^{3}_x}\|\bv\|_{H^{3}_x} \|\bw\|_{H^{\frac52+}_x}d\tau  + C\int^t_0  \|\partial \bv\|_{L^\infty_x} \| \bw\|^2_{H^3_x} d\tau
  \\
  & + C\int^t_0 \| \partial \bv\|_{L^\infty_x} \| \rho \|_{H^3_x} \| \bw \|_{H^3_x}d\tau
   +C \int^t_0 \| \bv \|_{H^3_x} \bw\|_{H^{\frac52+}_x} \|\bw \|_{H^3_x}d\tau
  \\
  & +  C\int^t_0 \| \partial \rho \|_{L^\infty_x} \| \bw \|_{H^2_x} \| \bw \|_{H^{\frac32+}_x} \| \bw \|_{H^3_x}d\tau.
  \end{split}
\end{equation}
Due to \eqref{hes3}, \eqref{hes5-b} and Young's inequality, we can show that
\begin{equation}\label{hes5-c}
  \begin{split}
 \| \bw \|^2_{\dot{H}^3_x}
  \leq & \|\bw_0 \|^2_{H^3_x}+\|\bv\|^2_{H_x^{\sstar+1}}\|\rho\|^4_{{H}^{\sstar+1}_x}
       + C\int^t_0  \|\partial \bv\|_{L^\infty_x} ( \| \bw\|^2_{H^3_x} + \| \rho \|^2_{H^3_x} ) d\tau
  \\
  &
   +C \int^t_0 \| \bv \|_{H^3_x}  \|\bw \|^2_{H^3_x}d\tau
  +  C\int^t_0 \| \partial \rho \|_{L^\infty_x} \| \bv \|_{H^{3}_x} \| \bw \|^2_{H^3_x}d\tau.
  \end{split}
\end{equation}
Using \eqref{hes1}, \eqref{hes2}, \eqref{hes5-c}, and Gronwall's inequality, we get \eqref{hes0}.
\end{proof}
\subsection{Energy estimates of Theorem \ref{dingli2}}\label{ES3}
\begin{theorem}\label{vve}%{(Energy estimates: type 2)}
	Let $(\rho,\bv,h)$ be a solution of \eqref{fc0}. Let $\bw$ be defined in \eqref{pw11} and $\bw$ satisfy the equation \eqref{W0}. Then the following energy estimates hold:
	\begin{equation}\label{hh0}
		\begin{split}
			\tilde{E}(t) \leq & C\tilde{E}_0 \exp \{ \tilde{E}_1 \int^t_0 \|d \bv, d \rho, d h)\|_{L^\infty_x} d\tau \}.
		\end{split}
	\end{equation}
	Here $\tilde{E}(t)$, $\tilde{E}_0$ and $\tilde{E}_1$ are defined by
	\begin{equation}\label{hh1}
		\tilde{E}(t)=\|\bv\|^2_{H_x^{\frac52}} + \|\rho\|^2_{H_x^{\frac52}}+\|h\|^2_{H_x^{\frac52+\epsilon}}+ \|\bw\|^2_{H_x^{\frac32+\epsilon}},
	\end{equation}
	and
	\begin{equation}\label{hh2}
		\begin{split}
			\tilde{E}_0  = & \tilde{E}(0)
			+   \| \rho_0 \|^2_{H^{\frac52}} \| h_0 \|^2_{H^{\frac52}} \exp\{   \int^t_0 \|(d \bv, d \rho, d h)\|_{L^\infty_x} d\tau \},
		\end{split}
	\end{equation}
	and
	\begin{equation}\label{hh3}
		\begin{split}
			\tilde{E}_1  =
			&   ( \|( \bv_0,  \rho_0 ,  h_0) \|_{H^{\frac52}} + \|( \bv_0,  \rho_0,  h_0) \|^2_{H^{\frac52}}  ) \exp\{   \int^t_0 \|(d \bv, d \rho, d h)\|_{L^\infty_x} d\tau \}  .
		\end{split}
	\end{equation}
\end{theorem}
\begin{proof}
	To prove \eqref{hh0}, we divide it into three steps.

	\textbf{Step 1: The energy estimate for $\bw$.} By elliptic estimates, we obtain that
	\begin{equation}\label{ww1}
		\| \bw \|^2_{\dot{H}^{\frac{3}{2}+\epsilon}_x} \leq 2\| \mathrm{curl} \bw \|^2_{\dot{H}^{\frac{1}{2}+\epsilon}_x}+2\| \mathrm{div} \bw\|^2_{\dot{H}^{\frac{1}{2}+\epsilon}_x}.
	\end{equation}
	Note \eqref{Wd1}. By H\"older's inequality and Young's inequality, we can bound $\| \mathrm{div} \bw\|_{\dot{H}^{\frac{1}{2}+\epsilon}_x}$ by
	\begin{equation}\label{ww2}
		\begin{split}
			\| \mathrm{div} \bw \|^2_{\dot{H}^{\frac{1}{2}+\epsilon}_x}
			\leq  & C\| \bw \|^2_{{H}^{1+\epsilon}_x} \| \rho\|^2_{{H}^{2+\epsilon}_x}
			\\
			\leq & \frac{1}{16}\| \bw \|^2_{{H}^{\frac{3}{2}+\epsilon}_x}+ C\| \bw\|^2_{L^{2}_x} \| \rho\|^6_{{H}^{2+\epsilon}_x}.
		\end{split}
	\end{equation}
	On the other hand, we have
	\begin{equation}\label{e010}
		\begin{split}
			\| \mathrm{curl}\bw\|^2_{\dot{H}^{\frac{1}{2}+\epsilon}_x} \leq & 2\| \mathrm{curl}\bw-\frac{1}{\gamma}\partial h \partial \bv+ \frac{1}{\gamma}\partial h \mathrm{div}\bv \|^2_{\dot{H}^{\frac{1}{2}+\epsilon}_x}+ 2 \| \frac{1}{\gamma}\partial h \partial \bv- \frac{1}{\gamma}\partial h \mathrm{div}\bv \|^2_{\dot{H}^{\frac{1}{2}+\epsilon}_x}
			\\
			\leq & 2\| \mathrm{curl}\bw-\frac{1}{\gamma}\partial h \partial \bv+ \frac{1}{\gamma}\partial h \mathrm{div}\bv \|^2_{\dot{H}^{\frac{1}{2}+\epsilon}_x}
			+ C \| h \|^2_{{H}^{2+\epsilon}_x} \| \bv \|^2_{{H}^{2+\epsilon}_x}.
		\end{split}
	\end{equation}
	By Lemma \ref{PW}, Lemma \ref{ce}, Lemma \ref{jh}, H\"older's inequality and \eqref{YL3}, \eqref{YL4}, we obtain
	\begin{equation}\label{ww4}
		\begin{split}
			& \frac{d}{dt} \| \mathrm{curl}\bw-\frac{1}{\gamma}\partial h \partial \bv+ \frac{1}{\gamma}\partial h \mathrm{div}\bv \|^2_{\dot{H}^{\frac{1}{2}+\epsilon}_x}
			\\
			\lesssim  & \|\partial \bv \|_{L^\infty_x} \| \mathrm{curl}\bw-\frac{1}{\gamma}\partial h \partial \bv+ \frac{1}{\gamma}\partial h \mathrm{div}\bv \|^2_{\dot{H}^{\frac{1}{2}+\epsilon}_x}
			\\
			& + \|\partial \bv \partial \bw \|_{\dot{H}^{\frac{1}{2}+\epsilon}_x} \| \mathrm{curl}\bw-\frac{1}{\gamma}\partial h \partial \bv+ \frac{1}{\gamma}\partial h \mathrm{div}\bv \|_{\dot{H}^{\frac{1}{2}+\epsilon}_x}
			\\
			& + \|(\partial \rho , \partial h)\cdot \partial^2 h \|_{\dot{H}^{\frac{1}{2}+\epsilon}_x} \| \mathrm{curl}\bw-\frac{1}{\gamma}\partial h \partial \bv+ \frac{1}{\gamma}\partial h \mathrm{div}\bv \|_{\dot{H}^{\frac{1}{2}+\epsilon}_x}
			\\
			& + \|(\partial \rho , \partial h, \partial \bv)\cdot (\partial \rho , \partial h, \partial \bv)\cdot (\partial \rho , \partial h, \partial \bv) \|_{\dot{H}^{\frac{1}{2}+\epsilon}_x} \| \mathrm{curl}\bw-\frac{1}{\gamma}\partial h \partial \bv+ \frac{1}{\gamma}\partial h \mathrm{div}\bv \|_{\dot{H}^{\frac{1}{2}+\epsilon}_x}
			\\
			\lesssim  & \|\partial \bv \|_{L^\infty_x} \|  \partial \bw \|^2_{{H}^{\frac{1}{2}+\epsilon}_x} + \|\partial \bv \|_{H^{\frac32}_x} \|  \partial \bw \|^2_{{H}^{\frac{1}{2}+\epsilon}_x} + ( \|\partial \rho \|_{L^\infty_x}+ \|\partial h \|_{L^\infty_x})  \|\partial^2 h \|_{\dot{H}^{\frac{1}{2}+\epsilon}_x}  \|\partial \bw \|_{\dot{H}^{\frac{1}{2}+\epsilon}_x}
			\\
			& +  ( \|\partial \rho \|_{H^{\frac32}_x}+ \|\partial h \|_{H^{\frac32}_x})  \|\partial^2 h \|_{\dot{H}^{\frac{1}{2}+\epsilon}_x}  \|\partial \bw \|_{\dot{H}^{\frac{1}{2}+\epsilon}_x}
			+\|\partial \bv \|_{L^\infty_x} \|\partial h \|_{L^6_x} \| \partial \bv \|_{W^{\frac{1}{2}+\epsilon,3}_x}\|\partial \bw \|_{\dot{H}^{\frac{1}{2}+\epsilon}_x}
			\\
			&
			+ \|\partial \bv \|_{L^\infty_x} \|\partial \rho \|_{L^6_x} \|\partial \bv \|_{W^{\frac{1}{2}+\epsilon,3}_x}\|\partial \bw \|_{\dot{H}^{\frac{1}{2}+\epsilon}_x} +   \|\partial \rho \|_{L^\infty_x} \|\partial h \|_{L^6_x} \| \partial \bv \|_{W^{\frac{1}{2}+\epsilon,3}_x}\|\partial \bw \|_{\dot{H}^{\frac{1}{2}+\epsilon}_x}
			\\
			& + \|\partial \bv, \partial \rho, \partial h \|_{L^\infty_x}
			\|\partial \bv , \partial \rho , \partial h \|_{\dot{H}^{\frac{1}{2}+\epsilon}_x}(\|\partial \bv \|^3_{L^6_x}+\|\partial \rho \|^3_{L^6_x}+\|\partial h \|^3_{L^6_x})
			\\
			\lesssim & \|\partial \bv \|_{L^\infty_x} \|   \bw \|^2_{{H}^{\frac{3}{2}+\epsilon}_x} +\|\partial \rho \|_{L^\infty_x} \| h \|_{{H}^{\frac{5}{2}+\epsilon}_x}  \| \bw \|_{{H}^{\frac{3}{2}+\epsilon}_x}
			\\
			& +  \|(\bv, h, \rho) \|_{H^{\frac52}_x} ( \|  h \|_{{H}^{\frac{5}{2}+\epsilon}_x} + \|  \bw \|_{{H}^{\frac{3}{2}+\epsilon}_x} )  \|  \bw \|_{{H}^{\frac{3}{2}+\epsilon}_x}
			\\
			&+ \|\partial \bv \|_{L^\infty_x} (\| h \|_{H^2_x}+\| \rho \|_{H^2_x}) \|\bv \|_{H^{2}_x}\| \bw \|_{{H}^{\frac{3}{2}+\epsilon}_x}
			\\
			& + (\|\partial \bv\|_{L^\infty_x}+\|\partial \rho\|_{L^\infty_x}+\|\partial h \|_{L^\infty_x}) (\| \bv \|^4_{{H}^{2}_x}+\|\rho \|^4_{{H}^{2}_x}+\| h \|^4_{{H}^{2}_x}).
		\end{split}
	\end{equation}
	Integrating \eqref{ww4} over $[0,t]$, we find that
	\begin{equation*}
		\begin{split}
			\| \mathrm{curl}\bw-\frac{1}{\gamma}\partial h \partial \bv+ \frac{1}{\gamma}\partial h \mathrm{div}\bv \|^2_{\dot{H}^{\frac{1}{2}+\epsilon}_x}
			\leq & \|\bw_0\|^2_{{H}^{\frac{3}{2}+\epsilon}_x}
			+C \| h_0 \|^2_{{H}^{2+\epsilon}_x} \| \bv_0 \|^2_{{H}^{2+\epsilon}_x}
			\\
			& +C\int^t_0 \|\partial \bv \|_{L^\infty_x} \tilde{E}(\tau) d\tau
			\\
			& +C\int^t_0 \|(\partial \bv, \partial \rho, \partial h) \|_{L^\infty_x} \|( \bv,  \rho, h) \|_{H^{2}_x}\tilde{E}(\tau) d\tau
			\\
			& +C\int^t_0 \|(\partial \bv, \partial \rho, \partial h) \|_{L^\infty_x} \|( \bv,  \rho, h) \|^2_{H^{2}_x}\tilde{E}(\tau) d\tau .
		\end{split}
	\end{equation*}
	The above estimate combining with \eqref{e010} yields
	\begin{equation}\label{e011}
		\begin{split}
			& \| \mathrm{curl}\bw\|^2_{\dot{H}^{\frac{1}{2}+\epsilon}_x}
			\\
			\leq & C\|\bw_0\|^2_{{H}^{\frac{3}{2}+\epsilon}_x}+C \| h_0 \|^2_{{H}^{2+\epsilon}_x} \| \bv_0 \|^2_{{H}^{2+\epsilon}_x}+C \| h \|^2_{{H}^{2+\epsilon}_x} \| \bv \|^2_{{H}^{2+\epsilon}_x}
			\\
			& +C\int^t_0 \|\partial \bv \|_{L^\infty_x} \tilde{E}(\tau) d\tau+C\int^t_0 \|(\partial \bv, \partial \rho, \partial h) \|_{L^\infty_x} \|( \bv,  \rho, h) \|_{H^{2}_x}\tilde{E}(\tau) d\tau
			\\
			& +C\int^t_0 \|(\partial \bv, \partial \rho, \partial h) \|_{L^\infty_x} \|( \bv,  \rho, h) \|^2_{H^{2}_x}\tilde{E}(\tau) d\tau .
		\end{split}
	\end{equation}

	\quad \textbf{Step 2: The energy estimate for $h$.} Let us now estimate $\| \Delta h \|_{\dot{H}^{\frac12+\epsilon}_x}$. Note $\Delta h= \mathrm{e}^\rho \partial_i H^i+ \partial_i \rho \partial^i h$. Thus, we have
	\begin{equation}\label{e014}
		\| \Delta h \|^2_{\dot{H}^{\frac12+\epsilon}_x} \leq 2 \|\mathrm{e}^{\rho} \partial_i H^i\|^2_{\dot{H}^{\frac12+\epsilon}_x} + 2 \|\partial_i \rho \partial^i h\|^2_{\dot{H}^{\frac12+\epsilon}_x}.
	\end{equation}
	Notice
	\begin{equation*}
		\|\partial_i \rho \partial^i h\|^2_{\dot{H}^{\frac12+\epsilon}_x} \leq C\|\rho \|^2_{{H}^{2+\epsilon}_x}\|\bv \|^2_{{H}^{2+\epsilon}_x}.
	\end{equation*}
	Taking derivatives $\Lambda^{\frac12+\epsilon}_x$ of \eqref{Hh1}, we have
	\begin{equation}\label{ww9}
		\begin{split}
			& \mathbf{T}\{ \Lambda^{\frac12+\epsilon}_x( \mathrm{e}^{\rho} \partial_i H^i)\} 
			\\
			=& -2\Lambda^{\frac12+\epsilon}_x(\partial_i v^j \partial_j  \partial^i h)   + \Lambda^{\frac12+\epsilon}_x( \epsilon^{ijm} \mathrm{e}^{\rho} w_m \partial^i \rho \partial_j h )+[\mathbf{T}, \Lambda^{\frac12+\epsilon}_x](\mathrm{e}^{\rho} \partial_i H^i).
		\end{split}
	\end{equation}
	Multiplying $\Lambda^{\frac12+\epsilon}_x( \mathrm{e}^{\rho} \partial_i H^i)$ by \eqref{ww9} and integrating over $[0,t] \times \mathbb{R}^3$ gives
	\begin{equation*}
		\begin{split}
			\|\mathrm{e}^{\rho} \partial_i H^i\|^2_{\dot{H}^{\frac12+\epsilon}_x} \leq &  C\|h_0\|^2_{H^{\frac52+\epsilon}}+C\|\rho_0 \|^2_{{H}^{2+\epsilon}_x}\|\bv_0 \|^2_{{H}^{2+\epsilon}_x}
			+ C\int^t_0  \|\partial \bv \|_{L^\infty_x} \cdot \|\mathrm{e}^{\rho} \partial_i H^i\|^2_{\dot{H}^{\frac12+\epsilon}_x}d\tau
			\\
			& +C \int^t_0   \| \partial \bv  \partial^2 h\|_{\dot{H}^{\frac12+\epsilon}_x}  \|\mathrm{e}^{\rho} \partial_i H^i\|_{\dot{H}^{\frac12+\epsilon}_x} d\tau
			+C \int^t_0  \| \bw \partial h  \partial \rho\|_{\dot{H}^{\frac12+\epsilon}_x} \|\mathrm{e}^{\rho} \partial_i H^i\|_{\dot{H}^{\frac12+\epsilon}_x} d\tau.
		\end{split}
	\end{equation*}
By H\"older's inequality, we obtain
		\begin{equation*}
		\begin{split}
		\|\mathrm{e}^{\rho} \partial_i H^i\|^2_{\dot{H}^{\frac12+\epsilon}_x} 
			\leq & \| h_0 \|^2_{{H}^{\frac52+\epsilon}_x} + C\|\rho_0 \|^2_{{H}^{2+\epsilon}_x}\|\bv_0 \|^2_{{H}^{2+\epsilon}_x}
			+ C\int^t_0  \|\partial \bv \|_{L^\infty_x}  \|h\|^2_{{H}^{\frac52+\epsilon}_x}d\tau
			\\
			& +C\int^t_0  \|\partial \bv \|_{L^\infty_x} \|\rho\|^2_{{H}^{2+\epsilon}_x} \|h\|^2_{{H}^{2+\epsilon}_x}d\tau
			+C\int^t_0  \| \bv \|_{H^{\frac52}_x}  \|h\|^2_{{H}^{\frac52+\epsilon}_x}d\tau
			\\
			&+C\int^t_0  \|\partial \bv \|_{L^\infty_x} \|\rho\|_{{H}^{2+\epsilon}_x}\|h\|_{{H}^{2+\epsilon}_x} \|h\|_{{H}^{\frac52+\epsilon}_x}d\tau
			\\
			&+C\int^t_0  \| \bv \|_{H^{\frac52}_x} \|\rho\|_{{H}^{2+\epsilon}_x}\|h\|_{{H}^{2+\epsilon}_x} \|h\|_{{H}^{\frac52+\epsilon}_x}d\tau
			\\
			& +C\int^t_0  \|(\partial \rho, \partial h) \|_{L^\infty_x} \|(h, \rho)\|^2_{{H}^{2}_x} \|\bw\|_{{H}^{\frac32+\epsilon}_x}\|h\|_{{H}^{\frac52+\epsilon}_x}d\tau
			\\
			& +C\int^t_0  \|(\partial \rho, \partial h) \|_{L^\infty_x} \|(h, \rho)\|^3_{{H}^{2}_x} \|\bw\|_{{H}^{\frac32+\epsilon}_x}d\tau .
		\end{split}
	\end{equation*}
	Therefore, we can conclude that
	\begin{equation}\label{e015}
		\begin{split}
			\| \Delta h \|^2_{\dot{H}^{\frac12+\epsilon}_x} \leq & C\|h_0\|^2_{H^{\frac52+\epsilon}}+C\|\rho_0 \|^2_{{H}^{2+\epsilon}_x}\|\bv_0 \|^2_{{H}^{2+\epsilon}_x}
			\\
			& +C\int^t_0   \|(\partial \bv, \partial h, \partial \rho) \|_{L^\infty_x} \tilde{E}(\tau) d\tau
			\\
			& + C\int^t_0  \|( \bv,  h,  \rho) \|_{H^{\frac52}_x}\tilde{E}(\tau) d\tau + + C\int^t_0  \|( \bv,  h,  \rho) \|^2_{H^{\frac52}_x}\tilde{E}(\tau) d\tau
			\\
			& + C\int^t_0  \|(\partial \bv, \partial h, \partial \rho) \|_{L^\infty_x} \|( \bv,  h,  \rho) \|_{H^2_x}\tilde{E}(\tau) d\tau
			\\
			& +C\int^t_0  \|(\partial \bv, \partial h, \partial \rho) \|_{L^\infty_x} \|( \bv,  h,  \rho) \|^2_{H^2_x}\tilde{E}(\tau) d\tau .
		\end{split}
	\end{equation}
	
	\textbf{Step 3: The total energy estimates.} By gathering \eqref{ww2}, \eqref{e011}, and \eqref{e015}, we have
	
	\begin{equation*}
		\begin{split}
			\tilde{E}(t) \leq & C\tilde{E}(0) + C\int^t_0 \{  \|( \bv,  h,  \rho) \|_{H^{\frac52}_x} + \|( \bv,  h,  \rho) \|^2_{H^{\frac52}_x}  \} \tilde{E}(\tau) d\tau
			\\
			&
			+ C  \| \rho_0 \|^2_{H^{\frac52}} \| h_0 \|^2_{H^{\frac52}} \exp\{ C  \int^t_0 \|(d \bv, d \rho, d h)\|_{L^\infty_x} d\tau \}
			\\
			& + C\int^t_0  \|(\partial \bv, \partial h, \partial \rho) \|_{L^\infty_x} \{ \|( \bv,  h,  \rho) \|_{H^2_x}+\|( \bv,  h,  \rho) \|^2_{H^2_x} \} \tilde{E}(\tau) d\tau .
		\end{split}
	\end{equation*}
	By Gronwall's inequality, we can prove that
	\begin{equation*}
		\begin{split}
			\tilde{E}(t) \leq & \tilde{E}_0 \exp \{ \tilde{E}_1 \int^t_0 \|(d \bv, d \rho, d h)\|_{L^\infty_x} d\tau \},
		\end{split}
	\end{equation*}
	where $\tilde{E}_0$ and $\tilde{E}_1$ are stated in \eqref{hh2} and \eqref{hh3}.
\end{proof}
\subsection{Uniqueness of the solution}
%By using the above energy estimates, we have
\begin{corollary}\label{cor}
Assume $2<s_0<s\leq\frac52$ and \eqref{HEw}. Suppose $(\bv_1,\rho_1, h_1)$ and $(\bv_2, \rho_2, h_2)$ to be solutions of \eqref{CEE} with the same initial data $(\bv_0, \rho_0, h_0) \in H^{s} \times H^s \times H^{s_0+1}$. We assume the initial specific vorticity $\bw_0=\bar{\rho}\mathrm{e}^{\rho_0} \mathrm{curl}\bv_0 \in H^{s_0}$. Then there exists a constant $T>0$ such that
  $(\bv_1,\rho_1, \bv_2, \rho_2) \in C([0,T],H_x^s) \cap C^1([0,T],H_x^{s-1})$, $\bw \in C([0,T],H_x^{s_0}) \cap C^1([0,T],H_x^{s_0-1})$ and $(d\bv_1, d\rho_1, dh_1, d\bv_2,d\rho_2, dh_2) \in {L^2_{[0,T]} L^\infty_x}$. Furthermore, we have
\begin{equation*}
  \bv_1=\bv_2, \quad \rho_1=\rho_2, \quad h_1=h_2.
\end{equation*}
\end{corollary}
\begin{proof}
	By using Lemma \ref{sh} and Strichartz estimates $(d\bv_1, d\rho_1, dh_1, d\bv_2,d\rho_2, dh_2) \in {L^2_{[0,T]} L^\infty_x}$, we can bound
	\begin{equation*}
		\|\bv_1-\bv_2, \rho_1-\rho_2, h_1-h_2\|_{L^2_x} \lesssim \|\bv_1-\bv_2, \rho_1-\rho_2, h_1-h_2\|_{L^2_x}(0)=0.
	\end{equation*}
So the solution is unique.
\end{proof}

Similarly, we can obtain the following corollary.
\begin{corollary}\label{cor2}
Assume $2<s\leq\frac52$ and \eqref{HEw}. Suppose $(\bv_1,\rho_1)$ and $(\bv_2, \rho_2)$ to be solutions of \eqref{fc1} with the same initial data $(\bv_0, \rho_0) \in H^{s} \times H^s $. We assume the initial specific vorticity $\bw_0=\bar{\rho}\mathrm{e}^{\rho_0} \mathrm{curl}\bv_0 \in H^{2}$. Then there exists a constant $T^*>0$ such that
  $(\bv_1,\rho_1, \bv_2, \rho_2) \in C([0,T^*],H_x^s) \cap C^1([0,T^*],H_x^{s-1})$, $\bw \in C([0,T^*],H_x^{2}) \cap C^1([0,T^*],H_x^{1})$ and $(d\bv_1, d\rho_1, d\bv_2,d\rho_2) \in {L^2_{[0,T^*]} L^\infty_x}$. Furthermore, we have
\begin{equation*}
  \bv_1=\bv_2, \quad \rho_1=\rho_2.
\end{equation*}
\end{corollary}

\begin{corollary}\label{cor3}
	Assume $0<\epsilon<\frac19$ and \eqref{HEw}. Suppose $(\bv_1,\rho_1,h_1)$ and $(\bv_2, \rho_2,h_2)$ to be solutions of \eqref{CEE} with the same initial data $(\bv_0, \rho_0,h_0) \in H^{\frac52} \times H^{\frac52}\times H^{\frac52+} $. We assume the initial specific vorticity $\bw_0=\bar{\rho}\mathrm{e}^{\rho_0} \mathrm{curl}\bv_0 \in H^{\frac{3}{2}+}$. Then there exists a constant $\bar{T}>0$ such that
	$(\bv_1,\rho_1, \bv_2, \rho_2) \in C([0,\bar{T}],H_x^{\frac52}) \cap C^1([0,\bar{T}],H_x^{\frac32})$, $\bw \in C([0,\bar{T}],H_x^{\frac{3}{2}+}) \cap C^1([0,\bar{T}],H_x^{\frac{1}{2}+})$, $(d\bv_1, d\rho_1, dh_1) \in {L^2_{[0,\bar{T}]} L^\infty_x}$, and $(d\bv_2,d\rho_2,dh_2) \in {L^2_{[0,\bar{T}]} L^\infty_x}$. Furthermore, we have
	\begin{equation*}
		\bv_1=\bv_2, \quad \rho_1=\rho_2,\quad h_1=h_2.
	\end{equation*}
\end{corollary}
\section{Proof of Theorem \ref{dingli}: the existence for small, smooth, compactly supported data}\label{Sec4}
Theorem \ref{dingli} includes two parts, the existence of solutions and the continuous dependence of solutions. Of course, we should prove the existence of solutions at first. For the continuous dependence of solutions in Theorem \ref{dingli}, we will prove it in Section \ref{Sub} as a separate part. In this section, our goal is to give a reduction of the existence of solutions in Theorem \ref{dingli} to the case of small, smooth, compactly supported data by using physical localization arguments. To obtain our goal, we first state Proposition \ref{p3} for considering smooth initial data, and then we prove Theorem \ref{dingli} by assuming Proposition \ref{p3} holds. After that, we introduce Proposition \ref{p1}, and we prove Proposition \ref{p3} by assuming Proposition \ref{p1}. So we reduce the proof of Theorem \ref{dingli} to Proposition \ref{p1}, which is about the existence result for small, smooth, compactly supported data.
\subsection{A reduction to the case of smooth initial data}
\begin{proposition}\label{p3}
Let $2<s_0<s\leq\frac52$ and \eqref{HEw} hold. For each $M_0>0$, there exists $T, M>0$ (depending on $C_0, c_0, s, s_0$, and $M_0$) such that, for each smooth initial data $(\bv_0, \rho_0, h_0, \bw_0)$ which satisfies
\begin{equation}\label{p30}
	\begin{split}
	&\|\bv_0, \rho_0 \|_{H^s} + \| \bw_0\|_{H^{s_0}}+ \| h_0\|_{H^{s_0+1}}  \leq M_0,
	\end{split}
	\end{equation}
there exists a smooth solution $(\bv, \rho, h, \bw)$ to \eqref{fc1} on $[0,T] \times \mathbb{R}^3$ satisfying
\begin{equation}\label{p31}
\begin{split}
 & \|\bv, \rho\|_{L^\infty_{[0,T]}H_x^s}+\|\bw\|_{L^\infty_{[0,T]}H_x^{s_0}}+\|h\|_{L^\infty_{[0,T]} H_x^{s_0+1}} \leq M,
 \\
 & \|\bv,\rho,h\|_{L^\infty_{[0,T] \times \mathbb{R}^3 }} \leq 1+C_0.
\end{split}
\end{equation}
Furthermore, the solution satisfies the following properties:

$\mathrm{(1)}$ dispersive estimate for $\bv$, $\rho$, $h$ and $\bv_+$
	\begin{equation}\label{p32}
	\|d \bv, d \rho, dh\|_{L^2_t C^\delta_x}+\|d \bv_+, d {\rho}, d \bv, dh\|_{L^2_t \dot{B}^{s_0-2}_{\infty,2}} \leq M,
	\end{equation}

$\mathrm{(2)}$ Let $f$ satisfy equation. For each $1 \leq r \leq s+1$, the Cauchy problem \eqref{linear} is well-posed in $H_x^r \times H_x^{r-1}$. Moreover, the following energy estimate
\begin{equation}\label{p333}
	\begin{split}
		\|{f}\|_{L^\infty_{[0,T]} H^{r}_x}+ \|\partial_t {f}\|_{L^\infty_{[0,T]} H^{r-1}_x} \leq  C_{M_0}(\| {f}_0\|_{H_x^r}+ \| {f}_1\|_{H_x^{r-1}}).
	\end{split}
\end{equation}
and Strichartz estimate
	\begin{equation}\label{p33}
	\|\left< \partial \right>^k f\|_{L^2_{[0,T]} L^\infty_x} \leq  C( \| f_0\|_{H^r}+ \| f_1\|_{H^{r-1}} ),\quad  k<r-1,
	\end{equation}
holds, and the same estimates hold with $\left< \partial \right>^k$ replaced by $\left< \partial \right>^{k-1}d$. Above, $ C_{M_0}$ is a constant depending on $C_0, c_0, s, s_0, M_0$.
\end{proposition}
In the following, we will use Proposition \ref{p3} to prove Theorem \ref{dingli}.
\subsection{Proof of Theorem \ref{dingli} by Proposition \ref{p3}}
%\medskip\begin{proof}[Proof of Theorem \ref{dingli} by Proposition \ref{p3}]
Consider the initial data $(\bv_0, \rho_0, h_0, \bw_0) \in H^s \times H^s \times H^{s_0+1} \times H^{s_0}$ in Theorem \ref{dingli} and
\begin{equation*}
	\|\bv_0\|_{H^s} + \|{\rho}_0 \|_{H^s} + \|\bw_0\|_{H^{s_0}} + \|h_0\|_{H^{s_0+1}} \leq M_0,
\end{equation*}
Let $\{(\bv_{0k}, \rho_{0k}, h_{0k}, \bw_{0k})\}_{k \in \mathbb{Z}^{+}}$ be a sequence of smooth data converging to $(\bv_0, \rho_0, h_0, \bw_0)$ in $H^s \times H^s  \times H^{s_0+1} \times H^{s_0}$, where $\bw_{0k}$ is defined by
\begin{equation*}
  \bw_{0k}=\bar{\rho}^{-1}\mathrm{e}^{-\rho_{0k}}\mathrm{curl}\bv_{0k}.
\end{equation*}
By Proposition \ref{p3}, for each data $(\bv_{0k}, \rho_{0k}, h_{0k}, \bw_{0k})$, there exists a solution $(\bv_k, \rho_k, h_k, \bw_k)$ to System \eqref{fc1}, and
 \begin{equation*}
   (\bv_k, \rho_k, h_k, \bw_k)|_{t=0}=(\bv_{0k}, \rho_{0k}, h_{0k}, \bw_{0k}), \quad \bw_k=\bar{\rho}^{-1}\mathrm{e}^{-{\rho}_{k}}\mathrm{curl}\bv_k.
 \end{equation*}
We should note that the solutions of \eqref{wte} also satisfy the symmetric hyperbolic system \eqref{sh}. Set $\bU_k=(p_k, \bv_k, h_k)^{\mathrm{T}}, k \in \mathbb{Z}^+$, where $p_k=\bar{\rho}^\gamma \mathrm{e}^{h_k+\gamma \rho_k}$. For $j \in \mathbb{Z}^+$, we then have
\begin{equation*}
\begin{split}
  &A^0( \bU_k )\partial_t \bU_k+ \sum^3_{i=1}A^i( \bU_k )\partial_{i}\bU_k = 0,
  \\
  &A^0( \bU_l ) \partial_t \bU_l+ \sum^3_{i=1}A^i(\bU_l)\partial_{i}\bU_l=0.
  \end{split}
\end{equation*}
The standard energy estimates tell us that
\begin{equation*}
  \frac{d}{dt}\|\bU_k - \bU_l \|_{H_x^{s-1}} \leq C_{\bU_k, \bU_l} \left(\| d\bU_k, d\bU_l\|_{L^\infty_x}\|\bU_k-\bU_l\|_{H_x^{s-1}}+ \|\bU_k-\bU_l\|_{L^\infty_x}\| \partial \bU_l\|_{H_x^{s-1}} \right),
\end{equation*}
where $C_{\bU_k, \bU_l}$ depends on the $L^\infty_x$ norm of $\bU_k, \bU_l$. By Strichartz estimates of $d\bv_k, d\rho_k, dh_k, k \in \mathbb{Z}^+$ in Proposition \ref{p3}, we derive that
\begin{equation*}
\begin{split}
  \|(\bU_k-\bU_l)(t,\cdot)\|_{H_x^{s-1}} &\lesssim \|(\bU_k-\bU_l)(0,\cdot)\|_{H_x^{s-1}}
  \\
  &\lesssim\|(\bv_{0k}-\bv_{0l}, \rho_{0k}-\rho_{0l}, h_{0k}-h_{0l}) \|_{H^{s-1}}.
  \end{split}
\end{equation*}
As a result, $\{(p_k,\bv_k,h_k)\}_{k \in \mathbb{Z}^+}$ is a Cauchy sequence in $C([-T,T];H_x^{s-1})$. Denote $(p,\bv,h)$ being the limit. Then
\begin{equation*}
  \lim_{k\rightarrow \infty}(p_k,\bv_k,h_k )=(p,\bv,h) \ \mathrm{in} \ C([-T,T];H_x^{s-1}).
\end{equation*}
Since $\rho_k=\frac{1}{\gamma}\ln(p_k \mathrm{e}^{-h_k}\bar{\rho}^{-\gamma})$, then
\begin{equation*}
  \lim_{k\rightarrow \infty}\rho_k=\rho \ \mathrm{in} \ C([-T,T];H_x^{s-1}).
\end{equation*}
Consider
\begin{equation*}
 \bw_k=\bar{\rho}^{-1}\mathrm{e}^{-\rho_k}\mathrm{curl}\bv_k, \quad k \in \mathbb{Z}^+.
\end{equation*}
Due to product estimates and elliptic estimates, we have
\begin{equation}\label{cx}
\begin{split}
  \|\bw_k-\bw_l\|_{H_x^{s_0-2}} &\lesssim (\| \bv_k-\bv_l\|_{H_x^{s_0-1}}+\|\rho_k-\rho_l\|_{H_x^{s_0-1}})
  \\
  & \lesssim (\| \bv_{0k}-\bv_{0l}\|_{H^{s}}+\|{\rho}_{0k}-{\rho}_{0l} \|_{H^{s}}).
\end{split}
\end{equation}
Above, in the last line we use Lemma \ref{jh0}. The estimate \eqref{cx} implies that $\{\bw_k\}_{k \in \mathbb{Z}^+}$ is a Cauchy sequence in $C([-T,T];H_x^{s_0-2})$. We denote $\bw$ being the limit, that is
\begin{equation*}
  \lim_{k\rightarrow \infty} \bw_k = \bw \ \mathrm{in} \ C([-T,T];H_x^{s_0-2}).
\end{equation*}
Since $(\bv_k,\rho_k, h_k)$ is uniformly bounded in $C([-T,T];H_x^{s})$ and $\bw_k$ is bounded in $C([-T,T];H_x^{s_0})$ respectively. Thus, $(\bv, \rho, h) \in C([-T,T];H_x^{s}), \bw \in C([-T,T];H_x^{s_0})$. Using Proposition \ref{p3} again, the sequence $\{(d\bv_k,d\rho_k,dh_k)\}_{k \in \mathbb{Z}^+}$ is uniformly bounded in $L^2([-T,T];C_x^\delta)$. Consequently, the sequence $\{(d\bv_k, d\rho_k,dh_k)\}_{k \in \mathbb{Z}^+}$ converges to $(d\bv, d\rho, dh)$ in $L^2([-T,T];L_x^\infty)$. That is,
\begin{equation}\label{ccr}
  \lim_{k \rightarrow \infty} (d\bv_k, d\rho_k,dh_k)=(d\bv, d\rho, dh).
\end{equation}
On the other hand, we also deduce that the sequence $\{(d \bv_k,d\rho_k,dh_k)\}_{k \in \mathbb{Z}^+}$ is bounded in $L^2([-T,T];\dot{B}^{s_0-2}_{\infty,2})$.
Combined with \eqref{ccr}, this yields
 \begin{equation}\label{cc1}
   (d \bv, d\rho, dh) \in L^2([-T,T];\dot{B}^{s_0-2}_{\infty,2}).
 \end{equation}
By using \eqref{p32}, we have that
 \begin{equation}\label{cc2}
	\bv, \rho, h \in L^\infty([-T,T]\times \mathbb{R}^3).
\end{equation}
Set
\begin{equation*}
  \bv=\bv_{+}+\bv_{-}, \quad \bv_{-}=(-\Delta)^{-1} (\mathrm{e}^{\rho} \mathrm{curl} \bw).
\end{equation*}
We have
\begin{equation}\label{cc2g}
\begin{split}
  \| \partial \bv_{+}\|_{\dot{B}^{s_0-2}_{\infty,2}} \leq \| \partial \bv\|_{\dot{B}^{s_0-2}_{\infty,2}}+\| \partial \bv_{-}\|_{\dot{B}^{s_0-2}_{\infty,2}}.
  \end{split}
\end{equation}
Using elliptic estimates and Sobolev inequalities, we have
\begin{equation*}
\begin{split}
  \| \partial \bv_{-}\|_{\dot{B}^{s_0-2}_{\infty, 2}} & \lesssim \| \partial(-\Delta)^{-1}(\mathrm{e}^{{\rho}} \mathrm{curl} \bw) \|_{\dot{B}^{s_0-\frac{1}{2}}_{2,2}}
   \\
  & \lesssim \|\mathrm{e}^{{\rho}} \mathrm{curl} \bw\|_{\dot{H}^{s_0-\frac{3}{2}}}
  %\lesssim (1+ \|{\rho}\|_{H^{2}}) \|\bw \|_{H^{s_0-\frac{1}{2}}}
 \\
 & \lesssim (1+\|{\rho}_0\|_{H^{s}})\|\bw_0 \|_{H^{s_0}}.
 \end{split}
\end{equation*}
As a result, we have $\partial \bv_{-} \in L^2([-T,T];\dot{B}^{s_0-2}_{\infty,2})$. By \eqref{cc2g}, we can obtain $\partial \bv_{+} \in L^2([-T,T];\dot{B}^{s_0-2}_{\infty,2})$. This completes the proof of Theorem \ref{dingli}.

Now, let us reduce Proposition \ref{p3} to the case of smooth, small initial data by compactness arguments, i.e. Proposition \ref{p1} in the following contents.
\subsection{A reduction to the case of small initial data}
\begin{proposition}\label{p1}
	Let $2<s_0<s \leq \frac52$ and $\delta_* \in [0,s-2)$. Assume \eqref{a0} and \eqref{a1} hold. Consider the Cauchy problem \eqref{fc1}. Suppose the initial data $(\bv_0, {\rho}_0, \bw_0, h_0)$ be smooth, supported in $B(0,c+2)$ and satisfying
	\begin{equation}\label{300}
	\begin{split}
	&\|\bv_0\|_{H^s} + \|\rho_0 \|_{H^s} + \|\bw_0\|_{H^{s_0}}+ \|h_0\|_{H^{s_0+1}} \leq \epsilon_3.
	\end{split}
	\end{equation}
	Then the Cauchy problem \eqref{fc1} admits a smooth solution $(\bv,\rho,\bw,h)$ on $[-1,1] \times \mathbb{R}^3$, which have the following properties:
	
	$\mathrm{(1)}$ energy estimates
	\begin{equation}\label{s402}
	\begin{split}
	&\|\bv\|_{L^\infty_t H_x^{s}}+\| \rho \|_{L^\infty_t H_x^{s}} + \| \bw\|_{H_x^{s_0}} + \| h \|_{H_x^{s_0+1}} \leq \epsilon_2,
	\\
	& \|\bv, \rho, h\|_{L^\infty_t L^\infty_x} \leq 1+C_0.
	\end{split}
	\end{equation}

	$\mathrm{(2)}$ dispersive estimate for $\bv$ and $\rho$
	\begin{equation}\label{s403}
	\|d \bv, d \rho, dh\|_{L^2_t C^{\delta_*}_x}+\| d \rho, \partial \bv_{+}, d \bv, dh\|_{L^2_t \dot{B}^{s_0-2}_{\infty,2}} \leq \epsilon_2,
	\end{equation}

	$\mathrm{(3)}$ dispersive estimate for the linear equation

For any $t_0 \in [-1,1)$, let $f$ satisfy
 the linear wave equation
 \begin{equation}\label{Linear}
 	\begin{cases}
 		& \square_g f=0, \qquad (t,x) \in (t_0,1]\times \mathbb{R}^3,
 		\\
 		&f(t_0,\cdot)=f_0 \in H^r(\mathbb{R}^3), \quad \partial_t f(t_0,\cdot)=f_1 \in H^{r-1}(\mathbb{R}^3),
 	\end{cases}
 \end{equation}
where $g$ has the formulations defined in \eqref{Met}. For each $1 \leq r \leq s+1$, the Cauchy problem \eqref{linear} admits a solution $f$ in $C([-1,1];H_x^r) \times C^1([-1,1];H_x^{r-1})$ if the data $(f_0,f_1) \in H_x^r \times H_x^{r-1}$. Moreover, for $k<r-1$, the following estimate holds:
	\begin{equation}\label{304}
	\|\left< \partial \right>^k f\|_{L^2_{[-1,1]} L^\infty_x} \lesssim  \| f_0\|_{H_x^r}+ \| f_1\|_{H_x^{r-1}},
	\end{equation}
and the same estimates hold with $\left< \partial \right>^k$ replaced by $\left< \partial \right>^{k-1}d$.
\end{proposition}
\subsection{Proof of Proposition \ref{p3} by Proposition \ref{p1}}
%\medskip\begin{proof}[Proof of Proposition \ref{p3} by Proposition \ref{p1}]
In Proposition \ref{p3}, the corresponding statement is considering the small, supported data. While, the initial data is large in Proposition \ref{p1}. Firstly, by using a scaling method, we can reduce the initial data in Proposition \ref{p1} to be small.

$\textbf{Step 1: Scaling}$. Assume
\begin{equation}\label{a4}
\begin{split}
&\|\bv_0\|_{H^s}+ \| \rho_0\|_{H^s} + \|\bw_0\|_{H^{s_0}} + \|h_0\|_{H^{s_0+1}}  \leq M_0.
\end{split}
\end{equation}
If taking the scaling
\begin{equation}\label{s1}
\begin{split}
&\underline{\bv}(t,x)=\bv(Tt,Tx),\quad \underline{\rho}(t,x)={\rho}(Tt,Tx), %\quad \underline{\bw}(t,x)=\bw(Tt,Tx) ,
\quad \underline{h}(t,x)=h(Tt,Tx),
\end{split}
\end{equation}
we then obtain
\begin{equation}\label{s2}
\begin{split}
&\|\underline{\bv}(0)\|_{\dot{H}^s} \leq M_0 T^{s-\frac{3}{2}}, \qquad \qquad \|\underline{\rho}(0)\|_{\dot{H}^s} \leq M_0 T^{s-\frac{3}{2}},
\\
& \|\underline{\bw}(0)\|_{\dot{H}^{s_0}} \leq M_0 T^{s_0-\frac{1}{2}},  \qquad \|\underline{h}(0)\|_{\dot{H}^{s_0+1}} \leq M_0 T^{s_0-\frac{1}{2}}.
\end{split}
\end{equation}
Here we use the fact that $\underline{\bw}(0)=\mathrm{e}^{ -\underline{\rho}(0)}\mathrm{curl}\underline{\bv}(0)$. Choose sufficiently small $T$
such that
\begin{equation}\label{s3}
\max\{ M_0 T^{s_0-\frac{1}{2}}, M_0 T^{s-\frac{3}{2}}\} \ll \epsilon_3.
\end{equation}
Therefore, we get
\begin{equation*}
\begin{split}
  &\|\underline{\bv}(0)\|_{\dot{H}^s} \leq \epsilon_3, \quad \|\underline{{\rho}}(0)\|_{\dot{H}^s} \leq \epsilon_3, \quad \|\underline{\bw}(0)\|_{\dot{H}^{s_0}} \leq \epsilon_3, \quad \|\underline{h}(0)\|_{\dot{H}^{s_0+1}} \leq \epsilon_3.
\end{split}
\end{equation*}
Using \eqref{HEw}, we also have
\begin{equation}\label{vv0}
	\begin{split}
		\|\underline{\bv}(0), \underline{\rho}(0),\underline{h}(0)\|_{L^\infty} \leq C_0.
	\end{split}
\end{equation}
Secondly, to reduce the initial data with support set, we need a physical localization technique.

$\textbf{Step 2: Localization}$.

For the propagation speed of \eqref{fc1} is finite, we set $c$ be the largest speed of \eqref{fc1}. Set $\chi$ be a smooth function supported in $B(0,c+2)$, and which equals $1$ in $B(0,c+1)$. For given $\by \in \mathbb{R}^3$, we define the localized initial data for the velocity and density near $\by$:
\begin{equation*}
\begin{split}
\widetilde{\bv}_0(\bx)=&\chi(\bx-\by)\left( \underline{\bv}_0(\bx)- \underline{\bv}_0(\by)\right),
\\
\widetilde{\rho}_0(\bx)=&\chi(\bx-\by)\left( \underline{\rho}_0(\bx)-\underline{\rho}_0(\by)\right),
\\
\widetilde{h}_0(\bx)=&\chi(\bx-\by)\left( \underline{h}_0(\bx)-\underline{h}_0(\by)\right).
\end{split}
\end{equation*}
To let $\widetilde{\bw}_0$ be the specific vorticity, we should keep the relation in \eqref{pw1}. That is,
\begin{equation}\label{wde}
  \widetilde{\bw}_0=\mathrm{e}^{ -\widetilde{\rho}_0}\mathrm{curl}\widetilde{\bv}_0.
  %=\bar{\rho} \mathrm{e}^{ -\widetilde{\rho}_0}\widetilde{\bW}_0,
\end{equation}
Then $\widetilde{\bw}_0$ is the specific vorticity. Since $s,s_0\in (2, \frac{5}{2})$, we can verify
\begin{equation}\label{cc3}
\begin{split}
&\| \widetilde{\bv}_0\|_{H^s} \lesssim \|\underline{\bv}_0\|_{\dot{H}^s} \lesssim \epsilon_3
\\
&\|\widetilde{\rho}_0\|_{H^s} \lesssim \|\underline{\rho}_0 \|_{\dot{H}^s} \lesssim \epsilon_3
\\
&\|\widetilde{h}_0\|_{H^{s_0+1}} \lesssim \|\underline{h}_0\|_{\dot{H}^{s_0+1}} \lesssim \epsilon_3.
\end{split}
\end{equation}
For $s>2$, using \eqref{cc3} and Sobolev imbeddings, we get
\begin{equation}\label{vv1}
	\begin{split}
		\| \widetilde{\bv}_0, \widetilde{\rho}_0, \widetilde{\rho}_0 \|_{L^\infty} \leq  \| \widetilde{\bv}_0, \widetilde{\rho}_0, \widetilde{\rho}_0 \|_{{H}^s} \leq 1.
	\end{split}
\end{equation}
A calculating starting from \eqref{wde} yields
\begin{equation}\label{WY1}
   \widetilde{\bw}_0= \mathrm{e}^{ -\widetilde{\rho }_0} \nabla \chi(\bx-\by)\cdot (\underline{\bv}_0(\bx)-\underline{\bv}_0(\by))+\chi(\bx-\by)  \mathrm{e}^{ -\widetilde{\rho}_0}\mathrm{curl}{\underline{\bv}}_0.
\end{equation}
Hence, we have
\begin{equation}\label{W0e}
  \| \widetilde{\bw}_0 \|_{L^2} \lesssim \| \partial \underline{\bv}_0 \|_{L^2} \lesssim \epsilon_3.
\end{equation}
Substituting $\underline{\bw}_0= \mathrm{e}^{ -{\underline{\rho} }_0}\mathrm{curl}\underline{\bv}_0$ into \eqref{WY1}, we can rewrite \eqref{WY1} as
\begin{equation*}
   \widetilde{\bw}_0=  \mathrm{e}^{- \widetilde{ \rho_0}} \nabla \chi(\bx-\by)\cdot (\underline{\bv}_0(\bx)-\underline{\bv}_0(\by))+\chi(\bx-\by) \mathrm{e}^{ -\widetilde{ \rho_0}+{\underline{\rho}}_0} \underline{\bw}_0.
\end{equation*}
Therefore, we can obtain
\begin{equation}\label{Whe}
\begin{split}
\| \widetilde{\bw}_0 \|_{\dot{H}^{s_0}} &= \|  \mathrm{e}^{ - \widetilde{ \rho_0}} \nabla \chi(\bx-\by)\cdot (\underline{\bv}_0(\bx)-\underline{\bv}_0(\by))+\chi(\bx-\by) \mathrm{e}^{ -\widetilde{{\rho}_0}+{\underline{\rho}}_0}\underline{\bw}_0 \|_{\dot{H}^{s_0}}
\\
&\lesssim \|\underline{\bw}_0 \|_{\dot{H}^{s_0}}+ \|\underline{\bv}_0 \|_{\dot{H}^{s_0}}+\| \underline{\rho}_0\|_{\dot{H}^{s_0}} \lesssim \epsilon_3.
\end{split}
\end{equation}
Adding \eqref{W0e} and \eqref{Whe}, we can get
\begin{equation}\label{cc4}
\| \widetilde{\bw}_0 \|_{H^{s_0}} \lesssim \epsilon_3.
\end{equation}
By Proposition \ref{p1}, there is a smooth solution $(\widetilde{\bv}, \widetilde{\rho}, \widetilde{h}, \widetilde{\bw})$ on $[-1,1]\times \mathbb{R}^3$ satisfying the following equation(citing \eqref{DDi})
\begin{equation}\label{p}
\begin{cases}
&\square_{\widetilde{{g}}} \widetilde{\bv}=-\mathrm{e}^{\widetilde{\rho}+\underline{\rho}_0(y)}\tilde{c}^2_s\mathrm{curl} \widetilde{\bw}+\widetilde{\bQ},
\\
&\square_{\widetilde{{g}}} \widetilde{\rho}=-\frac{1}{\gamma}\widetilde{c}_s^2\Delta \widetilde{h}+ \widetilde{{D}},
\\
&\square_{\widetilde{{g}}} \widetilde{h}=\widetilde{c}_s^2\Delta \widetilde{h} + \widetilde{E},
\\
& \widetilde{\mathbf{T}}h=0,
\\
& \widetilde{\mathbf{T}} \widetilde{\bw}=(\widetilde{\bw} \cdot \nabla) \widetilde{\bv}+\bar{\rho} \mathrm{e}^{\widetilde{h}+\underline{h}_0(y)+(\gamma-2)(\widetilde{\rho}+\underline{\rho}_0(y))}\epsilon^{iab}\partial_a \widetilde{\rho} \partial_b \widetilde{h}.
\end{cases}
\end{equation}
Above, the quantities $\tilde{c}^2_s$ and $\tilde{{g}}$ are defined by
\begin{equation}\label{DDE}
\begin{split}
\tilde{c}^2_s:&= \frac{dp}{d{\rho}}(\widetilde{\rho}+\underline{\rho}_0(y), \widetilde{h}+\underline{h}_0(y)),
\\
\widetilde{g}:&=g(\widetilde{\bv}+\underline{\bv}_0(y), \widetilde{h}+\underline{h}_0(y), \widetilde{\rho}+\underline{\rho}_0(y) ),
\\
\widetilde{\mathbf{T}}:&=\partial_t+ (\widetilde{\bv}+\underline{\bv}_0(y))\cdot \nabla,
\end{split}
\end{equation}
and
\begin{equation*}
  \begin{split}
  \widetilde{Q}^i=& \widetilde{Q}^{i\alpha\beta} \partial_\alpha \widetilde{\rho} \partial_\beta \widetilde{h}+\widetilde{Q}^{i\alpha\beta }_{1j} \partial_\alpha \widetilde{\rho} \partial_\beta \widetilde{v}^j+\widetilde{Q}^{i\alpha\beta}_{2j} \partial_\alpha \widetilde{h} \partial_\beta \widetilde{v}^j+\widetilde{Q}^{\alpha\beta}_{3j} \partial_\alpha \widetilde{v}^i \partial_\beta \widetilde{v}^j,
  \\
  \widetilde{D}=& \widetilde{D}_1^{\alpha\beta} \partial_\alpha \widetilde{\rho} \partial_\beta \widetilde{h}+\widetilde{D}_2^{\alpha\beta} \partial_\alpha \widetilde{h} \partial_\beta \widetilde{h}+\widetilde{D}_3^{\alpha\beta} \partial_\alpha \widetilde{\rho} \partial_\beta \widetilde{\rho}
  \\
  & + \widetilde{D}^{\alpha\beta }_{1j} \partial_\alpha \widetilde{\rho} \partial_\beta \widetilde{v}^j+\widetilde{D}^{\alpha\beta}_{2j} \partial_\alpha \widetilde{h} \partial_\beta \widetilde{v}^j+\widetilde{D}^{i\alpha\beta}_{3j} \partial_\alpha \widetilde{v}_i \partial_\beta \widetilde{v}^j,
  \\
  \widetilde{E}=& \widetilde{E}_1^{\alpha\beta} \partial_\alpha \widetilde{h} \partial_\beta \widetilde{\rho}+\widetilde{E}^{\alpha\beta }_{2} \partial_\alpha \widetilde{h} \partial_\beta \widetilde{h}+\widetilde{E}^{\alpha\beta}_{3j} \partial_\alpha \widetilde{h} \partial_\beta \widetilde{v}^j,
\end{split}
\end{equation*}
and $\widetilde{Q}^{i\alpha\beta}$, $\widetilde{Q}^{i\alpha\beta }_{1j}$, $\widetilde{Q}^{i\alpha\beta}_{2j}$,  $\widetilde{Q}^{\alpha\beta}_{3j}$, $\widetilde{D}_1^{\alpha\beta}$, $\widetilde{D}_2^{\alpha\beta}$, $\widetilde{D}_3^{\alpha\beta}$, $\widetilde{D}^{\alpha\beta }_{1j}$, $\widetilde{D}^{\alpha\beta}_{2j}$, $\widetilde{D}^{i\alpha\beta}_{3j}$, $\widetilde{E}_1^{\alpha\beta}$, $\widetilde{E}^{\alpha\beta }_{2}$, $\widetilde{E}^{\alpha\beta}_{3j}$ have the same formulations with $Q^{i\alpha\beta}$, $Q^{i\alpha\beta }_{1j}$, $Q^{i\alpha\beta}_{2j}$,  $Q^{\alpha\beta}_{3j}$, $D_1^{\alpha\beta}$, $D_2^{\alpha\beta}$, $D_3^{\alpha\beta}$, $D^{\alpha\beta }_{1j}$, $D^{\alpha\beta}_{2j}$, $D^{i\alpha\beta}_{3j}$, $E_1^{\alpha\beta}$, $E^{\alpha\beta }_{2}$, $E^{\alpha\beta}_{3j}$ by replacing $({\bv}, {\rho}, {h})$ to $(\widetilde{\bv}+\underline{\bv}_0(y), \widetilde{h}+\underline{h}_0(y), \widetilde{\rho}+\underline{\rho}_0(y))$.  Furthermore, the specific vorticity $\widetilde{\bw}$ satisfies
\begin{equation}\label{Wy}
  \widetilde{\bw}= \mathrm{e}^{ -\widetilde{\rho}}\mathrm{curl}\widetilde{\bv}.
\end{equation}
Set
\begin{equation}\label{dvp}
  \widetilde{\bv}_{+}:=\widetilde{\bv}-\widetilde{\bv}_{-}, \quad \Delta \widetilde{\bv}_{-}:=-\mathrm{e}^{\widetilde{\rho}}\mathrm{curl}\widetilde{\bw}.
\end{equation}
By Proposition \ref{p1} again, we can find the solution of \eqref{DDE} satisfying
\begin{equation}\label{see00}
  \|\widetilde{\bv}\|_{H_x^s}+ \|\widetilde{\rho}\|_{H_x^s}+ \| \widetilde{\bw} \|_{H_x^{s_0}}+ \| \widetilde{h} \|_{H_x^{s_0+1}} \leq C(\|\widetilde{\bv}_0\|_{H_x^s}+ \|\widetilde{\rho}_0\|_{H_x^s}+ \| \widetilde{\bw}_0 \|_{H_x^{s_0}}+ \| \widetilde{h}_0 \|_{H_x^{s_0+1}}),
\end{equation}
and
\begin{equation}\label{see10}
  \|d\widetilde{\bv}, d\widetilde{\rho},d\widetilde{h}\|_{L^2_tC^{\delta_*}_x}+ \| d\widetilde{\bv}, \partial \widetilde{\bv}_{+}, d\widetilde{\rho},d\widetilde{h} \|_{L^2_t \dot{B}^{s_0-2}_{\infty,2}} \leq C(\|\widetilde{\bv}_0\|_{H_x^s}+ \|\widetilde{\rho}_0\|_{H_x^s}+ \| \widetilde{\bw}_0 \|_{H_x^{s_0}}+ \| \widetilde{h}_0 \|_{H_x^{s_0+1}}).
\end{equation}
Furthermore, the linear equation
\begin{equation}\label{cc5}
\begin{cases}
&\square_{ \widetilde{{g}}} \tilde{f}=0,
\\
&\tilde{f}(t_0,\cdot)=\widetilde{f}_0, \ \partial_t\tilde{f}(t_0,)=\widetilde{f}_1,
\end{cases}
\end{equation}
admits a solution $\widetilde{f} \in C([0,T],H_x^r)\times C^1([0,T],H_x^{r-1})$, and the following estimate holds:
	\begin{equation}\label{s31}
	\|\left< \partial \right>^k \widetilde{f}\|_{L^2_t L^\infty_x} \lesssim  \| \widetilde{f}_0\|_{H_x^r}+ \| \widetilde{f}_1\|_{H_x^{r-1}},\quad k<r-1.
	\end{equation}
Consequently, the function $(\widetilde{\bv}+\underline{\bv}_0, \widetilde{\rho}+\underline{\rho}_0, \widetilde{h}+\underline{h}_0(y), \widetilde{\bw})$ is also a solution of \eqref{p}, and its initial data coincides with $(\bv_0, {\rho}_0, h_0, \bw_0)$ in $B(\by,c+1)$. Giving the restrictions, for $\by\in \mathbb{R}^3$,
\begin{equation}\label{RS}
  \left( \widetilde{\bv}+\underline{\bv}_0 \right)|_{\mathrm{K}^y},
  \quad (\widetilde{\rho}+\underline{\rho}_0 )|_{\mathrm{K}^y},
  \\
  \quad (\widetilde{h}+\underline{h}_0 )|_{\mathrm{K}^y},
  \\
  \quad \widetilde{\bw}|_{\mathrm{K}^y},
\end{equation}
where $\mathrm{K}^y=\left\{ (t,\bx): ct+|\bx-\by| \leq c+1, |t| <1 \right\}$, then the restrictions \eqref{RS} solve \eqref{CEE}-\eqref{id} on $\mathrm{K}^y$. By finite speed of propagation, a smooth solution $(\bar{\bv}, \bar{\rho}, \bar{h},  \bar{\bw})$ solves \eqref{fc1} in $[-1,1] \times \mathbb{R}^3$, where $\bar{\bv}, \bar{\rho}, \bar{h}$ and $\bar{\bw}$ is denoted by
\begin{equation}\label{vw}
\begin{split}
  \bar{\bv}(t,\bx)  &= \widetilde{\bv}+\underline{\bv}_0(y), \ \ \ \ \ (t,\bx) \in \mathrm{K}^y,
  \\
    \bar{\rho}(t,\bx)  &=\widetilde{ \rho}+ \underline{\rho}_0(y), \ \ \ \ \ \  (t,\bx) \in \mathrm{K}^y,
    \\
    \bar{h}(t,\bx)  &=\widetilde{ h}+ \underline{h}_0(y), \ \ \ \ \ \ (t,\bx) \in \mathrm{K}^y,
  \\
   \bar{\bw}(t,\bx)  &=\mathrm{e}^{ -\underline{{\rho}}_0(y)}\widetilde{\bw}, \qquad (t,x) \in \mathrm{K}^y.
\end{split}
\end{equation}
From \eqref{vw}, by time-space scaling $(t,\bx)$ to $(T^{-1}t,T^{-1}\bx)$, we also obtain
\begin{equation}\label{po0}
	\begin{split}
	 &(\bv, \rho, h, \bw)=(\bar{\bv}, \bar{\rho}, \bar{h},  \bar{\bw})(T^{-1}t,T^{-1}\bx),
	 \\
	 &(\bv, \rho, h, \bw)|_{t=0}=(\bar{\bv}, \bar{\rho}, \bar{h},  \bar{\bw})(0,T^{-1}\bx)=(\bv_0, \rho_0, h_0, \bw_0).
	\end{split}
\end{equation}
Therefore, the function $(\bv, \rho, h, \bw)$ defined in \eqref{po0} is the solution of \eqref{fc1} according to the uniqueness of solutions, i.e. Corollary \ref{cor}. To obtain the energy estimates for $(\bv,\rho,h,\bw)$ on $[0,T]\times \mathbb{R}^3$, let us use the cartesian grid $3^{-\frac12} \mathbb{Z}^3$ in $\mathbb{R}^3$, and a corresponding smooth partition of unity
\begin{equation*}
	\textstyle{\sum}_{\by \in 3^{-\frac12} \mathbb{Z}^3} \psi(\bx-\by)=1,
\end{equation*}
such that the function $\psi$ is supported in the unit ball. Therefore, we have
\begin{equation*}
	\begin{split}
		& \bar{\bv}=\textstyle{\sum}_{\by \in 3^{-\frac12} \mathbb{Z}^3} \psi(\bx-\by)(\widetilde{\bv}+\underline{\bv}_0(y) ),
		\\
		& \bar{\rho}=\textstyle{\sum}_{\by \in 3^{-\frac12} \mathbb{Z}^3} \psi(\bx-\by)(\widetilde{\rho}+\underline{\rho}_0(y) ),
		\\
		& \bar{h}=\textstyle{\sum}_{\by \in 3^{-\frac12} \mathbb{Z}^3} \psi(\bx-\by)(\widetilde{h}+\underline{h}_0(y) ),
		\\
		& \bar{\bw}=e^{-\bar{\rho} } \textrm{curl}\bar{\bv}.
	\end{split}
\end{equation*}
By \eqref{see10} and \eqref{vw}, we can see
\begin{equation}\label{po1}
	\begin{split}
		 & \|d\bar{\bv}, d\bar{\rho},d\bar{h}\|_{L^2_{[0,1]}C^{\delta_*}_x}+ \| d\bar{\bv}, \partial \bar{\bv}_{+}, d\bar{\rho},d\bar{h} \|_{L^2_{[0,1]} \dot{B}^{s_0-2}_{\infty,2}}
		 \\
		  \leq & \sup_{\by \in 3^{-\frac12} \mathbb{Z}^3} ( \|d\widetilde{\bv}, d\widetilde{\rho},d\widetilde{h}\|_{L^2_{[0,1]}{C^{\delta_*}_x}}+ \| d\widetilde{\bv}, \partial \widetilde{\bv}_{+}, d\widetilde{\rho},d\widetilde{h} \|_{L^2_{[0,1]} \dot{B}^{s_0-2}_{\infty,2}} )
		\\
		 \leq & C(\|\widetilde{\bv}_0\|_{H_x^s}+ \|\widetilde{\rho}_0\|_{H_x^s}+ \| \widetilde{\bw}_0 \|_{H_x^{s_0}}+ \| \widetilde{h}_0 \|_{H_x^{s_0+1}}).
	\end{split}
\end{equation}
Using \eqref{vv0}, \eqref{cc3}, \eqref{cc4}, and \eqref{see00}, we shall obtain
\begin{equation}\label{po03}
	\begin{split}
		& \|\bar{\bv}, \bar{\rho},\bar{h}\|_{L^\infty_{[0,1]}L^\infty_x} \leq \|\widetilde{\bv},\widetilde{\rho},\widetilde{h} \|_{H_x^{s}} + \|\underline{\bv}_0, \underline{\rho}_0, \underline{h}_0\|_{L^\infty_x} \leq 1+C_0.
	\end{split}
\end{equation}
Taking changing of coordinates and using ${\delta_*} \in (0,s-2)$ and $2<s_0<s$, we get
\begin{equation}\label{po2}
	\begin{split}
		& \|d{\bv}, d{\rho},d{h}\|_{L^2_{[0,T]}C^{\delta_*}_x}+ \| d{\bv}, \partial {\bv}_{+}, d{\rho},d{h} \|_{L^2_{[0,T]} \dot{B}^{s_0-2}_{\infty,2}}
		\\
		\leq & T^{-(\frac12+{\delta_*})}\|d\widetilde{\bv}, d\widetilde{\rho},d\widetilde{h}\|_{L^2_{[0,1]}C^{\delta_*}_x}+ \| d\widetilde{\bv}, \partial \widetilde{\bv}_{+}, d\widetilde{\rho},d\widetilde{h} \|_{L^2_{[0,1]} \dot{B}^{s_0-2}_{\infty,2}}
		\\
		\leq
		&C (T^{-(\frac12+{\delta_*})}+T^{-(\frac12+s_0-2)}) (\|\widetilde{\bv}_0\|_{H_x^s}+ \|\widetilde{\rho}_0\|_{H_x^s}+ \| \widetilde{\bw}_0 \|_{H_x^{s_0}}+ \| \widetilde{h}_0 \|_{H_x^{s_0+1}})
		\\
		\leq & C (T^{-(\frac12+{\delta_*})}+T^{-(\frac12+s_0-2)})(T^{s-\frac32}\|{\bv}_0\|_{\dot{H}_x^s}+ T^{s-\frac32}\|{\rho}_0\|_{\dot{H}_x^s}
		\\
		&\qquad \qquad \qquad \qquad \qquad \qquad + T^{s_0+1-\frac32}\| {\bw}_0 \|_{\dot{H}_x^{s_0}}+ T^{s_0+1-\frac32}\| {h}_0 \|_{\dot{H}_x^{s_0+1}})
		\\
		\leq & C (\|{\bv}_0\|_{H_x^s}+ \|{\rho}_0\|_{H_x^s}+ \| {\bw}_0 \|_{H_x^{s_0}}+ \| {h}_0 \|_{H_x^{s_0+1}}).
	\end{split}
\end{equation}
By using \eqref{po03} and changing of coordinates from $(t,\bx) \rightarrow (T^{-1}t,T^{-1}\bx)$, we get
\begin{equation}\label{po05}
	\begin{split}
		& \|{\bv}, {\rho},{h}\|_{L^\infty_{[0,T]\times \mathbb{R}^3}} = \|\bar{\bv}, \bar{\rho},\bar{h}\|_{L^\infty_{[0,1]}L^{\infty}_x}  \leq 1+C_0.
	\end{split}
\end{equation}
Using Theorem \ref{ve} and \eqref{po2}, we can obtain that $(\bv, \rho, h, \bw)$ satisfies \eqref{p31} and \eqref{p32}, which is stated in Proposition \ref{p3}. It remains for us to prove \eqref{p33} in Proposition \ref{p3}.

We consider the following homogeneous linear wave equation
\begin{equation}\label{po3}
	\begin{cases}
		\square_{{g}} f=0, \quad [0,T]\times \mathbb{R}^3,
		\\
		(f,f_t)|_{t=0}=(f_0,f_1)\in H_x^r \times H^{r-1}_x, \quad 1\leq r \leq s+1.
	\end{cases}
\end{equation}
For \eqref{po3} is scaling invariant, so we transfer the system \eqref{po3} to the localized linear equation
\begin{equation}\label{po4}
	\begin{cases}
		\square_{\widetilde{g}} {f}^y=0, \quad [0,1]\times \mathbb{R}^3,
		\\
		({f}^y,{f}^y_t)|_{t=0}=(f_0^y,f^y_1),
	\end{cases}
\end{equation}
where
\begin{equation}\label{po5}
	\begin{split}
		& f_0^y=\chi(\bx-\by)(\widetilde{f}_0-\widetilde{f}_0(\by)), \quad f^y_1=\chi(\bx-\by)\widetilde{f}_1,
		\\ & \widetilde{f}_0=f_0(T\bx),\quad \widetilde{f}_1=f_1(T\bx).
	\end{split}
\end{equation}
Let
\begin{equation}\label{po6}
	\widetilde{f}=\textstyle{\sum}_{\by\in 3^{-\frac12}\mathbb{Z}^3}\psi(\bx-\by)({f}^y+\widetilde{f}_0(\by))
\end{equation}
Using Proposition \ref{p1} again, for $k<r-1$, we shall obtain that
\begin{equation}\label{po7}
	\begin{split}
		\|\left< \partial \right>^{k-1} d {f}^y\|_{L^2_{[0,1]} L^\infty_x} \leq  & C(\| \chi(\bx-\by)(\widetilde{f}_0-\widetilde{f}_0(\by)) \|_{H_x^r}+ \| \chi(\bx-\by)\widetilde{f}_1 \|_{H_x^{r-1}})
		\\
		\leq & C(\|\widetilde{f}_0\|_{\dot{H}_x^r}+ \| \widetilde{f}_1 \|_{\dot{H}_x^{r-1}}).
	\end{split}
\end{equation}
We again use the finite speed of propagation to conclude that $\widetilde{f}={f}^y+\widetilde{f}_0(\by)$ on the region $K^y$. Using \eqref{po6} and \eqref{po7}, so we get
\begin{equation}\label{po8}
	\begin{split}
		& \|\left< \partial \right>^{k-1} d \widetilde{f}\|_{L^2_{[0,1]} L^\infty_x}
		\\
		= & \sup_{\by\in 3^{-\frac12}\mathbb{Z}^3}\|\left< \partial \right>^{k-1} d ({f}^y+\widetilde{f}_0(\by))\|_{L^2_{[0,1]} L^\infty_x}
		\\
		= & \sup_{\by\in 3^{-\frac12}\mathbb{Z}^3}\|\left< \partial \right>^{k-1} d {f}^y\|_{L^2_{[0,1]} L^\infty_x}
		\\
		\leq & C(\|\widetilde{f}_0\|_{\dot{H}_x^r}+ \| \widetilde{f}_1 \|_{\dot{H}_x^{r-1}}).
	\end{split}
\end{equation}
By changing of coordinates from $(t,\bx)\rightarrow (T^{-1}t,T^{-1}\bx)$, we have
\begin{equation}\label{po9}
	\begin{split}
		\|\left< \partial \right>^{k-1} d{f}\|_{L^2_{[0,T]} L^\infty_x} = & T^{k-\frac12}\|\left< \partial \right>^{k-1} d\widetilde{f}\|_{L^2_{[0,1]} L^\infty_x}.
	\end{split}
\end{equation}
Combing \eqref{po8} with \eqref{po9}, it follows
\begin{equation}\label{po10}
	\begin{split}
		\|\left< \partial \right>^{k-1} d{f}\|_{L^2_{[0,T]} L^\infty_x}
		\leq & CT^{r-1-k}(\|{f}_0\|_{\dot{H}_x^r}+ \| {f}_1 \|_{\dot{H}_x^{r-1}})
		\\
		\leq & C(\|{f}_0\|_{{H}_x^r}+ \| {f}_1 \|_{{H}_x^{r-1}}).
	\end{split}
\end{equation}
Above, we also use $k<r-1$ and $T\leq 1$. By using energy estimates for \eqref{po3} and using \eqref{po2}, then we obtain
\begin{equation}\label{po11}
	\begin{split}
		\|d {f}\|_{L^\infty_{[0,T]} H^{r-1}_x} \leq  & C(\| {f}_0\|_{H_x^r}+ \| {f}_1\|_{H_x^{r-1}}) \exp\{  \|d g \|_{L^1_{[0,T]L^\infty_x}}\}
		\\
		\leq  & C_{M_0}(\| {f}_0\|_{H_x^r}+ \| {f}_1\|_{H_x^{r-1}}) \exp\{  \|d \bv,d\rho,dh \|_{L^1_{[0,T]L^\infty_x}}\}
		\\
		\leq & C_{M_0}(\| {f}_0\|_{H_x^r}+ \| {f}_1\|_{H_x^{r-1}}).
	\end{split}
\end{equation}
By using
\begin{equation}\label{po12}
	\begin{split}
		\|f\|_{L^\infty_{[0,T]} L^{2}_x}=
		& \|\int^{t}_0 \partial_ t fd\tau -f_0\|_{L^\infty_{[0,T]} L^{2}_x}
		\\
		\leq & T \|d {f}\|_{L^\infty_{[0,T]} L^{2}_x}+ \| {f}_0\|_{L_x^2}
		\\
		\leq & C_{M_0}(\| {f}_0\|_{H_x^r}+ \| {f}_1\|_{H_x^{r-1}}).
	\end{split}
\end{equation}
Adding \eqref{po11} and \eqref{po12}, so we have proved
\begin{equation}\label{po13}
	\begin{split}
		\|{f}\|_{L^\infty_{[0,T]} H^{r}_x}+ \|\partial_t {f}\|_{L^\infty_{[0,T]} H^{r-1}_x} \leq  C_{M_0}(\| {f}_0\|_{H_x^r}+ \| {f}_1\|_{H_x^{r-1}}).
	\end{split}
\end{equation}
Combining \eqref{po12}, \eqref{po13}, we have proved \eqref{p33} and \eqref{p333} in Proposition \ref{p3}. So we complete the proof of Proposition \ref{p3}.
%\end{proof}
\section{A bootstrap argument}\label{ABA}
For Proposition \ref{p1} is for the existence of small solutions, so we can consider it as a small perturbation of $(\mathbf{0},0,\mathbf{0},0)$. Therefore, the acoustic metric $g$ in \eqref{Met} is a perturbation near a flat Minkowski metric. So it's possible for us to give a proof of Proposition \ref{p1} by a bootstrap argument, i.e. using Proposition \ref{p4} to prove Proposition \ref{p1}. To reduce the proof of Proposition \ref{p1} to Proposition \ref{p4}. Let us make some preparations. We denote $\mathbf{m}$ being a standard Minkowski metric,
\begin{equation*}
  \mathbf{m}^{00}=-1, \quad \mathbf{m}^{ij}=\delta^{ij}, \quad i, j=1,2,3.
\end{equation*}
Taking $\bv=0, h=0, \rho=0$ in $g$ (please see the formulation of $g$ in \eqref{Met}), the inverse matrix of the metric $g$ is
\begin{equation*}
g^{-1}(0)=
\left(
\begin{array}{ccccc}
-1 & 0 & 0 & 0 & 0\\
0 & c^2_s(0) & 0& 0 & 0\\
0 & 0 & c^2_s(0) & 0 &0
\\
0 & 0 & 0 &   c^2_s(0)&0
\\
0 & 0 & 0 & 0 &  c^2_s(0)
\end{array}
\right ).
\end{equation*}
By a linear change of coordinates which preserves $dt$, we may assume that $g^{\alpha \beta}(0)=\mathbf{m}^{\alpha \beta}$.  Let $\chi$ be a smooth cut-off function in the region $B(0,3+2c) \times [-2, 2]$ which equals to $1$ in the region $B(0,2+2c) \times [-1, 1]$ and equals to $0$ in $B(0,3+2c) \times [-2, -\frac{3}{2}]$. Set
\begin{equation}\label{boldg}
  \mathbf{g}=\chi(t,x)(g-g(0))+g(0).
\end{equation}
Here
\begin{equation*}
  {g}=-dt\otimes dt+c_s^{-2}\sum_{a=1}^{3}\left( dx^a-v^adt\right)\otimes\left( dx^a-v^adt\right).
\end{equation*}
Now we are ready to consider the Cauchy problem of the following system
\begin{equation}\label{CS}
	\begin{cases}
\square_{\mathbf{g}} v^i=-\mathrm{e}^{\rho}c_s^2 \mathrm{curl} \bw^i+Q^i,
\\
\square_{\mathbf{g}} {\rho}=-\frac{1}{\gamma}c_s^2\Delta h+ D,
\\
\square_{\mathbf{g}} h=c_s^2\Delta h + E,
\\
 \mathbf{T}w^i=w^a \partial_a v^i+\bar{\rho} \mathrm{e}^{h+(\gamma-2)\rho}\epsilon^{iab}\partial_a \rho \partial_b h.
\end{cases}
	\end{equation}
where  $Q^i, E^i$ ($i=1,2,3$), and $\mathcal{D}$ has the same formulations in \eqref{DDi}, and $\mathbf{g}$ is defined in \eqref{boldg}. We denote by $\mathcal{H}$ the family of smooth solutions $(\bv, \rho, h, \bw)$ to equation \eqref{CS} for $t \in [-2,2]$, with the initial data $(\bv_0, \rho_0, h_0, \bw_0)$ supported in $B(0,2+c)$, and for which
\begin{equation}\label{401}
\|\bv_0\|_{H^s} + \|\rho_0\|_{H^s}+ \|h_0\|_{H^{s_0+1}}+\| \bw_0\|_{H^{s_0}}  \leq \epsilon_3, \ \ \ \ \ \ \ \ \ \ \ \ \ \
\end{equation}
\begin{equation}\label{402}
 \| \bv\|_{L^\infty_t H_x^{s}}+\|\rho\|_{L^\infty_t H_x^{s}}+\|h\|_{L^\infty_t H_x^{s_0+1}}+\| \bw\|_{L^\infty_t H_x^{s_0}} \leq 2 \epsilon_2,
\end{equation}
\begin{equation}\label{403}
  \| d \bv, d \rho, dh, \partial \bv_{+}\|_{L^2_t C_x^\delta}+ \|d \rho,dh, \partial \bv_{+}, d\bv\|_{L^2_t \dot{B}^{s_0-2}_{\infty,2}} \leq 2 \epsilon_2.
\end{equation}
Therefore, the bootstrap argument can be stated as follows:
\begin{proposition}\label{p4}
Assume that \eqref{a0} holds. Then there is a continuous functional $G: \mathcal{H} \rightarrow \mathbb{R}^{+}$, satisfying $G(0)=0$, so that for each $(\bv, \rho, h, \bw) \in \mathcal{H}$ satisfying $G(\bv, \rho, h) \leq 2 \epsilon_1$ the following hold:

$\mathrm{(1)}$ The function $\bv, \rho, h$, and $\bw$ satisfies $G(\bv, \rho, h) \leq \epsilon_1$.

$\mathrm{(2)}$ The following estimate holds,
\begin{equation}\label{404}
\|\bv\|_{L^\infty_t H_x^{s}}+ \|\rho\|_{L^\infty_t H_x^{s}}+ \|h\|_{L^\infty_t H_x^{s_0+1}} + \|\bw\|_{L^\infty_t H_x^{s_0}} \leq \epsilon_2,
\end{equation}
\begin{equation}\label{405}
\|d \bv, d \rho, dh\|_{L^2_t C^\delta_x}+\|\partial \bv_{+}, d \rho, dh,  d\bv\|_{L^2_t \dot{B}^{s_0-2}_{\infty,2}} \leq \epsilon_2. \ \ \ \ \ \ \ \ \
\end{equation}

$\mathrm{(3)}$ For $1 \leq k \leq s+1$, the linear wave equation \eqref{Linear} endowed with the metric $\mathbf{g}$ is well-posed when $(f_0,f_1) \in H_x^r \times H_x^{r-1}$, and the Strichartz estimates \eqref{304} hold.
\end{proposition}
\subsection{Proof of Proposition \ref{p1} by Proposition \ref{p4}}
%\medskip\begin{proof}[{Proof of Proposition \ref{p1} by Proposition \ref{p4}}]
Note the initial data in Proposition \ref{p1} satisfying
\begin{equation*}
  \|\bv_0\|_{H^s} + \|\rho_0\|_{H^s}+ \|h_0\|_{H^{s_0+1}}+\| \bw_0\|_{H^{s_0}}  \leq \epsilon_3.
\end{equation*}
We denote by $\text{A}$ the subset of those $\gamma \in [0,1]$ such that the equation \eqref{CS} admits a smooth solution $(\bv_\gamma,\rho_\gamma,h_\gamma,\bw_\gamma)$ having the initial data
\begin{equation*}
  \begin{split}
  \bv_\gamma(0)=&\gamma \bv_0,
  \\
  \rho_\gamma(0)=&\gamma \rho_0,
  \\
  h_\gamma(0)=&\gamma h_0,
  \\
  \bw_\gamma(0)=& \bar{\rho} \mathrm{e}^{-\rho_\gamma(0)}\mathrm{curl}\bv_\gamma(0),
  \end{split}
\end{equation*}
and such that $G(\bv_\gamma, \rho_\gamma, h_\gamma) \leq \epsilon_1$ and \eqref{404}, \eqref{405} hold.

If $\gamma=0$, then
\begin{equation*}
  (\bv_\gamma, \rho_\gamma, h_\gamma, \bw_\gamma)(t,x)=(\mathbf{0},0,0,\mathbf{0}).
\end{equation*}
is a smooth solution of \eqref{CS} with initial data
\begin{equation*}
  (\bv_\gamma, \rho_\gamma, h_\gamma, \bw_\gamma)(0,x)=(\mathbf{0},0,0,\mathbf{0}).
\end{equation*}
Thus, the set $\text{A}$ is not empty. If we can prove that $\text{A}=[0,1]$, then $1 \in \text{A}$. Also, \eqref{300}-\eqref{304} follow from Proposition \ref{p4}. As a result, Proposition \ref{p1} holds. It is sufficient to prove that $\text{A}$ is both open and closed in $[0,1]$.

(1) $\text{A}$ is open. Let $\gamma \in \text{A}$. Then $(\bv_\gamma, \rho_\gamma, h_\gamma, \bw_\gamma)$ is a smooth solution to \eqref{CS}, where
\begin{equation*}
  \bw_\gamma= \bar{\rho}^{-1} \mathrm{e}^{-\rho_\gamma}\mathrm{curl}\bv_\gamma.
\end{equation*}
Let $\beta$ be close to $\gamma$. By the continuity of $G$, it follows that
\begin{equation*}
  G(\bv_\beta, \rho_\beta, h_\beta) \leq 2\epsilon_1,
\end{equation*}
and also \eqref{401}-\eqref{403} holds. Using Proposition \ref{p4}, we have
\begin{equation*}
  G(\bv_\beta, \rho_\beta, h_\beta) \leq \epsilon_1,
\end{equation*}
and the estimates \eqref{404}, \eqref{405} hold. Thus, we have showed that $\beta \in \text{A}$.

(2) $\text{A}$ is closed. Let $\gamma_k \in \text{A}, k \in \mathbb{N}^+$. Let $\gamma$ be a limit satisfying $\lim_{k \rightarrow \infty} \gamma_k = \gamma$.
Then there exists a sequence $\{(\bv_{\gamma_k}, \rho_{\gamma_k}, h_{\gamma_k}, \bw_{\gamma_k}) \}_{k \in \mathbb{N}^+}$ being the smooth solutions to \eqref{CS} and
\begin{align*}
 & \|(\bv_{\gamma_k}, \rho_{\gamma_k})\|_{L^\infty_t H_x^{s}}+ \|(\partial_t \bv_{\gamma_k}, \partial_t \rho_{\gamma_k})\|_{L^\infty_t H_x^{s-1}}+ \|h_{\gamma_k}\|_{L^\infty_t H_x^{s_0+1}}+ \|\partial_t h_{\gamma_k}\|_{L^\infty_t H_x^{s_0}}
 \\
 &+ \| \bw_{\gamma_k} \|_{L^\infty_t H_x^{s_0}}+\| \partial_t \bw_{\gamma_k} \|_{L^\infty_t H_x^{s_0-1}}+  \| d \bv_{\gamma_k}, d \rho_{\gamma_k}, d h_{\gamma_k} \|_{L^2_t C_x^\delta}+ \| d \bv_{\gamma_k}, d \rho_{\gamma_k},dh_{\gamma_k} \|_{L^2_t \dot{B}_{\infty,2}^{s_0-2}} \leq \epsilon_2.
 \end{align*}
Then there exists a subsequence of $\{(\bv_{\gamma_k}, \rho_{\gamma_k}, h_{\gamma_k}, \bw_{\gamma_k}) \}_{k \in \mathbb{N}^+}$ such that $ (\bv_{\gamma}, \rho_{\gamma}, h_{\gamma}, \bw_{\gamma}) $ is the limit, and the limit satisfies
\begin{equation*}
\begin{split}
 &\|(\bv_{\gamma}, \rho_{\gamma})\|_{L^\infty_t H_x^{s}}+ \|(\partial_t \bv_{\gamma}, \partial_t \rho_{\gamma})\|_{L^\infty_t H_x^{s-1}}+\| h_{\gamma} \|_{L^\infty_t H_x^{s_0+1}}+\| \partial_t h_{\gamma} \|_{L^\infty_t H_x^{s_0}}
 \\
 &+ \| \bw_{\gamma} \|_{L^\infty_t H_x^{s_0}}+\| \partial_t \bw_{\gamma} \|_{L^\infty_t H_x^{s_0-1}}+  \| d \bv_{\gamma}, d \rho_{\gamma},dh \|_{L^2_t C_x^\delta}+  \| d \bv_{\gamma}, d \rho_{\gamma},dh_{\gamma} \|_{L^2_t \dot{B}_{\infty,2}^{s_0-2}} \leq \epsilon_2,
\end{split}
\end{equation*}
Thus, $G(\bv_\gamma, \rho_\gamma ,h_\gamma) \leq \epsilon_1$, which implies $\gamma \in \text{A}$. This completes the proof of Proposition \ref{p1}.
%\end{proof}
\section{Regularity of the characteristic hypersurface}\label{sec6}
To prove Proposition \ref{p4}, we need to give a definition of $G$ (please see \eqref{500}) and also prove the Strichartz estimates \eqref{404}-\eqref{405}. Inspired by Smith-Tataru's result \cite{ST} for quasi-linear wave equations, we try to find the character of the system \eqref{CS} and analysing the regularity of characteristic hypersurface. Next, we will prove Proposition \ref{p4} by using Proposition \ref{r2}, Proposition \ref{r3}, and Proposition \ref{r4}.

Following \cite{ST}, for $\theta \in \mathbb{S}^2$, $\bx \in \mathbb{R}^3$, and let $\Sigma_{\theta, r}$ be the flowout of the set $\theta \cdot \bx = r-2$ along the null geodesic flow with respect of $\mathbf{g}$ in the direction $\theta$ at $t=-2$.
Let $\bx'_{\theta}$ be given orthonormal coordinates on the hyperplane in $\mathbb{R}^3$ perpendicular to $\theta$. Then, $\Sigma_{\theta,r}$ is of the form
\begin{equation*}
	\Sigma_{\theta,r}=\left\{ (t,\bx): \theta\cdot \bx-\phi_{\theta, r}=0  \right\}
\end{equation*}
for a smooth function $\phi_{\theta, r}(t,\bx'_{\theta})$.
For a given $\theta$, the family $\{\Sigma_{\theta,r}, \ r \in \mathbb{R}\}$ defines a foliation of $[-2,2]\times \mathbb{R}^3$, by characteristic hypersurfaces with respect to $\mathbf{g}$.

We now introduce two norms for functions defined on $[-2,2] \times \mathbb{R}^3$,
\begin{equation}\label{d0}
  \begin{split}
  &\vert\kern-0.25ex\vert\kern-0.25ex\vert f\vert\kern-0.25ex\vert\kern-0.25ex\vert_{s_0, \infty} = \sup_{-2 \leq t \leq 2} \sup_{0 \leq j \leq 1} \| \partial_t^j f(t,\cdot)\|_{H^{s_0-j}(\mathbb{R}^3)},
  \\
  & \vert\kern-0.25ex\vert\kern-0.25ex\vert  f\vert\kern-0.25ex\vert\kern-0.25ex\vert_{s_0,2} = \big( \sup_{0 \leq j \leq 1} \int^{2}_{-2} \| \partial_t^j f(t,\cdot)\|^2_{H^{s_0-j}(\mathbb{R}^3)} dt \big)^{\frac{1}{2}}.
  \end{split}
\end{equation}
The same notation applies for functions in $[-2,2] \times \mathbb{R}^3$. We denote
\begin{equation*}
  \vert\kern-0.25ex\vert\kern-0.25ex\vert f\vert\kern-0.25ex\vert\kern-0.25ex\vert_{s_0,2,\Sigma_{\theta,r}}=\vert\kern-0.25ex\vert\kern-0.25ex\vert f|_{\Sigma_{\theta,r}} \vert\kern-0.25ex\vert\kern-0.25ex\vert_{s_0,2},
\end{equation*}
where the right hand side is the norm of the restriction of $f$ to ${\Sigma_{\theta,r}}$, taken over the $(t,\bx'_{\theta})$ variables used to parametrise ${\Sigma_{\theta,r}}$. Similarly, the notation $\|f\|_{H^{s_0}(\Sigma_{\theta,r})}$ denotes the $H^{s_0}(\mathbb{R}^2)$ norm of $f$ restricted to the time $t$ slice of ${\Sigma_{\theta,r}}$ using the $\bx'_{\theta}$ coordinates on ${\Sigma^t_{\theta,r}}$.

We now set
\begin{equation}\label{500}
  G(\bv, \rho,h)= \sup_{\theta, r} \vert\kern-0.25ex\vert\kern-0.25ex\vert d \phi_{\theta,r}-dt\vert\kern-0.25ex\vert\kern-0.25ex\vert_{s_0,2,{\Sigma_{\theta,r}}}.
\end{equation}

\begin{proposition}\label{r1}
Let $(\bv, \rho, h, \bw) \in \mathcal{H}$ so that $G(\bv, \rho, h) \leq 2 \epsilon_1$. Then
\begin{equation}\label{501}
  \vert\kern-0.25ex\vert\kern-0.25ex\vert  {\mathbf{g}}^{\alpha \beta}-\mathbf{m}^{\alpha \beta} \vert\kern-0.25ex\vert\kern-0.25ex\vert_{s_0,2,{\Sigma_{\theta,r}}} + \vert\kern-0.25ex\vert\kern-0.25ex\vert 2^{j}({\mathbf{g}}^{\alpha \beta}-S_j {\mathbf{g}}^{\alpha \beta}), d \Delta_j {\mathbf{g}}^{\alpha \beta}, 2^{-j} \partial_x \Delta_j d {\mathbf{g}}^{\alpha \beta}  \vert\kern-0.25ex\vert\kern-0.25ex\vert_{s_0-1,2,{\Sigma_{\theta,r}}} \lesssim \epsilon_2.
\end{equation}
\end{proposition}
\begin{proposition}\label{r2}
Let $(\bv, \rho, h, \bw) \in \mathcal{H}$ so that $G(\bv, \rho, h) \leq 2 \epsilon_1$. Then

\begin{equation}\label{G}
G(\bv, \rho, h) \lesssim \epsilon_2.
\end{equation}
Furthermore, for each $t$ it holds that
\begin{equation}\label{502}
  \|d \phi_{\theta,r}(t,\cdot)-dt \|_{C^{1,\delta}_{x'}} \lesssim \epsilon_2+  \| d {\mathbf{g}}(t,\cdot) \|_{C^\delta_x(\mathbb{R}^3)}.
\end{equation}
\end{proposition}
\subsection{Energy estimates on the characteristic hypersurface}
Let $(\bv, \rho, h, \bw) \in \mathcal{H}$. Then the following estimates hold:
\begin{equation}\label{5021}
   \vert\kern-0.25ex\vert\kern-0.25ex\vert \bv,\rho,h\vert\kern-0.25ex\vert\kern-0.25ex\vert_{s,\infty}+\vert\kern-0.25ex\vert\kern-0.25ex\vert \bw \vert\kern-0.25ex\vert\kern-0.25ex\vert_{s_0,\infty} +\|d\bv,d\rho,dh, \partial \bv_{+}\|_{L^2_tC^\delta_x}+ \|d\bv, d\rho,dh, \partial \bv_{+}\|_{L^2_t \dot{B}^{s_0-2}_{\infty,2}} \lesssim \epsilon_2.
\end{equation}
It is sufficient to prove Proposition \ref{r1} and Proposition \ref{r2} for $\theta=(0,0,1)$ and $r=0$. We fix this choice, and suppress $\theta$ and $r$ in our notation. We use $(x_3, \bx')$ instead of $(x_{\theta}, \bx'_{\theta})$, where $x_{\theta}=\theta \cdot \bx$. Then $\Sigma$ is defined by
\begin{equation*}
  \Sigma=\left\{ x_3- \phi(t,\bx')=0 \right\}.
\end{equation*}
The hypothesis $G \leq 2 \epsilon_1$ implies that
\begin{equation}\label{503}
 \vert\kern-0.25ex\vert\kern-0.25ex\vert d \phi_{\theta,r}(t,\cdot)-dt \vert\kern-0.25ex\vert\kern-0.25ex\vert_{s_0,2, \Sigma} \leq 2 \epsilon_1.
\end{equation}
By Sobolev's imbedding, we have
\begin{equation}\label{504}
  \|d \phi(t,\bx')-dt \|_{L^2_t C^{1,\delta}_{x'}} + \| \partial_t d \phi(t,\bx')\|_{L^2_t C^{\delta}_{x'}} \lesssim \epsilon_1.
\end{equation}
Let us now introduce two lemmas, which is from \cite{ST}.
\begin{Lemma}\label{te0}\cite{ST}
Let $\tilde{f}(t,\bx)=f(t,\bx',x_3+\phi(t,\bx'))$. Then we have
\begin{equation*}
  \vert\kern-0.25ex\vert\kern-0.25ex\vert \tilde{f}\vert\kern-0.25ex\vert\kern-0.25ex\vert_{s_0,\infty}\lesssim \vert\kern-0.25ex\vert\kern-0.25ex\vert f\vert\kern-0.25ex\vert\kern-0.25ex\vert_{s_0,\infty}, \quad \|d\tilde{f}\|_{L^2_tL_x^\infty}\lesssim \|d f\|_{{L^2_tL_x^\infty}}, \quad  \|\tilde{f}\|_{H^{s_0}_{x}}\lesssim \|f\|_{H^{s_0}_{x}}.
\end{equation*}
\end{Lemma}
\begin{Lemma}\label{te2}\cite{ST}
For $r\geq 1$, we have
\begin{equation*}
\begin{split}
  \sup_{t\in[-2,2]} \| f\|_{H^{r-\frac{1}{2}}(\mathbb{R}^n)} & \lesssim \vert\kern-0.25ex\vert\kern-0.25ex\vert f \vert\kern-0.25ex\vert\kern-0.25ex\vert_{r,2},
  \\
  \sup_{t\in[-2,2]} \| f\|_{H^{r-\frac{1}{2}}(\Sigma^t)} & \lesssim \vert\kern-0.25ex\vert\kern-0.25ex\vert f \vert\kern-0.25ex\vert\kern-0.25ex\vert_{r,2,\Sigma}.
\end{split}
\end{equation*}
If $r> \frac{n+1}{2}$, then
\begin{equation*}
  \vert\kern-0.25ex\vert\kern-0.25ex\vert f_1f_2\vert\kern-0.25ex\vert\kern-0.25ex\vert_{r,2}\lesssim  \vert\kern-0.25ex\vert\kern-0.25ex\vert f_2 \vert\kern-0.25ex\vert\kern-0.25ex\vert_{r,2} \vert\kern-0.25ex\vert\kern-0.25ex\vert f_1\vert\kern-0.25ex\vert\kern-0.25ex\vert_{r,2}.
\end{equation*}
Similarly, if $r>\frac{n}{2}$, then
\begin{equation*}
  \vert\kern-0.25ex\vert\kern-0.25ex\vert f_1f_2\vert\kern-0.25ex\vert\kern-0.25ex\vert_{r,2,\Sigma}\lesssim  \vert\kern-0.25ex\vert\kern-0.25ex\vert f_2 \vert\kern-0.25ex\vert\kern-0.25ex\vert_{r,2,\Sigma} \vert\kern-0.25ex\vert\kern-0.25ex\vert f_1 \vert\kern-0.25ex\vert\kern-0.25ex\vert_{r,2,\Sigma}.
\end{equation*}
\end{Lemma}
\begin{Lemma}\label{te1}
{Suppose $\bU$ satisfy the hyperbolic symmetric system}
\begin{equation}\label{505}
  A^0(\bU)\bU_t+ \sum^{3}_{i=1}A^i(\bU)\bU_{x_i}= 0.
\end{equation}
Then
\begin{equation}\label{te10}
\begin{split}
 \vert\kern-0.25ex\vert\kern-0.25ex\vert  \bU\vert\kern-0.25ex\vert\kern-0.25ex\vert_{s_0,2,\Sigma} & \lesssim \|d \bU \|_{L^2_t L^{\infty}_x}+ \| \bU\|_{L^{\infty}_tH_x^{s_0}}.
   \end{split}
\end{equation}
\end{Lemma}

\begin{proof}
Choosing the change of coordinates $x_3 \rightarrow x_3-\phi(t,x')$ and setting $\tilde{\bU}(t,x)=U(t,\bx',x_3+\phi(t,\bx'))$, the system \eqref{te10} is transformed to
\begin{equation*}
  A^0(\tilde{\bU}) \partial_t \tilde{\bU}+ \sum_{i=1}^3 A^i(\tilde{\bU}) \partial_{x_i} \tilde{\bU}= - \partial_t \phi  \partial_3 \tilde{\bU} - \sum_{\alpha=0}^3 A^\alpha(\tilde{\bU}) \partial_{x_\alpha}\phi \partial_\alpha \tilde{\bU}.
\end{equation*}
Since $\phi$ is independent of $x_3$, this yields
\begin{equation}\label{U}
  A^0(\tilde{\bU}) \partial_t \tilde{\bU}+ \sum_{i=1}^3 A^i(\tilde{\bU}) \partial_{x_i} \tilde{\bU}= - \partial_t \phi  \partial_3 \tilde{\bU} - \sum_{\alpha=0}^2 A^\alpha(\tilde{\bU}) \partial_{x_\alpha}\phi \partial_\alpha \tilde{\bU}.
\end{equation}
To prove \eqref{te10}, we first establish the $0$-order estimate. A direct calculation on$[-2,2]\times \mathbb{R}^3$ shows that
\begin{equation*}
\begin{split}
   \vert\kern-0.25ex\vert\kern-0.25ex\vert \tilde{\bU}\vert\kern-0.25ex\vert\kern-0.25ex\vert^2_{0,2,\Sigma} & \lesssim \| d\tilde{\bU} \|_{L^1_t L_x^\infty}\|\tilde{\bU}\|_{L_x^2} + \|\partial d\phi\|_{L^1_t L_{x'}^\infty}\|\tilde{\bU}\|_{L_x^2}
   \\
   & \lesssim \| d\tilde{\bU} \|_{L^1_t L_x^\infty}\|\tilde{\bU}\|_{L_x^2} + \|\partial d\phi\|_{L^1_t L_x^\infty}\|\tilde{\bU}\|_{L_x^2}.
\end{split}
\end{equation*}
By using Lemma \ref{te0}, \eqref{5021} and \eqref{504}, we can prove that
\begin{equation}\label{U0}
 \vert\kern-0.25ex\vert\kern-0.25ex\vert \bU\vert\kern-0.25ex\vert\kern-0.25ex\vert_{0,2,\Sigma} \lesssim \|d \bU \|_{L^2_t L^{\infty}_x}+ \| \bU\|_{L^{\infty}_tL_x^2}.
\end{equation}
We now establish the $s_0$-order estimate. Taking the derivative of $\partial^{\beta}_{x'}$($1 \leq |\beta| \leq s_0$) on \eqref{U} and integrating it on $[-2,2]\times \mathbb{R}^3$, we get
\begin{equation}\label{U1}
  \begin{split}
 \| \partial^{\beta}_{x'} \tilde{\bU}\|^2_{L^2_{\Sigma}} & \lesssim \| d \tilde{\bU} \|_{L^1_t L^{\infty}_x} \| \partial^{\beta}_{x} \tilde{\bU}\|_{L^{\infty}_tL_x^2} +|I_1|+|I_2|,
  \end{split}
\end{equation}
where
\begin{equation*}
\begin{split}
  &I_1= -\int_{-2}^2\int_{\mathbb{R}^3} \partial^{\beta}_{x'}  \big( \partial_t \phi  \partial_3 \tilde{\bU} \big) \cdot \Lambda^{\beta}_{x'} \tilde{\bU}  dxd\tau,
\\
& I_2= -\sum^2_{\alpha=0}\int_{-2}^2\int_{\mathbb{R}^3} \partial^{\beta}_{x'} \big( A^\alpha(\tilde{\bU}) \partial_{x_\alpha}\phi \partial_\alpha \tilde{\bU} \big) \cdot \partial^{\beta}_{x'} \tilde{\bU}   dxd\tau.
\end{split}
\end{equation*}
We can write $I_1$ as
\begin{equation*}
\begin{split}
  I_1 =& -\int_{-2}^2\int_{\mathbb{R}^3} \big( \partial^{\beta}_{x'}(\partial_t \phi  \partial_3 \tilde{\bU})-\partial_t \phi \partial_3 \partial^{\beta}_{x'} \tilde{\bU} \big)   \partial^{\beta}_{x'} \tilde{\bU}dxd\tau
  \\
  & + \int_{-2}^2\int_{\mathbb{R}^3} \partial_t \phi \cdot \partial_3 \partial^{\beta}_{x'} \tilde{\bU} \cdot \partial^{\beta}_{x'} \tilde{\bU}  dx d\tau,
 \\
 =& -\int_{-2}^2\int_{\mathbb{R}^3} [ \partial^{\beta}_{x'}, \partial_t \phi  \partial_3] \tilde{\bU} \cdot \partial^{\beta}_{x'} \tilde{\bU}dxd\tau
\end{split}
\end{equation*}
We also write
\begin{equation*}
\begin{split}
  I_2 = & -\sum^2_{\alpha=0}\int_{-2}^2\int_{\mathbb{R}^3} \big( \partial^{\beta}_{x'} \big( A^\alpha(\tilde{\bU}) \partial_{x_\alpha}\phi \partial_\alpha \tilde{\bU}) -  A^\alpha(\tilde{\bU}) \partial_{x_\alpha}\phi \partial_\alpha \partial^{\beta}_{x'} \tilde{\bU} \big) \cdot \partial^{\beta}_{x'} \tilde{\bU}  dxd\tau
  \\
  & +\sum^2_{\alpha=0} \int_{-2}^2\int_{\mathbb{R}^3} \big( A^\alpha(\tilde{\bU}) \partial_{x_\alpha}\phi \big) \cdot \partial_\alpha(\partial^{\beta}_{x'} \tilde{\bU}) \cdot \partial^{\beta}_{x'} \tilde{\bU}dxd\tau.
\end{split}
\end{equation*}
By commutator estimates in Lemma \ref{jh}, we can get
\begin{equation}\label{U2}
\begin{split}
  |I_1|
  & \lesssim  \big( \|\partial^\beta \tilde \bU\|_{L^{\infty}_tL_x^2} \| d \partial_t \phi \|_{L^1_tL_{x}^\infty}+ \sup_{\theta, r} \|\partial^{\beta}_{x'} \partial_t \phi\|_{L^2(\Sigma_{\theta,r})} \| d \tilde{\bU}\|_{L^1_tL_x^\infty} \big)\cdot\|\partial^\beta \tilde \bU\|_{L^\infty_tL_x^2}
\end{split}
\end{equation}
and
\begin{equation}\label{U3}
\begin{split}
  |I_2| \lesssim  & \big(\| \partial^{\beta}_{x'} \tilde{\bU} \|_{L^2_tL_x^2} \|d \partial \phi\|_{L^1_t L^\infty_x} + \|d\tilde{\bU}\|_{L^1_tL_x^\infty} \sup_{\theta,r}\|\partial^{\beta}_{x'}d \phi\|_{L^2(\Sigma_{\theta,r})}  \big) \cdot \|\partial^{\beta}_{x'} \tilde{\bU} \|_{L^\infty_tL_x^2}
  \\
  & + \big( \| d \tilde{\bU} \|_{L^2_t L_x^\infty} \| \partial \phi\| _{L^2_tL_x^\infty}+ \|\tilde{\bU}\|_{L^2_t L_x^\infty} \|\partial^2\phi\|_{L^2_tL_{x}^\infty} \big)  \cdot \|\partial^\beta \tilde{\bU} \|^2_{L^\infty_t L_x^2}.
\end{split}
\end{equation}
Taking sum of $1\leq \beta \leq s_0$ on \eqref{U1}, due to Lemma \ref{te0}, \eqref{U2}, \eqref{U3}, \eqref{5021}, and \eqref{504}, we obtain
\begin{equation}\label{U4}
\begin{split}
 \vert\kern-0.25ex\vert\kern-0.25ex\vert \partial_{x'} \bU\vert\kern-0.25ex\vert\kern-0.25ex\vert_{s_0-1,2,\Sigma} & \lesssim \|d \bU \|_{L^2_t L^{\infty}_x}+\|d \bU\|_{L^{\infty}_tH_x^{s_0-1}}.
   \end{split}
\end{equation}
Taking derivatives on \eqref{505}, we have
\begin{equation*}
  \begin{split}
  A^0(\bU)(\partial \bU)_t + \sum^3_{i=1} A^i(\bU)(\partial \bU)_{x_i}=-\sum^3_{\alpha=0}\partial(A^\alpha(\bU))\bU_{x_\alpha},
  \end{split}
\end{equation*}
In a similar process, we can obtain
\begin{equation}\label{U40}
\begin{split}
 \vert\kern-0.25ex\vert\kern-0.25ex\vert \partial \bU \vert\kern-0.25ex\vert\kern-0.25ex\vert_{s_0-1,2,\Sigma} & \lesssim \|d \bU \|_{L^2_t L^{\infty}_x}+\|d \bU\|_{L^{\infty}_tH_x^{s_0-1}}.
   \end{split}
\end{equation}
Using $\partial_t \bU=- \sum^3_{i=1}(A^0)^{-1}A^i(\bU)\partial_i \bU$ and Lemma \ref{te2}, we have
\begin{equation}\label{U5}
\begin{split}
\vert\kern-0.25ex\vert\kern-0.25ex\vert \partial_t \bU \vert\kern-0.25ex\vert\kern-0.25ex\vert_{s_0-1,2,\Sigma}
& \lesssim \vert\kern-0.25ex\vert\kern-0.25ex\vert \bU \vert\kern-0.25ex\vert\kern-0.25ex\vert_{s_0-1,2,\Sigma} \vert\kern-0.25ex\vert\kern-0.25ex\vert\partial_t \bU\vert\kern-0.25ex\vert\kern-0.25ex\vert_{s_0-1,2,\Sigma}
 \lesssim \|d \bU \|_{L^2_t L^{\infty}_x}.
\end{split}
\end{equation}
Combining \eqref{U0}, \eqref{U4}, \eqref{U5}, and \eqref{U5}, we obtain \eqref{te10}. Thus, the proof is finished.
\end{proof}
Using Lemma \ref{te1}, and combining with \eqref{402} and \eqref{403}, we can obtain the following corollary.
\begin{corollary}\label{vte}
Suppose that $(\bv, \rho, h, \bw) \in \mathcal{H}$. Then the following estimate
\begin{equation}\label{e017}
\vert\kern-0.25ex\vert\kern-0.25ex\vert \bv \vert\kern-0.25ex\vert\kern-0.25ex\vert_{s_0,2,\Sigma}+ \vert\kern-0.25ex\vert\kern-0.25ex\vert \rho \vert\kern-0.25ex\vert\kern-0.25ex\vert_{s_0,2,\Sigma}+ \vert\kern-0.25ex\vert\kern-0.25ex\vert h \vert\kern-0.25ex\vert\kern-0.25ex\vert_{s_0,2,\Sigma} \lesssim \epsilon_2
\end{equation}
holds.
\end{corollary}
The estimate \eqref{e017} is not sufficient to prove Proposition \ref{r2}, for we also need the higher-order estimate of $h$ and $\bw$ along $\Sigma$. The next goal is to establish the characteristic energy estimates for $\bw$. Compared with the energy estimates for $\bw$, the characteristic energy is along the hypersurface, not the Cauchy slices. Therefore, it's not trivial. To prove the energy estimates of $\bw$ along the null hypersurface, let us give a lemma firstly.
\begin{Lemma}\label{te3}
Let $f$ satisfy the following transport equation
\begin{equation}\label{333}
  \mathbf{T} f=F.
\end{equation}
Set $L= \partial (-\Delta)^{-1}\mathrm{curl}$. Then
\begin{equation}\label{teE}
\begin{split}
  \vert\kern-0.25ex\vert\kern-0.25ex\vert  Lf\vert\kern-0.25ex\vert\kern-0.25ex\vert^2_{s_0-2,2,\Sigma} \lesssim & \ \big( \| \partial \bv\|_{L^2_t\dot{B}^{s_0-2}_{\infty,2}}+\vert\kern-0.25ex\vert\kern-0.25ex\vert d\phi-dt \vert\kern-0.25ex\vert\kern-0.25ex\vert_{s_0,2,\Sigma} \big) \|f\|^2_{{H}^{s_0-2}_x}(1+\vert\kern-0.25ex\vert\kern-0.25ex\vert d\phi-dt \vert\kern-0.25ex\vert\kern-0.25ex\vert_{s_0,2,\Sigma})
  \\
  & + \sum_{\sigma\in\{0,s_0-2 \}}\big|\int^t_0 \int_{\mathbb{R}^3} \Lambda^{\sigma}_{x'}{LF} \cdot \Lambda^{\sigma}_{x'}Lfdxd\tau \big|.
\end{split}
\end{equation}
\end{Lemma}
\begin{proof}
Applying the operator $L$ to \eqref{333}, we have
\begin{equation*}
  \mathbf{T} Lf=LF+[L,\mathbf{T}]f.
\end{equation*}
Choosing the change of coordinates $x_3 \rightarrow x_3-\phi(t,x')$ and setting $\widetilde{f}=f(x_1,x_2,x_3-\phi(t,x'))$, then the above equation transforms to
\begin{equation*}
\begin{split}
  (\partial_t+ \partial_t \phi \partial_{x_3}) \widetilde{Lf}+ \tilde{v}^i \cdot (\partial_{x_i}+\partial_{x_i} \phi \partial_{x_3} ) \widetilde{Lf}= & \widetilde{LF}+\widetilde{[L,\mathbf{T}]f}.
  \end{split}
\end{equation*}
Rewrite it as
\begin{equation}\label{Q}
\begin{split}
  \partial_t\widetilde{Lf}+ (\tilde{\bv} \cdot \nabla)\widetilde{Lf}= & \widetilde{LF}+\widetilde{[L,\mathbf{T}]f}
  %\\
%  &
- \partial_t \phi \partial_{x_3}\widetilde{Lf}-(\tilde{\bv} \cdot \nabla) \phi \partial_{x_3} \widetilde{Lf} .
  \end{split}
\end{equation}
Multiplying $\widetilde{Lf}$ and integrating it on $\mathbb{R}^{+} \times \mathbb{R}^3$, we can show that
\begin{equation}\label{LL}
  \begin{split}
  \vert\kern-0.25ex\vert\kern-0.25ex\vert Lf\vert\kern-0.25ex\vert\kern-0.25ex\vert^2_{0,2,\Sigma} \lesssim \ &\big| \int^2_{-2} \int_{\mathbb{R}^3} LF \cdot Lfdx d\tau \big|+ \| \partial \bv\|_{L^2_t L^\infty_x}(1+ \|\partial \phi\|_{L^\infty_x})\|Lf\|^2_{L^2_x}
  \\
  & \ +\|[L,\mathbf{T}]f\|_{L^2_t L^2_x}\|Lf\|^2_{L^2_x}
  \end{split}
\end{equation}
Considering $L= \partial (-\Delta)^{-1}\mathrm{curl}$, by commutator estimates in Lemma \ref{ceR}, we get
\begin{equation}\label{LL0}
\begin{split}
  \|[L,\mathbf{T}]f\|_{L^2_x} \lesssim & \ \|\partial \bv\|_{L^\infty}\|f \|_{L^2_x}.
\end{split}
\end{equation}
By elliptic estimates, we also have
\begin{equation}\label{LL1}
  \|Lf\|^2_{L^2_x} \lesssim \|f\|^2_{L^2_x}.
\end{equation}
By using \eqref{LL}, \eqref{LL0}, and \eqref{LL1}, it can give us
\begin{equation}\label{te30}
  \vert\kern-0.25ex\vert\kern-0.25ex\vert Lf\vert\kern-0.25ex\vert\kern-0.25ex\vert^2_{0,2,\Sigma} \lesssim  \ \| \partial v\|_{L^2_t L^\infty_x}(1+ \|\partial \phi\|_{L^2_tL^\infty_x})\|f\|^2_{L^2_x} + \big| \int^2_{-2} \int_{\mathbb{R}^3} LF \cdot Lfdx d\tau \big|.
\end{equation}
It remains for us to estimate the high order term. Taking derivatives $\Lambda^{s_0-2}_{x'}$ on \eqref{Q}, we have
\begin{equation}\label{Q0}
\begin{split}
  \partial_t \Lambda^{s_0-2}_{x'}\widetilde{Lf}+ (\tilde{\bv} \cdot \nabla) \Lambda^{s_0-2}_{x'}\widetilde{Lf}=
  & \Lambda^{s_0-2}_{x'}\left( \widetilde{[L,\mathbf{T}]f} \right)-\Lambda^{s_0-2}_{x'}( \partial_t \phi \partial_{x_3}\widetilde{Lf})
  \\
  & +\Lambda^{s_0-2}_{x'}\widetilde{LF} -\Lambda^{s_0-2}_{x'}\{(\tilde{\bv} \cdot \nabla) \phi \partial_{x_3} \widetilde{Lf}\}
  \\
  & -[\Lambda^{s_0-2}_{x'}, \tilde{v}^i \partial_{x_i}]\widetilde{Lf} .
  \end{split}
\end{equation}
Multiplying $\Lambda^{s_0-2}_{x'}\widetilde{Lf}$ on \eqref{Q0} and integrating it on $[-2,2]\times \mathbb{R}^3$, we have
\begin{equation*}
  \begin{split}
  \|\Lambda^{s_0-2}_{x'}\widetilde{Lf}\|_{L^2_{\Sigma}} \lesssim & \ \|d\bv\|_{L^1_t L^\infty_x}\|\Lambda^{s_0-2}_{x'}Lf\|^2_{L^{2}_x}+
   \|\Lambda^{s_0-2}_{x'} ( {[L,\mathbf{T}]f} )\|_{L^1_t L^2_x}\|\Lambda^{s_0-2}_{x'}Lf\|_{L^{2}_x}
  \\
  & \ +\big|\int^t_0 \Lambda^{s_0-2}_{x'}{LF} \cdot \Lambda^{s_0-2}_{x'}Lfdxd\tau \big|
  + \| [\Lambda^{s_0-2}_{x'}, \tilde{\bv} \cdot \nabla ]{Lf}\|_{L^1_t L^2_x}\|\Lambda^{s_0-2}_{x'}Lf\|_{L^{2}_x}
  \\
  & \ +\big| \int^2_{-2}\int_{\mathbb{R}^3} \big( \Lambda^{s_0-2}_{x'}( \partial_t \phi \partial_{x_3}{Lf})-\Lambda^{s_0-2}_{x'}\{(\tilde{\bv} \cdot \nabla) \phi \partial_{x_3} {Lf}\} \big)\Lambda^{s_0-2}_{x'}{Lf}dxd\tau \big|.
  \end{split}
\end{equation*}
We will estimate the right terms one by one. By using elliptic estimates, we can prove
\begin{equation}\label{Q1}
  \|d\bv\|_{L^1_t L^\infty_x}\|\Lambda^{s_0-2}_{x'}Lf\|^2_{L^{2}_x} \leq \|d\bv\|_{L^2_t L^2_x}\|f\|^2_{\dot{H}^{s_0-2}_x},
\end{equation}
and
\begin{equation}\label{Q40}
\begin{split}
   \|\Lambda^{s_0-2}_{x'} ( {[L,\mathbf{T}]f} )\|_{L^1_t L^2_x} \|\Lambda^{s_0-2}_{x'}Lf\|_{L^{2}_x}
  \lesssim  \ \|[L,\mathbf{T}]f\|_{L^1_t \dot{H}^{s_0-2}_x} \|f\|_{\dot{H}^{s_0-2}_x}.
\end{split}
\end{equation}
For the right term $\|[L,\mathbf{T}]f\|_{L^1_t \dot{H}^{s_0-2}_x}$ in \eqref{Q40}, by using Lemma \ref{ceR}, we have
\begin{equation*}
\begin{split}
  \|[L,\mathbf{T}]f\|_{L^1_t \dot{H}^{s_0-2}_x} & \lesssim \|\bv\|_{L^1_t\dot{B}^1_{\infty,\infty}}   \| f \|_{\dot{H}^{s_0-2}_x}+ \| \bv\|_{L^1_t\dot{B}^{s_0-1}_{\infty,\infty}}  \| f \|_{L^2_x}
  \\
 & \lesssim \| \partial \bv\|_{L^2_t L^{\infty}_x}\|f \|_{\dot{H}^{s_0-2}_x}+ \| \partial \bv\|_{L^2_t \dot{B}^{s_0-2}_{\infty,2}} \|f \|_{L^2_x}
 \\
 & \lesssim \big( \| \partial \bv\|_{L^2_t\dot{B}^{s_0-2}_{\infty,2}}+\| \partial \bv\|_{L^2_tL^{\infty}_x} \big) \|f\|_{{H}^{s_0-2}_x}.
\end{split}
\end{equation*}
Substituting in \eqref{Q40}, we obtain
\begin{equation}\label{Q4}
\begin{split}
   \|\Lambda^{s_0-2}_{x'} ( {[L,\mathbf{T}]f} )\|_{L^1_t L^2_x} \|\Lambda^{s_0-2}_{x'}Lf\|_{L^{2}_x}
  \lesssim  \big( \| \partial \bv\|_{L^2_t\dot{B}^{s_0-2}_{\infty,2}}+\| \partial \bv\|_{L^2_tL^{\infty}_x} \big) \|f\|^2_{{H}^{s_0-2}_x}.
\end{split}
\end{equation}
Using Lemma \ref{ce}, we also have
\begin{equation}\label{Q5}
\begin{split}
 \| [\Lambda^{s_0-2}_{x'}, \tilde{\bv} \cdot \nabla]{Lf}\|_{L^1_t L^2_x}\|\Lambda^{s_0-2}_{x'}Lf\|_{L^{2}_x}
\leq & \ \| [\Lambda^{s_0-2}_{x}, \tilde{\bv} \cdot \nabla]{Lf}\|_{L^1_t L^2_x}\|\Lambda^{s_0-2}_{x}Lf\|_{L^{2}_x}
\\
\lesssim & \ \| [\Lambda^{s_0-2}_{x}, \tilde{\bv} \cdot \nabla]{Lf}\|_{L^1_t L^2_x}\|f\|_{\dot{H}^{s_0-2}_x}
\\
\lesssim & \ \| \partial \bv\|_{L^1_t \dot{B}^0_{\infty,2}} \|Lf\|_{\dot{H}^{s_0-2}_x}\|f\|_{\dot{H}^{s_0-2}_x}
\\
\lesssim & \ \| \partial \bv\|_{L^2_t \dot{B}^0_{\infty,2}} \|f\|^2_{\dot{H}^{s_0-2}_x}.
\end{split}
\end{equation}
For $\phi$ is independent with $x_3$, we have
\begin{equation}\label{jhz}
\begin{split}
  & \int^2_{-2}\int_{\mathbb{R}^3} \left( \Lambda^{s_0-2}_{x'}( \partial_t \phi \partial_{x_3}{Lf})-\Lambda^{s_0-2}_{x'}\{ (\tilde{\bv} \cdot \nabla) \phi \partial_{x_3} {Lf}\} \right)\Lambda^{s_0-2}_{x'}{Lf}dxd\tau
  \\
  = \ & \int^2_{-2}\int_{\mathbb{R}^3} \left( [\Lambda^{s_0-2}_{x'}, \partial_t \phi \partial_{x_3}]{Lf})-[\Lambda^{s_0-2}_{x'}, (\tilde{\bv} \cdot \nabla) \phi \partial_{x_3}]{Lf} \right)\Lambda^{s_0-2}_{x'}{Lf}dxd\tau
  \\
  & + \int^2_{-2}\int_{\mathbb{R}^3} \left( \partial_t \phi \partial_{x_3}\Lambda^{s_0-2}_{x'}{Lf} - (\tilde{\bv} \cdot \nabla) \phi \partial_{x_3}\Lambda^{s_0-2}_{x'}{Lf} \right)\Lambda^{s_0-2}_{x'}{Lf}dxd\tau
  \\
  = \ & \int^2_{-2}\int_{\mathbb{R}^3} \left( [\Lambda^{s_0-2}_{x'}, \partial_t \phi \partial_{x_3}]{Lf})-[\Lambda^{s_0-2}_{x'}, (\tilde{\bv} \cdot \nabla) \phi \partial_{x_3}]{Lf} \right)\Lambda^{s_0-2}_{x'}{Lf}dxd\tau
  \\
  & + \int^2_{-2}\int_{\mathbb{R}^3} \partial_{x_3} \tilde{v}^i \partial_{x_i} \phi \Lambda^{s_0-2}_{x'}{Lf} \Lambda^{s_0-2}_{x'}{Lf}dxd\tau.
\end{split}
\end{equation}
We will estimate the right terms on \eqref{jhz}. For the first one and the second one, we use Lemma \ref{ce} and Lemma \ref{LPE} to bound
\begin{equation}\label{r1E}
\begin{split}
  & \big( \|[\Lambda^{s_0-2}_{x'}, \partial_t \phi \partial_{x_3}]{Lf}\|_{L^1_tL^2_x}+\|[\Lambda^{s_0-2}_{x'}, (\tilde{\bv} \cdot \nabla) \phi \partial_{x_3}]{Lf}\|_{L^1_tL^2_x} \big) \|Lf\|_{\dot{H}^{s_0-2}_x}
  \\
  \lesssim & \big( \|\partial(\partial_t \phi)\|_{L^1_t \dot{H}^{s_0-2}_x} + \partial(\bv \partial \phi)\|_{L^1_t \dot{H}^{s_0-2}_x} \big)\|Lf\|^2_{\dot{H}^{s_0-2}_x}
  \\
  \lesssim & \big(\| \partial(d \phi)\|_{L^2_t \dot{B}^{s_0-2}_{\infty,2}} + \|\partial \bv\|_{L^2_t \dot{B}^{s_0-2}_x} \|\partial \phi\|_{L^2_t L^\infty_x} + \|\partial \phi\|_{L^2_t C^{\beta}_x}\|\partial \bv\|_{L^2_t L^\infty_x} \big)\|Lf\|^2_{\dot{H}^{s_0-2}_x},
\end{split}
\end{equation}
where we take $\beta$=$s_0-\frac32-\epsilon_0>s_0-2$. By Sobolev imbedding, we can get
\begin{equation*}
  \| \partial(d \phi)\|_{L^2_t \dot{B}^{s_0-2}_{\infty,2}} \lesssim \| \partial(d \phi)\|_{L^2_t C^{\beta}} \lesssim  \| \partial(d \phi-dt)\|_{L^2_t H_{x'}^{s_0-1}(\Sigma)} \lesssim \vert\kern-0.25ex\vert\kern-0.25ex\vert d\phi-dt \vert\kern-0.25ex\vert\kern-0.25ex\vert_{s_0,2,\Sigma},
\end{equation*}
and
\begin{equation*}
  \|\partial \phi\|_{L^2_t L^\infty_x} + \|\partial \phi\|_{L^2_t C^{\beta}_x} \leq 1+\vert\kern-0.25ex\vert\kern-0.25ex\vert d \phi-dt\vert\kern-0.25ex\vert\kern-0.25ex\vert_{L^2_t C^{\beta}_x}\lesssim 1+ \vert\kern-0.25ex\vert\kern-0.25ex\vert d\phi-dt\vert\kern-0.25ex\vert\kern-0.25ex\vert_{s_0,2,\Sigma}.
\end{equation*}
Substituting in \eqref{r1E}, we can rewrite \eqref{r1E} as
\begin{equation}\label{r2E}
\begin{split}
  & \big( \|[\Lambda^{s_0-2}_{x'}, \partial_t \phi \partial_{x_3}]{Lf}\|_{L^1_tL^2_x}+\|[\Lambda^{s_0-2}_{x'}, (\tilde{\bv} \cdot \nabla) \phi \partial_{x_3}]{Lf}\|_{L^1_tL^2_x} \big) \|Lf\|_{\dot{H}^{s_0-2}_x}
  \\
  \lesssim & \|f\|^2_{\dot{H}^{s_0-2}_x} \big( \vert\kern-0.25ex\vert\kern-0.25ex\vert d\phi-dt\vert\kern-0.25ex\vert\kern-0.25ex\vert_{s_0,2,\Sigma}+ (\|\partial \bv\|_{L^2_t \dot{B}^{s_0-2}_x}+\|\partial \bv\|_{L^2_t L^\infty_x})(1+\vert\kern-0.25ex\vert\kern-0.25ex\vert d\phi-dt\vert\kern-0.25ex\vert\kern-0.25ex\vert_{s_0,2,\Sigma}) \big).
\end{split}
\end{equation}
The third term on the right hand side of \eqref{jhz} can be bounded by
\begin{equation}\label{r3E}
\begin{split}
  & \left|\int^2_{-2}\int_{\mathbb{R}^3} \partial_{x_3}\tilde{v}^i \partial_{x_i} \phi \Lambda^{s_0-2}_{x'}{Lf} \Lambda^{s_0-2}_{x'}{Lf}dxd\tau \right|
  \\
  \lesssim & \| \partial \bv\|_{L^2_t L^\infty_x} \|\partial \phi\|_{L^2_t L^\infty_x} \|Lf\|^2_{\dot{H}^{s_0-2}_x}
  \\
  \lesssim & (1+ \vert\kern-0.25ex\vert\kern-0.25ex\vert d\phi-dt\vert\kern-0.25ex\vert\kern-0.25ex\vert_{s_0,2,\Sigma})\| \partial \bv\|_{L^2_t L^\infty_x}\|f\|^2_{H^{s_0-2}_x}.
\end{split}
\end{equation}
Substituting \eqref{r2E} and \eqref{r3E}, we can estimate \eqref{jhz} by
\begin{equation}\label{Q6}
 \begin{split}
 & \left|\int^2_{-2}\int_{\mathbb{R}^3} \big( \Lambda^{s_0-2}_{x'}( \partial_t \phi \partial_{x_3}{Lf})-\Lambda^{s_0-2}_{x'}\{ (\bv \cdot \nabla) \phi \partial_{x_3} {Lf} \} \big) \Lambda^{s_0-2}_{x'}{Lf}  dxd\tau \right|
\\
 \lesssim & \|f\|^2_{\dot{H}^{s_0-2}_x} \big( \vert\kern-0.25ex\vert\kern-0.25ex\vert d\phi-dt\vert\kern-0.25ex\vert\kern-0.25ex\vert_{s_0,2,\Sigma}+ (\|\partial \bv\|_{L^2_t \dot{B}^{s_0-2}_x}+\|\partial \bv\|_{L^2_t L^\infty_x})(1+\vert\kern-0.25ex\vert\kern-0.25ex\vert d\phi-dt\vert\kern-0.25ex\vert\kern-0.25ex\vert_{s_0,2,\Sigma}) \big)
 \\
 \lesssim & \|f\|^2_{\dot{H}^{s_0-2}_x} \big( \vert\kern-0.25ex\vert\kern-0.25ex\vert d\phi-dt\vert\kern-0.25ex\vert\kern-0.25ex\vert_{s_0,2,\Sigma}+ \|\partial \bv\|_{L^2_t \dot{B}^{s_0-2}_x})(1+\vert\kern-0.25ex\vert\kern-0.25ex\vert d\phi-dt\vert\kern-0.25ex\vert\kern-0.25ex\vert_{s_0,2,\Sigma}) \big),
  \end{split}
\end{equation}
where we use the fact that $\|\partial \bv\|_{L^2_t L_x^\infty} \leq \|\partial \bv\|_{L^2_t \dot{B}^{s_0-2}_x}$.
Combining \eqref{Q1} to \eqref{Q6}, we can get
\begin{equation}\label{te31}
\begin{split}
  \vert\kern-0.25ex\vert\kern-0.25ex\vert \Lambda^{s_0-2}_{x'} Lf\vert\kern-0.25ex\vert\kern-0.25ex\vert^2_{0,2,\Sigma} \lesssim & \ \big( \| \partial \bv\|_{L^2_t\dot{B}^{s_0-2}_{\infty,2}}+\vert\kern-0.25ex\vert\kern-0.25ex\vert d\phi-dt \vert\kern-0.25ex\vert\kern-0.25ex\vert_{s_0,2,\Sigma} \big) \|f\|^2_{{H}^{s_0-2}_x}(1+\vert\kern-0.25ex\vert\kern-0.25ex\vert d\phi-dt \vert\kern-0.25ex\vert\kern-0.25ex\vert_{s_0,2,\Sigma})
  \\
  & + \left|\int^t_0 \int_{\mathbb{R}^3} \Lambda^{s_0-2}_{x'}{LF} \cdot \Lambda^{s_0-2}_{x'}Lfdxd\tau \right|.
\end{split}
\end{equation}
Adding \eqref{te30} and \eqref{te31}, we have
\begin{equation*}
\begin{split}
  \vert\kern-0.25ex\vert\kern-0.25ex\vert  Lf\vert\kern-0.25ex\vert\kern-0.25ex\vert^2_{s_0-2,2,\Sigma} \lesssim & \ \big( \| \partial \bv\|_{L^2_t\dot{B}^{s_0-2}_{\infty,2}}+\vert\kern-0.25ex\vert\kern-0.25ex\vert d\phi-dt \vert\kern-0.25ex\vert\kern-0.25ex\vert_{s_0,2,\Sigma} \big) \|f\|^2_{{H}^{s_0-2}_x}(1+\vert\kern-0.25ex\vert\kern-0.25ex\vert d\phi-dt \vert\kern-0.25ex\vert\kern-0.25ex\vert_{s_0,2,\Sigma})
  \\
  & + \sum_{\sigma\in\{0,s_0-2 \}} \left|\int^t_0 \int_{\mathbb{R}^3} \Lambda^{\sigma}_{x'}{LF} \cdot \Lambda^{\sigma}_{x'}Lfdxd\tau \right|.
\end{split}
\end{equation*}
Therefore, we complete the proof of Lemma \ref{te3}.
\end{proof}

We also need to give the estimate of $\mathrm{curl} \bw$ along the characteristic hypersurfaces, for $\mathrm{curl} \bw$ is a nonlinear term in the wave equation of the velocity. Then, $\mathrm{curl} \bw$ decides the regularity of the characteristic hypersurfaces.
\begin{Lemma}\label{te20}
Suppose that $(\bv, \rho, h, \bw) \in \mathcal{H}$. %Assuming $|||W|||_{s_0,2,\Sigma} \lesssim \epsilon_1$,
Then
\begin{equation}\label{te201}
\begin{split}
 \vert\kern-0.25ex\vert\kern-0.25ex\vert  \mathrm{curl}\bw  \vert\kern-0.25ex\vert\kern-0.25ex\vert_{s_0-1,2,\Sigma} \lesssim  \epsilon_2.
   \end{split}
\end{equation}
\end{Lemma}
\begin{proof}
By \eqref{d0}, we see
\begin{equation}\label{C1}
\begin{split}
  \vert\kern-0.25ex\vert\kern-0.25ex\vert  \mathrm{curl}\bw  \vert\kern-0.25ex\vert\kern-0.25ex\vert_{s_0-1,2,\Sigma}
  = & \| \mathrm{curl}\bw \|_{L^2_t H^{s_0-1}_{x'}(\Sigma)}+ \| \partial_t \mathrm{curl}\bw\|_{L^2_t H^{s_0-2}_{x'}(\Sigma)}.
\end{split}
\end{equation}
We will divide it into several steps for estimating the right terms on \eqref{C1}.

\textbf{Step 1: $\| \mathrm{curl}\bw \|_{L^2_t H^{s_0-1}_{x'}(\Sigma)}$}. Noting \eqref{W1}, we can simply write it as
\begin{equation}\label{TOE0}
\mathbf{T}  \mathrm{curl} \bw^i = ( \mathrm{curl} \bw \cdot \nabla) v^i+J^i.
\end{equation}
Here
\begin{equation*}
  \begin{split}
  J^i=&-  \mathrm{curl}\bw^i \mathrm{div}\bv-2\epsilon^{imn}\partial_m v^j \partial_n w_j
+  \bar{\rho}^{\gamma-1}  \partial^l ( \mathrm{e}^{h+(\gamma-2)\rho}) \partial_l \rho \partial^i h
\\
&- \bar{\rho}^{\gamma-2}  \partial^l ( \mathrm{e}^{h+(\gamma-2)\rho}) \partial^i \rho \partial_l h
+\bar{\rho}^{\gamma-1} \mathrm{e}^{h+(\gamma-2)\rho}    \Delta \rho \partial^i h + \bar{\rho}^{\gamma-1} \mathrm{e}^{h+(\gamma-2)\rho}  \partial^m \rho \partial_m\partial^i h
 \\
 &- \bar{\rho}^{\gamma-1} \mathrm{e}^{h+(\gamma-2)\rho}    \partial_m \partial^i \rho  \partial^m h - \bar{\rho}^{\gamma-1} \mathrm{e}^{h+(\gamma-2)\rho}  \partial^i \rho \Delta h.
  \end{split}
\end{equation*}
By changing of coordinates $x_3 \rightarrow x_3-\phi(t,x')$, then \eqref{TOE0} becomes to
\begin{equation}\label{TOE}
\begin{split}
  \partial_t\widetilde{\mathrm{curl}\bw^i}+ (\tilde{\bv} \cdot \nabla)\widetilde{\mathrm{curl}\bw^i}= & (\widetilde{\mathrm{curl}\bw}\cdot \nabla)\widetilde{v^i}+\widetilde{\mathrm{curl}\bw^j} \partial_j \phi \partial_{x_3}\widetilde{v^i}
  \\
  &
- \partial_t \phi \partial_{x_3}\widetilde{\mathrm{curl}\bw^i}-(\tilde{\bv} \cdot \nabla) \phi \partial_{x_3} \widetilde{\mathrm{curl}\bw^i}+\widetilde{J}^i.
  \end{split}
\end{equation}
Multiplying $\widetilde{\mathrm{curl}\bw_i}$ and integrating it on $[-2,2]\times \mathbb{R}^3$, we can obtain
\begin{equation*}
  \begin{split}
  \| \mathrm{curl}\bw \|^2_{L^2_t L^2_{x'}(\Sigma)} \lesssim &  \|\partial \bv\|_{L^1_tL^\infty_x}  \| \mathrm{curl}\bw \|^2_{L^\infty_t L^2_x}(1+\|\partial \phi \|_{L^\infty_t L^\infty_x})
  \\
  & +  \|\bJ \|_{L^1_t L^2_x} \| \mathrm{curl}\bw \|_{L^\infty_t L^2_x}
  %\\
%  &+ \|\partial \bv\|_{L^1_tL^\infty_x} \| \bW\|_{L^\infty_t L^2_x} \|\bW\|_{L^\infty_t L^2_x}\|\partial \phi \|_{L^\infty_t L^\infty_x}
  + \|\partial d\phi \|_{L^2_t L^\infty_x}\|\mathrm{curl}\bw\|^2_{L^\infty_t L^2_x}( 1+ \|\partial \bv\|_{L^2_tL^\infty_x} ).
  \end{split}
\end{equation*}
Using \eqref{402}, and \eqref{403}, we have
\begin{equation}\label{OM0}
  \| \mathrm{curl}\bw \|^2_{L^2_t L^2_{x'}(\Sigma)} \lesssim \epsilon^3_2.
\end{equation}
It remains for us to bound $\| \Lambda^{s_0-1}_{x'}(\mathrm{curl}\bw)\|_{L^2_t L^2_{x'}(\Sigma)}$. Let us first estimate $\|  \partial (\mathrm{curl}\bw)\|_{L^2_t H^{s_0-2}_{x'}(\Sigma)}$. By the Hodge's decomposition,
\begin{equation}\label{OM1}
\begin{split}
  \|  \partial (\mathrm{curl}\bw)\|_{L^2_t H^{s_0-2}_{x'}(\Sigma)} &= \| L ( \mathrm{curl} \mathrm{curl}\bw)+H ( \mathrm{div} \mathrm{curl}\bw)\|_{L^2_t H^{s_0-2}_{x'}(\Sigma)}
 \\
  & = \| L ( \mathrm{curl} \mathrm{curl}\bw)\|_{L^2_t H^{s_0-2}_{x'}(\Sigma)},
\end{split}
\end{equation}
where the operators $L$ and $H$ are given by $\partial(-\Delta_x)^{-1}\mathrm{curl}$ and $\partial(-\Delta_x)^{-1}\nabla$ respectively. By \eqref{W2}, we can get
\begin{equation}\label{TO}
\begin{split}
 \mathbf{T} \underline{F}^i
=  \partial^i \mathcal{G} + K^i+\mathbf{T}(-\frac{1}{\gamma}\mathrm{e}^{-\rho} \epsilon^{ijk} \partial_k h \Delta v_j),
\end{split}
\end{equation}
%We set $f=(f^1,f^2,f^3)^{\mathrm{T}}$ and $F=(F^1,F^2,F^3)^{\mathrm{T}}$, where
where we set $\underline{\bF}=(\underline{F}^1,\underline{F}^2,\underline{F}^3)$, $G=2 \partial_n v_a \partial^n w^a$, and
\begin{equation*}
 \underline{F}^i= \mathrm{curl} \mathrm{curl}\bw^i -\epsilon^{ijk} \partial_j \rho \cdot \mathrm{curl}\bw_k- 2\partial^a {\rho} \partial^i w_a+2 \mathrm{e}^{-\rho} \epsilon^{ijk} \partial_j v^m  \partial_m\partial_k h.
\end{equation*}
Operating the operator $L$ on \eqref{TO}, and using Lemma \ref{te3}, we can deduce that
\begin{equation}\label{L0}
\begin{split}
  \vert\kern-0.25ex\vert\kern-0.25ex\vert L \underline{\bF} \vert\kern-0.25ex\vert\kern-0.25ex\vert^2_{s_0-2,2,\Sigma} \lesssim  \ & \ \big( \| \partial \bv \|_{L^2_t\dot{B}^{s_0-2}_{\infty,2}}+\vert\kern-0.25ex\vert\kern-0.25ex\vert d\phi-dt \vert\kern-0.25ex\vert\kern-0.25ex\vert_{s_0,2,\Sigma} \big) \|\underline{\bF}\|^2_{{H}^{s_0-2}_x}(1+\vert\kern-0.25ex\vert\kern-0.25ex\vert d\phi-dt \vert\kern-0.25ex\vert\kern-0.25ex\vert_{s_0,2,\Sigma})
  \\
  & +  \sum_{\sigma \in \{0,s_0-2 \}}\left| \int^2_{-2} \int_{\mathbb{R}^3}  \Lambda_{x'}^{\sigma}\widetilde{LK^i} \cdot \Lambda_{x'}^{\sigma} \widetilde{L\underline{F}_i}dxd\tau    \right|
  \\
  & +  \sum_{\sigma \in \{0,s_0-2 \}}\left| \int^2_{-2} \int_{\mathbb{R}^3}  \Lambda_{x'}^{\sigma}\left\{(\partial^i+\partial^i \phi \partial^3)\widetilde{L\mathcal{G}}\right\} \cdot \Lambda_{x'}^{\sigma} \widetilde{L\underline{F}_i}dxd\tau    \right|
  \\
  & +  \sum_{\sigma \in \{0,s_0-2 \}}\left| \int^2_{-2} \int_{\mathbb{R}^3}  \Lambda_{x'}^{\sigma} \left\{\mathbf{T}\widetilde{\{L(-\mathrm{e}^{-\rho} \epsilon^{ijk} \partial_k h \Delta v_j)\}}\right\} \cdot \Lambda_{x'}^{\sigma} \widetilde{L\underline{F}_i}dxd\tau    \right|.
  \end{split}
\end{equation}
We set
\begin{equation*}
\begin{split}
\mathrm{I}&=\sum_{\sigma \in \{0,s_0-2 \}}\left| \int^2_{-2} \int_{\mathbb{R}^3}  \Lambda_{x'}^{\sigma}\widetilde{LK^i} \cdot \Lambda_{x'}^{\sigma} \widetilde{L\underline{F}_i}dxd\tau    \right|,
\\
\mathrm{J}&=   \sum_{\sigma \in \{0,s_0-2 \}}\left| \int^2_{-2} \int_{\mathbb{R}^3}  \Lambda_{x'}^{\sigma}\left\{(\partial^i+\partial^i \phi \partial^3)\widetilde{L\mathcal{G}}\right\} \cdot \Lambda_{x'}^{\sigma} \widetilde{L\underline{F}_i}dxd\tau    \right|,
\\
\bar{\mathrm{J}}&=\sum_{\sigma \in \{0,s_0-2 \}}\left| \int^2_{-2} \int_{\mathbb{R}^3}  \Lambda_{x'}^{\sigma} \left\{\mathbf{T}\widetilde{\{L(-\mathrm{e}^{-\rho} \epsilon^{ijk} \partial_k h \Delta v_j)\}}\right\} \cdot \Lambda_{x'}^{\sigma} \widetilde{L\underline{F}_i}dxd\tau    \right|.
\end{split}
\end{equation*}
For $\sigma \in \{0,s_0-2 \}$, we use H\"older's inequality to give the bound
\begin{equation}\label{J20}
\begin{split}
  \mathrm{I}
  \leq  &   \textstyle{\sum_{\sigma \in \{0,s_0-2 \}}}\| \Lambda_{x'}^{\sigma} L \bK \|_{L^1_t L^2_x} \|\Lambda_{x'}^{\sigma} L\underline{\bF } \|_{L^\infty_t L^2_x}
  \\
  \lesssim & \| \bK\|_{L^1_t H^{s_0-2}_x} \| \underline{\bF}\|_{L^\infty_t H^{s_0-2}_x}.
\end{split}
\end{equation}
Using \eqref{rF} and Sobolev inequalities, we get
\begin{equation}\label{R}
  \begin{split}
 \| \bK\|_{L^1_t H^{s_0-2}_x}
 \lesssim   & \|d \rho, d \bv, dh\|_{L^1_t \dot{B}^{s_0-2}_{\infty,2}}(\|\bv\|_{H^{s}}+\| \rho \|_{H^{s}}+\| \bw \|_{H^{s_0}}+\| h \|_{H^{s_0+1}})
 \\
 & + \|d \rho, d \bv, dh\|_{L^1_t \dot{B}^{s_0-2}_{\infty,2}}(\|\bv\|^2_{H^{s}}+\| \rho \|^2_{H^{s}}+\| \bw \|^2_{H^{s_0}}+\| h \|^2_{H^{s_0+1}})
 \\
 & + \|d \rho, d \bv, dh\|_{L^1_t \dot{B}^{s_0-2}_{\infty,2}}(\|\bv\|^3_{H^{s}}+\| \rho \|^3_{H^{s}}+\| \bw \|^3_{H^{s_0}}+\| h \|^3_{H^{s_0+1}})
 \\
 &+\|\bv\|^2_{H^{s}}+\| \rho \|^2_{H^{s}}+\| \bw \|^2_{H^{s_0}}+\| h \|^2_{H^{s_0+1}}
 +\|\bv\|^3_{H^{s}}+\| \rho \|^3_{H^{s}}
 \\
 &+\| \bw \|^3_{H^{s_0}}+\| h \|^3_{H^{s_0+1}}
 +\|\bv\|^4_{H^{s}}+\| \rho \|^4_{H^{s}}+\| \bw \|^4_{H^{s_0}}+\| h \|^4_{H^{s_0+1}}.
  \end{split}
\end{equation}
By H\"older's inequality, we also get
\begin{equation}\label{R1}
  \|\underline{\bF}\|_{L^\infty_t H^{s_0-2}_x} \lesssim \| \bw \|_{H_x^{s_0}}+ (\|\bv\|_{H_x^s}+\| \rho \|_{H^{s}}+\| \bw \|_{H^{s_0}}+\| h \|_{H^{s_0+1}})^2.
\end{equation}
Inserting \eqref{R} and \eqref{R1} into \eqref{J2} and using \eqref{401}-\eqref{403}, we have
\begin{equation}\label{Ij2}
\begin{split}
  \mathrm{I} \lesssim &\ \epsilon^2_2.
\end{split}
\end{equation}
For $\mathrm{J}$, we separate it as
\begin{equation*}
  \mathrm{J} \leq \sum_{\sigma \in \{0,s_0-2 \}}(|J^1|+ |J^2|+|J^3|+|J^4|+|J^5|+|J^6|+|J^7|+|J^8|),
\end{equation*}
where
\begin{equation*}
  \begin{split}
  J_1= &   \int^2_{-2} \int_{\mathbb{R}^3}  (\partial^i+\partial^i \phi \partial_{x_3}) \Lambda_{x'}^{\sigma}(\widetilde{L\mathcal{G}}) \cdot  \Lambda_{x'}^{\sigma} \widetilde{(L\mathrm{curl}\mathrm{curl}\bw_i)}dxd\tau ,
  \\
  J_2= &   \int^2_{-2} \int_{\mathbb{R}^3} [\Lambda_{x'}^{\sigma}, \partial^i+\partial^i \phi \partial_{x_3}](\widetilde{L\mathcal{G}}) \cdot \Lambda_{x'}^{\sigma} \widetilde{(L\mathrm{curl}\mathrm{curl}\bw_i)} dxd\tau ,
  \\
  J_3= &   \int^2_{-2} \int_{\mathbb{R}^3} (\partial^i+\partial^i \phi \partial_{x_3}) \Lambda_{x'}^{\sigma} \widetilde{L \big( 2 \partial_n v^a \partial^n w_a \big)} \cdot \Lambda_{x'}^{\sigma} \widetilde{L \big(-2 \partial_a \rho \partial_i w^a \big)}  dxd\tau ,
  \\
   J_4= &   \int^2_{-2} \int_{\mathbb{R}^3} [\Lambda_{x'}^{\sigma}, \partial^i+\partial^i \phi \partial_{x_3}]  \widetilde{L \big( 2  \partial_n v^a \partial^n w_a \big)} \cdot \Lambda_{x'}^{\sigma} \widetilde{L \big(-2 \partial_a \rho \partial_i w^a \big)}  dxd\tau
   \\
  J_5= &   \int^2_{-2} \int_{\mathbb{R}^3} (\partial^i+\partial^i \phi \partial_{x_3}) \Lambda_{x'}^{\sigma} \widetilde{L \big( 2  \partial_n v^a \partial^n w_a \big)} \cdot \Lambda_{x'}^{\sigma} \widetilde{L \big(2\mathrm{e}^{-\rho} \epsilon_{ijk}\partial^j v_m \partial^m\partial^k h \big)}  dxd\tau ,
  \\
   J_6= &   \int^2_{-2} \int_{\mathbb{R}^3} [\Lambda_{x'}^{\sigma}, \partial^i+\partial^i \phi \partial_{x_3}]  \widetilde{L \big( 2  \partial_n v^a \partial^n w_a \big)} \cdot \Lambda_{x'}^{\sigma} \widetilde{L \big( 2\mathrm{e}^{-\rho} \epsilon_{ijk}\partial^j v_m \partial^m\partial^k h \big)}  dxd\tau,
   \\
  J_7= &   \int^2_{-2} \int_{\mathbb{R}^3} (\partial^i+\partial^i \phi \partial_{x_3}) \Lambda_{x'}^{\sigma} \widetilde{L \big( 2  \partial_n v^a \partial^n w_a \big)} \cdot \Lambda_{x'}^{\sigma} \widetilde{L \big(-\mathrm{e}^{-\rho} \epsilon_{ijk}\partial^k h  \Delta v^j \big)}  dxd\tau ,
  \\
   J_8= &   \int^2_{-2} \int_{\mathbb{R}^3} [\Lambda_{x'}^{\sigma}, \partial^i+\partial^i \phi \partial_{x_3}]  \widetilde{L \big( 2  \partial_n v^a \partial^n w_a \big)} \cdot \Lambda_{x'}^{\sigma} \widetilde{L \big( -\mathrm{e}^{-\rho} \epsilon_{ijk}\partial^k h  \Delta v^j \big)}  dxd\tau.
  \end{split}
  \end{equation*}
Let us bound the above terms one by one. For $J_1$, we have
\begin{equation*}
\begin{split}
  J_1= &  \int^2_{-2} \int_{\mathbb{R}^3}  (\partial^i+\partial^i \phi \partial_{x_3}) \Lambda_{x'}^{\sigma}(\widetilde{L\mathcal{G}}) \cdot  \Lambda_{x'}^{\sigma} \widetilde{(L\mathrm{curl}\mathrm{curl}\bw_i )}dxd\tau
  \\
  =& -\int^2_{-2} \int_{\mathbb{R}^3}  \Lambda_{x'}^{\sigma}(\widetilde{L\mathcal{G}}) \cdot  (\partial^i+\partial^i \phi \partial_{x_3}) \big\{ \Lambda_{x'}^{\sigma} \widetilde{(L\mathrm{curl}\mathrm{curl}\bw_i)} \big\} dxd\tau
  \\
  =& -\int^2_{-2} \int_{\mathbb{R}^3}  \Lambda_{x'}^{\sigma}(\widetilde{L\mathcal{G}}) \cdot  [\partial^i+\partial^i \phi \partial_{x_3}, \Lambda_{x'}^{\sigma}]  \widetilde{(L\mathrm{curl}\mathrm{curl}\bw_i)}  dxd\tau
  \\
  & + \int^2_{-2} \int_{\mathbb{R}^3}  \Lambda_{x'}^{\sigma}(\widetilde{L\mathcal{G}}) \cdot  \Lambda_{x'}^{\sigma}  \big\{ (\partial^i+\partial^i \phi \partial_{x_3}) \widetilde{(L\mathrm{curl}\mathrm{curl}\bw_i)} \big\} dxd\tau .
\end{split}
\end{equation*}
Using the fact
\begin{equation*}
\begin{split}
  (\partial^i+\partial^i \phi \partial_{x_3}) \widetilde{(L\mathrm{curl}\mathrm{curl}\bw_i)}&= (\partial^i L\mathrm{curl}\mathrm{curl}\bw_i)(x_1,x_2,x_3+\phi(t,x'))
  \\
  &= \{L(\mathrm{div} \mathrm{curl}\mathrm{curl}\bw_i)\}(x_1,x_2,x_3+\phi(t,x')) =0,
\end{split}
\end{equation*}
then
\begin{equation*}
\begin{split}
  J_1= &  -\int^2_{-2} \int_{\mathbb{R}^3}  \Lambda_{x'}^{\sigma}(\widetilde{L\mathcal{G}}) \cdot  [\partial^i+\partial^i \phi \partial_{x_3}, \Lambda_{x'}^{\sigma}]  \widetilde{(L\mathrm{curl} \mathrm{curl}\bw_i)}  dxd\tau.
\end{split}
\end{equation*}
For $\sigma \in \{0,s_0-2 \}$, by H\"older's inequality and commutator estimates in Lemma \ref{ce}, we have
\begin{equation}\label{J1}
\begin{split}
  \textstyle{\sum_{\sigma \in \{0,s_0-2 \}}}(|J_1|+|J_2|)= &  \| L\mathcal{G} \|_{{H}_x^{s_0-2}} \|\partial \phi\|_{L^2_t L^\infty_x} \|L\mathrm{curl}\mathrm{curl}\bw\|_{{H}_x^{s_0-2}}
  \\
  \lesssim & \|\partial \phi\|_{L^2_t L^\infty_x}(\|\bv\|^3_{H^s_x}+ \|\bw\|^3_{H^{s_0}_x}+\|\rho\|^3_{{H}_x^{s}}).
\end{split}
\end{equation}
For $J_3$, we use the Plancherel formula in $\mathbb{R}^3$ which gives
\begin{equation}\label{J36}
\begin{split}
  J_3= &\int^2_{-2} \int_{\mathbb{R}^3} \partial^i \Lambda_{x'}^{\sigma} \widetilde{L \big( 2  \partial_n v^a \partial^n w_a \big)} \cdot  \Lambda_{x'}^{\sigma} \widetilde{L \big(-2 \partial_a \rho \partial_i w^a \big)}  dxd\tau
  \\
  & + \int^2_{-2} \int_{\mathbb{R}^3}   \partial_{x_3} \Lambda_{x'}^{\sigma} \widetilde{L \big( 2 \partial_n v^a \partial^n w_a \big)} \cdot \partial^i \phi  \Lambda_{x'}^{\sigma} \widetilde{L \big(-2 \partial_a \rho \partial_i w^a \big)}  dxd\tau
  \\
  = &\int^2_{-2} \int_{\mathbb{R}^3} \Lambda_x^{-\frac{1}{2}}\partial^i \Lambda_{x'}^{\sigma} \widetilde{L \big( 2 \partial_n v^a \partial^n w_a \big)} \cdot \Lambda_x^{\frac{1}{2}} \Lambda_{x'}^{\sigma} \widetilde{L \big(-2 \partial_a \rho \partial_i w^a \big)}  dxd\tau
  \\
  & + \int^2_{-2} \int_{\mathbb{R}^3}   \Lambda_x^{-\frac{1}{2}} \partial_{x_3} \Lambda_{x'}^{\sigma} \widetilde{L \big( 2 \partial_n v^a \partial^n w_a \big)} \cdot \Lambda_x^{\frac{1}{2}} \left( \partial^i \phi  \Lambda_{x'}^{\sigma} \widetilde{ L \big(-2 \partial_a \rho \partial_i w^a \big)} \right)  dxd\tau.
\end{split}
\end{equation}
For $\sigma \in \{0,s_0-2 \}$, by H\"older's inequality, we can obtain
\begin{equation*}
\begin{split}
  & \textstyle{\sum_{\sigma \in \{0,s_0-2 \}}}|J_3|
  \\
  = & \|  \partial \bv \partial \bw  \|_{L^\infty_t H^{s_0-\frac{3}{2}}_x}\|  \partial \rho \partial \bw  \|_{L^\infty_t H^{s_0-\frac{3}{2}}_x}+\|  \partial \phi  \|_{L^2_t C^{\frac{1}{2}}_x} \|  \partial \bv \partial \bw  \|_{L^\infty_t H^{s_0-\frac{3}{2}}_x}\|  \partial \rho \partial \bw  \|_{L^\infty_t H^{s_0-\frac{3}{2}}_x}
  \\
  \lesssim & \ \| \rho \|_{H^{s_0}} \| \bv \|_{H^{s_0}}\|\bw\|_{H^{s_0}}(1+ \vert\kern-0.25ex\vert\kern-0.25ex\vert d\phi-dt\vert\kern-0.25ex\vert\kern-0.25ex\vert_{s_0,2,\Sigma})
 %  \\
 % \lesssim & \ \|d \boldsymbol{\rho}, d v\|^2_{L^2_t L^\infty_x}\|W\|^2_{H^2} + \|  W\|^2_{L^\infty_t H^2_x}\| v\|_{L^\infty_t H^2_x}\|d \boldsymbol{\rho}\|_{L^2_t L^\infty_x}.
\end{split}
\end{equation*}
By using \eqref{401}-\eqref{403}, we get
\begin{equation}\label{J3}
\begin{split}
  \textstyle{\sum_{\sigma \in \{0,s_0-2 \}}}|J_3| \lesssim & \ \epsilon^2_2.
\end{split}
\end{equation}
For $\sigma \in \{0,s_0-2 \}$, by H\"older's inequality, Lemma \ref{ce}, and \eqref{401}-\eqref{403}, we have
\begin{equation}\label{J4}
\begin{split}
  \textstyle{\sum_{\sigma \in \{0,s_0-2 \}}}|J_4| \lesssim \vert\kern-0.25ex\vert\kern-0.25ex\vert d\phi-dt\vert\kern-0.25ex\vert\kern-0.25ex\vert_{s_0,2,\Sigma}(\| \rho \|^4_{H^{s_0}}+ \| \bv \|^4_{H^{s}}+\|\bw\|^4_{H^{s_0}})\lesssim \epsilon^2_2.
\end{split}
\end{equation}
Calculate
\begin{equation}\label{j10}
  \begin{split}
  &(\partial^i+\partial^i \phi \partial_{x_3}) \Lambda_{x'}^{\sigma} \widetilde{L \big( 2  \partial_n v^a \partial^n w_a \big)} \cdot \Lambda_{x'}^{\sigma} \widetilde{L \big(2\mathrm{e}^{-\rho} \epsilon_{ijk}\partial^j v_m \partial^m\partial^k h \big)}
  \\
  =& (\partial^i+\partial^i \phi \partial_{x_3}) \left\{ \Lambda_{x'}^{\sigma} \widetilde{L \big( 2 \partial_n v^a \partial^n w_a \big)} \cdot \Lambda_{x'}^{\sigma} \widetilde{L \big(2\mathrm{e}^{-\rho} \epsilon_{ijk}\partial^j v_m \partial^m\partial^k h \big)} \right\}
  \\
  & \ - \Lambda_{x'}^{\sigma} \widetilde{L \big( 2 \partial_n v^a \partial^n w_a \big)} \cdot \Lambda_{x'}^{\sigma} (\partial^i+\partial^i \phi \partial_{x_3})\widetilde{L \big(2\mathrm{e}^{-\rho} \epsilon_{ijk}\partial^j v_m \partial^m\partial^k h \big)}
  \\
  & \ + \Lambda_{x'}^{\sigma} \widetilde{L \big( 2  \partial_n v^a \partial^n w_a \big)} \cdot [\Lambda_{x'}^{\sigma},\partial^i+\partial^i \phi \partial_{x_3}]\widetilde{L \big(2\mathrm{e}^{-\rho} \epsilon_{ijk}\partial^j v_m \partial^m\partial^k h \big)}
  \\
  =& (\partial^i+\partial^i \phi \partial_{x_3}) \left\{ \Lambda_{x'}^{\sigma} \widetilde{L \big( 2  \partial_n v^a \partial^n w_a \big)} \cdot \Lambda_{x'}^{\sigma} \widetilde{L \big(2\mathrm{e}^{-\rho} \epsilon_{ijk}\partial^j v_m \partial^m\partial^k h \big)} \right\}
  \\
  & \ - \Lambda_{x'}^{\sigma} \widetilde{L \big( 2  \partial_n v^a \partial^n w_a \big)} \cdot \Lambda_{x'}^{\sigma} \widetilde{\partial^i L \big(2\mathrm{e}^{-\rho} \epsilon_{ijk}\partial^j v_m \partial^m\partial^k h \big)}
  \\
  & \ + \Lambda_{x'}^{\sigma} \widetilde{L \big( 2  \partial_n v^a \partial^n w_a \big)} \cdot [\Lambda_{x'}^{\sigma},\partial^i+\partial^i \phi \partial_{x_3}]\widetilde{L \big(2\mathrm{e}^{-\rho} \epsilon_{ijk}\partial^j v_m \partial^m\partial^k h \big)}.
  \end{split}
\end{equation}
Taking derivatives, we have
\begin{equation*}
\begin{split}
  \partial^i L \big(2\mathrm{e}^{-\rho} \epsilon_{ijk}\partial^j v_m \partial^m\partial^k h \big)=  & L \big(-2\mathrm{e}^{-\rho}  \epsilon_{ijk}\partial^i \rho \partial^j v_m \partial^m\partial^k h \big)
  + L \big(2\mathrm{e}^{-\rho} \epsilon_{ijk}\partial^j v_m \partial^i\partial^m\partial^k h \big),
\end{split}
\end{equation*}
so \eqref{j10} yields
\begin{equation}\label{j02}
  \begin{split}
  &(\partial^i+\partial^i \phi \partial_{x_3}) \Lambda_{x'}^{\sigma} \widetilde{L \big( 2 \partial_n v^a \partial^n w_a \big)} \cdot \Lambda_{x'}^{\sigma} \widetilde{L \big(2\mathrm{e}^{-\rho} \epsilon_{ijk}\partial^j v_m \partial^m\partial^k h \big)}
  \\
  =& (\partial^i+\partial^i \phi \partial_{x_3}) \{ \Lambda_{x'}^{\sigma} \widetilde{L \big( 2\partial_n v^a \partial^n w_a \big)} \cdot \Lambda_{x'}^{\sigma} \widetilde{L \big(2\mathrm{e}^{-\rho} \epsilon_{ijk}\partial^j v_m \partial^m\partial^k h \big)} \}
  \\
  & \ - \Lambda_{x'}^{\sigma} \widetilde{L \big( 2 \partial_n v^a \partial^n w_a \big)} \cdot \Lambda_{x'}^{\sigma} \widetilde{ L \big(2\mathrm{e}^{-\rho} \epsilon_{ijk}\partial^j v_m \partial^i \partial^m\partial^k h \big)}
  \\
  & \ + \Lambda_{x'}^{\sigma} \widetilde{L \big( 2  \partial_n v^a \partial^n w_a \big)} \cdot \Lambda_{x'}^{\sigma} \widetilde{ L \big(2\mathrm{e}^{-\rho} \epsilon_{ijk}\partial^i \rho \partial^j v_m  \partial^m \partial^k h \big)}
  \\
  & \ + \Lambda_{x'}^{\sigma} \widetilde{L \big( 2  \partial_n v^a \partial^n w_a \big)} \cdot [\Lambda_{x'}^{\sigma},\partial^i+\partial^i \phi \partial_{x_3}]\widetilde{L \big(2\mathrm{e}^{-\rho} \epsilon_{ijk}\partial^j v_m \partial^m\partial^k h \big)}.
  \end{split}
\end{equation}
Here we use the fact that $2\mathrm{e}^{-\rho} \epsilon_{ijk}\partial^i\partial^j v_m \partial^m\partial^k h=0$.

Inserting \eqref{j02} in $J_5$, and using H\"older's inequality, \eqref{401}-\eqref{403}, we  obtain
\begin{equation}\label{J5}
  \begin{split}
  \textstyle{\sum_{\sigma \in \{0,s_0-2 \}}}|J_5| \lesssim & (\|d\bv\|_{L^2_t L^\infty_x}+ \|d\bv\|_{L^2_t \dot{B}^{s_0-2}_{\infty,2}})(\| \rho \|^3_{H^{s_0}}+ \| \bv \|^3_{H^{s}}+\|\bw\|^3_{H^{s_0}}+\|h\|^3_{H^{s_0+1}})
   \\
   &+(1+\|d\phi-dt\|_{s_0,2,\Sigma})(\| \rho \|^4_{H^{s_0}}+ \| \bv \|^4_{H^{s}}+\|\bw\|^4_{H^{s_0}}+\|h\|^4_{H^{s_0+1}}) \\
  &+(1+\|d\phi-dt\|_{s_0,2,\Sigma})(\| \rho \|^5_{H^{s_0}}+ \| \bv \|^5_{H^{s}}+\|\bw\|^5_{H^{s_0}}+\|h\|^5_{H^{s_0+1}})
\\
  \lesssim & \epsilon^2_2.
  \end{split}
\end{equation}
For $J_7$, we use a similar idea on handling $J_5$. Calculate
\begin{equation}\label{j07}
\begin{split}
    &(\partial^i+\partial^i \phi \partial_{x_3}) \Lambda_{x'}^{\sigma} \widetilde{L \big( 2  \partial_n v^a \partial^n w_a \big)} \cdot \Lambda_{x'}^{\sigma} \widetilde{L \big(-\mathrm{e}^{-\rho} \epsilon_{ijk}\partial^k h  \Delta v^j \big)}
    \\
    =&(\partial^i+\partial^i \phi \partial_{x_3}) \left\{ \Lambda_{x'}^{\sigma} \widetilde{L \big( 2  \partial_n v^a \partial^n w_a \big)} \cdot \Lambda_{x'}^{\sigma} \widetilde{L \big(-\mathrm{e}^{-\rho} \epsilon_{ijk}\partial^k h  \Delta v^j \big)} \right\}
    \\
    &- \Lambda_{x'}^{\sigma} \widetilde{L \big( 2  \partial_n v^a \partial^n w_a \big)} \cdot \Lambda_{x'}^{\sigma} (\partial^i+\partial^i \phi \partial_{x_3}) \widetilde{L \big(-\mathrm{e}^{-\rho} \epsilon_{ijk}\partial^k h  \Delta v^j \big)}
    \\
    &+ \Lambda_{x'}^{\sigma} \widetilde{L \big( 2  \partial_n v^a \partial^n w_a \big)} \cdot [\Lambda_{x'}^{\sigma}, \partial^i+\partial^i \phi \partial_{x_3}] \widetilde{L \big(-\mathrm{e}^{-\rho} \epsilon_{ijk}\partial^k h  \Delta v^j \big)}
    \\
    =&(\partial^i+\partial^i \phi \partial_{x_3}) \left\{ \Lambda_{x'}^{\sigma} \widetilde{L \big( 2  \partial_n v^a \partial^n w_a \big)} \cdot \Lambda_{x'}^{\sigma} \widetilde{L \big(-\mathrm{e}^{-\rho} \epsilon_{ijk}\partial^k h  \Delta v^j \big)} \right\}
    \\
    &- \Lambda_{x'}^{\sigma} \widetilde{L \big( 2  \partial_n v^a \partial^n w_a \big)} \cdot \Lambda_{x'}^{\sigma} \widetilde{L \{ \partial^i \big(-\mathrm{e}^{-\rho} \epsilon_{ijk}\partial^k h  \Delta v^j \big) \} }
    \\
    &+ \Lambda_{x'}^{\sigma} \widetilde{L \big( 2 \partial_n v^a \partial^n w_a \big)} \cdot [\Lambda_{x'}^{\sigma}, \partial^i+\partial^i \phi \partial_{x_3}] \widetilde{L \big(-\mathrm{e}^{-\rho} \epsilon_{ijk}\partial^k h  \Delta v^j \big)}
    \\
    =&(\partial^i+\partial^i \phi \partial_{x_3}) \left\{ \Lambda_{x'}^{\sigma} \widetilde{L \big( 2  \partial_n v^a \partial^n w_a \big)} \cdot \Lambda_{x'}^{\sigma} \widetilde{L \big(-\mathrm{e}^{-\rho} \epsilon_{ijk}\partial^k h  \Delta v^j \big)} \right\}
    \\
    &+ \Lambda_{x'}^{\sigma} \widetilde{L \big( 2  \partial_n v^a \partial^n w_a \big)} \cdot \Lambda_{x'}^{\sigma} \widetilde{L \{ \big(-2\mathrm{e}^{-\rho} \partial^i \rho \epsilon_{ijk}\partial^k h  \Delta v^j+\mathrm{e}^{-\rho}  \partial^k h \Delta ( \mathrm{e}^{\rho}  \bw_k) \big) \} }
    \\
    &+ \Lambda_{x'}^{\sigma} \widetilde{L \big( 2  \partial_n v^a \partial^n w_a \big)} \cdot [\Lambda_{x'}^{\sigma}, \partial^i+\partial^i \phi \partial_{x_3}] \widetilde{L \big(-\mathrm{e}^{-\rho} \epsilon_{ijk}\partial^k h  \Delta v^j \big)},
\end{split}
\end{equation}
where we use the fact $\mathrm{e}^{-\rho} \epsilon_{ijk}\partial^i(\partial^k h  \Delta v^j)=\mathrm{e}^{-\rho}  \partial^k h \Delta ( \mathrm{e}^{\rho}  \bw_k)$. Inserting \eqref{j07} to $J_7$, we get
\begin{equation}\label{J7}
  \begin{split}
  \textstyle{\sum_{\sigma \in \{0,s_0-2 \}}}|J_7| \lesssim & (\|(d\bv,dh)\|_{L^2_t L^\infty_x}+ \|(d\bv,dh)\|_{L^2_t \dot{B}^{s_0-2}_{\infty,2}})(\| \rho,\bv \|^3_{H^{s}}+\|\bw\|^3_{H^{s_0}}+\|h\|^3_{H^{s_0+1}})
   \\
   &+(1+\|d\phi-dt\|_{s_0,2,\Sigma})(\| \rho,\bv \|^4_{H^{s}}+\|\bw\|^4_{H^{s_0}}+\|h\|^4_{H^{s_0+1}}) \\
  %&+(1+\|d\phi-dt\|_{s_0,2,\Sigma})(\| \rho \|^5_{H^{s_0}}+ \| \bv \|^5_{H^{s}}+\|\bw\|^5_{H^{s_0}}+\|h\|^5_{H^{s_0+1}})
%\\
  \lesssim & \epsilon^2_2.
  \end{split}
\end{equation}
By Lemma \ref{ce} and H\"older's inequality, we have
\begin{equation}\label{J6}
  \begin{split}
  \textstyle{\sum_{\sigma \in \{0,s_0-2 \}}}(|J_6|+|J_8|) \lesssim & \|d\partial \phi\|_{L^1_t L^\infty_x}(\| \rho \|^4_{H^{s_0}}+ \| \bv \|^4_{H^{s}}+\|\bw\|^4_{H^{s_0}}+\|h\|^4_{H^{s_0+1}})
  \\
   \lesssim  & \|d\phi-dt\|_{s_0,2,\Sigma}(\| \rho \|^4_{H^{s_0}}+ \| \bv \|^4_{H^{s}}+\|\bw\|^4_{H^{s_0}}+\|h\|^4_{H^{s_0+1}}) \\
  %&+(1+\|d\phi-dt\|_{s_0,2,\Sigma})(\| \rho \|^5_{H^{s_0}}+ \| \bv \|^5_{H^{s}}+\|\bw\|^5_{H^{s_0}}+\|h\|^5_{H^{s_0+1}})
%\\
  \lesssim & \epsilon^2_2.
  \end{split}
\end{equation}
Combining \eqref{J1}, \eqref{J3}, \eqref{J4}, \eqref{J5}, \eqref{J7}, and \eqref{J6}, we know
\begin{equation}\label{Jj}
  \mathrm{J}\lesssim  \epsilon^2_2.
\end{equation}
For $\bar{\mathrm{J}}$, we calculate
\begin{equation*}
  \begin{split}
  \bar{\mathrm{J}} \leq &\sum_{\sigma \in \{0,s_0-2 \}}\left| \int^2_{-2} \int_{\mathbb{R}^3}  \Lambda_{x'}^{\sigma} \left\{\mathbf{T}\widetilde{\{L(-\mathrm{e}^{-\rho} \epsilon^{ijk} \partial_k h \Delta v_j)\}}\right\} \cdot \Lambda_{x'}^{\sigma} \widetilde{L\underline{F}_i}dxd\tau    \right|
  \\
  \leq &\sum_{\sigma \in \{0,s_0-2 \}}\left| \int^2_{-2} \int_{\mathbb{R}^3}   [\mathbf{T},\Lambda_{x'}^{\sigma}]\left\{\widetilde{L(-\mathrm{e}^{-\rho} \epsilon^{ijk} \partial_k h \Delta v_j)}\right\} \cdot \Lambda_{x'}^{\sigma} \widetilde{L\underline{F}_i}dxd\tau    \right|
  \\
  & + \sum_{\sigma \in \{0,s_0-2 \}}\left| \int^2_{-2} \int_{\mathbb{R}^3}   \mathbf{T}\left\{\Lambda_{x'}^{\sigma}\widetilde{\{L(-\mathrm{e}^{-\rho} \epsilon^{ijk} \partial_k h \Delta v_j)\}}\right\} \cdot \Lambda_{x'}^{\sigma} \widetilde{L\underline{F}_i}dxd\tau    \right|
  \\
  =&\bar{\mathrm{J}}_1+\bar{\mathrm{J}}_2,
  \end{split}
\end{equation*}
where
\begin{equation*}
  \begin{split}
  \bar{\mathrm{J}}_1= &\sum_{\sigma \in \{0,s_0-2 \}}\left| \int^2_{-2} \int_{\mathbb{R}^3}   [\mathbf{T},\Lambda_{x'}^{\sigma}]\left\{\widetilde{L(-\mathrm{e}^{-\rho} \epsilon^{ijk} \partial_k h \Delta v_j)}\right\} \cdot \Lambda_{x'}^{\sigma} \widetilde{L\underline{F}_i}dxd\tau    \right|,
  \\
  \bar{\mathrm{J}}_2 =&  \sum_{\sigma \in \{0,s_0-2 \}}\left| \int^2_{-2} \int_{\mathbb{R}^3}   \mathbf{T}\left\{\Lambda_{x'}^{\sigma}\widetilde{\{L(-\mathrm{e}^{-\rho} \epsilon^{ijk} \partial_k h \Delta v_j)\}}\right\} \cdot \Lambda_{x'}^{\sigma} \widetilde{L\underline{F}_i}dxd\tau    \right|.
  \end{split}
\end{equation*}
By Lemma \ref{ce} and \eqref{401}-\eqref{403}, we have
\begin{equation}\label{J11}
  \begin{split}
  \bar{\mathrm{J}}_1 \lesssim & \sum_{\sigma \in \{0,s_0-2 \}} \int^{2}_{-2} \| \partial \bv \|_{L^2_t L^\infty_x} \|\mathrm{e}^{-\rho} \partial h \Delta \bv\|_{\dot{H}_x^\sigma} \|\underline{\bF}\|_{\dot{H}_x^\sigma}d\tau
  \\
  \lesssim & \| \partial \bv \|_{L^2_t L^\infty_x }\| \partial \bv \|_{L^2_t\dot{B}^{s_0-2}_{\infty,2}}(\|\bv\|^3_{H_x^s}+\| \rho \|^3_{H^{s}}+\| \bw \|^3_{H^{s_0}}+\| h \|^3_{H^{s_0+1}})
  \\
  &+ \| \partial \bv \|_{L^2_t L^\infty_x}\| \partial \bv \|_{L^2_t\dot{B}^{s_0-2}_{\infty,2}}(\|\bv\|^4_{H_x^s}+\| \rho \|^4_{H^{s}}+\| \bw \|^4_{H^{s_0}}+\| h \|^4_{H^{s_0+1}})
  \\
  \lesssim & \epsilon^2_2.
  \end{split}
\end{equation}
For $\bar{\mathrm{J}}_2$, we have
\begin{equation*}
  \begin{split}
  \bar{\mathrm{J}}_2 
  =& \bar{\mathrm{J}}_{21}+\bar{\mathrm{J}}_{22}+\bar{\mathrm{J}}_{23},
  \end{split}
\end{equation*}
where
\begin{equation*}
  \bar{\mathrm{J}}_{21}=\sum_{\sigma \in \{0,s_0-2 \}}  | \int^2_{-2} \int_{\mathbb{R}^3} \frac{d}{dt}\left\{\Lambda_{x'}^{\sigma}\widetilde{\{L(-\mathrm{e}^{-\rho} \epsilon^{ijk} \partial_k h \Delta v_j)\}} \cdot \Lambda_{x'}^{\sigma} \widetilde{L\underline{F}_i}\right\} dxd\tau  |,\ \ \ \ \ \
\end{equation*}
\begin{equation*}
  \bar{\mathrm{J}}_{22}=\sum_{\sigma \in \{0,s_0-2 \}} | \int^2_{-2} \int_{\mathbb{R}^3} (\bv \cdot \nabla) \left\{\Lambda_{x'}^{\sigma}\widetilde{\{L(-\mathrm{e}^{-\rho} \epsilon^{ijk} \partial_k h \Delta v_j)\}} \cdot \Lambda_{x'}^{\sigma} \widetilde{L\underline{F}_i}\right\} dxd\tau    |,
\end{equation*}
\begin{equation*}
  \bar{\mathrm{J}}_{23}=\sum_{\sigma \in \{0,s_0-2 \}}  | \int^2_{-2} \int_{\mathbb{R}^3}  \Lambda_{x'}^{\sigma}\widetilde{\{L(-\mathrm{e}^{-\rho} \epsilon^{ijk} \partial_k h \Delta v_j)\}} \cdot \mathbf{T}\{\Lambda_{x'}^{\sigma} \widetilde{L\underline{F}_i}\} dxd\tau  |. \ \ \ \ \ \ \ \ \ \
\end{equation*}
For $\bar{\mathrm{J}}_{21}$, we calculate that
\begin{equation}\label{J21}
\begin{split}
 \bar{\mathrm{J}}_{21}=&\sup_{t \in [-2,2]} (\|\bv\|^3_{H_x^s}+\| \rho \|^3_{H^{s}}+\| \bw \|^3_{H^{s_0}}+\| h \|^3_{H^{s_0+1}})(t)
 \\
 &+(\|\bv\|^4_{H_x^s}+\| \rho \|^4_{H^{s}}+\| \bw \|^4_{H^{s_0}}+\| h \|^4_{H^{s_0+1}})(t)
 \\
 \lesssim & \epsilon^2_2.
\end{split}
\end{equation}
By Lemma \ref{ce}, we can estimate $\bar{\mathrm{J}}_{22}$ by
\begin{equation}\label{J22}
\begin{split}
 \bar{\mathrm{J}}_{22}=&\| \partial \bv \|_{L^2_t L^\infty_x }\| \partial \bv \|_{L^2_t \dot{B}^{s_0-2}_{\infty,2} }(\|\bv\|^2_{H_x^s}+\| \rho \|^2_{H^{s}}+\| \bw \|^2_{H^{s_0}}
 \\
 &+\| h \|^2_{H^{s_0+1}} +\|\bv\|^3_{H_x^s}+\| \rho \|^3_{H^{s}}+\| \bw \|^3_{H^{s_0}}+\| h \|^3_{H^{s_0+1}})
 \\
 \lesssim & \epsilon^2_2.
\end{split}
\end{equation}
We then bound $\bar{\mathrm{J}}_{23}$ by
\begin{equation}\label{J23}
\begin{split}
  \bar{\mathrm{J}}_{23}=& \sum_{\sigma \in \{0,s_0-2 \}}  | \int^2_{-2} \int_{\mathbb{R}^3}  \Lambda_{x'}^{\sigma}\widetilde{\{L(-\mathrm{e}^{-\rho} \epsilon^{ijk} \partial_k h \Delta v_j)\}} \cdot  [\mathbf{T} , \Lambda_{x'}^{\sigma}] (\widetilde{L\underline{F}_i}) dxd\tau  |
  \\
  & + \sum_{\sigma \in \{0,s_0-2 \}}  | \int^2_{-2} \int_{\mathbb{R}^3}  \Lambda_{x'}^{\sigma}\widetilde{\{L(-\mathrm{e}^{-\rho} \epsilon^{ijk} \partial_k h \Delta v_j)\}} \cdot \Lambda_{x'}^{\sigma} \{ \mathbf{T} (\widetilde{L\underline{F}_i}) \} dxd\tau  
  \\
  \leq & \bar{\mathrm{J}}_{231}+\bar{\mathrm{J}}_{232}+\bar{\mathrm{J}}_{233}+\bar{\mathrm{J}}_{234}.
\end{split}
\end{equation}
Here we set
\begin{equation*}
  \begin{split}
 \bar{\mathrm{J}}_{231} = & \sum_{\sigma \in \{0,s_0-2 \}}  | \int^2_{-2} \int_{\mathbb{R}^3}  \Lambda_{x'}^{\sigma}\widetilde{\{L(-\mathrm{e}^{-\rho} \epsilon^{ijk} \partial_k h \Delta v_j)\}} \cdot  [\mathbf{T} , \Lambda_{x'}^{\sigma}] (\widetilde{L\underline{F}_i}) dxd\tau  |,
  \\
  \bar{\mathrm{J}}_{232}=&  \sum_{\sigma \in \{0,s_0-2 \}}  | \int^2_{-2} \int_{\mathbb{R}^3}  \Lambda_{x'}^{\sigma}\widetilde{\{L(-\mathrm{e}^{-\rho} \epsilon^{ijk} \partial_k h \Delta v_j)\}} \cdot \Lambda_{x'}^{\sigma} \{ \partial_i \widetilde{\mathcal{G}}+ K_i  \} dxd\tau  |,
  \\
  \bar{\mathrm{J}}_{233}=&  \sum_{\sigma \in \{0,s_0-2 \}}  | \int^2_{-2} \int_{\mathbb{R}^3}  \Lambda_{x'}^{\sigma}\widetilde{\{L(-\mathrm{e}^{-\rho} \epsilon^{ijk} \partial_k h \Delta v_j)\}} \cdot [\Lambda_{x'}^{\sigma}, \mathbf{T}] \widetilde{\{ L(-\mathrm{e}^{-\rho} \epsilon^{ijk} \partial_k h \Delta v_j) \}} dxd\tau  |,
  \\
  \bar{\mathrm{J}}_{234}= & \sum_{\sigma \in \{0,s_0-2 \}}  | \int^2_{-2} \int_{\mathbb{R}^3}  \Lambda_{x'}^{\sigma}\widetilde{\{L(-\mathrm{e}^{-2\rho} \epsilon^{ijk} \partial_k h \Delta v_j)\}} \cdot \mathbf{T} \widetilde{\Lambda_{x'}^{\sigma} \{ L(-\mathrm{e}^{-2\rho} \epsilon^{ijk} \partial_k h \Delta v_j) \}} dxd\tau  |.
  \end{split}
\end{equation*}
By Lemma \ref{ce} and H\"older's inequality, we get
\begin{equation}\label{J231}
\begin{split}
  \bar{\mathrm{J}}_{231}+\bar{\mathrm{J}}_{233} \lesssim & \| \partial \bv \|^2_{L^2_t \dot{B}^{s_0-2}_{\infty,2} }(\|\bv\|^3_{H_x^s}+\| \rho \|^3_{H^{s}}+\| \bw \|^3_{H^{s_0}}
 \\
 &+\| h \|^3_{H^{s_0+1}} +\|\bv\|^4_{H_x^s}+\| \rho \|^4_{H^{s}}+\| \bw \|^4_{H^{s_0}}+\| h \|^4_{H^{s_0+1}}).
\end{split}
\end{equation}
\begin{equation}\label{J230}
\begin{split}
  \bar{\mathrm{J}}_{232}=&  \sum_{\sigma \in \{0,s_0-2 \}}  | \int^2_{-2} \int_{\mathbb{R}^3}  \Lambda_{x'}^{\sigma}\widetilde{\{L(-\mathrm{e}^{-\rho} \epsilon^{ijk} \partial_k h \Delta v_j)\}} \cdot \Lambda_{x'}^{\sigma} \{ \partial_i \widetilde{\mathcal{G}}+ \widetilde{K}_i  \} dxd\tau  |
  \\
  =& \sum_{\sigma \in \{0,s_0-2 \}}  | \int^2_{-2} \int_{\mathbb{R}^3}  \Lambda_{x'}^{\sigma}\widetilde{\{L(-\mathrm{e}^{-\rho} \epsilon^{ijk} \partial_k h \Delta v_j)\}} \cdot \Lambda_{x'}^{\sigma}  \widetilde{K}_i   dxd\tau  |
  \\
  & + \sum_{\sigma \in \{0,s_0-2 \}}  | \int^2_{-2} \int_{\mathbb{R}^3}  \Lambda_{x'}^{\sigma}\partial_i \widetilde{\{L(-\mathrm{e}^{-\rho} \epsilon^{ijk} \partial_k h \Delta v_j)\}} \cdot \Lambda_{x'}^{\sigma}   \widetilde{\mathcal{G}}   dxd\tau  | .
\end{split}
\end{equation}
By using $-\mathrm{e}^{-\rho} \epsilon^{ijk}\partial_i( \partial_k h \Delta v_j)=\mathrm{e}^{-\rho}\partial_k h \Delta ( \mathrm{e}^\rho \bw^k)$,  \eqref{J230} yields
\begin{equation}\label{J232}
\begin{split}
  \bar{\mathrm{J}}_{232}
  =& \sum_{\sigma \in \{0,s_0-2 \}}  | \int^2_{-2} \int_{\mathbb{R}^3}  \Lambda_{x'}^{\sigma}\widetilde{\{L(-\mathrm{e}^{-\rho} \epsilon^{ijk} \partial_k h \Delta v_j)\}} \cdot \Lambda_{x'}^{\sigma}  \widetilde{K}_i   dxd\tau  |
  \\
  & + \sum_{\sigma \in \{0,s_0-2 \}}  | \int^2_{-2} \int_{\mathbb{R}^3}  \Lambda_{x'}^{\sigma} \widetilde{\{L(2\mathrm{e}^{-\rho}\partial_i \rho \epsilon^{ijk} \partial_k h \Delta v_j)\}} \cdot \Lambda_{x'}^{\sigma}   \widetilde{G}   dxd\tau  |
  \\
  & + \sum_{\sigma \in \{0,s_0-2 \}}  | \int^2_{-2} \int_{\mathbb{R}^3}  \Lambda_{x'}^{\sigma} \widetilde{\{L(-\mathrm{e}^{-\rho}  \partial_k h \Delta (\mathrm{e}^{\rho} \bw^k)\}} \cdot \Lambda_{x'}^{\sigma}   \widetilde{G}   dxd\tau  |
  \\
  \lesssim &\sum_{\sigma \in \{0,s_0-2 \}}  \| \partial h \Delta \bv \|_{L^2_t \dot{H}^\sigma_{x}}\|\mathbf{K}\|_{L^2_t \dot{H}^\sigma_{x}}+\| \partial \rho \partial h \Delta \bv \|_{L^1_t \dot{H}^\sigma_{x}}\|G\|_{L^\infty_t \dot{H}^\sigma_{x}}
  \\
  &\quad + (\| \partial h \Delta \bw \|_{L^1_t \dot{H}^\sigma_{x}}+ \| \partial h \partial^2 \rho \bw \|_{L^1_t \dot{H}^\sigma_{x}} )\|G\|_{L^\infty_t \dot{H}^\sigma_{x}}
  \\
  \lesssim & (\| d \bv \|^2_{L^2_t \dot{B}^{s_0-2}_{\infty,2}}+\| d \rho \|^2_{L^2_t \dot{B}^{s_0-2}_{\infty,2}}+\| d h \|^2_{L^2_t \dot{B}^{s_0-2}_{\infty,2}})\{ \|\bv\|^3_{H_x^s}+\| \rho \|^3_{H^{s}}
  \\
  & +\| \bw \|^3_{H^{s_0}}+\| h \|^3_{H^{s_0+1}}+\|\bv\|^4_{H_x^s}+\| \rho \|^4_{H^{s}}+\| \bw \|^4_{H^{s_0}}+\| h \|^4_{H^{s_0+1}} \}.
\end{split}
\end{equation}
Integrating by parts on $\bar{\mathrm{J}}_{234}$, we obtain
\begin{equation}\label{J234}
\begin{split}
  \bar{\mathrm{J}}_{234}= & \frac12 \sum_{\sigma \in \{0,s_0-2 \}}  | \int^2_{-2} \int_{\mathbb{R}^3}  \mathbf{T}| \Lambda_{x'}^{\sigma}\widetilde{\{L(-\mathrm{e}^{-\rho} \epsilon^{ijk} \partial_k h \Delta v_j)\}} |^2 dxd\tau
  \\
  \leq & \sum_{\sigma \in \{0,s_0-2 \}} \|\mathrm{e}^{-\rho} \partial h \Delta \bv\|^2_{\dot{H}^{\sigma}_x}(1+ \|d\bv\|_{L^1_t L^\infty_x})
  \\
  \lesssim & (\| h \|^2_{H^{s_0+1}}+\|\bv\|^2_{H_x^s})(1+ \|d\bv\|_{L^1_t L^\infty_x})(1+ \|\rho\|^2_{H_x^s}).
\end{split}
\end{equation}
By using \eqref{J231}, \eqref{J232}, \eqref{J234}, and \eqref{401}-\eqref{403}, we prove that
\begin{equation}\label{J2}
  \bar{\mathrm{J}}_{2}
  \lesssim \epsilon^2_2.
\end{equation}
The estimate \eqref{J2} adding to \eqref{J11}, we have
\begin{equation}\label{bJ}
  \bar{\mathrm{J}} \lesssim \epsilon^2_2.
\end{equation}
Substituting \eqref{Ij2}, \eqref{Jj}, \eqref{bJ} to \eqref{L0} and using \eqref{401}-\eqref{403}, \eqref{L0} yields
\begin{equation}\label{L01}
\begin{split}
  \vert\kern-0.25ex\vert\kern-0.25ex\vert L\underline{\bF}\vert\kern-0.25ex\vert\kern-0.25ex\vert^2_{s_0-2,2,\Sigma} \lesssim  \epsilon^2_2.
  \end{split}
\end{equation}
For
\begin{equation*}
  L\underline{F}^i=L(\mathrm{curl} \mathrm{curl} \bw^i) +L(\epsilon^{ijk} \partial_j \rho \cdot \mathrm{curl}\bw_k) - 2L( \partial^a {\rho} \partial^i w_a)+2L( \mathrm{e}^{-\rho} \epsilon^{ijk} \partial_j v^m  \partial_m\partial_k h).
\end{equation*}
we then get
\begin{equation}\label{Rr}
\begin{split}
  \vert\kern-0.25ex\vert\kern-0.25ex\vert L\underline{F}^i\vert\kern-0.25ex\vert\kern-0.25ex\vert_{s_0-2,2,\Sigma} \geq & \vert\kern-0.25ex\vert\kern-0.25ex\vert L\mathrm{curl}\mathrm{curl}\bw^i\vert\kern-0.25ex\vert\kern-0.25ex\vert_{s_0-2,2,\Sigma}- 2 \vert\kern-0.25ex\vert\kern-0.25ex\vert L(\partial^a {\rho} \partial^i w_a) \vert\kern-0.25ex\vert\kern-0.25ex\vert_{s_0-2,2,\Sigma}
  \\
  & -2\vert\kern-0.25ex\vert\kern-0.25ex\vert L( \mathrm{e}^{-\rho} \epsilon^{ijk} \partial_j v^m  \partial_m\partial_k h)\vert\kern-0.25ex\vert\kern-0.25ex\vert_{s_0-2,2,\Sigma}- \vert\kern-0.25ex\vert\kern-0.25ex\vert L(\epsilon^{ijk} \partial_j \rho \cdot \mathrm{curl}\bw_k) \vert\kern-0.25ex\vert\kern-0.25ex\vert.
\end{split}
\end{equation}
By trace theorem and elliptic estimate, we have
\begin{equation}\label{RR}
\begin{split}
 \vert\kern-0.25ex\vert\kern-0.25ex\vert L\left( \partial_a \rho \partial^i w^a \right)\vert\kern-0.25ex\vert\kern-0.25ex\vert_{s_0-2,2,\Sigma}
\leq &  C\| L\left( \partial_a \rho \partial^i w^a \right)\|_{H^{s_0-\frac{3}{2}+(s-s_0)}_x}
\\
\leq &  C\| \partial_a \rho \partial^i w^a \|_{H^{s_0-\frac{3}{2}+(s-s_0)}_x}
\\
\leq &   C\|\rho \|_{H_x^{s}}\| \bw \|_{H_x^{s_0}} \lesssim \epsilon^2_2.
\end{split}
\end{equation}
The same technique tells us
\begin{equation}\label{RR1}
\begin{split}
& \vert\kern-0.25ex\vert\kern-0.25ex\vert L(\epsilon^{ijk} \partial_j \rho \cdot \mathrm{curl}\bw_k) \vert\kern-0.25ex\vert\kern-0.25ex\vert+ \vert\kern-0.25ex\vert\kern-0.25ex\vert  L( \mathrm{e}^{-\rho} \epsilon^{ijk} \partial_j v^m  \partial_m\partial_k h) \vert\kern-0.25ex\vert\kern-0.25ex\vert_{s_0-2,2,\Sigma}
\\
\leq &   C( \|\bv \|_{H_x^{s}}\| h \|_{H_x^{s_0+1}} +\|\rho \|_{H_x^{s}}\| \bw \|_{H_x^{s_0}}) \lesssim \epsilon^2_2.
\end{split}
\end{equation}
Substituting \eqref{Rr}, \eqref{RR}, and \eqref{RR1} to \eqref{L01}, and using \eqref{OM1}, we therefore get
\begin{equation}\label{PO}
  \vert\kern-0.25ex\vert\kern-0.25ex\vert \partial \mathrm{curl} \bw \vert\kern-0.25ex\vert\kern-0.25ex\vert_{s_0-2,2,\Sigma}= \vert\kern-0.25ex\vert\kern-0.25ex\vert L ( \mathrm{curl}  \mathrm{curl} \bw )\vert\kern-0.25ex\vert\kern-0.25ex\vert_{s_0-2,2,\Sigma}\lesssim \ \epsilon^2_2 \lesssim \ \epsilon_2.
\end{equation}
Due to
\begin{equation*}
  \partial_{x'} \mathrm{curl}  \mathrm{curl} \bw = \partial \mathrm{curl} \bw \cdot (0, \partial_{x'}\phi)^{\mathrm{T}},
\end{equation*}
by H\"older inequality and product estimates in Lemma \ref{cj}, we can deduce the following estimate
\begin{equation}\label{OH}
\begin{split}
  \|\partial_{x'} \mathrm{curl} \bw \|_{L^2_t H^{s_0-2}_{x'}(\Sigma)} &\lesssim \| \partial \mathrm{curl} \bw  \cdot (0, \partial_{x'}\phi)^{\mathrm{T}}\|_{L^2_t H^{s_0-2}_{x'}(\Sigma)}
  \\
  &\lesssim \| \partial \mathrm{curl} \bw  \|_{L^2_t H^{s_0-2}_{x'}(\Sigma)} \| (0, \partial_{x'}\phi)^{\mathrm{T}}\|_{L^\infty_t H^{s_0-1}_{x'}(\Sigma)}
  \\
  &\lesssim \| \partial \mathrm{curl} \bw \|_{L^2_t H^{s_0-2}_{x'}(\Sigma)} \| \partial_{x'}\phi \|_{L^\infty_t H^{s_0-1}_{x'}(\Sigma)}
  \\
  &\lesssim \| \partial \mathrm{curl} \bw  \|_{L^2_t H^{s_0-2}_{x'}(\Sigma)} \| \partial_{x'}\phi \|_{s_0, 2, \Sigma} \lesssim \epsilon^2_2.
\end{split}
\end{equation}
Above, we use the trace theorem
$\| \partial_{x'}\phi \|_{L^\infty_t H^{s_0-1}_{x'}(\Sigma)} \lesssim  \vert\kern-0.25ex\vert\kern-0.25ex\vert \partial_{x'}\phi \vert\kern-0.25ex\vert\kern-0.25ex\vert_{s_0, 2, \Sigma}\lesssim  \vert\kern-0.25ex\vert\kern-0.25ex\vert d\phi-dt \vert\kern-0.25ex\vert\kern-0.25ex\vert_{s_0, 2, \Sigma}$.
Through \eqref{OM0} and \eqref{OH}, we can see
\begin{equation}\label{O1}
  \| \mathrm{curl} \bw  \|_{L^2_t H^{s_0-1}_{x'}(\Sigma)} \lesssim \epsilon_2 .
\end{equation}
\textbf{Step 2: $\| \partial_t \mathrm{curl} \bw  \|_{L^2_t H^{s_0-2}_{x'}(\Sigma)}$}. 
By using \eqref{W1} and Sobolev imbeddings, we also have
\begin{equation}\label{OR1}
\begin{split}
  \| \partial_t \mathrm{curl} \bw  \|_{L^2_t H^{s_0-2}_{x'}(\Sigma)}  & \leq    \|  (\bv \cdot \nabla) \mathrm{curl} \bw  \|_{L^2_t H^{s_0-2}_{x'}(\Sigma)}+ \|\partial \bv \cdot \partial \mathrm{curl} \bw  \|_{L^2_t H^{s_0-2}_{x'}(\Sigma)}
  \\
  & \leq \| (\bv \cdot \nabla) \mathrm{curl} \bw  \|_{L^2_t H^{s_0-2}_{x'}(\Sigma)} + \| \partial \bv \cdot \partial \bw \|_{L^2_t H^{s_0-2}_{x'}(\Sigma)}
  \\
  & \leq \| \bv \|_{L^\infty_t H^{s_0-\frac12}_{x'}(\Sigma)} \| \partial \mathrm{curl} \bw  \|_{L^2_t H^{s_0-2}_{x'}(\Sigma)} + \| \partial \bv \cdot \partial \bw \|_{L^\infty_t H^{s-\frac{3}{2}}_{x}},
  \\
  & \leq \| \bv \|_{L^\infty_t H^{s}_{x}} \| \partial \mathrm{curl} \bw  \|_{L^2_t H^{s_0-2}_{x'}(\Sigma)} + \| \partial \bv \|_{L^\infty_t H^{s-1}_x}\| \partial \bw \|_{L^\infty_t H^{1}_{x}}.
\end{split}
\end{equation}
Due to \eqref{O1}, \eqref{OR1}, and \eqref{401}-\eqref{403}, we have proved \eqref{te201}.
\end{proof}

\begin{Lemma}\label{te21}
Suppose that $(\bv, \rho, h, \bw) \in \mathcal{H}$. %Assuming $|||W|||_{s_0,2,\Sigma} \lesssim \epsilon_1$,
Then
\begin{equation}\label{te211}
\begin{split}
 \vert\kern-0.25ex\vert\kern-0.25ex\vert  \Delta h  \vert\kern-0.25ex\vert\kern-0.25ex\vert_{s_0-1,2,\Sigma} \lesssim  \epsilon_2.
   \end{split}
\end{equation}
\end{Lemma}
\begin{proof}
Let us recall \eqref{Hh2}. By changing of coordinates $x_3 \rightarrow x_3-\phi(t,x')$, the equation \eqref{Hh2} becomes to
\begin{equation}\label{th0}
\begin{split}
  \partial_t \left\{ \widetilde{ \partial^k( \mathrm{e}^{\rho} \partial_i H^i } ) \right\}+ (\tilde{\bv}\cdot \nabla) \left\{ \widetilde{ \partial^k( \mathrm{e}^{\rho} \partial_i H^i } ) \right\}
   =&\widetilde{ -2\partial^k(\partial_i v^j) \partial_j  \partial^i h}+\widetilde{Y^k}
   \\
   &-(\partial_t \phi+ \tilde{v}^j \partial_j \phi) \partial_3 \left\{ \widetilde{ \partial^k( \mathrm{e}^{\rho} \partial_i H^i } ) \right\}.
\end{split}
\end{equation}
Here we consider $\partial^k ( \mathrm{e}^{\rho} \partial_i H^i )$ as a whole. Taking derivatives $\Lambda^{\sigma}_{x'}$($\sigma \in \{0,s_0-2\}$) on \eqref{th0}, then multiplying $\Lambda^{\sigma}_{x'}\left\{ \widetilde{ \partial^k( \mathrm{e}^{\rho} \partial_i H^i } ) \right\}$ and integrating it on $[-2,2]\times \mathbb{R}^3$, we can obtain
\begin{equation}\label{th1}
\begin{split}
 & \vert\kern-0.25ex\vert\kern-0.25ex\vert  \partial^k ( \mathrm{e}^{\rho} \partial_i H^i ) \vert\kern-0.25ex\vert\kern-0.25ex\vert^2_{s_0-2,2,\Sigma} \leq \mathrm{Q_1}+\mathrm{Q_2}+\mathrm{Q_3}+\mathrm{Q_4}+\mathrm{Q_5}.
\end{split}
\end{equation}
where we denote
\begin{equation*}
\begin{split}
 \mathrm{Q_1}=&
\sum_{ \sigma \in \{0,s_0-2\} }  \int^2_{-2} \int_{\mathbb{R}^3} |\mathrm{div}\tilde{\bv}|\cdot |\Lambda^{\sigma}_{x'} \widetilde{ \partial^k( \mathrm{e}^{\rho} \partial_i H^i } |^2dxd\tau ,
\\
 \mathrm{Q_2}=&
\sum_{ \sigma \in \{0,s_0-2\} }  | \int^2_{-2} \int_{\mathbb{R}^3} \widetilde{Y^k} \cdot \Lambda^{\sigma}_{x'} \widetilde{ \partial^k( \mathrm{e}^{\rho} \partial_i H^i } ) dxd\tau | ,
\\
 \mathrm{Q_3}=& \sum_{ \sigma \in \{0,s_0-2\} }| \int^2_{-2} \int_{\mathbb{R}^3}  [\Lambda^{\sigma}_{x'}, \tilde{\bv}\cdot \nabla] \{ \widetilde{ \partial^k( \mathrm{e}^{\rho} \partial_i H^i } ) \} \cdot \Lambda^{\sigma}_{x'} \widetilde{ \partial^k( \mathrm{e}^{\rho} \partial_i H^i } ) dxd\tau |,
\\
\mathrm{Q_4}=& \sum_{ \sigma \in \{0,s_0-2\} }| \int^2_{-2} \int_{\mathbb{R}^3} (\partial_t \phi+ \tilde{v}^j \partial_j \phi) \partial_3 \{ \widetilde{ \partial^k( \mathrm{e}^{\rho} \partial_i H^i } ) \} \cdot \Lambda^{\sigma}_{x'} \widetilde{ \partial^k( \mathrm{e}^{\rho} \partial_i H^i } ) dxd\tau |,
\\
\mathrm{Q_5}=& \sum_{ \sigma \in \{0,s_0-2\} }| \int^2_{-2} \int_{\mathbb{R}^3} \widetilde{ -2\partial^k(\partial_i v^j) \partial_j  \partial^i h} \cdot \Lambda^{\sigma}_{x'} \widetilde{ \partial^k( \mathrm{e}^{\rho} \partial_i H^i }) dxd\tau |.
\end{split}
\end{equation*}
By H\"older's inequality, we can bound $\mathrm{Q_1}$ by
\begin{equation}\label{eQ1}
 \mathrm{Q_1} \lesssim \|d\bv\|_{L^2_t L^\infty_x} (\| h \|^2_{H^{s_0+1}}+\|\rho\|^2_{H_x^s}\| h \|^2_{H^{s_0+1}}) \lesssim \epsilon^2_2.
\end{equation}
By H\"older's inequality and Lemma \ref{lpe}, we can bound $\mathrm{Q_2}$ by
\begin{equation}\label{eQ2}
\begin{split}
 \mathrm{Q_2} \lesssim  & \sum_{ \sigma \in \{0,s_0-2\} } \| \bY \|_{L^1_t \dot{H}^{\sigma}_x} \|\mathrm{e}^{\rho} \partial_i H^i\|_{\dot{H}^{1+\sigma}_x}
 \\
 \lesssim  & (\|d\bv\|_{L^2_t L^\infty_x} + \|d\bv\|_{L^2_t \dot{B}^{s_0-2}_{\infty,2}}) (\| \bv \|^2_{H^{s}}+\|\rho\|^2_{H_x^s}+\|\bw\|^2_{H_x^{s_0}}+\| h \|^2_{H^{s_0+1}})
 \\
 & + (\|d\bv\|_{L^2_t L^\infty_x} + \|d\bv\|_{L^2_t \dot{B}^{s_0-2}_{\infty,2}}) (\| \bv \|^3_{H^{s}}+\|\rho\|^3_{H_x^s}+\|\bw\|^3_{H_x^{s_0}}+\| h \|^3_{H^{s_0+1}})
 \\
 & + (\|d\bv\|_{L^2_t L^\infty_x} + \|d\bv\|_{L^2_t \dot{B}^{s_0-2}_{\infty,2}})  (\| \bv \|^4_{H^{s}}+\|\rho\|^4_{H_x^s}+\|\bw\|^4_{H_x^{s_0}}+\| h \|^4_{H^{s_0+1}})
 \\
 &+ (\|d\bv\|_{L^2_t L^\infty_x} + \|d\bv\|_{L^2_t \dot{B}^{s_0-2}_{\infty,2}})  (\| \bv \|^5_{H^{s}}+\|\rho\|^5_{H_x^s}+\|\bw\|^5_{H_x^{s_0}}+\| h \|^5_{H^{s_0+1}})
 \\
\lesssim & \epsilon^2_2,
\end{split}
\end{equation}
where $\bY=(Y^1,Y^2,Y^3)$. By Lemma \ref{ce}, we have
\begin{equation}\label{eQ3}
\begin{split}
\mathrm{Q_3} \lesssim &\|d\bv\|_{L^2_t L^\infty_x } (\| h \|^2_{H^{s_0+1}}+\|\rho\|^4_{H_x^s}+\| h \|^4_{H^{s_0+1}})
 \\
 \lesssim & \epsilon^2_2.
\end{split}
\end{equation}
Integrating $\mathrm{Q_4}$ by parts, we have
\begin{equation}\label{eQ4}
 \mathrm{Q_4} \lesssim \|d\bv\|_{L^2_t L^\infty_x}\|d\phi\|_{s_0,2,\Sigma} (\| h \|^2_{H^{s_0+1}}+\|\rho\|^2_{H_x^s}\| h \|^2_{H^{s_0+1}}) \lesssim \epsilon^2_2,
\end{equation}
where we use the fact that $\phi$ is independent with $x_3$. For $\mathrm{Q_5}$, we use the similar idea on estimating $\tilde{I}_4$, i.e. \eqref{I08}-\eqref{I411}, we then conclude that
\begin{equation}\label{eQ5}
\begin{split}
 \mathrm{Q_5} \leq & C\| \rho \|_{{H}^{2}_{x}} \| h \|_{{H}^{s_0+\frac12}_{x}}\| h \|_{{H}^{s_0+1}_{x}}+C\| h \|_{{H}^{\frac52+}_{x}}\| \rho \|_{{H}^{2}_{x}}\| \rho \|_{{H}^{s_0}_{x}} \| h \|_{{H}^{s_0+\frac12}}
\\
& + C\int^t_0 \|\partial \bv \|_{L^\infty_x} \| h \|^2_{H^{\frac52+}}\| \rho \|^2_{H^{s_0}} d\tau
 + C\int^t_0   ( \| \bv \|^2_{{H}^{2}_{x}} +  \| \rho \|^2_{{H}^{2}_{x}}) \| h \|^2_{{H}^{s_0+1}_{x}}  d\tau
\\
&+ C\int^t_0 \|\partial \bv \|_{\dot{B}^{s_0-2}_{\infty,2}}( \|h\|^2_{H^{s_0+1}_x}+ \|\bw\|^2_{H^{s_0}_x}+ \|\rho\|^2_{H^{s}_x}+ \|\bv\|^2_{H^{s}_x})d\tau
  \\
  &+C\int^t_0 \|(\partial \bv, \partial \rho, \partial h) \|_{\dot{B}^{s_0-2}_{\infty,2}}\|( h,\rho,\bv)\|_{H^{s_0}_x}( \|h\|^2_{H^{s_0+1}_x}+ \|\bw\|^2_{H^{s_0}_x}+ \| (\rho,\bv) \|^2_{H^{s}_x})d\tau
  \\
  &+C\int^t_0 \|(\partial \bv, \partial \rho, \partial h) \|_{\dot{B}^{s_0-2}_{\infty,2}} \|( h,\rho,\bv) \|^2_{H^{s_0}_x}( \|h\|^2_{H^{s_0+1}_x}+ \|\bw\|^2_{H^{s_0}_x}+ \|(\rho, \bv)\|^2_{H^{s}_x})d\tau
  \\
  & +C\int^t_0 \|(\partial \bv, \partial \rho, \partial h) \|_{\dot{B}^{s_0-2}_{\infty,2}} \|( h,\rho, \bv)\|^3_{H^{s_0}_x}( \|h\|^2_{H^{s_0+1}_x}+ \|\bw\|^2_{H^{s_0}_x}+ \|(\rho,\bv)\|^2_{H^{s}_x})d\tau.
 \\
  \lesssim & \epsilon^2_2.
\end{split}
\end{equation}
Here we use \eqref{Mt} and the assumptions \eqref{401}-\eqref{403}. Gathering \eqref{eQ1}, \eqref{eQ2}, $\cdots$, \eqref{eQ5}, yields
\begin{equation*}
  \vert\kern-0.25ex\vert\kern-0.25ex\vert  \partial^k ( \mathrm{e}^{\rho} \partial_i H^i ) \vert\kern-0.25ex\vert\kern-0.25ex\vert_{s_0-2,2,\Sigma}  \lesssim \epsilon_2, \quad k=1,2,3.
\end{equation*}
Due to
\begin{equation*}
  \partial_{x'} ( \mathrm{e}^{\rho} \partial_i H^i ) = \partial ( \mathrm{e}^{\rho} \partial_i H^i ) \cdot (0, \partial_{x'}\phi)^{\mathrm{T}},
\end{equation*}
by H\"older's inequality and product estimates in Lemma \ref{cj}, we can deduce the following estimate
\begin{equation}\label{th02}
  \vert\kern-0.25ex\vert\kern-0.25ex\vert \partial_{x'} ( \mathrm{e}^{\rho} \partial_i H^i ) \vert\kern-0.25ex\vert\kern-0.25ex\vert_{s_0-2,2,\Sigma} \lesssim \epsilon_2.
\end{equation}
Utilizing \eqref{Hh1} and product estimates, we have
\begin{equation}\label{th03}
\begin{split}
  \vert\kern-0.25ex\vert\kern-0.25ex\vert \partial_{t} ( \mathrm{e}^{\rho} \partial_i H^i ) \vert\kern-0.25ex\vert\kern-0.25ex\vert_{s_0-2,2,\Sigma} =&
  \vert\kern-0.25ex\vert\kern-0.25ex\vert \bv \cdot \nabla ( \mathrm{e}^{\rho} \partial_i H^i ) \vert\kern-0.25ex\vert\kern-0.25ex\vert_{s_0-2,2,\Sigma}
  + \vert\kern-0.25ex\vert\kern-0.25ex\vert  \partial \bv \partial^2 h \vert\kern-0.25ex\vert\kern-0.25ex\vert_{s_0-2,2,\Sigma}
  + \vert\kern-0.25ex\vert\kern-0.25ex\vert  \bw \partial \rho \partial h \vert\kern-0.25ex\vert\kern-0.25ex\vert_{s_0-2,2,\Sigma}
  \\
  \lesssim & \vert\kern-0.25ex\vert\kern-0.25ex\vert \bv \vert\kern-0.25ex\vert\kern-0.25ex\vert_{s_0,2,\Sigma} \vert\kern-0.25ex\vert\kern-0.25ex\vert \nabla ( \mathrm{e}^{\rho} \partial_i H^i ) \vert\kern-0.25ex\vert\kern-0.25ex\vert_{s_0-2,2,\Sigma}
  + \vert\kern-0.25ex\vert\kern-0.25ex\vert   \bv \vert\kern-0.25ex\vert\kern-0.25ex\vert_{s_0,2,\Sigma} \vert\kern-0.25ex\vert\kern-0.25ex\vert   h \vert\kern-0.25ex\vert\kern-0.25ex\vert_{s_0+1,2,\Sigma}
  \\
  & + \vert\kern-0.25ex\vert\kern-0.25ex\vert  \bw  \vert\kern-0.25ex\vert\kern-0.25ex\vert_{s_0,2,\Sigma} \vert\kern-0.25ex\vert\kern-0.25ex\vert  \rho \vert\kern-0.25ex\vert\kern-0.25ex\vert_{s_0,2,\Sigma} \vert\kern-0.25ex\vert\kern-0.25ex\vert   h \vert\kern-0.25ex\vert\kern-0.25ex\vert_{s_0+1,2,\Sigma}.
\end{split}
\end{equation}
Using \eqref{401}-\eqref{403},  \eqref{th03} yields
\begin{equation}\label{th04}
  \frac12 \vert\kern-0.25ex\vert\kern-0.25ex\vert \partial_{t} ( \mathrm{e}^{\rho} \partial_i H^i ) \vert\kern-0.25ex\vert\kern-0.25ex\vert_{s_0-2,2,\Sigma} \lesssim \epsilon_2.
\end{equation}
Combining \eqref{th02} and \eqref{th04}, we have
\begin{equation}\label{th05}
  \vert\kern-0.25ex\vert\kern-0.25ex\vert  \mathrm{e}^{\rho} \partial_i H^i  \vert\kern-0.25ex\vert\kern-0.25ex\vert_{s_0-1,2,\Sigma} \lesssim \epsilon_2.
\end{equation}
By calculating $\mathrm{e}^{\rho} \partial_i H^i=\Delta h - \partial_i \rho \partial^i h$, we can infer that
\begin{equation*}
\begin{split}
\vert\kern-0.25ex\vert\kern-0.25ex\vert  \Delta h \vert\kern-0.25ex\vert\kern-0.25ex\vert_{s_0-1,2,\Sigma}\leq & \vert\kern-0.25ex\vert\kern-0.25ex\vert  \mathrm{e}^{\rho} \partial_i H^i  \vert\kern-0.25ex\vert\kern-0.25ex\vert_{s_0-1,2,\Sigma}+\vert\kern-0.25ex\vert\kern-0.25ex\vert  \partial_i \rho \partial^i h  \vert\kern-0.25ex\vert\kern-0.25ex\vert_{s_0-1,2,\Sigma}
\\
\leq & \vert\kern-0.25ex\vert\kern-0.25ex\vert  \mathrm{e}^{\rho} \partial_i H^i  \vert\kern-0.25ex\vert\kern-0.25ex\vert_{s_0-1,2,\Sigma}+\vert\kern-0.25ex\vert\kern-0.25ex\vert  \partial_i \rho  \vert\kern-0.25ex\vert\kern-0.25ex\vert_{s_0-1,2,\Sigma}\vert\kern-0.25ex\vert\kern-0.25ex\vert   \partial^i h  \vert\kern-0.25ex\vert\kern-0.25ex\vert_{s_0-1,2,\Sigma}
\\
\lesssim & \epsilon_2.
\end{split}
\end{equation*}
Here we use that $H_{x'}^{s_0-1}(\Sigma)$ is a algebra for $s_0>2$.
\end{proof}

\begin{Lemma}\label{fre}
{Let $\bU=(p(\rho), \bv, h)^{\mathrm{T}}$ satisfy the assumption in Lemma \ref{te1}. Then}
\begin{equation}\label{508}
  \vert\kern-0.25ex\vert\kern-0.25ex\vert  2^j(\bU-S_j \bU), d S_j \bU, 2^{-j} d \partial_x S_j \bU \vert\kern-0.25ex\vert\kern-0.25ex\vert_{s_0-1,2,\Sigma} \lesssim \epsilon_2. %\|U\|_{L^\infty_t H^{s_0}_x} +  \|dU\|_{L^2_t L_x^\infty}.
\end{equation}
\end{Lemma}
\begin{proof}
Let $\Delta_0$ be a standard multiplier of order $0$ on $\mathbb{R}^3$, such that $\Delta_0$ is additionally bounded on $L^\infty(\mathbb{R}^3)$. Clearly,
\begin{equation*}
  A^0(\bU)(\Delta_0U)_t+ \sum^3_{i=1}A^i(\bU)(\Delta_0 \bU)_{x_i}= -[\Delta_0, A^\alpha(\bU)]\partial_{x_\alpha}\bU.
\end{equation*}
Due to Lemma \ref{te1}, we have
\begin{equation}\label{60}
  \vert\kern-0.25ex\vert\kern-0.25ex\vert \Delta_0U\vert\kern-0.25ex\vert\kern-0.25ex\vert^2_{s_0,2,\Sigma} \lesssim \|d \bU \|_{L^2_t L^{\infty}_x}\| \Delta_0 \bU\|^2_{L^{\infty}_tH_x^{s}}+\| [\Delta_0, A^\alpha(\bU)]\partial_{x_\alpha}\bU \|_{L^{2}_tH_x^{s_0}}\| \Delta_0 \bU\|_{L^{\infty}_tH_x^{s_0}}.
\end{equation}
By commutator estimates, we obtain
\begin{equation*}
  \| [\Delta_0, A^\alpha(\bU)]\partial_{x_\alpha}\bU \|_{L^{2}_tH_x^{s_0}} \lesssim \|d \bU \|_{L^2_t L^{\infty}_x}\| \Delta_0 \bU\|_{L^{\infty}_tH_x^{s_0}}.
\end{equation*}
Hence, \eqref{60} yields
\begin{equation}\label{60e}
  \vert\kern-0.25ex\vert\kern-0.25ex\vert \Delta_0 \bU \vert\kern-0.25ex\vert\kern-0.25ex\vert^2_{s_0,2,\Sigma} \lesssim \|d \bU \|_{L^2_t L^{\infty}_x}\| \Delta_0 \bU\|^2_{L^{\infty}_tH_x^{s_0}}.
\end{equation}
To control the norm of $2^j(\bU-S_j \bU)$, we write
\begin{equation*}
  2^j(\bU-S_j \bU)= 2^j \sum_{m\geq j}\Delta_m \bU,
\end{equation*}
where $\Delta_m \bU$ satisfies the above conditions for $\Delta_0 \bU$. Applying \eqref{60e} and replacing $s_0$ to $s_0-1$, we get
\begin{equation*}
\begin{split}
   \vert\kern-0.25ex\vert\kern-0.25ex\vert 2^j(\bU-S_j \bU) \vert\kern-0.25ex\vert\kern-0.25ex\vert^2_{s_0-1,2,\Sigma}
  \lesssim & \sum_{m\geq j} \vert\kern-0.25ex\vert\kern-0.25ex\vert 2^m \Delta_m \bU \vert\kern-0.25ex\vert\kern-0.25ex\vert^2_{s_0-1,2,\Sigma}
  \\
  \lesssim & \|d \bU \|_{L^2_t L^{\infty}_x} \sum_{m\geq j} (2^m\|\Delta_m \bU\|^2_{L^{\infty}_tH_x^{s_0}})
  \\
   \lesssim & \|d \bU \|_{L^2_t L^{\infty}_x} \| \bU \|^2_{L^{\infty}_tH_x^{s_0}} \lesssim \epsilon^2_2.
\end{split}
\end{equation*}
Taking square of it, we can see
\begin{equation*}
   \vert\kern-0.25ex\vert\kern-0.25ex\vert 2^j(\bU-S_j \bU) \vert\kern-0.25ex\vert\kern-0.25ex\vert_{s_0-1,2,\Sigma}
  \lesssim  \epsilon_2.
\end{equation*}
Finally, applying \eqref{60} to $\Delta_0=S_{j}$ and $\Delta_0=2^{-j}\partial_x S_{j}$ shows that
\begin{equation*}
  \vert\kern-0.25ex\vert\kern-0.25ex\vert d S_{j}\bU\vert\kern-0.25ex\vert\kern-0.25ex\vert_{s_0-1,2,\Sigma} +\vert\kern-0.25ex\vert\kern-0.25ex\vert 2^{-j} d \partial_x S_{j}\bU \vert\kern-0.25ex\vert\kern-0.25ex\vert_{s_0-1,2,\Sigma} \lesssim \epsilon_2.
  %\|U\|_{L^\infty_t H^{s_0}_x} +  \|dU\|_{L^2_t L^\infty}.
\end{equation*}
Therefore, the proof of Lemma \ref{fre} is completed.
\end{proof}
We are now ready to prove Proposition \ref{r1}.
%\subsection{Proof of Proposition  \ref{r1}}
\medskip\begin{proof}[Proof of Proposition \ref{r1}]
Note $(\bv, \rho, h, \bw) \in \mathcal{H}$. Using Lemma \ref{fre}, it is sufficient to verify that
\begin{equation*}
\begin{split}
  \vert\kern-0.25ex\vert\kern-0.25ex\vert {\mathbf{g}}^{\alpha \beta}-\mathbf{m}^{\alpha \beta}\vert\kern-0.25ex\vert\kern-0.25ex\vert_{s_0,2,\Sigma_{\theta,r}} \lesssim \epsilon_2.
\end{split}
\end{equation*}
By Lemma \ref{te1} and \eqref{5021}-\eqref{504}, we have
\begin{equation*}
  \sup_{\theta,r}\vert\kern-0.25ex\vert\kern-0.25ex\vert \bv \vert\kern-0.25ex\vert\kern-0.25ex\vert_{s_0,2,\Sigma_{\theta,r}}+ \sup_{\theta,r}\vert\kern-0.25ex\vert\kern-0.25ex\vert \rho \vert\kern-0.25ex\vert\kern-0.25ex\vert_{s_0,2,\Sigma_{\theta,r}} \lesssim \epsilon_2.
\end{equation*}
Noting the expression of $\mathbf{g}$, by Lemma \ref{te2}, we derive that
\begin{equation*}
\begin{split}
  \vert\kern-0.25ex\vert\kern-0.25ex\vert  {\mathbf{g}}^{\alpha \beta}-\mathbf{m}^{\alpha \beta}\vert\kern-0.25ex\vert\kern-0.25ex\vert _{s_0,2,\Sigma_{\theta,r}} & \lesssim \vert\kern-0.25ex\vert\kern-0.25ex\vert \bv\vert\kern-0.25ex\vert\kern-0.25ex\vert_{s_0,2,\Sigma_{\theta,r}}+\vert\kern-0.25ex\vert\kern-0.25ex\vert \bv \cdot \bv\vert\kern-0.25ex\vert\kern-0.25ex\vert_{s_0,2,\Sigma_{\theta,r}}+\vert\kern-0.25ex\vert\kern-0.25ex\vert c_s^2-c_s^2(0)\vert\kern-0.25ex\vert\kern-0.25ex\vert_{s_0,2,\Sigma_{\theta,r}}
  \\
  & \lesssim \epsilon_2.
\end{split}
\end{equation*}
Thus, the conclusion of Proposition \ref{r1} holds.
\end{proof}
To prove Proposition \ref{r2}, let us introduce a null frame of $\Sigma$ and give estimates of the coefficients of connection.
\subsection{The null frame}
We introduce a null frame along $\Sigma$ as follows. Let
\begin{equation*}
  V=(dr)^*,
\end{equation*}
where $r$ is the defining function of the foliation $\Sigma$, and where $*$ denotes the identification of covectors and vectors induced by $\mathbf{g}$. Then $V$ is the null geodesic flow field tangent to $\Sigma$. Let
\begin{equation}\label{600}
  \sigma=dt(V), \qquad l=\sigma^{-1} V.
\end{equation}
Thus $l$ is the g-normal field to $\Sigma$ normalized so that $dt(l)=1$, hence
\begin{equation}\label{601}
  l=\left< dt,dx_3-d\phi\right>^{-1}_{\mathbf{g}} \left( dx_3-d \phi \right)^*,
\end{equation}
so the coefficients $l^j$ are smooth functions of $v, \rho$ and $d \phi$. Conversely,
\begin{equation}\label{602}
 dx_3-d \phi =\left< l,\partial_{x_3}\right>^{-1}_{\mathbf{g}} l^*,
\end{equation}
so that $d \phi$ is a smooth function of $\bv, \rho, h$ and the coefficients of $l$.

Next we introduce the vector fields $e_a, a=1,2$ tangent to the fixed-time slice $\Sigma^t$ of $\Sigma$. We do this by applying Gram-Schmidt orthogonalization in the metric $\mathbf{g}$ to the $\Sigma^t$-tangent vector fields $\partial_{x_a}+ \partial_{x_a} \phi \partial_{x_3}, a=1, 2$.

Finally, we let
\begin{equation*}
  \underline{l}=l+2\partial_t.
\end{equation*}
It follows that $\{l, \underline{l}, e_1, e_2 \}$ form a null frame in the sense that
\begin{align*}
  & \left<l, \underline{l} \right>_{\mathbf{g}} =2, \qquad \qquad \ \left< e_a, e_b\right>_{\mathbf{g}}=\delta_{ab} \ (a,b=1,2),
  \\
  & \left<l, l \right>_{\mathbf{g}} =\left<\underline{l}, \underline{l} \right>_{\mathbf{g}}=0, \quad \left<l, e_a \right>_{\mathbf{g}}=\left<\underline{l}, e_a \right>_{\mathbf{g}}=0 \ (a=1,2).
\end{align*}
The coefficients of each of the fields is a smooth function of $(\bv,\rho,h)$ and $d \phi$, and by assumption we also have the pointwise bound
\begin{equation*}
  | e_1 - \partial_{x_1} | +| e_2 - \partial_{x_2} | + | l- (\partial_t+\partial_{x_3}) | + | \underline{l} - (-\partial_t+\partial_{x_3})|  \lesssim \epsilon_1.
\end{equation*}
After that, we can state the following corollary concerning to the decomposition of curvature tensor.
\begin{Lemma}\label{LLQ}\cite{ST}
Suppose $f$ satisfies
$$\mathbf{g}^{\alpha \beta} \partial^2_{\alpha \beta}f=F.$$
Let $(t,x',\phi(t,x'))$ denote the parametrization of $\Sigma$, and for $0 \leq \alpha, \beta \leq 2$, let $/\kern-0.55em \partial_\alpha$ denote differentiation along $\Sigma$ in the induced coordinates. Then, for $0 \leq \alpha, \beta \leq 2$, one can write
\begin{equation*}
  /\kern-0.55em \partial_\alpha /\kern-0.55em \partial_\beta (f|_{\Sigma}) = l(f_2)+ f_1,
\end{equation*}
where
\begin{equation*}
  \| f_2 \|_{L^2_t H^{s_0-1}_{x'}(\Sigma)}+\| f_1 \|_{L^1_t H^{s_0-1}_{x'}(\Sigma)} \lesssim \|df\|_{L^\infty_t H_x^{s_0-1}}+ \|df\|_{L^2_t L_x^\infty}+ \| F\|_{L^2_t H^{s_0-1}_x}+ \| F\|_{L^1_t H^{s_0-1}_{x'}(\Sigma)}.
\end{equation*}
\end{Lemma}
%\begin{proof}
%Let
%\end{proof}
\begin{corollary}\label{Rfenjie}
Let $R$ be the Riemann curvature tensor of the metric ${\mathbf{g}}$. Let $e_0=l$. Then for any $0 \leq a, b, c,d \leq 2$, we can write
\begin{equation}\label{603}
  \left< R(e_a, e_b)e_c, e_d \right>_{\mathbf{g}}|_{\Sigma}=l(f_2)+f_1,
\end{equation}
where $|f_1|\lesssim |\mathrm{curl} \bw|+ |d \mathbf{g} |^2+ |\Delta h|$ and $|f_2| \lesssim |d {\mathbf{g}}|$. Moreover, the characteristic energy estimates
\begin{equation}\label{604}
  \|f_2\|_{L^2_t H^{s_0-1}_{x'}(\Sigma)}+\|f_1\|_{L^1_t H^{s_0-1}_{x'}(\Sigma)} \lesssim \epsilon_2,
\end{equation}
holds. Additionally, for any $t \in [-2,2]$,
\begin{equation}\label{605}
  \|f_2(t,\cdot)\|_{C^\delta_{x'}(\Sigma^t)} \lesssim \|d \mathbf{g}\|_{C^\delta_x(\mathbb{R}^3)}.
\end{equation}
\end{corollary}
\begin{proof}
According to the expression of curvature tensor, we have
\begin{equation*}
  \left< R(e_a, e_b)e_c, e_d \right>_{\mathbf{g}}= R_{\alpha \beta \mu \nu}e^\alpha_a e^\beta_b e_c^\mu e_d^\nu,
\end{equation*}
where
\begin{equation*}
  R_{\alpha \beta \mu \nu}= \frac12 \left[ \partial^2_{\alpha \mu} \mathbf{g}_{\beta \nu}+\partial^2_{\beta \nu} \mathbf{g}_{\alpha \mu}-\partial^2_{\beta \mu} \mathbf{g}_{\alpha \nu}-\partial^2_{\alpha \nu} \mathbf{g}_{\beta \mu} \right]+ Q(\mathbf{g}^{\alpha \beta}, d \mathbf{g}_{\alpha \beta}),
\end{equation*}
where $Q$ is a sum of products of coefficients of $\mathbf{g}^{\alpha \beta} $ with quadratic forms in $d \mathbf{g}_{\alpha \beta}$. By using Proposition \ref{r1}, then $Q$ satisfies the bound required of $f_1$. It is sufficient to consider
\begin{equation*}
   \frac12 e^\alpha_a e^\beta_b e_c^\mu e_d^\nu  \left[ \partial^2_{\alpha \mu} \mathbf{g}_{\beta \nu}+\partial^2_{\beta \nu} \mathbf{g}_{\alpha \mu}-\partial^2_{\beta \mu} \mathbf{g}_{\alpha \nu}-\partial^2_{\alpha \nu} \mathbf{g}_{\beta \mu} \right].
\end{equation*}
We therefore look at the term $ e^\alpha_a e_c^\mu \partial^2_{\alpha \mu} \mathbf{g}_{\beta \nu} $, which is typical. By \eqref{503}, Proposition \ref{r1}, and Lemma \ref{te3}, we get
\begin{equation}\label{LLL}
  \vert\kern-0.25ex\vert\kern-0.25ex\vert  l^\alpha - \delta^{\alpha 0} \vert\kern-0.25ex\vert\kern-0.25ex\vert _{s_0,2,\Sigma} +\vert\kern-0.25ex\vert\kern-0.25ex\vert  \underline{l}^\alpha + \delta^{\alpha 0}-2\delta^{\alpha n} \vert\kern-0.25ex\vert\kern-0.25ex\vert _{s_0,2,\Sigma} + \vert\kern-0.25ex\vert\kern-0.25ex\vert  e^\alpha_a- \delta^{\alpha a} \vert\kern-0.25ex\vert\kern-0.25ex\vert _{s_0,2,\Sigma} \lesssim \epsilon_1.
\end{equation}
By \eqref{LLL} and Proposition \ref{r1}, the term $ e_a (e_c^\mu) \partial_{ \mu} \mathbf{g}_{\beta \nu}$ satisfies the bound required of $f_1$, so we consider $e_a(e_c(\mathbf{g}_{\beta \nu}))$. Finally, since the coefficients of $e_c$ in the basis $/\kern-0.55em \partial_\alpha$ have tangential derivatives bounded in $L^2_tH^{s_0-1}_{x'}(\Sigma)$, we are reduced by Lemma \ref{LLQ} to verifying that
\begin{equation*}
 \| \mathbf{g}^{\alpha \beta } \partial^2_{\alpha \beta} \mathbf{g}_{\mu \nu} \|_{L^1_t H^{s_0-1}_{x'}(\Sigma)} \lesssim \epsilon_2.
\end{equation*}
Note
\begin{equation*}
  \mathbf{g}^{\alpha \beta } \partial^2_{\alpha \beta} \mathbf{g}_{\mu \nu}= \square_{\mathbf{g}} \mathbf{g}_{\mu \nu}.
\end{equation*}
By Lemma \ref{te20}, Lemma \ref{te21}, and Corollary \ref{vte}, we have
\begin{equation*}
\begin{split}
\| \square_{\mathbf{g}} \mathbf{g}_{\mu \nu}\|_{L^1_t H^{s_0-1}_{x'}(\Sigma)}
 \lesssim & \ \| \square_{\mathbf{g}} \bv\|_{L^2_t H^{s_0-1}_{x'}(\Sigma)}+\| \square_{\mathbf{g}} \rho \|_{L^2_t H^{s_0-1}_{x'}(\Sigma)}+\| \square_{\mathbf{g}} h \|_{L^2_t H^{s_0-1}_{x'}(\Sigma)}
 \\
 \lesssim & \ \| \mathrm{curl} \bw\|_{L^2_t H^{s_0-1}_{x'}(\Sigma)}+ \| (d{\mathbf{g}})^2\|_{L^1_t H^{s_0-1}_{x'}(\Sigma)}+ \| \Delta h\|_{L^1_t H^{s_0-1}_{x'}(\Sigma)}
 \\
 \lesssim & \ \| \mathrm{curl} \bw\|_{L^2_t H^{s_0-1}_{x'}(\Sigma)}+ \| d{\mathbf{g}} \|_{L^2_t L^\infty_x }\| d{\mathbf{g}}\|_{L^2_t H^{s_0-1}_{x'}(\Sigma)}+ \| \Delta h\|_{L^1_t H^{s_0-1}_{x'}(\Sigma)}
 \\
 \lesssim & \ \| \mathrm{curl} \bw\|_{L^2_t H^{s_0-1}_{x'}(\Sigma)}+ \| d\bv, d\rho \|_{L^2_t L^\infty_x }\| d\bv, d\rho\|_{L^2_t H^{s_0-1}_{x'}(\Sigma)}+ \| \Delta h\|_{L^2_t H^{s_0-1}_{x'}(\Sigma)}
 \\
 \lesssim & \ \epsilon_2.
\end{split}
\end{equation*}
Above, $\mathrm{curl} \bw, \Delta h$ and $(d{\mathbf{g}})^2$ are included in $f_1$. We hence complete the proof of Corollary \ref{Rfenjie}.
\end{proof}

\subsection{Estimate of the connection coefficients}
Define
\begin{equation*}
  \chi_{ab} = \left<D_{e_a}l,e_b \right>_{\mathbf{g}}, \qquad l(\ln \sigma)=\frac{1}{2}\left<D_{l}\underline{l},l \right>_{\mathbf{g}}, \quad \mu_{0ab} = \left<D_{l}e_a,e_b \right>_{\mathbf{g}}.
\end{equation*}
For $\sigma$, we set the initial data $\sigma=1$ at the time $-2$. Thanks to Proposition \ref{r1}, we have
\begin{equation}\label{606}
  \|\chi_{ab}\|_{L^2_t H^{s_0-1}_{x'}(\Sigma)} + \| l(\ln \sigma)\|_{L^2_t H^{s_0-1}_{x'}(\Sigma)} + \|\mu_{0ab}\|_{L^2_t H^{s_0-1}_{x'}(\Sigma)} \lesssim \epsilon_1.
\end{equation}
In a similar way, if we expand $l=l^\alpha /\kern-0.55em \partial_\alpha$ in the tangent frame $\partial_t, \partial_{x'}$ on $\Sigma$, then
\begin{equation}\label{607}
  l^0=1, \quad \|l^1\|_{s_0,2,\Sigma} \lesssim \epsilon_1.
\end{equation}
\begin{Lemma}\label{chi}
Let $\chi$ be defined as before. Then
\begin{equation}\label{608}
  \|\chi_{ab}\|_{L^2_t H^{s_0-1}_{x'}(\Sigma)} \lesssim \epsilon_2.
\end{equation}
Furthermore, for any $t \in [-2,2]$,
\begin{equation}\label{609}
\| \chi_{ab} \|_{C^{\delta}_{x'}(\Sigma^t)} \lesssim \epsilon_2+ \|d \mathbf{g} \|_{C^{\delta}_{x}(\mathbb{R}^3)}.
\end{equation}
\end{Lemma}
\begin{proof}
The well-known transport equation for $\chi$ along null hypersurfaces (see references \cite{KR2} and \cite{ST}) can be described as
\begin{equation*}
  l(\chi_{ab})=\left< R(l,e_a)l, e_b \right>_{\mathbf{g}}-\chi_{ac}\chi_{cb}-l(\ln \sigma)\chi_{ab}+\mu_{0ab} \chi_{cb}+ \mu_{0bc}\chi_{ac}.
\end{equation*}
Due to Corollary \ref{Rfenjie}, we can write the above equation in the following
\begin{equation}\label{610}
  l(\chi_{ab}-f_2)=f_1-\chi_{ac}\chi_{cb}-l(\ln \sigma)\chi_{ab}+\mu_{0ab} \chi_{cb}+ \mu_{0bc}\chi_{ac},
\end{equation}
where
\begin{equation}\label{611}
  \|f_2\|_{L^2_t H^{s_0-1}_{x'}(\Sigma)}+\|f_1\|_{L^1_t H^{s_0-1}_{x'}(\Sigma)} \lesssim \epsilon_2,
\end{equation}
and for any $t \in [0,T]$,
\begin{equation}\label{612}
  \|f_2(t,\cdot)\|_{C^\delta_{x'}(\Sigma^t)} \lesssim \|d \mathbf{g}\|_{C^\delta_x(\mathbb{R}^3)}.
\end{equation}
Let $\Lambda_{x'}^{s_0-1}$ be the fractional derivative operator in the $x'$ variables. Set
\begin{equation*}
  G=f_1-\chi_{ac}\chi_{cb}-l(\ln \sigma)\chi_{ab}+\mu_{0ab} \chi_{cb}+ \mu_{0bc}\chi_{ac}.
\end{equation*}
We have
\begin{equation}\label{613}
  \begin{split}
 & \|\Lambda_{x'}^{s_0-1}(\chi_{ab}-f_2)(t,\cdot) \|_{L^2_{x'}(\Sigma^t)}
  \\
  \lesssim \ &\| [\Lambda_{x'}^{s_0-1},l](\chi_{ab}-f_2) \|_{L^1_tL^2_{x'}(\Sigma^t)}+ \| \Lambda_{x'}^{s_0-1}G \|_{L^1_tL^2_{x'}(\Sigma^t)}.
  \end{split}
\end{equation}
A direct calculation shows that
\begin{equation}\label{614}
  \begin{split}
 \| \Lambda_{x'}^{s-1}G \|_{L^1_tL^2_{x'}(\Sigma^t)} &\lesssim \|f_1\|_{L^1_tH^{s_0-1}_{x'}(\Sigma^t)}+ \|\chi\|^2_{L^2_tH^{s_0-1}_{x'}(\Sigma^t)}
  \\
  & \quad + \|\chi\|_{L^2_tH^{s_0-1}_{x'}(\Sigma^t)}\cdot\|l(\ln \sigma)\|_{L^2_tH^{s_0-1}_{x'}(\Sigma^t)}
  \\
  & \quad + \|\mu\|_{L^2_tH^{s_0-1}_{x'}(\Sigma^t)}\cdot\|\chi\|_{L^2_tH^{s_0-1}_{x'}(\Sigma^t)},
  \end{split}
\end{equation}
where we use the fact that $H^{s_0-1}_{x'}(\Sigma^t)$ is an algebra for $s_0>2$.

We next bound
\begin{align*}
  \| [\Lambda_{x'}^{s_0-1},l](\chi_{ab}-f_2) \|_{L^2_{x'}(\Sigma^t)} &\leq \| /\kern-0.55em \partial_{\alpha} l^{\alpha} (\chi_{ab}-f_2)(t,\cdot) \|_{H^{s_0-1}_{x'}(\Sigma^t)}
  \\
  & \quad \ + \|[\Lambda_{x'}^{s_0-1} /\kern-0.55em \partial_{\alpha}, l^{\alpha}](\chi-f_2)(t,\cdot) \|_{L^{2}_{x'}(\Sigma^t)}.
\end{align*}
By Kato-Ponce commutator estimate and Sobolev embeddings, the above is bounded by
\begin{equation}\label{615}
  \|l^1(t,\cdot)\|_{H^{s_0-1}_{x'}(\Sigma^t)} \| \Lambda^{s_0-1}(\chi_{ab}-f_2)(t,\cdot) \|_{L^{2}_{x'}(\Sigma^t)} .
\end{equation}
Gathering \eqref{606}, \eqref{607}, \eqref{611} and \eqref{613}-\eqref{615} together, we thus prove that
\begin{equation*}
  \sup_t \|(\chi_{ab}-f_2)(t,\cdot)\|_{H^{s_0-1}_{x'}(\Sigma^t)}  \lesssim \epsilon_2.
\end{equation*}
%The conclusion now follows by Corollary \ref{Rfenjie} and Sobolev imbedding.
From \eqref{610}, we see that
\begin{equation}\label{616}
\begin{split}
  \| \chi_{ab}-f_2\|_{C^{\delta}_{x'}} & \lesssim \| f_1 \|_{L^1_tC^{\delta}_{x'}}+ \|\chi_{ac}\chi_{cb}\|_{L^1_tC^{\delta}_{x'} }+\|l(\ln \sigma)\chi_{ab}\|_{L^1_tC^{\delta}_{x'}}
  \\
  & \quad + \|\mu_{0ab} \chi_{cb}\|_{L^1_tC^{\delta}_{x'} }+\|\mu_{0bc}\chi_{ac}\|_{L^1_tC^{\delta}_{x'} }.
  \end{split}
\end{equation}
By Sobolev imbedding $H^{s_0-1}(\mathbb{R}^2)\hookrightarrow C^{\delta}(\mathbb{R}^2), \delta \in (0, s_0-2)$ and Gronwall's inequality, we derive that
\begin{equation*}
\| \chi_{ab} \|_{C^{\delta}_{x'}(\Sigma^t)} \lesssim \epsilon_2+ \|d \mathbf{g}\|_{C^{\delta}_{x}(\mathbb{R}^3)}.
\end{equation*}
\end{proof}
\subsection{The proof of Proposition \ref{r2}}
Note
\begin{equation*}
  G(\bv, \rho, h)= \vert\kern-0.25ex\vert\kern-0.25ex\vert d\phi(t,x')-dt\vert\kern-0.25ex\vert\kern-0.25ex\vert_{s_0,2, \Sigma}.
\end{equation*}
%so it suffices for us to control the latter quantity by $\epsilon_2$.
Using \eqref{602} and the estimate of $\vert\kern-0.25ex\vert\kern-0.25ex\vert \mathbf{g}-\mathbf{m} \vert\kern-0.25ex\vert\kern-0.25ex\vert_{s_0,2,\Sigma}$ in Proposition \ref{r1}, then the estimate \eqref{G} follows from the bound
\begin{equation*}
  \vert\kern-0.25ex\vert\kern-0.25ex\vert l-(\partial_t-\partial_{x_3})\vert\kern-0.25ex\vert\kern-0.25ex\vert_{s_0,2,\Sigma} \lesssim \epsilon_2,
\end{equation*}
where it is understood that one takes the norm of the coefficients of $l-(\partial_t-\partial_{x_3})$ in the standard frame on $\mathbb{R}^{3+1}$. The geodesic equation, together with the bound for Christoffel symbols $\|\Gamma^\alpha_{\beta \gamma}\|_{L^2_t L^\infty_x} \lesssim \|d {\mathbf{g}} \|_{L^2_t L^\infty_x}\lesssim \epsilon_2$, imply that
\begin{equation*}
  \|l-(\partial_t-\partial_{x_3})\|_{L^\infty_{t,x}} \lesssim \epsilon_2,
\end{equation*}
so it is sufficient to bound the tangential derivatives of the coefficients of $l-(\partial_t-\partial_{x_3})$ in the norm $L^2_t H^{s_0-1}_{x'}(\Sigma)$. By using Proposition \ref{r1}, we can estimate the Christoffel symbols
\begin{equation*}
  \|\Gamma^\alpha_{\beta \gamma} \|_{L^2_t H^{s_0-1}_{x'}(\Sigma^t)} \lesssim \epsilon_2.
\end{equation*}
Note that $H^{s_0-1}_{x'}(\Sigma^t)$ is a algebra if $s_0>2$. We then deduce that
\begin{equation*}
  \|\Gamma^\alpha_{\beta \gamma} e_1^\beta l^\gamma\|_{L^2_t H^{s_0-1}_{x'}(\Sigma^t)} \lesssim \epsilon_2.
\end{equation*}
We are now in a position to establish the following bound,
\begin{equation*}
  \| \left< D_{e_a}l, e_b \right>\|_{L^2_t H^{s_0-1}_{x'}(\Sigma^t)}+ \| \left< D_{e_a}l, \underline{l} \right>\|_{L^2_t H^{s_0-1}_{x'}(\Sigma^t)}+\|\left< D_{l}l, \underline{l} \right>\|_{L^2_t H^{s_0-1}_{x'}(\Sigma^t)} \lesssim \epsilon_2.
\end{equation*}
The first term is $\chi_{ab}$, which has been bounded by Lemma \ref{chi}. For the second term, noting
\begin{equation*}
  \left< D_{e_a}l, \underline{l} \right>=\left< D_{e_a}l, 2\partial_t \right>=-2\left< D_{e_a}\partial_t,l \right>,
\end{equation*}
then it can be bounded via Proposition \ref{r1}. Similarly, we can control the last term through Proposition \ref{r1}. At this stage, it remains for us to show that
\begin{equation*}
  \| d \phi(t,x')-dt \|_{C^{1,\delta}_{x'}(\mathbb{R}^2)}  \lesssim \epsilon_2+ \| d\mathbf{g}(t,\cdot)\|_{C^{\delta}_x(\mathbb{R}^3)}.
\end{equation*}
To do that, it is sufficient to show that
\begin{equation*}
  \|l(t,\cdot)-(\partial_t-\partial_{x_3})\|_{C^{1,\delta}_{x'}(\mathbb{R}^2)} \lesssim \epsilon_2+ \| d \mathbf{g} (t,\cdot)\|_{C^{\delta}_x(\mathbb{R}^3)}.
\end{equation*}
The coefficients of $e_1$ are small in $C^{\delta}_{x'}(\Sigma^t)$ perturbations of their constant coefficient analogs, so it is sufficient to show that
\begin{equation*}
 \|\left< D_{e_1}l, e_1 \right>(t,\cdot)\|_{C^{\delta}_{x'}(\Sigma^t)}
 +\|\left< D_{e_1}l, \underline{l} \right>(t,\cdot)\|_{C^{\delta}_{x'}(\Sigma^t)}  \lesssim \epsilon_2+ \| d\mathbf{g}(t,\cdot)\|_{C^{\delta}_x(\mathbb{R}^3)}.
\end{equation*}
Above, the first term is bounded by Lemma \ref{chi}, and the second term can be estimated by using
\begin{equation*}
  \|\left< D_{e_1}\partial_t, l \right>(t,\cdot)\|_{C^{\delta}_{x'}(\Sigma^t)} \lesssim  \| d\mathbf{g}(t,\cdot)\|_{C^{\delta}_x(\mathbb{R}^3)}.
\end{equation*}
Consequently, we complete the proof of Proposition \ref{r2}.

\section{Strichartz estimates}\label{SEs}
In this part, let us prove the Strichartz estimates of linear wave equations including in Proposition \ref{p4}. To state it clearly, let us introduce it as follows.
\begin{proposition}\label{r5}
Suppose that $(\bv,\rho, h, \bw) \in \mathcal{{H}}$ and $G(\bv,\rho, h)\leq 2 \epsilon_1$. For each $1 \leq r \leq s+1$, and for each $t_0 \in [-2,2]$, then the linear, homogeneous equation
\begin{equation*}
	\begin{cases}
	& \square_{\mathbf{g}} f=0,\qquad \quad [-2,2]\times \mathbb{R}^3,
	\\
	&f(t,x)|_{t=t_0}=f_0, \quad \partial_t f(t,x)|_{t=t_0}=f_1,
	\end{cases}
	\end{equation*}
is well-posed on the time-interval $[-2,2]$ if the initial data $(f_0, f_1)\in H_x^r \times H_x^{r-1}$. Moreover, the solution satisfies, respectively, the energy estimates
\begin{equation*}
  \| f\|_{L^\infty_{[-2,2]}H_x^r}+ \| \partial_t f\|_{L^\infty_{[-2,2]}H_x^{r-1}} \leq C \big( \|f_0\|_{H_x^r}+ \|f_1\|_{H_x^{r-1}} \big),
\end{equation*}
and the Strichartz estimates\footnote{here the constant $C$ is universal if the solutions $(\bv,\rho,h,\bw) \in \mathcal{{H}}$.}
\begin{equation}\label{SL}
  \| \left<\partial \right>^k f\|_{L^2_{[-2,2]}L^\infty_x} \leq C \big( \|f_0\|_{H_x^r}+ \|f_1\|_{H_x^{r-1}} \big), \quad k<r-1.
\end{equation}
\end{proposition}
We will postpone the proof of Proposition \ref{r5} to Section \ref{Ap}. Based on Proposition \ref{r5} and Lemma \ref{LD}, we have the following result.
\begin{proposition}\label{r3}
	Suppose that $(\bv,\rho, h, \bw) \in \mathcal{{H}}$ and $G(\bv,\rho, h)\leq 2 \epsilon_1$.
	For any $1 \leq r \leq s+1$, and for each $t_0 \in [-2,2]$, the linear, non-homogeneous equation
	\begin{equation*}
		\begin{cases}
			& \square_{\mathbf{g}} f=\mathbf{T}F, \qquad (t,x) \in (t_0,2]\times \mathbb{R}^3,
			\\
			&f(t,x)|_{t=t_0}=f_0 \in H_x^r(\mathbb{R}^3),
			\\
			&\partial_t f(t,x)|_{t=t_0}=f_1-F(t_0,\cdot) \in H_x^{r-1}(\mathbb{R}^3),
		\end{cases}
	\end{equation*}
	admits a solution $f \in C([-2,2],H_x^r) \times C^1([-2,2],H_x^{r-1})$ and the following estimates holds:
	\begin{equation}\label{lw0}
		\begin{split}
			\| f\|_{L_{[-2,2]}^\infty H_x^r}+ \|\partial_t f\|_{L_{[-2,2]}^\infty H_x^{r-1}} \lesssim & \|f_0\|_{H_x^r}+ \|f_1\|_{H_x^{r-1}}
			 +\| F \|_{L^1_{[-2,2]}H_x^r}.
		\end{split}
	\end{equation}
	Additionally, the following estimates hold, provided $k<r-1$,
	\begin{equation}\label{lw1}
		\| \left<\partial \right>^k f\|_{L^2_{[-2,2]}L^\infty_x} \lesssim \|f_0\|_{H_x^r}+ \|f_1\|_{H_x^{r-1}}+\| F \|_{L^1_{[-2,2]}H_x^r},
	\end{equation}
and the similar estimates also hold when we replace $\left<\partial \right>^k$ by $\left<\partial \right>^{k-1}d$.
\end{proposition}
\medskip\begin{proof}
	Let $V$ satisfy the the linear, homogeneous equation
	\begin{equation}\label{Vf}
		\begin{cases}
			& \square_{\mathbf{g}} V=0,
			\\
			&V(t,x)|_{t=t_0}=f_0, \quad \partial_t V(t,x)|_{t=t_0}=f_1,
		\end{cases}
	\end{equation}
	Let $Q$ satisfy the the linear, non-homogeneous equation
	\begin{equation}\label{Qf}
		\begin{cases}
			& \square_{\mathbf{g}} Q=\mathbf{T}F,
			\\
			&Q(t,x)|_{t=t_0}=0, \quad \partial_t Q(t,x)|_{t=t_0}=-F(t_0,x),
		\end{cases}
	\end{equation}
By superposition principle for linear wave equation, and referring \eqref{Vf} and \eqref{Qf}, then $f= V+Q$ satisfies
\begin{equation*}
	\begin{cases}
		& \square_{\mathbf{g}} f=\mathbf{T}F,
		\\
		&(f, \partial_t f)|_{t=t_0}=(f_0, f_1-F(t_0,x)).
	\end{cases}
\end{equation*}
Using Proposition \ref{r5}, it follows
\begin{equation}\label{Vf1}
	\| V\|_{L^\infty_tH_x^r}+ \| \partial_t V\|_{L^\infty_tH_x^{r-1}} \leq C \big( \|f_0\|_{H_x^r}+ \|f_1\|_{H_x^{r-1}} \big),
\end{equation}
and
\begin{equation}\label{Vf2}
	\| \left<\partial \right>^k V\|_{L^2_{t}L^\infty_x} \leq C \big( \|f_0\|_{H_x^r}+ \|f_1\|_{H_x^{r-1}} \big), \quad k<r-1.
\end{equation}
To estimate $Q$, we need to rewrite \eqref{Qf}. Using $\mathbf{T}=\partial_t+\bv \cdot \nabla$, $Q(t_0,x)=0$, and $\partial_t Q(t_0,x)=-F(t_0,x)$, it yields
\begin{equation}\label{Qf1}
\mathbf{T}Q(t,x)|_{t=t_0}=\partial_tQ(t_0,x)+(\bv \cdot \nabla)Q(t_0,x)= -F(t_0,x).
\end{equation}
Combining \eqref{Qf} and \eqref{Qf1}, we obtain
	\begin{equation}\label{Qf2}
		\begin{cases}
			& \square_{\mathbf{g}} Q=\mathbf{T}F,
			\\
			&Q(t,x)|_{t=t_0}=0, \quad \mathbf{T}Q(t,x)|_{t=t_0}=-F(t_0,x).
		\end{cases}
	\end{equation}
Using Lemma \ref{LD} and Proposition \ref{r5} again, we can prove
	\begin{equation}\label{QE}
		\| Q\|_{L^\infty_t H_x^r}+ \| \partial_t Q\|_{L^\infty_t H_x^{r-1}} \leq C \|F\|_{L^1_{[-2,2]}H_x^{r}},
	\end{equation}
	and
	\begin{equation}\label{SQ}
		\| \left<\partial \right>^k Q\|_{L^2_{t}L^\infty_x} \leq C \|F\|_{L^1_{[-2,2]}H_x^{r}}, \quad k<r-1.
	\end{equation}
Adding \eqref{Vf1} and \eqref{QE}, we get \eqref{lw0}. Adding \eqref{Vf2} and \eqref{SQ}, we get \eqref{lw1}. At this stage, we have finished the proof of Proposition \ref{r3}.
\end{proof}

\begin{proposition}\label{r6}
Suppose $(\bv, \rho, h, \bw) \in \mathcal{H}$ and $G(\bv, \rho, h)\leq 2 \epsilon_1$. Let $2<s_0<s\leq\frac52$. Let $\bv_{+}$ be defined in \eqref{dvc}. Then for $k<s-1$, we have
\begin{equation}\label{fgh}
\begin{split}
	 \|\left< \partial \right>^{k} (\rho+\frac{1}{\gamma}h), \left< \partial \right>^{k-1} \bv_{+}\|_{L^2_{[-2,2]} \dot{B}^{0}_{\infty, 2}}
	\lesssim  & \| \rho, \bv, h\|_{L^\infty_{[-2,2]} H^s}
	\\
	&+\| h\|_{L^2_{[-2,2]} H^{\frac{5}{2}+}}
	 +\| \bw \|_{L^2_{[-2,2]} H^{\frac{3}{2}+}},
\end{split}
\end{equation}
and the similar estimates hold if we replace $\left< \partial \right>^{k} $ with $\left< \partial \right>^{k-1} d$.
\end{proposition}
\begin{proof}
Let us recall that $\rho+\frac{1}{\gamma}h$ and $\bv_{+}$ satisfy the equation
\begin{equation}\label{fcr}
\begin{cases}
 & \square_{\mathbf{g}} (\rho+\frac{1}{\gamma}h)= D+ \frac{1}{\gamma}E,
\\
&\square_{\mathbf{g}} \bv_{+}=\mathbf{T}\mathbf{T} \bv_{-}+ \bQ.
\end{cases}
\end{equation}
Using Lemma \ref{yux}, we get
\begin{equation}\label{mlo}
  \|D, E, \bQ\|_{L^1_tH_x^{s-1}}  \lesssim  \|d \bv, d\rho, dh\|_{L^1_t L_x^\infty} \|d \bv, d\rho, dh\|_{L^\infty_tH_x^{s-1}}.
\end{equation}
Operating $\Delta_j$ on \eqref{fcr} and matching its initial data, then $\Delta_j \rho$ satisfies the following Cauchy problems
\begin{equation}\label{pq0}
\begin{cases}
 &\square_{\mathbf{g}} \Delta_j (\rho+\frac{1}{\gamma}h)= \Delta_j ( D+ \frac{1}{\gamma}E )+ [\square_{\mathbf{g}}, \Delta_j](\rho+\frac{1}{\gamma}h),
 \\
  & \Delta_j (\rho+\frac{1}{\gamma}h)|_{t=0}= \Delta_j \rho_0+ \frac{1}{\gamma}\Delta_j h_0,
  \\
  & \Delta_j \mathbf{T}(\rho+\frac{1}{\gamma}h)|_{t=0}= - \Delta_j \mathrm{div}\bv_0,
\end{cases}
\end{equation}
and
\begin{equation}\label{pq1}
\begin{cases}
& \square_{\mathbf{g}} \Delta_j v^i_{+}=
\mathbf{T}\{\Delta_j (\mathbf{T} v_{-}^i)\}+ \Delta_j Q^i+ [\square_{\mathbf{g}}, \Delta_j]v^i_{+}+[\Delta_j, \mathbf{T}](\mathbf{T} v_{-}^i),
\\
& \Delta_j v^i_{+}|_{t=0}=\Delta_j (v^i_0-v^i_{-0}), \\
& \Delta_j (\mathbf{T}v^i_{+})|_{t=0}=-\Delta_j (c^2_s(0)\partial^i \rho_0+ \frac{1}{\gamma}c^2_s(0)\partial^i h_0)- \Delta_j (\mathbf{T}v^i_{-0}).
\end{cases}
\end{equation}
Above, we use the \eqref{fc0} such that
\begin{equation*}
\begin{split}
  & \Delta_j \mathbf{T}(\rho_0+\frac{1}{\gamma}h_0)=- \Delta_j \mathrm{div}\bv_0, \quad v^i_{+0}=v^i_0-v^i_{-0},
  \\
  & \mathbf{T}v^i_{+0}=-c^2_s(0)\partial^i \rho_0- \frac{1}{\gamma}c^2_s(0)\partial^i h_0- \mathbf{T} v^i_{-0}.
\end{split}
\end{equation*}
By using the Strichartz estimate in Proposition \ref{r5} and Proposition \ref{r3} (taking $r=s$ and $k<s-1$), we obtain
\begin{equation}\label{ise1}
\begin{split}
  & \| \left< \partial \right>^k \Delta_j (\rho+\frac{1}{\gamma}h) \|_{L^2_{[-2,2]} L^\infty_x} + \| \left< \partial \right>^k \Delta_j \bv_{+}\|_{L^2_{[-2,2]} L^\infty_x}
  \\
  \lesssim  & \|\Delta_j \rho_0 \|_{H^{s}}+\|\Delta_j \bv_0 \|_{H^{s}}+\|\Delta_j h_0 \|_{H^{s}}  + \| \Delta_j \bv_{-0} \|_{H^{s}}
  \\
  &+ \|\Delta_j \mathbf{T} \bv_{-}\|_{L^1_{[-2,2]}H_x^{s-1}}+ \| [\square_{\mathbf{g}}, \Delta_j]\rho \|_{L^1_{[-2,2]} H_x^{s-1}}+ \| \Delta_j E \|_{L^1_{[-2,2]} H_x^{s-1}}
  \\
  &+ \| \Delta_j \bQ \|_{L^1_{[-2,2]} H_x^{s-1}}
    + \|[\square_{\mathbf{g}}, \Delta_j]\bv_{+}\|_{L^1_{[-2,2]} H_x^{s-1}}
  + \| [\Delta_j, \mathbf{T}](\mathbf{T} \bv_{-}) \|_{L^1_{[-2,2]} H_x^{s-1}}.
\end{split}
\end{equation}
Taking square of \eqref{ise1} and summing it over $j\geq 1$, we get
\begin{equation}\label{fgh8}
\begin{split}
  \|\left< \partial \right>^k (\rho+\frac{1}{\gamma}h), \left< \partial \right>^k \bv_{+}\|_{L^2_t \dot{B}^{0}_{\infty, 2}} \lesssim & \| \rho_0\|_{H_x^s}+\| \bv_0\|_{H_x^s}+\| h_0\|_{H_x^s}+\| \bv_{-0}\|_{H_x^s}
  \\
  &+\{ \| [\square_{\mathbf{g}}, \Delta_j]\rho \|_{L^1_{[-2,2]} H_x^{s-1}} \}_{l^2_j}
   + \|E,\bQ \|_{L^1_{[-2,2]} H_x^{s-1}}
   \\
   & + \{ \|[\square_{\mathbf{g}}, \Delta_j]\bv_{+}\|_{L^1_{[-2,2]} H_x^{s-1}}  \}_{l^2_j}
  \\
  & + \{ \| [\Delta_j, \mathbf{T}](\mathbf{T} \bv_{-}) \|_{L^1_{[-2,2]} H_x^{s-1}} \}_{l^2_j}.
\end{split}
\end{equation}
By Lemma \ref{YR}, we can see
\begin{equation}\label{fh00}
	\begin{split}
			\{ \| [\Delta_j, \mathbf{T}](\mathbf{T} \bv_{-}) \|_{L^1_{[-2,2]} H_x^{s-1}} \}_{l^2_j} \lesssim & \|\partial \bv \|_{L^1_{[-2,2]}L^\infty_x} \|\mathbf{T} \bv_{-}\|_{L^\infty_{[-2,2]}H^{s-1}_x}
		\\& + \|\bv\|_{L^\infty_{[-2,2]} H^s_x} \|\mathbf{T} \bv_{-}\|_{L^1_{[-2,2]} H^1_x}
	\end{split}
\end{equation}
By Lemma \ref{yx}, we also get
\begin{equation}\label{fh01}
\begin{split}
 \{ \| [\square_{\mathbf{g}}, \Delta_j]\rho \|_{L^1_{[-2,2]} H_x^{s-1}} \}_{l^2_j} \lesssim & \|d \rho \|_{L^1_{[-2,2]}L^\infty_x} \|d\bv, d\rho, dh\|_{L^\infty_{[-2,2]}H^{s-1}_x}
 \\& +  \|d\bv, d \rho, dh\|_{L^1_{[-2,2]}L^\infty_x} \|d\bv, d\rho, dh\|_{L^\infty_{[-2,2]}H^{s-1}_x}.
\end{split}
\end{equation}
Using Lemma \ref{yx} again, it follows
\begin{equation}\label{fh02}
	\begin{split}
		\{ \| [\square_{\mathbf{g}}, \Delta_j]\bv_{+} \|_{L^1_{[-2,2]} H_x^{s-1}} \}_{l^2_j} \lesssim & \|d \bv_{+} \|_{L^1_{[-2,2]}L^\infty_x} \|d\bv, d\rho, dh\|_{L^\infty_{[-2,2]}H^{s-1}_x}
		\\& +  \|d\bv, d \rho, dh\|_{L^2_{[-2,2]}L^\infty_x} \|d\bv_{+}\|_{L^2_{[-2,2]}H^{s-1}_x}.
	\end{split}
\end{equation}
By \eqref{YYE}, it yields
\begin{equation}\label{fh03}
	\begin{split}
		\|E,\bQ \|_{L^1_{[-2,2]} H_x^{s-1}} \lesssim & \|d \rho \|_{L^2_{[-2,2]}L^\infty_x} ( \|\bv, \rho\|_{L^\infty_{[-2,2]}H^{s}_x} + \|\bw\|_{L^2_{[-2,2]}H^{\frac32+}_x} + \|h\|_{L^2_{[-2,2]}H^{\frac52+}_x}).
	\end{split}
\end{equation}
Using \eqref{eta} and $(\bv,\rho,h,\bw)\in \mathcal{{H}}$(please refer \eqref{401}-\eqref{403}), we have
\begin{equation}\label{fh04}
	\|\mathbf{T}\bv_{-}\|_{L^2_{[-2,2]}H^{s-1}_x} \lesssim \|\bv, \rho, h\|_{L^\infty_{[-2,2]}H^{s}_x}+ \|\bw\|_{L^2_{[-2,2]}H^{\frac32+}_x} + \|h\|_{L^2_{[-2,2]}H^{\frac52+}_x}.
\end{equation}
Using \eqref{fh04}, we have
\begin{equation}\label{fh05}
	\begin{split}
		\|d\bv_{+}\|_{L^2_{[-2,2]}H^{s-1}_x} \leq & \|d\bv_{-}\|_{L^2_{[-2,2]}H^{s-1}_x}+ \|d\bv\|_{L^2_{[-2,2]}H^{s-1}_x}
		\\
		\leq & \|\mathbf{T}\bv_{-}\|_{L^2_{[-2,2]}H^{s-1}_x}+\|(\bv \cdot \nabla)\bv_{-}\|_{L^2_{[-2,2]}H^{s-1}_x} + \|d\bv\|_{L^2_{[-2,2]}H^{s-1}_x}
		\\
		\lesssim  & \|\bv, \rho, h\|_{L^\infty_{[-2,2]}H^{s}_x}+ \|\bw\|_{L^2_{[-2,2]}H^{\frac32+}_x} + \|h\|_{L^2_{[-2,2]}H^{\frac52+}_x}.
	\end{split}
\end{equation}
Inserting \eqref{fh00}-\eqref{fh05} in \eqref{fgh8}, \eqref{fgh8} yields
\begin{equation}\label{fh06}
	\begin{split}
		\|\left< \partial \right>^k (\rho+\frac{1}{\gamma}h), \left< \partial \right>^k \bv_{+}\|_{L^2_{[-2,2]} \dot{B}^{0}_{\infty, 2}} \lesssim & \| \rho\|_{L^\infty_{[-2,2]} H_x^s}+\| \bv\|_{L^\infty_{[-2,2]} H_x^s}+\| h\|_{L^\infty_{[-2,2]} H_x^s}
		\\
	& + \|\bw\|_{L^2_{[-2,2]}H^{\frac32+}_x} + \|h\|_{L^2_{[-2,2]}H^{\frac52+}_x} .
	\end{split}
\end{equation}
So we have proved \eqref{fgh}.
\end{proof}

\begin{remark}
We expect that Proposition \ref{r6} yields a Strichartz estimate of solutions with  sharp regularity of the velocity, density, entropy, and vorticity.
\end{remark}
\begin{proposition}\label{r4}
Suppose $(\bv, \rho, h, \bw) \in \mathcal{H}$ and $G(\bv, \rho, h)\leq 2 \epsilon_1$. Let $2<s_0<s<\frac52$ and $\delta_*\in [0,s-2)$. Let $\bv_{+}$ be defined in \eqref{dvc}. Then the following Strichartz estimates and energy estimates hold
\begin{equation}\label{strr}
\|d \bv, d \rho, dh, \partial \bv_{+}\|_{L^2_t C^{\delta_*}_x}+\|\partial \bv_{+}, d \rho, d\bv, dh\|_{L^2_t \dot{B}^{s_0-2}_{\infty,2}} \leq \epsilon_2,
\end{equation}
and the energy estimates
\begin{equation}\label{eef}
\| \bv\|_{L^\infty_tH_x^s} +\|\rho\|_{L^\infty_tH_x^s}+\| h \|_{L^\infty_tH_x^{s_0+1}}+\| \bw\|_{L^\infty_tH_x^{s_0}}\leq \epsilon_2,
\end{equation}
hold.
\end{proposition}
\begin{proof}
If we take $k=s_0-1$ in Proposition \ref{r6}, we then obtain
\begin{equation}\label{epp0}
\begin{split}
\|d(\rho+\frac{1}{\gamma}h), d\bv_{+} \|_{L^2_{[-2,2]} \dot{B}^{s_0-2}_{\infty, 2}}
\lesssim  & \| \rho, \bv, h\|_{L^\infty_{[-2,2]} H^s}+\| h\|_{L^2_{[-2,2]} H^{\frac{5}{2}+}}
\\
& +\| \bw \|_{L^2_{[-2,2]} H^{\frac{3}{2}+}}+\| \bw \|_{L^\infty_{[-2,2]} H^{\frac{1}{2}+} },
\end{split} 		
\end{equation}
Using Proposition \ref{r6} again, for $\delta_* \in [0,s-2)$, we shall get
\begin{equation}\label{epp1}
	\begin{split}
		\|d(\rho+\frac{1}{\gamma}h), d\bv_{+} \|_{L^2_{[-2,2]}C^{\delta_*}_x}
		\lesssim  & \| \rho, \bv, h\|_{L^\infty_{[-2,2]} H^s}+\| h\|_{L^2_{[-2,2]} H^{\frac{5}{2}+}}
		\\
		& +\| \bw \|_{L^2_{[-2,2]} H^{\frac{3}{2}+}},
	\end{split} 		
\end{equation}
By Sobolev's imbeddings, we have
\begin{equation}\label{epp2}
	\begin{split}
		\|\nabla h\|_{L^2_{[-2,2]} C^{\delta_*}_x}
		\lesssim  & \| h\|_{L^2_{[-2,2]} H^{\frac{5}{2}+\delta+}}.
	\end{split} 		
\end{equation}
Using \eqref{epp2} and \eqref{fc0}, it follows
\begin{equation}\label{epp3}
	\begin{split}
		\|\partial_t h\|_{L^2_{[-2,2]} C^{\delta_*}_x}
		\lesssim  & \| h\|_{L^2_{[-2,2]} H_x^{\frac{5}{2}+\delta_*+}}+ \| \bv \|_{L^\infty_{[-2,2]} H_x^{s}}.
	\end{split} 		
\end{equation}
Similarly, we also get
\begin{equation}\label{epp4}
	\begin{split}
		\|\nabla \bv_{-}\|_{L^2_{[-2,2]} C^{\delta_*}_x}
		\lesssim  & \| \bw\|_{L^2_{[-2,2]} H^{\frac{3}{2}+\delta_*+}}.
	\end{split} 		
\end{equation}
Using \eqref{epp4} and \eqref{fh04}, it follows
\begin{equation}\label{epp5}
	\begin{split}
		\|\partial_t \bv_{-}\|_{L^2_{[-2,2]} C^{\delta_*}_x}
		\lesssim  & \|\bv, \rho, h\|_{L^\infty_{[-2,2]}H^{s}_x}+ \|\bw\|_{L^2_{[-2,2]}H^{\frac32+\delta_*+}_x} + \|h\|_{L^2_{[-2,2]}H^{\frac52+\delta_*+}_x}.
	\end{split} 		
\end{equation}
As a result, using \eqref{epp1}-\eqref{epp5}, we get
\begin{equation}\label{epp6}
	\begin{split}
		\|d\rho, dh, d\bv \|_{L^2_{[-2,2]} C^{\delta_*}_x}
		\lesssim  & \| \rho, \bv, h\|_{L^\infty_{[-2,2]} H_x^s}+\| h\|_{L^2_{[-2,2]} H_x^{\frac{5}{2}+\delta_*+}}
		+\| \bw \|_{L^2_{[-2,2]} H_x^{\frac{3}{2}+\delta_*+}}
		\\
		\lesssim  & \| \rho, \bv, h\|_{L^\infty_{[-2,2]} H_x^s}+\| h\|_{L^2_{[-2,2]} H_x^{\frac{5}{2}+\delta_*+(s_0-\frac32-\delta_*)}}
		\\
		& \ +\| \bw \|_{L^2_{[-2,2]} H_x^{\frac{3}{2}+\delta_*+(s_0-\frac32-\delta_*)}}
		\\
		\lesssim  & \| \rho, \bv, h\|_{L^\infty_{[-2,2]} H_x^s}+\| h\|_{L^2_{[-2,2]} H_x^{1+s_0}}    +\| \bw \|_{L^2_{[-2,2]} H_x^{s_0}}.
	\end{split} 		
\end{equation}
Above, we use the fact that $\delta*<s-2\leq \frac12$ and $s_0-\frac32-\delta_*>s_0-2>0$.

Similarly, using \eqref{epp0} and Sobolev's imbeddings, we have
\begin{equation}\label{epp7}
	\begin{split}
		\|d\rho,dh, d\bv \|_{L^2_{[-2,2]} \dot{B}^{s_0-2}_{\infty, 2}}
		\lesssim  & \| \rho, \bv, h\|_{L^\infty_{[-2,2]} H_x^s}+\| h\|_{L^2_{[-2,2]} H_x^{s_0+\frac{1}{2}}}   +\| \bw \|_{L^2_{[-2,2]} H_x^{s_0-\frac{1}{2}}}
		\\
		\lesssim  & \| \rho, \bv, h\|_{L^\infty_{[-2,2]} H_x^s}+\| h\|_{L^2_{[-2,2]} H_x^{1+s_0}}   +\| \bw \|_{L^2_{[-2,2]} H_x^{s_0}}.
	\end{split} 		
\end{equation}
By using Theorem \ref{ve} and $(\bv,\rho,h,\bw)$ \eqref{401}-\eqref{403}, it follows that
\begin{equation}\label{epp8}
	\begin{split}
		& \|\rho, \bv \|_{L^\infty_{[-2,2]} H_x^s}+ \|h \|_{L^\infty_{[-2,2]} H_x^{s_0+1}}+ \| \bw \|_{L^\infty_{[-2,2]} H_x^{s_0}}
		\\
		\leq  & C (\epsilon_3+ \epsilon^2_3)\exp \left\{ C\|d\rho,dh, d\bv \|_{L^2_{[-2,2]} \dot{B}^{s_0-2}_{\infty, 2}} \exp(C\|d\rho,dh, d\bv \|_{L^2_{[-2,2]} \dot{B}^{s_0-2}_{\infty, 2}}) \right\}
		\\
		\leq & C\epsilon_3 \mathrm{e}^{C\epsilon_1 \mathrm{e}^{C\epsilon_1} }
		\\
		\leq & \epsilon_2.
	\end{split} 		
\end{equation}
Therefore, by \eqref{epp8} and \eqref{epp7}, it turns out
\begin{equation}\label{epp9}
	\begin{split}
		\|d\rho,dh, d\bv \|_{L^2_{[-2,2]} \dot{B}^{s_0-2}_{\infty, 2}}
		\leq  & C\| \rho, \bv, h\|_{L^\infty_{[-2,2]} H^s}+C\| h\|_{L^2_{[-2,2]} H^{\frac{5}{2}+}}
		\\
		& +C\| \bw \|_{L^2_{[-2,2]} H^{\frac{3}{2}+}}+C\| \bw \|_{L^\infty_{[-2,2]} H^{\frac{1}{2}+} }
		\\
		\leq & C\epsilon_3 \mathrm{e}^{C\epsilon_1 \mathrm{e}^{C\epsilon_1} }
		\\
		\leq & \epsilon_2.
	\end{split} 		
\end{equation}
Using \eqref{epp8} and \eqref{epp6} we prove that
\begin{equation}\label{epp10}
	\begin{split}
		\|d\rho, dh, d\bv \|_{L^2_{[-2,2]} C^{\delta_*}_x}
		\lesssim   \epsilon_2.
	\end{split} 		
\end{equation}
Gathering \eqref{epp1}, \eqref{epp8}, \eqref{epp9} and \eqref{epp10}, we have finished this proof.
\end{proof}
\section{The proof of Proposition \ref{r5}}\label{Ap}
In this part, we will give a proof of Proposition \ref{r5} by using Smith-Tataru's idea in \cite{ST}. The method relies on a wave-packet construction for the nontruncated metric $\mathbf{g}$. We mention that the use of wave-packets is introduced by Smith \cite{Sm}, and developed by Wolff \cite{Wo} and Smith-Tataru \cite{ST}. Following this idea, so we can represent approximate solutions to the linear equation as a square summable superposition of wave packets, and the wave packets are localized in phase space. As a result, we can obtain Strichartz estimates of linear wave equations with sufficient regularity of the foliations. Therefore, it's quite different with Wang's paper \cite{WQEuler} in establishing Strichartz estimates, where Wang \cite{WQEuler} use the vector field approach and work with foliations on null cones.

In this part, we set $(\bv,\rho, h, \bw) \in \mathcal{{H}}$ and $G(\bv,\rho,h)\leq 2 \epsilon_1$, where $\mathcal{{H}}$ and $G$ are defined in \eqref{401}-\eqref{403} and \eqref{500} respectively. Then $\|d \mathbf{g}\|_{L^2_t L^\infty_x}:=\|d \mathbf{g}\|_{L^2_{[-2,2]} L^\infty_x} \leq \epsilon_0$ and $\|d \mathbf{g}\|_{L^2_t C^\delta_x}:=\|d \mathbf{g}\|_{L^2_{[-2,2]} C^\delta_x} \leq \epsilon_0$, where $\epsilon_0$ is stated in \eqref{a0}.

We divide the proof of Proposition \ref{r5} into several subparts. The first step is to reduce the problem in a phase space.
\subsection{A reduction to the phase decomposition}
Let $P_\lambda$ be a Littlewood-Paley operator with frequency $\{\xi\in\mathbb{R}^3: \frac{\lambda}{8}\leq |\xi|\leq 8\lambda\}$, where $\lambda\geq 1$. Given such a frequency scale $\lambda$, we localize the coefficients of the acoustic metric in frequency. More precisely, we define the smoothed coefficients
\begin{equation*}
	\mathbf{g}_{\lambda}= \sum_{\lambda'<\lambda}P_{\lambda'} \mathbf{g}.
\end{equation*}
We can state the following result:
\begin{proposition}\label{AA1}
	Suppose that $(\bv,\rho, h, \bw) \in \mathcal{{H}}$ and $G(\bv,\rho,h)\leq 2 \epsilon_1$. Suppose $f$ satisfy
	\begin{equation}\label{linearA}
		\begin{cases}
			\square_{\mathbf{g}} f=0, \quad (t,\bx)\in [-2,2]\times \mathbb{R}^3,\\
			f|_{t=t_0}=f_0, \quad \partial_t f|_{t=t_0}=f_1.
		\end{cases}
	\end{equation}
	Then for each $(f_0,f_1) \in H^1 \times L^2$ there exists a function $f_{\lambda} \in C^\infty([-2,2]\times \mathbb{R}^3)$, with
	\begin{equation*}
		\mathrm{supp} \widehat{f_\lambda(t,\cdot)} \subseteq \{ \xi: \frac{\lambda}{8} \leq |\xi| \leq 8\lambda \},
	\end{equation*}
	such that
	\begin{equation}\label{Yee}
		\begin{cases}
			& \| \square_{\mathbf{g}_\lambda} f_{\lambda} \|_{L^1_{[-2,2]} L^2_x} \lesssim \epsilon_0 (\| f_0\|_{H^1}+\| f_1 \|_{L^2} ),
			\\
			& f_\lambda(t_0)=P_\lambda f_0, \quad \partial_t f_{\lambda} (t_0)=P_{\lambda} f_1.
		\end{cases}
	\end{equation}
	Additionally, if $r>1$, then the following Strichartz estimates holds:
	\begin{equation}\label{Ase}
		\| f_{\lambda} \|_{L^2_{[-2,2]} L^\infty_x} \lesssim \epsilon_0^{-\frac{1}{2}} \lambda^{r-1} ( \| f_0 \|_{H^1} + \| f_1 \|_{L^2} ).
	\end{equation}
\end{proposition}
\begin{remark}\label{rel}
	In fact, Proposition \ref{AA1} tells us we can find a good approximate solution $f_\lambda$ for the problem
	\begin{equation*}
		\begin{cases}
			\square_{\mathbf{g}_\lambda} f_\lambda=0, \quad (t,x)\in [-2,2]\times \mathbb{R}^3,\\
			(f_\lambda, \partial_t f_\lambda)|_{t=t_0}=(P_\lambda f_0, P_\lambda f_1).
		\end{cases}
	\end{equation*}
	For $\epsilon_0 \lambda\leq 1$, the result in Proposition \ref{AA1} is almost trivial. Indeed, we can set $f_\lambda=P_\lambda f$, where $f$ is the exact solution of
	\begin{equation}\label{linearD}
		\begin{cases}
			\square_{\mathbf{g}_\lambda} f=0, \quad (t,x)\in [-2,2]\times \mathbb{R}^3,\\
			(f, \partial_t f)|_{t=t_0}=(P_\lambda f_0, P_\lambda f_1).
		\end{cases}
	\end{equation}
	Using energy estimates for \eqref{linearD}, we can obtain
	\begin{equation}\label{eete}
		\|df\|_{L^\infty_{[-2,2]} L^2_x} \lesssim \| P_\lambda f_0\|_{H^1}+\|P_\lambda f_1\|_{L^2}.
	\end{equation}
	Moreover, for $\mathbf{g}^{00}_\lambda=-1$, so we have
	\begin{equation*}
		\begin{split}
			\| \square_{\mathbf{g}_\lambda} f_\lambda \|_{L^2_{[-2,2]} L^2_x} 
			\lesssim & 
			\| [ \mathbf{g}^{\alpha i}_\lambda, \partial_\alpha P_\lambda ] \partial_i f \|_{ L^2_{[-2,2]} L^2_x } 
			+ \| P_\lambda (\partial_\alpha \mathbf{g}^{\alpha i}_\lambda ) \partial_i f \|_{ L^2_{[-2,2]} L^2_x }
			\\
			\lesssim & \| d \mathbf{g}_\lambda \|_{L^2_{[-2,2]} L^\infty_x} \| d {f} \|_{L^\infty_{[-2,2]} L^2_x}
			\\
			\lesssim & \epsilon_0 ( \| P_\lambda f_0\|_{H^1}+\|P_\lambda f_1\|_{L^2} ).
		\end{split}
	\end{equation*}
	The Strichartz estimate \eqref{Ase} follows from Sobolev imbeddings and \eqref{eete}. Hence, we will restrict to establishing Proposition \ref{AA1} in the case that
	\begin{equation*}
		\epsilon_0 \lambda \gg 1.
	\end{equation*}
\end{remark}
\subsection{The proof of Proposition \ref{AA1}}
%\subsubsection{A normalized wave packets}\label{cwp}
To prove Proposition \ref{AA1}, we construct an approximate solutions to \eqref{linearA} by using wave packets, and a wave packet has a finer spatial localization than a plane wave solutions. We divide the proof into two parts, namely \eqref{Yee} and \eqref{Ase}. The conclusion \eqref{Yee} is proved in Proposition \ref{szy} and Proposition \ref{szi}, and the conclusion \eqref{Ase} is obtained by Proposition \eqref{szt}.

We introduce a spatially localized mollifier $T_\lambda$ by
\begin{equation*}
	T_\lambda f = \chi_\lambda * f, \quad \chi_\lambda=\lambda^3 \chi(\lambda^{-1} y),
\end{equation*}
where $\chi \in C^\infty_0(\mathbb{R}^3)$ is supported in the ball $|x| \leq \frac{1}{32}$, and has integral $1$. By choosing $\chi$ appropriately, any function $u$ with frequency support contained in $\{\xi:|\xi|\leq 4\lambda\}$ can be factored $u=T_\lambda \widetilde{u}$, where $\| \widetilde{u} \|_{L^2_x} \approx \| u \|_{L^2_x}$.

Let us introduce the wave packets which is related to $\mathbf{g}$ and the null hypersurfaces $\Sigma$.
\begin{definition}[\cite{ST}]
	Let the hypersurface $\Sigma_{\omega,r}$ and the geodesic $\gamma$ be defined in Section \ref{sec6}. A normalized wave packet around $\gamma$ is a function $f$ of the form
	\begin{equation*}
		f=\epsilon_0^{\frac12} \lambda^{-1} T_\lambda(u \zeta),
	\end{equation*}
	where
	\begin{equation*}
		u(t,x)=\delta(x_{\omega}-\phi_{\omega,r}(t,x'_\omega)), \quad \zeta=\zeta_0( (\epsilon_0 \lambda)^{\frac12}(x'_\omega-\gamma_\omega(t))  ).
	\end{equation*}
	Here, $\zeta_0$ is a smooth function supported in the set $|x'| \leq 1$, with uniform bounds on its derivatives $|\partial^\alpha_{x'} \zeta_0(x')| \leq c_\alpha$.
\end{definition}
We give two notations here. We denote $L(u,\zeta)$ to denote a translation invariant bilinear operator of the form
\begin{equation*}
	L(u,\zeta)(x)=\int K(y,z)u(x+y)\zeta(x+z)dydz,
\end{equation*}
where $K(y,z)$ is a finite measure. If $X$ is a Sobolev spaces, we then denote $X_a$ the same space but with the norm obtained by dimensionless rescaling by $a$,
\begin{equation*}
	\| u \|_{X_a}=\| u(a \cdot)\|_{X}.
\end{equation*}
Since $2(s_0-1)>2$, then for $a<1$ we have $\| u \|_{H^{s_0-1}_a(\mathbb{R}^{2})} \lesssim \| u \|_{H^{s_0-1}(\mathbb{R}^{2})}$. In the following, let us introduce what we can get when we take the operator $\square_{\mathbf{g}_\lambda}$ on wave packets.
\subsubsection{A normalized wave packet}
\begin{proposition}\label{np}
	Let $f$ be a normalized packet. Then there is another normalized wave packet $\tilde{f}$, and functions $\phi_m(t,x'_\omega), m=0,1,2$, so that
	\begin{equation}\label{np1}
		\square_{\mathbf{g}_\lambda} P_\lambda f= L(d \mathbf{g}, d \tilde{P}_\lambda \tilde{f})+ \epsilon_0^{\frac12}\lambda^{-1}P_{\lambda}T_{\lambda}\sum_{m=0,1,2}
		\psi_m\delta^{(m)}(x'_\omega-\phi_{\omega,r}),
	\end{equation}
	where the functions $\psi_m=\psi_m(t,x'_\omega)$ satisfy the scaled Sobolev estimates
	\begin{equation}\label{np2}
		\| \psi_m\|_{L^2_t H^{s_0-1}_{a,x'_\omega}} \lesssim \epsilon_0 \lambda^{1-m}, \quad m=0,1,2, \quad a=(\epsilon_0 \lambda)^{-\frac12}.
	\end{equation}
\end{proposition}
\begin{proof}
%\medskip\begin{proof}[Proof of Proposition \ref{np}]
	For brevity, we consider the case $\omega=(0,0,1)$. Then $x_\omega=x_3$, and $\bx'_\omega=\bx'$. We write
	\begin{equation}\label{js0}
		\square_{\mathbf{g}_\lambda} P_\lambda f
		= \lambda^{-1} ( [\square_{\mathbf{g}_\lambda}, P_\lambda T_\lambda]+P_\lambda T_\lambda \square_{\mathbf{g}_\lambda} )(uh).
	\end{equation}
	For the first term in \eqref{js0}, noting $\mathbf{g}\lambda$ supported at frequency $\leq \frac \lambda8$, then we can write
	\begin{equation*}
		[\square_{\mathbf{g}_\lambda}, P_\lambda T_\lambda]=[\square_{\mathbf{g}_\lambda}, P_\lambda T_\lambda]\tilde{P}_\lambda \tilde{T}_\lambda
	\end{equation*}
	for some multipliers $\tilde{P}_\lambda, \tilde{T}_\lambda$ which have the same properties as $P_\lambda, T_\lambda$. Therefore, by using the kernel bounds for $P_\lambda T_\lambda$, we conclude that
	\begin{equation*}
		[\square_{\mathbf{g}_\lambda}, P_\lambda T_\lambda]f=L(d\mathbf{g}, df).
	\end{equation*}
	For the second term in \eqref{js0}, we use the Leibniz rule
	\begin{equation}\label{js1}
		\square_{\mathbf{g}_\lambda}(u\zeta)=\zeta \square_{\mathbf{g}_\lambda} u+(\mathbf{g}^{\alpha \beta}_\lambda+\mathbf{g}^{\beta \alpha}_\lambda)\partial_\alpha u \partial_\beta \zeta + u \square_{\mathbf{g}_\lambda} \zeta.
	\end{equation}
	We let $\nu$ denote the conormal vector field along $\Sigma$, $\nu=dx_3-d\phi(t,\bx')$. In the following, we take the Greek indices $0 \leq \alpha, \beta \leq 2$.

	For the first term in \eqref{js1}, by calculation, we have
	\begin{equation*}
		\begin{split}
			\mathbf{g}_{\lambda}^{\alpha \beta} \partial_{\alpha \beta} u =& \mathbf{g}_{\lambda}^{\alpha \beta}(t,x',\phi)\nu_\alpha \nu_\beta \delta^{(2)}_{x_3-\phi}
			-2 (\partial_3 \mathbf{g}_{\lambda}^{\alpha \beta})(t,\bx',\phi) \nu_\alpha \nu_\beta \delta^{(1)}_{x_3-\phi}
			\\
			&+(\partial^2_3 \mathbf{g}_{\lambda}^{\alpha \beta})(t,\bx',\phi) \nu_\alpha \nu_\beta \delta^{(0)}_{x_3-\phi}- \mathbf{g}_{\lambda}^{\alpha \beta})(t,\bx',\phi) \partial_{\alpha \beta}\phi \delta^{(1)}_{x_3-\phi}
			\\
			&+ \partial_3\mathbf{g}_{\lambda}^{\alpha \beta})(t,\bx',\phi) \partial_{\alpha \beta}\phi \delta^{(0)}_{x_3-\phi}.
		\end{split}
	\end{equation*}
	Here, $\delta^{(m)}_{x_3-\phi}=(\partial^m \delta)(x_3-\phi) $. By Leibniz rule, we can take
	\begin{equation*}
		\begin{split}
			\psi_0&=\zeta \big\{ (\partial^2_3 \mathbf{g}_{\lambda}^{\alpha \beta})(t,\bx',\phi)\nu_\alpha \nu_\beta+ (\partial_3 \mathbf{g}_{\lambda}^{\alpha \beta})(t,\bx',\phi)\partial_{\alpha \beta}\phi  \big\},
			\\
			\psi_1&=\zeta \big\{ 2(\partial_3 \mathbf{g}_{\lambda}^{\alpha \beta})(t,\bx',\phi)\nu_\alpha \nu_\beta-  \mathbf{g}_{\lambda}^{\alpha \beta}(t,\bx',\phi)\partial_{\alpha \beta}\phi \big\},
			\\
			\psi_2&=\zeta( \mathbf{g}_{\lambda}^{\alpha \beta}-\mathbf{g}^{\alpha \beta})\nu_\alpha \nu_\beta,
		\end{split}
	\end{equation*}
	By using \eqref{G}, Proposition \ref{r1}, and Corollary \ref{vte}, we can conclude that this settings of $\psi_0, \psi_1,$ and $\psi_2$ satisfy the estimates \eqref{np2}. For the second term in \eqref{js0}, we have
	\begin{equation*}
		\begin{split}
			(\mathbf{g}^{\alpha \beta}_\lambda+\mathbf{g}^{\beta \alpha}_\lambda)\partial_\alpha u \partial_\beta \zeta=& \frac12\nu_\alpha (\mathbf{g}^{\alpha \beta}_\lambda+\mathbf{g}^{\beta \alpha}_\lambda)(t,\bx',\phi)\partial_\beta \zeta \delta^{(1)}_{x_3-\phi}
			\\
			&- \frac12\nu_\alpha \partial_3 (\mathbf{g}^{\alpha \beta}_\lambda+\mathbf{g}^{\beta \alpha}_\lambda)(t,\bx',\phi)\partial_\beta \zeta \delta^{(0)}_{x_3-\phi}.
		\end{split}
	\end{equation*}
	We then take
	\begin{equation*}
		\psi_0= \frac12\nu_\alpha \partial_3 (\mathbf{g}^{\alpha \beta}_\lambda+\mathbf{g}^{\beta \alpha}_\lambda)(t,\bx',\phi)\partial_\beta \zeta, \quad \psi_1= \frac12\nu_\alpha (\mathbf{g}^{\alpha \beta}_\lambda+\mathbf{g}^{\beta \alpha}_\lambda)(t,\bx',\phi)\partial_\beta \zeta.
	\end{equation*}
	By using \eqref{G}, Proposition \ref{r1}, and Corollary \ref{vte}, we can conclude that this settings of $\psi_0$ and $\psi_1$ satisfy the estimates \eqref{np2}.
	For the third term in \eqref{js0}, we can take
	\begin{equation*}
		\psi_0=\mathbf{g}^{\alpha \beta}_\lambda(t,\bx',\phi)\partial_{\alpha \beta}\zeta.
	\end{equation*}
	By using \eqref{G}, Proposition \ref{r1}, and Corollary \ref{vte}, we can conclude that this settings of $\psi_0$ satisfies the estimates \eqref{np2}.
\end{proof}
Therefore, a single wave packet is not sufficient for us to construct approximate solutions to a linear wave equation, so we need to consider the superposition of wave packets.
\subsubsection{Superpositions of wave packets}\label{aFs}
Firstly, let us introduce some notations. The index $\omega$ stands for the initial orientation of the wave packet at $t=-2$, which varies over a maximal collection of approximately $\epsilon_0^{-1}\lambda$ unit vectors separated by at least $\epsilon_0^{\frac12}\lambda^{-\frac12}$. For each $\omega$, we have the orthonormal coordinate system $(x_\omega, \bx'_\omega)$ of $\mathbb{R}^2$, where $x_\omega=x \cdot \omega$, and $\bx'_\omega$ are projective along $\omega$.

We denote $\mathbb{R}^3$ by a parallel tiling of rectangles, with length $(8\lambda)^{-1}$ in the $x_{\omega}$ direction, and $(4\epsilon_0 \lambda)^{-\frac12}$ in the other directions $\bx'_{\omega}$. The index $j$ corresponds to a counting of the rectangles in this decomposition. Let $R_{\omega,j}$ denote the collection of the doubles of these rectangles, and $\Sigma_{\omega,j}$ denote the element of the $\Sigma_{\omega}$ ($\Sigma_{\omega}=\cup_{r} \Sigma_{\omega,r}$) foliation upon which $R_{\omega,j}$ is centered. Let $\gamma_{\omega,j}$ denote the null geodesic contained in $\Sigma_{\omega,j}$ which passes through the center of $R_{\omega,j}$ at time $t=-2$.

We let $T_{\omega,j}$ be the set
\begin{equation}\label{twj}
	T_{\omega,j}=\Sigma_{\omega,j} \cap \{ |\bx'_{\omega}-\gamma_{\omega,j}| \leq (\epsilon_0 \lambda)^{-\frac12}\} .
\end{equation}
By \eqref{600} and \eqref{606}, then the estimate
\begin{equation*}
	|dr_{\theta}-(\theta \cdot d \bx-dt)| \lesssim \epsilon_1,
\end{equation*}
holds pointwise uniformly on $[-2,2]\times \mathbb{R}^3$. This also implies that
\begin{equation}\label{pr0}
	| \phi_{\theta,r}(t,\bx'_{\theta})-\phi_{\theta,r'}(t,\bx'_{\theta})-(r-r')| \lesssim \epsilon_1|r-r'|.
\end{equation}
On the other hand, if we set $\bar{\rho}(t)=\| d \mathbf{g} \|_{C^\delta_x}$, then \eqref{502} tells us
\begin{equation}\label{pr1}
	\| d^2_{\bx'_{\omega}}\phi_{\omega,r}(t,\bx'_{\omega})-d^2_{\bx'_{\omega}}\phi_{\omega,r'}(t,\bx'_{\omega})\|_{L^\infty_{\bx'_{\omega}}} \lesssim \epsilon_2+ \bar{\rho}(t).
\end{equation}
By using \eqref{pr0} and \eqref{pr1}, we get
\begin{equation}\label{pr2}
	\| d_{\bx'_{\omega}}\phi_{\omega,r}(t,\bx'_{\omega})-d_{\bx'_{\omega}}\phi_{\omega,r'}(t,\bx'_{\omega})\|_{L^\infty_{\bx'_{\omega}}} \lesssim (\epsilon_2+ \bar{\rho}(t))^{\frac12}|r-r'|^{\frac12}.
\end{equation}
For $dx_{\omega}-d\phi_{\omega,r}$ is null and also $d \mathbf{g} \leq \bar{\rho}(t)$, this also implies H\"older-$\frac12$ bounds on $d\phi_{\omega,r}$. So we suppose that $(t,x)\in \Sigma_{\omega,r}$ and $(t,y)\in \Sigma_{\omega,r'}$, that $|\bx'_{\omega}-\by'_{\omega}| \leq 2(\epsilon_0 \lambda)^{-\frac12}$, and that $|r-r'|\leq 2 \lambda^{-1}$. Using \eqref{601}, we can obtain
\begin{equation*}
	|l_{\omega}(t,x)-l_{\omega}(t,y)| \lesssim \epsilon_0^{\frac12}\lambda^{-\frac12}+\epsilon_0^{-\frac12}\bar{\rho}(t)\lambda^{-\frac12}.
\end{equation*}
Due to $\dot{\gamma}_{\omega}=l_{\omega}$, and $\|\bar{\rho}\|_{L^2_t} \lesssim \epsilon_0$, so any geodesic in $\Sigma_{\omega}$ which intersects a slab $T_{\omega,j}$ should be contained in the similar slab of half the scale.

We are ready to introduce a lemma about a superpositions of wave packets from a certain fixed time.
\begin{Lemma}\label{SWP}
	Let $0<\mu < \delta$. Let a scalar function $\bar{v}(t,\bx)$ be formulated by
	\begin{equation}\label{swp}
		\bar{v}(t,\bx)=\epsilon_0^{\frac12}P_{\lambda} \sum_{\omega,j}T_\lambda(\psi^{\omega,j}\delta_{\bx_{\omega}-\phi_{\omega,j}(t,\bx'_{\omega})}).
	\end{equation}
	Set $\bar{\rho}(t)=\| d \mathbf{g} \|_{C^\delta_x}$ and $a=(\epsilon_0 \lambda)^{-\frac12}$. Then we have
	\begin{equation}\label{swp0}
		\| \bar{v}(t) \|^2_{L^2_x}\lesssim \sum_{\omega,j}\| \psi^{\omega,j}\|^2_{H^{1+\mu}_a}, \qquad \qquad \ \ \text{if} \ \ \bar{\rho}(t) \leq \epsilon_0,
	\end{equation}
	and
	\begin{equation}\label{swp1}
		\| \bar{v}(t) \|^2_{L^2_x}\lesssim \epsilon_0^{-1} \bar{\rho}(t)  \sum_{\omega,j}\| \psi^{\omega,j}\|^2_{H^{1+\mu}_a},  \qquad \text{if} \ \ \bar{\rho}(t) \geq \epsilon_0 .
	\end{equation}
	\begin{proof}
		This proof follows ideas of Smith-Tataru \cite{ST} (Lemma 8.5 and Lemma 8.6 on page 340-345). Therefore we omit the details here.
	\end{proof}
	\begin{remark}
		By \eqref{swp0} and \eqref{swp1}, we get
		\begin{equation}\label{swp3}
			\| \bar{v}(t) \|^2_{L^2_x}\lesssim \left(1+\epsilon_0^{-1} \bar{\rho}(t)\right)  \sum_{\omega,j}\| \psi^{\omega,j}\|^2_{H^{1+\mu}_a}.
		\end{equation}
	\end{remark}
\end{Lemma}
\begin{proposition}[\cite{ST}]\label{szy}
	Let $f=\sum_{\omega,j}a_{\omega,j}f^{\omega,j}$, where $f^{\omega,j}$ are normalized wave packets supported in $T_{\omega,j}$. Then we have
	\begin{equation}\label{ese}
		\| d P_\lambda f\|_{L^\infty_t L^2_x} \lesssim (\sum_{\omega,j} a^2_{\omega,j})^{\frac12},
	\end{equation}
	and
	\begin{equation}\label{ese1}
		\| \square_{\mathbf{g}_{\lambda}} P_\lambda f \|_{L^1_t L^2_x} \lesssim \epsilon_0 (\sum_{\omega,j} a^2_{\omega,j})^{\frac12}.
	\end{equation}
\end{proposition}
\begin{proof}
	We first prove a weaker estimate comparing with \eqref{ese}
	\begin{equation}\label{ese2}
		\| d P_\lambda f\|_{L^2_t L^2_x} \lesssim (\sum_{\omega,j} a^2_{\omega,j})^{\frac12}.
	\end{equation}
	By using \eqref{swp3} and replacing $P_\lambda$ by $\lambda^{-1}\partial P_\lambda$, and $\psi^{\omega,j}=a_{\omega,j}\zeta^{\omega,j}$, we have
	\begin{equation*}%\label{ese3}
		\| \partial P_\lambda f(t)\|^2_{L^2_x} \lesssim (1+\epsilon_0^{-1}\bar{\rho}(t)) \sum_{\omega,j} a^2_{\omega,j}.
	\end{equation*}
	Due to the fact $\|\bar{\rho}\|_{L^2_t} \lesssim \epsilon_0$, we can see that
	\begin{equation}\label{ese3}
		\| \partial P_\lambda f\|^2_{L^2_t L^2_x} \lesssim \sum_{\omega,j} a^2_{\omega,j} .
	\end{equation}
	We also need to get the similar estimate for the time derivatives. We can calculate
	\begin{equation*}
		\partial_t \zeta = \dot{\gamma}(t) (\epsilon_0 \lambda)^{\frac12} \tilde{\zeta}, \quad \partial_t \delta(x_\omega-\phi_{\omega,j})=\partial_t \phi_{\omega,j} \delta^{(1)}(x_\omega-\phi_{\omega,j}).
	\end{equation*}
	For $\dot{\gamma} \in L^\infty_t$ and $\partial_t \phi_{\omega,j} \in L^\infty_t$, then we have
	\begin{equation}\label{ese4}
		\| \partial_t P_\lambda f\|^2_{L^2_t L^2_x} \lesssim \sum_{\omega,j} a^2_{\omega,j} .
	\end{equation}
	Combining \eqref{ese3} and \eqref{ese4}, we have proved \eqref{ese2}. To prove \eqref{ese1}, we use the formula \eqref{np1}. Considering the right hand of \eqref{np1}, by using \eqref{ese2}, we can bound the first term by
	\begin{equation}
		\| L(d\mathbf{g}, d \tilde{P}_{\lambda} \tilde{f}) \|_{L^1_tL^2_x} \lesssim \|d\mathbf{g}\|_{L^2_tL^\infty_x} \|d \tilde{P}_{\lambda} \tilde{f} \|_{L^2_tL^2_x} \lesssim \epsilon_0 (\sum_{\omega,j} a^2_{\omega,j})^{\frac12}.
	\end{equation}
	It only remains for us to estimate the second right term on \eqref{np1}. If we set
	\begin{equation*}
		\Phi=\epsilon_0^{\frac12} P_\lambda T_\lambda \left(\sum_{\omega,j}a_{\omega,j}\cdot \sum_{m=0,1,2}\psi^{\omega,j}_m \delta^{(m)}_{x_{\omega}-\phi_{\omega,j}} \right),
	\end{equation*}
	and use \eqref{np2}, then we have
	\begin{equation*}
		\begin{split}
\| \Phi \|^2_{L^2_x}=&(1+\bar{\rho}(t)\epsilon_0^{-1}) \sum_{\omega,j}a^2_{\omega,j} \sum_{m=0,1,2} \lambda^{m-1}\|\psi^{\omega,j}_m (t) \|^2_{H^{1+\mu}_a}
\\
\lesssim  &(1+\bar{\rho}(t)\epsilon_0^{-1}) \epsilon_0^2  \sum_{\omega,j}a^2_{\omega,j},
\end{split}
	\end{equation*}
	we therefore get
	\begin{equation*}
		\begin{split}
			 \| \Phi \|_{L^1_t L^2_x}
			\lesssim & \epsilon_0 (\sum_{\omega,j} a^2_{\omega,j})^{\frac12} \left(\int^2_{-2}[1+\bar{\rho}(t)\epsilon_0^{-1}]^{\frac12}dt \right)
			\\
			\lesssim & \epsilon_0 (\sum_{\omega,j} a^2_{\omega,j})^{\frac12} \left(\int^2_{-2}[1+\bar{\rho}(t)\epsilon_0^{-1}]dt \right)^{\frac12}
\\
			\lesssim & \epsilon_0 (\sum_{\omega,j} a^2_{\omega,j})^{\frac12} .
		\end{split}
	\end{equation*}
	Hence, the second right term on \eqref{np1} can be estimated by $\epsilon_0 (\sum_{\omega,j} a^2_{\omega,j})^{\frac12}$. So we have proved \eqref{ese1}. Using \eqref{ese2} and \eqref{ese1}, and classical energy estimates for linear wave equation, we obtain \eqref{ese}.
\end{proof}
\subsubsection{Matching the initial data}
Although we have constructed the approximate solutions using superpositions of normalized wave packets, we also need to complete this construction, i.e. matching the initial data for this type of solutions. Since the metric\footnote{Please see \eqref{boldg}.} $\mathbf{g}$ equals to the Minkowski metric for $t \in [-2,-\frac32]$, so it's natural for us to work with wave packets near $t=-2$ for the Minkowski wave operator. We refer to the construction from \cite[Proposition 8.7, page 346]{ST}
%(Proposition 8.7 on page 346)
or Smith \cite[page 815, Section 4]{Sm}.
%(pages 815, Section 4).
\begin{proposition}[\cite{ST}]\label{szi}
	Given any initial data $(f_0,f_1) \in H^1 \times L^2$, there exists a function of the form
	\begin{equation*}
		f=\sum_{\omega,j}a_{\omega,j}f^{\omega,j},
	\end{equation*}
	where the function $f^{\omega,j}$ are normalized wave packets, such that
	\begin{equation*}
		P_\lambda f(-2)=P_\lambda f_0, \quad \partial_t P_\lambda f(-2)=P_\lambda f_1.
	\end{equation*}
	Furthermore,
	\begin{equation*}
		\sum_{\omega,j}a_{\omega,j}^2 \lesssim \| f_0 \|^2_{H^1}+ \| f_1 \|^2_{L^2}.
	\end{equation*}
\end{proposition}

\begin{remark}
Considering the definition of a wave packet, together with the regularity of the $\Sigma_{\omega,r}$, then we can use the quantities $\phi_{\omega,r}$ and $\gamma_{\omega,r}$ to keep the type of $f$ for $t\in[-2,2]$.

If $t_0=-2$ in Proposition \eqref{AA1}, then the function $f$ is the approximate solution matching the initial data by using Proposition \ref{szi}.

If $t_0\in(-2,2)$, then we can use an iteration process as follows. Basically, we hope to find the exact solution $\Phi$ to the Cauchy problem
\begin{equation*}
  \begin{cases}
  & \square_{\mathbf{g}_{\lambda}} \Phi=0, \qquad (t,\bx) \in [-2,2] \times \mathbb{R}^3,
  \\
  & (\Phi, \partial_t \Phi)|_{t=t_0}=(\Phi_0,\Phi_1).
  \end{cases}
\end{equation*}
Let $\Phi_\lambda=P_{\lambda} \Phi$ be the approximate solution with the initial data $ (\Phi_\lambda, \partial_t \Phi_\lambda)|_{t=-2}=({P}_\lambda \Phi(-2), {P}_\lambda \partial_t \Phi(-2))$. Then we have
\begin{equation*}
  \square_{\mathbf{g}_{\lambda}} \Phi_\lambda= %P_\lambda \square_{\mathbf{g}_{\lambda}} \Phi+
  [\square_{\mathbf{g}_{\lambda}},P_{\lambda}]\Phi.
\end{equation*}
The energy estimates tell us that
\begin{equation*}
  \| \square_{\mathbf{g}_{\lambda}} \Phi_\lambda \|_{L^1_t L^2_x} \lesssim \| d\mathbf{g}\|_{L^2_t L^\infty_x}\| d\Phi\|_{L^\infty_t L^2_x} \lesssim \epsilon_0 (\|\Phi_0\|_{H^1}+\|\Phi_1\|_{L^2}).
\end{equation*}
%Moreover, the Strichartz estimate \eqref{Ase} holds for $\Phi_\lambda$.
However, it does not match the data at $t=t_0$. Fortunately, we can obtain
\begin{equation*}
  \begin{split}
  & P_\lambda \Phi_0 -\Phi_\lambda (t_0)=P_\lambda \Phi_0^1, \qquad \Phi_0^1=P_\lambda \Phi_0-\Phi_\lambda (t_0),
  \\
  & P_\lambda \Phi_1 -\partial_t \Phi_\lambda (t_0)=P_\lambda \Phi_1^1, \ \quad \Phi_1^1=P_\lambda \Phi_1-\partial_t \Phi_\lambda (t_0).
  \end{split}
\end{equation*}
By energy estimates and commutator estimates, we have
\begin{equation*}
  \begin{split}
  \|\Phi_0^1\|_{H^1} + \|\Phi_1^1\|_{L^2} \lesssim & \| \square_{\mathbf{g}_{\lambda}} ( P_\lambda \Phi - \Phi) \|_{L^1_t L^2_x}
  \\
  \lesssim & \| \square_{\mathbf{g}_{\lambda}}  \Phi_\lambda \|_{L^1_t L^2_x}+ \| [ \square_{\mathbf{g}_{\lambda}} ( \Phi_\lambda] \Phi) \|_{L^1_t L^2_x}
  \\
  \lesssim & \epsilon_0 (\|\Phi_0\|_{H^1}+\|\Phi_1\|_{L^2}).
  \end{split}
\end{equation*}
Since the norm of the error is much smaller than the initial size of the data,
we can repeat this process with data $(\Phi_0^1,\Phi_1^1)$, and sum this series to
obtain a smooth function $\Phi_\lambda$ with data $(P_\lambda \Phi_0,P_\lambda \Phi_1)$ at time $t=t_0$. As a result, it can also match the given data at time $t=t_0$.
\end{remark}
\subsubsection{Overlap estimates}
Since the foliations $\Sigma_{\omega,r}$ varying with $\omega$ and $r$, so a fixed $\Sigma_{\omega,r}$ may intersect with other $\Sigma_{\omega',r'}$. As a result, we should be clear about the number of $\lambda$-slabs which contain two given points in the space-time $[-2,2]\times \mathbb{R}^3$.
\begin{corollary}[\cite{ST}]\label{corl}
	For all points $P_1=(t_1,\bx_1)$ and $P_2=(t_2,\bx_2)$ in space-time $\mathbb{R}^{+} \times \mathbb{R}^3$ , and $\epsilon_0 \lambda \geq 1$, the number $N_{\lambda}(P_1,P_2)$ of slabs of scale $\lambda$ that contain both $P_1$ and $P_2$ satisfies the bound
	\begin{equation*}
		\begin{split}
			N_{\lambda}(P_1,P_2)\lesssim & \epsilon_0^{-1} |t_1-t_2|^{-1}.
		\end{split}
	\end{equation*}
\end{corollary}
\subsubsection{The proof of Estimate \eqref{Ase}}
After the construction of approximate solutions, we still need to prove the key estimate \eqref{Ase}. This can be directly derived by the following result. Before that, let us define $\mathcal{T}=\cup_{\omega,j}T_{\omega,j}$, where $T_{\omega,j}$ is set in \eqref{twj}. We also denote $\chi_{{J}}$ be a smooth cut-off function, and $\chi_J=1$ on a set $J$.
\begin{proposition}\label{szt}%[\cite{ST}]
	Let $t \in [-2,2]$ and
	\begin{equation*}
		f={\sum}_{J \in \mathcal{T}}a_{J}\chi_{J}f_{J},
	\end{equation*}
	where $\sum_{J \in \mathcal{T}}a_{J}^2 \leq 1$ and $f_J$ are normalized wave packets in $J$. Then\footnote{We need $r>1$ in \eqref{Ase} for there is a factor $(\ln \lambda)^{\frac12}$ in \eqref{Aswr}.}
	\begin{equation}\label{Aswr}
		\|f \|_{L^2_t L^\infty_x} \lesssim \epsilon_0^{-\frac12}(\ln \lambda)^{\frac12}.
	\end{equation}
\end{proposition}
\begin{proof}
	This proof also follows Smith-Tataru \cite{ST} (Proposition 10.1 on page 353). Here we use slightly different factors in \eqref{Aswr}.
	
	Let us first make a partition of the time-interval $[-2,2]$. By decomposition, there exists a partition $\left\{ I_j \right\}$ of the  interval $[-2,2]$ into disjoint subintervals $I_j$ such that with the size of each $I_j$, $|I_j| \simeq \lambda^{-1}$, and the number of subintervals $j \simeq\lambda$. By mean value theorem, there exists a number $t_j$ such that
	\begin{equation}\label{Aswrq}
		\|f \|^2_{L^2_t L^\infty_x} \leq \textstyle{\sum}_j \|f \|^2_{L^2_{I_j} L^\infty_x} \leq \textstyle{\sum}_j \|f(t_j,\cdot) \|^2_{ L^\infty_x} \lambda^{-1},%\lesssim \epsilon_0^{-\frac12}(\ln \lambda)^{\frac12}.
	\end{equation}
	where $t_j$ is located in $I_j$, and $|t_{j+1}-t_j| \simeq \lambda^{-1}$. Let us explain the \eqref{Aswrq} as follows. For simplicity, let  $I_{j_0}=[0,\lambda^{-1}]$ and $I_{j_0+1}=[\lambda^{-1},2\lambda^{-1}]$. By mean value theorem, on $I_{j_0}$, we have
	\begin{equation*}%\label{Aswrt}
		\|f \|^2_{L^2_{I_{j_0}} L^\infty_x} = \|f(t_{j_0},\cdot) \|^2_{ L^\infty_x}\lambda^{-1}.%\lesssim \epsilon_0^{-\frac12}(\ln \lambda)^{\frac12}.
	\end{equation*}
	We also have
	\begin{equation*}%\label{Aswrb}
		\|f \|^2_{L^2_{I_{j_0+1}} L^\infty_x} = \|f(t_{j_0+1},\cdot) \|^2_{ L^\infty_x}\lambda^{-1}.%\lesssim \epsilon_0^{-\frac12}(\ln \lambda)^{\frac12}.
	\end{equation*}
	If $t_{j_0} \leq \frac12 \lambda^{-1}$ or $t_{j_0+1} \geq \frac32 \lambda^{-1}$, then $|t_{j_0+1}-t_{j_0}| \geq \frac12 \lambda^{-1}$. Otherwise, $t_{j_0} \in [\frac12 \lambda^{-1}, \lambda^{-1}] $ and $t_{j_0+1} \in [\lambda^{-1}, \frac32 \lambda^{-1}]$. In this case,
	we combine $I_{j_0}, I_{j_0+1}$ together and set $I^*_{j_0}= I_{j_0}\cup I_{j_0+1}$. On the new interval  $I^*_{j_0}$, we can see that
	\begin{equation*}
		\|f \|^2_{L^2_{I^*_{j_0}} L^\infty_x} = 2\|f(t^*_{j_0},\cdot) \|^2_{ L^\infty_x}\lambda^{-1},
	\end{equation*}
	and
	\begin{equation*}
		2\|f(t^*_{j_0},\cdot) \|^2_{ L^\infty_x} = \|f(t_{j_0},\cdot) \|^2_{ L^\infty_x}+\|f(t_{j_0+1},\cdot) \|^2_{ L^\infty_x}. 
	\end{equation*}
	Since $f$ is a contituous function, we have
	\begin{equation*}
		\|f(t^*_{j_0},\cdot) \|_{ L^\infty_x} = \left\{  \frac12 \big( \|f(t_{j_0},\cdot) \|^2_{ L^\infty_x} +\|f(t_{j_0+1},\cdot) \|^2_{ L^\infty_x} \big) \right\}^{\frac12}, \qquad t^*_{j_0} \in [\frac12\lambda^{-1}, \frac32\lambda^{-1}].
	\end{equation*}
In case $t^*_{j_0} \in [\frac12\lambda^{-1}, \frac32\lambda^{-1}]$, we set $t_{j_0}:=t^*_{j_0}$. On the next time-interval $I_{j_0+2}=[2\lambda^{-1}, 3\lambda^{-1}]$, we obtain
	\begin{equation*}
		\|f \|^2_{L^2_{I_{j_0+2}} L^\infty_x} = \|f(t_{j_0+1},\cdot) \|^2_{ L^\infty_x}\lambda^{-1}, \quad t_{j_0+1} \in [2\lambda^{-1}, 3\lambda^{-1}].
	\end{equation*}
	Hence, we have proved $|t_{j_0+1}-t_{j_0}| \geq \frac12 \lambda^{-1}$. So we can use this way to decompose $[-2,2]$. Therefore, to prove \eqref{Aswr}, and taking \eqref{Aswrq} into account, we only need to show that
	\begin{equation}\label{wee0}
		\textstyle{\sum}_j |f(t_j,x_j)|^2 \lesssim \epsilon_0^{-1}\lambda \ln \lambda,
	\end{equation}
	where $x_j$ is arbitrarily chosen. We then set the points $P_j=(t_j,x_j)$.

Since each points lies in at most $\approx \epsilon_0^{-\frac12}\lambda^{\frac12}$ slabs, so we may assume that $|a_J| \geq \epsilon_0^{\frac12}\lambda^{-\frac12}$. Then we decompose the sum $f=\sum_{J \in \mathcal{T}}a_{J}\chi_{J}f_{J}$ dyadically with respect to the size of $a_J$. 
	We next decompose the sum over $j$ via a dyadic decomposition in the numbers of slabs containing $(t_j,\bx_j)$. We may assume that we are summing over $M$ points\footnote{We also remark that $M$ is $\simeq \lambda$, for $t_j \in [-2,2]$ and $|t_{j+1}-t_j|\simeq \lambda^{-1}$.} $(t_j,\bx_j)$, each of which is contained in approximately $L$ slabs\footnote{So $L \lesssim  \epsilon_0^{-1}\lambda$ by using Corollary \ref{corl}.}. Then $|f(t_j,\bx_j)|\lesssim N^{-\frac12}L$ and
	\begin{equation}\label{wee1}
		\textstyle{\sum}_j |f(t_j,\bx_j)|^2 \lesssim L^2 M N^{-1}.
	\end{equation}
	Combining \eqref{wee0} with \eqref{wee1}, we only need to prove
	\begin{equation}\label{wee2}
		L^2 M N^{-1} \lesssim \epsilon_0^{-1}\lambda \ln \lambda.
	\end{equation}
This is a counting problem, and we shall prove \eqref{wee2} by calculating in two different ways. First consider the number $K$ of pairs $(i,j)$ for which $P_i$ and $P_j$ are contained in a common slab, counted with multiplicity. For $J \in \mathcal{T}$, we denote by $n_J$ the number of points $P_j$ contained\footnote{For $J$, if it contains $3$ points $P_1,P_2,P_3$, then we will count it $(P_1,P_2), (P_2,P_1), (P_1,P_3), (P_3,P_1),(P_2,P_3),(P_3,P_2)$ in $K$. Therefore it's $3^2$.} in $J$. Then
	\begin{equation*}
		K = \sum_{n_J \geq 2} n_J^2.
	\end{equation*}
By Jensen's inequality, we have
\begin{equation*}
		K = \sum_{n_J \geq 2} n_J^2 \gtrsim N^{-1}  (\sum_{n_J \geq 2} n_J)^2.
	\end{equation*}
	Note that $ \sum_{J \in \mathcal{T}} n_J \approx ML$. We consider it into two cases. If
	\begin{equation*}
		\sum_{n_J \geq 2} n_J \leq \sum_{n_J =1 } n_J,
	\end{equation*}
	then $N\approx ML$. In this case, combining with the fact that $L \lesssim \epsilon_0^{-1}\lambda$, then \eqref{wee2} holds. Otherwise, we have
	\begin{equation}\label{wee3}
		K \gtrsim N^{-1}M^2L^2.
	\end{equation}
	In this case, by using Corollary \ref{corl}, we obtain
	\begin{equation}\label{wee4}
		K \lesssim \epsilon_0^{-1} \sum_{1\leq i,j \leq M, i\neq j} |t_i-t_j|^{-1}. %\lesssim M \epsilon_0^{-1} \lambda \ln \lambda.
	\end{equation}
	The sum is maximized in the case that $t_j$ are close as possible, i.e. if the $t_j$ are consecutive multiples of $\lambda^{-1}$. Therefore,  \eqref{wee4} yields
	\begin{equation}\label{wee5}
		K \lesssim \epsilon_0^{-1} \lambda \sum_{1\leq i,j \leq M, i\neq j} |i-j|^{-1} \lesssim M \epsilon_0^{-1} \lambda \ln \lambda, %\lesssim M \epsilon_0^{-1} \lambda \ln \lambda.
	\end{equation}
where we use the fact that $M \simeq \lambda$. Combining \eqref{wee5} and \eqref{wee3}, we get \eqref{wee2}. So we have finished the proof of Proposition \ref{szt}.
\end{proof}

\subsection{The proof of Proposition \ref{r5}}
In this part, we will use Proposition \ref{AA1} to prove Proposition \ref{r5}.	We first prove the case of $r=1$ and $r=2$. Considering the fractional derivatives on spatial and using Kato-Ponce commutator estimates and Gagliardo-Nirenberg inequality, we divide the proof into several cases according to Proposition \ref{AA1}.
	
	\textbf{Step 1:} The case $r=1$. Using basic energy estimates for \eqref{linear}, we have
	\begin{equation*}
		\begin{split}
			\| \partial_t f \|_{L^2_x} + \|\nabla f \|_{L^2_x}  \lesssim & \ (\| f_0\|_{H^1}+ \| f_1\|_{L^2}) \exp(\int^t_0 \| d \mathbf{g} \|_{L^\infty_x} d\tau)
			\\
			\lesssim & \ \| f_0\|_{H^1}+ \| f_1\|_{L^2}.
		\end{split}
	\end{equation*}
	Then the Cauchy problem \eqref{linear} holds a unique solution $f \in C([-2,2],H_x^1)$ and $\partial_t f \in C([-2,2],L_x^2)$. It remains to show that the solution $f$ also satisfies the Strichartz estimate \eqref{SL}.

	Without loss of generality, we take $t_0=0$ in Proposition \ref{r5}. For any given initial data $(f_0,f_1) \in H^1 \times L^2$, and $t_0 \in [-2,2]$, we take a Littlewood-Paley decomposition
	\begin{equation*}
		f_0=\textstyle{\sum}_{\lambda}P_{\lambda}f_0, \qquad  f_1=\textstyle{\sum}_{\lambda}P_{\lambda}f_1,
	\end{equation*}
	and for each $\lambda$ we take the corresponding $f_{\lambda}$ as in Proposition \ref{AA1}.  If we set
	\begin{equation}\label{fstar}
		f^*=\textstyle{\sum}_{\lambda}f_{\lambda},
	\end{equation}
	then $f^*$ matches the initial data $(f_0,f_1)$ at the time $t=t_0$, and also satisfies the Strichartz estimates \eqref{Ase}. We claim that $f^*$ is also an approximate solution for $\square_{\mathbf{g}}$ in the sense that
	\begin{equation}\label{QQ}
		\| \square_{\mathbf{g}} f^* \|_{L^1_t L^2_x} \lesssim \epsilon_0(\| f_0 \|_{H^1}+\| f_1 \|_{L^2} ).
	\end{equation}
	We can derive the above bound by using the decomposition
	\begin{equation*}
		\square_{\mathbf{g}} f^*=\textstyle{\sum_{\lambda}} \square_{\mathbf{g}_\lambda} f_\lambda+ \textstyle{\sum_{\lambda}} \square_{\mathbf{g}-\mathbf{g}_\lambda} f_\lambda.
	\end{equation*}
	The first one $\textstyle{\sum_{\lambda}} \square_{\mathbf{g}_\lambda} f_\lambda$ can be controlled by Proposition \ref{AA1}. As for the second one $\textstyle{\sum_{\lambda}} \square_{\mathbf{g}-\mathbf{g}_\lambda} f_\lambda$, noting $\mathbf{g}^{00}=-1$, we then have
	\begin{equation*}
		\begin{split}
			\textstyle{\sum_{\lambda}} \square_{\mathbf{g}-\mathbf{g}_\lambda} f_\lambda= \textstyle{\sum_{\lambda}} ({\mathbf{g}-\mathbf{g}_\lambda}) \partial  df_\lambda.
		\end{split}
	\end{equation*}
	It follows from H\"older inequality that
	\begin{equation}\label{y1}
		\begin{split}
			\| \textstyle{\sum_{\lambda}} \square_{\mathbf{g}-\mathbf{g}_\lambda} f_\lambda \|_{L^2_x} \lesssim \mathrm{sup}_{\lambda} \left(\lambda \| {\mathbf{g}-\mathbf{g}_\lambda} \|_{L^\infty_x} \right) \left( \sum_{\lambda}  \| df_\lambda \|^2_{L^2_x} \right)^{\frac12}.
		\end{split}
	\end{equation}
	On the other hand, we have
	\begin{equation}\label{y2}
		\begin{split}
			\mathrm{sup}_{\lambda} \left(\lambda \| {\mathbf{g}-\mathbf{g}_\lambda} \|_{L^\infty_x} \right) \lesssim & \ \mathrm{sup}_{\lambda} \big(\lambda \sum_{\mu > \lambda}\| \mathbf{g}_\mu \|_{L^\infty_x} \big)
			\\
			\lesssim & \ \mathrm{sup}_{\lambda} \big(\lambda \sum_{\mu > \lambda}\mu^{-(1+\delta)}\|d \mathbf{g}_\mu \|_{C^\delta_x} \big)
			\\
			\lesssim & \ \|d \mathbf{g} \|_{C^\delta_x}(\textstyle{\sum_{\mu}} \mu^{-\delta}) \lesssim \ \|d \mathbf{g} \|_{C^\delta_x}.
		\end{split}
	\end{equation}
	Gathering \eqref{y1} and \eqref{y2}, we can prove that
	\begin{equation*}
		\begin{split}
			\textstyle{\sum_{\lambda}} \square_{\mathbf{g}-\mathbf{g}_\lambda} f_\lambda \lesssim \epsilon_0(\| f_0 \|_{H^1}+\| f_1 \|_{L^2} ),
		\end{split}
	\end{equation*}
where we use the fact that $\|d \mathbf{g} \|_{L^2_t C^\delta_x} \lesssim \epsilon_0$. Hence, we have proved \eqref{QQ}. But $f^*$ is not the exact solution to $\square_{\mathbf{g}} f=0$ with the Cauchy data $(f_0,f_1)$. So we seek another function $\mathbf{M}F$ satisfying
\begin{equation}\label{AQ1}
		\begin{cases}
			&\square_{\mathbf{g}} \mathbf{M}F= - \square_{\mathbf{g}} f^*,
\\
& (\mathbf{M}F, \partial_t \mathbf{M}F)=(0,0).
		\end{cases}
	\end{equation}
For given $F=- \square_{\mathbf{g}} f^* \in L^1_t L^2_x$, we now form the function
	\begin{equation*}
		F^{(1)}=\int^t_0 f^{(1)}(\tau;t,x)d\tau,
	\end{equation*}
	where $f^{(1)}(\tau;t,x)$ is the approximate solution of $\square_{\mathbf{g}} \tilde{f}=0$ and $f^{(1)}$ is formed like $f^*$(cf.  \eqref{fstar}) with the Cauchy data
	\begin{equation*}
\begin{cases}
&\square_{\mathbf{g}} f^{(1)}=0, \qquad (t,\bx) \in (\tau,2] \times \mathbb{R}^3,
\\
&		f^{(1)}(\tau,t,x)|_{t=\tau}=0, \quad \partial_t f^{(1)}(\tau,t,x)|_{t=\tau}=F(\tau,\cdot).
	\end{cases}
\end{equation*}
By calculating
	\begin{equation*}
		\begin{split}
		\partial_t F^{(1)}=& \int^t_0 \partial_t f^{(1)}(\tau;t,x) d\tau+ f^{(1)}(t;t,x)=\int^t_0 \partial_t f^{(1)}(\tau;t,x) d\tau,
		\\
		\partial_{tt} F^{(1)}=& \int^t_0 \partial_{tt} f^{(1)}(\tau;t,x) d\tau+ \partial_t f^{(1)}(\tau;t,x)=\int^t_0 \partial_{tt} f^{(1)}(t;t,x) d\tau+F(t,x),
\\
\partial_{ti} F^{(1)}=& \int^t_0 \partial_{ti} f^{(1)}(t;t,x) d\tau, \quad i=1,2,3,
\\
\partial_{ij} F^{(1)}=& \int^t_0 \partial_{ij} f^{(1)}(t;t,x) d\tau, \quad i,j=1,2,3,
          \end{split}	
\end{equation*}
we then have
	\begin{equation*}
		\square_{\mathbf{g}} F^{(1)}=\int^t_0 \square_{\mathbf{g}} f^{(1)}(\tau;t,x) d\tau+F.
	\end{equation*}
By using \eqref{QQ}, we derive that
	\begin{equation*}
		\| \square_{\mathbf{g}} F^{(1)}-F\|_{L^1_t L^2_x} \lesssim \| \square_{\mathbf{g}}f^{(1)}(\tau;t,x)\|_{L^1_t L^2_x} \lesssim \epsilon_0 \| F \|_{L^1_t L^2_x}.
	\end{equation*}
For $F^{m}, m\geq 2$, we perform an iteration as follows. Let
	\begin{equation*}
		F^{(2)}=\int^t_0 f^{(2)}(\tau;t,x)d\tau,
	\end{equation*}
	where $f^{(2)}(\tau;t,x)$ is the approximate solution of $\square_{\mathbf{g}} \tilde{f}=0$ and $f^{(2)}$ is formed like $f^*$ (cf. \eqref{fstar}) with Cauchy data
	\begin{equation*}
\begin{cases}
&\square_{\mathbf{g}} f^{(2)}=0,\qquad (t,\bx) \in (\tau,2] \times \mathbb{R}^3,
\\
&		f^{(2)}(\tau;t,x)|_{t=\tau}=0, \quad \partial_t f^{(2)}(\tau;t,x)|_{t=\tau}=F-\square_{\mathbf{g}} F^{(1)}.
\end{cases}
	\end{equation*}
Then we can prove that
	\begin{equation*}
		\|\square_{\mathbf{g}}F^{(2)}- (F-\square_{\mathbf{g}} F^{(1)}) \| \lesssim \epsilon_0 \| F-\square_{\mathbf{g}} F^{(1)} \|_{L^1_t L^2_x} \lesssim \epsilon^2_0 \| F \|_{L^1_t L^2_x}.
	\end{equation*}
We rewrite the above equality as
	\begin{equation*}
		\|\square_{\mathbf{g}}(F^{(2)}+ F^{(1)})-F \| \lesssim \epsilon_0 \| F-\square_{\mathbf{g}} F^{(1)} \|_{L^1_t L^2_x} \lesssim \epsilon^2_0 \| F \|_{L^1_t L^2_x}.
	\end{equation*}
For $m \geq 2$, we can set
	\begin{equation*}
		F^{(m)}=\int^t_0 f^{(m)}(\tau;t,x)d\tau,
	\end{equation*}
	where $f^{(m)}(\tau;t,x)$ is the approximate solution of $\square_{\mathbf{g}} \tilde{f}=0$ and $f^{(m)}$ is formed like $f^*$ (cf. \eqref{fstar}) with Cauchy data
	\begin{equation*}
\begin{cases}
&\square_{\mathbf{g}} f^{(m)}=0,\qquad (t,\bx) \in (\tau,2] \times \mathbb{R}^3,
\\
&		f^{(m)}(\tau;t,x)|_{t=\tau}=0, \quad \partial_t f^{(2)}(\tau;t,x)|_{t=\tau}=F-\square_{\mathbf{g}} (F^{(1)}+F^{(2)}+\cdots F^{(m-1)}).
\end{cases}
	\end{equation*}
we therefore obtain that
	\begin{equation}\label{AQ4}
		\|\square_{\mathbf{g}}(F^{(m)}+F^{(m-1)}+\cdots F^{(1)})- F \|_{L^1_t L^2_x} \lesssim \epsilon^m_0 \| F \|_{L^1_t L^2_x}.
	\end{equation}
By using the contraction principle, we can see that there is a limit $  \mathbf{M}F=\lim_{m\rightarrow \infty} (F^{(m)}+F^{(m-1)}+\cdots F^{(1)})$. Hence, we get
\begin{equation}\label{AQ0}
		\square_{\mathbf{g}} \mathbf{M}F = - \square_{\mathbf{g}} f^*,  \quad (\mathbf{M}F, \partial_t \mathbf{M}F)=(0,0).
	\end{equation}
Gathering \eqref{AQ1} and \eqref{AQ0}, we can write the solution $f$ in the form
	\begin{equation*}
		f= f^*+ \mathbf{M}F,
	\end{equation*}
	where $f^*$ is the approximation solution formed in \eqref{fstar} for initial data $(f_0,f_1)$ specified at time $t=0$, and it satisfies the Strichartz estimates \eqref{Ase}, and
	\begin{equation*}
		\| F\|_{L^1_t L^2_x} \lesssim \epsilon_0 ( \| f_0 \|_{H^1}+\| f_1\|_{L^2}).
	\end{equation*}
	The Strichartz estimates of $\mathbf{M}F$ now follow since they holds for each $f^{(m)}(\tau;t,x)$, $\tau \in [0,t]$. Therefore, the estimate \eqref{SL} holds.

\textbf{Step 2:} The case $r=2$. We also consider $t_0=0$. Given data $(f_0,f_1) \in H^2\times H^1$, we find a solution of the form $f=\left< \partial \right>^{-1}\bar{f}$. Then we have
\begin{equation*}
\square_{\mathbf{g}} \bar{f}=( \square_{\mathbf{g}}- \left< \partial \right> \square_{\mathbf{g}} \left< \partial \right>^{-1}) \bar{f}=[\mathbf{g}^{\alpha \beta}, \left< \partial \right>]\left< \partial \right>^{-1} \partial_{\alpha \beta} \bar{f}.
\end{equation*}
Hence, $\bar{f}$ solves
	\begin{equation*}
		\begin{cases}
			\square_{\mathbf{g}} \bar{f}= [\mathbf{g}^{\alpha \beta}, \left< \partial \right>]\left< \partial \right>^{-1} \partial_{\alpha \beta} \bar{f}, \quad (t,\bx)\in [-2,2]\times \mathbb{R}^3,
			\\
			\bar{f}|_{t=0}=\left< \partial \right>f_0, \quad \partial_t\bar{f}|_{t=0}=\left< \partial \right>f_1.
		\end{cases}
	\end{equation*}

For $\bar{F}=[\mathbf{g}^{\alpha \beta}, \left< \partial \right>]\left< \partial \right>^{-1} \partial_{\alpha \beta} \bar{f} \in L_t^1 L^2_x$, we form $\mathbf{M}\bar{F}$ as above in Step 1. So we have $\square_{\mathbf{g}} \mathbf{M}\bar{F}=\bar{F}$, and $\mathbf{M}\bar{F}$ has vanishing Cauchy data at $t_0=0$. Let $\bar{f}^*$ be the solution for the homogeneous equation
 	\begin{equation*}
		\begin{cases}
			\square_{\mathbf{g}} \bar{f}^*=0, \quad (t,\bx)\in [-2,2]\times \mathbb{R}^3,
			\\
			\bar{f}^*|_{t=0}=\left< \partial \right>f_0, \quad \partial_t\bar{f}^*|_{t=0}=\left< \partial \right>f_1.
		\end{cases}
	\end{equation*}
Then we can seek a solution $\bar{f}$ of the form $\bar{f}=\bar{f}^*+\mathbf{M}\bar{F}$, provided that
\begin{equation}\label{AWQ}
\| [\mathbf{g}^{\alpha \beta}, \left< \partial \right>]\left< \partial \right>^{-1} \partial_{\alpha \beta} \mathbf{M}\bar{F}\|_{L^1_tL^2_x} \lesssim \epsilon_0 \|\bar{F}\|_{L^1_tL^2_x}.
\end{equation}
Let us prove \eqref{AWQ}. Considering $\mathbf{g}^{00}=-1$, by using Coifman-Meyer estimate, we derive that
\begin{equation}\label{AQ2}
\|  [\mathbf{g}^{\alpha \beta}, \left< \partial \right>] \left< \partial \right>^{-1} \partial_{\alpha \beta} \mathbf{M}\bar{F} \|_{L^2_x} \lesssim \|d\mathbf{g}\|_{L^\infty_x} \| d \mathbf{M}\bar{F} \|_{L^2_x}.
\end{equation}
By using \eqref{AQ0}, we also obtain
\begin{equation}\label{AQ3}
\| d \mathbf{M}\bar{F}\|_{L^\infty_tL^2_x} \lesssim \|\bar{F}\|_{L^1_t L^2_x}.
\end{equation}
Combining \eqref{AQ2} with \eqref{AQ3}, and using $\|\|d\mathbf{g}\|_{L^2_tL^\infty_x} \lesssim \epsilon_0$, we can get \eqref{AWQ}.

Therefore, by Duhamel's principle, in the case of $r=1$ or $r=2$, we can also obtain the Strichartz estimates for the linear, non-homogeneous wave equation
	\begin{equation*}
		\begin{cases}
			\square_{\mathbf{g}} f= G, \quad (t,\bx)\in [-2,2]\times \mathbb{R}^3,
			\\
			f|_{t=0}=f_0, \quad \partial_tf|_{t=0}=f_1,
		\end{cases}
	\end{equation*}
as
	\begin{equation}\label{eAQ4}
		\| \left< \partial \right>^k f \|_{L^2_t L^\infty_x} \lesssim  \|f_0\|_{H^r}+ \| f_1 \|_{H^{r-1}}+ \|G\|_{L^1_t H_x^{r-1}} , \quad   k<r-1.
	\end{equation}
Now, we claim that the following bounds\footnote{Comparing with \eqref{eAQ4}, the bound \eqref{AQ5} is stronger, for it includes the time-derivatives of $f$. It plays an important role in the next step.} hold for $r=1$ and $r=2$,
	\begin{equation}\label{AQ5}
		\| \left< \partial \right>^k df \|_{L^2_t L^\infty_x} \lesssim  \|f_0\|_{H^r}+ \| f_1 \|_{H^{r-1}}+ \|G\|_{L^1_t H_x^{r-1}} , \quad   k<r-2.
	\end{equation}
Let us now use \eqref{eAQ4} to prove \eqref{AQ5}.

To prove \eqref{AQ5} for $r=2$, we first consider $G=0$. That is,
$\square_{\mathbf{g}} f=0, (f,\partial_t f)|_{t=0}=(f_0,f_1)\in H^2 \times H^1$. For $\mathbf{g}^{00}=-1$, we then have
\begin{equation}\label{AQ6}
\square_{\mathbf{g}} d f= d\mathbf{g}^{\alpha \beta} \partial_{\alpha \beta} f \in L^1_t L^2_x.
\end{equation}
The right hand side of \eqref{AQ6} can be derived by energy estimates and Gronwall's inequality. That is
\begin{equation*}
\|\partial_t d f\|_{L^2_x}+ \|\partial d f\|_{L^2_x} \lesssim ( \|f_0\|_{H^2}+\|f_1\|_{H^1} ) \exp\{\int^{2}_{0} \|d \mathbf{g}\|_{L^\infty_x }dt \} \lesssim  \|f_0\|_{H^2}+\|f_1\|_{H^1} . %\in L^1_t L^2_x,
\end{equation*}
Then $d\mathbf{g}^{\alpha \beta} \partial_{\alpha \beta} f \in L^1_t L^2_x$ holds. Thus, \eqref{AQ5} holds for $r=2$ and $G=0$. If $G\neq 0$, we can use the Duhamel principle to handle it. Decompose $f=f^a+f^b$, where
\begin{equation}\label{aw0}
\begin{cases}
&\square_{\mathbf{g}} f^a=0,
\\
&(f^a,\partial_t f^a)|_{t=0}=(f_0,f_1) \in H^2 \times H^1,
\end{cases}
\end{equation}
and
\begin{equation}\label{aw1}
\begin{cases}
&\square_{\mathbf{g}} f^b=G,
\\
&(f^a,\partial_t f^b)|_{t=0}=(0,0).
\end{cases}
\end{equation}
So we can get the Strichartz estimates for $f^a$
\begin{equation*}
		\| \left< \partial \right>^k df^a \|_{L^2_t L^\infty_x} \lesssim  \|f_0\|_{H^r}+ \| f_1 \|_{H^{r-1}}, \quad   k<r-2.
	\end{equation*}
For $f^b$, we use the iteration process like $\mathbf{M}F$ in Step 1, and the following Strichartz estimates for $f^b$
 \begin{equation*}
		\| \left< \partial \right>^k df^b \|_{L^2_t L^\infty_x} \lesssim  \|G\|_{L^1_t H_x^{r-1}} , \quad   k<r-2.
	\end{equation*}
So the conclusion \eqref{AQ5} holds for $f=f^a+f^b$.

To prove \eqref{AQ5} for $r=1$, we note that, $\left< \partial \right>^{-1}f$ has the Cauchy data in $H^2 \times H^1$ if the Cauchy data of $f$ is in $H^1 \times L^2$, and
\begin{equation*}
  \begin{split}
  \| \square_{\mathbf{g}} \left< \partial \right>^{-1}f \|_{L^1_t H^1_x}= & \|\left< \partial \right> \square_{\mathbf{g}} \left< \partial \right>^{-1}f \|_{L^1_t L^2_x}
  \\
  \lesssim & \|[ \left< \partial \right>, \mathbf{g}^{\alpha \beta}] \left< \partial \right>^{-1}\partial_{\alpha \beta}f \|_{L^1_t L^2_x}+ \|G \|_{L^1_t L^2_x}.
  \end{split}
\end{equation*}
Noting $\mathbf{g}^{00}=-1$, then we give the following bounds by Coifman-Meyer estimate and energy estimates for $f$
\begin{equation*}
  \begin{split}
  \|[ \left< \partial \right>, \mathbf{g}^{\alpha \beta}] \left< \partial \right>^{-1}\partial_{\alpha \beta}f \|_{L^1_t L^2_x}  \lesssim & \|d\mathbf{g}\|_{L^1_t L^2_x}\|df\|_{L^\infty_t L^2_x}
  \\
  \lesssim & \epsilon_0 ( \|f_0\|_{H^1} + \|f_1\|_{L^2}).
  \end{split}
\end{equation*}
\quad \textbf{Step 3:} The case $1 < r <2$. We also set $t_0=0$. Based on the above result in Step 1 and Step 2, we also transform the initial data in $H^1 \times L^2$. Applying $\left< \partial \right>^{r-1}$ to  \eqref{linearA}, we have
	\begin{equation*}
		\square_{\mathbf{g}} \left< \partial \right>^{r-1}f=-[\square_{\mathbf{g}}, \left< \partial \right>^{r-1}]f.
	\end{equation*}
	Let $\left< \partial \right>^{r-1}f=\bar{f}$. Then $\bar{f}$ is a solution to
	\begin{equation}\label{Qd}
		\begin{cases}
			\square_{\mathbf{g}}\bar{f}=-[\square_{\mathbf{g}}, \left< \partial \right>^{r-1}]\left< \partial \right>^{1-r}\bar{f},
			\\
			(\bar{f}, \partial_t \bar{f})|_{t=0} \in H^1 \times L^2.
		\end{cases}
	\end{equation}
Considering $\mathbf{g}^{00}=-1$, let us calculate
\begin{equation}\label{Qd0}
\begin{split}
-[\square_{\mathbf{g}}, \left< \partial \right>^{r-1}]\left< \partial \right>^{1-r}\bar{f}=& [\mathbf{g}^{\alpha i}-\mathbf{m}^{\alpha i}, \left< \partial \right>^{r-1} \partial_i ] \partial_{\alpha} \left< \partial \right>^{1-r} \bar{f}
\\
& \ + \left< \partial \right>^{r-1}( \partial_i \mathbf{g}^{\alpha i} \partial_{\alpha} \left< \partial \right>^{1-r} \bar{f}).
\end{split}
\end{equation}
For $r-1 \in (0,1)$, by product estimates and a refined Kato-Ponce inequalities (please cite Theorem 1.9 in \ref{LD}), we have
	\begin{equation}\label{Qd1}
\begin{split}
		& \| [\mathbf{g}^{\alpha i}-\mathbf{m}^{\alpha i}, \left< \partial \right>^{r-1} \partial_i ] \partial_{\alpha} \left< \partial \right>^{1-r} \bar{f}  \|_{L^2_x}+ \|\left< \partial \right>^{r-1}( \partial_i \mathbf{g}^{\alpha i} \partial_{\alpha} \left< \partial \right>^{1-r} \bar{f})\|_{L^2_x}
\\
\lesssim & \| d \mathbf{g} \|_{L^\infty_x} \| d \bar{f} \|_{L^2_x}+  \| \left< \partial \right>^{r} (\mathbf{g}-\mathbf{m})\|_{L^{p_1}_x}   \|\left< \partial \right>^{1-r} d \bar{f} \|_{L^{q_1}_x},
\end{split}
	\end{equation}
where $p_1=\frac{3}{\frac{3}{2}-s+r}$ and $q_1=\frac{3}{s-r}$. By Sobolev's inequality, we get
\begin{equation}\label{Qd2}
  \| \left< \partial \right>^{r} (\mathbf{g}-\mathbf{m})\|_{L^{p_1}_x} \lesssim \| \mathbf{g}- \mathbf{m}\|_{H^{s}_x}.
\end{equation}
By Gagliardo-Nirenberg inequality, we can see that\footnote{The number $s$ and $s_0$ satisfying $2<s_0<s \leq \frac52$.}
\begin{equation}\label{Qd3}
  \|\left< \partial \right>^{1-r} d \bar{f} \|_{L^{q_1}_x} \lesssim \| d \bar{f} \|^{\frac{s-s_0}{\frac52-s_0}}_{L^{2}_x} \|\left< \partial \right>^{1-s_0} d \bar{f} \|^{1-\frac{s-s_0}{\frac52-s_0}}_{L^{\infty}_x}
\end{equation}
Considering\footnote{please see the definition of $\mathcal{H}$ and \eqref{401}-\eqref{403}}
	\begin{equation*}
		\| d \mathbf{g} \|_{L^2_t L^\infty_x} +  \| \mathbf{g}- \mathbf{m}\|_{L^\infty_tH^{s}_x} \lesssim \epsilon_0,
	\end{equation*}
and combining with \eqref{Qd0}-\eqref{Qd3}, we have
	\begin{equation}\label{eh3}
		\| [\square_{\mathbf{g}}, \left< \partial \right>^{r-1}]\left< \partial \right>^{1-r}\bar{f}\|_{L^1_t L^2_x} \lesssim \epsilon_0( \| d\bar{f}\|_{L^\infty_t L^2_x}+  \| \left< \partial \right>^{1-s_0} d \bar{f} \|_{L^2_t L^\infty_x}).
	\end{equation}
Using the discussion in case $r=1$, namely \eqref{AQ5}, we then get the Strichartz estimates for $\bar{f}$ for $\theta<-1$,
	\begin{equation}\label{eh}
		\begin{split}
			\| \left< \partial \right>^\theta d \bar{f} \|_{L^2_t L^\infty_x} \lesssim  & \  \epsilon_0( \|f_0\|_{H^r}+ \| f_1 \|_{H^{r-1}}+ \| d \bar{f} \|_{L^2_x}+  \|\left< \partial \right>^{1-s_0} d \bar{f} \|_{L^2_t L^\infty_x} ).
			%\\
			%\lesssim & \ \epsilon_0( \|f_0\|_{H^r}+ \| f_1 \|_{H^{r-1}}) , \quad   \theta<0.
		\end{split}
	\end{equation}
Note $1-s_0<-1$. So we take $\theta=1-s_0$, and then we obtain
	\begin{equation}\label{eh0}
		\begin{split}
			\| \left< \partial \right>^{1-s_0} d \bar{f} \|_{L^2_t L^\infty_x} \lesssim  & \  \epsilon_0( \|f_0\|_{H^r}+ \| f_1 \|_{H^{r-1}}+ \| d \bar{f} \|_{L^2_x} ).
		\end{split}
	\end{equation}
Using \eqref{eh0},  \eqref{eh} yields
	\begin{equation}\label{eh1}
		\begin{split}
			\| \left< \partial \right>^\theta d \bar{f} \|_{L^2_t L^\infty_x} \lesssim  & \  \epsilon_0( \|f_0\|_{H^r}+ \| f_1 \|_{H^{r-1}}+ \| d \bar{f} \|_{L^2_x} ), \quad \theta <-1.
		\end{split}
	\end{equation}
By \eqref{eh0}, \eqref{eh3} and the energy estimates for $\bar{f}$, we have
	\begin{equation}\label{AQ7}
		 \| d \bar{f} \|_{L^2_x} \lesssim \| \bar{f}(0) \|_{H^1_x}+ \| \partial_t \bar{f}(0) \|_{L^2_x} \lesssim \|f_0\|_{H^r}+ \| f_1 \|_{H^{r-1}}.
	\end{equation}
Substituting $\left< \partial \right>^{r-1}f=\bar{f}$ in \eqref{eh1} and using \eqref{AQ7}, we can prove that
	\begin{equation*}
			\| \left< \partial \right>^\theta d {f} \|_{L^2_t L^\infty_x} \lesssim   \ \epsilon_0( \|f_0\|_{H^r}+ \| f_1 \|_{H^{r-1}}) , \quad   \theta<r-2.
	\end{equation*}
As a result, we also have
\begin{equation*}
			\| \left< \partial \right>^k f \|_{L^2_t L^\infty_x}
			\lesssim   \epsilon_0( \|f_0\|_{H^r}+ \| f_1 \|_{H^{r-1}} ) , \quad   k<r-1.
	\end{equation*}

\textbf{Step 4:} The case $2 < r \leq s$. We also set $t_0=0$. Based on the above result in Step 1 and Step 2, we transform the initial data in $H^1 \times L^2$. Applying $\left< \partial \right>^{r-1}$ to \eqref{linearA}, we have
	\begin{equation*}
		\square_{\mathbf{g}} \left< \partial \right>^{r-1}f=-[\square_{\mathbf{g}}, \left< \partial \right>^{r-1}]f.
	\end{equation*}
	Let $\left< \partial \right>^{r-1}f=\bar{f}$. Then $\bar{f}$ is a solution to
	\begin{equation}\label{qd}
		\begin{cases}
			\square_{\mathbf{g}}\bar{f}=-[\square_{\mathbf{g}}, \left< \partial \right>^{r-1}]\left< \partial \right>^{1-r}\bar{f},
			\\
			(\bar{f}, \partial_t \bar{f})|_{t=0} \in H^1 \times L^2.
		\end{cases}
	\end{equation}
Considering $\mathbf{g}^{00}=-1$, let us calculate
\begin{equation}\label{AQ9}
\begin{split}
-[\square_{\mathbf{g}}, \left< \partial \right>^{r-1}]\left< \partial \right>^{1-r}\bar{f}=& [\mathbf{g}^{\alpha i}-\mathbf{m}^{\alpha i}, \left< \partial \right>^{r-1} \partial_i ] \partial_{\alpha} \left< \partial \right>^{1-r} \bar{f}
\\
& \ + \left< \partial \right>^{r-1}( \partial_i \mathbf{g}^{\alpha i} \partial_{\alpha} \left< \partial \right>^{1-r} \bar{f}).
\end{split}
\end{equation}
By Kato-Ponce commutator estimates and product estimates, we have
	\begin{equation}\label{qdy}
\begin{split}
		& \| [\mathbf{g}^{\alpha i}-\mathbf{m}^{\alpha i}, \left< \partial \right>^{r-1} \partial_i ] \partial_{\alpha} \left< \partial \right>^{1-r} \bar{f}  \|_{L^2_x}+ \|\left< \partial \right>^{r-1}( \partial_i \mathbf{g}^{\alpha i} \partial_{\alpha} \left< \partial \right>^{1-r} \bar{f})\|_{L^2_x}
\\
\lesssim & \| d \mathbf{g} \|_{L^\infty_x} \| d \bar{f} \|_{L^2_x}+  \| \mathbf{g}^{\alpha \beta}- \mathbf{m}^{\alpha \beta}\|_{H^{r}_x}   \|\left< \partial \right>^{1-r} d \bar{f} \|_{L^\infty_x},
\end{split}
	\end{equation}
Considering\footnote{please see the definition of $\mathcal{H}$ and \eqref{401}-\eqref{403}}
	\begin{equation*}
		\| d \mathbf{g} \|_{L^2_t L^\infty_x} +  \| \mathbf{g}^{\alpha \beta}- \mathbf{m}^{\alpha \beta}\|_{L^\infty_tH^{r}_x} \lesssim \epsilon_0, \qquad 2 < r \leq s,
	\end{equation*}
and combining with \eqref{AQ9}, \eqref{qdy}, we have
	\begin{equation}\label{Eh3}
		\| [\square_{\mathbf{g}}, \left< \partial \right>^{r-1}]\left< \partial \right>^{1-r}\bar{f}\|_{L^1_t L^2_x} \lesssim \epsilon_0( \| d\bar{f}\|_{L^\infty_t L^2_x}+  \| \left< \partial \right>^{1-r} d \bar{f} \|_{L^2_t L^\infty_x}).
	\end{equation}
Using the discussion in case $r=1$, namely \eqref{AQ5}, we then get the Strichartz estimates for $\bar{f}$ for $\theta<-1$,
	\begin{equation}\label{Eh}
		\begin{split}
			\| \left< \partial \right>^\theta d \bar{f} \|_{L^2_t L^\infty_x} \lesssim  & \  \epsilon_0( \|f_0\|_{H^r}+ \| f_1 \|_{H^{r-1}}+ \| d \bar{f} \|_{L^2_x}+  \|\left< \partial \right>^{1-r} d \bar{f} \|_{L^2_t L^\infty_x} ).
		\end{split}
	\end{equation}
Note $r>2$ and $1-r<-1$. We then take $\theta=1-r$. So we have
	\begin{equation}\label{Eh0}
		\begin{split}
			\| \left< \partial \right>^{1-r} d \bar{f} \|_{L^2_t L^\infty_x} \lesssim  & \  \epsilon_0( \|f_0\|_{H^r}+ \| f_1 \|_{H^{r-1}}+ \| d \bar{f} \|_{L^2_x} ).
		\end{split}
	\end{equation}
Using \eqref{eh0}, \eqref{Eh} yields
	\begin{equation}\label{Eh1}
		\begin{split}
			\| \left< \partial \right>^\theta d \bar{f} \|_{L^2_t L^\infty_x} \lesssim  & \  \epsilon_0( \|f_0\|_{H^r}+ \| f_1 \|_{H^{r-1}}+ \| d \bar{f} \|_{L^2_x} ), \quad \theta <-1.
		\end{split}
	\end{equation}
By \eqref{Eh0}, \eqref{Eh3} and the energy estimates for $\bar{f}$, we have
	\begin{equation}\label{Eh4}
		 \| d \bar{f} \|_{L^2_x} \lesssim \| \bar{f}(0) \|_{H^1_x}+ \| \partial_t \bar{f}(0) \|_{L^2_x} \lesssim \|f_0\|_{H^r}+ \| f_1 \|_{H^{r-1}}.
	\end{equation}
Substituting $\left< \partial \right>^{r-1}f=\bar{f}$ in \eqref{Eh1} and using \eqref{Eh4}, we can prove that
	\begin{equation*}
			\| \left< \partial \right>^\theta d {f} \|_{L^2_t L^\infty_x} \lesssim   \ \epsilon_0( \|f_0\|_{H^r}+ \| f_1 \|_{H^{r-1}}) , \quad   \theta<r-2.
	\end{equation*}
As a result, we also have
\begin{equation*}
			\| \left< \partial \right>^k f \|_{L^2_t L^\infty_x}
			\lesssim   \epsilon_0( \|f_0\|_{H^r}+ \| f_1 \|_{H^{r-1}} ) , \quad   k<r-1.
	\end{equation*}

\textbf{Step 5:} The case $s < r \leq 3$. We also set $t_0=0$. Based on the above result in Step 1 and Step 2, we transform the initial data in $H^1 \times L^2$. Applying $\left< \partial \right>^{r-1}$ to \eqref{linearA}, we have
	\begin{equation*}
		\square_{\mathbf{g}} \left< \partial \right>^{r-1}f=-[\square_{\mathbf{g}}, \left< \partial \right>^{r-1}]f.
	\end{equation*}
	Let $\left< \partial \right>^{r-1}f=\bar{f}$. Then $\bar{f}$ is a solution to
	\begin{equation}\label{sy0}
		\begin{cases}
			\square_{\mathbf{g}}\bar{f}=-[\square_{\mathbf{g}}, \left< \partial \right>^{r-1}]\left< \partial \right>^{1-r}\bar{f},
			\\
			(\bar{f}, \partial_t \bar{f})|_{t=0} \in H^1 \times L^2.
		\end{cases}
	\end{equation}
Considering $\mathbf{g}^{00}=-1$, let us calculate
\begin{equation}\label{sy1}
\begin{split}
-[\square_{\mathbf{g}}, \left< \partial \right>^{r-1}]\left< \partial \right>^{1-r}\bar{f}=& [\mathbf{g}^{\alpha i}-\mathbf{m}^{\alpha i}, \left< \partial \right>^{r-1}  ] \partial_{\alpha} \left< \partial \right>^{1-r} \partial_i\bar{f} .
\end{split}
\end{equation}
By Kato-Ponce commutator estimates, we have
	\begin{equation}\label{sy2}
\begin{split}
		& \| [\mathbf{g}^{\alpha i}-\mathbf{m}^{\alpha i}, \left< \partial \right>^{r-1}  ] \partial_{\alpha} \left< \partial \right>^{1-r} \partial_i \bar{f}  \|_{L^2_x}
\\
\lesssim & \| d \mathbf{g} \|_{L^\infty_x} \| d \bar{f} \|_{L^2_x}+  \| \left< \partial \right>^{r-1}( \mathbf{g}- \mathbf{m})\|_{L^{p_2}_x}   \|\left< \partial \right>^{2-r} d \bar{f} \|_{L^{q_2}_x},
\end{split}
	\end{equation}
where $p_2=\frac{3}{\frac{1}{2}-s+r}$ and $q_2=\frac{3}{1+s-r}$. By Sobolev's inequality, we get
\begin{equation}\label{sy3}
  \| \left< \partial \right>^{r-1} (\mathbf{g}-\mathbf{m})\|_{L^{p_2}_x} \lesssim \| \mathbf{g}- \mathbf{m}\|_{H^{s}_x}.
\end{equation}
By Gagliardo-Nirenberg inequality, we can see that\footnote{The number $s$ and $s_0$ satisfying $2<s_0<s \leq \frac52$.}
\begin{equation}\label{sy4}
  \|\left< \partial \right>^{2-r} d \bar{f} \|_{L^{q_2}_x} \lesssim \| d \bar{f} \|^{\frac{s-2}{3-s}}_{L^{2}_x} \|\left< \partial \right>^{-\frac{s}{2}} d \bar{f} \|^{\frac{5-2s}{3-s}}_{L^{\infty}_x}
\end{equation}
Considering\footnote{please see the definition of $\mathcal{H}$ and \eqref{401}-\eqref{403}}
	\begin{equation*}
		\| d \mathbf{g} \|_{L^2_t L^\infty_x} +  \| \mathbf{g}^{\alpha \beta}- \mathbf{m}^{\alpha \beta}\|_{L^\infty_tH^{r-1}_x} \lesssim \epsilon_0, \qquad s < r \leq 3,
	\end{equation*}
and combining with \eqref{sy1}, \eqref{sy2}, we have
	\begin{equation*}%\label{sy3}
		\| [\square_{\mathbf{g}}, \left< \partial \right>^{r-1}]\left< \partial \right>^{1-r}\bar{f}\|_{L^1_t L^2_x} \lesssim \epsilon_0( \| d\bar{f}\|_{L^\infty_t L^2_x}+  \| \left< \partial \right>^{-\frac{s}{2}} d \bar{f} \|_{L^2_t L^\infty_x}).
	\end{equation*}
Using the discussion in case $r=1$, namely \eqref{AQ5}, we then get the Strichartz estimates for $\bar{f}$ for $\theta<-1$,
	\begin{equation}\label{syu4}
		\begin{split}
			\| \left< \partial \right>^\theta d \bar{f} \|_{L^2_t L^\infty_x} \lesssim  & \  \epsilon_0( \|f_0\|_{H^r}+ \| f_1 \|_{H^{r-1}}+ \| d \bar{f} \|_{L^2_x}+  \|\left< \partial \right>^{2-r} d \bar{f} \|_{L^2_t L^\infty_x} ).
			%\\
			%\lesssim & \ \epsilon_0( \|f_0\|_{H^r}+ \| f_1 \|_{H^{r-1}}) , \quad   \theta<0.
		\end{split}
	\end{equation}
Note $-\frac{s}{2}<-1$. We then take $\theta=-\frac{s}{2}$. So we have
	\begin{equation}\label{sy5}
		\begin{split}
			\| \left< \partial \right>^{2-r} d \bar{f} \|_{L^2_t L^\infty_x} \lesssim  & \  \epsilon_0( \|f_0\|_{H^r}+ \| f_1 \|_{H^{r-1}}+ \| d \bar{f} \|_{L^2_x} ).
			%\\
			%\lesssim & \ \epsilon_0( \|f_0\|_{H^r}+ \| f_1 \|_{H^{r-1}}) , \quad   \theta<0.
		\end{split}
	\end{equation}
Using \eqref{sy5},  \eqref{syu4} yields
	\begin{equation}\label{sy6}
		\begin{split}
			\| \left< \partial \right>^\theta d \bar{f} \|_{L^2_t L^\infty_x} \lesssim  & \  \epsilon_0( \|f_0\|_{H^r}+ \| f_1 \|_{H^{r-1}}+ \| d \bar{f} \|_{L^2_x} ), \quad \theta <-1.
		\end{split}
	\end{equation}
By \eqref{sy4}, \eqref{sy5} and the energy estimates for $\bar{f}$, we have
	\begin{equation}\label{sy7}
		 \| d \bar{f} \|_{L^2_x} \lesssim \| \bar{f}(0) \|_{H^1_x}+ \| \partial_t \bar{f}(0) \|_{L^2_x} \lesssim \|f_0\|_{H^r}+ \| f_1 \|_{H^{r-1}}.
	\end{equation}
Substituting $\left< \partial \right>^{r-1}f=\bar{f}$ in \eqref{sp6} and using \eqref{sp7}, we can prove that
	\begin{equation*}
			\| \left< \partial \right>^\theta d {f} \|_{L^2_t L^\infty_x} \lesssim   \ \epsilon_0( \|f_0\|_{H^r}+ \| f_1 \|_{H^{r-1}}) , \quad   \theta<r-2.
	\end{equation*}
As a result, we also have
\begin{equation*}
			\| \left< \partial \right>^k f \|_{L^2_t L^\infty_x}
			\lesssim   \epsilon_0( \|f_0\|_{H^r}+ \| f_1 \|_{H^{r-1}} ) , \quad   k<r-1.
	\end{equation*}

\textbf{Step 6:} The case $3 < r \leq s+1$. We also set $t_0=0$. Based on the above result in Step 1 and Step 2, we transform the initial data in $H^1 \times L^2$.  Applying $\left< \partial \right>^{r-1}$ to \eqref{linearA}, we have
	\begin{equation*}
		\square_{\mathbf{g}} \left< \partial \right>^{r-1}f=-[\square_{\mathbf{g}}, \left< \partial \right>^{r-1}]f.
	\end{equation*}
	Let $\left< \partial \right>^{r-1}f=\bar{f}$. Then $\bar{f}$ is a solution to
	\begin{equation}\label{sp0}
		\begin{cases}
			\square_{\mathbf{g}}\bar{f}=-[\square_{\mathbf{g}}, \left< \partial \right>^{r-1}]\left< \partial \right>^{1-r}\bar{f},
			\\
			(\bar{f}, \partial_t \bar{f})|_{t=0} \in H^1 \times L^2.
		\end{cases}
	\end{equation}
Considering $\mathbf{g}^{00}=-1$, let us calculate
\begin{equation}\label{sp1}
\begin{split}
-[\square_{\mathbf{g}}, \left< \partial \right>^{r-1}]\left< \partial \right>^{1-r}\bar{f}=& [\mathbf{g}^{\alpha i}-\mathbf{m}^{\alpha i}, \left< \partial \right>^{r-1}  ] \partial_{\alpha} \left< \partial \right>^{1-r} \partial_i\bar{f} .
\end{split}
\end{equation}
By Kato-Ponce commutator estimates, we have
	\begin{equation}\label{sp2}
\begin{split}
		& \| [\mathbf{g}^{\alpha i}-\mathbf{m}^{\alpha i}, \left< \partial \right>^{r-1}  ] \partial_{\alpha} \left< \partial \right>^{1-r} \partial_i \bar{f}  \|_{L^2_x}
\\
\lesssim & \| d \mathbf{g} \|_{L^\infty_x} \| d \bar{f} \|_{L^2_x}+  \| \mathbf{g}^{\alpha \beta}- \mathbf{m}^{\alpha \beta}\|_{H^{r-1}_x}   \|\left< \partial \right>^{2-r} d \bar{f} \|_{L^\infty_x},
\end{split}
	\end{equation}
Considering\footnote{please see the definition of $\mathcal{H}$ and \eqref{401}-\eqref{403}}
	\begin{equation*}
		\| d \mathbf{g} \|_{L^2_t L^\infty_x} +  \| \mathbf{g}^{\alpha \beta}- \mathbf{m}^{\alpha \beta}\|_{L^\infty_tH^{r-1}_x} \lesssim \epsilon_0, \qquad 3 < r \leq s+1,
	\end{equation*}
and combining with \eqref{sp1}, \eqref{sp2}, we have
	\begin{equation}\label{sp3}
		\| [\square_{\mathbf{g}}, \left< \partial \right>^{r-1}]\left< \partial \right>^{1-r}\bar{f}\|_{L^1_t L^2_x} \lesssim \epsilon_0( \| d\bar{f}\|_{L^\infty_t L^2_x}+  \| \left< \partial \right>^{2-r} d \bar{f} \|_{L^2_t L^\infty_x}).
	\end{equation}
Using the discussion in case $r=1$, namely \eqref{AQ5}, we then get the Strichartz estimates for $\bar{f}$ for $\theta<-1$,
	\begin{equation}\label{sp4}
		\begin{split}
			\| \left< \partial \right>^\theta d \bar{f} \|_{L^2_t L^\infty_x} \lesssim  & \  \epsilon_0( \|f_0\|_{H^r}+ \| f_1 \|_{H^{r-1}}+ \| d \bar{f} \|_{L^2_x}+  \|\left< \partial \right>^{2-r} d \bar{f} \|_{L^2_t L^\infty_x} ).
			%\\
			%\lesssim & \ \epsilon_0( \|f_0\|_{H^r}+ \| f_1 \|_{H^{r-1}}) , \quad   \theta<0.
		\end{split}
	\end{equation}
Note $r>3$ and $2-r<-1$. We then take $\theta=2-r$. So we have
	\begin{equation}\label{sp5}
		\begin{split}
			\| \left< \partial \right>^{2-r} d \bar{f} \|_{L^2_t L^\infty_x} \lesssim  & \  \epsilon_0( \|f_0\|_{H^r}+ \| f_1 \|_{H^{r-1}}+ \| d \bar{f} \|_{L^2_x} ).
			%\\
			%\lesssim & \ \epsilon_0( \|f_0\|_{H^r}+ \| f_1 \|_{H^{r-1}}) , \quad   \theta<0.
		\end{split}
	\end{equation}
Using \eqref{sp5}, \eqref{sp4} yields
	\begin{equation}\label{sp6}
		\begin{split}
			\| \left< \partial \right>^\theta d \bar{f} \|_{L^2_t L^\infty_x} \lesssim  & \  \epsilon_0( \|f_0\|_{H^r}+ \| f_1 \|_{H^{r-1}}+ \| d \bar{f} \|_{L^2_x} ), \quad \theta <-1.
		\end{split}
	\end{equation}
By \eqref{Eh0}, \eqref{Eh3} and the energy estimates for $\bar{f}$, we have
	\begin{equation}\label{sp7}
		 \| d \bar{f} \|_{L^2_x} \lesssim \| \bar{f}(0) \|_{H^1_x}+ \| \partial_t \bar{f}(0) \|_{L^2_x} \lesssim \|f_0\|_{H^r}+ \| f_1 \|_{H^{r-1}}.
	\end{equation}
Substituting $\left< \partial \right>^{r-1}f=\bar{f}$ in \eqref{sp6} and using \eqref{sp7}, we can prove that
	\begin{equation*}
			\| \left< \partial \right>^\theta d {f} \|_{L^2_t L^\infty_x} \lesssim   \ \epsilon_0( \|f_0\|_{H^r}+ \| f_1 \|_{H^{r-1}}) , \quad   \theta<r-2.
	\end{equation*}
As a result, we also have
\begin{equation*}
			\| \left< \partial \right>^k f \|_{L^2_t L^\infty_x}
			\lesssim   \epsilon_0( \|f_0\|_{H^r}+ \| f_1 \|_{H^{r-1}} ) , \quad   k<r-1.
	\end{equation*}
At this stage, we have finished the proof of Proposition \ref{r5}.
\section{Continuous dependence of solutions on initial data} \label{Sub}
In this part, we will study the continuous dependence of solutions of isentropic compressible Euler equations. We will discuss the problem of continuous dependence by extending a frequency envelope approach from the incompressible Euler equations \cite{Tao} to the compressible Euler system, This approach was originally devised by Tao and developed by Ifrim-Tataru for quasilinear hyperbolic systems \cite{IT1}. We should mention that there are some works considering the continuous dependence for incompressible Euler equations \cite{BS,KL} and incompressible Neo-Hookean materials \cite{AK}, and they need some additional regularity conditions for the vorticity. Different with these historical results, we are concerned with the rough solutions for a wave-transport system, and the proof crucially relies on Strichartz estimates of a linear wave equation endowed with the acoustic metric.
\begin{corollary}
	[Continuous dependence on initial data]\label{cv} Let $2<s_0<s\leq \frac{5}{2}$. Let the initial data $(\bv_{0}, \rho_{0}, h_{0}, \bw_{0})$ be stated in Theorem \ref{dingli}. Then
	\begin{equation}\label{Id}
		\| \bv_0\|_{H^s}+ \| \rho_0\|_{H^s} + \| h_0 \|_{H^{s_0+1}} + \| \bw_0 \|_{H^{s_0}} \leq M_0.
	\end{equation}
Let $\{(\bv_{0m}, \rho_{0m}, h_{0m}, \bw_{0m})\}_{m \in \mathbb{N}^+}$ be a sequence of initial data converging to $(\bv_0,\rho_0, h_0, \bw_0)$ in space $H^s \times H^s \times H^{s_0+1}\times H^{s_0}$. Then for every $m \in \mathbb{N}^+$, there exists $T>0$ such that the Cauchy problem \eqref{fc0} with the initial data $(\bv_{0m}, \rho_{0m}, h_{0m}, \bw_{0m})$ has a unique solution $(\bv_{m}, \rho_{m}, h_m, \bw_{m})$ on $[0,T]$, and the Cauchy problem \eqref{fc0} with the initial data $(\bv_{0}, \rho_{0}, h_{0}, \bw_{0})$ has a unique solution $(\bv, \rho, h, \bw)$ on $[0,T]$. Moreover, the sequence $\{(\bv_{m}, \rho_{m}, h_m, \bw_{m})\}_{m \in \mathbb{N}^+}$ converge to $(\bv,\rho,h,\bw)$ on $[0,T]$ in $\in H^s_x \times H_x^s \times H_x^{s_0+1} \times H_x^{s_0}$. %Moreover,
\end{corollary}
Before we prove Corollary \ref{cv}, let us introduce a definition of frequency envelope, which is proposed by Tao \cite{Tao}.
\begin{definition}\label{sfe}
	We say that $\{c_k\}_{k \in \mathbb{N}^+} \in \ell^2$ is a frequency envelope for a function $f$ in $H^s$ if we have the
	following two properties:

	(1) Energy bound:
	\begin{equation}\label{f0}
		\|P_k f \|_{H^s} \lesssim c_k,
	\end{equation}
	\quad (2)  Slowly varying:
	\begin{equation}\label{f1}
		\frac{c_k}{c_j} \lesssim 2^{\delta|j-k|}, \quad j,k \in \mathbb{N}^+.
	\end{equation}
	We call such envelopes sharp, if
	\begin{equation*}
		\| f\|^2_{H^s} \approx \sum_{k \geq 0} c^2_k.
	\end{equation*}
\end{definition}
\medskip\begin{proof}[Proof of Corollary \ref{cv}]
	For $\{(\bv_{0m}, \rho_{0m},h_{0m}, \bw_{0m})\}_{m \in \mathbb{N}^+}$ is a sequence of initial data converging to $(\bv_0,\rho_0, h_0, \bw_0)$ in space $H^s \times H^s \times H^{s_0+1} \times H^{s_0}$, by using \eqref{Id}, we  obtain
	\begin{equation*}
		\| \bv_{0m}\|_{H^s}+ \| \rho_{0m}\|_{H^s}+ \| h_{0m}\|_{H^{s_0+1}} + \| \bw_{0m} \|_{H^{s_0}} \leq 100 M_0, \quad \text{for \ large} \ m,
	\end{equation*}
	By using the existence result in Theorem \ref{dingli}, there is $T>0$ such that the Cauchy problem \eqref{fc0} with the initial data $(\bv_{0m}, \rho_{0m},h_{0m}, \bw_{0m})$ has a unique solution $(\bv_{m}, \rho_{m}, h_m, \bw_{m})$ on $[0,T]$, and the Cauchy problem \eqref{fc0} with the initial data $(\bv_{0}, \rho_{0},h_0, \bw_{0})$ has a unique solution $(\bv, \rho,h, \bw)$ on $[0,T]$. We then divide the proof of convergence into three steps.

\textit{Step 1: Convergence in weaker spaces}. By $L^2_x$ energy estimates, and using the Strichartz estimates, for $m,l \in \mathbb{N}^+$, we have
	\begin{equation*}
		\begin{split}
		&	\|(\bv_m-\bv_l, \rho_m-\rho_l, h_m-h_l)(t,\cdot)\|_{L_x^{2}}+\|(\bw_m-\bw_l)(t,\cdot)\|_{L_x^{2}}
			\\
			\lesssim &\|(\bv_{0m}-\bv_{0l}, \rho_{0m}-\rho_{0l})\|_{H_x^{s}}+ \|h_{0m}-h_{0l})\|_{H_x^{s_0+1}}+\|\bw_{0m}-\bw_{0l}\|_{H_x^{s_0}}.
		\end{split}
	\end{equation*}
	This implies that, for $t\in [0,T]$,
	\begin{equation*}
		\lim_{m \rightarrow \infty}(\bv_m, \rho_m, h_m, \bw_m)  =  (\bv, \rho, h, \bw) \quad \mathrm{in} \ \ L^{2}_x.
	\end{equation*}
	By interpolation formula, we then have
	\begin{equation*}
		\|\bv_m-\bv\|_{H_x^\sigma}  \lesssim \|\bv_m-\bv\|^{1-\frac{\sigma}{s}}_{L_x^{2}} \|\bv_m-\bv\|^{\frac{\sigma}{s}}_{H_x^s}, \quad 0 \leq \sigma <s.
	\end{equation*}
	This leads to
	\begin{equation*}
		\lim_{m\rightarrow \infty}\bv_m  =  \bv \ \ \mathrm{in} \ \ H^{\sigma}_x, \ \ 0 \leq \sigma <s,
	\end{equation*}
In a similar way, we can obtain that
	\begin{equation}\label{er0}
		\begin{split}
			& \lim_{m\rightarrow \infty}\rho_m =  \rho \ \ \mathrm{in} \ \ H^{\sigma}_x, \ \ 0 \leq \sigma <s,
			\\
			& \lim_{m\rightarrow \infty}h_m =  h \ \ \mathrm{in} \ \ H^{\gamma}_x, \ \ 0 \leq \gamma <s_0+1,
			\\
			& \lim_{m\rightarrow \infty}\bw_m =  \bw \ \ \mathrm{in} \ \ H^{\theta}_x, \ \ 0 \leq \theta <s_0.
		\end{split}
	\end{equation}
It remains for us to prove the convergence with the highest derivatives of $(\bv_m,\rho_m,h_m,\bw_m)$.

	\textit{Step 2: Constructing smooth solutions}. We set $\bU_0=(v^1_{0},v^2_{0},v^3_{0},p(\rho_0),h)^{\mathrm{T}}$. By \cite{Tao,IT1}, there exists a sharp frequency envelope for $v^1_{0},v^2_{0}, v^3_{0}$ $\rho_0$ and $h_0$ respectively. Let $\{ c^{(i)}_{k} \}_{k \geq 0}$($i=1,2,3,4$) be a sharp frequency envelope respectively for $v^1_0, v^2_0, v^3_0, \rho_0$ in $ H^s$. Let $\{ c^{(5)}_{k} \}_{k \geq 0}$ be a sharp frequency envelope respectively for $h_0$ in $ H^{s_0+1}$. Let $ d^{(1)}_{k}, d^{(2)}_{k}, d^{(3)}_{k}$ be a sharp frequency envelope for $w^1_0, w^2_0$ and $w^3_0$ in $ H^{s_0}$. We choose a family of regularizations $\bU^{l}_0=(v^{1l}_0, v^{2l}_0,v^{3l}_0, p(\rho^{l}_0), h^{l}_0)^{\mathrm{T}}\in \cap^\infty_{a=0}H^a$ at frequencies $\lesssim 2^l$ where $l$ is a frequency parameter and $l \in \mathbb{N}^+$. Denote
	\begin{equation*}
		\bw^l_0=\bar{\rho}^{-1}\mathrm{e}^{-\rho^{l}_0}\mathrm{curl}\bv^l_0.
	\end{equation*}
	Then the functions $v^{1l}_0, v^{2l}_0, v^{3l}_0, \rho^{l}_0$, and $\bw^l_0=(w^{1l}_0,w^{2l}_0,w^{3l}_0)^{\mathrm{T}}$ have the following properties:

	(i)  uniform bounds
	\begin{equation}\label{pqr0}
		\begin{split}
		& \| P_k v^{il}_0 \|_{H_x^s} \lesssim c^{(i)}_k, \quad \| P_k w^{il}_0 \|_{H_x^{s_0}} \lesssim d^{(i)}_k, \qquad i=1,2,3,
		\\
		 &  \| P_k \rho^{l}_0 \|_{H_x^s} \lesssim c^{(4)}_k, \ \quad \| P_k h^{l}_0 \|_{H_x^{s_0+1}} \lesssim c^{(5)}_k,
\end{split}
\end{equation}
	
	(ii)  high frequency bounds
	\begin{equation}\label{pqr1}
		\begin{split}
		& \|   P_k v^{il}_0 \|_{H_x^{s+1}} \lesssim 2^{k}c^{(i)}_k, \quad \|P_k  w^{il}_0 \|_{H_x^{s_0+1}} \lesssim 2^{k}d^{(i)}_k, \qquad i=1,2, 3,
		\\
		&\|   P_k \rho^{l}_0 \|_{H_x^{s+1}} \lesssim 2^{k}c^{(4)}_k,   \quad \|  P_k h^{l}_0 \|_{H_x^{s_0+2}} \lesssim 2^{k}c^{(5)}_k,
\end{split}	
\end{equation}
	
	(iii)  difference bounds
	\begin{equation}\label{pqr2}
		\begin{split}
			&\|  v^{i(l+1)}_0- v^{il}_0\|_{L_x^{2}} \lesssim 2^{-sl}c^{(i)}_l, \quad  \|  w^{i(l+1)}_0- w^{il}_0\|_{L_x^{2}} \lesssim 2^{-s_0l}d^{(i)}_l, \quad i=1,2,3,
			\\
			&\| \rho^{l+1}_0- \rho^{l}_0 \|_{L_x^{2}} \lesssim 2^{-{s}l}c^{(4)}_l, \quad \quad \| h^{l+1}_0- h^{l}_0 \|_{L_x^{2}} \lesssim 2^{-(s_0+1)l}c^{(5)}_l,
		\end{split}
	\end{equation}
	
	(iv)  limit
	\begin{equation}\label{li}
		\begin{split}
			&\lim_{l \rightarrow \infty}\bv^l_0=\bv_0,  \mathrm{in}\ H^s, \qquad \quad \lim_{l \rightarrow \infty} \rho_0^l =\rho_0,  \mathrm{in}\ H^s,
			\\
			&\lim_{l\rightarrow \infty}\bw^l_0=\bw_0,   \mathrm{in}\ H^{s_0}, \qquad  \lim_{l\rightarrow \infty}h^l_0=h_0,   \mathrm{in}\ H^{s_0+1}.
		\end{split}
	\end{equation}
	We set $\bv^{l}=(v^{1l},v^{2l},v^{3l})$. Consider a Cauchy problem \eqref{CEE} with the initial data
	with the initial data
	\begin{equation*}
		\begin{split}
			& (\bv^l,\rho^l,h^l, \bw^l)|_{t=0}=(\bv^l_0,\rho^l_0,h^l_0,\bw^l_0).
			%\\
			%& (\partial_t \bv^l, \partial_t \rho^l)|_{t=0}=(-\bv^l_0\cdot \nabla \bv^l_0+c^2_s\nabla \rho^l_0,-\bv^l_0 \cdot \nabla \rho^l_0-\mathrm{div}\bv^l_0).
		\end{split}
	\end{equation*}
	By Proposition \ref{p3}, we can obtain a family of smooth solutions $(\bv^{l}, \rho^{l}, h^{l}, \bw^l)$ on a time interval $[0,T]$ ($T$ depends only on the size of $\|\bv_0\|_{H^s}+\|\rho_0\|_{H^s}+\|h_0\|_{H^{s_0+1}}+\|\bw_0\|_{H^{s_0}}$). Furthermore, based on the estimates of \eqref{pqr0}-\eqref{li}, using Proposition \ref{p3} again, we can claim that

	$\bullet$  high frequency bounds
	\begin{equation}\label{dbsh}
		\|  \bv^{l}\|_{H_x^{s+1}}+\| \rho^{l} \|_{H_x^{s+1}}+ \| \bw^l \|_{H_x^{s_0+1}}+ \| h^l \|_{H_x^{s_0+2}} \lesssim 2^{l} M_0 ,
	\end{equation}
	
	$\bullet$  difference bounds
	\begin{equation}\label{dbs}
		\|  (\bv^{l+1}- \bv^{l}, \rho^{l+1}- \rho^{l}, h^{l+1}- h^{l}) \|_{L_x^{2}} \lesssim 2^{-sl}c_l,
	\end{equation}
	and
	\begin{equation}\label{dbs2}
		\| h^{l+1}- h^{l} \|_{\dot{H}_x^{s_0+1}} \lesssim c_l,
	\end{equation}
and
	\begin{equation}\label{dbs1}
		\| \bw^{l+1}- \bw^{l} \|_{L_x^{2}} \lesssim 2^{-(s-1)l} c_l,\quad \| \bw^{l+1}- \bw^{l} \|_{\dot{H}_x^{s_0}} \lesssim c_l,
	\end{equation}
	where $c_l=\sum^{5}_{a=1}c^{(a)}_l  + \sum^{3}_{a=1}d^{(1)}_l+d^{(2)}_l+d^{(3)}_l$. By using \eqref{p31} in Proposition \ref{p3} and combining with \eqref{pqr1}, then the estimate \eqref{dbsh} holds.
 Let us postpone to prove \eqref{dbs}-\eqref{dbs1}. By using
	\begin{equation*}
		\rho- \rho^{l}=\sum^\infty_{k=l}( \rho^{k+1}-\rho^k ),
	\end{equation*}
	and the estimate \eqref{dbs}, we can conclude
	\begin{equation}\label{ei0}
		\| \rho- \rho^{l} \|_{H_x^{s}} \lesssim c_{\geq l}.
	\end{equation}
	Here $c_{\geq l}=\sum_{j \geq l}c_j$. Similarly, by using \eqref{dbs2} and \eqref{dbs1}, we also have
	\begin{equation}\label{ei1}
		\|  \bv- \bv^{l}\|_{H_x^{s}} \lesssim c_{\geq l}, \quad \|  h- h^{l}\|_{H_x^{s_0+1}} \lesssim c_{\geq l}, \quad \| \bw- \bw^{l} \|_{H_x^{s_0}} \lesssim c_{\geq l}.
	\end{equation}
	\quad \textit{Step 3: Convergence of the highest derivatives}. For the initial data $(\bv_{0m}, \rho_{0m}, h_{0m}, \bw_{0m})$ and $\bv_{0m}=(v^1_{0m},v^2_{0m},v^3_{0m})$, let $\{ c^{(i)m}_{k}\}_{k\geq 0}$ be frequency envelopes for the initial data $v^{i}_{0m}$ in $H^s (i=1,2,3)$. Let $\{ c^{(4)m}_{k}\}_{k\geq 0}$ be frequency envelopes for the initial data $\rho_{0m}$ in $H^{s_0+1}$. Let $\{ c^{(5)m}_{k}\}_{k\geq 0}$ be frequency envelopes for the initial data $h_{0m}$ in $H^{s_0}$. Let $\{ d^{(i)m}_{k}\}_{k\geq 0}$ be frequency envelopes for the initial data $w^i_{0m}$ in $H^{s_0}(i=1,2,3)$. We choose a family of regularizations $(v^{1l}_{0m}, v^{2l}_{0m}, v^{3l}_{0m}, \rho^{l}_{0m}, h^{l}_{0m})\in \cap^\infty_{a=0}H^a$ at frequencies $\lesssim 2^l$. Denote
	\begin{equation*}
		\bw^l_{0m}=\bar{\rho}^{-1}\mathrm{e}^{-\rho^{l}_{0m}}\mathrm{curl}\bv^l_{0m}.
	\end{equation*}
	Then the functions $v^{1l}_{0m}, v^{2l}_{0m}, \rho^{l}_{0m}, h^{l}_{0m}$, and $\bw^l_{0m}$ can also have

	(1)  uniform bounds
	\begin{equation*}
		\begin{split}
		& \| P_k v^{il}_{0m} \|_{H_x^s} \lesssim c^{(i)m}_k, \quad \| P_k w^{il}_{0m} \|_{H_x^{s_0}} \lesssim d^{(i)m}_k, \quad i=1,2,3,
		\\
		&  \| P_k \rho^{l}_{0m} \|_{H_x^s} \lesssim c^{(4)m}_k, \quad \| P_k h^{l}_{0m} \|_{H_x^{s_0+1}} \lesssim c^{(5)m}_k,
	\end{split}
\end{equation*}
	
	(2)  high frequency bounds for $a\geq 0$
	\begin{equation*}
		\begin{split}
			& \|P_k   v^{i l}_{0m} \|_{H_x^{s+a}} \lesssim 2^{ak}c^{(i)}_k, \quad i=1,2, 3,
			\\
			& \|P_k   \rho^{l}_{0m} \|_{H_x^{s+a}} \lesssim 2^{ak}c^{(4)}_k, \quad  \|P_k h^{l}_{0m} \|_{H_x^{s_0+1+a}} \lesssim 2^{ak}c^{(5)}_k,
		\end{split}
	\end{equation*}
	
	(3)  difference bounds
	\begin{equation*}
		\begin{split}
			&\|  v^{i(l+1)}_{0m}- v^{il}_{0m}\|_{L_x^{2}} \lesssim 2^{-sl}c^{(i)m}_l, \quad  \|  w^{i(l+1)}_{0m}- w^{il}_{0m}\|_{L_x^{2}} \lesssim 2^{-s_0l}d^{(i)m}_l, \qquad i=1,2,3,
			\\
			&\| \rho^{l+1}_{0m}- \rho^{l}_{0m} \|_{L_x^{2}} \lesssim 2^{-sl}c^{(4)m}_l, \quad \ \ \ \| h^{i(l+1)}_{0m}- h^{il}_{0m} \|_{L_x^{2}} \lesssim 2^{-(s_0+1) l}c^{(5)m}_l,
		\end{split}
	\end{equation*}
	
	(4)  limit
	\begin{equation}\label{SSS}
		\begin{split}
			&\lim_{l\rightarrow \infty}\bv^l_{0m}=\bv_{0m},  \ \mathrm{in}\ H^s,\qquad \lim_{l\rightarrow \infty}\rho_{0m}^l= \rho_{0m},  \quad \mathrm{in}\ H^s,
			\\
			&\lim_{l\rightarrow \infty}\bw^l_{0m}=\bw_{0m},  \ \mathrm{in}\ H^{s_0}, \qquad \lim_{l\rightarrow \infty}h^l_{0m}=h_{0m},  \ \mathrm{in}\ H^{s_0+1},
		\end{split}
	\end{equation}
	By Proposition \ref{p3}, we can obtain a family of smooth solutions $(\bv_j^{l}, \rho_j^{l}, h_j^{l}, \bw_j^l)$ on a time interval\footnote{The time $T$ is uniform for $j$, because its initial data is uniformly bounded.} $[0,T]$ with the initial data $ (\bv_{0m}^{l}, \rho_{0m}^{l}, \rho_{0m}^{l}, \bw_{0m}^l)$. Let us now prove
	\begin{equation}\label{ta}
		\lim_{m\rightarrow \infty} \rho_m = \rho,  \quad \mathrm{in} \ \ H^s_x.
	\end{equation}
Inserting $\rho^{l}_m$ into $\rho_{m}-\rho$, we have
		\begin{equation}\label{cd1}
			\begin{split}
				\|\rho_{m}-\rho\|_{H_x^s}
				\leq & \|\rho^{l}_m-\rho^{l}\|_{H_x^s}+ \|\rho^{l}-\rho\|_{H_x^s}+ \|\rho^{l}_m-\rho_m\|_{H_x^s}.
			\end{split}
		\end{equation}
Due to \eqref{SSS}, we have
			\begin{equation}\label{cd4}
				\lim_{m\rightarrow \infty}\rho^l_{0m}= \rho^l_0 \quad  \mathrm{in} \ {H^\sigma_x}, \ 0\leq \sigma < \infty.
			\end{equation}
Using similar idea of the proof of \eqref{er0}, we can derive that
			\begin{equation}\label{cd5}
				\lim_{m\rightarrow \infty}\rho^l_{m}= \rho^l \quad  \mathrm{in} \ {H^\sigma_x}, \ 0 \leq \sigma < \infty.
			\end{equation}
From \eqref{cd4}, we have
			\begin{equation}\label{cd6}
				c^{(4)j}_{k} \rightarrow c^{(4)}_{k} , \ j\rightarrow \infty.
			\end{equation}
Using \eqref{cd1} and \eqref{ei0},  \eqref{cd1} yields
			\begin{equation}\label{cd3}
				\begin{split}
					\|\rho_{m}-\rho\|_{H_x^s}
					\lesssim \|\rho^{l}_m-\rho^{l}\|_{H_x^s}+ c_{\geq l}+ c^{(4)j}_{\geq l},
				\end{split}
			\end{equation}
Taking the limit of \eqref{cd3} for $j\rightarrow \infty$, by using \eqref{cd5} and \eqref{cd6}, this leads to
			\begin{equation}\label{Cd3}
				\begin{split}
					\lim_{m\rightarrow \infty }\|\rho_{m}-\rho\|_{H_x^s}
					\lesssim c_{\geq l}+ c^{(4)}_{\geq l} \lesssim c_{\geq l}.
				\end{split}
			\end{equation}
Finally, taking $l\rightarrow \infty$ for \eqref{Cd3} and using $\lim_{l\rightarrow \infty} c_{\geq l}=0$, we have
			\begin{equation}\label{cd7}
				\begin{split}
					\lim_{m\rightarrow \infty}\|\rho_{m}-\rho\|_{H_x^s}=0.
				\end{split}
			\end{equation}
Similarly, by using \eqref{dbs}-\eqref{dbs1}, we also get
			\begin{equation*}
				\lim_{m\rightarrow \infty}\|\bv_{m}-\bv\|_{H_x^s}=0, \quad \lim_{m\rightarrow \infty}\|h_{m}-h\|_{H_x^{s_0+1}}=0, \quad
			\lim_{m\rightarrow \infty}\|\bw_{m}-\bw\|_{H_x^{s_0}}=0.
			\end{equation*}
Hence, we have finished the proof of Theorem \ref{dingli3}.
		\end{proof}
It only remains for us to prove \eqref{dbs}-\eqref{dbs1}. By using the Strichartz estimates of a linear wave equation endowed with the acoustic metric, namely in Proposition \ref{p3}.

\medskip\begin{proof}[Proof of Estimates \eqref{dbs}-\eqref{dbs1}.] To prove \eqref{dbs}-\eqref{dbs1}, we also use the hyperbolic system to discuss $\bv^{l+1}-\bv^l$, $\rho^{l+1}-\rho^l$ and $h^{l+1}-h^l$. But for $\bw^{l+1}-\bw^l$, we use it's transport equations of $\bw^{l+1}$ and $\bw^l$. However, if we consider the difference terms $(\bv^{l+1}-\bv^l, \rho^{l+1}-\rho^l,h^{l+1}-h^l, \bw^{l+1}-\bw^l)$, then the original system \eqref{fc0} is destroyed or it's a disturbed system. To avoid the loss of derivatives of $\bv^{l+1}-\bv^l$ and $\rho^{l+1}-\rho^l$, the Strichartz estimates in Proposition \ref{r5} plays a key role. While, the difference $\bw^{l+1}-\bw^l$ and $h^{l+1}-h^l$  is much worse than $\bv^{l+1}-\bv^l$ for there is no dispersion.

We will prove \eqref{dbs}-\eqref{dbs1} by bootstrap arguments. At time $t=0$, by Bernstein's inequality, the data satisfies
\begin{equation}\label{lyr0}
\begin{split}
	& \|  \bv_0^{l+1}- \bv_0^{l} \|_{L_x^{2}} \leq C 2^{-sl}c_l, \qquad \ \ \|  \rho_0^{l+1}- \rho_0^{l} \|_{L_x^{2}} \leq C 2^{-sl}c_l,
\\
& \| h_0^{l+1}- h_0^{l} \|_{L_x^{2}} \leq C 2^{-(s_0+1)l} c_l,  \quad  \| \bw_0^{l+1}- \bw_0^{l} \|_{L_x^{2}} \leq C 2^{-s_0l} c_l.
\end{split}
\end{equation}
For $t \leq T$, we assume
	\begin{equation}\label{lyr1}
	\begin{split}
	& \|  (\bv^{l+1}- \bv^{l}, \rho^{l+1}- \rho^{l}, h^{l+1}- h^{l}) \|_{L_x^{2}} \leq 10C 2^{-sl}c_l,\quad \| h^{l+1}- h^{l} \|_{\dot{H}_x^{s_0+1}} \leq 10C c_l,
	\\
		& \| \bw^{l+1}- \bw^{l} \|_{L_x^{2}} \leq 10C 2^{-(s-\frac12)l} c_l,\qquad \qquad \qquad \qquad  \| \bw^{l+1}- \bw^{l} \|_{\dot{H}_x^{s_0}} \leq 10C c_l,
\end{split}
	\end{equation}
and
\begin{equation}\label{lyr2}
  \begin{split}
  & 2^l \|  \bv^{l+1}-  \bv^{l}, \rho^{l+1}-  \rho^{l}, h^{l+1}-  h^{l} \|_{L^2_{[0,t]} L^\infty_x}
  \\
  & \ + 2^l\|  \bv^{l+1}-  \bv^{l}, \rho^{l+1}-  \rho^{l}, h^{l+1}- h^{l} \|_{L^2_{[0,t]} \dot{B}^{s_0-2}_{\infty,2}}  \leq Cc_l.
\end{split}
\end{equation}
where $C$ is the constant stated in \eqref{lyr1}. We will establish the desired estimates into four steps.

\textit{Step 1: The difference of velocity and density}. Let $\bU^{l}=(\bv^l, p(\rho^l), h^l)^\mathrm{T} $. Then $\bU^{l+1}-\bU^{l}$ satisfies
			\begin{equation}\label{Fhe}
				\begin{cases}
					& A^0(\bU^{l+1}) \partial_t ( \bU^{l+1}- \bU^{l}) + A^i(\bU^{l+1}) \partial_i ( \bU^{l+1}- \bU^{l})=\Pi^l,
					\\
					& ( \bU^{l+1}-\bU^{l} )|_{t=0}= \bU_0^{l+1}-\bU_0^{l},
				\end{cases}
			\end{equation}
			where
			\begin{equation}\label{Fh}
				\Pi^l=-[A^0(\bU^{l+1})-A^0(\bU^{l}) ]\partial_t  \bU^{l}- [A^i(\bU^{l+1})-A^i(\bU^{l}) ]\partial_i  \bU^{l}.
			\end{equation}
			By energy estimates, for $a \geq 0$, we have
			\begin{equation}\label{lqe}
				\begin{split}
					\frac{d}{dt} \| \bU^{l+1}-\bU^{l} \|_{\dot{H}^a_x} \lesssim &  \| (d \bU^{l+1}, d \bU^{l}) \|_{L^\infty_x}\| \bU^{l+1}-\bU^{l} \|_{\dot{H}^a_x}
					+ \| d \bU^{l} \|_{\dot{H}^{a}_x} \|  \bU^{l+1}-  \bU^{l} \|_{L^\infty_x}
					\\
					\lesssim &  \| ( d \bU^{l+1} , d \bU^{l} )\|_{L^\infty_x} \| \bU^{l+1}-\bU^{l} \|_{\dot{H}^a_x}
					+  2^l \|  \bU^{l+1}-  \bU^{l} \|_{L^\infty_x} \| \partial \bU^{l} \|_{\dot{H}^{a-1}_x}.
				\end{split}
			\end{equation}
By using \eqref{fc0} and \eqref{lqe}, we then get
			\begin{equation}\label{fff1}
				\begin{split}
					&\| \bv^{l+1}-\bv^{l},\rho^{l+1}-\rho^{l},h^{l+1}-h^{l}\|_{L^2_x}
					\\
\lesssim &  \| \bv_0^{l+1}-\bv_0^{l}, \rho_0^{l+1}-\rho_0^{l}, h_0^{l+1}-h_0^{l} \|_{L^2_x}
					\\
					& + \int^t_0 \| (\partial \bv^{l+1}, \partial \rho^{l+1}, \partial h^{l+1}) \|_{L^\infty_x}\| \bv^{l+1}-\bv^{l},\rho^{l+1}-\rho^{l},h^{l+1}-h^{l} \|_{L^2_x} d\tau
					\\
					& + \int^t_0 \| (\partial \bv^{l}, \partial \rho^{l}, \partial h^{l}) \|_{L^\infty_x}\| \bv^{l+1}-\bv^{l},\rho^{l+1}-\rho^{l},h^{l+1}-h^{l} \|_{L^2_x} d\tau
					\\
					&
					+  \int^t_0 \| \partial \bv^{l}, \partial \rho^{l}, \partial h^{l}\|_{ L^2_x } \|  \bv^{l+1}-  \bv^{l}, \rho^{l+1}-  \rho^{l},h^{l+1}-  h^{l} \|_{ L^\infty_x} d\tau.
				\end{split}
			\end{equation}
By Gronwall's inequality, we have
	\begin{equation}\label{fff2}
		\begin{split}
		\| \bv^{l+1}-\bv^{l},\rho^{l+1}-\rho^{l} \|_{L^2_x}  \leq & \| \bv_0^{l+1}-\bv_0^{l}, \rho_0^{l+1}-\rho_0^{l}\|_{L^2_x} A(t)
 \leq  2^{-sl}c_l A(t),
	\end{split}
	\end{equation}
where
\begin{equation*}
	\begin{split}
		A(t)= & C\exp \left\{ \int^t_0 \big( \| (\partial \bv^{l}, \partial \rho^{l}, \partial h^{l}) \|_{L^\infty_x}  + 2^l \|  \bv^{l+1}-  \bv^{l}, \rho^{l+1}-  \rho^{l}, h^{l+1}- h^{l} \|_{ L^\infty_x} \big) d\tau \right\}.
	\end{split}
\end{equation*}
By using \eqref{lyr2}, for $t \in [0, \frac{1}{400(1+C_0)^2}T]$, \eqref{fff2} yields
	\begin{equation}\label{lyr3}
		\begin{split}
		\| \bv^{l+1}-\bv^{l},\rho^{l+1}-\rho^{l},h^{l+1}-h^{l} \|_{L^2_x}  \leq 4C 2^{-sl}c_l.
	\end{split}
	\end{equation}
\textit{Step 2: The difference of the specific vorticity}. Due to $\bw^l=\mathrm{e}^{-\rho^l}\mathrm{curl}\bv^{l}$, so we obtain
 \begin{equation}\label{lyr4}
   \begin{split}
  \|\bw^{l+1}- \bw^l\|_{L^2_x} \lesssim \|\partial (\bv^{l+1}- \bv^l)\|_{L^2_x} \leq 8C 2^{-(s-1)l}c_l.
\end{split}
\end{equation}
For $\| \bw^{l+1} - \bw^l \|_{\dot{H}^{s_0}}$, using elliptic estimates, we have
\begin{equation}
	\| \bw^{l+1} - \bw^l \|_{\dot{H}^{s_0}} \lesssim \| \textrm{div}(\bw^{l+1} - \bw^l) \|_{\dot{H}^{s_0-1}}+\| \textrm{curl}(\bw^{l+1} - \bw^l) \|_{\dot{H}^{s_0-1}}.
\end{equation}
Noting
\begin{equation*}
  \textrm{div}(\bw^{l+1}-\bw^l)=\partial_i (\bw^{l+1}-\bw^l) \partial^i \rho^{l+1}+\partial_i \bw^l \partial^i (\rho^{l+1}-\rho^l),
\end{equation*}
and using \eqref{gg0}, H\"older's inequality and \eqref{lyr1}, we derive that
\begin{equation}\label{divwl}
\begin{split}
	\| \textrm{div}(\bw^{l+1} - \bw^l) \|_{\dot{H}_x^{s_0-1}} \leq & \|\bw^{l+1} - \bw^l \|_{{H}_x^{\frac{3}{2}+}} \|\rho^l\|_{{H}_x^{s_0}}+\| \bw^l \|_{{H}_x^{\frac{3}{2}+}} \| \rho^{l+1}- \rho^l\|_{{H}_x^{s_0}}
\\
 \leq & 3C c_l.
\end{split}
\end{equation}
On the other hand, we also have
\begin{equation}\label{gg0}
	\begin{split}
		& \mathbf{T}^l ( \mathrm{curl}^i \mathrm{curl} \bw^{l}+F^{il}  )
		=   \partial^i \big( 2 \partial_n v^l_a \partial^n w^{al} \big) + K^{il}, \quad i=1,2,3,
	\end{split}
\end{equation}
where $K^{il}$ has the same formulations with $K^i$ by replacing $(\bv,\rho,h,\bw)$ to $(\bv^l,\rho^l,h^l,\bw^l)$, and
\begin{equation*}
	F^{il}= -\epsilon^{ijk} \partial_j \rho^l \cdot \mathrm{curl}_k\bw^l- 2\partial^a {\rho^l} \partial^i w^l_a+2 \mathrm{e}^{-\rho^l} \epsilon^{ijk} \partial_j v^{ml}  \partial_m\partial_k h^l-\mathrm{e}^{-\rho^l} \epsilon^{ijk} \partial_k h^l \Delta v^l_j.
\end{equation*}
So we have
\begin{equation}\label{divf}
	\begin{split}
	\mathrm{div} \bF^{l}=\partial_i F^{il}= & -\epsilon^{ijk} \partial_j \rho^l \cdot \partial_i \mathrm{curl}_k\bw^l- 2\partial_i \partial^a {\rho^l} \partial^i w^l_a
	\\
	& - 2\partial^a {\rho^l} \Delta w^l_a-2 \mathrm{e}^{-\rho^l}  \epsilon^{ijk} \partial \rho^l \partial_j v^{lm}  \partial_m\partial_k h^l
\\
&	+\mathrm{e}^{-\rho^l} \epsilon^{ijk} \partial_i \rho^l \partial_k h^l \Delta v^l_j-\mathrm{e}^{-\rho^l}  \partial_k h^l \Delta (\mathrm{e}^{\rho^l}w^{il}).
\end{split}
\end{equation}
Using the Hodge decomposition, we have
\begin{equation}\label{curlwl}
	\| \textrm{curl}(\bw^{l+1} - \bw^l) \|_{\dot{H}^{s_0-1}} \leq \|\textrm{curl}\textrm{curl}(\bw^{l+1} - \bw^l) \|_{\dot{H}_x^{s_0-2}} .
\end{equation}
Using \eqref{gg0}, we have
\begin{equation}\label{gg1}
	\begin{split}
		& \mathbf{T}^l \{  \mathrm{curl}^i \mathrm{curl} ( \bw^{l+1}-\bw^{l})+F^{i(l+1)}- F^{il}  \}
		\\
		=&    \partial^i \{  2 \partial_n (v^{l+1}_a-v^{l}_a) \partial^n w^{a(l+1)} \} +\partial^i \{  2 \partial_n v^l_a \partial^n (w^{a(l+1)}-w^{al} ) \}+ K^{i(l+1)}-K^{il}.
	\end{split}
\end{equation}
For simplicity, we denote
\begin{equation}\label{Gt}
	G(t)= \| \mathrm{curl}^i \mathrm{curl} ( \bw^{l+1}-\bw^{l})+F^{i(l+1)}- F^{il} \|_{\dot{H}^{s_0-2}_x}.
\end{equation}Operating $\Lambda^{s_0-2}_x$ on \eqref{gg1}, multiplying $\Lambda^{s_0-2}_x\left\{ \mathrm{curl}_i \mathrm{curl} ( \bw^{l+1}-\bw^{l})+F_i^{(l+1)}- F_i^{l}  \right\} $, and integrating it on $[0,t] \times \mathbb{R}^3$, we then obtain
\begin{equation}\label{G2t}
	\begin{split}
		G^2(t) \leq & G^2(0) + G^a+G^b+G^c,
	\end{split}
\end{equation}
where
	\begin{equation*}%\label{G2t0}
		\begin{split}
			G^a=&
			| \int^t_0 \int_{\mathbb{R}^3} \Lambda^{s_0-2}_x \partial^i \{  2 \partial_n (v^{l+1}_a-v^{l}_a) \partial^n w^{a(l+1)} \}
			\cdot \Lambda^{s_0-2}_x \{  \mathrm{curl}_i \mathrm{curl} ( \bw^{l+1}-\bw^{l})+F_i^{l+1}- F_i^{l}  \} dxd\tau|,
			\\
			G^b=&  | \int^t_0 \int_{\mathbb{R}^3} \Lambda^{s_0-2}_x \partial^i \{  2 \partial_n v^l_a \partial^n (w^{a(l+1)}-w^{al} ) \}
			\cdot \Lambda^{s_0-2}_x \{  \mathrm{curl}_i \mathrm{curl} ( \bw^{l+1}-\bw^{l})+F_i^{l+1}- F_i^{l}  \} dxd\tau|,
			\\
			G^c=&  | \int^t_0 \int_{\mathbb{R}^3} \Lambda^{s_0-2}_x \{ K^{i(l+1)}-K^{il} \}
			\cdot \Lambda^{s_0-2}_x \{  \mathrm{curl}_i \mathrm{curl} ( \bw^{l+1}-\bw^{l})+F_i^{l+1}- F_i^{l} \} dxd\tau|.
		\end{split}
	\end{equation*}
Using \eqref{lyr1}, it's direct for us to get
\begin{equation}\label{G2t0}
G^2(0) \leq \frac32(Cc_l)^2.
\end{equation}
Integrating by parts for $G^a$, we have
	\begin{equation*}
		\begin{split}
			G^a\leq &
			\left| \int^t_0 \int_{\mathbb{R}^3} \Lambda^{s_0-2}_x  \{  2 \partial_n (v^{l+1}_a-v^{l}_a) \partial^n w^{a(l+1)} \}
			\cdot \Lambda^{s_0-2}_x \partial^i \{  \mathrm{curl}_i \mathrm{curl} ( \bw^{l+1}-\bw^{l})+F_i^{(l+1)}- F_i^{l}  \} dxd\tau \right|
			\\
			=&  \left| \int^t_0 \int_{\mathbb{R}^3} \Lambda^{s_0-2}_x  \{  2 \partial_n (v^{l+1}_a-v^{l}_a) \partial^n w^{a(l+1)} \}
			\cdot \Lambda^{s_0-2}_x \partial^i \{  F_i^{(l+1)}- F_i^{l}  \} dxd\tau \right|.
		\end{split}
	\end{equation*}
By H\"older's inequality and Lemma \ref{lpe} and \eqref{divf}, we have
\begin{small}
	\begin{equation*}
	\begin{split}
		 G^a
		\lesssim & \int^t_0 \| \bv^{l+1}-\bv^{l} \|_{H^{\frac32+}_x}\| \bw^{l} \|^2_{H^{s_0}_x}   \| \rho^{l+1}-\rho^{l} \|_{H^{s_0}_x} d\tau
		 \\
& +  \int^t_0 \| \bv^{l+1}-\bv^{l} \|_{H^{\frac32+}_x}\| \bw^{l} \|_{H^{s_0}_x} ( \|\partial \rho^{l} \|_{\dot{B}^{s_0-2}_{\infty,2} }+ \|\partial h^{l} \|_{\dot{B}^{s_0-2}_{\infty,2} }) \| \bw^{l+1}-\bw^{l} \|^2_{H^{s_0}_x}
		  \\
		 & +  \int^t_0 \| \bv^{l+1}-\bv^{l} \|_{H^{\frac32+}_x}\| \bw^{l} \|_{H^{s_0}_x} ( \|\partial( \rho^{l+1}-\rho^{l} ) \|_{\dot{B}^{s_0-2}_{\infty,2} }+\|\partial( h^{l+1}-h^{l} ) \|_{\dot{B}^{s_0-2}_{\infty,2} }) \|\bw^l \|_{H^{s_0}_x}d\tau
		 \\
& + \int^t_0 \|\bv^{l+1}-\bv^{l}  \|_{H^2_x}\|\bw^l \|_{H^{s_0}_x} ( \|\rho^{l+1}-\rho^{l}  \|_{H^2_x} \|\bw^l \|^2_{H^{s_0}_x}d\tau
 + \|\bw^{l+1}-\bw^{l}  \|_{H^2_x} \|\rho^l \|_{H^{s_0}_x}) d\tau
\\
& + \int^t_0 \|\rho^{l+1}-\rho^{l}  \|_{H^2_x}\|\bw^{l+1}-\bw^{l}  \|_{H^2_x} \|\bv^l \|_{H^{s_0}_x}\|\rho^l \|_{H^{s_0}_x}d\tau
	\\
&	 + \int^t_0 \| \bv^{l+1}-\bv^{l} \|_{H^{\frac32+}_x}\| \bw^{l} \|^2_{H^{s_0}_x} \|\rho^l \|_{H^{2}_x}\|h^l \|_{H^{3}_x}  \| \bv^{l+1}-\bv^{l} \|_{H^{s_0}_x} d\tau
		\\
		& +\int^t_0 \| \bv^{l+1}-\bv^{l} \|_{H^{\frac32+}_x}\| \bw^{l} \|^2_{H^{s_0}_x} \|\rho^{l+1}-\rho^l \|_{H^{2}_x}\|h^l \|_{H^{3}_x}  \| \bv^{l} \|_{H^{s_0}_x} d\tau
		 \\
& +\int^t_0 \| \bv^{l+1}-\bv^{l} \|_{H^{\frac32+}_x}\| \bw^{l} \|^2_{H^{s_0}_x} \|\rho^l \|_{H^{2}_x}\|h^{l+1}-h^l \|_{H^{3}_x}  \| \bv^{l} \|_{H^{s_0}_x} d\tau
		\\
		& +\int^t_0 \| \bv^{l+1}-\bv^{l} \|_{H^{\frac32+}_x}\| \bw^{l} \|^2_{H^{s_0}_x} \|\rho^l \|_{H^{2}_x}\|h^l \|_{H^{3}_x}  \| \bv^{l+1}-\bv^l \|_{H^{s_0}_x} d\tau.
	\end{split}
\end{equation*}
\end{small}
By using \eqref{lyr1}, for $t\in [0, \frac{1}{64(1+C)^2}T]$, we have
\begin{equation}\label{G2t1}
  G^a \leq \frac32(Cc_l)^2.
\end{equation}
Using a similar way to handle $G^a$, we can bound $G^b$ by
\begin{small}
\begin{equation*}%\label{Gb}
	\begin{split}
		G^b \lesssim & \int^t_0 \| \bv^{l+1}-\bv^{l} \|_{H^{\frac32+}_x}\| \bw^{l} \|^2_{H^{s_0}_x}   \| \rho^{l+1}-\rho^{l} \|_{H^{s_0}_x} d\tau
		 \\
& +  \int^t_0 \| \bw^{l+1}-\bw^{l} \|_{H^{\frac32+}_x}\| \bv^{l} \|_{H^{s_0}_x} ( \|\partial \rho^{l} \|_{\dot{B}^{s_0-2}_{\infty,2} }+ \|\partial h^{l} \|_{\dot{B}^{s_0-2}_{\infty,2} }) \| \bw^{l+1}-\bw^{l} \|^2_{H^{s_0}_x}
		  \\
		 & +  \int^t_0 \| \bw^{l+1}-\bw^{l} \|_{H^{\frac32+}_x}\| \bv^{l} \|_{H^{s_0}_x} ( \|\partial( \rho^{l+1}-\rho^{l} ) \|_{\dot{B}^{s_0-2}_{\infty,2} }+\|\partial( h^{l+1}-h^{l} ) \|_{\dot{B}^{s_0-2}_{\infty,2} }) \|\bw^l \|_{H^{s_0}_x}d\tau
		 \\
& + \int^t_0 \|\bw^{l+1}-\bw^{l}  \|_{H^2_x}\|\bv^l \|_{H^{s_0}_x} ( \|\rho^{l+1}-\rho^{l}  \|_{H^2_x} \|\bw^l \|^2_{H^{s_0}_x}d\tau
 + \|\bw^{l+1}-\bw^{l}  \|_{H^2_x} \|\rho^l \|_{H^{s_0}_x}) d\tau
\\
& + \int^t_0 \|\rho^{l+1}-\rho^{l}  \|_{H^2_x}\|\bw^{l+1}-\bw^{l}  \|_{H^2_x} \|\bv^l \|_{H^{s_0}_x}\|\rho^l \|_{H^{s_0}_x}d\tau
	\\
&	 + \int^t_0 \| \bw^{l+1}-\bw^{l} \|_{H^{\frac32+}_x}\| \bv^{l} \|_{H^{s_0}_x}\| \bw^{l} \|_{H^{s_0}_x} \|\rho^l \|_{H^{2}_x}\|h^l \|_{H^{3}_x}  \| \bv^{l+1}-\bv^{l} \|_{H^{s_0}_x} d\tau
		\\
		& +\int^t_0 \| \bw^{l+1}-\bw^{l} \|_{H^{\frac32+}_x}\|_{H^{s_0}_x}\| \bw^{l} \|_{H^{s_0}_x} \|\rho^{l+1}-\rho^l \|_{H^{2}_x}\|h^l \|_{H^{3}_x}  \| \bv^{l} \|_{H^{s_0}_x} d\tau
		 \\
& +\int^t_0 \| \bw^{l+1}-\bw^{l} \|_{H^{\frac32+}_x}\|_{H^{s_0}_x}\| \bw^{l} \|_{H^{s_0}_x} \|\rho^l \|_{H^{2}_x}\|h^{l+1}-h^l \|_{H^{3}_x}  \| \bv^{l} \|_{H^{s_0}_x} d\tau
		\\
		& +\int^t_0 \| \bw^{l+1}-\bw^{l} \|_{H^{\frac32+}_x}\|_{H^{s_0}_x}\| \bw^{l} \|_{H^{s_0}_x} \|\rho^l \|_{H^{2}_x}\|h^l \|_{H^{3}_x}  \| \bv^{l+1}-\bv^l \|_{H^{s_0}_x} d\tau.
	\end{split}
\end{equation*}
\end{small}
By using \eqref{lyr1}, for $t\in [0, \frac{1}{400(1+C)^2}T]$, we have
\begin{equation}\label{G2t2}
  G^b \leq \frac32(Cc_l)^2.
\end{equation}
By H\"older's inequality and Lemma \ref{lpe}, we have
\begin{equation*}%\label{Gc}
	\begin{split}
		G^c \lesssim  &  \int^t_0 ( \|\partial( \bv^{l+1}-\bv^{l} ) \|_{\dot{B}^{s_0-2}_{\infty,2}} \| \bw^l \|_{ H^{s_0}_x}+\|\partial \bv^{l} \|_{\dot{B}^{s_0-2}_{\infty,2}} \| \bw^{l+1}-\bw^l \|_{H^{s_0}_x} ) G(\tau) d\tau
		\\
		&  + \int^t_0 ( \|\partial( \rho^{l+1}-\rho^{l} ) \|_{ \dot{B}^{s_0-2}_{\infty,2}} \| h^l \|_{ H^{s_0+1}_x}+\|\partial \rho^{l} \|_{\dot{B}^{s_0-2}_{\infty,2}} \| h^{l+1}-h^l \|_{H^{s_0+1}_x} ) G(\tau)d\tau
		\\
		&  + \int^t_0 ( \|\partial( h^{l+1}-h^{l} ) \|_{\dot{B}^{s_0-2}_{\infty,2}} \| h^l \|_{ H^{s_0+1}_x}+\|\partial h^{l} \|_{ \dot{B}^{s_0-2}_{\infty,2}} \| h^{l+1}-h^l \|_{ H^{s_0+1}_x} ) G(\tau)d\tau
		\\
		&+  \int^t_0 ( \| \bv^{l+1}-\bv^l \|_{ H^{s_0}_x}\|\bv^l \|_{ H^{s_0}_x}\|h^l \|_{ H^{s_0+1}_x}
	 + \| \bv^l \|^2_{ H^{s_0}_x} \|h^{l+1}-h^l \|_{ H^{s_0+1}_x} )G(\tau)d\tau
	 	\\
	 &+  \int^t_0 ( \| \rho^{l+1}-\rho^l \|_{ H^{s_0}_x}\|h^l \|^2_{H^{s_0+1}_x}
	 + \| \rho^l \|^2_{ H^{s_0}_x}\|h^l \|_{ H^{s_0+1}_x} \|h^{l+1}-h^l \|_{ H^{s_0+1}_x} )G(\tau)d\tau
	 \\
	 &+  \int^t_0 ( \| \rho^{l+1}-\rho^l \|_{ H^{s_0}_x}\|h^l \|_{ H^{s_0+1}_x}\|\rho^l \|^2_{ H^{s_0}_x}
	 + \| \rho^l \|^2_{ H^{s_0}_x}\|h^{l+1}-h^l \|_{ H^{s_0+1}_x}  )G(\tau)d\tau
	 \\
	 & + \int^t_0 ( \| \rho^{l+1}-\rho^l \|_{ H^{s_0}_x}\|h^l \|_{ H^{s_0+1}_x}\|\rho^l \|^2_{ H^{s_0}_x}
	 + \| \rho^l \|^3_{H^{s_0}_x}\|h^{l+1}-h^l \|_{ H^{s_0+1}_x}  )G(\tau)d\tau .
	\end{split}
\end{equation*}
By using \eqref{lyr1}, for $t\in [0, \frac{1}{400(1+C)^2}T]$, we have
\begin{equation}\label{G2t3}
  G^c \leq \frac32(Cc_l)^2.
\end{equation}
Gathering \eqref{G2t}-\eqref{G2t3}, we show that
\begin{equation}\label{G2t4}
  G(t) \leq 3Cc_l.
\end{equation}
Hence, we get
\begin{equation}\label{G2t5}
\begin{split}
  \|\textrm{curl}\textrm{curl}(\bw^{l+1} - \bw^l) \|_{\dot{H}_x^{s_0-2}} \leq &G(t)+ \|\bF^{l+1}-\bF^l \|_{\dot{H}_x^{s_0-2}}
  \\
  \leq & 5Cc_l.
\end{split}
\end{equation}
By \eqref{divwl}, \eqref{curlwl} and \eqref{G2t5}, we have proved
\begin{equation}\label{G2t6}
  \|\bw^{l+1} - \bw^l \|_{\dot{H}_x^{s_0}} \leq 8Cc_l.
\end{equation}
\textit{Step 3: The difference of the highest derivatives of entropy}. Due to \eqref{lyr3}, we only need to discuss $\|h^{l+1}-h^{l}\|_{\dot{H}_x^{1+s_0}}$. Note $h^l$ satisfying \eqref{Hh2}. Hence, for $l \in \mathbb{Z}^+$, we get
\begin{equation}\label{fff3}
	\begin{split}
		(\partial_t+ \bv^l \cdot \nabla) \left\{ \partial^k( \Delta h^l- \partial^i h^l \partial_i \rho^l) \right\}=&-2\partial^k(\partial_i v^{jl}) \partial_j  \partial^i h^l+Y^{kl},
	\end{split}
\end{equation}
where
\begin{equation}\label{fff4}
	\begin{split}
		Y^{kl}=& -2\partial_i v^{lj} \partial^k  \partial^i (\partial_j h^l) + \epsilon^{ijm}\mathrm{e}^{\rho^l} w^l_m \partial^i \rho^l \partial^k\rho^l \partial_j h^l + \epsilon^{ijm}\mathrm{e}^{\rho^l} \partial^k w^l_m \partial^i \rho^l \partial_j h^l
		\\
		& + \epsilon^{ijm}\mathrm{e}^{\rho^l} w^l_m \partial_j h^l \partial_k \partial^i \rho^l   + \epsilon^{ijm}\mathrm{e}^{\rho^l} w^l_m \partial^i \rho^l \partial^k \partial_j h^l .
	\end{split}
\end{equation}
As a result, we find that
\begin{equation}\label{fff5}
	\begin{split}
		& (\partial_t+ \bv^{l+1} \cdot \nabla) \left\{ \partial^k( \Delta h^{l+1}-\Delta h^{l}- \partial^i h^{l+1} \partial_i \rho^{l+1}- \partial^i h^{l} \partial_i \rho^{l}) \right\}
		\\
		=&(\bv^{l+1}-\bv^{l})\cdot \nabla \left\{ \partial^k( \Delta h^l- \partial^i h^l \partial_i \rho^l) \right\}
		 -2\partial^k(\partial_i v^{j(l+1)}-\partial_i v^{jl}) \partial_j  \partial^i h^{l+1}
		\\
		&-2\partial^k\partial_i v^{jl} \partial_j  \partial^i (h^{l+1}-h^{l})+Y^{k(l+1)k}-Y^{kl},
	\end{split}
\end{equation}
where
\begin{equation}\label{fffa}
	\begin{split}
		Y^{lk}=& -2\partial_i v^{lj} \partial^k  \partial^i (\partial_j h^l) + \epsilon^{ijm}\mathrm{e}^{\rho^l} w^l_m \partial^i \rho^l \partial^k\rho^l \partial_j h^l + \epsilon^{ijm}\mathrm{e}^{\rho^l} \partial^k w^l_m \partial^i \rho^l \partial_j h^l
		\\
		& + \epsilon^{ijm}\mathrm{e}^{\rho^l} w^l_m \partial_j h^l \partial_k \partial^i \rho^l   + \epsilon^{ijm}\mathrm{e}^{\rho^l} w^l_m \partial^i \rho^l \partial^k \partial_j h^l .
	\end{split}
\end{equation}
For simplicity, we denote
\begin{equation}\label{BT}
	B(t)=\| \Delta h^{l+1}-\Delta h^{l}- \partial^i h^{l+1} \partial_i \rho^{l+1}- \partial^i h^{l} \partial_i \rho^{l} \|_{\dot{H}_x^{s_0-1}}.
\end{equation}
Operating $\Lambda^{s_0-2}_x$ on \eqref{fff5}, and multiplying $\Lambda^{s_0-2}_x\left\{ \partial^k( \Delta h^{l+1}-\Delta h^{l}- \partial^i h^{l+1} \partial_i \rho^{l+1}- \partial^i h^{l} \partial_i \rho^{l}) \right\} $, and integrating it on $[0,t] \times \mathbb{R}^3$, we then obtain
\begin{equation}\label{fff6}
	\begin{split}
		B^2(t)
		\leq &  B^2(0) +C\int^t_0 \|\partial \bv^{l+1}\|_{L^\infty_x}B^2(\tau) d\tau + V_1+ V_2+ V_3+ V_4,
		%\\
		%\leq &  B^2(0) +C_0 \int^t_0 \|\partial \bv^{l+1}\|_{L^2_tL^\infty_x}B^2(\tau)d\tau + V_1+ V_2+ V_3+ V_4,
	\end{split}
\end{equation}
where
\begin{equation*}
	\begin{split}
	V_1=	&  C\big|\int^t_0 \int_{\mathbb{R}^3} \Lambda^{s_0-2}_x\left\{ (\bv^{l+1}-\bv^{l})\cdot \nabla \partial^k( \Delta h^l- \partial^i h^l \partial_i \rho^l) \right\}
	\\
& \quad \cdot \Lambda^{s_0-2}_x\left\{ \partial_k( \Delta h^{l+1}-\Delta h^{l}- \partial^i h^{l+1} \partial_i \rho^{l+1}- \partial^i h^{l} \partial_i \rho^{l}) \right\}  dx d\tau \big|
		\\
	V_2=	& C\big| \int^t_0 \int_{\mathbb{R}^3} \Lambda^{s_0-2}_x \left\{ \partial^k(\partial_i v^{j(l+1)}-\partial_i v^{jl}) \partial_j  \partial^i h^{l+1} \right\}
	\\
	& \quad  \cdot \Lambda^{s_0-2}_x\left\{ \partial_k( \Delta h^{l+1}-\Delta h^{l}- \partial^i h^{l+1} \partial_i \rho^{l+1}- \partial^i h^{l} \partial_i \rho^{l}) \right\}  dxd\tau \big|
		\\
	V_3=	& C \big| \int^t_0 \int_{\mathbb{R}^3} \Lambda^{s_0-2}_x \left\{ \partial^k\partial_i v^{jl} \partial_j  \partial^i (h^{l+1}-h^{l}) \right\}
	\\
	& \quad \cdot \Lambda^{s_0-2}_x\left\{ \partial_k( \Delta h^{l+1}-\Delta h^{l}- \partial^i h^{l+1} \partial_i \rho^{l+1}- \partial^i h^{l} \partial_i \rho^{l}) \right\}  dxd\tau \big|
		\\
	V_4=	& C \big| \int^t_0  \int_{\mathbb{R}^3} \Lambda^{s_0-2}_x(Y^{k(l+1)}-Y^{kl})
	\\
	& \quad \cdot \Lambda^{s_0-2}_x\left\{ \partial_k( \Delta h^{l+1}-\Delta h^{l}- \partial^i h^{l+1} \partial_i \rho^{l+1}- \partial^i h^{l} \partial_i \rho^{l}) \right\}  dxd\tau \big|.
	\end{split}
\end{equation*}
We will estimate the right terms in \eqref{fff6} one by one. By Lemma \ref{lpe}, Bernstein's inequality and \eqref{lyr1}, we can bound $V_1$ as
\begin{equation}\label{V2t0}
	\begin{split}
		V_1 \leq	&  C\int^t_0 2^l\| \bv^{l+1}-\bv^{l} \|_{\dot{B}^{s_0-2}_{\infty,2}}\cdot \| \Delta h^l- \partial^i h^l \partial_i \rho^l\|_{L^\infty_t \dot{H}^{s_0-1}_x } \cdot B(\tau)d\tau
\\
& \leq \frac32 (Cc_l)^2.
	\end{split}
\end{equation}
Using Lemma \ref{lpe}, the Bernstein inequality and \eqref{lyr1} again, we have
\begin{equation}\label{V2t1}
	\begin{split}
		V_3 \leq	&  C\int^t_0 \| \partial \bv^l \|_{\dot{B}^{s_0-2}_{\infty,2}}\cdot 2^l\|  h^{l+1}- h^l \|_{ \dot{H}^{s_0}_x } \cdot B(\tau) d\tau
\\
& \leq \frac32 (Cc_l)^2.
	\end{split}
\end{equation}
For $V_4$, by using \eqref{fffa}, Lemma \ref{lpe}, and Lemma \ref{wql}, we can derive that
\begin{equation*}
	\begin{split}
		 V_4  \leq	&  C\int^t_0 \| \partial \bv^l \|_{\dot{B}^{s_0-2}_{\infty,2}}\cdot \|  h^{l+1}- h^l \|_{ \dot{H}^{s_0+1}_x } \cdot B(\tau)d\tau
		\\
		& + C\int^t_0 \| \partial ( \bv^{l+1} - \bv^{l}) \|_{\dot{B}^{s_0-2}_{\infty,2}}\cdot \|  h^l \|_{ \dot{H}^{s_0+1}_x } \cdot B(\tau)d\tau
		\\
		& + C\int^t_0 \|  \bw^{l+1} - \bw^{l} \|_{H_x^{s_0}}\cdot \|  \rho^l \|_{ {H}^{s}_x }(1+\|  \rho^l \|_{ {H}^{s}_x }) \|  h^l \|_{ {H}^{s_0+1}_x } \cdot B(\tau)d\tau
		\\
		& + C\int^t_0 \| \bw^{l} \|_{ H_x^{s_0}}\cdot \| \rho^{l+1}- \rho^l \|_{ {H}^{s}_x }(1+\|  \rho^l \|_{ {H}^{s}_x }) \|  h^l \|_{{H}^{s_0+1}_x } \cdot B(\tau)d\tau
		\\
		& + C\int^t_0 \|  \bw^{l} \|_{ H_x^{s_0}}\cdot \|  \rho^l \|_{ {H}^{s}_x }(1+\|  \rho^l \|_{ {H}^{s}_x }) \|  h^{l+1}-h^l \|_{ {H}^{s_0+1}_x } \cdot B(\tau)d\tau.
	\end{split}
\end{equation*}
By using \eqref{lyr1}, we get
\begin{equation}\label{V2t2}
	\begin{split}
		V_3  \leq \frac32 (Cc_l)^2.
	\end{split}
\end{equation}
For $V_2$, let us integrate it by parts
\begin{equation*}
	\begin{split}
		 V_2 =	& C\big| \int^t_0 \int_{\mathbb{R}^3} \Lambda^{s_0-2}_x \left\{ \partial^k(\partial_i v^{j(l+1)}-\partial_i v^{jl}) \partial_j  \partial^i h^{l+1} \right\}
			\\
			& \quad  \cdot \Lambda^{s_0-2}_x\left\{ \partial_k( \Delta h^{l+1}-\Delta h^{l}- \partial^i h^{l+1} \partial_i \rho^{l+1}- \partial^i h^{l} \partial_i \rho^{l}) \right\}  dxd\tau \big|
		\\
		= &  C\big| \int^t_0 \int_{\mathbb{R}^3} \Lambda^{s_0-2}_x \left\{ \partial^k(\Delta v^{j(l+1)}-\Delta v^{jl}) \partial_j   h^{l+1} \right\}
		\\
		& \quad  \cdot \Lambda^{s_0-2}_x\left\{ \partial_k( \Delta h^{l+1}-\Delta h^{l}- \partial^i h^{l+1} \partial_i \rho^{l+1}- \partial^i h^{l} \partial_i \rho^{l}) \right\}  dxd\tau \big|
		\\
		&+ C\big| \int^t_0 \int_{\mathbb{R}^3} \Lambda^{s_0-2}_x \left\{ \partial^k(\partial_i v^{j(l+1)}-\partial_i v^{jl}) \partial_j   h^{l+1} \right\}
		\\
		& \quad  \cdot \Lambda^{s_0-2}_x\left\{ \partial^i \partial_k( \Delta h^{l+1}-\Delta h^{l}- \partial^i h^{l+1} \partial_i \rho^{l+1}- \partial^i h^{l} \partial_i \rho^{l}) \right\}  dxd\tau \big|
		\\
		= &  C\big| \int^t_0 \int_{\mathbb{R}^3} \Lambda^{s_0-2}_x \left\{ \partial^k(\Delta v^{j(l+1)}-\Delta v^{jl}) \partial_j   h^{l+1} \right\}
		\\
		& \quad  \cdot \Lambda^{s_0-2}_x\left\{ \partial_k( \Delta h^{l+1}-\Delta h^{l}- \partial^i h^{l+1} \partial_i \rho^{l+1}- \partial^i h^{l} \partial_i \rho^{l}) \right\}  dxd\tau \big|
		\\
		&+ C\big| \int^t_0 \int_{\mathbb{R}^3} \Lambda^{s_0-2}_x \left\{ \partial_i(\Delta v^{j(l+1)}-\Delta v^{jl}) \partial_j   h^{l+1} \right\}
		\\
		& \qquad  \cdot \Lambda^{s_0-2}_x\left\{ \partial^i ( \Delta h^{l+1}-\Delta h^{l}- \partial^i h^{l+1} \partial_i \rho^{l+1}- \partial^i h^{l} \partial_i \rho^{l}) \right\}  dxd\tau \big|
		\\
		=& V_2^A + V_2^B .
	\end{split}
\end{equation*}
Here
\begin{equation*}
	\begin{split}
		V_2^A=& C\big| \int^t_0 \int_{\mathbb{R}^3} \Lambda^{s_0-2}_x \left\{ \partial^k(\Delta v^{j(l+1)}-\Delta v^{jl}) \partial_j   h^{l+1} \right\}
		\\
		& \quad  \cdot \Lambda^{s_0-2}_x\left\{ \partial_k( \Delta h^{l+1}-\Delta h^{l}- \partial^i h^{l+1} \partial_i \rho^{l+1}- \partial^i h^{l} \partial_i \rho^{l}) \right\}  dxd\tau \big|,
		\\
		V_2^B=& C\big| \int^t_0 \int_{\mathbb{R}^3} \Lambda^{s_0-2}_x \left\{ \partial_i(\Delta v^{j(l+1)}-\Delta v^{jl}) \partial_j   h^{l+1} \right\}
		\\
		& \qquad  \cdot \Lambda^{s_0-2}_x\left\{ \partial^i ( \Delta h^{l+1}-\Delta h^{l}- \partial^i h^{l+1} \partial_i \rho^{l+1}- \partial^i h^{l} \partial_i \rho^{l}) \right\}  dxd\tau \big|.
	\end{split}
\end{equation*}
For $V_2^A$ and $V_2^B$, there is a derivative loss, so we use the Hodge decomposition and \eqref{fc0} to transfer the derivatives. By
%the Hodge decomposition and
Equation \eqref{Hod}, we have
\begin{equation*}
	\Delta  v^{jl}= -\mathbf{T}^l (\partial^j {\rho}^l)+ \partial^j v^{kl} \partial_k {\rho}^l+ \mathrm{curl}^j \big( \mathrm{e}^{{\rho}^l}\bw^l \big),
\end{equation*}
where $\mathbf{T}^l=\partial_t  + \bv^l \cdot \nabla$. As a result, we get
\begin{equation}\label{ffg0}
\begin{split}
	\Delta  (v^{j(l+1)}-v^{jl})= &-\mathbf{T}^{l+1} (\partial^j {\rho}^{l+1}-\partial^j {\rho}^{l})+(\bv^{l}-\bv^{l+1})\cdot \nabla \partial^j {\rho}^{l}
	\\
	& + \mathrm{curl}^j \big( \mathrm{e}^{{\rho}^{l+1}}(\bw^{l+1}-\bw^l) \big)
	\ -\epsilon^{jmk}\partial_m (\mathrm{e}^{{\rho}^{l+1}}-\mathrm{e}^{{\rho}^{l}})\bw^{l}_k
	\\
	& + \partial^j v^{k(l+1)} \partial_k {\rho}^{l+1}-\partial^j v^{kl} \partial_k {\rho}^l .
\end{split}
\end{equation}
Substituting \eqref{ffg0} in $V^A_2$, gives
\begin{equation}\label{V2A}
	 V^A_2 \leq  V^{A1}_2+V^{A2}_2+V^{A3}_2+V^{A4}_2,
\end{equation}
where
\begin{equation*}
	\begin{split}
		 V^{A1}_2  =& C\big| \int^t_0 \int_{\mathbb{R}^3} \Lambda^{s_0-2}_x \left\{ \partial^k\big(\mathbf{T}^{l+1} (\partial^j {\rho}^{l+1}-\partial^j {\rho}^{l})\big) \partial_j   h^{l+1} \right\}
		\\
		& \qquad  \cdot \Lambda^{s_0-2}_x\left\{ \partial_k( \Delta h^{l+1}-\Delta h^{l}- \partial^i h^{l+1} \partial_i \rho^{l+1}- \partial^i h^{l} \partial_i \rho^{l}) \right\}  dxd\tau \big|,
		\\
		V^{A2}_2  =& C\big| \int^t_0 \int_{\mathbb{R}^3} \Lambda^{s_0-2}_x \left\{ \partial^k\big((\bv^{l}-\bv^{l+1})\cdot \nabla \partial^j {\rho}^{l} \big) \partial_j   h^{l+1} \right\}
		\\
		& \qquad  \cdot \Lambda^{s_0-2}_x \left\{ \partial_k( \Delta h^{l+1}-\Delta h^{l}- \partial^i h^{l+1} \partial_i \rho^{l+1}- \partial^i h^{l} \partial_i \rho^{l}) \right\}  dxd\tau \big|,
		\\
		V^{A3}_2  =	&  C\big| \int^t_0 \int_{\mathbb{R}^3} \Lambda^{s_0-2}_x  \left\{ \partial^k  \big( \mathrm{curl}^j ( \mathrm{e}^{{\rho}^{l+1}}(\bw^{l+1}-\bw^l) ) -\epsilon^{jmk}\partial_m (\mathrm{e}^{{\rho}^{l+1}}-\mathrm{e}^{{\rho}^{l}})\bw^{l}_k \big) \right\}
		\\
		& \qquad  \cdot \Lambda^{s_0-2}_x \left\{ \partial_k( \Delta h^{l+1}-\Delta h^{l}- \partial^i h^{l+1} \partial_i \rho^{l+1}- \partial^i h^{l} \partial_i \rho^{l}) \right\}  dxd\tau \big|,
		\\
			V^{A4}_2  =& C\big| \int^t_0 \int_{\mathbb{R}^3} \Lambda^{s_0-2}_x \left\{ \partial^k (  \partial^j v^{k(l+1)} \partial_k {\rho}^{l+1}-\partial^j v^{kl} \partial_k {\rho}^l ) \right\}
		\\
		& \qquad  \cdot \Lambda^{s_0-2}_x\left\{ \partial_k( \Delta h^{l+1}-\Delta h^{l}- \partial^i h^{l+1} \partial_i \rho^{l+1}- \partial^i h^{l} \partial_i \rho^{l}) \right\}  dxd\tau \big|.
	\end{split}
\end{equation*}
By Lemma \ref{lpe} and H\"older's inequality, we can estimate $V^{A1}_2, V^{A2}_2, V^{A3}_2$ and $V^{A4}_2$ as follows,
\begin{equation}\label{V2A2}
	\begin{split}
		V^{A2}_2  \leq & C \int^t_0  \| \partial( \bv^{l+1}-\bv^{l})\|_{\dot{B}^{s_0-2}_{\infty,2}} \| {\rho}^{l} \|_{H^{s_0}_x} \|  h^{l+1} \|_{H^{s_0+1}_x} B(\tau) d\tau,
\\
		V^{A3}_2  \leq	&  C \int^t_0 \|  \bw^{l+1}-\bw^{l}\|_{H^{s_0}_x}  B(\tau)d\tau
		  + C \int^t_0 \| \partial( \rho^{l+1}-\rho^{l})\|_{\dot{B}^{s_0-2}_{\infty,2}} \| \bw \|_{H^{s_0}_x}  B(\tau)d\tau,
\\
		V^{A4}_2  \leq &  C \int^t_0  \{ \| \partial \rho^{l+1}\|_{\dot{B}^{s_0-2}_{\infty,2}} \| \bv^{l+1}-\bv^{l} \|_{H^{s_0}_x}  +   \| \partial \bv^{l}\|_{\dot{B}^{s_0-2}_{\infty,2}} \| \rho^{l+1}-\rho^{l} \|_{H^{s_0}_x} \}  B(\tau)d\tau.
	\end{split}
\end{equation}
For $V^{A1}_2$, by commutator rule, we write it as
\begin{equation*}
	\begin{split}
		V^{A1}_2=& C\big| \int^t_0 \int_{\mathbb{R}^3} \Lambda^{s_0-2}_x \left\{\mathbf{T}^{l+1} \big( \partial^k (\partial^j {\rho}^{l+1}-\partial^j {\rho}^{l})\big) \partial_j   h^{l+1} \right\}
		\\
		& \qquad  \cdot \Lambda^{s_0-2}_x\left\{ \partial_k( \Delta h^{l+1}-\Delta h^{l}- \partial^i h^{l+1} \partial_i \rho^{l+1}- \partial^i h^{l} \partial_i \rho^{l}) \right\}  dxd\tau \big|
		\\
		& +C\big| \int^t_0 \int_{\mathbb{R}^3} \Lambda^{s_0-2}_x \left\{ [\partial^k, \mathbf{T}^{l+1}]  (\partial^j {\rho}^{l+1}-\partial^j {\rho}^{l}) \partial_j   h^{l+1} \right\}
		\\
		& \qquad  \cdot \Lambda^{s_0-2}_x\left\{ \partial_k( \Delta h^{l+1}-\Delta h^{l}- \partial^i h^{l+1} \partial_i \rho^{l+1}- \partial^i h^{l} \partial_i \rho^{l}) \right\}  dxd\tau \big|.
	\end{split}
\end{equation*}
By a further calculation, we have
\begin{equation}\label{VA12}
	\begin{split}
		V^{A1}_2
		\leq & V^{A11}_2+V^{A12}_2+V^{A13}_2+V^{A14}_2,
	\end{split}
\end{equation}
where we set
\begin{equation*}
	\begin{split}
		V^{A11}_2
		= & C\big| \int^t_0 \int_{\mathbb{R}^3} \mathbf{T}^{l+1} \left\{ \Lambda^{s_0-2}_x  \big( \partial^k (\partial^j {\rho}^{l+1}-\partial^j {\rho}^{l})\big) \partial_j   h^{l+1} \right\}
		\\
	& \qquad  \cdot \Lambda^{s_0-2}_x \left\{ \partial_k( \Delta h^{l+1}-\Delta h^{l}- \partial^i h^{l+1} \partial_i \rho^{l+1}- \partial^i h^{l} \partial_i \rho^{l}) \right\}  dxd\tau \big|,
	\\
		V^{A12}_2
	= &  C\big| \int^t_0 \int_{\mathbb{R}^3}   \Lambda^{s_0-2}_x  \left\{ \big( \partial^k (\partial^j {\rho}^{l+1}-\partial^j {\rho}^{l})\big) \mathbf{T}^{l+1}(\partial_j   h^{l+1}) \right\}
	\\
	& \qquad  \cdot \Lambda^{s_0-2}_x \left\{ \partial_k( \Delta h^{l+1}-\Delta h^{l}- \partial^i h^{l+1} \partial_i \rho^{l+1}- \partial^i h^{l} \partial_i \rho^{l}) \right\}  dxd\tau \big|,
		\\
			V^{A13}_2
		= & C\big| \int^t_0 \int_{\mathbb{R}^3} [\Lambda^{s_0-2}_x, \mathbf{T}^{l+1}] \left\{   \big( \partial^k (\partial^j {\rho}^{l+1}-\partial^j {\rho}^{l})\big) \partial_j   h^{l+1} \right\}
		\\
		& \qquad  \cdot \Lambda^{s_0-2}_x\left\{ \partial_k( \Delta h^{l+1}-\Delta h^{l}- \partial^i h^{l+1} \partial_i \rho^{l+1}- \partial^i h^{l} \partial_i \rho^{l}) \right\}  dxd\tau \big|,
		\\
		V^{A14}_2
	=	& C \big| \int^t_0 \int_{\mathbb{R}^3} \Lambda^{s_0-2}_x \left\{ [\partial^k, \mathbf{T}^{l+1}]  (\partial^j {\rho}^{l+1}-\partial^j {\rho}^{l}) \partial_j   h^{l+1} \right\}
		\\
		& \qquad  \cdot \Lambda^{s_0-2}_x\left\{ \partial_k( \Delta h^{l+1}-\Delta h^{l}- \partial^i h^{l+1} \partial_i \rho^{l+1}- \partial^i h^{l} \partial_i \rho^{l}) \right\}  dxd\tau \big|.
	\end{split}
\end{equation*}
Note $\mathbf{T}^{l+1}(\partial_j   h^{l+1})=\partial \bv^{l+1} \partial  h^{l+1}$. By Lemma \ref{lpe}, Lemma \ref{ce} and H\"older's inequality, we derive that
\begin{equation}\label{VA122}
	\begin{split}
		V^{A12}_2+V^{A13}_2+V^{A14}_2
		\leq &  C \int^t_0 \| {\rho}^{l+1}- {\rho}^{l} \|_{H_x^{s_0}} \| \partial \bv\|_{\dot{B}_{\infty,2}^{s_0-2} } \| h^{l+1} \|_{H_x^{s_0} } B(\tau)d\tau.
	\end{split}
\end{equation}
For $V^{A121}_2$, we calculate it directly
\begin{equation*}
	\begin{split}
			V^{A11}_2
		= & C\big| \int^t_0 \int_{\mathbb{R}^3} \mathbf{T}^{l+1} \big\{ \Lambda^{s_0-2}_x  \big( \partial^k (\partial^j {\rho}^{l+1}-\partial^j {\rho}^{l})\big) \partial_j   h^{l+1}
		\\
		& \qquad \cdot  \Lambda^{s_0-2}_x \big( \partial_k( \Delta h^{l+1}-\Delta h^{l}- \partial^i h^{l+1} \partial_i \rho^{l+1}- \partial^i h^{l} \partial_i \rho^{l}) \big) \big\}  dxd\tau \big|
		\\
		& + C\big| \int^t_0 \int_{\mathbb{R}^3}  \left\{ \Lambda^{s_0-2}_x  \big( \partial^k (\partial^j {\rho}^{l+1}-\partial^j {\rho}^{l})\big) \partial_j   h^{l+1} \right\}
		\\
		& \qquad \cdot \mathbf{T}^{l+1} \left\{ \Lambda^{s_0-2}_x  \partial_k( \Delta h^{l+1}-\Delta h^{l}- \partial^i h^{l+1} \partial_i \rho^{l+1}- \partial^i h^{l} \partial_i \rho^{l}) \right\}  dxd\tau \big|
		\\
		= & C\big| \int^t_0 \int_{\mathbb{R}^3} \mathbf{T}^{l+1} \big\{ \Lambda^{s_0-2}_x  \big( \partial^k (\partial^j {\rho}^{l+1}-\partial^j {\rho}^{l})\big) \partial_j   h^{l+1}
		\\
		& \qquad \cdot  \Lambda^{s_0-2}_x \big( \partial_k( \Delta h^{l+1}-\Delta h^{l}- \partial^i h^{l+1} \partial_i \rho^{l+1}- \partial^i h^{l} \partial_i \rho^{l}) \big) \big\}  dxd\tau \big|
		\\
		& + C\big| \int^t_0 \int_{\mathbb{R}^3}  \left\{ \Lambda^{s_0-2}_x  \big( \partial^k (\partial^j {\rho}^{l+1}-\partial^j {\rho}^{l})\big) \partial_j   h^{l+1} \right\}
		\\
		& \qquad \cdot   \Lambda^{s_0-2}_x \mathbf{T}^{l+1} \left\{ \partial_k( \Delta h^{l+1}-\Delta h^{l}- \partial^i h^{l+1} \partial_i \rho^{l+1}- \partial^i h^{l} \partial_i \rho^{l}) \right\}  dxd\tau \big|
		\\
		& + C\big| \int^t_0 \int_{\mathbb{R}^3}  \left\{ \Lambda^{s_0-2}_x  \big( \partial^k (\partial^j {\rho}^{l+1}-\partial^j {\rho}^{l})\big) \partial_j   h^{l+1} \right\}
		\\
		& \qquad \cdot   [\Lambda^{s_0-2}_x, \mathbf{T}^{l+1} ] \left\{ \partial_k( \Delta h^{l+1}-\Delta h^{l}- \partial^i h^{l+1} \partial_i \rho^{l+1}- \partial^i h^{l} \partial_i \rho^{l}) \right\}  dxd\tau \big|
	\end{split}
\end{equation*}
Using \eqref{fff5}, Lemma \ref{lpe}, Lemma \ref{ce}, and H\"older's inequality, we obtain
\begin{equation}\label{VA112}
	\begin{split}
		V^{A11}_2 \leq & C \| {\rho}_0^{l+1}- {\rho}_0^{l} \|_{H_x^{s_0}}  \| h_0^{l+1} \|_{ L^\infty_tH_x^{s_0+1} } B(0)
		\\
		& + C \int^t_0\| {\rho}^{l+1}- {\rho}^{l} \|_{H_x^{s_0}} \| \partial \bv\|_{\dot{B}_{\infty,2}^{s_0-2} } \| h^{l+1} \|_{H_x^{s_0} } B(\tau)d\tau.
	\end{split}
\end{equation}
By using \eqref{lyr1} and \eqref{V2A}-\eqref{VA112}, for $t\in [0,\frac{1}{100(1+C)^2}T]$, we have
\begin{equation}\label{V2t3}
  V_2 \leq  \frac32(Cc_l)^2.
\end{equation}
By using \eqref{lyr1} again and \eqref{fff6}, \eqref{V2t0},\eqref{V2t1},\eqref{V2t2} and \eqref{V2t3}, we can see that
\begin{equation*}
 B(t) \leq 3Cc_l, \qquad t\in [0,\frac{1}{100(1+C)^2}T].
\end{equation*}
Hence, we can get
\begin{equation}\label{B2t}
 \| h^{l+1}-h^{l}\|_{\dot{H}_x^{s_0+1}} \leq  B(t)+ \| \partial ( h^{l+1}-h^l) \partial \rho^{l+1}\|_{\dot{H}_x^{s_0-1}}+ \| \partial h^{l} \partial ( \rho^{l+1}- \rho^{l}) \|_{\dot{H}_x^{s_0-1}}
 \leq  6Cc_l.
\end{equation}
\textit{Step 4: Strichartz estimates of the difference}. Let us now use the wave equation \eqref{fc} for $(\bv^{l}, \rho^{l}, h^{l},\bw^{l})$. That is,
			\begin{equation*}
				\begin{cases}
					& \square_{g^{l}} v_{+}^{il}= \mathbf{T}^l\mathbf{T}^l v^{il}_{-}+Q^{il},
					\\
					& \square_{g^l} (\rho^{l}+\frac{1}{\gamma} h^l)={D}^{l}+\frac{1}{\gamma}{E}^{l},
				\end{cases}
			\end{equation*}
			Above, $g^l, \bv^l_{+}, Q^{il}, E^{l}$ and ${D}^l$ have the same formulation with $g, Q^{i}, E$ and ${D}$ in \eqref{Met},\eqref{etad},\eqref{DDi} if we replace $(\bv,\rho,h,\bw)$ to $(\bv^l,\rho^l,h^l,\bw^l)$. Thus, the quantities $\bv^{l+1}-\bv^{l}$ and  $\rho^{l+1}-\rho^{l}$ satisfies
			\begin{equation}\label{Ss4}
				\begin{cases}
					& \square_{g^{l+1}} (\bv_{+}^{l+1}-\bv_{+}^{l})= \mathbf{T}^{l+1}(\mathbf{T}^{l+1}\bv^{l+1}_{-}-\mathbf{T}^{l}\bv^{l}_{-})
					-( g^{l+1}_{\alpha i}-g^l_{\alpha i}) \partial^{i \alpha} \bv_{+}^{l}+\bF^l,
					\\
					& \square_{g^{l+1}} (\rho^{l+1}+\frac{1}{\gamma} h^{l+1}-\rho^{l}-\frac{1}{\gamma} h^l)={G}^{l+1}-G^l-( g^{l+1}_{\alpha i}-g^l_{\alpha i}) \partial^{i \alpha} (\rho^{l}+\frac{1}{\gamma} h^l),
				\end{cases}
			\end{equation}
where $\bF^l=\bQ^{l+1}-\bQ^{l}+( \mathbf{T}^{l+1}-\mathbf{T}^{l}) \mathbf{T}^{l} \bv_{-}^l$ and $G^l={D}^{l}+\frac{1}{\gamma}{E}^{l}$. Let us now estimate the right hand side of \eqref{Ss4} one by one. Calculate $\mathbf{T}^{l+1}-\mathbf{T}^{l}=( \bv^{l+1}- \bv^{l}) \cdot \nabla$. By product estimates and \eqref{fc0}, we have
			\begin{equation}\label{Fh0}
				\begin{split}
					\|  \bF^{l} \|_{H^{s-1}_x} \leq & \|\bQ^{l+1}-\bQ^{l} \|_{H^{s-1}_x}+ \|( \bv^{l+1}- \bv^{l}) \cdot \nabla (\mathbf{T}^{l} \bv_{-}^l) \|_{H^{s-1}_x}
					\\
					\lesssim & \|(d\rho^l, d\bv^l,dh^l)\|_{L^\infty_x}
					\| (d\bv^{l+1}-d\bv^{l},d\rho^{l+1}-d\rho^{l},dh^{l+1}-dh^{l}) \|_{H^{s-1}_x}
					\\
					& +\|(d\rho^{l+1},d\bv^{l+1},dh^{l+1})\|_{L^\infty_x}
					\| (d\bv^{l+1}-d\bv^{l},d\rho^{l+1}-d\rho^{l},dh^{l+1}-dh^{l}) \|_{H^{s-1}_x}
					\\
					& + 2^l \| \bv^{l+1}- \bv^{l} \|_{L^\infty_x} \| \mathbf{T}^{l} \bv_{-}^l \|_{H^{s-1}_x}.
					\\
					\leq &  \|(\partial \rho^l, \partial \bv^l,\partial h^l)\|_{L^\infty_x}
					\| \partial(\bv^{l+1}-\bv^{l},\rho^{l+1}-\rho^{l},h^{l+1}-h^{l}) \|_{H^{s-1}_x}
					\\
					& +\|\partial(\rho^{l+1},\bv^{l+1},h^{l+1})\|_{L^\infty_x}
					\| \partial(\bv^{l+1}-\bv^{l},\rho^{l+1}-\rho^{l},h^{l+1}-h^{l}) \|_{H^{s-1}_x}
					\\
					&+ 2^l \| \bv^{l+1}- \bv^{l} \|_{L^\infty_x}  \| \bw^{l+1}-\bw^{l} \|_{H^{s-1}_x}  .
				\end{split}
			\end{equation}
		Using $\Delta \bv^l_{-}= \mathrm{e}^{\rho^l}\mathrm{curl}\bw^{l}$, and recalling the proof of \eqref{eta}, we find out
	\begin{equation}\label{Fha}
		\begin{split}
			& \| \mathbf{T}^{l+1}\bv^{l+1}_{-}-\mathbf{T}^{l}\bv^{l}_{-}\|_{H^{s}_x}
\\
\lesssim & \| \partial \bv^{l+1}  \partial \bw^{l+1}-\partial \bv^{l}  \partial \bw^{l} \|_{H_x^{s-2}}
			 + \| \partial^2 \bv^{l+1}  \partial \bv^{l+1}_{-}- \partial^2 \bv^{l}  \partial \bv^{l}_{-} \|_{H_x^{s-2}}
			\\
			& + \| \partial^2 \bv^{l+1}_{-}  \partial \bv^{l+1}- \partial^2 \bv^{l}_{-}  \partial \bv^{l} \|_{H_x^{s-2}}
			 + \| \partial^2 \rho^{l+1}  \partial h^{l+1} - \partial^2 \rho^{l}  \partial h^{l}\|_{H_x^{s-2}}
			\\
			& + \| \partial h^{l+1}  \partial \rho^{l+1}  \partial h^{l+1}-\partial h^{l}  \partial \rho^{l}  \partial h^{l} \|_{H_x^{s-2}}
			+\| \partial \rho^{l+1}  \partial \rho^{l+1}  \partial h^{l+1}-\partial \rho^{l}  \partial \rho^{l}  \partial h^{l} \|_{H_x^{s-2}}
			\\
			& +  \| \partial \rho^{l+1}   \partial^2 h^{l+1}- \partial \rho^{l}  \partial^2 h^{l} \|_{H_x^{s-2}}.
		\end{split}
	\end{equation}
By product estimates, \eqref{Fha} yields
	\begin{equation}\label{Fhb}
	\begin{split}
		& \| \mathbf{T}^{l+1}\bv^{l+1}_{-}-\mathbf{T}^{l}\bv^{l}_{-}\|_{H^{s}_x}
		\\
		\lesssim
		 & \| \partial (\bv^{l+1} - \bv^{l}) \|_{H_x^{s-1}} \| \partial \bw^{l} \|_{H_x^{\frac{1}{2}+}} + \| \partial \bv^{l} \|_{H_x^{s-1}} \| \partial ( \bw^{l+1}-\bw^{l} ) \|_{H_x^{\frac{1}{2}+}}
		 \\
		 & + \| \partial^2 ( \bv^{l+1} - \bv^{l}) \|_{H_x^{s-2}} \| \partial \bv^{l}_{-} \|_{H_x^{\frac{3}{2}+}}
		  + \| \partial^2 \bv^{l} \|_{H_x^{s-2}} \| \partial (\bv^{l+1}_{-}- \bv^{l}_{-} )\|_{H_x^{\frac{3}{2}+}}
		 \\
		 & + \| \partial ( \bv^{l+1} - \bv^{l})\|_{H_x^{s-1}} \| \partial^2 \bv^{l}_{-} \|_{H_x^{\frac{1}{2}+}}+\| \partial \bv^{l} \|_{H_x^{s-1}} \| \partial^2 ( \bv^{l+1}_{-} - \bv^{l}_{-} )\|_{H_x^{\frac{1}{2}+}}
		\\
		& + (\| \partial ( h^{l+1} - h^{l} ) \|_{H_x^{s-1}}+\| \partial ( \rho^{l+1}- \partial \rho^{l}) \|_{H_x^{s-1}}) \| \partial \rho^{l} \|_{H_x^{s-1}} \| \partial h^{l} \|_{H_x^{s-1}}
		\\
		& + (\| \partial h^{l} \|_{H_x^{s-1}}+\| \partial \rho^{l} \|_{H_x^{s-1}}) \| \partial ( \rho^{l+1}- \rho^{l} ) \|_{H_x^{s-1}} \| \partial h^{l} \|_{H_x^{s-1}}
			\\
		& + (\| \partial h^{l} \|_{H_x^{s-1}}+\| \partial \rho^{l} \|_{H_x^{s-1}}) \| \partial \rho^{l} \|_{H_x^{s-1}} \| \partial ( h^{l+1} - h^{l}) \|_{H_x^{s-1}}
		\\
		&+  \| \partial ( \rho^{l+1} - \rho^{l})  \|_{H_x^{s-1}} \| \partial^2 h^{l} \|_{H_x^{\frac{1}{2}+}}
		+  \| \partial \rho^{l}  \|_{H_x^{s-1}} \| \partial^2 ( h^{l+1}- h^{l}) \|_{H_x^{\frac{1}{2}+}}
		\\
		& + \| \partial^2 ( \rho^{l+1} - \rho^{l} )\|_{H_x^{s-2}} \| \partial h^{l} \|_{H_x^{\frac{3}{2}+}}
		+ \| \partial^2 \rho^{l}  \|_{H_x^{s-2}} \| \partial ( h^{l+1} - h^{l} )\|_{H_x^{\frac{3}{2}+}}.
	\end{split}
\end{equation}
By product estimates, Bernstein's inequalities and \eqref{fc0}, we obtain
			\begin{equation}\label{Fh2}
				\begin{split}
					\| ( g^{l+1}_{\alpha i}-g^l_{\alpha i}) \partial^{i \alpha} \bv^{l} \|_{H^{s-1}_x}
					\lesssim &   \|\partial d\bv^l \|_{L^\infty_x}  \| (\bv^{l+1}-\bv^{l},\rho^{l+1}-\rho^{l},h^{l+1}-h^{l}) \|_{H^{s-1}_x}
					\\
					& +   \| (\bv^{l+1}-\bv^{l}, \rho^{l+1}-\rho^{l}, h^{l+1}-h^{l}) \|_{L^\infty_x} \| \partial d\bv^{l} \|_{H^{s-1}_x}
					\\
					\leq &   \| \partial \bv^l \|_{L^\infty_x}\cdot 2^l  \| (\bv^{l+1}-\bv^{l},\rho^{l+1}-\rho^{l},h^{l+1}-h^{l}) \|_{H^{s-1}_x}
					\\
					& +   2^l\| (\bv^{l+1}-\bv^{l}, \rho^{l+1}-\rho^{l},h^{l+1}-h^{l}) \|_{L^\infty_x} \| \partial \bv^l \|_{H^{s-1}_x}.
				\end{split}
			\end{equation}
Similarly, we also deduce
			\begin{equation}\label{Fh3}
				\begin{split}
				& \| ( g^{l+1}_{\alpha i}-g^l_{\alpha i}) \partial^{i \alpha} (\rho^{l}+\frac{1}{\gamma} h^l) \|_{H^{s-1}_x}
				\\
					\lesssim &   \|\partial (d\rho^l,dh^l) \|_{L^\infty_x}  \| (\bv^{l+1}-\bv^{l},\rho^{l+1}-\rho^{l},h^{l+1}-h^{l}) \|_{H^{s-1}_x}
					\\
					& +   \| (\bv^{l+1}-\bv^{l}, \rho^{l+1}-\rho^{l},h^{l+1}-h^{l}) \|_{L^\infty_x} \| \partial ( d\rho^{l},dh^l )\|_{H^{s-1}_x}
					\\
					\leq &  \| \partial \rho^l,\partial h^l \|_{L^\infty_x}  \cdot 2^l\| (\bv^{l+1}-\bv^{l},\rho^{l+1}-\rho^{l},h^{l+1}-h^{l}) \|_{H^{s-1}_x}
					\\
					& +    2^l\| (\bv^{l+1}-\bv^{l}, \rho^{l+1}-\rho^{l},h^{l+1}-h^{l}) \|_{L^\infty_x} \|\partial (\rho^l, h^l) \|_{H^{s-1}_x},
				\end{split}
			\end{equation}
and
			\begin{equation}\label{Fh4}
				\begin{split}
				\| G^{l+1}- G^{l} \|_{H^{s-1}_x} \leq & C\|(\partial \rho^l,\partial \bv^l,\partial h^l)\|_{L^\infty_x}
				  \| \partial(\bv^{l+1}-\bv^{l},\rho^{l+1}-\rho^{l},h^{l+1}-h^{l}) \|_{H^{s-1}_x}.
			\end{split}
		\end{equation}
To the system \eqref{Ss4}, by using Proposition \ref{r5}, Lemma \ref{LD} and using \eqref{lyr1} and \eqref{Fh0}-\eqref{Fh4}, we can derive that
			\begin{equation}\label{ff2}
				\begin{split}
					& 2^l \| (\bv_{+}^{l+1}-\bv_{+}^{l}, \rho^{l+1}+\frac{1}{\gamma}h^{l+1}-\rho^{l}-\frac{1}{\gamma}h^{l}) \|_{L^2_{[0,\frac{1}{100(1+C)^2}T]} L^\infty_x}
					\\
					& \quad +2^l\|(\bv_{+}^{l+1}-\bv_{+}^{l}, \rho^{l+1}+\frac{1}{\gamma}h^{l+1}-\rho^{l}-\frac{1}{\gamma}h^{l}) \|_{L^2_{[0,\frac{1}{100(1+C)^2}T]} \dot{B}^{s_0-2}_{\infty,2}}
					\\
					\leq & Cc^2_l \leq \frac{C}{2}c_l,
				\end{split}
			\end{equation}
where we use the fact $c_l\rightarrow 0, l\rightarrow \infty$. Moreover, by Sobolev imbeddings, we also have
		\begin{equation}\label{B3t}
			\begin{split}
			& 2^l\| (\bv_{-}^{l+1}-\bv_{-}^{l},h^{l+1}-h^{l}) \|_{L^2_t L^{\infty}_x} + 2^l \| \partial (\bv_{-}^{l+1}-\bv_{-}^{l},h^{l+1}-h^{l}) \|_{L^2_t \dot{B}^{s_0-2}_{\infty,2}}
			\\
			\lesssim & \| \bw^{l+1}-\bw^{l}  \|_{L^2_t H^{s_0-\frac12}_x} + \|  h^{l+1}-h^{l}  \|_{L^2_t H^{s_0+\frac12}_x}
\\
\leq & C2^{-\frac{s}{2(s_0+1)}l}Cc_l
\leq  \frac14 Cc_l.
		\end{split}
	\end{equation}
Due to \eqref{ff2} and \eqref{B3t}, we
			\begin{equation}\label{B4t}
				\begin{split}
	& 2^l \| (\bv^{l+1}-\bv^{l}, \rho^{l+1}-\rho^{l}, h^{l+1}-h^{l}) \|_{L^2_{[0,\frac{1}{100(1+C)^2}T]}L^\infty_x}
\\
& \ + 2^l \| (\bv^{l+1}-\bv^{l}, \rho^{l+1}-\rho^{l}, h^{l+1}-h^{l}) \|_{L^2_{[0,\frac{1}{100(1+C)^2}T]} \dot{B}^{s_0-2}_{\infty,2}}
					\leq \frac{3C}{4}c_l.
				\end{split}
			\end{equation}
Therefore, by gathering \eqref{lyr3}, \eqref{lyr4}, \eqref{G2t6}, \eqref{B2t} and \eqref{B4t}, we have proved \eqref{dbs}-\eqref{dbs2}.
		\end{proof}
\begin{remark}
	The estimate \eqref{B4t} is essential for us to prove the continuous dependence. If we consider the classical solutions, then we could use Sobolev imbeddings to handle \eqref{Ss4}. However, for rough solutions, it is necessary to use equation \eqref{Ss4} and the Strichartz estimates \eqref{SE1} of the linear wave equation \eqref{linear} to establish \eqref{B4t}. We also mention that the continuous dependence of solutions of \eqref{CEE}-\eqref{id} has not been proved in \cite{WQEuler} and \cite{DLS}.
\end{remark}	
	
\section{Proof of Theorem \ref{dingli3}}\label{Sec8}
In this part, our goal is to prove Theorem \ref{dingli3}. We first reduce the proof to the case of smooth initial data with  bounded frequency,  that is, using Proposition \ref{DDL2} to prove Theorem \ref{dingli3}. Then we use Proposition \ref{p1} to obtain Proposition \ref{DDL3}. Finally, we use Proposition \ref{DDL3} and develop a ``semiclassical analysis" approach to prove Proposition \ref{DDL2}.

\subsection{Proof of Theorem \ref{dingli3}}\label{keypra}
To prove Theorem \ref{dingli3}, we will prove the well-posedness of solutions as a limit of a sequence. To find this sequence, let us introduce a sequence of initial data $(\bv_{0j},\rho_{0j})$ with frequency support such that
\begin{equation}\label{Dss0}
	\bv_{0j}=P_{\leq j} \bv_0, \quad \rho_{0j}=P_{\leq j} \rho_0,
\end{equation}
where $P_{\leq j}=\textstyle{\sum}_{k\leq j}P_k$, $P_j$ be a Littlewood-Paley operator with frequency $\{\xi\in\mathbb{R}^3: 2^{j-3}\leq |\xi|\leq 2^{j+3}, j\in\mathbb{Z}\}$, and $(\bv_0,\rho_0)$ is stated as \eqref{ids} in Theorem \ref{dingli3}. Following \eqref{pw11}, so we define
\begin{equation}\label{Qss0}
	\bw_{0j}=\mathrm{e}^{-\rho_{0j}}\mathrm{curl}\bv_{0j}.
\end{equation}
Using \eqref{ids}, so we have
\begin{equation}\label{Ess0}
	\begin{split}
		\|\bw_{0j}\|_{H^2}=& \| \mathrm{e}^{\rho_{0j}}\mathrm{curl}\bv_{0j} \|_{H^2}
		\\
		\leq  & \| \mathrm{e}^{\rho_{0j}}P_{\leq j}(\mathrm{e}^{\rho_0}\bw_{0}) \|_{H^2}
		\\
		\leq & C(\|\mathrm{e}^{\rho_0}\bw_{0} \|_{H^2}+ \|\rho_{0}\|_{H^2}\|\bw_{0}\|_{H^2})
		\\
		\leq & C(M_*+M^2_*).
	\end{split}
\end{equation}
Adding \eqref{Dss0} and \eqref{Ess0}, we can see
\begin{equation}\label{pu0}
	E(0):=\|\bv_{0j}\|_{H^s}+ \|\rho_{0j}\|_{H^s}+ \| \bw_{0j} \|_{H^2} \leq C(M_*+M^2_*).
\end{equation}
Using \eqref{HEw}, \eqref{Dss0} and \eqref{Qss0}, we get
\begin{equation}\label{pu00}
	| \bv_{0j}, \rho_{0j} | \leq C_0, \quad c_s|_{t=0}\geq c_0>0.
\end{equation}

Before we give a proof of Theorem \ref{dingli3}, let us now introduce Proposition \ref{DDL2}.
\begin{proposition}\label{DDL2}
Let $s$ and \eqref{HEw}-\eqref{chuzhi3} be stated in Theorem \ref{dingli3}. Let $(\bv_{0j}, \rho_{0j}, \bw_{0j})$ be stated in \eqref{Dss0} and \eqref{Qss0}. For each $j\geq 1$, consider Cauchy problem \eqref{fc0s} with the initial data $(\bv_{0j}, \rho_{0j}, \bw_{0j})$. Then for all $j \geq 1$, there exists two positive constants $T^{*}>0$ and $\bar{M}_{2}>0$ ($T^*$ and $\bar{M}_{2}$ only depends on ${M}_{*}$ and $s$) such that \eqref{fc0s} has a unique solution $(\bv_{j},\rho_{j})\in C([0,T^*];H_x^s)\cap C^1([0,T^*];H_x^{s-1})$, $\bw_{j}\in C([0,T^*];H_x^2) \cap C^1([0,T^*];H_x^1)$. To be precise,

\begin{enumerate}
	\item  the solution $\bv_j, \rho_j$ and $\bw_j$ satisfy the energy estimates
\begin{equation}\label{Duu0}
  \|\bv_j\|_{L^\infty_{[0,T^*]}H_x^{\sstar}}+\|\rho_j\|_{L^\infty_{[0,T^*]}H_x^{\sstar}}+ \|\bw_j\|_{L^\infty_{[0,T^*]}H_x^{2}} \leq \bar{M}_{2},
\end{equation}
and
\begin{equation}\label{Duu1}
  \|\partial_t \bv_j, \partial_t\rho_j\|_{L^\infty_{[0,T^*]}H_x^{\sstar-1}}+ \|\partial_t \bw_j\|_{L^\infty_{[0,T^*]}H_x^{1}} \leq \bar{M}_{2},
\end{equation}
and
\begin{equation}\label{Duu00}
	\|\bv_j, \rho_j\|_{L^\infty_{[0,T^*]\times \mathbb{R}^3}} \leq 2+C_0.
\end{equation}

\item the solution $\bv_j$ and $\rho_j$ satisfy the Strichartz estimate
\begin{equation}\label{Duu2}
  \|d\bv_j, d\rho_j\|_{L^2_{[0,T^*]}L_x^\infty} \leq \bar{M}_{2}.
\end{equation}

\item for $\frac{s}{2} \leq r \leq 3$, consider the following linear wave equation
\begin{equation}\label{Duu21}
	\begin{cases}
		\square_{{g}_j} f_j=0, \qquad [0,T^*]\times \mathbb{R}^3,
		\\
		(f_j,\partial_t f_j)|_{t=0}=(f_{0j},f_{1j}),
	\end{cases}
\end{equation}
where $(f_{0j},f_{1j})=(P_{\leq j}f_0,P_{\leq j}f_1)$ and $(f_0,f_1)\in H_x^r \times H^{r-1}_x$. Then there is a unique solution $f_j$ on $[0,T^*]\times \mathbb{R}^3$. Moreover, for $a\leq r-\frac{s}{2}$, we have
\begin{equation}\label{Duu22}
	\begin{split}
		&\|\left< \partial \right>^{a-1} d{f}_j\|_{L^2_{[0,T^*]} L^\infty_x}
		\leq  \bar{M}_4(\|{f}_0\|_{{H}_x^r}+ \|{f}_1 \|_{{H}_x^{r-1}}),
		\\
		&\|{f}_j\|_{L^\infty_{[0,T^*]} H^{r}_x}+ \|\partial_t {f}_j\|_{L^\infty_{[0,T^*]} H^{r-1}_x} \leq  \bar{M}_4 (\| {f}_0\|_{H_x^r}+ \| {f}_1\|_{H_x^{r-1}}),
	\end{split}
\end{equation}
where $\bar{M}_4$ is a universal constant depends on $C_0, c_0, M_*, s$.
\end{enumerate}
\end{proposition}
Based on Proposition \ref{DDL2}, we are now ready to prove Theorem \ref{dingli3}.
\medskip\begin{proof}[Proof of Theorem \ref{dingli3} by Proposition \ref{DDL2}] We will divide the proof into four steps, The following Step 1 and Step 2 is for Point 1 and Point 2 in Theorem \ref{dingli3}. Step 3 is for giving a proof to Point 3. At last, we prove Point 4 in Step 4.

\textit{Step 1: Obtaining a solution of Theorem \ref{dingli3} in a weaker spaces by taking a limit.}
If we set $\bU_j=(p(\rho_j), \bv_j)^{\mathrm{T}}$, by Lemma \ref{sh} and \ref{W0}, for $j \in \mathbb{N}^{+}$ we have
\begin{equation*}
	\begin{split}
		A^0 (\bU_j) \partial_t \bU_j+ A^i (\bU_j) \partial_i \bU_j=0,
		\\
		\partial_t \bw_j + (\bv_j \cdot \nabla)\bw_j=(\bw_j \cdot \nabla)\bv_j.
	\end{split}
\end{equation*}
Then for any $j, l \in \mathbb{N}^{+}$, we have
\begin{equation}\label{ur}
	\begin{split}
		A^0 (\bU_j) \partial_t (\bU_j-\bU_l)+ A^i (\bU_j) \partial_i (\bU_j-\bU_l)=&-\{ A^\alpha (\bU_j)-A^\alpha (\bU_l)\} \partial_\alpha \bU_l,
		\\
		\partial_t (\bw_j-\bw_l) + (\bv_j \cdot \nabla)(\bw_j-\bw_l)=&-\{(\bv_j-\bv_l) \cdot \nabla\}\bw_l
		\\
		& \ + \{(\bw_j-\bw_l) \cdot \nabla)\}\bv_j
		\\
		& \ + (\bw_l \cdot \nabla) (\bv_j-\bv_l).
	\end{split}
\end{equation}
Using \eqref{Duu0}-\eqref{Duu2}, we can show that
\begin{equation*}
	\| \bU_{j}-\bU_{l}\|_{L^\infty_{[0, T^*]}H^1_x}+\| \bw_{j}-\bw_{l}\|_{L^\infty_{[0, T^*]}H^1_x} \leq C_{\bar{M}_2} (\| \bv_{0j}-\bv_{0l}\|_{H^{\sstar}}+ \|\rho_{0j}-\rho_{0l}\|_{H^{\sstar}}+ \| \bw_{0j}-\bw_{0l}  \|_{H^{2}}).
\end{equation*}
Here $C_{\bar{M}_2}$ is a constant depending on $\bar{M}_2$. By Lemma \ref{jh0}, so we get
\begin{equation*}
	\| \bv_{j}-\bv_{l},\rho_{j}-\rho_{l}\|_{L^\infty_{[0, T^*]}H^1_x}+\| \bw_{j}-\bw_{l}\|_{L^\infty_{[0, T^*]}H^1_x} \leq C_{\bar{M}_2} (\| \bv_{0j}-\bv_{0l}\|_{H^{\sstar}}+ \|\rho_{0j}-\rho_{0l}\|_{H^{\sstar}}+ \| \bw_{0j}-\bw_{0l}  \|_{H^{2}}).
\end{equation*}
Using \eqref{fc0s} and \eqref{W0} again, we also have
\begin{equation*}
	\|\partial_t ( \bU_{j}-\bU_{l} )\|_{L^\infty_{[0, T^*]}L^2_x}+\|\partial_t (  \bw_{j}-\bw_{l} )\|_{L^\infty_{[0, T^*]}L^2_x} \leq C_{C_*} (\| \bv_{0j}-\bv_{0l}\|_{H^{\sstar}}+ \|\rho_{0j}-\rho_{0l}\|_{H^{\sstar}}+ \| \bw_{0j}-\bw_{0l}  \|_{H^{2}}).
\end{equation*}
So the sequence $\{(\bv_{j}, \rho_{j}, \bw_{j})\}_{j\in \mathbb{N}^+}$ is a Cauchy sequence in $H^1_x$, and $\{(\partial_t \bv_{j}, \partial_t\rho_{j}, \partial_t\bw_{j})\}_{j\in \mathbb{N}^+}$ is a Cauchy sequence in $L^2_x$. So there is a limit $( \bv, \rho, \bw)$ such that
\begin{equation}\label{rwe}
	\begin{split}
		(\bv_{j}, \rho_{j}, \bw_{j})\rightarrow & ( \bv, \rho, \bw) \quad \quad \ \ \quad \text{in} \ \ H_x^{1} \times H_x^{1} \times H_x^{1},
		\\
		(\partial_t \bv_{j}, \partial_t\rho_{j}, \partial_t\bw_{j})\rightarrow & ( \partial_t\bv, \partial_t\rho, \partial_t\bw) \quad \text{in} \ \ L_x^{2} \times L_x^{2} \times L_x^{2}.
	\end{split}
\end{equation}
By using \eqref{Duu0} and \eqref{Duu1}, then a subsequence of $\{(\bv_{j}, \rho_{j}, \bw_{j})\}_{j\in \mathbb{N}^+}$ and $\{(\partial_t \bv_{j}, \partial_t\rho_{j}, \partial_t\bw_{j})\}_{j\in \mathbb{N}^+}$ are weakly convergent. Therefore, if $m\rightarrow \infty$, then
\begin{equation}\label{rwe0}
	\begin{split}
		(\bv_{j_m}, \rho_{j_m}, \bw_{j_m})\rightharpoonup & ( \bv, \rho, \bw) \quad \qquad \ \ \text{in} \ \ H_x^{\sstar} \times H_x^{\sstar} \times H_x^{2},
		\\
		(\partial_t \bv_{j_m}, \partial_t\rho_{j_m}, \partial_t\bw_{j_m})\rightharpoonup & ( \partial_t\bv, \partial_t\rho, \partial_t\bw) \quad \text{in} \ \ H_x^{\sstar-1} \times H_x^{\sstar-1} \times H_x^{1}.
	\end{split}
\end{equation}
By \eqref{rwe}, \eqref{rwe0} and interpolation formula, we can obtain
\begin{equation}\label{rwe1}
	(\bv_{j_m}, \rho_{j_m})\rightarrow( \bv, \rho) \quad \text{in} \ \ H_x^{a}(0 \leq a <\sstar), \quad \bw_{j_m} \rightarrow \bw \quad \text{in} \ \ H_x^{b} (0\leq b<2).
\end{equation}
By \eqref{rwe1} and \eqref{fc0s}, \eqref{fc1s} we also have
\begin{equation}\label{rwe2}
	\begin{split}
		& (\partial_t\bv_{j_m}, \partial_t\rho_{j_m})\rightarrow(\partial_t \bv, \partial_t \rho) \ \qquad \text{in} \ \ H_x^{\alpha}(0 \leq \alpha <\sstar-1),
		\\
		& \partial_t\bw_{j_m} \rightarrow \partial_t\bw \qquad \qquad \qquad \qquad \text{in} \ \ H_x^{\beta} (0 \leq \beta<1).
	\end{split}
\end{equation}
The weak convergence \eqref{rwe0} and the strong convergence \eqref{rwe1}-\eqref{rwe2} makes us easy to check that $(\bv, \rho, \bw)$ is a strong solution of \eqref{fc1s} with the initial data $(\bv, \rho,\bw)|_{t=0}=(\bv_0,\rho_0,\bw_0)$,
and for $1\leq a<\sstar, 1 \leq b<2$,
\begin{equation}\label{rwe4}
	\begin{split}
		& (\bv,\rho)\in C([0,T^*];H^a),
		\qquad \qquad \quad \bw \in C([0,T^*];H^b),
		\\
		&(\partial_t \bv,\partial_t \rho)\in C([0,T^*];H_x^{a-1}),
		\qquad \partial_t\bw \in C([0,T^*];H_x^{b-1}),
		\\
		& (\bv,\rho)\in L^\infty([0,T^*];H_x^{\sstar}), \qquad \qquad \bw \in L^\infty([0,T^*];H_x^2),
		\\
		& (\partial_t\bv,\partial_t\rho)\in L^\infty([0,T^*];H_x^{\sstar-1}), \ \quad \partial_t\bw \in L^\infty([0,T^*];H_x^1).
	\end{split}
\end{equation}
Furthermore, by using \eqref{Duu0}, \eqref{Duu1}, we obtain
\begin{equation*}%\label{Dii4}
	\begin{split}
		& \mathcal{E}(t)= \|\bv\|_{L^\infty_{[0,T^*]}H_x^{\sstar}}+ \|\rho\|_{L^\infty_{[0,T_*]}H_x^{\sstar}}+ \|\bw\|_{L^\infty_{[0,T^*]}H_x^{2}} \leq \bar{M}_2,
	\end{split}
\end{equation*}
and
\begin{equation*}%\label{ii5}
	\begin{split}
		& \|d \bv\|_{L^2_{[0,T^*]}L_x^\infty}+\| d \rho\|_{L^2_{[0,T^*]}L_x^\infty}
		\leq \bar{M}_2,
		\\
		& \|\partial_t \bv\|_{L^\infty_{[0,T^*]}H_x^{\sstar-1}}+ \|\partial_t \rho\|_{L^\infty_{[0,T^*]}H_x^{\sstar-1}}+ \|\partial_t \bw\|_{L^\infty_{[0,T^*]}H_x^{1}} \leq \bar{M}_2.
	\end{split}
\end{equation*}
It also remains for us that the convergence \eqref{rwe1} and \eqref{rwe2} should also hold in the highest derivatives. To prove it, let us use a similar idea in discussing the continuous dependence in Section \ref{Sub}.

\textit{Step 2: Strong convergence of the limit in the highest derivatives.} Denote
\begin{equation}\label{Dss23}
	\begin{split}
		e^{(a)}_{j} = & \| v^a_{0(j+1)} - v^a_{0j} \|_{L^2_x}, \quad a=1,2,3,
		\\
		d^{(a)}_{j} = & \| w^a_{0(j+1)} - w^a_{0j} \|_{L^2_x},\quad a=1,2,3,
		\\
		e^{(4)}_{j} = & \| \rho_{0(j+1)} - \rho_{0j} \|_{L^2_x}.
	\end{split}
\end{equation}
Set
\begin{equation}\label{Dss24}
	\begin{split}
		e_{j} = \textstyle{\sum}_{a=1}^4 e^{(a)}_{j}+\textstyle{\sum}_{a=1}^3 d^{(a)}_{j}.
	\end{split}
\end{equation}
Therefore, by using \eqref{Dss23} and \eqref{Dss24}, we can deduce that
\begin{equation}\label{Dss25}
	\begin{split}
		\lim_{j \rightarrow \infty} e_{j} = 0.
	\end{split}
\end{equation}
Note \eqref{Dss0} and \eqref{Qss0}. So we have these properties:

(i)  uniform bounds
\begin{equation}\label{mp0}
	\begin{split}
		& \| \bv_{0j},\rho_{0j} \|_{H^{\sstar}} \lesssim M_*+M^2_*, \quad \| \bw_{0j} \|_{H^{2}} \lesssim M_*+M^2_*,
	\end{split}
\end{equation}

(ii)  high frequency bounds
\begin{equation}\label{mp1}
	\begin{split}
		& \|   \bv_{0j},\rho_{0j} \|_{H^{\sstar+1}}+\|  \bw_{0j} \|_{H^{3}} \lesssim 2^{j}(M_*+M^2_*),
	\end{split}	
\end{equation}

(iii)  difference bounds
\begin{equation}\label{mp2}
	\begin{split}
		&\|  \bv_{0(j+1)}- \bv_{0j},\rho_{0(j+1)}- \rho_{0j}\|_{L_x^{2}} \lesssim 2^{-\sstar j}e_j,
		\\
		&\|  \bw_{0(j+1)}- \bw_{0j}\|_{L_x^{2}} \lesssim 2^{-2j}e_j,
	\end{split}
\end{equation}

(iv)  limit
\begin{equation}\label{mp3}
	\lim_{l \rightarrow \infty}( \bv_{0j}, \rho_{0j}, \bw_{0j})=(\bv_0, \rho_0, \bw_0) \quad  \mathrm{in} \quad H^{\sstar} \times H^{\sstar}\times H^{2}.
\end{equation}
For solutions $(\bv_j,\rho_j, \bw_j)$, we claim that

$\bullet$  uniform bounds
\begin{equation}\label{ebs0}
	\| \bv_{j} \|_{L^\infty_{[0,T^*]}H_x^{\sstar}}+\|\rho_{j} \|_{L^\infty_{[0,T^*]}H_x^{\sstar}}+\| \bw_{j}\|_{L^\infty_{[0,T^*]}H_x^{2}} \lesssim \bar{M}_2,
\end{equation}

$\bullet$  higher-order norms
\begin{equation}\label{ebs1}
	\|  \bv_{j} \|_{L^\infty_{[0,T^*]}H_x^{\sstar+1}}+\| \rho_{j} \|_{L^\infty_{[0,T^*]}H_x^{\sstar+1}}+\| \bw_{j}  \|_{L^\infty_{[0,T^*]}H_x^{3}} \lesssim 2^{j}\bar{M}_2,
\end{equation}

$\bullet$  difference bounds
\begin{equation}\label{ebs2}
	\|  (\bv_{j+1}- \bv_{j}, \rho_{j+1}- \rho_{j}) \|_{L^\infty_{[0,T^*]} L_x^{2}} \lesssim 2^{-\sstar j}e_j,
\end{equation}
and
\begin{equation}\label{ebs00}
	\| \bw_{j+1}- \bw_{j} \|_{L^\infty_{[0,T^*]} L_x^{2}} \lesssim 2^{-(\sstar-\frac12)j} e_j,
\end{equation}
and
\begin{equation}\label{ebs3}
	\| \bw_{j+1}- \bw_{j} \|_{L^\infty_{[0,T^*]} \dot{H}_x^{2}} \lesssim e_j,
\end{equation}

$\bullet$  Strichartz estimates of the difference
\begin{equation}\label{ebs4}
	2^j\| \bv_{j+1}- \bv_{j}, \rho_{j+1}- \rho_{j}\|_{L^2_{[0,T^*]} L_x^{\infty}} \lesssim e_j.
\end{equation}
By using \eqref{rwe}, \eqref{ebs2} and \eqref{ebs3}, we get
\begin{equation*}
	\|\bv_j-\bv\|_{H_x^{\sstar}}\lesssim \textstyle{\sum}_{l \geq j}e_{l}.
\end{equation*}
Hence, combining with \eqref{Dss25}, we have
\begin{equation}\label{ebs5}
	\lim_{j\rightarrow \infty}\|\bv_j-\bv\|_{H_x^{\sstar}}\lesssim  \lim_{j\rightarrow \infty} \textstyle{\sum}_{l \geq j}e_{l}=0.
\end{equation}
Similarly, using \eqref{ebs2} and \eqref{ebs3}, so we also conclude that
\begin{equation}\label{ebs6}
	\begin{split}
		\lim_{j\rightarrow \infty}\|\rho_j-\rho\|_{H_x^{\sstar}}\lesssim  \lim_{j\rightarrow \infty} \textstyle{\sum}_{l \geq j}e_{l}=0,
		\\
		\lim_{j\rightarrow \infty}\|\bw_j-\bw\|_{H_x^{2}}\lesssim  \lim_{j\rightarrow \infty} \textstyle{\sum}_{l \geq j}e_{l}=0.
	\end{split}
\end{equation}
Combining \eqref{ebs5} and \eqref{ebs6}, we have proved the strong convergence $\lim_{j\rightarrow \infty}(\bv_j,\rho_j, \bw_j)=(\bv,\rho,\bw)$ in $H_x^{\sstar} \times H_x^{\sstar} \times H_x^{2}$. In a similar way, we can show that $\lim_{j\rightarrow \infty}(\partial_t \bv_j,\partial_t\rho_j, \partial_t\bw_j)=(\partial_t\bv,\partial_t\rho,\partial_t\bw)$ in $H_x^{\sstar-1} \times H_x^{\sstar-1} \times H_x^{1}$. Therefore, we conclude that $(\bv, \rho, \bw)$ is a strong solution of \eqref{fc1s} with the initial data $(\bv, \rho,\bw)|_{t=0}=(\bv_0,\rho_0,\bw_0)$. Moreover,
\begin{equation}\label{key1}
	\begin{split}
		& (\bv,\rho)\in C([0,T^*];H^{\sstar}),
		\qquad \qquad \quad \bw \in C([0,T^*];H^2),
		\\
		&(\partial_t \bv,\partial_t \rho)\in C([0,T^*];H_x^{\sstar-1}),
		\qquad \partial_t\bw \in C([0,T^*];H_x^{1}).
	\end{split}
\end{equation}

It now remains for us to prove \eqref{ebs0}-\eqref{ebs4}. By using \eqref{Duu0}, we get \eqref{ebs0}. By Corollary \ref{hes} , \eqref{Duu0}, \eqref{Duu2} and \eqref{mp1}, we can obtain
\begin{equation*}
	\begin{split}
		\|  \bv_{j} \|_{L^\infty_{[0,T^*]}H_x^{\sstar+1}}+\| \rho_{j} \|_{L^\infty_{[0,T^*]}H_x^{\sstar+1}}+\| \bw_{j}  \|_{L^\infty_{[0,T^*]}H_x^{3}} \lesssim  2^j \bar{M}_2.
	\end{split}
\end{equation*}
So we have proved \eqref{ebs1}.

It remains for us to prove \eqref{ebs2}-\eqref{ebs4}. We will use bootstrap arguments to conclude the proof. By assuming \eqref{ebs4}, if we can prove
\begin{equation}\label{ebs7}
	2^j\| \bv_{j+1}- \bv_{j}, \rho_{j+1}- \rho_{j}\|_{L^2_{[0,T^*]} L_x^{\infty}} \lesssim 2^{-\delta_1 j}e_j,
\end{equation}
then \eqref{ebs7} holds. Let us now prove \eqref{ebs2}-\eqref{ebs3}, \eqref{ebs7} if we have \eqref{ebs7}. Taking $h=0$ in \eqref{ff2} and using \eqref{ebs7}, \eqref{Duu0}, we get \eqref{ebs2}. By using \eqref{ur}, we get
\begin{equation}\label{Dss26}
	\begin{split}
		\partial_t (\bw_j-\bw_l) + (\bv_j \cdot \nabla)(\bw_j-\bw_l)=&-\{(\bv_j-\bv_l) \cdot \nabla\}\bw_l
	 + \{(\bw_j-\bw_l) \cdot \nabla)\}\bv_j
		\\
		& \ + (\bw_l \cdot \nabla) (\bv_j-\bv_l)
		\\
		=&-\{(\bv_j-\bv_l) \cdot \nabla\}\bw_l
		+ \{(\bw_j-\bw_l) \cdot \nabla)\}\bv_j
		\\
		& \ + \nabla\cdot \left\{ \bw_l (\bv_j-\bv_l)  \right\}+  (\bv_j-\bv_l) \textrm{div}\bw_l.
	\end{split}
\end{equation}
Multiplying \eqref{Dss26} with $\bw_{m+1}- \bw_{m}$ and integrating on $[0,T^*]\times \mathbb{R}^3$, and using $\textrm{div}\bw_l=\bw^l \cdot \nabla \rho^l$, we obtain
\begin{equation*}
	\begin{split}
		\| \bw_{m+1}- \bw_{m} \|^2_{L^\infty_{[0,T_*]} L^2_x} \leq  & \| (\bw_{m+1}- \bw_{m})(0) \|^2_{L^2_x}
		\\
		& + C\int^{T_*}_0  \| \nabla \bv_{m+1}, \nabla \rho_{m+1}\|_{L^\infty_x}\| \bw_{m+1}- \bw_{m} \|^2_{L^2_x}d\tau
		\\
		& + C\int^{T_*}_0 \| \bv_{m+1}-\bv_{m} \|_{H^{\frac12}_x}\| \bw_{m+1}- \bw_{m} \|_{L^2_x}\|\nabla \bw_{m} \|_{H^1_x}d\tau
		\\
		&+ C \int^{T_*}_0 \| \rho_{m+1}-\rho_{m} \|_{H^{\frac12}_x}\| \bw_{m+1}- \bw_{m} \|_{L^2_x}\|\nabla \bw_{m} \|_{H^1_x}d\tau
		\\
		&+ C \int^{T_*}_0 \| \rho_{m+1}-\rho_{m} \|_{L^2_x}\| \nabla ( \rho_{m+1}- \rho_{m}) \|_{L^2_x}\|\bw_{m} \|^2_{L^\infty_x}d\tau.
	\end{split}
\end{equation*}
This combining with \eqref{Duu0}, \eqref{Duu1} \eqref{ebs0}, \eqref{ebs2}, then \eqref{ebs00} holds.

Using \eqref{divwl}, and seeing \eqref{ebs0} and \eqref{ebs2}, we have
\begin{equation}\label{ebs8}
	\begin{split}
		\| \textrm{div}(\bw_{j+1} - \bw_j) \|_{\dot{H}_x^{1}} \leq & \|\bw_{j+1} - \bw_j \|_{{H}_x^{\frac{3}{2}+}} \|\rho_j\|_{{H}_x^{2}}+\| \bw_j \|_{{H}_x^{\frac{3}{2}+}} \| \rho_{j+1}- \rho_j\|_{{H}_x^{2}}
		\\
		\lesssim & 2^{-\delta_1 j}e_j\lesssim  e_j.
	\end{split}
\end{equation}
Using \eqref{DB1} and \eqref{DB2}, we obtain
\begin{equation}\label{ecs0}
	\begin{split}
		& \mathbf{T} ( \mathrm{curl} \mathrm{curl} \bw_j^i+\Omega_j^i  )
		=   \partial^i \big( 2 \partial_n v_{aj} \partial^n w_j^a \big) + \Gamma_j^i,
	\end{split}
\end{equation}
where
\begin{align} \label{eq:OM}
\Omega_j^i=-\epsilon^{ilk} \partial_l \rho_j \cdot \mathrm{curl}_k \bw_j- 2\partial^a {\rho_j} \delta_{ab}\partial^i w_j^b
\end{align}
where $\mathrm{curl}_k \bw_j$ the $k$:th component of $\mathrm{curl} \bw_j$, and $\Gamma_j^i$ is defined in the same way as above, but with $\bv, \rho, \bw, \mathrm{curl}\bw$ replaced by $\bv_j, \rho_j, \bw_j, \mathrm{curl}\bw_j$ in $\Gamma^i$ (cf.  \eqref{DB3}).
For simplicity, we denote
\begin{equation}\label{ecs2}
	Z(t)= \| \mathrm{curl}^i \mathrm{curl} ( \bw_{j+1}-\bw_j)+\Omega^{i}_{j+1}- \Omega^{i}_{j} \|_{L^{2}_x}.
\end{equation}
Multiplying $ \mathrm{curl}_i \mathrm{curl} ( \bw_{j+1}-\bw_{j})+\Omega_{i(j+1)}- \Omega_{ij} $, and integrating it on $[0,t] \times \mathbb{R}^3$($t \leq T^*$), we then obtain
\begin{equation}\label{ecs3}
	\begin{split}
		Z^2(t) \leq & Z^2(0) + Z^a+Z^b+Z^c,
	\end{split}
\end{equation}
where
\begin{equation*}
	\begin{split}
		Z^a=&
		| \int^t_0 \int_{\mathbb{R}^3} \partial^i \{  2 \partial_n (v_{a(j+1)}-v_{aj}) \partial^n w^{a}_{j+1} \}
		\cdot  \{  \mathrm{curl}_i \mathrm{curl} ( \bw_{j+1}-\bw_{j})+\Omega_{i(j+1)}- \Omega_{ij}  \} dxd\tau|,
		\\
		Z^b=&  | \int^t_0 \int_{\mathbb{R}^3} \partial^i \{  2 \partial_n v_{aj} \partial^n (w^{a}_{j+1}-w^{a}_{j} ) \}
		\cdot \{  \mathrm{curl}_i \mathrm{curl} ( \bw_{j+1}-\bw_{j})+\Omega_{i(j+1)}- \Omega_{ij}  \} dxd\tau|,
		\\
		Z^c=&  | \int^t_0 \int_{\mathbb{R}^3} ( \Gamma^{i}_{j+1}-\Gamma^{i}_{j} )
		\cdot \{  \mathrm{curl}_i \mathrm{curl} ( \bw_{j+1}-\bw_{j})+\Omega_{i(j+1)}- \Omega_{ij} \} dxd\tau|.
	\end{split}
\end{equation*}
Using \eqref{mp2}, it's direct for us to get
\begin{equation}\label{ecs4}
	G^2(0) \lesssim e^2_j.
\end{equation}
Integrating by parts in $Z^a$, we have
\begin{equation*}
	\begin{split}
		Z^a\leq &
		\left| \int^t_0 \int_{\mathbb{R}^3} \{  2 \partial_n (v_{a(j+1)}-v_{aj}) \partial^n w^{a}_{j+1} \}
		\cdot \partial^i \{  \mathrm{curl}_i \mathrm{curl} ( \bw_{j+1}-\bw_{j})+\Omega_{i(j+1)}- \Omega_{ij}  \} dxd\tau \right|
		\\
		=&  \left| \int^t_0 \int_{\mathbb{R}^3} \{  2 \partial_n (v_{a(j+1)}-v_{aj}) \partial^n w^{a}_{j+1} \}
		\cdot \partial^i \{  \Omega_{i(j+1)}- \Omega_{ij} \} dxd\tau \right|.
	\end{split}
\end{equation*}
By H\"older's inequality, Lemma \ref{lpe}, and \eqref{divf}, we have
\begin{equation*}
	\begin{split}
		Z^a
		\lesssim & \int^t_0 \| \partial(\bv_{j+1}-\bv_{j}) \|_{L^\infty_x}\| \bw_{j} \|^2_{H^{2}_x}   \| \rho_{j+1}-\rho_{j} \|_{H^{2}_x} d\tau
		\\
		& +  \int^t_0 \| \partial(\bv_{j+1}-\bv_{j}) \|_{L^\infty_x}\| \bw_{j} \|_{H^{2}_x}\| \rho_{j} \|_{H^{2}_x}   \| \bw_{j+1}-\bw_{j} \|_{H^{\frac32+}_x} d\tau
		\\
		&+ \int^t_0 \| \partial(\bv_{j+1}-\bv_{j}) \|_{L^\infty_x}\| \bw_{j} \|^2_{H^{2}_x} \| \rho_{j+1}-\rho_{j} \|_{H^{\frac32+}_x} d\tau
		\\ &		  +   \int^t_0 \| \bv_{j+1}-\bv_{j} \|_{H^{\frac32+}_x}\| \bw_{j} \|_{H^{2}_x} \| \rho_{j} \|_{H^{2}_x} \| \bw_{j+1}-\bw_{j} \|_{H^{2}_x} d\tau
	\end{split}
\end{equation*}
By using \eqref{ebs0}, \eqref{ebs2}, \eqref{ebs00}, \eqref{ebs4}, for $t\in [0, T_*]$, we have
\begin{equation}\label{ecs5}
	Z^a \lesssim 2^{-2\delta_1j}e^2_j.
\end{equation}
Using a similar way to handle $Z^a$, we can bound $Z^b$ by
\begin{equation*}
	\begin{split}
		Z^b
		\lesssim & \int^t_0 \| \partial\bv_{j}\|_{L^\infty_x}\| \bw_{j} \|_{H^{2}_x} \| \bw_{j+1}-\bw_{j} \|_{H^{\frac32+}_x}  \| \rho_{j+1}-\rho_{j} \|_{H^{2}_x} d\tau
		\\
		& +  \int^t_0 \| \partial \bv_{j} \|_{L^\infty_x}\| \bw_{j+1}-\bw_{j} \|_{H^{2}_x}\| \rho_{j} \|_{H^{2}_x}   \| \bw_{j+1}-\bw_{j} \|_{H^{\frac32+}_x} d\tau
		\\
		&+ \int^t_0 \| \partial(\rho_{j+1}-\rho_{j}) \|_{L^\infty_x} \| \bw_{j} \|_{H^{2}_x} \| \bv_{j} \|_{H^{2}_x} \| \bw_{j+1}-\bw_{j} \|_{H^{\frac32+}_x} d\tau
		\\ &		  +   \int^t_0 \|\partial \bv_{j}\|_{L^\infty_x}\| \bw_{j+1}-\bw_j \|_{H^{2}_x} \| \rho_{j} \|_{H^{2}_x} \| \bw_{j+1}-\bw_{j} \|_{H^{\frac32+}_x} d\tau.
	\end{split}
\end{equation*}
By \eqref{Duu0}, \eqref{ebs0}, \eqref{ebs2}, \eqref{ebs00}, \eqref{ebs4}, for $t\in [0, T^*]$, we have
\begin{equation}\label{ecs6}
	Z^b \lesssim 2^{-2\delta_1j}e^2_j.
\end{equation}
By H\"older's inequality and Lemma \ref{lpe}, we have
\begin{equation*}
	\begin{split}
		Z^c \lesssim  &  \int^t_0 ( \|\partial( \bv_{j+1}-\bv_{j} ) \|_{L^\infty_x} \| \bw_j \|_{ H^{2}_x}+\|\partial \bv_{j} \|_{L^\infty_x} \| \bw_{j+1}-\bw_j \|_{H^{2}_x} ) Z(\tau) d\tau.
	\end{split}
\end{equation*}
By using \eqref{lyr1}, for $t\in [0, T^*]$, we have
\begin{equation}\label{ecs7}
	Z^c \leq e^2_j.
\end{equation}
Gathering \eqref{ecs4}-\eqref{ecs7}, we show that
\begin{equation}\label{ecs8}
	Z^2(t) \lesssim e^2_j.
\end{equation}
Hence, by using \eqref{ebs2} and \eqref{ebs00}, for $t \in [0,T^*]$, we get
\begin{equation}\label{ecs9}
	\begin{split}
		\|\textrm{curl}\textrm{curl}(\bw^{l+1} - \bw^l) \|_{L_x^{2}} \leq & G(t)+ \|\Omega_{j+1}-\Omega_j \|_{L_x^{2}}
		\\
		\leq & G(t)+ \|\rho_{j+1}-\rho_j \|_{H_x^{\frac32+}}\|\bw_j \|_{H_x^{2}}+ \|\bw_{j+1}-\bw_j \|_{H_x^{\frac32+}}\|\rho_j \|_{H_x^{2}}
		\\
		\lesssim  & e_j.
	\end{split}
\end{equation}
Combining \eqref{ebs8} and \eqref{ecs9}, we have proved \eqref{ebs3}. By using \eqref{Duu21}-\eqref{Duu22} and taking $r=\frac{s}{2}=1+5\delta_{1}$, we can bound \eqref{ebs7} by
\begin{equation*}
	\begin{split}
		&2^j\| \bv_{j+1}- \bv_{j}, \rho_{j+1}- \rho_{j}\|_{L^2_{[0,T^*]} L_x^{\infty}}
		\\
		\leq & C2^j ( \| \bv_{j+1}- \bv_{j}, \rho_{j+1}- \rho_{j}\|_{L^\infty_{[0,T^*]} H_x^{1+5\delta_{1}}} + \| \bw_{j+1}- \bw_{j} \|_{L^\infty_{[0,T^*]} H_x^{\frac12+5\delta_1}})
		\\
		\leq & C2^{-3\delta_1 j} e_j.
	\end{split}
\end{equation*}
This proves \eqref{ebs4}.

\textit{Step 3: Obtaining a solution of linear wave by taking a limit.}
Reapting the proof in Step 1 and Step 2, we shall also find
\begin{equation}\label{key2}
	\begin{split}
	\lim_{j\rightarrow \infty} \|f_j-f\|_{H^r_x}=0, \quad \lim_{j\rightarrow \infty} \|\partial_t f_j- \partial_t f\|_{H^{r-1}_x}=0.
	\end{split}
\end{equation}
So taking $j\rightarrow \infty$, we have
	\begin{equation*}
	\begin{cases}
		& \square_g f=0, \qquad (t,x) \in (t_0,T^*]\times \mathbb{R}^3,
		\\
		&(f, \partial_t f)|_{t=0}=(f_0,f_1) \in H^r\times H^{r-1}.
	\end{cases}
\end{equation*}
Using \eqref{Duu22}, it follows
\begin{equation}\label{key3}
	\begin{split}
		&\|\left< \partial \right>^{a-1} d{f}\|_{L^2_{[0,T^*]} L^\infty_x}
		\leq  \bar{M}_4 (\|{f}_0\|_{{H}_x^r}+ \|{f}_1 \|_{{H}_x^{r-1}}),
		\\
		&\|{f}\|_{L^\infty_{[0,T^*]} H^{r}_x}+ \|\partial_t {f}\|_{L^\infty_{[0,T^*]} H^{r-1}_x} \leq  \bar{M}_4 (\| {f}_0\|_{H_x^r}+ \| {f}_1\|_{H_x^{r-1}}).
	\end{split}
\end{equation}
\textit{Step 4: Continuous dependence.}
Based on \eqref{key3}, we are able to prove the continuous dependence of solutions on the initial data, which can follow Section \ref{Sub}.

At this stage, we have finished the proof of Theorem \ref{dingli3}.
\end{proof}

It remains for us to prove Proposition \ref{DDL2}. Our idea is to prove Proposition \ref{DDL2} via small solutions, and using scaling technique and Strichartz estimates to obtain the solutions satisfying \eqref{Duu0}-\eqref{Duu2}. We next turn to introduce Proposition \ref{DDL3}, which is about small solutions.
\subsection{Proposition \ref{DDL3} for small data}
Let us now introduce Proposition \ref{DDL3} for small data, where the initial regularity of the vorticity \eqref{DP30} is better than \eqref{chuzhi3}.
\begin{proposition}\label{DDL3}
Let $\epsilon_2$ and $\epsilon_3$ be stated in \eqref{a0}. Let\footnote{We take $\delta_1$ be a fixed number.} $\delta_1=\frac{\sstar-2}{10}>0$. For each small, smooth initial data $(\bv_0, \rho_0, \bw_0)$ supported in $B(0,c+2)$ which satisfies\footnote{To obtain the smallness norm of $\| \bw_0\|_{H^{2+\delta_1}}$, the price we pay is that the time interval is smaller, please see \eqref{DTJ}.}
\begin{equation}\label{DP30}
	\begin{split}
	&\|(\bv_0, \rho_0) \|_{H^{\sstar}} + \| \bw_0\|_{H^{2+\delta_1}}  \leq \epsilon_3,
	\end{split}
	\end{equation}
there exists a smooth solution $(\bv, \rho, \bw)$ to \eqref{fc1s} on $[-2,2] \times \mathbb{R}^3$ satisfying
\begin{equation}\label{DP31}
 \|(\bv, \rho)\|_{L^\infty_{[-2,2]}H_x^{\sstar}}+\|\bw\|_{L^\infty_{[-2,2]}H_x^{2+\delta_1}} \leq \epsilon_2.
\end{equation}
Furthermore, the solution satisfies the properties

$\mathrm{(1)}$ dispersive estimate for $\bv$, $\rho$, and $\bv_+$
	\begin{equation}\label{DP32}
	\|d \bv, d \rho\|_{L^2_{[-2,2]} C^{\delta_*}_x}+\|d \bv_+, d {\rho}, d \bv\|_{L^2_{[-2,2]} L^\infty_x} \leq \epsilon_2,
	\end{equation}
	with
\begin{equation}\label{eq:d*def}
\delta_*\in (0, \sstar-2).
\end{equation}

$\mathrm{(2)}$ Let $f$ satisfy equation \eqref{Linear}\footnote{Here the acoustic metric is better than in Proposition \ref{DDL2}, for we improve the regularity of vorticity to $2+\delta_1$, and $\delta_1$ is a fixed number.}. For each $1 \leq r \leq s+1$, the Cauchy problem \eqref{linear} is well-posed in $H^r \times H^{r-1}$, and the following estimate holds\footnote{This Strichartz estimates is stronger than \eqref{Duu2}, for the regularity of the vorticity is different between Proposition \ref{DDL2} and Proposition \ref{DDL3}.}:
	\begin{equation}\label{DP33}
	\|\left< \partial \right>^k f\|_{L^2_{[-2,2]} L^\infty_x} \lesssim  \| f_0\|_{H^r}+ \| f_1\|_{H^{r-1}},\quad \ k<r-1,
	\end{equation}
and the same estimates hold with $\left< \partial \right>^k$ replaced by $\left< \partial \right>^{k-1}d$.
\end{proposition}
\begin{proof}
%\medskip\begin{proof}[Proof of Proposition \ref{DDL3} ]
Setting $h=0$ in Proposition \ref{p1} and replacing the regularity exponent $s$ and $s_0$ to $s=\sstar, s_0=2+\delta_1$, then we can directly obtain the conclusion in Proposition \ref{DDL3}.
\end{proof}
\subsection{Proof of Proposition \ref{DDL2} by using Proposition \ref{DDL3}}\label{keypro}
In Proposition \ref{DDL3}, the corresponding statement considers the small, supported data. While, the initial data is large in Proposition \ref{DDL2}, and regularity of the vorticity in Proposition \ref{DDL2} is weaker than Proposition \ref{DDL3}. So we divide the proof into three subparts. Firstly, we can only get a solution of Proposition \ref{DDL2} on a short time interval, which is presented in subsection \ref{esess}. Then we are able to get some good Strichartz estimates for low-frequency and mid-frequency of solutions on these short time intervals by paying a loss of derivatives, which is proved in subsection \ref{esest}. After that, we can extend these solutions on a regular time-interval in subsection \ref{finalk}. At the same time, the solutions of linear wave can also be extended to this regular time-interval, which is obtained in \ref{finalq}.
\subsubsection{Energy estimates and Strichartz estimates on a short time-interval.} \label{esess}
Firstly, by using a scaling method, we can reduce the initial data in Proposition \ref{DDL2} to be small. Considering \eqref{Dss0}, take space-time scaling
\begin{equation*}
	\begin{split}
		& \widetilde{\bv}_{0j}={\bv}_{0j}(Tt,Tx), \quad \widetilde{\rho}_{0j}=\rho_{0j}(Tt,Tx).
	\end{split}
\end{equation*}
Referring to \eqref{pw11}, we set
\begin{equation*}
		\widetilde{\bw}_{0j}=\mathrm{e}^{\widetilde{\rho}_{0j}}\mathrm{curl} \widetilde{\bv}_{0j},
\end{equation*}
then we have the bounds
\begin{equation}\label{Dss2}
	\begin{split}
		 \| \widetilde{\bv}_{0j} \|_{\dot{H}^{\sstar}} +  \| \widetilde{\rho}_{0j} \|_{\dot{H}^{\sstar}} \leq  & T^{\sstar-\frac32}\| (\| {\bv}_{0j} \|_{\dot{H}^{\sstar}}+ {\rho}_{0j} \|_{\dot{H}^{\sstar}}),
		\\
		\| \widetilde{\bw}_{0j}\|_{\dot{H}^{2+\delta_1}} =& \| \mathrm{e}^{\widetilde{\rho}_{0j}}\mathrm{curl} \widetilde{\bv}_{0j} \|_{\dot{H}^{2+\delta_1}}
		\\
		\leq & T^{\frac32+\delta_{1}} \|\mathrm{e}^{{\rho}_{0j}}\mathrm{curl} {\bv}_{0j} \|_{\dot{H}^{2+\delta_1}}
		\\
		= & T^{\frac32+\delta_{1}} \|{\bw}_{0j} \|_{\dot{H}^{2+\delta_1}}
		\\
		\leq & T^{\frac32+\delta_{1}} 2^{\delta_{1} j }\|{\bw}_{0j} \|_{\dot{H}^{2}}.
	\end{split}
\end{equation}
Set
\begin{equation}\label{DTJ}
	T^*_j:=2^{-\delta_1j}[E(0)]^{-1}.
\end{equation}
Taking $T$ in \eqref{Dss2} as $T^*_j$,  \eqref{Dss2} yields
\begin{equation}\label{pp7}
	\begin{split}
		& \| \widetilde{\bv}_{0j} \|_{\dot{H}^{\sstar}}+\| \widetilde{\rho}_{0j}\|_{\dot{H}^{\sstar}} \leq 2^{-\delta_1 j({\frac12+10\delta_{1}})}[E(0)]^{-(\frac12+10\delta_{1})},
		\\
		& \| \widetilde{\bw}_{0j}\|_{\dot{H}^{2+\delta_1}}  \leq  2^{-(\frac{\delta_1}{2}+\delta_{1}^2)j} [E(0)]^{-(\frac12+\delta_{1})}.
	\end{split}
\end{equation}
Note $\| \widetilde{\bv}_{0j}\|_{L^\infty}\leq  \|{\bv}_{0j}\|_{L^\infty}$ and $\| \widetilde{\rho}_{0j}\|_{L^\infty}  \leq  \|{\rho}_{0j}\|_{L^\infty} $. By \eqref{pu00}, we have
\begin{equation}\label{pp60}
	\begin{split}
		\| \widetilde{\bv}_{0j}\|_{L^\infty}+ \| \widetilde{\rho}_{0j}\|_{L^\infty}  \leq  \|{\bv}_{0j}\|_{L^\infty} + \|{\rho}_{0j}\|_{L^\infty} \leq  C_0.
	\end{split}
\end{equation}
hold. For a small parameter $\epsilon_3$ stated in Proposition \ref{DDL3}, we can choose\footnote{Recall that $\delta_{1}=\frac{s-2}{10}$, $E(0)=C(M_*+M_*^2)$. So $N_0$ depends on $s, C_0, c_0$ and $M_*$.} $N_0=N_0(\delta_{1},E(0))$ such that
\begin{equation}\label{pp8}
\begin{split}
	& 2^{-\delta_1 N_0}(1+E^6(0))\ll \epsilon_3,
	\\
	& 2^{-\delta_{1} N_0 } (1+C^3_*)\{ 1+ \frac{C_*}{E(0)}(1+ E^3(0))^{-1} \} \leq 1,
	\\
	& C2^{-\frac{\delta_1}{2}N_0} (1+E^3(0))E^{-\frac12}(0)  [\frac{1}{3}(1-2^{-\delta_{1}})]^{-2} \leq 2.
\end{split}	
\end{equation}
Above, $C_*=C_*(\delta_{1}, E(0))$ is denoted by
\begin{equation}\label{Cstar}
	\begin{split}
		C_*=C(E(0)+E^3(0))\exp\{  2(1+E^3(0)) \textrm{e}^2 \} .
	\end{split}	
\end{equation}
As a result, for $j \geq N_0$, we have
\begin{equation}\label{Dss3}
	\begin{split}
		\| \widetilde{\bv}_{0j} \|_{\dot{H}^{\sstar}}+\| \widetilde{\rho}_{0j}\|_{\dot{H}^{\sstar}}+\| \widetilde{\bw}_{0j}\|_{\dot{H}^{2+\delta_1}}  \leq  \epsilon_3.
	\end{split}
\end{equation}
For \eqref{Dss3} is about the homogeneous norm, so we need to use physical localization to get the same bound for in-homogeneous norm. Utilize the standard physical localization as in Section 4. Since the speed of propagation of \eqref{fc1s} is finite, we may set $c$ be the largest speed of propagation of \eqref{fc1s}. Then the solution in a unit cylinder $[-1,1]\times B(\by,1)$ is uniquely determined by the initial data in the ball $B(\by,1+c)$. Hence
it is natural to truncate the initial data in a slightly larger region. Set $\chi$ be a smooth function supported in $B(0,c+2)$,  and which equals $1$ in $B(0,c+1)$. For given $\by \in \mathbb{R}^3$, we define the localized initial data for the velocity and density near $\by$:
\begin{equation}\label{yyy0}
	\begin{split}
		\bar{\bv}_{0j}(\bx)=&\chi(\bx-\by)\left( \widetilde{\bv}_{0j}(\bx)- \widetilde{\bv}_{0j}(\by)\right),
		\\
		\bar{\rho}_{0j}(\bx)=&\chi(\bx-\by)\left( \widetilde{\rho}_{0j}(\bx)-\widetilde{\rho}_{0j}(\by)\right).
	\end{split}
\end{equation}
Referring to \eqref{pw11}, we set
\begin{equation}\label{yyy1}
	\bar{\bw}_{0j}=%\bar{\rho}^{-1}
	\mathrm{e}^{-\bar{\rho}_{0j}}\mathrm{curl}\bar{\bv}_{0j}.
\end{equation}
By \eqref{Dss3}, \eqref{yyy0} and \eqref{yyy1}, so we have
\begin{equation}\label{sS4}
	\begin{split}
		\| \bar{\bv}_{0j} \|_{{H}^{\sstar}}+\| \bar{\rho}_{0j}\|_{{H}^{\sstar}}+\| \bar{\bw}_{0j} \|_{{H}^{2+\delta_1}}  \leq  & C(\| \widetilde{\bv}_{0j} \|_{\dot{H}^{\sstar}}+\| \widetilde{\rho}_{0j}\|_{\dot{H}^{\sstar}}+\| \widetilde{\bw}_{0j} \|_{\dot{H}^{2+\delta_1}})
		\\
		\leq & \epsilon_3.
	\end{split}
\end{equation}
For every $j$, by Proposition \ref{DDL3}, there is a smooth solution $(\bar{\bv}_j, \bar{\rho}_j, \bar{\bw}_j)$ on $[-2,2]\times \mathbb{R}^3$ satisfying
\begin{equation}\label{yz0}
	\begin{cases}
		\square_{\widetilde{g}_j} \bar{v}_j^i=-\mathrm{e}^{\bar{\rho}_j+\tilde{\rho}_{0j}(y)}\widetilde{c}_s^2 \mathrm{curl} \bar{\bw}^i+\widetilde{\mathcal{Q}}^i,
		\\
		\square_{\widetilde{g}_j} \bar{\rho}_j=\widetilde{\mathcal{D}},
		\\
		\widetilde{\mathbf{T}} \bar{\bw}_j=\bar{\bw}_j\cdot \nabla \bar{\bv}_j,
		\\
		(\bar{\bv}_j, \bar{\rho}_j, \bar{\bw}_j)|_{t=0}=(\bar{\bv}_{0j}, \bar{\rho}_{0j}, \bar{\bw}_{0j}).
	\end{cases}
\end{equation}
Above, the quantities $\widetilde{c}^2_s$ and $\widetilde{{g}_j}$ are defined by
\begin{equation}\label{QRY}
	\begin{split}
		\widetilde{c}^2_s:&= \frac{dp}{d{\rho}}(\bar{\rho}_j+\tilde{\rho}_{0j}(y)),
		\\
		\widetilde{g}_j:&=g({\bar\bv}_j+\tilde{\bv}_{0j}(y), {\bar\rho}_j+\tilde{\rho}_{0j}(y)),
		\\
		\widetilde{\mathbf{T}}:&=\partial_t+ ({\bar\bv}_j+\tilde{\bv}_{0j}(y))\cdot \nabla,
	\end{split}
\end{equation}
and
\begin{equation*}
	\begin{split}
		\widetilde{\mathcal{Q}}^i=& \widetilde{\mathcal{Q}}^{i\alpha\beta }_{j} \partial_\alpha \bar{\rho} \partial_\beta \bar{v}^j+ \widetilde{\mathcal{Q}}^{\alpha\beta }_{2j} \partial_\alpha \bar{v}^i \partial_\beta \bar{v}^j,
		\\
		\widetilde{\mathcal{D}}=& \widetilde{\mathcal{D}}^{\alpha\beta }_{j} \partial_\alpha \bar{\rho} \partial_\beta \bar{v}^j+ \widetilde{\mathcal{D}}^{\alpha\beta }_{2} \partial_\alpha \bar{\rho} \partial_\beta \bar{\rho}+ \widetilde{\mathcal{D}}^{i\alpha\beta }_{3j} \partial_\alpha \bar{v}_i \partial_\beta \bar{v}^j,
	\end{split}
\end{equation*}
and $\widetilde{\mathcal{Q}}^{i\alpha\beta }_{j}$, $\widetilde{\mathcal{Q}}^{\alpha\beta}_{2j}$, $\widetilde{\mathcal{D}}^{\alpha\beta }_{j}$, $\widetilde{\mathcal{D}}^{\alpha\beta}_{2}$, and $\widetilde{\mathcal{D}}^{i\alpha\beta}_{3j}$ have the same formulations with $\mathcal{Q}^{i\alpha\beta }_{j}$, $\mathcal{Q}^{\alpha\beta}_{2j}$, $\mathcal{D}^{\alpha\beta }_{j}$, $\mathcal{D}^{\alpha\beta}_{2}$,   $\mathcal{D}^{i\alpha\beta}_{3j}$ by replacing $({\bv}, {\rho}, {h})$ to $(\bar{\bv}_j+\tilde{\bv}_0(\by), \bar{\rho}_j+\tilde{\rho}_0(\by) )$.

Using Proposition \ref{DDL3} again, on $[-2,2]\times \mathbb{R}^3$, we find that
\begin{equation}\label{Dsee0}
	\|\bar{\bv}_j\|_{L^\infty_{[-2,2]}H_x^{\sstar}}+ \|\bar{\rho}_j\|_{L^\infty_{[-2,2]}H_x^{\sstar}}+ \| \bar{\bw}_j \|_{L^\infty_{[-2,2]}H_x^{2+\delta_1}} \leq \epsilon_2,
\end{equation}
and
\begin{equation}\label{Dsee1}
	\|d\bar{\bv}_j, d\bar{\rho}_j \|_{L^2_{[-2,2]}C^{\delta_*}_x} \leq \epsilon_2.
\end{equation}
for $\delta_*$ as in \eqref{eq:d*def}. Furthermore, using \eqref{DP33} in Proposition \ref{DDL3}, for $1\leq r_1 \leq s+1$, the linear equation
\begin{equation}\label{Dsee2}
	\begin{cases}
		&\square_{ \widetilde{g}_j } F=0,\quad [-2,2]\times \mathbb{R}^3,
		\\
		&{F}(t_0,x)={F}_0, \ \partial_t {F}(t_0,x)={F}_1, \quad t_0 \in [-2,2],
	\end{cases}
\end{equation}
admits a solution ${F} \in C([-2,2],H_x^{r_1})\times C^1([-2,2],H_x^{r_1-1})$ and for $k<r_1-1$, the following estimate holds\footnote{For all $j$, the initial norm of the initial data \eqref{sS4} is uniformly controlled by the same small parameter $\epsilon_3$, and the regularity of the initial data \eqref{sS4} only depends on $\sstar$, so the constant in \eqref{3s1} is uniform for all $\widetilde{{g}}$ depending on $j$.}:
\begin{equation}\label{3s1}
	\begin{split}
	\|\left< \partial \right>^{k-1}d{F}\|_{L^2_{[-2,2]} L^\infty_x}+ \|\left< \partial \right>^{k-1}d{F}\|_{L^2_{[-2,2]} L^\infty_x} \leq  & C(\| {F}_0\|_{H_x^r}+ \| {F}_1\|_{H_x^{r-1}} ).
\end{split}
\end{equation}
Here $\widetilde{g}_j$ is given by \eqref{QRY}.

Note \eqref{yz0}. Then the function $(\bar{\bv}_j+\widetilde{\bv}_{0j}(\by), \bar{\rho}_j+\widetilde{\rho}_{0j}(\by), \mathrm{e}^{-\tilde{\rho}_{0j}(\by)} \bar{\bw}_j)$ is also a solution of the following system
\begin{equation}\label{yz1}
	\begin{cases}
		\square_{\widetilde{g}_j} ( \bar{\bv}_j+\widetilde{\bv}_{0j}(\by) )=-\mathrm{e}^{\bar{\rho}_j+\tilde{\rho}_{0j}(\by)}\widetilde{c}_s^2 \mathrm{curl} (\mathrm{e}^{-\tilde{\rho}_{0j}(\by)} \bar{\bw}_j)+\widetilde{\mathcal{Q}}^i,
		\\
		\square_{\widetilde{g}_j} ( \bar{\rho}_j+\widetilde{\rho}_{0j}(\by) )=\widetilde{\mathcal{D}},
		\\
		\widetilde{\mathbf{T}} ( \mathrm{e}^{-\tilde{\rho}_{0j}(\by)} \bar{\bw}_j)=\mathrm{e}^{-\tilde{\rho}_{0j}(\by)}\bar{\bw}_j \cdot \nabla ( \bar{\bv}_j+\widetilde{\bv}_{0j}(\by) ),
		\\
		(\bar{\bv}_j+\widetilde{\bv}_{0j}(y), \bar{\rho}_j+\widetilde{\rho}_{0j}(y), \mathrm{e}^{-\tilde{\rho}_{0j}(y)} \bar{\bw}_j )|_{t=0}=(\widetilde{\bv}_{0j}, \widetilde{\rho}_{0j}, \widetilde{\bw}_{0j}).
	\end{cases}
\end{equation}
Consider the restrictions, for $\by\in \mathbb{R}^3$,
\begin{equation}\label{Dsee3}
	\left( \bar{\bv}_j+\widetilde{\bv}_{0j} (\by) \right)|_{\mathrm{K}^y},
	\quad (\bar{\rho}_j+\widetilde{\rho}_{0j} (\by) )|_{\mathrm{K}^y},
	% \\
	%  \quad (\bar{h}_j+\widetilde{h}_{0j} (y) )|_{\mathrm{K}^y},
	\\
	\quad \bar{\bw}_j|_{\mathrm{K}^y},
\end{equation}
where $\mathrm{K}^y=\left\{ (t,\bx): ct+|\bx-\by| \leq c+1, |t| <1 \right\}$, then the restrictions \eqref{Dsee3} solve \eqref{CEE}-\eqref{id} on $\mathrm{K}^y$. By  finite speed of propagation, a smooth solution $(\bar{\bv}_j+\widetilde{\bv}_{0j}(\by), \bar{\rho}_j+ \widetilde{\rho}_{0j}(\by), \mathrm{e}^{-\widetilde{\rho}_{0j}(y)} \bar{\bw}_j)$ solves \eqref{fc1} on $\mathrm{K}^y$. Therefore, if we set
\begin{equation}\label{Dsee4}
	\begin{split}
		\widetilde{\bv}_j(t,\bx)  &=\textstyle{\sum}_{\by \in 3^{-\frac12} \mathbb{Z}^3 }\psi(\bx-\by) (\bar{\bv}_j+\widetilde{\bv}_{0j}(\by)),
		\\
		\widetilde{\rho}_j(t,\bx)  &=\textstyle{\sum}_{\by \in 3^{-\frac12} \mathbb{Z}^3}\psi(\bx-\by) ( \bar{\rho}_j+ \widetilde{\rho}_{0j}(\by)),
		\\
		\widetilde{\bw}_j(t,\bx)  &=\mathrm{e}^{-\widetilde{\rho}_j}\mathrm{curl}\widetilde{\bv}_j,%\sum_{\by \in 3^{-\frac12} \mathbb{Z}^3}\psi(\bx-\by) \bar{\bw}_j,
	\end{split}
\end{equation}
then $(\widetilde{\bv}_j,\widetilde{\rho}_j,\widetilde{\bw}_j)$ is a smooth solution of \eqref{fc1} on $[-1,1]\times \mathbb{R}^3$ with the initial data $(\widetilde{\bv}_j, \widetilde{\rho}_j, \widetilde{\bw}_j )|_{t=0}=(\widetilde{\bv}_{0j}, \widetilde{\rho}_{0j}, \widetilde{\bw}_{0j})$, where $\psi$ is supported in the unit ball such that
\begin{equation*}
	\textstyle{\sum}_{\by \in 3^{-\frac12} \mathbb{Z}^3 } \psi(\bx-\by)=1.
\end{equation*}
On the other hand, we recall that the initial data $(\widetilde{\bv}_{0j}, \widetilde{\rho}_{0j}, \widetilde{\bw}_{0j})$ is a scaling of $(\bv_{0j}, \rho_{0j}, \bw_{0j})$ with the space-time scale $T^*_j$, and the system \eqref{fc1s} is scaling-invariant. Then, the function
\begin{equation*}
	({\bv}_{j}, {\rho}_{j}, {\bw}_{j})=(\widetilde{\bv}_{j},\widetilde{\rho}_{j}, \widetilde{\bw}_{j}) ((T^*_j)^{-1}t,(T^*_j)^{-1}x),
\end{equation*}
is a solution of \eqref{fc1s} on the space-time $[0,T^*_j]\times \mathbb{R}^3$($T^*_j$ is stated in \eqref{DTJ}), and it has the initial data
\begin{equation*}
	({\bv}_{j}, {\rho}_{j}, {\bw}_{j})|_{t=0}=({\bv}_{0j}, {\rho}_{0j}, {\bw}_{0j}).
\end{equation*}
Referring \eqref{Dsee4} and using \eqref{Dsee2}-\eqref{3s1}, we can see
\begin{equation}\label{Dsee5}
	\begin{split}
		 \|d\widetilde{\bv}_j, d\widetilde{\rho}_j\|_{L^2_{[0,1]}C^{\delta_*}_x}
		\leq & \sup_{\by \in 3^{-\frac12} \mathbb{Z}^3}  \|d\bar{\bv}_j, d\bar{\rho}_j\|_{L^2_{[0,1]}{C^{\delta_*}_x}}
		\\
		\leq & C(\|\bar{\bv}_{0j}\|_{H_x^s}+ \|\bar{\rho}_{0j}\|_{H_x^s}+ \| \bar{\bw}_{0j} \|_{H_x^{2}}).
	\end{split}
\end{equation}
By changing of coordinates $(t,\bx)\rightarrow ((T_j^*)^{-1}t,(T_j^*)^{-1}\bx)$ for each $j\geq 1$, we get
\begin{equation}\label{Dsee6}
	\begin{split}
		\|d{\bv}_j, d{\rho}_j\|_{L^2_{[0,T^*_j]}C^{\delta_*}_x}
		\leq & (T^*_j)^{-(\frac12+{\delta_*})}\|d\widetilde{\bv}_j, d\widetilde{\rho}_j\|_{L^2_{[0,1]}C^{\delta_*}_x}.
	\end{split}
\end{equation}
Using \eqref{Dsee5} and \eqref{Dsee6}, and \eqref{DTJ}, and ${\delta_*} \in (0,s-2)$, it follows
\begin{equation}\label{Dsee7}
	\begin{split}
		\|d{\bv}_j, d{\rho}_j\|_{L^2_{[0,T^*_j]}C^{\delta_*}_x}
	\leq	&C (T^*_j)^{-(\frac12+{\delta_*})}(\|\bar{\bv}_{0j}\|_{H_x^s}+ \|\bar{\rho}_{0j}\|_{H_x^s}+ \| \bar{\bw}_{0j} \|_{H_x^{2}})
		\\
		\leq & C (T^*_j)^{-(\frac12+{\delta_*})} \{  (T^*_j)^{s-\frac32}\|{\bv}_{0j}, {\rho}_{0j}\|_{\dot{H}_x^s}
		+ (T^*_j)^{\frac32}\| {\bw}_{0j} \|_{\dot{H}_x^{2}} \}
		\\
		\leq & C (\|{\bv}_{0j},{\rho}_{0j}\|_{H_x^s}+ \| {\bw}_{0j} \|_{H_x^{2}}).
	\end{split}
\end{equation}
Combining \eqref{Dsee7} and \eqref{pu0}, we obtain
\begin{equation}\label{yz4}
	\begin{split}
		\|d{\bv}_j, d{\rho}_j\|_{L^2_{[0,T^*_j]}C^{\delta_*}_x}
		\leq & C (1+E(0)).
	\end{split}
\end{equation}
Set
\begin{equation}\label{Etr}
	E(T^*_{j})=\|\bv_{j}\|_{L^\infty_{[0,T^*_{j}]} H^{\sstar}_x}+\|\rho_{j}\|_{L^\infty_{[0,T^*_{j}]} H^{\sstar}_x}+
	\|\bw_{j}\|_{L^\infty_{[0,T^*_{j}]} H^{2}_x}.
\end{equation}
By using \eqref{yz4} and Theorem \ref{TT2}, we have\footnote{Note that $\| \bw_j \|_{L^\infty_{[0,T^*_{j}]}H_x^{2+}}$ is not uniform bounded. But $\| \bar \bw_j \|_{L^\infty_{[-2,2]}H_x^{2+}} \lesssim \epsilon_2$ for all $j \geq 1$.}
\begin{equation}\label{AMM3}
	E(T^*_j) \leq C( E(0)+E^3(0) )\mathrm{e}^{C(1+E(0))}.
\end{equation}
We also expect some estimates for the homogeneous linear wave equation on a short time interval $[0,T^*_j]$. Following the proof of \eqref{po10} and \eqref{po13} for Equation \eqref{po3}, we deduce that for $1\leq r_1 \leq s+1$, then
\begin{equation}\label{ru03}
	\begin{cases}
		\square_{{g}_j} F=0, \quad [0,T^*_j]\times \mathbb{R}^3,
		\\
		(F,F_t)|_{t=0}=(F_0,F_1)\in H_x^{r_1} \times H^{r_1-1}_x,
	\end{cases}
\end{equation}
has a unique solution on $[0,T^*_j]$. Moreover, for $a_1<r_1-1$, we have
\begin{equation}\label{ru04}
	\begin{split}
		\|\left< \partial \right>^{a_1-1} dF\|_{L^2_{[0,T^*_j]} L^\infty_x}
		\leq & C(\|{F}_0\|_{{H}_x^{r_1}}+ \| {F}_1 \|_{{H}_x^{r_1-1}}),
	\end{split}
\end{equation}
and
\begin{equation}\label{ru05}
	\begin{split}
		\|{F}\|_{L^\infty_{[0,T^*_j]} H^{r_1}_x}+ \|\partial_t {F}\|_{L^\infty_{[0,T_j]} H^{r_1-1}_x} \leq  C(\| {F}_0\|_{H_x^{r_1}}+ \| {F}_1\|_{H_x^{r_1-1}}).
	\end{split}
\end{equation}
Although there is a uniform bound for $\bv_{j}, \rho_{j}$ and $\bw_{j}$,  the size of the time interval $[0,T^*_{j}]$ is not uniform. To construct a solution as a limit of these smooth solutions, we have to extend these solutions on a regular time-interval.

\subsubsection{A loss of Strichartz estimates on a short time-interval.} \label{esest}
To see the energy estimates and Strichartz estimates of the sequence $({\bv}_{j}, {\rho}_{j}, {\bw}_{j})$ very clearly, let us give some precise analysis on $({\bv}_j,{\rho}_j,{\bw}_j)$, for there is a strong Strichartz estimate \eqref{yz4} and \eqref{ru03}-\eqref{ru05}.
Next, we will discuss it into two cases\footnote{Here, we mainly inspired by Tataru \cite{T1}, Bahouri-Chemin \cite{BC2}, and Ai-Ifrim-Tataru \cite{AIT}. Of course, our work based on the property of Strichartz estimates in Proposition \ref{r6} and careful analysis on vorticity. Moreover, we eventually conclude it by induction method.}, the high frequency and low frequency for $({\bv}_j, {\rho}_j)$.

\textit{Case 1: High frequency($k \geq j$).} Using \eqref{yz4}, we get
\begin{equation}\label{yz6}
	\begin{split}
		& \| d \bv_j, d \rho_j \|_{L^2_{[0,T^*_j]}C^a_x} \leq C(1+E(0)), \quad  a \in [0,s-2).
	\end{split}
\end{equation}
For $k \geq j$, by Bernstein inequality and \eqref{yz6}, we have
\begin{equation}\label{yz9}
	\begin{split}
		\| P_{k} d \bv_j, P_{k} d \rho_j \|_{L^2_{[0,T^*_j]}L^\infty_x}
		\leq & 2^{-ka}\| P_k d \bv_j, P_k d \rho_j \|_{L^2_{[0,T^*_j]}C^a_x}
		\\
		\leq & C2^{-ja} \| d \bv_j, d \rho_j \|_{L^2_{[0,T^*_j]}C^a_x}
		\\
		\leq &  C2^{-ja} (1+E(0)).
	\end{split}
\end{equation}
Taking $a=9\delta_1$ in \eqref{yz9}, so we obtain
\begin{equation}\label{yz10}
	\begin{split}
		\| P_{k} d \bv_j, P_{k} d \rho_j \|_{L^2_{[0,T^*_j]}L^\infty_x}
		\leq & C   2^{-\delta_{1}k} \cdot (1+E^3(0)) 2^{-7\delta_{1}j},, \quad k \geq j,
	\end{split}
\end{equation}
\textit{Case 2: Low frequency($k<j$).} In this case, it's a little different from $k \geq j$. Fortunately, there is some good estimates for difference terms $P_k(d{\rho}_{m+1}-d{\rho}_m)$ and $P_k(d{\bv}_{m+1}-d{\bv}_{m})$. Following \eqref{dvc} and \eqref{etad}, we set
\begin{equation}\label{etad0}
	-\Delta \bv_{-m}=\mathrm{e}^{\rho_m}\mathrm{curl} \bw_m, \quad \bv_m=\bv_{+m}+ \bv_{-m}.
\end{equation}
We will obtain some good estimates of $P_k(d{\bv}_{m+1}-d{\bv}_m)$ and $P_k(d{\rho}_{m+1}-d{\rho}_m)$ by using a Strichartz estimates \eqref{ru04} of \eqref{ru03}. The good estimates is from a loss of derivatives of Strichartz estimates.

By Bernstein inequality, we get
\begin{equation}\label{yz12}
	\begin{split}
		\|{\bv}_{0(m+1)}-{\bv}_{0m}\|_{L^2}
		\lesssim & 2^{-sm} \|{\bv}_{0(m+1)}-{\bv}_{0m}\|_{\dot{H}^s}.
	\end{split}	
\end{equation}
Similarly, we obtain
\begin{equation}\label{yz13}
	\begin{split}
		& \|{\rho}_{0(m+1)}-{\rho}_{0m}\|_{L^2} \lesssim   2^{-sm}\|{\bv}_{0(m+1)}-{\bv}_{0m}\|_{\dot{H}^s},
		\\
		& \|{\bw}_{0(m+1)}-{\bw}_{0m}\|_{L^2} \lesssim   2^{-2m}\|{\bw}_{0(m+1)}-{\bw}_{0m}\|_{\dot{H}^2}.
	\end{split}	
\end{equation}
We claim that
\begin{align}\label{yu0}
	& \|\bv_{m+1}-\bv_{m}, {\rho}_{m+1}-{\rho}_{m}\|_{L^\infty_{ [0,T^*_{m+1}] } L^2_x} \leq 2^{-(\sstar- \delta_1)m}E(0),
	\\\label{yu1}
	& \|\bw_{m+1}-\bw_{m}\|_{L^\infty_{ [0,T^*_{m+1}]  } L^2_x} \leq 2^{- (\sstar-\frac12- \delta_1)m}E(0).
\end{align}
Let ${\bU}_{m}=(\bv_{m},p(\rho_{m}))^\mathrm{T} $. Then ${\bU}_{m+1}-{\bU}_{m}$ satisfies
\begin{equation}\label{yz14}
	\begin{cases}
		& A^0({\bU}_{m+1}) \partial_t ( {\bU}_{m+1}- {\bU}_{m}) + A^i({\bU}_{m+1}) \partial_i ( {\bU}_{m=1}- {\bU}_{m})=\Pi_m,
		\\
		& ( {\bU}_{m+1}-{\bU}_{m} )|_{t=0}= {\bU}_{0(m+1)}-{\bU}_{0m},
	\end{cases}
\end{equation}
where
\begin{equation}\label{Fhz}
	\Pi_m=-[A^0({\bU}_{m+1})-A^0({\bU}_{m}) ]\partial_t  {\bU}_{m}- [A^i({\bU}_{m+1})-A^i({\bU}_{m}) ]\partial_i  {\bU}_{m}.
\end{equation}
By \eqref{Fhz}, we can see
\begin{equation}\label{Fhz0}
	|\Pi_m|\lesssim |{\bU}_{m+1} - {\bU}_{m}| \cdot |  d{\bU}_{m} |.
\end{equation}
Multiplying ${\bU}_{m+1}- {\bU}_{m}$ on \eqref{yz14} and integrating it on $\times \mathbb{R}^3$, we have
\begin{equation}\label{yz15}
	\frac{d}{dt}	\|{\bU}_{m+1}-{\bU}_{m} \|^2_{L^2_x} \lesssim   \| (d {\bU}_{m+1}, d{\bU}_{m}) \|_{L^\infty_x}\| {\bU}_{m+1}-{\bU}_{m} \|^2_{L^2_x}.
\end{equation}
Integrating \eqref{yz15} on $[0,T^*_{m+1}]$ and using Gronwall's inequality, \eqref{yz4}, \eqref{yz12}, and \eqref{yz13}, it yields
\begin{equation*}
	\begin{split}
		\| {\bU}_{m+1}-{\bU}_{m} \|_{L^\infty_{[0,T^*_{m+1} ]}L^2_x} \lesssim &  \| {\bU}_{0(m+1)}-{\bU}_{0m} \|_{L^2_x} \leq 2^{-(\sstar- \delta_1)m}E(0).
	\end{split}
\end{equation*}
Then we get
\begin{equation*}%\label{piw}
	\begin{split}
		\| {\bv}_{m+1}-{\bv}_{m}, {\rho}_{m+1}-{\rho}_{m} \|_{L^\infty_{[0,T^*_{m+1}]}L^2_x} \leq 2^{-(\sstar- \delta_1)m}E(0).
	\end{split}
\end{equation*}
So we have proved \eqref{yu0}. To prove \eqref{yu1}, let us consider the transport equation of ${\bw}_{m+1} - {\bw}_m$, that is,
\begin{equation}\label{AMM7}
	\begin{split}
		&\partial_t ({\bw}_{m+1}- {\bw}_{m}) + ({\bv}_{m+1} \cdot \nabla) ({\bw}_{m+1}- {\bw}_{m})
		\\
		=& -({\bv}_{m+1}-{\bv}_{m}) \cdot \nabla {\bw}_{m}+ ({\bw}_{m+1}- {\bw}_{m})\cdot \nabla {\bv}_{m+1}+ {\bw}_{m} \cdot \nabla ({\bv}_{m+1}-{\bv}_{m}).
	\end{split}
\end{equation}
Multiplying with ${\bw}_{m+1}- {\bw}_{m}$ on \eqref{AMM7} and integrating it on $[0,t]\times \mathbb{R}^3$, we derive
\begin{equation}\label{AMM8}
	\begin{split}
		\|{\bw}_{m+1}- {\bw}_{m} \|^2_{L^2_x} \leq  & \| {\bw}_{0(m+1)}- {\bw}_{0m} \|^2_{L^2_x}+ C\int^t_{0}  \| \nabla {\bv}_{m+1}\|_{L^\infty_x}\| {\bw}_{m+1}- {\bw}_{m} \|^2_{L^2_x}d\tau
		\\
		& + C\int^t_{0} \| {\bv}_{m+1}-{\bv}_{m} \|_{L^3_x}\| {\bw}_{m+1}-{\bw}_{m} \|_{L^2_x}\|\nabla {\bw}_{m} \|_{L^6_x}d\tau
		\\
		&+ C| \int^t_{0} \int_{\mathbb{R}^3}({\bw}_{m} \cdot \nabla) ({\bv}_{m+1}-{\bv}_{m}) \cdot ({\bw}_{m+1}- {\bw}_{m}) dx d\tau |.
	\end{split}
\end{equation}
By calculating
\begin{equation*}
	\begin{split}
		& \int^t_0 \int_{\mathbb{R}^3} ({\bw}_{m} \cdot \nabla) ({\bv}_{m+1}-{\bv}_{m}) \cdot ({\bw}_{m+1}- {\bw}_{m}) dx d\tau
		\\
		= & \int^t_0 \int_{\mathbb{R}^3} {w}^a_{m} \delta_{ai}\partial^i  ({v}^j_{m+1}-{v}^j_{m}) \delta_{bj}({w}^b_{m+1}- {w}^b_{m}) dx d\tau
		\\
		= & \int^t_0 \int_{\mathbb{R}^3} {w}^a_{m} \delta_{ai}   \{ \partial^i({v}^j_{m+1}-{v}^j_{m})- \partial^j({v}^i_{m+1}-{v}^i_{m}) \} \delta_{bj}({w}^b_{m+1}- {w}^b_{m}) dx d\tau
		\\
		& + \int^t_0 \int_{\mathbb{R}^3} {w}^a_{m} \delta_{ai}  \partial^j({v}^i_{m+1}-{v}^i_{m}) \delta_{bj}({w}^b_{m+1}- {w}^b_{m}) dx d\tau
		\\
		= & \int^t_0 \int_{\mathbb{R}^3} {w}^a_{m} \delta_{ai}  \epsilon^{ij}_{\ \ k} (\mathrm{e}^{{\rho}_{m+1}} {w}^k_{m+1} - \mathrm{e}^{{\rho}_{m}} {w}^k_{m})  \delta_{bj}({w}^b_{m+1}- {w}^b_{m}) dx d\tau
		\\
		& + \int^t_0 \int_{\mathbb{R}^3} {w}^a_{m} \delta_{ai}  \partial^j({v}^i_{m+1}-{v}^i_{m}) \delta_{bj}({w}^b_{m+1}- {w}^b_{m}) dx d\tau,
	\end{split}
\end{equation*}
and
\begin{equation*}
	\begin{split}
		& \int^t_0 \int_{\mathbb{R}^3} {w}^a_{m} \delta_{ai}  \partial^j({v}^i_{m+1}-{v}^i_{m}) \delta_{bj}({w}^b_{m+1}- {w}^b_{m}) dx d\tau
		\\
		=& - \int^t_0 \int_{\mathbb{R}^3} \partial^j {w}^a_{m} \delta_{ai}  ({v}^i_{m+1}-{v}^i_{m}) \delta_{bj}({w}^b_{m+1}- {w}^b_{m}) dx d\tau
		\\
		& - \int^t_0 \int_{\mathbb{R}^3} {w}^a_{m} \delta_{ai}  ({v}^i_{m+1}-{v}^i_{m}) (\mathrm{div} {\bw}_{m+1}- \mathrm{div} {\bw}_{m}) dx d\tau
		\\
		=& - \int^t_0 \int_{\mathbb{R}^3} \partial^j {w}^a_{m} \delta_{ai}  ({v}^i_{m+1}-{v}^i_{m}) \delta_{bj}({w}^b_{m+1}- {w}^b_{m}) dx d\tau
		\\
		& - \int^t_0 \int_{\mathbb{R}^3}  {w}^a_{m} \delta_{ai}  ({v}^i_{m+1}-{v}^i_{m}) ({\bw}_{m+1} \cdot \nabla {\rho}_{m+1} - {\bw}_{m} \cdot \nabla{\rho}_{m}) dx d\tau,
	\end{split}
\end{equation*}
then we get
\begin{equation}\label{AMM9}
	\begin{split}
		& \int^t_0 \int_{\mathbb{R}^3} ({\bw}_{m} \cdot \nabla) ({\bv}_{m+1}-{\bv}_{m}) \cdot ({\bw}_{m+1}- {\bw}_{m}) dx d\tau
		\\
		= & \int^t_0 \int_{\mathbb{R}^3} {w}^a_{m} \delta_{ai}  \epsilon^{ij}_{\ \ k} (\mathrm{e}^{{\rho}_{m+1}} {w}^k_{m+1} - \mathrm{e}^{{\rho}_{m}} {w}^k_{m})  \delta_{bj}({w}^b_{m+1}- {w}^b_{m}) dx d\tau
		\\
		& - \int^t_0 \int_{\mathbb{R}^3} \partial^j{w}^a_{m} \delta_{ai}  ({v}^i_{m+1}-{v}^i_{m}) \delta_{bj}({w}^b_{m+1}- {w}^b_{m}) dx d\tau
		\\
		& - \int^t_0 \int_{\mathbb{R}^3} {w}^a_{m} \delta_{ai}  ({v}^i_{m+1}-{v}^i_{m}) ({\bw}_{m+1}- {\bw}_{m}) \cdot \nabla {\rho}_{m+1}  dx d\tau.
		\\
		& + \int^t_0 \int_{\mathbb{R}^3} {w}^a_{m} \delta_{ai}  ({v}^i_{m+1}-{v}^i_{m})  {\bw}_{m} \cdot \nabla ({\rho}_{m+1} -  {\rho}_{m}) dx d\tau.
	\end{split}
\end{equation}
Inserting \eqref{AMM9} in \eqref{AMM8},  \eqref{AMM8} yields
\begin{equation}\label{M04}
	\begin{split}
		\| {\bw}_{m+1}- {\bw}_{m} \|^2_{L^2_x} \leq  & \| {\bw}_{0(m+1)} - {\bw}_{0m} \|^2_{L^2_x}
		\\
		& + C\int^t_0  \| \partial {\bv}_{m+1}, \partial {\rho}_{m+1}\|_{L^\infty_x}\| {\bw}_{m+1}- {\bw}_{m} \|^2_{L^2_x}d\tau
		\\
		& + C\int^t_0 \| {\bv}_{m+1}-{\bv}_{m} \|_{H^{\frac12}_x}\| {\bw}_{m+1}- {\bw}_{m} \|_{L^2_x}\|\partial {\bw}_{m} \|_{H^1_x}d\tau
		\\
		&+ C \int^t_0 \| {\rho}_{m+1}-{\rho}_{m} \|_{H^{\frac12}_x}\| {\bw}_{m+1}- {\bw}_{m} \|_{L^2_x}\|\partial {\bw}_{m} \|_{H^1_x}d\tau
		\\
		&+ C \int^t_0 \| {\rho}_{m+1}-{\rho}_{m} \|_{L^2_x}\| \partial ({\rho}_{m+1}- {\rho}_{m}) \|_{L^2_x}\|{\bw}_{m} \|^2_{L^\infty_x}d\tau.
	\end{split}
\end{equation}
Seeing \eqref{M04} and using \eqref{yz4}, and \eqref{yu0}, we can see that
\begin{equation*}
	\|{\bw}_{m+1}-{\bw}_{m}\|_{L^\infty_{ [0, T^*_{m+1}] } L^2_x} \leq 2^{- (\sstar-\frac12- \delta_1)m}(1+E^2(0)).
\end{equation*}
So we have finished the proof of \eqref{yu1}. On the other hand, using Strichartz estimates \eqref{ru04}, then
\begin{equation}\label{see80}
	\begin{split}
		& \|\partial({\rho}_{m+1}-{\rho}_{m}) \|_{L^2_{[0,T^*_{m+1}]} C^{\delta_{1}}_x}+\|\partial ({\bv}_{+(m+1)}-{\bv}_{+m}) \|_{L^2_{[0,T^*_{m+1}]} C^{\delta_{1}}_x}
		\\
		\leq  &  C\| {\rho}_{m+1}-{\rho}_{m}, {\bv}_{m+1}-{\bv}_{m} \|_{L^\infty_{[0,T^*_{m+1}]} H^{2+2\delta_{1}}_x}+\| {\bw}_{m+1}-{\bw}_{m} \|_{L^2_{[0,T^*_{m+1}]} H^{\frac32+2\delta_{1}}_x}.
	\end{split}
\end{equation}
Noting $s=2+10\delta_{1}$ and using \eqref{yu0}-\eqref{yu1}, we can bound \eqref{see80} by
\begin{equation}\label{see99}
	\begin{split}
		& \|\partial({\rho}_{m+1}-{\rho}_{m}) \|_{L^2_{[0,T^*_{m+1}]} C^{\delta_{1}}_x}+\|\partial ({\bv}_{+(m+1)}-{\bv}_{+m}) \|_{L^2_{[0,T^*_{m+1}]} C^{\delta_{1}}_x}
		\leq   2^{-7\delta_1m} E(0).
	\end{split}
\end{equation}
By \eqref{yu1} and Sobolev imbeddings, we shall get
\begin{equation}\label{siq}
	\begin{split}
		\|\partial ({\bv}_{-(m+1)}-{\bv}_{-m}) \|_{L^2_{[0,T^*_{m+1}]} L^\infty_x}
		\leq & \| \partial ({\bv}_{-(m+1)}-{\bv}_{-m}) \|_{L^2_{[0,T^*_{m+1}]} L^\infty_x}
		\\
		\leq  & C \|{\bw}_{m+1}- {\bw}_{m} \|_{L^\infty_{[0,T^*_{m+1}]} H^{\frac32+\delta_{1}}_x}
		\\
		\leq & C (1+E^2(0)) 2^{-8\delta_{1} m} .
	\end{split}
\end{equation}
If we add \eqref{see99} and \eqref{siq}, then we have
\begin{equation}\label{siw}
	\begin{split}
		& \|\partial( {\rho}_{m+1}-{\rho}_{m} ) \|_{L^2_{[0,T^*_{m+1}]} C^{\delta_{1}}_x}+\| \partial ({\bv}_{m+1}-{\bv}_{m}) \|_{L^2_{[0,T^*_{m+1}]} C^{\delta_{1}}_x}\leq  C (1+E^2(0)) 2^{-6\delta_{1} m}.
	\end{split}
\end{equation}
By using \eqref{fc0s} and \eqref{siw}, we shall obtain
\begin{equation}\label{sie}
	\begin{split}
		& \|\partial_t({\rho}_{m+1}-{\rho}_{m})\|_{L^2_{[0,T^*_{m+1}]} C^{\delta_{1}}_x}+\| \partial_t({\bv}_{m+1}-{\bv}_{m}) \|_{L^2_{[0,T^*_{m+1}]} C^{\delta_{1}}_x}
		\\
		\leq  & \|\partial({\rho}_{m+1}-{\rho}_{m}) , \partial ({\bv}_{m+1}-{\bv}_{m}) \|_{L^2_{[0,T^*_{m+1}]} C^{\delta_{1}}_x} \cdot (1+ \|\bv_m, \rho_m\|_{L^\infty_{[0,T^*_{m+1}]} H^s_x} )
			\\
			\leq & C (1+E^3(0)) 2^{-6\delta_{1} m}.
	\end{split}
\end{equation}
Due to \eqref{siw} and \eqref{sie}, for $k<j$, it yields
\begin{equation}\label{Sia}
		 \|P_k d({\rho}_{m+1}-{\rho}_{m}), P_k d ({\bv}_{m+1}-{\bv}_{m}) \|_{L^2_{[0,T^*_{m+1}]} L^\infty_x}
		\leq   2^{-\delta_1k}\cdot C (1+E^3(0)) 2^{-6\delta_1m},
\end{equation}
and
\begin{equation}\label{kz4}
	\|P_k d({\rho}_{m+1}-{\rho}_{m}), P_k d ({\bv}_{m+1}-{\bv}_{m}) \|_{L^1_{[0,T^*_{m+1}]} L^\infty_x}
	\leq   2^{-\delta_1k} \cdot C (1+E^3(0)) 2^{-6\delta_1m}.
\end{equation}

\subsubsection{Uniform energy and Strichartz estimates on a regular time-interval $[0,T^*_{N_0}]$.}\label{finalk}
Recall $T_{N_0}^*=[E(0)]^{-1}2^{-\delta_{1}N_0}$, which is stated in \eqref{DTJ}. Recall  \eqref{yz10} and \eqref{Sia}. Therefore, when $k\geq j$ and $j\geq N_0$, using \eqref{yz10} and H\"older's inequality, we have
\begin{equation}\label{kf01}
	\begin{split}
		\| P_{k} d \bv_j, P_{k} d \rho_j \|_{L^1_{[0,T^*_j]}L^\infty_x} \leq & (T^*_j)^{\frac12}\| P_{k} d \bv_j, P_{k} d \rho_j \|_{L^2_{[0,T^*_j]}L^\infty_x}
		\\
		\leq & (T^*_{N_0})^{\frac12} \cdot C  (1+E^3(0))2^{-\delta_{1}k} 2^{-7\delta_1j}.
	\end{split}
\end{equation}
When $k< j$ and $m\geq N_0$, using \eqref{Sia} and H\"older's inequality, we also see
\begin{equation}\label{kf02}
	\|P_k d({\rho}_{m+1}-{\rho}_{m}), P_k d({\bv}_{m+1}-{\bv}_{m}) \|_{L^1_{[0,T^*_{m+1}]} L^\infty_x}
	\leq    (T^*_{N_0})^{\frac12} \cdot C  (1+E^3(0))2^{-\delta_{1}k} 2^{-6\delta_1m}.
\end{equation}
On the other side, when $k\geq j$ and $j< N_0$, using \eqref{yz10}, we have
\begin{equation}\label{kf03}
	\begin{split}
		\| P_{k} d \bv_j, P_{k} d \rho_j \|_{L^1_{[0,T^*_{N_0}]}L^\infty_x} \leq & (T^*_{N_0})^{\frac12} \| P_{k} d \bv_j, P_{k} d \rho_j \|_{L^2_{[0,T^*_{N_0}]}L^\infty_x}
		\\
		\leq & (T^*_{N_0})^{\frac12} \| P_{k} d \bv_j, P_{k} d \rho_j \|_{L^2_{[0,T^*_j]}L^\infty_x}
		\\
		\leq & (T^*_{N_0})^{\frac12} \cdot C  (1+E^3(0))2^{-\delta_{1}k} 2^{-7\delta_1j}.
	\end{split}
\end{equation}
When $k< j$ and $m< N_0$, using \eqref{Sia} and H\"older's inequality, we also see
\begin{equation}\label{kf04}
	\begin{split}
	& \|P_k d({\rho}_{m+1}-{\rho}_{m}), P_k d({\bv}_{m+1}-{\bv}_{m}) \|_{L^1_{[0,T^*_{N_0}]} L^\infty_x}
	\\
	\leq & (T^*_{N_0})^{\frac12} \|P_k d({\rho}_{m+1}-{\rho}_{m}), P_k d({\bv}_{m+1}-{\bv}_{m}) \|_{L^2_{[0,T^*_{N_0}]} L^\infty_x}
	\\
	\leq & (T^*_{N_0})^{\frac12} \|P_k d({\rho}_{m+1}-{\rho}_{m}), P_k d({\bv}_{m+1}-{\bv}_{m}) \|_{L^2_{[0,T^*_{m+1}]} L^\infty_x}
	\\
	\leq    & (T^*_{N_0})^{\frac12} \cdot C  (1+E^3(0))2^{-\delta_{1}k} 2^{-6\delta_1m}.
	\end{split}
\end{equation}
Due to a different time-interval for the sequency  $(\bv_j,\rho_j,\bw_j)$, we will discuss the solutions $(\bv_j,\rho_j,\bw_j)$ if  $j \leq N_0$ or  $j \geq N_0+1$ as follows.

\textit{Case 1: $j \leq N_0$.} In this case, $(\bv_j, \rho_j, \bw_j)$ exists on $[0,T_j^*]$, and $[0,T^*_{N_0}]\subseteq [0,T^*_{j}]$. So we don't need to extend solutions $(\bv_j, \rho_j, \bw_j)$ if  $j \leq N_0$. Using \eqref{yz4} and H\"older's inequality, we get
\begin{equation*}
	\| d\bv_j, d\rho_j\|_{L^1_{[0,T^*_{N_0}]}L^\infty_x} \leq (T^*_{N_0})^{\frac12} \| d\bv_j, d\rho_j\|_{L^2_{[0,T^*_{N_0}]}L^\infty_x} \leq C(T^*_{N_0})^{\frac12} (1+E(0)).
\end{equation*}
By \eqref{pp8}, this yields
\begin{equation}\label{ky0}
	\| d\bv_j, d\rho_j\|_{L^1_{[0,T^*_{N_0}]}L^\infty_x} \leq 2.
\end{equation}
By the Newton-Leibniz formula and \eqref{pu00}, it follows that
\begin{equation}\label{ky1}
	\|\bv_j, \rho_j\|_{L^\infty_{[0,T^*_{N_0}] \times \mathbb{R}^3}}\leq |\bv_{0j}, \rho_{0j}|+ \| d\bv_j, d\rho_j\|_{L^1_{[0,T^*_{N_0}]}L^\infty_x} \leq 2+C_0.
\end{equation}
Using Theorem \ref{ve}, we get
\begin{equation}\label{ky2}
	E(T^*_{N_0}) \leq C(E(0)+E^3(0))\exp\{  2(1+E^3(0)) \textrm{e}^2 \} .
\end{equation}
\textit{Case 2: $j \geq N_0+1$.} In this case, we have to extend the time interval $I_1=[0,T^*_j]$ to $[0,T^*_{N_0}]$. So we will use \eqref{yz10}, \eqref{Sia}, and Theorem \ref{TT2} to calculate the energy at time $T_j^*$. Then we start at $T_j^*$ and get a new time-interval.

To be simple, we set
\begin{equation}\label{ti1}
	I_1=[0,T^*_j]=[t_0,t_1], \quad \quad |I_1|=[E(0)]^{-1}2^{-\delta_{1} j}.
\end{equation}
By frequency decomposition, we get
\begin{equation}\label{kz03}
	\begin{split}
		d \bv_j= & \textstyle{\sum}^{\infty}_{k=j} d \bv_j+ \textstyle{\sum}^{j-1}_{k=1}P_k d \bv_j
		\\
		=& P_{\geq j}d \bv_j+ \textstyle{\sum}^{j-1}_{k=1} \textstyle{\sum}_{m=k}^{j-1} P_k  (d\bv_{m+1}-d\bv_m)+ \textstyle{\sum}^{j-1}_{k=1}P_k d \bv_k .
	\end{split}
\end{equation}
Similarly, we also have
\begin{equation}\label{kz04}
	\begin{split}
		d \rho_j= & P_{\geq j}d \rho_j+ \textstyle{\sum}^{j-1}_{k=1} \textstyle{\sum}_{m=k}^{j-1} P_k  (d \rho_{m+1}-d \rho_m)+ \textstyle{\sum}^{j-1}_{k=1}P_k d \rho_k.
	\end{split}
\end{equation}
When $k\geq j$, using \eqref{kf01} and \eqref{kf03}, we have\footnote{In the case $j\geq N_0$ we use \eqref{kf01}. In the case $j < N_0$, we take $T^*_j = T^*_{N_0}$ and use \eqref{kf03} to give a bound on $[0,T^*_{N_0}]$.}
\begin{equation}\label{kz01}
	\begin{split}
		\| P_{k} d \bv_j, P_{k} d \rho_j \|_{L^1_{[0, T^*_j]}L^\infty_x} \leq   &   (T^*_{N_0})^{\frac12} \cdot C(1+E^3(0)) 2^{-\delta_{1}k} 2^{-7\delta_{1}j}.
	\end{split}
\end{equation}
When $k< j$, using \eqref{kf02} and \eqref{kf04}, we also see\footnote{In the case $m\geq N_0$ we use \eqref{kf02}. In the case $m < N_0$, we take $T^*_{m+1} = T^*_{N_0}$ and use \eqref{kf04} to give a bound on $[0,T^*_{N_0}]$.}
\begin{equation}\label{kz02}
  \|P_k d({\rho}_{m+1}-{\rho}_{m}), P_k d({\bv}_{m+1}-{\bv}_{m}) \|_{L^1_{[0,T^*_{m+1}]} L^\infty_x}
\leq    (T^*_{N_0})^{\frac12} \cdot C  (1+E^3(0))2^{-\delta_{1}k} 2^{-6\delta_1m}.
\end{equation}
%Considering \eqref{kz03}-\eqref{kz04}, the terms in \eqref{kz04} is much different.
We will use  and \eqref{kz01}-\eqref{kz02} to give a precise analysis of \eqref{kz03}-\eqref{kz04} and get some new time-intervals, and then we try to extend $\rho_j$ from $[0,T^*_j]$ to $[0,T^*_{N_0}]$. Our strategy is as follows. In step 1, we extend it from $[0,T^*_j]$ to $[0,T^*_{j-1}]$. In step 2, we use induction methods to conclude these estimates and also extend it to $[0,T_{N_0}^*]$.

\textbf{Step 1: Extending $[0,T^*_j]$ to $[0,T^*_{j-1}] \ (j \geq N_0+1)$.} To start, referring \eqref{DTJ}, so we need to calculate $E(T^*_j)$ for obtaining a length of a new time-interval. Then we shall calculate $\|d \bv_j, d \rho_j\|_{L^1_{[0,T^*_j]}L^\infty_x}$. Using \eqref{kz03} and \eqref{kz04}, we derive that
\begin{equation*}
	\begin{split}
 \|d \bv_j, d \rho_j\|_{L^1_{[0,T^*_j]} L^\infty_x }
\leq & 	\|P_{\geq j}d\bv_j, P_{\geq j}d \rho_j\|_{L^1_{[0,T^*_j]} L^\infty_x }  + \textstyle{\sum}^{j-1}_{k=1} \|P_k d\bv_k, P_k d\rho_k\|_{L^1_{[0,T^*_j]}L^\infty_x}\\
& + \textstyle{\sum}^{j-1}_{k=1} \textstyle{\sum}_{m=k}^{j-1}  \|P_k  (d\bv_{m+1}-d\bv_m), P_k  (d\rho_{m+1}-d\rho_m)\|_{L^1_{[0,T^*_j]}L^\infty_x}
\\
\leq & 	\|P_{\geq j}d\bv_j, P_{\geq j}d \rho_j\|_{L^1_{[0,T^*_j]}L^\infty_x} + \textstyle{\sum}^{j-1}_{k=1} \|P_k d\bv_k, P_k d\rho_k\|_{L^1_{[0,T^*_k]}L^\infty_x}\\
& + \textstyle{\sum}^{j-1}_{k=1} \textstyle{\sum}_{m=k}^{j-1}\| P_k  (d\bv_{m+1}-d\bv_m), P_k  (d\rho_{m+1}-d\rho_m)\|_{L^1_{[0,T^*_{m+1}]}L^\infty_x}
\\
\leq & C(1+E^3(0))(T^*_{N_0})^{\frac12}  \textstyle{\sum}_{k=j}^{\infty} 2^{-\delta_{1} k} 2^{-7\delta_{1} j}
\\
& + C(1+E^3(0))(T^*_{N_0})^{\frac12} \textstyle{\sum}^{j-1}_{k=1} 2^{-\delta_{1} k} 2^{-7\delta_{1} k}
\\
& +  C(1+E^3(0))(T^*_{N_0})^{\frac12} \textstyle{\sum}^{j-1}_{k=1} \textstyle{\sum}_{m=k}^{j-1} 2^{-\delta_{1} k} 2^{-6\delta_{1} m}
\\
\leq & C(1+E^3(0))(T^*_{N_0})^{\frac12} [\frac{1}{3}(1-2^{-\delta_{1}})]^{-2}.
	\end{split}
\end{equation*}
So we get
\begin{equation}
	\begin{split}\label{kz05}
		 \|d \bv_j, d \rho_j\|_{L^1_{[0,T^*_j]}L^\infty_x}
		\leq  C(1+E^3(0))(T^*_{N_0})^{\frac12} [\frac{1}{3}(1-2^{-\delta_{1}})]^{-2}.
	\end{split}
\end{equation}
By using \eqref{kz05} and \eqref{pu00}, we get
\begin{equation}
	\begin{split}\label{kp05}
	\| \bv_j,  \rho_j\|_{L^\infty_{[0,T^*_j]}L^\infty_x} \leq \| \bv_{0j},  \rho_{0j}\|+	\|d \bv_j, d \rho_j\|_{L^1_{[0,T^*_j]}L^\infty_x}
		\leq  C_0+2.
	\end{split}
\end{equation}
By \eqref{kz05} and \eqref{pp8} and Theorem \ref{TT2}, we have
\begin{equation}\label{kz06}
	\begin{split}
	E(T^*_{j}) \leq C(E(0)+E^3(0))\exp\{  2\textrm{e}^2(1+E^3(0))  \} =C_*.
	\end{split}
\end{equation}
Above, $C_*$ is stated in \eqref{Cstar}. Starting at the time $T^*_j$, seeing \eqref{DTJ} and \eqref{kz06}, we shall get an extending time-interval with a length of $(C_*)^{-1}2^{-\delta_{1} j}$. But, if $T^*_j + (C_*)^{-1}2^{-\delta_{1} j} \geq T^*_{j-1}$, so we have finished this step. Else, we need to extend it again.

$\mathit{Case 1: T^*_j + (C_*)^{-1}2^{-\delta_{1} j} \geq T^*_{j-1} }$, then we get a new interval
\begin{equation}\label{deI2}
	I_2=[T^*_j, T^*_{j-1}], \quad |I_2|= (2^{\delta_{1}}-1) E(0)^{-1}2^{-\delta_{1} j}.
\end{equation}
Moreover, referring \eqref{kz01} and \eqref{kz02}, we can deduce
\begin{equation}\label{kz07}
	\| P_{k} d \bv_j, P_{k} d \rho_j \|_{L^1_{[T^*_j, T^*_{j-1}]}L^\infty_x}
	\leq  (T^*_{N_0})^{\frac12}  \cdot C  (1+C^3_*) 2^{-\delta_{1}k} 2^{-7\delta_{1}j}, \quad k \geq j,
\end{equation}
and $k<j$,
\begin{equation}\label{kz08}
	\|P_k (d{\rho}_{j}-d{\rho}_{j-1}), P_k (d{\bv}_{j}-d{\bv}_{j-1}) \|_{L^1_{[T^*_j, T^*_{j-1}]} L^\infty_x}
	\leq   (T^*_{N_0})^{\frac12}  \cdot C  (1+C^3_*) 2^{-\delta_{1}k} 2^{-6\delta_1(j-1)}.
\end{equation}
Using \eqref{pp8} and $j \geq N_0+1$, \eqref{kz07} and \eqref{kz08} yields
\begin{equation}\label{kz09}
	\| P_{k} d \bv_j, P_{k} d \rho_j \|_{L^1_{[T^*_j, T^*_{j-1}]}L^\infty_x}
	\leq  (T^*_{N_0})^{\frac12}  \cdot C  (1+E^3(0)) 2^{-\delta_{1}k} 2^{-6\delta_{1}j}, \quad k \geq j,
\end{equation}
and $k<j$,
\begin{equation}\label{kz10}
	\|P_k (d{\rho}_{j}-d{\rho}_{j-1}), P_k (d{\bv}_{j}-d{\bv}_{j-1}) \|_{L^1_{[T^*_j, T^*_{j-1}]} L^\infty_x}
	\leq    (T^*_{N_0})^{\frac12}  \cdot C   (1+E^3(0)) 2^{-\delta_{1}k} 2^{-5\delta_1(j-1)}.
\end{equation}
Therefore, we could get the following estimate
\begin{equation}\label{kz11}
	\begin{split}
		& \|d \bv_j, d\rho_j\|_{L^1_{[0,T^*_{j-1}]} L^\infty_x}
		\\
		\leq & 	\|P_{\geq j}d\bv_j, P_{\geq j}d \rho_j\|_{L^1_{[0,T^*_{j-1}]} L^\infty_x}  + \textstyle{\sum}^{j-1}_{k=1} \|P_k d\bv_k, P_k d\rho_k\|_{L^1_{[0,T^*_{j-1}]} L^\infty_x}
		\\
		& + \textstyle{\sum}^{j-1}_{k=1} \|P_k  (d\bv_{j}-d\bv_{j-1}), P_k  (d\rho_{j}-d\rho_{j-1})\|_{L^1_{[0,T^*_{j-1}]} L^\infty_x}
		\\
		& + \textstyle{\sum}^{j-2}_{k=1} \textstyle{\sum}_{m=k}^{j-2}  \|P_k  (d\bv_{m+1}-d\bv_m), P_k  (d\rho_{m+1}-d\rho_m)\|_{L^1_{[0,T^*_{j-1}]} L^\infty_x}.
	\end{split}
\end{equation}
Due to \eqref{kz01} and \eqref{kz09}, it yields
\begin{equation}\label{kz12}
	\begin{split}
	\|P_{\geq j}d\bv_j, P_{\geq j}d\rho_j\|_{L^1_{[0,T^*_{j-1}]} L^\infty_x}
	\leq & C   (1+E^3(0)) (T^*_{N_0})^{\frac12}  \textstyle{\sum}^{\infty}_{k=j}2^{-\delta_{1}k} (2^{-7\delta_1 j}+ 2^{-6\delta_{1} j})
	\\
	\leq & C   (1+E^3(0)) (T^*_{N_0})^{\frac12}  \textstyle{\sum}^{\infty}_{k=j} 2^{-\delta_{1}k} 2^{-6\delta_{1} j}\times 2.
	\end{split}
\end{equation}
Due to \eqref{kz02} and \eqref{kz10}, it yields
\begin{equation}\label{kz13}
	\begin{split}
		& \textstyle{\sum}^{j-1}_{k=1} \|P_k  (d\bv_{j}-d\bv_{j-1}), P_k  (d\rho_{j}-d\rho_{j-1})\|_{L^1_{[0,T^*_{j-1}]} L^\infty_x}
		\\
		\leq & C   (1+E^3(0)) (T^*_{N_0})^{\frac12}  \textstyle{\sum}^{j-1}_{k=1} 2^{-\delta_{1}k} (2^{-6\delta_1 (j-1)}+ 2^{-5\delta_{1} (j-1)})
		\\
		\leq & C   (1+E^3(0)) (T^*_{N_0})^{\frac12}  \textstyle{\sum}^{j-1}_{k=1} 2^{-\delta_{1}k} 2^{-5\delta_{1} (j-1)} \times 2.
	\end{split}
\end{equation}
Inserting \eqref{kz12}-\eqref{kz13} into \eqref{kz11}, we derive that
\begin{equation}\label{kz14}
	\begin{split}
		 \|d \bv_j, d \rho_j\|_{L^1_{[0,T^*_{j-1}]} L^\infty_x}
		\leq & 	C(1+E^3(0)) (T^*_{N_0})^{\frac12}   \textstyle{\sum}_{k=j}^{\infty} 2^{-\delta_{1} k} 2^{-6\delta_{1} j}\times 2
		\\
		& + C   (1+E^3(0)) (T^*_{N_0})^{\frac12}  \textstyle{\sum}^{j-1}_{k=1} 2^{-\delta_{1}k} 2^{-5\delta_{1} (j-1)} \times 2
		\\
		& + C(1+E^3(0)) (T^*_{N_0})^{\frac12}  \textstyle{\sum}^{j-1}_{k=1} 2^{-\delta_{1} k} 2^{-7\delta_{1} k}
		\\
		& +  C(1+E^3(0)) (T^*_{N_0})^{\frac12}  \textstyle{\sum}^{j-2}_{k=1} \textstyle{\sum}_{m=k}^{j-2} 2^{-\delta_{1} k} 2^{-6\delta_{1} m}
		\\
		\leq & 	C(1+E^3(0))(T^*_{N_0})^{\frac12}  [\frac13(1-2^{-\delta_{1}})]^{-2}.
	\end{split}
\end{equation}
By \eqref{kz14}, \eqref{pp8} and Theorem \ref{TT2}, we also prove
\begin{equation}\label{kz15}
	\begin{split}
		&|\bv_j, \rho_j| \leq 2+C_0,
		\\
		& E(T^*_{j-1}) \leq C_*.
	\end{split}
\end{equation}
$\mathit{Case 2: T^*_j + (C_*)^{-1}2^{-\delta_{1} j} < T^*_{j-1} }$. In this situation, we will record
\begin{equation*}
	I_2=[T^*_j, t_2], \quad |I_2| = (C_*)^{-1}2^{-\delta_{1} j}.
\end{equation*}
Referring \eqref{kz01} and \eqref{kz02}, we also deduce that
\begin{equation}\label{kz16}
	\| P_{k} d \bv_j, P_{k} d \rho_j \|_{L^1_{I_2}L^\infty_x}
	\leq  (T^*_{N_0})^{\frac12} \cdot C  (1+C^3_*) 2^{-\delta_{1}k} 2^{-7\delta_{1}j}, \quad k \geq j,
\end{equation}
\begin{equation}\label{kz17}
	\|P_k (d{\rho}_{j}-d{\rho}_{j-1}), P_k  (d{\bv}_{j}-d{\bv}_{j-1}) \|_{L^1_{I_2} L^\infty_x}
	\leq    (T^*_{N_0})^{\frac12} \cdot C  (1+C^3_*) 2^{-\delta_{1}k} 2^{-6\delta_1(j-1)}, \quad k<j.
\end{equation}
Similarly, we can also get
\begin{equation*}
	\begin{split}
		 \|d\bv_j, d\rho_j\|_{L^1_{ I_1 \cup I_2 } L^\infty_x }
		\leq & 	\|P_{\geq j}d\bv_j, P_{\geq j}d \rho_j\|_{L^1_{I_1 \cup I_2} L^\infty_x }   + \textstyle{\sum}^{j-1}_{k=1} \|P_k d\bv_k, P_k d\rho_k\|_{L^1_{[0,T^*_k]} L^\infty_x }
		\\
		 & + \textstyle{\sum}^{j-1}_{k=1} \|P_k  (d\bv_{j}-d\bv_{j-1}), P_k  (d\rho_{j}-d\rho_{j-1})\|_{L^1_{I_1 \cup I_2 } L^\infty_x }
		\\
		& + \textstyle{\sum}^{j-2}_{k=1} \textstyle{\sum}_{m=k}^{j-2}  \|P_k  (d\bv_{m+1}-d\bv_m), P_k  (d\rho_{m+1}-d\rho_m)\|_{L^1_{I_1 \cup I_2} L^\infty_x } .
\end{split}
\end{equation*}
Noting $ I_1 \cup I_2 \subseteq T^*_k$ if $k \leq j-1$, then it follows that
\begin{equation}\label{kz11A}
	\begin{split}
		 & \|d \bv_j, d \rho_j\|_{L^1_{ I_1 \cup I_2 } L^\infty_x }
		\\
			\leq & 	\|P_{\geq j}d\bv_j, P_{\geq j}d \rho_j\|_{L^1_{I_1 \cup I_2} L^\infty_x }   + \textstyle{\sum}^{j-1}_{k=1} \|P_k d\bv_k, P_k d\rho_k\|_{L^1_{I_1 \cup I_2 } L^\infty_x }
		\\
		 & + \textstyle{\sum}^{j-1}_{k=1} \|P_k  (d \bv_{j}-d \bv_{j-1}), P_k  (d\rho_{j}-d\rho_{j-1})\|_{L^1_{I_1 \cup I_2 } L^\infty_x }
		\\
		& + \textstyle{\sum}^{j-2}_{k=1} \textstyle{\sum}_{m=k}^{j-2}  \|P_k  (d\bv_{m+1}-d\bv_m), P_k  (d\rho_{m+1}-d\rho_m)\|_{L^1_{[0,T^*_{m+1}]} L^\infty_x }
\end{split}
\end{equation}
Inserting \eqref{kz01}, \eqref{kz02} and \eqref{kz16} and \eqref{kz17} to \eqref{kz11A}, we have
\begin{equation}\label{kz11a}
	\begin{split}		
		 \|d \bv_j, d \rho_j\|_{L^1_{ I_1 \cup I_2 } L^\infty_x }
		\leq & 	C(1+E^3(0)) (T^*_{N_0})^{\frac12} \textstyle{\sum}_{k=j}^{\infty} 2^{-\delta_{1} k} 2^{-6\delta_{1} j}\times 2
		\\
		& + C   (1+E^3(0)) (T^*_{N_0})^{\frac12} \textstyle{\sum}^{j-1}_{k=1} 2^{-\delta_{1}k} 2^{-5\delta_{1} (j-1)} \times 2
		\\
		& + C(1+E^3(0)) (T^*_{N_0})^{\frac12} \textstyle{\sum}^{j-1}_{k=1} 2^{-\delta_{1} k} 2^{-7\delta_{1} k}
		\\
		& +  C(1+E^3(0)) (T^*_{N_0})^{\frac12} \textstyle{\sum}^{j-2}_{k=1} \textstyle{\sum}_{m=k}^{j-2} 2^{-\delta_{1} k} 2^{-6\delta_{1} m}
		\\
		\leq & 	C(1+E^3(0))(T^*_{N_0})^{\frac12} [\frac13(1-2^{-\delta_{1}})]^{-2}.
	\end{split}
\end{equation}
By \eqref{kz11a} and Theorem \ref{TT2}, we also prove
\begin{equation}\label{kz15f}
	\begin{split}
		& |\bv_j,\rho_j| \leq 2+C_0,
		\\
		& E(t_2) \leq C_*.
	\end{split}
\end{equation}
Therefore, we can repeat the process with a length with $(C_*)^{-1}2^{-\delta_{1} j}$ till extending it to $T^*_{j-1}$. Moreover, on every new time-interval with $(C_*)^{-1}2^{-\delta_{1} j}$ \eqref{kz16} and \eqref{kz17} hold. Set
\begin{equation}\label{times1}
	X_1= \frac{T^*_{j-1}-T^*_j}{C^{-1}_*2^{-\delta_{1} j}}= (2^{\delta_{1}}-1)C_* E^{-1}(0).
\end{equation}
So we need a maximum of $X_1$-times to reach the time $T^*_{j-1}$ both in case 2(it's also adapt to case 1 for calculating the times). As a result, we shall calculate
\begin{equation*}
	\begin{split}
		 \|d \bv_j, d \rho_j\|_{L^1_{[0,T^*_{j-1}]} L^\infty_x}
		\leq & 	\|P_{\geq j}d\bv_j, P_{\geq j}d \rho_j\|_{L^1_{[0,T^*_{j-1}]} L^\infty_x}
		+ \textstyle{\sum}^{j-1}_{k=1} \|P_k d\bv_k, P_k d\rho_k\|_{L^1_{[0,T^*_{k}]} L^\infty_x}
		\\
		+ & \textstyle{\sum}^{j-1}_{k=1} \|P_k  (d\bv_{j}-d\bv_{j-1}), P_k  (d\rho_{j}-d\rho_{j-1})\|_{L^1_{[0,T^*_{j-1}]} L^\infty_x}.
		\\
		& + \textstyle{\sum}^{j-2}_{k=1} \textstyle{\sum}_{m=k}^{j-2}  \|P_k  (d\bv_{m+1}-d\bv_m), P_k  (d\rho_{m+1}-d\rho_m)\|_{L^1_{[0,T^*_{m}]} L^\infty_x}
\end{split}
\end{equation*}	
Due to \eqref{kz01}, \eqref{kz02}, \eqref{kz16}, \eqref{kz17}, and \eqref{pp8}, we obtain
\begin{equation}\label{kzqt}
	\begin{split}
		& \|d \bv_j, d \rho_j\|_{L^1_{[0,T^*_{j-1}]} L^\infty_x}
\\
		\leq & C(1+E^3(0)) (T^*_{N_0})^{\frac12} \textstyle{\sum}_{k=j}^{\infty} 2^{-\delta_{1} k} 2^{-6\delta_{1} j}\times (2^{\delta_{1}}-1)C_* E^{-1}(0)
		\\
		& + C   (1+E^3(0)) (T^*_{N_0})^{\frac12} \textstyle{\sum}^{j-1}_{k=1} 2^{-\delta_{1}k} 2^{-5\delta_{1} (j-1)} \times (2^{\delta_{1}}-1)C_* E^{-1}(0)
		\\
		& + C(1+E^3(0))(T^*_{N_0})^{\frac12} \textstyle{\sum}^{j-1}_{k=1} 2^{-\delta_{1} k} 2^{-7\delta_{1} k}
		\\
		& +  C(1+E^3(0))(T^*_{N_0})^{\frac12} \textstyle{\sum}^{j-2}_{k=1} \textstyle{\sum}_{m=k}^{j-2} 2^{-\delta_{1} k} 2^{-6\delta_{1} m}
		\\
		\leq & 	C(1+E^3(0))(T^*_{N_0})^{\frac12}[\frac13(1-2^{-\delta_{1}})]^{-2}.
	\end{split}
\end{equation}
Therefore, by \eqref{kzqt} and Theorem \eqref{TT2}, we can see that
\begin{equation}\label{kz270}
	\begin{split}
		& \| \bv_j, \rho_j \|_{L^\infty_{ [0,T^*_{j-1}]\times \mathbb{R}^3}} \leq 2+C_0,
		\\
	& E(T^*_{j-1}) \leq C_*.
	\end{split}
\end{equation}
At this moment, both in case 1 or case 2, seeing from \eqref{kz270}, \eqref{kzqt}, \eqref{kz14}, \eqref{kz15}, through a maximum of $X_1=(2^{\delta_{1}}-1)C_* E^{-1}(0)$ times with each length $C_*^{-1}2^{-\delta_{1}j}$ or $(2^{\delta_{1}}-1)E(0)^{-1}2^{-\delta_{1} j}$, we shall extend the solutions $(\bv_j,\rho_j,\bw_j)$ from $[0,T^*_j]$ to $[0,T^*_{j-1}]$, and
\begin{equation}\label{kz27}
\begin{split}
 & \| \bv_j, \rho_j \|_{L^\infty_{ [0,T^*_{j-1}]\times \mathbb{R}^3}} \leq 2+C_0,, \quad E(T^*_{j-1}) \leq C_*,
\\
& \|d \bv_j, d \rho_j\|_{L^1_{[0,T^*_{j-1}]} L^\infty_x}
\leq  	C(1+E^3(0))(T^*_{N_0})^{\frac12}[\frac13(1-2^{-\delta_{1}})]^{-2}.
\end{split}		
\end{equation}
As a result, we also have extended the solutions $(\bv_m,\rho_m,\bw_m)$ ($m\in[N_0, j-1]$) from $[0,T^*_m]$ to $[0,T^*_{m-1}]$. Moreover, referring \eqref{kz01} and \eqref{kz02}, \eqref{kz27}, and \eqref{times1}, we get
\begin{equation}\label{kz29}
\begin{split}
	& \| P_{k} d \bv_m, P_{k} d \rho_m \|_{L^1_{[0,T^*_{m-1}]}L^\infty_x}
	\\
	\leq  & \| P_{k} d\bv_m, P_{k} d \rho_m \|_{L^1_{[0,T^*_{m}]}L^\infty_x}
	\\
	& +\| P_{k} d \bv_m, P_{k} d \rho_m \|_{L^1_{[T^*_{m},T^*_{m-1}]}L^\infty_x}
	\\
	\leq  & C(1+E^3(0))(T^*_{N_0})^{\frac12} 2^{-\delta_{1}k} 2^{-7\delta_{1}m} + C(1+C^2_*) 2^{-\delta_{1}k} 2^{-7\delta_{1}m} \times (2^{\delta_{1}}-1)C_* E^{-1}(0)
	\\
	\leq  & C(1+E^3(0))(T^*_{N_0})^{\frac12} 2^{-\delta_{1}k} 2^{-7\delta_{1}m} + C(1+E^2(0)) 2^{-\delta_{1}k} 2^{-6\delta_{1}m} \times (2^{\delta_{1}}-1).
\end{split}	
\end{equation}
For $k \geq m\geq N_0+1$, using \eqref{pp8},  \eqref{kz29} yields
\begin{equation}\label{kz31}
	\begin{split}
		 \| P_{k} d \bv_m, P_{k} d \rho_m \|_{L^1_{[0,T^*_{m-1}]}L^\infty_x}
		\leq & C(1+E^3(0))(T^*_{N_0})^{\frac12} 2^{-\delta_{1}k} 2^{-5\delta_{1}m} \times 2^{\delta_{1}},
	\end{split}	
\end{equation}
Similarly, if $m\geq N_0+1$, using \eqref{pp8}, then the following estimate holds for $k<j$
\begin{equation}\label{kz34}
\begin{split}
& \|P_k (d{\rho}_{m}-d{\rho}_{m-1}), P_k (d{\bv}_{m}-d{\bv}_{m-1}) \|_{L^1_{[0,T^*_{m-1}]} L^\infty_x}
\\
\leq    &  \|P_k (d{\rho}_{m}-d{\rho}_{m-1}), P_k  (d{\bv}_{m}-d{\bv}_{m-1}) \|_{L^1_{[0,T^*_{m}]} L^\infty_x}
\\
& +\|P_k (d {\rho}_{m}-d {\rho}_{m-1}), P_k (d {\bv}_{m}-d {\bv}_{m-1}) \|_{L^1_{[T^*_m,T^*_{m-1}]} L^\infty_x}
\\
 \leq & C  (1+E(0)^3)(T^*_{N_0})^{\frac12} 2^{-\delta_{1}k} 2^{-6\delta_1(m-1)}
 \\
 & + C  (1+C^3_*) 2^{-\delta_{1}k} 2^{-6\delta_1(m-1)} \times (2^{\delta_{1}}-1)C_* E^{-1}(0)
\\
\leq & C(1+E^3(0))(T^*_{N_0})^{\frac12} 2^{-\delta_{1}k} 2^{-5\delta_{1}(m-1)} \times 2^{\delta_{1}}.
\end{split}	
\end{equation}
\quad \textbf{Step 2: Extending time interval $[0,T^*_j]$ to $[0,T^*_{N_0}]$}. Based the above analysis in Step 1, we can give a induction by achieving the goal. We assume the solutions $(\bv_j, \rho_j, \bw_j)$ can be extended from $[0,T^*_j]$ to $[0,T^*_{j-l}]$ through a maximal $X_l$ times and
\begin{equation}\label{kz41}
	\begin{split}
		X_l=& \frac{T_{j-l}^*- T_{j}^*}{C^{-1}_* 2^{-\delta_{1} j}}
		\\
		=& \frac{E(0)^{-1}( 2^{-\delta_{1}(j-l)} - 2^{-\delta_{1}j} )}{C^{-1}_* 2^{-\delta_{1} j}}
		\\
		=& \frac{C_*}{E(0)} (2^{\delta_{1}l}-1).
	\end{split}	
\end{equation}
Moreover, the following bounds
\begin{equation}\label{kz42}
	\begin{split}
		\| P_{k} d \bv_j, P_{k} d \rho_j \|_{L^1_{[0,T^*_{j-l}]}L^\infty_x}
		\leq & C(1+E^3(0))(T^*_{N_0})^{\frac12} 2^{-\delta_{1}k} 2^{-5\delta_{1}j} \times 2^{\delta_{1}l}, \qquad k \geq j,
	\end{split}	
\end{equation}
and
\begin{equation}\label{kz43}
	\begin{split}
		& \|P_k (d{\rho}_{m+1}-d{\rho}_{m}), P_k (d{\bv}_{m+1}-d{\bv}_{m}) \|_{L^1_{[0,T^*_{m-l}]} L^\infty_x}
		\\
		\leq  & C(1+E^3(0))(T^*_{N_0})^{\frac12} 2^{-\delta_{1}k} 2^{-5\delta_{1}m } \times 2^{\delta_{1}l}, \quad k<j,
	\end{split}	
\end{equation}
and
\begin{equation}\label{kz44}
	\|d \bv_j, d \rho_j\|_{L^1_{[0,T^*_{j-l}]} L^\infty_x}
	\leq  	C(1+E^3(0))(T^*_{N_0})^{\frac12}[\frac13(1-2^{-\delta_{1}})]^{-2}.
\end{equation}
and
\begin{equation}\label{kz45}
	|\bv_j, \rho_j| \leq 2+C_0, \quad E(T^*_{j-l}) \leq C_*.
\end{equation}
In the following, we will check these estimates \eqref{kz41}-\eqref{kz45} hold when $l=1$, and then also holds when we replace $l$ by $l+1$.

Using \eqref{kz27}, \eqref{kz31}, \eqref{times1}, and \eqref{kz34}, then \eqref{kz42}-\eqref{kz45} hold by taking $l=1$. Let us now check it for $l+1$. In this case, it implies that  $T^*_{j-l} \leq T^*_{N_0}$. Therefore, $j-l\geq N_0+1$ should hold. Starting at the time $T^*_{j-l}$, seeing \eqref{DTJ} and \eqref{kz45}, we shall get an extending time-interval of $(\bv_j,\rho_j,\bw_j)$ with a length of $(C_*)^{-1}2^{-\delta_{1} j}$. So we can go to the case 2 in step 1, and the length every new time-interval is $C_*^{-1} 2^{-\delta_{1} j}$. So the times is
\begin{equation}\label{kz46}
	X= \frac{T^*_{j-(l+1)}-T^*_{j-l}}{C^{-1}_*2^{-\delta_{1} j}}= 2^{\delta_{1} l}(2^{\delta_{1}}-1)C_* E^{-1}(0).
\end{equation}
Thus, we can deduce that
\begin{equation}\label{kz40}
	X_{l+1}=X_l+X= (2^{\delta_{1}(l+1)}-1)C_* E^{-1}(0).
\end{equation}
Moreover, for $k \geq j$, we have
\begin{equation}\label{kz48}
	\begin{split}
		\| P_{k} d \bv_j, P_{k} d\rho_j \|_{L^1_{[0, T^*_{j-(l+1)}]}L^\infty_x}\leq & \| P_{k} d \bv_j, P_{k} d\rho_j \|_{L^1_{[0,T^*_{j-l}]}L^\infty_x}
		\\
		& +	\| P_{k} d \bv_j, P_{k} d\rho_j \|_{L^1_{[T^*_{j-l}, T^*_{j-(l+1)}]}L^\infty_x}.
	\end{split}
\end{equation}
Using \eqref{kz01} and \eqref{kz45}
\begin{equation}\label{kz49}
\begin{split}
		& \| P_{k} d \bv_j, P_{k} d\rho_j \|_{L^1_{[T^*_{j-l}, T^*_{j-(l+1)}]}L^\infty_x}
	\\
	\leq   & C  (1+C^3_*)(T^*_{N_0})^{\frac12} 2^{-\delta_{1}k} 2^{-7\delta_{1}j}\times 2^{\delta_{1} l}(2^{\delta_{1}}-1)C_* E^{-1}(0), \quad k\geq j.
\end{split}
\end{equation}
Due to \eqref{pp8}, we can see
\begin{equation*}
	\begin{split}
	 (1+C^3_*) C_* E^{-1}(0)(1+E^3(0))^{-1}2^{-\delta_{1} N_0} \leq 1.
	\end{split}
\end{equation*}
Hence, from \eqref{kz49} we have
\begin{equation}\label{kz50}
	\begin{split}
		& \| P_{k} d \bv_j, P_{k} d\rho_j \|_{L^1_{[T^*_{j-l}, T^*_{j-(l+1)}]}L^\infty_x}
		\\
		\leq   & C  (1+E^3(0)) (T^*_{N_0})^{\frac12} 2^{-\delta_{1}k} 2^{-5\delta_{1}j}\times 2^{\delta_{1} l}(2^{\delta_{1}}-1), \quad k\geq j.
	\end{split}
\end{equation}
Using \eqref{kz42} and \eqref{kz50}, so we get
\begin{equation}\label{kz51}
	\begin{split}
		\| P_{k} d \bv_j, P_{k} d \rho_j \|_{L^1_{[0,T^*_{j-(l+1)}]}L^\infty_x}
		\leq & C(1+E^3(0)) (T^*_{N_0})^{\frac12} 2^{-\delta_{1}k} 2^{-5\delta_{1}j} \times 2^{\delta_{1}(l+1)}, \qquad k \geq j.
	\end{split}	
\end{equation}
If $k<j$, then we have
\begin{equation}\label{kz52}
	\begin{split}
		& \|P_k (d{\rho}_{m+1}-d{\rho}_{m}), P_k (d{\bv}_{m+1}-d{\bv}_{m}) \|_{L^1_{[0,T^*_{m-(l+1)}]} L^\infty_x}
		\\
		\leq  & \|P_k (d{\rho}_{m+1}-d{\rho}_{m}), P_k (d{\bv}_{m+1}-d{\bv}_{m}) \|_{L^1_{[0,T^*_{m-l}]} L^\infty_x}
		\\
		& + \|P_k (d{\rho}_{m+1}-d{\rho}_{m}), P_k (d{\bv}_{m+1}-d{\bv}_{m}) \|_{L^1_{[T^*_{m-l},T^*_{m-(l+1)}]} L^\infty_x}.
	\end{split}	
\end{equation}
When we extend the solutions $(\bv_j, \rho_j, \bw_j)$ from $[0,T^*_{j-l}]$ to $[0,T^*_{j-(l+1)}]$, then the solutions $(\bv_m, \rho_m, \bw_m)$ is also extended from $[0,T^*_{m-l}]$ to $[0,T^*_{m-(l+1)}]$. Seeing \eqref{kz02} and \eqref{kz45}, we can obtain
\begin{equation}\label{kz53}
	\begin{split}
		& \|P_k (d{\rho}_{m+1}-d{\rho}_{m}), P_k (d{\bv}_{m+1}-d{\bv}_{m}) \|_{L^1_{[T^*_{m-l},T^*_{m-(l+1)}]} L^\infty_x}
		\\
		\leq  & C(1+C_*^3)(T^*_{N_0})^{\frac12} 2^{-\delta_{1}k} 2^{-6\delta_{1}m } \times 2^{\delta_{1} l}(2^{\delta_{1}}-1)C_* E^{-1}(0).
	\end{split}	
\end{equation}
Hence, for $m-l \geq N_0+1$, using \eqref{pp8}, it yields
\begin{equation*}
	 (1+C_*^3)  C_* E^{-1}(0) (1+E^3(0))^{-1} 2^{-\delta_{1}N_0 } \leq 1.
\end{equation*}
So \eqref{kz53} becomes
\begin{equation}\label{kz54}
	\begin{split}
		& \|P_k (d{\rho}_{m+1}-d{\rho}_{m}), P_k (d{\bv}_{m+1}-d{\bv}_{m}) \|_{L^1_{[T^*_{m-l},T^*_{m-(l+1)}]} L^\infty_x}
		\\
		\leq  & C(1+E^3(0))(T^*_{N_0})^{\frac12} 2^{-\delta_{1}k} 2^{-5\delta_{1}m } \times 2^{\delta_{1} l}(2^{\delta_{1}}-1).
	\end{split}	
\end{equation}
Inserting \eqref{kz43} and \eqref{kz54} to \eqref{kz52}, we have
\begin{equation}\label{kz55}
	\begin{split}
		& \|P_k (d{\rho}_{m+1}-d{\rho}_{m}), P_k (d{\bv}_{m+1}-d{\bv}_{m}) \|_{L^1_{[0,T^*_{m-(l+1)}]} L^\infty_x}
		\\
		\leq  & C(1+E^3(0))(T^*_{N_0})^{\frac12} 2^{-\delta_{1}k} 2^{-5\delta_{1}m } \times 2^{\delta_{1}(l+1)}, \quad k<j,
	\end{split}	
\end{equation}
Let us now bound the following Strichartz estimate
\begin{equation}\label{kz56}
	\begin{split}
		& \|d \bv_j, d \rho_j\|_{L^1_{[0,T^*_{j-(l+1)}]} L^\infty_x}
		\\
		\leq & \|P_{\geq j}d\bv_j, P_{\geq j}d \rho_j\|_{L^1_{[0,T^*_{j-(l+1)}]} L^\infty_x}
		\\
		& + \textstyle{\sum}^{j-1}_{k=j-l}\textstyle{\sum}^{j-1}_{m=k} \|P_k  (d\bv_{m+1}-d\bv_{m}), P_k  (d\rho_{m+1}-d\rho_{m})\|_{L^1_{[0,T^*_{j-(l+1)}]} L^\infty_x}
		\\
		& + \textstyle{\sum}^{j-1}_{k=1} \|P_k d\bv_k, P_k d\rho_k\|_{L^1_{[0,T^*_{j-(l+1)}]} L^\infty_x}
		\\
		=& 	\|P_{\geq j}d\bv_j, P_{\geq j}d \rho_j\|_{L^1_{[0,T^*_{j-(l+1)}]} L^\infty_x}
		\\
		& + \textstyle{\sum}^{j-1}_{k=j-l} \|P_k d\bv_k, P_k d\rho_k\|_{L^1_{[0,T^*_{j-(l+1)}]} L^\infty_x}
		+ \textstyle{\sum}^{j-(l+1)}_{k=1} \|P_k d\bv_k, P_k d\rho_k\|_{L^1_{[0,T^*_{j-(l+1)}]} L^\infty_x}
		\\
		& + \textstyle{\sum}^{j-1}_{k=1} \textstyle{\sum}^{j-1}_{m=j-(l+1)} \|P_k  (d\bv_{m+1}-d\bv_{m}), P_k  (d\rho_{m+1}-d\rho_{m})\|_{L^1_{[0,T^*_{j-(l+1)}]} L^\infty_x}
		\\
		& + \textstyle{\sum}^{j-(l+2)}_{k=1} \textstyle{\sum}_{m=k}^{j-(l+2)}  \|P_k  (d\bv_{m+1}-d\bv_m), P_k  (d\rho_{m+1}-d\rho_m)\|_{L^1_{[0,T^*_{j-(l+1)}]} L^\infty_x}
		\\
		=& \Theta_1+  \Theta_2+ \Theta_3+ \Theta_4+ \Theta_5,
	\end{split}
\end{equation}
where
\begin{equation}\label{Theta}
	\begin{split}
		\Theta_1= &  \|P_{\geq j}d\bv_j, P_{\geq j}d \rho_j\|_{L^1_{[0,T^*_{j-(l+1)}]} L^\infty_x} ,
		\\
		\Theta_2= & \textstyle{\sum}^{j-(l+2)}_{k=1}\textstyle{\sum}^{j-(l+2)}_{m=k} \|P_k  (d\bv_{m+1}-d\bv_{m}), P_k  (d\rho_{m+1}-d\rho_{m})\|_{L^1_{[0,T^*_{j-(l+1)}]} L^\infty_x},
		\\
		\Theta_3= & \textstyle{\sum}^{j-1}_{k=1}\textstyle{\sum}^{j-1}_{m=j-(l+1)} \|P_k  (d\bv_{m+1}-d\bv_{m}), P_k  (d\rho_{m+1}-d\rho_{m})\|_{L^1_{[0,T^*_{j-(l+1)}]} L^\infty_x},
		\\
		\Theta_4 =& \textstyle{\sum}^{j-(l+1)}_{k=1}	\|P_{k}d\bv_k, P_{k}d\rho_k\|_{L^1_{[0,T^*_{j-(l+1)}]} L^\infty_x},
		\\
		\Theta_5 =& \textstyle{\sum}^{j-1}_{k=j-l}	\|P_{k}d\bv_k, P_{k}d\rho_k\|_{L^1_{[0,T^*_{j-(l+1)}]} L^\infty_x}.
	\end{split}
\end{equation}
To get the estimate on time-interval $[0,T^*_{j-(l+1)}]$ for $\Theta_1, \Theta_2, \Theta_3, \Theta_4$ and $\Theta_5$, we should note that
there is no growth for $\Theta_2$ and $\Theta_4$ in this extending process. For example, considering $\Theta_2$, the existing time-interval of $P_k  (d\bv_{m+1}-d\bv_{m})$ is actually $[0,T^*_{m+1}]$, and $[0,T^*_{j-(l+1)}] \subseteq [0,T^*_{m+1}]$ if $m \geq j-(l+2)$. So we can use the initial bounds \eqref{kz01} and \eqref{kz02} to handle $\Theta_2$ and $\Theta_4$. While, considering $\Theta_1, \Theta_3$, and $\Theta_5$, we need to calculate the growth in the Strichartz estimate. Based on this idea, let us give a precise analysis on \eqref{Theta}.

According to \eqref{kz51}, we can estimate $\Theta_1$ as
\begin{equation}\label{kz57}
	\begin{split}
		\Theta_1 \leq & C(1+E^3(0))(T^*_{N_0})^{\frac12} \textstyle{\sum}^{\infty}_{k=j} 2^{-\delta_{1}k} 2^{-5\delta_{1}j} \times 2^{\delta_{1}(l+1)}.
	\end{split}
\end{equation}
Due to \eqref{kz01}, we have
\begin{equation}\label{kz58}
	\begin{split}
	\Theta_2\leq &	\textstyle{\sum}^{j-(l+2)}_{k=1}\textstyle{\sum}^{j-(l+2)}_{m=k} \|P_k  (d\bv_{m+1}-d\bv_{m}), P_k  (d\rho_{m+1}-d\rho_{m})\|_{L^1_{[0,T^*_{m+1}]} L^\infty_x},
	\\
	\leq &  C(1+E^3(0))(T^*_{N_0})^{\frac12} \textstyle{\sum}^{j-(l+2)}_{k=1}\textstyle{\sum}^{j-(l+2)}_{m=k} 2^{-\delta_{1}k} 2^{-6\delta_{1}m}.
	\end{split}	
\end{equation}
For $1 \leq k \leq j-1$ and $m \leq j-1$, using \eqref{kz55}, it follows
\begin{equation}\label{kz59}
\begin{split}
	\Theta_3\leq & \textstyle{\sum}^{j-1}_{k=1}\textstyle{\sum}^{j-1}_{m=j-(l+1)} \|P_k  (d\bv_{m+1}-d\bv_{m}), P_k  (d\rho_{m+1}-d\rho_{m})\|_{L^1_{[0,T^*_{m-(l+1)}]} L^\infty_x}
	\\
	\leq &  C  (1+E^3(0))(T^*_{N_0})^{\frac12} \textstyle{\sum}^{j-1}_{k=1}\textstyle{\sum}^{j-1}_{m=j-(l+1)}  2^{-\delta_{1}k} 2^{-5\delta_{1}m} \times 2^{\delta_{1}(l+1)}.
\end{split}
\end{equation}
Due to \eqref{kz01}, we can see
\begin{equation}\label{kz60}
	\begin{split}
		\Theta_4 \leq & \textstyle{\sum}^{j-(l+1)}_{k=1}	\|P_{k}d\bv_k, P_{k}d\rho_k\|_{L^1_{[0,T^*_{k}]} L^\infty_x}
		\\
		\leq & C  (1+E^3(0))(T^*_{N_0})^{\frac12} \textstyle{\sum}^{j-(l+1)}_{k=1} 2^{-\delta_{1}k} 2^{-7\delta_{1}k}.
	\end{split}
\end{equation}
If $j-l \leq k\leq j-1$, then $k+l+1-j \leq l$. By using \eqref{kz42}, we can estimate
\begin{equation}\label{kz61}
	\begin{split}
	\Theta_5 =& \textstyle{\sum}^{j-1}_{k=j-l}	\|P_{k}d\bv_k, P_{k}d\rho_k\|_{L^1_{[0,T^*_{j-(l+1)}]} L^\infty_x}
	\\
	= & \textstyle{\sum}^{j-1}_{k=j-l}	\|P_{k}d\bv_k, P_{k}d\rho_k\|_{L^1_{[0,T^*_{k-(k+l+1-j)}]} L^\infty_x}
	\\
	\leq	& C  (1+E^3(0))(T^*_{N_0})^{\frac12} \textstyle{\sum}^{j-1}_{k=j-l} 2^{-\delta_{1}k} 2^{-5\delta_{1}j}2^{\delta_{1}(k+l+1-j)}.
	\end{split}
\end{equation}
Inserting \eqref{kz57}-\eqref{kz61} to \eqref{kz56}, it follows
\begin{equation}\label{kz62}
	\begin{split}
		& \|d \bv_j, d \rho_j\|_{L^1_{[0,T^*_{j-(l+1)}]} L^\infty_x}
		\\
		\leq & C(1+E^3(0)) (T^*_{N_0})^{\frac12} \textstyle{\sum}^{\infty}_{k=j} 2^{-\delta_{1}k} 2^{-5\delta_{1}j} \times 2^{\delta_{1}(l+1)}
		\\
		& + C(1+E^3(0))(T^*_{N_0})^{\frac12} \textstyle{\sum}^{j-(l+2)}_{k=1}\textstyle{\sum}^{j-(l+2)}_{m=k} 2^{-\delta_{1}k} 2^{-6\delta_{1}m}
		\\
		&+C  (1+E^3(0))(T^*_{N_0})^{\frac12} \textstyle{\sum}^{j-1}_{k=1}\textstyle{\sum}^{j-1}_{m=j-(l+1)}  2^{-\delta_{1}k} 2^{-5\delta_{1}m} \times 2^{\delta_{1}(l+1)}
		\\
		& + C  (1+E^3(0))(T^*_{N_0})^{\frac12} \textstyle{\sum}^{j-(l+1)}_{k=1} 2^{-\delta_{1}k} 2^{-7\delta_{1}k}
		\\
		& + C  (1+E^3(0))(T^*_{N_0})^{\frac12} \textstyle{\sum}^{j-1}_{k=j-l} 2^{-\delta_{1}k} 2^{-5\delta_{1}k}2^{\delta_{1}(k+l+1-j)}
	\end{split}
\end{equation}
In the case of $j-l \geq N_0+1$ and $j\geq N_0+1$, \eqref{kz62} yields
\begin{equation}\label{kz63}
	\begin{split}
		& \|d \bv_j, d \rho_j\|_{L^1_{[0,T^*_{j-(l+1)}]} L^\infty_x}
		\\
		\leq & C(1+E^3(0))(T^*_{N_0})^{\frac12}(1-2^{-\delta_{1}})^{-2} \big\{
		2^{-\delta_{1}j} 2^{\delta_{1}(l+1)}+ 2^{-6\delta_{1}}
		\\
		& \quad +2^{-5\delta_{1}[j-(l+1)]}  2^{\delta_{1}(l+1)}+ 2^{-6\delta_{1}}+ 2^{-5\delta_{1}(j-l)}2^{\delta_{1}(l+1-j)}
		\big\}
		\\
		\leq & C(1+E^3(0))(T^*_{N_0})^{\frac12}(1-2^{-\delta_{1}})^{-2} \left\{
		2^{-6\delta_{1}N_0} + 2^{-6\delta_{1}} +	2^{-5\delta_{1}N_0} + 2^{-6\delta_{1}}+ 	2^{-6\delta_{1}N_0} 	\right\}
		\\
		\leq & C(1+E^3(0))(T^*_{N_0})^{\frac12}[\frac13(1-2^{-\delta_{1}})]^{-2}.
	\end{split}
\end{equation}
By using \eqref{kz63}, \eqref{pu00}, \eqref{pp8}, and Theorem \ref{TT2}, we have proved
\begin{equation}\label{kz64}
	\begin{split}
		|\bv_j, \rho_j|\leq 2+C_0, \quad  E(T^*_{j-(l+1)}) \leq C_*.
	\end{split}
\end{equation}
Gathering \eqref{kz40}, \eqref{kz51}, \eqref{kz55}, \eqref{kz63} and \eqref{kz64}, we know that \eqref{kz41}-\eqref{kz45} holding for $l+1$. So our induction hold \eqref{kz41}-\eqref{kz45} for $l=1$ to $l=j-N_0$. Therefore, we can extend the solutions $(\bv_j,\rho_j,\bw_j)$ from $[0,T^*_j]$ to $[0,T^*_{N_0}]$ when $j \geq N_0$. Let us denote
\begin{equation}\label{Tstar}
	T^*=T^*_{N_0}=[E(0)]^{-1}2^{-\delta_{1} N_0}.
\end{equation}
Setting $l=j-N_0$ in \eqref{kz44}-\eqref{kz45}, we therefore get
\begin{equation}\label{kz65}
	\begin{split}
	& E(T^*) \leq  C_*, \quad \|\bv_j, \rho_j\|_{L^\infty_{[0,T^*]} L^\infty_x} \leq 2+C_0,
	\\
	& \|d \bv_j, d \rho_j\|_{L^1_{[0,T^*]} L^\infty_x}
	\leq  C(1+E^3(0))(T^*_{N_0})^{\frac12}[\frac13(1-2^{-\delta_{1}})]^{-2},
	\end{split}
\end{equation}
where $N_0$ and $C_*$ depends on $C_0, c_0, s, M_*$. It is stated in \eqref{pp8} and \eqref{Cstar}. In this process, we can also conclude
\begin{equation}\label{kz66}
	\begin{split}
		\|d \bv_j, d \rho_j\|_{L^2_{[0,T^*]} L^\infty_x}
		\leq & C(1+E^3(0))[\frac13(1-2^{-\delta_{1}})]^{-2}.
	\end{split}
\end{equation}
\subsubsection{Strichartz estimates of linear wave equation on time-interval $[0,T^*_{N_0}]$.}\label{finalq}
We still expect the behavior of a linear wave equation endowed with $g_j=g(\bv_j,\rho_j)$. So we claim a theorem as follows
\begin{proposition}\label{rut}
	For $\frac{s}{2} \leq r \leq 3$, there is a solution $f_j$ on $[0,T^*_{N_0}]\times \mathbb{R}^3$ satisfying the following linear wave equation
	\begin{equation}\label{ru01}
		\begin{cases}
			\square_{{g}_j} f_j=0,
			\\
			(f_j,\partial_t f_j)|_{t=0}=(f_{0j},f_{1j}),
		\end{cases}
	\end{equation}
where $(f_{0j},f_{1j})=(P_{\leq j}f_0,P_{\leq j}f_1)$ and $(f_0,f_1)\in H_x^r \times H^{r-1}_x$. Moreover, for $a\leq r-\frac{s}{2}$, we have
	\begin{equation}\label{ru02}
		\begin{split}
			&\|\left< \partial \right>^{a-1} d{f}_j\|_{L^2_{[0,T^*_{N_0}]} L^\infty_x}
			\leq  C_{M_*}(\|{f}_0\|_{{H}_x^r}+ \|{f}_1 \|_{{H}_x^{r-1}}),
			\\
			&\|{f}_j\|_{L^\infty_{[0,T^*_{N_0}]} H^{r}_x}+ \|\partial_t {f}_j\|_{L^\infty_{[0,T^*_{N_0}]} H^{r-1}_x} \leq  C_{M_*}(\| {f}_0\|_{H_x^r}+ \| {f}_1\|_{H_x^{r-1}}).
		\end{split}
	\end{equation}
\end{proposition}
\begin{proof}[Proof of Proposition \ref{rut}.] Our proof also relies a Strichartz estimates on a short time-interval. Then a loss of derivatives are then obtained by summing up the short time estimates with respect to these time intervals.

Using \eqref{ru01}, \eqref{ru03} and \eqref{ru04}, we have
\begin{equation}\label{ru06}
	\begin{split}
		\| \left< \partial \right>^{a-1+\frac92\delta_{1}} df_j\|_{L^2_{[0,T^*_j]} L^\infty_x} \leq & C( \| f_{0j} \|_{H^r_x} + \| f_{1j} \|_{H^{r-1}_x}  )
		\\
		\leq & C ( \| f_{0} \|_{H^r_x} + \| f_{1} \|_{H^{r-1}_x}  ),
	\end{split}
\end{equation}
where we use
\begin{equation}\label{ru060}
	a-1+\frac92\delta_{1} \leq r-\frac{s}{2}+\frac92\delta_{1} <r-1.
\end{equation}
By Bernstein inequality, for $k\geq j$, we shall obtain
\begin{equation}\label{ru07}
	\begin{split}
		\| \left< \partial \right>^{a-1} P_k df_j\|_{L^2_{[0,T^*_j]} L^\infty_x}
		=& C2^{-\frac92\delta_{1} k}\| \left< \partial \right>^{a-1+\frac92\delta_{1}} P_k df_j\|_{L^2_{[0,T^*_j]} L^\infty_x}
		\\
		\leq & C2^{-\frac12\delta_{1} k} 2^{-4\delta_{1} j} \| \left< \partial \right>^{a-1+\frac92\delta_{1}} df_j\|_{L^2_{[0,T^*_j]} L^\infty_x}.
	\end{split}
\end{equation}
Combining \eqref{ru06} and \eqref{ru07}, for $a \leq r-\frac{s}{2}$, we get
\begin{equation}\label{ru070}
	\begin{split}
		\| \left< \partial \right>^{a-1} P_k df_j\|_{L^2_{[0,T^*_j]} L^\infty_x}
		\leq & C2^{-\frac12\delta_{1} k} 2^{-4\delta_{1} j} ( \| f_{0} \|_{H^r_x} + \| f_{1} \|_{H^{r-1}_x}  ), \quad k\geq j.
	\end{split}
\end{equation}
On the other hand, for any integer $m\geq 1$, we also have
\begin{equation*}
	\begin{cases}
		\square_{{g}_m} f_m=0, \quad [0,T^*_{m}]\times \mathbb{R}^3,
		\\
		(f_m,\partial_t f_m)|_{t=0}=(f_{0m},f_{1m}),
	\end{cases}
\end{equation*}
and
\begin{equation*}
	\begin{cases}
		\square_{{g}_{m+1}} f_{m+1}=0, \quad [0,T^*_{m+1}]\times \mathbb{R}^3,
		\\
		(f_{m+1},\partial_t f_{m+1})|_{t=0}=(f_{0(m+1)},f_{1(m+1)}).
	\end{cases}
\end{equation*}
So the difference term $f_{m+1}-f_m$ satisfies
\begin{equation}\label{ru08}
	\begin{cases}
		\square_{{g}_{m+1}} (f_{m+1}-f_m)=({g}^{\alpha i}_{m+1}-{g}^{\alpha i}_{m} )\partial_{\alpha i} f_m, \quad [0,T^*_{m+1}]\times \mathbb{R}^3,
		\\
		(f_{m+1}-f_m,\partial_t (f_{m+1}-f_m))|_{t=0}=(f_{0(m+1)}-f_{0m},f_{1(m+1)}-f_{1m}).
	\end{cases}
\end{equation}
By Duhamel's principle, \eqref{ru04} and \eqref{ru060}, it yields
\begin{equation}\label{ru09}
	\begin{split}
		& \| \left< \partial \right>^{a-1} P_k (d f_{m+1}-df_m)\|_{L^2_{[0,T^*_{m+1}]} L^\infty_x}
		\\
		\leq & C\| f_{0(m+1)}-f_{0m}\|_{H_x^{r-\frac92\delta_{1}}} + C\|f_{1(m+1)}-f_{1m}\|_{H_x^{r-1-\frac92\delta_{1}}}
		\\
		& \ + C\|({g}_{m+1}-{g}_{m}) \cdot\nabla d f_m\|_{L^1_{[0,T^*_{m+1}]} H^{r}_x}
		\\
		\leq & C2^{-\frac92\delta_{1} m}( \| f_{0} \|_{H^r_x} + \| f_{1} \|_{H^{r-1}_x}  ) + C \| {g}_{m+1}-{g}_{m}\|_{L^1_{[0,T^*_{m+1}]} L^\infty_x} \| \partial d f_m\|_{L^\infty_{[0,T^*_{m+1}]} H^{r-1}_x} .
	\end{split}
\end{equation}
By using the energy estimates
\begin{equation}\label{ru10}
	\begin{split}
		\| \partial d f_m\|_{L^\infty_{[0,T^*_{m+1}]} H^{r-1}_x} \leq & C\| d f_m\|_{L^\infty_{[0,T_{m+1}]} H^{r}_x}
		\\
		\leq & C(\|f_{0m}\|_{H^{r+1}_x}+\|f_{1m}\|_{H^{r}_x})
		\\
		\leq & C2^m (\|f_{0m}\|_{H^r_x}+\|f_{1m}\|_{H^{r-1}_x}) .
	\end{split}
\end{equation}
By using Strichartz estimates \eqref{ru04} and Lemma \ref{LD} again, we shall prove that
\begin{equation}\label{ru11}
	\begin{split}
		& 2^m \| {g}_{m+1}-{g}_{m}\|_{L^1_{[0,T^*_{m+1}]} L^\infty_x}
		\\
		\leq & 2^m (T^*_{m+1})^{\frac12}\| {\bv}_{m+1}-{\bv}_{m},  {\rho}_{m+1}-{\rho}_{m}\|_{L^2_{[0,T^*_{m+1}]} L^\infty_x}
		\\
		\leq & C2^m ( \| {\bv}_{m+1}-{\bv}_{m},  {\rho}_{m+1}-{\rho}_{m}\|_{L^\infty_{[0,T^*_{m+1}]} H^{1+\delta_{1}}_x} + \| {\bw}_{m+1}-{\bw}_{m}\|_{L^2_{[0,T^*_{m+1}]} H^{\frac12+\delta_{1}}_x})
		\\
		\leq & C2^{-7\delta_{1}m } (1+E^3(0)).
	\end{split}
\end{equation}
Above, we use \eqref{yu0}, \eqref{yu1}, and \eqref{DTJ}.

Due to \eqref{ru10} and \eqref{ru11}, for $k<j$, and $k\leq m$,  \eqref{ru09} yields
\begin{equation}\label{ru12}
	\begin{split}
		\| \left< \partial \right>^{a-1} P_k (df_{m+1}-df_m)\|_{L^2_{[0,T^*_{m+1}]} L^\infty_x}
		\leq  C2^{-\frac12\delta_{1} k}2^{-4\delta_{1} m} \| (f_{0},f_1) \|_{H^r_x\times H^{r-1}_x}  (1+E^3(0)).
	\end{split}
\end{equation}

In the following part, let us consider the energy estimates. Taking the operator $\left< \partial \right>^{r-1}$ on \eqref{ru01}, then we get
	\begin{equation}\label{ru15}
	\begin{cases}
		\square_{{g}_j} \left< \partial \right>^{r-1}f_j=[ \square_{{g}_j} ,\left< \partial \right>^{r-1}] f_j,
		\\
		(\left< \partial \right>^{r-1}f_j,\partial_t \left< \partial \right>^{r-1}f_j)|_{t=0}=(\left< \partial \right>^{r-1}f_{0j},\left< \partial \right>^{r-1}f_{1j}).
	\end{cases}
\end{equation}
To estimate $[ \square_{{g}_j} ,\left< \partial \right>^{r-1}] f_j$, we will divide it into two cases $\frac{s}{2}<r\leq s$ and $s<r\leq 3$.

\textit{Case 1: $\frac{s}{2}<r\leq s$.} For $\frac{s}{2}<r\leq s$, note that
\begin{equation}\label{ru14}
		\begin{split}
	 [ \square_{{g}_j} ,\left< \partial \right>^{r-1}] f_j
	=& [{g}^{\alpha i}_j-\mathbf{m}^{\alpha i}, \left< \partial \right>^{r-1} \partial_i ] \partial_{\alpha}f_j + \left< \partial \right>^{r-1}( \partial_i g_j \partial_{\alpha}f_j )
	\\
	 = &[{g}_j-\mathbf{m}, \left< \partial \right>^{r-1} \nabla] df_j + \left< \partial \right>^{r-1}( \nabla g_j df).
	\end{split}
	\end{equation}
By \eqref{ru14} and Kato-Ponce estimates, we have\footnote{If $r=s$, then $L^{\frac{3}{s-r}}_x=L^\infty_x$.}
\begin{equation}\label{ru16}
	\| [ \square_{{g}_j} ,\left< \partial \right>^{r-1}] f_j \|_{L^2_x} \leq C ( \|dg_j\|_{L^\infty_x} \|d f_j \|_{H^{r-1}_x} + \|\left< \partial \right>^{r} (g_j-\mathbf{m}) \|_{L^{\frac{3}{\frac32-s+r}}_x} \| df_j \|_{L^{\frac{3}{s-r}}_x} )
\end{equation}
By Sobolev's inequality, it follows
\begin{equation}\label{ru17}
\|\left< \partial \right>^{r} (g_j-\mathbf{m}) \|_{L^{\frac{3}{\frac32-s+r}}_x} \leq C\|g_j-\mathbf{m}\|_{H^s_x}.
\end{equation}
By Gagliardo-Nirenberg inequality and Young's inequality, we can get
\begin{equation}\label{ru18}
	\begin{split}
	\| df_j \|_{L^{\frac{3}{s-r}}_x} \leq & C\|\left< \partial \right>^{r-1} df_j \|^{\frac{s-2}{3-s}}_{L^{2}_x} \|\left< \partial \right>^{r-\frac{s}{2}-1} df_j \|^{\frac{5-2s}{3-s}}_{L^{\infty}_x}
	\\
	\leq & C( \|df_j \|_{H^{r-1}_x}+ \|\left< \partial \right>^{r-\frac{s}{2}-1} df_j \|_{L^{\infty}_x})
	\end{split}
\end{equation}
\textit{Case 2: $s<r\leq 3$.} For $s<r\leq 3$, using Kato-Ponce estimates, we have
\begin{equation}\label{ru150}
	\begin{split}
		&  \| [ \square_{{g}_j} ,\left< \partial \right>^{r-1}] f_j \|_{L^2_x}
		\\
	= & \| {g}^{\alpha i}_j-\mathbf{m}^{\alpha i},\left< \partial \right>^{r-1}] \partial_{\alpha i}f_j \|_{L^2_x}
	\\
	\leq & C ( \|dg_j\|_{L^\infty_x} \|d f_j \|_{H^{r-1}_x} + \|\left< \partial \right>^{r-1} (g_j-\mathbf{m}) \|_{L^{\frac{3}{\frac12-s+r}}_x} \|\nabla df_j \|_{L^{\frac{3}{1+s-r}}_x} ) .
	\end{split}
\end{equation}
By Sobolev's inequality, it follows
\begin{equation}\label{ru151}
	\|\left< \partial \right>^{r-1} (g_j-\mathbf{m}) \|_{L^{\frac{3}{\frac12-s+r}}_x} \leq C\|g_j-\mathbf{m}\|_{H^s_x}.
\end{equation}
By Gagliardo-Nirenberg inequality and Young's inequality, we can get
\begin{equation}\label{ru152}
	\begin{split}
		\| \nabla df_j \|_{L^{\frac{3}{1+s-r}}_x} \leq & C\|\left< \partial \right>^{r-1} df_j \|^{\frac{s-2}{3-s}}_{L^{2}_x} \|\left< \partial \right>^{r-1-\frac{s}{2}} df_j \|^{\frac{5-2s}{3-s}}_{L^{\infty}_x}
		\\
		\leq & C( \|df_j \|_{H^{r-1}_x}+ \|\left< \partial \right>^{r-\frac{s}{2}-1} df_j \|_{L^{\infty}_x})
	\end{split}
\end{equation}
For $\frac{s}{2}<r \leq 3$, gathering \eqref{ru16}-\eqref{ru18} with \eqref{ru150}- \eqref{ru152}, so we have
\begin{equation}\label{ru19}
\begin{split}
	\| [ \square_{{g}_j} ,\left< \partial \right>^{r-1}] f_j \|_{L^2_x}\|d f_j \|_{H^{r-1}_x} \leq & C  \|dg_j\|_{L^\infty_x} \|d f_j \|^2_{H^{r-1}_x}
	+ C  \| g_j-\mathbf{m} \|_{H^s_x}\|d f_j \|^2_{H^{r-1}_x}
	\\
	& + C  \|g_j-\mathbf{m}\|_{H^s_x} \|\left< \partial \right>^{r-\frac{s}{2}-1} df_j \|_{L^{\infty}_x} \|d f_j \|_{H^{r-1}_x}
	\\
	\leq & C  ( \|dg_j\|_{L^\infty_x}+ \|g_j-\mathbf{m}\|_{H^s_x}+ \|\left< \partial \right>^{r-\frac{s}{2}-1} df_j \|_{L^{\infty}_x} ) \|d f_j \|^2_{H^{r-1}_x}
	\\
	&+ C  \|g_j-\mathbf{m}\|^2_{H^s_x}  \|\left< \partial \right>^{r-\frac{s}{2}-1} df_j \|_{L^{\infty}_x} .
\end{split}	
\end{equation}
By \eqref{ru15} and \eqref{ru19}, we get
\begin{equation}\label{ru20}
	\begin{split}
		\frac{d}{dt}\|d f_j \|^2_{H^{r-1}_x}
		\leq & C  (\|dg_j\|_{L^\infty_x}+ \|g_j-\mathbf{m}\|_{H^s_x}+ \|\left< \partial \right>^{r-\frac{s}{2}-1} df_j \|_{L^{\infty}_x} ) \|d f_j \|^2_{H^{r-1}_x}
		\\
		&+ C  \|g_j-\mathbf{m}\|^2_{H^s_x}  \|\left< \partial \right>^{r-\frac{s}{2}-1} df_j \|_{L^{\infty}_x} .
	\end{split}	
\end{equation}
Using Gronwall's inequality, we get
\begin{equation}\label{ru21}
	\begin{split}
		\|d f_j(t) \|^2_{H^{r-1}_x}
		\leq & C  (\|d f_j(0) \|^2_{H^{r-1}_x} + C \int^t_0 \|g_j-\mathbf{m}\|^2_{H^s_x}  \|\left< \partial \right>^{r-\frac{s}{2}-1} df_j \|_{L^{\infty}_x}d\tau)
		\\
		& \qquad \cdot\exp\{\int^t_0  \|dg_j\|_{L^\infty_x}+ \|g_j-\mathbf{m}\|_{H^s_x}+ \|\left< \partial \right>^{r-\frac{s}{2}-1} df_j \|_{L^{\infty}_x} d\tau \} .
	\end{split}	
\end{equation}
Note
\begin{equation}\label{ru22}
	\begin{split}
		\|f_j(t) \|_{L^{2}_x}
		\leq & \| f_j(0) \|_{L^2_x} + \int^t_0 \| \partial_t f_j \|_{L^2_x} d\tau.
	\end{split}	
\end{equation}
By \eqref{ru21}, \eqref{ru22} and \eqref{kz65}, if $t\in[0,T^*_{N_0}]$, then it follows
\begin{equation}\label{ru23}
	\begin{split}
		\| f_j(t) \|^2_{H^{r}_x} + \| d f_j(t) \|^2_{H^{r-1}_x}
		\leq & C  \textrm{e}^{C_*}(\|f_0 \|^2_{H^{r}_x}+\|f_1 \|^2_{H^{r-1}_x} + C_*)
		\\
		& \quad \cdot\exp\{\int^t_0   \|\left< \partial \right>^{r-\frac{s}{2}-1} df_j \|_{L^{\infty}_x} d\tau \cdot \textrm{e}^{\int^t_0   \|\left< \partial \right>^{r-\frac{s}{2}-1} df_j \|_{L^{\infty}_x} d\tau}   \} .
	\end{split}	
\end{equation}
Based on \eqref{ru070}, \eqref{ru12},  and \eqref{ru23}, following the extending method in subsection \ref{esest}, for $a\leq r-\frac{s}{2}$ we shall obtain
\begin{equation}\label{ru234}
	\begin{split}
		\| \left< \partial \right>^{a-1} P_k df_j\|_{L^2_{[0,T^*_{N_0}]} L^\infty_x}
		\leq & C2^{-\frac12\delta_{1} k} 2^{-4\delta_{1} j} ( \| f_{0} \|_{H^r_x} + \| f_{1} \|_{H^{r-1}_x}  ) \times \{ 2^{\delta_{1} (j-N_0)} \}^{\frac12}
		\\
		\leq & C2^{-\frac12\delta_{1}N_0} 2^{-\frac12\delta_{1} k} 2^{-2\delta_{1} j} ( \| f_{0} \|_{H^r_x} + \| f_{1} \|_{H^{r-1}_x}  ), \quad k \geq j,
	\end{split}
\end{equation}
and
\begin{equation}\label{ru25}
	\begin{split}
		& \| \left< \partial \right>^{a-1} P_k (df_{m+1}-df_m)\|_{L^2_{[0,T^*_{N_0}]} L^\infty_x}
		\\
		\leq & C2^{-\frac12\delta_{1} k}2^{-3\delta_{1} m}( \| f_{0} \|_{H^r_x} + \| f_{1} \|_{H^{r-1}_x}  )(1+E^3(0)) \times \{ 2^{\delta_{1} (m+1-N_0)} \}^{\frac12}
		\\
		\leq & C2^{-\frac12\delta_{1} k}2^{-2\delta_{1} m}2^{-\frac12\delta_{1}N_0}( \| f_{0} \|_{H^r_x} + \| f_{1} \|_{H^{r-1}_x}  ) (1+E^3(0)), \quad k<j.
	\end{split}
\end{equation}
By phase decomposition, we have
\begin{equation*}
	df_j= P_{\geq j} df_j+ \textstyle{\sum}^{j-1}_{k=1} \textstyle{\sum}^{j-1}_{m=k} P_k (df_{m+1}-df_m)+ \textstyle{\sum}^{j-1}_{k=1} P_k df_k.
\end{equation*}
By using \eqref{ru234} and \eqref{ru25}, we get
\begin{equation}\label{ru230}
	\begin{split}
		 \| \left< \partial \right>^{a-1} df_{j} \|_{L^2_{[0,T^*_{N_0}]} L^\infty_x}
		\leq & C( \| f_{0} \|_{H^r_x} + \| f_{1} \|_{H^{r-1}_x}  )(1+E^3(0))[\frac13(1-2^{-\delta_{1}})]^{-2}.
	\end{split}
\end{equation}
Using \eqref{ru23} and \eqref{ru230}, we get
\begin{equation}\label{ru26}
	\begin{split}
		 \|f_j\|_{L^\infty_{[0,T^*_{N_0}]} H^{r}_x}+ \|d f_j\|_{L^\infty_{[0,T^*_{N_0}]} H^{r-1}_x }
		\leq & A_* \exp ( B_* \mathrm{e}^{B_*} ),
	\end{split}	
\end{equation}
where $A_{*}=C  \textrm{e}^{C_*}(\|f_0 \|_{H^{r}_x}+\|f_1 \|_{H^{r-1}_x} + C_*)$, $B_*=C(\| f_{0} \|_{H^r_x} + \| f_{1} \|_{H^{r-1}_x})(1+E^3(0))[\frac13(1-2^{-\delta_{1}})]^{-2}$ and $C_*$ is stated in \eqref{Cstar}. Therefore, the conclusion \eqref{ru02} hold.
\end{proof}

\section{Proof of Theorem \ref{dingli2}}\label{Sec9}
In this part, our goal is to prove Theorem \ref{dingli2}. We first reduce the proof to the case of initial data with frequency support, that is, using Proposition \ref{DL2} to prove Theorem \ref{dingli2}. Then we use Proposition \ref{p1} to obtain Proposition \ref{DDL3}. Finally, we use Proposition \ref{DDL3} and develop a method to prove Proposition \ref{DDL2}.

\subsection{Proof of Theorem \ref{dingli2}}\label{keyc} To prove Theorem \ref{dingli2}, we will prove the well-posedness of solutions as a limit of a sequence. To find this sequence, let us introduce a sequence of initial data $(\bv_{0j},\rho_{0j}, h_{0j})$ with frequency support such that
\begin{equation}\label{dss0}
	(\bv_{0j}, \rho_{0j}, h_{0j})=(P_{\leq j} \bv_0, P_{\leq j} \rho_0, P_{\leq j} h_0).
\end{equation}
Above, the data $(\bv_{0}, \rho_{0}, h_{0})$ is stated as \eqref{chuzhi2} in Theorem \ref{dingli2}. Following \eqref{pw11}, we define
\begin{equation}\label{dss1}
 \bw_{0j}=\mathrm{e}^{-\rho_{0j}}\mathrm{curl}\bv_{0j}.
\end{equation}
Set
\begin{equation}\label{qu0}
	\mathbb{E}(0)=	\| \bv_{0j} \|_{H^{\frac52}}+ \| \rho_{0j} \|_{H^{\frac52}} + \| \bw_{0j} \|_{H^{\frac52}}+  \| h_{0j} \|_{H^{\frac52}}.
\end{equation}
Using \eqref{chuzhi2} and \eqref{dss0} and \eqref{dss1}, we therefore get
\begin{equation}\label{qu01}
	\mathbb{E}(0) \leq C(\bar{M}_0+ \bar{M}^2_0).
\end{equation}
Before we give a proof of Theorem \ref{dingli2}, let us now introduce Proposition \ref{DL2}.
\begin{proposition}\label{DL2}
	Let $\epsilon$ and \eqref{HEw}-\eqref{chuzhi2} be stated in Theorem \ref{dingli2}. Let $(\bv_{0j}, \rho_{0j},h_{0j}, \bw_{0j})$ be stated in \eqref{dss0} and \eqref{dss1}. For each $j\geq 1$, consider Cauchy problem \eqref{fc0} with the initial data $(\bv_{0j}, \rho_{0j}, h_{0j},\bw_{0j})$. Then for all $j \geq 1$, there exists two positive constants $\bar{T}>0$ and $\bar{M}_{1}>0$ ($\bar{T}$ and $\bar{M}_{1}$ only depending on $\bar{M}_{0}$ and $\epsilon$) such that \eqref{fc0} has a unique solution $(\bv_j,\rho_j) \in C([0,\bar{T}],H_x^{\frac52})$, $h_j \in C([0,\bar{T}],H_x^{\frac52+\epsilon})$, and $\bw_j \in C([0,\bar{T}],H_x^{\frac32+\epsilon})$. To be precise,

	$\mathrm{(1)}$ the solution $\bv_j, \rho_j, h_j$ and $\bw_j$ satisfy the energy estimates
	\begin{equation}\label{uu0}
		\|\bv_j,\rho_j\|_{L^\infty_{[0,\bar{T}]}H_x^{\frac52}}+ \|h_j\|_{L^\infty_{[0,\bar{T}]}H_x^{\frac52+\epsilon}}+ \|\bw_j\|_{L^\infty_{[0,\bar{T}]}H_x^{\frac32+\epsilon}}+ \|\bv_j,\rho_j, h_j\|_{L^\infty_{[0,\bar{T}]\times \mathbb{R}^3}} \leq \bar{M}_{1},
	\end{equation}
	and
	\begin{equation}\label{uu1}
	\|\partial_t \bv_j\|_{L^\infty_{[0,\bar{T}]}H_x^{\frac32}}+\|\partial_t \rho_j\|_{L^\infty_{[0,\bar{T}]}H_x^{\frac32}}+ \|\partial_t h_j\|_{L^\infty_{[0,\bar{T}]}H_x^{\frac32+\epsilon}}+ \|\partial_t \bw_j\|_{L^\infty_{[0,\bar{T}]}H_x^{\frac12+\epsilon}} \leq \bar{M}_{1},
	\end{equation}

	$\mathrm{(2)}$ the solution $\bv_j$, $\rho_j$, and $h_j$ satisfy the Strichartz estimate
	\begin{equation}\label{uu2}
		\|d\bv_j, d\rho_j, dh_j\|_{L^2_{[0,\bar{T}]}L_x^\infty} \leq \bar{M}_{1}.
	\end{equation}

	$\mathrm{(3)}$ for $\frac{7}{6} \leq r < \frac72$, consider the following linear wave equation
	\begin{equation}\label{uu21}
		\begin{cases}
			\square_{{g}_j} f_j=0, \qquad [0,\bar{T}]\times \mathbb{R}^3,
			\\
			(f_j,\partial_t f_j)|_{t=0}=(f_{0j},f_{1j}),
		\end{cases}
	\end{equation}
	where $(f_{0j},f_{1j})=(P_{\leq j}f_0,P_{\leq j}f_1)$ and $(f_0,f_1)\in H_x^r \times H^{r-1}_x$. Then there is a unique solution $f_j$ on $[0,\bar{T}]\times \mathbb{R}^3$. Moreover, for $a\leq r-\frac{7}{6}$, we have
	\begin{equation}\label{uu22}
		\begin{split}
			&\|\left< \partial \right>^{a-1} d{f}_j\|_{L^2_{[0,\bar{T}]} L^\infty_x}
			\leq  \bar{M}_3(\|{f}_0\|_{{H}_x^r}+ \|{f}_1 \|_{{H}_x^{r-1}}),
			\\
			&\|{f}_j\|_{L^\infty_{[0,\bar{T}]} H^{r}_x}+ \|\partial_t {f}_j\|_{L^\infty_{[0,\bar{T}]} H^{r-1}_x} \leq  \bar{M}_3(\| {f}_0\|_{H_x^r}+ \| {f}_1\|_{H_x^{r-1}}),
		\end{split}
	\end{equation}
	where $ {\bar{M}_3}$ is a universal constant depending on $C_0, c_0, \bar{M}_0$ and $\epsilon$.
\end{proposition}
Based on Proposition \ref{DL2}, we are now ready to prove Theorem \ref{dingli2}.
\medskip\begin{proof}[Proof of Theorem \ref{dingli2} by using Proposition \ref{DL2}]
	Using \eqref{uu0}-\eqref{uu1}, there is a limit for a subsequence of $\{(\bv_j, \rho_j, h_j)\}_{j \in \mathbb{Z}}$, i.e.
	\begin{equation*}
		\lim_{n\rightarrow \infty} (\bv_{j_n}, \rho_{j_n}, h_{j_n}, \bw_{j_n})=(\bv, \rho, h, \bw) \quad \text{in} \ H_x^{\frac52} \times H_x^{\frac52} \times H_x^{\frac52+\epsilon} \times H_x^{\frac32+\epsilon}.
	\end{equation*}
	Thus, $(\bv, \rho, h, \bw)$ is a solution to \eqref{fc0}-\eqref{id} and
	\begin{equation*}
		(\bv, \rho, h, \bw)|_{t=0}=(\bv_0,\rho_0,h_0, \bw_0),
	\end{equation*}
Moreover, it satisfies
	\begin{equation*}
	\|\bv,\rho\|_{L^\infty_{[0,\bar{T}]}H_x^{\frac52}}+ \|\bw\|_{L^\infty_{[0,\bar{T}]}H_x^{\frac32+\epsilon}}+ \|h\|_{L^\infty_{[0,\bar{T}]}H_x^{\frac52+\epsilon}}+ \|\bv, \rho,h\|_{L^\infty_{ [0,\bar{T}] \times \mathbb{R}^3}} \leq \bar{M}_1,
\end{equation*}
and
\begin{equation*}
	\|\partial_t \bv\|_{L^\infty_{[0,\bar{T}]}H_x^{\frac32}}+\|\partial_t\rho\|_{L^\infty_{[0,\bar{T}]}H_x^{\frac32}}+ \|\partial_t \bw\|_{L^\infty_{[0,\bar{T}]}H_x^{\frac12+\epsilon}}+ \|\partial_t h\|_{L^\infty_{[0,\bar{T}]}H_x^{\frac32+\epsilon}} \leq \bar{M}_1,
\end{equation*}	
and
\begin{equation*}
	\|d\bv, d\rho, dh\|_{L^2_{[0,\bar{T}]}L_x^\infty} \leq \bar{M}_1.
\end{equation*}
Similarly, there is a limit for a subsequence of $\{f_j\}_{j \in \mathbb{Z}}$,
\begin{equation*}
	\lim_{n\rightarrow \infty} (f_{j_n}, \partial_t f_{j_n}) =(f, f_t) \quad \text{in} \ H_x^{r} \times  H_x^{r-1}.
\end{equation*}
Using \eqref{uu21} and \eqref{uu22} and taking $j\rightarrow \infty$, the limit $f$ is a solution of the following Cauchy problem
	\begin{equation*}
	\begin{cases}
		\square_{{g}} f=0, \qquad [0,\bar{T}]\times \mathbb{R}^3,
		\\
		(f,\partial_t f)|_{t=0}=(f_{0},f_{1}).
	\end{cases}
\end{equation*}
Furthermore, by \eqref{uu22}, $f$ has the following estimates
\begin{equation*}
	\begin{split}
		&\|\left< \partial \right>^{a-1} d{f}\|_{L^2_{[0,\bar{T}]} L^\infty_x}
		\leq  C_{\bar{M}_0}(\|{f}_0\|_{{H}_x^r}+ \|{f}_1 \|_{{H}_x^{r-1}}),
		\\
		&\|{f}\|_{L^\infty_{[0,\bar{T}]} H^{r}_x}+ \|\partial_t {f}\|_{L^\infty_{[0,\bar{T}]} H^{r-1}_x} \leq  C_{\bar{M}_0}(\| {f}_0\|_{H_x^r}+ \| {f}_1\|_{H_x^{r-1}}).
	\end{split}
\end{equation*}
Using a similar idea in Section \ref{Sub}, we can also prove the continuous dependence. So we have finished the proof of Theorem \ref{dingli2}.
\end{proof}
It remains for us to prove Proposition \ref{DL2}. Our idea is to prove Proposition \ref{DL2} via small solutions, and using scaling technique and Strichartz estimates to obtain the solutions satisfying \eqref{uu0}-\eqref{uu2}. We next turn to introduce Proposition \ref{DL3}, which is about small solutions.
\subsection{Proposition \ref{DL3} for small data}
Our object is to prove Proposition \ref{DL2} via small solutions, and using scaling technique and Strichartz estimates to obtain the solutions satisfying \eqref{uu0}-\eqref{uu2}. Let us now state Proposition \ref{DL3} considering small data.
\begin{proposition}\label{DL3}
	Let $\epsilon_2$ and $\epsilon_3$ be stated in \eqref{a0}. Let $0<\epsilon<\frac{1}{9}$. For each small, smooth initial data $(\bv_0, \rho_0, h_0, \bw_0)$ which satisfies
	\begin{equation}\label{P30}
		\begin{split}
			&\|(\bv_0, \rho_0) \|_{H^{\frac52}} + \| \bw_0\|_{H^{2+\epsilon}}+ \| h_0\|_{H^{3+\epsilon}}  \leq \epsilon_3,
		\end{split}
	\end{equation}
	there exists a unique smooth solution $(\bv, \rho, h, \bw)$ to \eqref{fc1} on $[-2,2] \times \mathbb{R}^3$ satisfying
	\begin{equation}\label{P31}
		\|(\bv, \rho)\|_{L^\infty_{[-2,2]}H_x^{\frac52}}+\|\bw\|_{L^\infty_{[-2,2]}H_x^{2+\epsilon}}+\|h\|_{L^\infty_{[-2,2]}H_x^{3+\epsilon}} \leq \epsilon_2.
	\end{equation}
	Furthermore, the solution satisfies the properties

	$\mathrm{(1)}$ dispersive estimate for $\bv$, $\rho$, and $h$
	\begin{equation}\label{P32}
		\|d \bv, d \rho, dh\|_{L^2_{[-2,2]} C^a_{x}} \leq \epsilon_2, \quad a\in [0,\frac12).
	\end{equation}

	$\mathrm{(2)}$ Let $f$ satisfy equation \eqref{linear}. For each $1 \leq r \leq \frac52+1$, the Cauchy problem \eqref{linear} is well-posed in $H_x^r \times H_x^{r-1}$, and the following estimate holds:
	\begin{equation}\label{P33}
		\|\left< \partial \right>^k f\|_{L^2_{[-2,2]} L^\infty_x} \lesssim  \| f_0\|_{H_x^r}+ \| f_1\|_{H_x^{r-1}},\quad  k<r-1,
	\end{equation}
and the similar estimates hold if we replace $\left< \partial \right>^k$ by $\left< \partial \right>^{k-1} d$.
\end{proposition}
\medskip\begin{proof}[Proof of Proposition \ref{DL3}]
Taking $s=\frac52, s_0=2+\epsilon$ in Proposition \ref{p1}, then Proposition \ref{DL3} holds.
\end{proof}
Based on Proposition \ref{DL3}, we are now ready to give a proof of Proposition \ref{DL2}.
\subsection{Proof of Proposition \ref{DL2}}
\medskip\begin{proof}[Proof of Proposition \ref{DL2} by using Proposition \ref{DL3}]
	In Proposition \ref{DL3}, the corresponding statement is considering the small, supported data. While, the initial data is large in Proposition \ref{DL2}. Firstly, using a scaling method, we can reduce the initial sequence $\{(\bv_{0j}, \rho_{0j}, \bw_{0j}, h_{0j})\}$ to be small. Taking scaling
\begin{equation*}
		\begin{split}
			& \widetilde{\bv}_{0j}={\bv}_{0j}(Tt,Tx), \quad \widetilde{\rho}_{0j}=\rho_{0j}(Tt,Tx),  \quad \widetilde{h}_{0j}={h}_{0j}(Tt,Tx),
		\end{split}
	\end{equation*}
and setting
\begin{equation*}
	\widetilde{\bw}_{0j}=\mathrm{e}^{-\widetilde{\rho}_{0j}} \mathrm{curl} \widetilde{\bv}_{0j},
\end{equation*}
we then have
	\begin{equation}\label{qu02}
		\begin{split}
			& \| \widetilde{\bv}_{0j} \|_{\dot{H}^{\frac52}} \leq T \| {\bv}_{0j} \|_{\dot{H}^{\frac52}} \leq T\cdot \mathbb{E}(0),
			\\
			&
			\| \widetilde{\rho}_{0j}\|_{\dot{H}^{\frac52}} \leq T \| {\rho}_{0j} \|_{\dot{H}^{\frac52}} \leq T \cdot \mathbb{E}(0),
			\\
			& \| \widetilde{h}_{0j}\|_{\dot{H}^{3+\epsilon}} \leq T^{\frac32+\epsilon} \| {h}_{0j}\|_{\dot{H}^{3+\epsilon}} \leq T^{\frac{3+\epsilon}{2}}2^{\frac{j}{2}}\cdot T^{\frac{\epsilon}{2}} \mathbb{E}(0) ,
			\\
			& \| \widetilde{\bw}_{0j}\|_{\dot{H}^{2+\epsilon}}\leq T^{\frac32+\epsilon}\|{\bw}_{0j}\|_{\dot{H}^{2+\epsilon}} \leq T^{\frac{3+\epsilon}{2}}2^{\frac{j}{2}}\cdot T^{\frac{\epsilon}{2}} \mathbb{E}(0).
		\end{split}
	\end{equation}
Set
	\begin{equation}\label{qu03}
	{T}_j:=2^{-\frac{1}{3+\epsilon}j} [\mathbb{E}(0)]^{-1}.
\end{equation}
Hence, if we replace $T$ with $T_j$ in \eqref{qu02}, then we have
	\begin{equation*}
		\begin{split}
			\| \widetilde{\bv}_{0j} \|_{\dot{H}^{\frac52}}+\| \widetilde{\rho}_{0j}\|_{\dot{H}^{\frac52}}+\| \widetilde{h}_{0j}\|_{\dot{H}^{3+\epsilon}}+\| \widetilde{\bw}_{0j}\|_{\dot{H}^{2+\epsilon}}  \leq  2^{-\frac{\epsilon}{2(3+\epsilon)}j}(1+\mathbb{E}(0)).
		\end{split}
	\end{equation*}
By using \eqref{HEw} and \eqref{dss0}, we also have
\begin{equation}\label{qu00}
 \| \widetilde{\bv}_{0j},\widetilde{\rho}_{0j},\widetilde{h}_{0j}\|_{L^\infty_x} \leq \| {\bv}_{0j},{\rho}_{0j},{h}_{0j}\|_{L^\infty_x} \leq C_0.
\end{equation}
	Here, for a small parameter $\epsilon_3$, we can choose ${N}_1=N_1(\epsilon, \bar{M}_0)$ such that
	\begin{equation}\label{qu04}
	\begin{split}
	2^{-\frac{\epsilon N_1}{2(3+\epsilon)}}(1+\mathbb{E}(0)) \ll \epsilon_3,
	\\
	(1+\mathbb{E}^2_*)2^{-\frac{\epsilon}{40}N_1}(1+\mathbb{E}^2(0))^{-1} \leq 1.
	\end{split}	
	\end{equation}
Above, $\mathbb{E}_*$ is denoted by
	\begin{equation}\label{qu0q}
	\mathbb{E}_*=C(\mathbb{E}(0)+\mathbb{E}^2(0) ) \exp\{C(1+\mathbb{E}^2(0))[\frac13(1-2^{-\frac{\epsilon}{40}})]^{-2} \}.
\end{equation}	
As a result, for $j \geq {N}_{1}$, we have
	\begin{equation}\label{qu05}
		\begin{split}
			\| \widetilde{\bv}_{0j} \|_{\dot{H}^{\frac52}}+\| \widetilde{\rho}_{0j}\|_{\dot{H}^{\frac52}}+\| \widetilde{h}_{0j}\|_{\dot{H}^{3+\epsilon}}+\| \widetilde{\bw}_{0j}\|_{\dot{H}^{2+\epsilon}}  \leq  \epsilon_3.
		\end{split}
	\end{equation}
For \eqref{qu05} is about the homogeneous norm, we need to use physical localization to get the same bound for in-homogeneous norm. Utilize the similar standard physical localization(page 44, Step 2). For the propagation speed of \eqref{fc1} is finite, we set $c$ be the largest speed of \eqref{fc1}. Set $\chi$ be a smooth function supported in $B(0,c+2)$, and which equals $1$ in $B(0,c+1)$. For given $\by \in \mathbb{R}^3$, we define the localized initial data for the velocity and density near $\by$:
	\begin{equation}\label{qu06}
		\begin{split}
			\bar{\bv}_{0j}(\bx)=&\chi(\bx-\by)\left( \widetilde{\bv}_{0j}(\bx)- \widetilde{\bv}_{0j}(\by)\right),
			\\
			\bar{\rho}_{0j}(\bx)=&\chi(\bx-\by)\left( \widetilde{\rho}_{0j}(\bx)-\widetilde{\rho}_{0j}(\by)\right),
			\\
			\bar{h}_{0j}(\bx)=&\chi(\bx-\by) ( \widetilde{h}_{0j}(\bx)-\widetilde{h}_{0j}(\by) ).
		\end{split}
	\end{equation}
Following \eqref{pw11}, we shall set
	\begin{equation}\label{qu07}
		\bar{\bw}_{0j}=\mathrm{e}^{-\bar{\rho}_{0j}}\mathrm{curl}\bar{\bv}_{0j}.
	\end{equation}
Based on \eqref{qu06}, \eqref{qu07} and \eqref{qu05}, it follows that
	\begin{equation}\label{ss4}
		\begin{split}
			\| \bar{\bv}_{0j} \|_{{H}^{\frac52}}+\| \bar{\rho}_{0j}\|_{{H}^{\frac52}}+\| \bar{h}_{0j} \|_{{H}^{3+\epsilon}}+\| \bar{\bw}_{0j} \|_{{H}^{2+\epsilon}}  \leq  \epsilon_3.
		\end{split}
	\end{equation}
	For each fixed $j$, using Proposition \ref{DL3}, there is a smooth solution $(\bar{\bv}_j, \bar{\rho}_j, \bar{h}_j, \bar{\bw}_j)$ on $[-2,2]\times \mathbb{R}^3$ satisfying
\begin{equation}\label{ss3}
\begin{cases}
\square_{\widetilde{g}_j} v^i=-\mathrm{e}^{\rho+\widetilde{\rho}_{0j}(y)}\widetilde{c}_s^2 \mathrm{curl} \bw^i+\widetilde{Q}^i,
\\
\square_{\widetilde{g}_j} {\rho}=-\frac{1}{\gamma}\widetilde{c}_s^2\Delta h+ \widetilde{D},
\\
\square_{\widetilde{g}_j} h=\widetilde{c}_s^2\Delta h + \widetilde{E},
\\
\widetilde{\mathbf{T}} h =0,
\\
\widetilde{\mathbf{T}} w^i=w^a \partial_a v^i+\bar{\rho}^{\gamma-1} \mathrm{e}^{(h+\widetilde{h}_{0j}(y))+(\gamma-2)(\rho+\widetilde{\rho}_{0j}(y))}\epsilon^{iab}\partial_a \rho \partial_b h,
\\
(\bv, \rho, h, \bw)|_{t=0}=(\bar{\bv}_{0j},\bar{\rho}_{0j},\bar{h}_{0j},\bar{\bw}_{0j} ).
\end{cases}
\end{equation}
Above, the quantities $\widetilde{c}^2_s$ and $\widetilde{{g}}$ are given by
\begin{equation}\label{DDEt}
\begin{split}
& \widetilde{c}^2_s= \frac{dp}{d{\rho}}({\rho}+\widetilde{\rho}_{0j}(y), {h}+\widetilde{h}_{0j}(y)),
\\
& \widetilde{g}_j=g({\bv}+\widetilde{\bv}_{0j}(y), {\rho}+\widetilde{\rho}_{0j}(y), {h}+\widetilde{h}_{0j}(y)),
\\
& \widetilde{\mathbf{T}}=\partial_t+ ({\bv}+\widetilde{\bv}_{0j}(y))\cdot \nabla,
\end{split}
\end{equation}
and
\begin{equation*}
  \begin{split}
  \widetilde{Q}^i=& \widetilde{Q}^{i\alpha\beta} \partial_\alpha \rho \partial_\beta h+\widetilde{Q}^{i\alpha\beta }_{1j} \partial_\alpha \rho \partial_\beta v^j+\widetilde{Q}^{i\alpha\beta}_{2j} \partial_\alpha h \partial_\beta v^j +\widetilde{Q}^{\alpha\beta}_{3j} \partial_\alpha v^i \partial_\beta v^j,
  \\
  \widetilde{D}=& \widetilde{D}_1^{\alpha\beta} \partial_\alpha \rho \partial_\beta h+\widetilde{D}_2^{\alpha\beta} \partial_\alpha h \partial_\beta h+\widetilde{D}_3^{\alpha\beta} \partial_\alpha \rho \partial_\beta \rho
  \\
  &+\widetilde{D}^{\alpha\beta }_{1j} \partial_\alpha \rho \partial_\beta v^j+\widetilde{D}^{\alpha\beta}_{2j} \partial_\alpha h \partial_\beta v^j+\widetilde{D}^{i\alpha\beta}_{3j} \partial_\alpha v_i \partial_\beta v^j,
  \\
  \widetilde{E}=& \widetilde{E}_1^{\alpha\beta} \partial_\alpha h \partial_\beta \rho+\widetilde{E}^{\alpha\beta }_{2} \partial_\alpha h \partial_\beta h+\widetilde{E}^{\alpha\beta}_{3j} \partial_\alpha h \partial_\beta v^j,
\end{split}
\end{equation*}
and $\widetilde{Q}^{i\alpha\beta}$, $\widetilde{Q}^{i\alpha\beta }_{1j}$, $\widetilde{Q}^{i\alpha\beta}_{2j}$,  $\widetilde{Q}^{\alpha\beta}_{3j}$, $\widetilde{D}_1^{\alpha\beta}$, $\widetilde{D}_2^{\alpha\beta}$, $\widetilde{D}_3^{\alpha\beta}$, $\widetilde{D}^{\alpha\beta }_{1j}$, $\widetilde{D}^{\alpha\beta}_{2j}$, $\widetilde{D}^{i\alpha\beta}_{3j}$, $\widetilde{E}_1^{\alpha\beta}$, $\widetilde{E}^{\alpha\beta }_{2}$, $\widetilde{E}^{\alpha\beta}_{3j}$ have the same forms with $Q^{i\alpha\beta}$, $Q^{i\alpha\beta }_{1j}$, $Q^{i\alpha\beta}_{2j}$,  $Q^{\alpha\beta}_{3j}$, $D_1^{\alpha\beta}$, $D_2^{\alpha\beta}$, $D_3^{\alpha\beta}$, $D^{\alpha\beta }_{1j}$, $D^{\alpha\beta}_{2j}$, $D^{i\alpha\beta}_{3j}$, $E_1^{\alpha\beta}$, $E^{\alpha\beta }_{2}$, $E^{\alpha\beta}_{3j}$ by replacing $({\bv}, {\rho}, {h})$ to $({\bv}+\widetilde{\bv}_{0j}(y), {\rho}+\widetilde{\rho}_{0j}(y), {h}+\widetilde{h}_{0j}(y))$. Seeing \eqref{dvc} and \eqref{etad}, let us set
	\begin{equation*}
		\bar{\bv}_{+j}=\bar{\bv}_j-\bar{\bv}_{-j}, \quad \Delta \bar{\bv}_{-j}=-\mathrm{e}^{ \bar{\rho}_j} \mathrm{curl}\bar{\bw}_j.
	\end{equation*}
	Using Proposition \ref{DL3} again, we could find
	\begin{equation}\label{see0}
		\|\bar{\bv}_j\|_{L^\infty_{[-2,2]}H_x^{\frac52}}+ \|\bar{\rho}_j\|_{L^\infty_{[-2,2]} H_x^{\frac52}}+ \| \bar{\bw}_j \|_{L^\infty_{[-2,2]} H_x^{2+\epsilon}}+ \| \bar{h}_j \|_{L^\infty_{[-2,2]} H_x^{3+\epsilon}} \leq \epsilon_2,
	\end{equation}
	and
	\begin{equation}\label{see1}
		\|d\bar{\bv}_j, d\bar{\rho}_j,d\bar{h}_j\|_{L^2_{[-2,2]}C^{a}_x} \leq \epsilon_2, \quad a \in [0,\frac12).
	\end{equation}
Furthermore, the linear equation
	\begin{equation}\label{see2}
		\begin{cases}
			&\square_{ \widetilde{{g}}_j} f=0, \qquad [-2,2]\times \mathbb{R}^3,
			\\
			&{f}(t_0,x)={f}_0, \ \partial_t {f}(t_0,x)={f}_1,\quad t_0 \in [-2,2],
		\end{cases}
	\end{equation}
	admits a solution ${f} \in C([-2,2],H_x^r)\times C^1([-2,2],H_x^{r-1})$. Moreover, for $k<r-1$, we have
\begin{equation}\label{so31}
\begin{split}
		\|\left< \partial \right>^k{f}\|_{L^2_{[-2,2]} L^\infty_x} \leq  & C(\| {f}_0\|_{H_x^r}+ \| {f}_1\|_{H_x^{r-1}}),
\end{split}
\end{equation}
and the similar estimates hold when we replace $\left< \partial \right>^k$ by $\left< \partial \right>^{k-1}d$. Here $\widetilde{g}=g(\bar{\bv}_j+\widetilde{\bv}_{0j}(y),\bar{\rho}_j+\widetilde{\rho}_{0j}(y),\bar{h}_j+\widetilde{h}_{0j}(y))$(citing \eqref{DDEt}). Note \eqref{ss3}. Then the function $(\bar{\bv}_j+\widetilde{\bv}_{0j}(\by), \bar{\rho}_j+\widetilde{\rho}_{0j}(\by), \bar{h}_j+\widetilde{h}_{0j}(\by), \mathrm{e}^{\widetilde{\rho}_{0j}(\by)}\bar{\bw}_j)$ is also a solution of the following system
\begin{equation}\label{qu08}
	\begin{cases}
		&\square_{\widetilde{g}}(\bar{v}^i_j+\widetilde{v}^i_{0j}(\by))=-\mathrm{e}^{\rho+\widetilde{\rho}_{0j}(y)}\widetilde{c}_s^2 \mathrm{curl} (\mathrm{e}^{\widetilde{\rho}_{0j}(\by)} \bw^i)+\widetilde{Q}^i,
		\\
		&\square_{\widetilde{g}} (\bar{\rho}_j+\widetilde{\rho}_{0j}(\by))=-\frac{1}{\gamma}\widetilde{c}_s^2\Delta (\bar{h}_j+\widetilde{h}_{0j}(\by))+ \widetilde{D},
		\\
		&\square_{\widetilde{g}} (\bar{h}_j+\widetilde{h}_{0j}(\by))=\widetilde{c}_s^2\Delta (\bar{h}_j+\widetilde{h}_{0j}(\by)) + \widetilde{E},
		\\
		&\widetilde{\mathbf{T}} (\bar{h}_j+\widetilde{h}_{0j}(\by)) =0,
		\\
		&\widetilde{\mathbf{T}} (\mathrm{e}^{\widetilde{\rho}_{0j}(\by)} w^i)=\widetilde{\rho}_{0j}(\by)\widetilde{w}^a \partial_a (\bar{v}^i_j+\widetilde{v}^i_{0j}(\by))
		\\
		&\qquad \qquad \qquad \quad +\bar{\rho}^{\gamma-1} \mathrm{e}^{(h+\widetilde{h}_{0j}(y))+(\gamma-2)(\rho+\widetilde{\rho}_{0j}(y))}\epsilon^{iab}\partial_a (\bar{\rho}_j+\widetilde{\rho}_{0j}(\by)) \partial_b (\bar{h}_j+\widetilde{h}_{0j}(\by)),
		\\
		&(\bar{\bv}_j+\widetilde{\bv}_{0j}(\by), \bar{\rho}_j+\widetilde{\rho}_{0j}(\by), \bar{h}_j+\widetilde{h}_{0j}(\by), \mathrm{e}^{\widetilde{\rho}_{0j}(\by)}\bar{\bw}_j))|_{t=0}
		\\
		& \qquad \qquad \qquad \qquad \qquad \qquad \qquad \qquad \qquad =(\widetilde{\bv}_{0j},\widetilde{\rho}_{0j},\widetilde{h}_{0j},\widetilde{\bw}_{0j} ) \ \text{in} \ B(\by,c+1).
	\end{cases}
\end{equation}
Consider the restrictions, for $\by\in \mathbb{R}^3$,
	\begin{equation}\label{qu09}
		\left( \bar{\bv}_j+\widetilde{\bv}_{0j} (\by) \right)|_{\mathrm{K}^y},
		\quad (\bar{\rho}_j+\widetilde{\rho}_{0j} (\by) )|_{\mathrm{K}^y},
		\\
		\quad (\bar{h}_j+\widetilde{h}_{0j} (\by) )|_{\mathrm{K}^y},
		\\
		\quad \mathrm{e}^{\widetilde{\rho}_{0j}(\by)} \bar{\bw}_j|_{\mathrm{K}^y},
	\end{equation}
	where $\mathrm{K}^y=\left\{ (t,\bx): ct+|\bx-\by| \leq c+1, |t| <1 \right\}$, then the restrictions \eqref{qu08} solve \eqref{CEE}-\eqref{id} on $\mathrm{K}^y$. By finite speed of propagation, a smooth solution $(\widetilde{\bv}_{j}, \widetilde{\rho}_j, \widetilde{h}_j,  \widetilde{\bw}_j)$ solves \eqref{fc1} in $[-1,1] \times \mathbb{R}^3$, where $\widetilde{\bv}_j, \widetilde{\rho}_j, \widetilde{h}_j$ and $\widetilde{\bw}_j$ is denoted by
	\begin{equation}\label{qu10}
		\begin{split}
			\widetilde{\bv}_j(t,\bx) \  &= \bar{\bv}_j+\widetilde{\bv}_{0j}(\by), \ \ \ (t,\bx) \in \mathrm{K}^y,
			\\
			\widetilde{\rho}_j(t,\bx) \  &=\bar{\rho}_j+ \widetilde{\rho}_{0j}(\by), \ \ \ (t,\bx) \in \mathrm{K}^y,
			\\
			\widetilde{h}_j(t,\bx) \  &=\bar{h}_j+ \widetilde{h}_{0j}(\by), \ \ \ (t,\bx) \in \mathrm{K}^y,
			\\
			\widetilde{\bw}_j(t,\bx)  &=\bar{\bw}_j, \ \ \qquad \qquad  (t,x) \in \mathrm{K}^y.
		\end{split}
	\end{equation}
Therefore, if we set
\begin{equation}\label{qu11}
	\begin{split}
		\widetilde{\bv}_j(t,\bx)  &=\textstyle{\sum_{\by \in 3^{-\frac12} \mathbb{Z}^3 }} \psi(\bx-\by) (\bar{\bv}_j+\widetilde{\bv}_{0j}(\by)),
		\\
		\widetilde{\rho}_j(t,\bx)  &=\textstyle{\sum_{\by \in 3^{-\frac12} \mathbb{Z}^3 }} \psi(\bx-\by) ( \bar{\rho}_j+ \widetilde{\rho}_{0j}(\by)),
		\\
		\widetilde{h}_j(t,\bx)  &=\textstyle{\sum_{\by \in 3^{-\frac12} \mathbb{Z}^3 }} \psi(\bx-\by) ( \bar{h}_j+ \widetilde{h}_{0j}(\by)),
		\\
		\widetilde{\bw}_j(t,\bx)  &=\mathrm{e}^{\widetilde{\rho}_j}\mathrm{curl}\widetilde{\bv}_j,
	\end{split}
\end{equation}
then $(\widetilde{\bv}_j,\widetilde{\rho}_j,\widetilde{h}_j,\widetilde{\bw}_j)$ is a smooth solution of \eqref{fc1} on $[-1,1]\times \mathbb{R}^3$ with the initial data $(\widetilde{\bv}_j, \widetilde{\rho}_j, \widetilde{h}_j,\widetilde{\bw}_j )|_{t=0}=(\widetilde{\bv}_{0j}, \widetilde{\rho}_{0j},\widetilde{h}_{0j}, \widetilde{\bw}_{0j})$, where $\psi$ is supported in the unit ball such that
\begin{equation*}
	\textstyle{\sum_{\by \in 3^{-\frac12} \mathbb{Z}^3 }} \psi(\bx-\by)=1.
\end{equation*}
	Therefore, the function $(\widetilde{\bv}_j, \widetilde{\rho}_j, \widetilde{h}_j, \widetilde{\bw}_j)$ defined in \eqref{qu11} is the solution of \eqref{fc1} according to the uniqueness of solutions, i.e. Corollary \ref{cor}. By \eqref{see0} and \eqref{see1}, we can obtain that $(\widetilde{\bv}_j, \widetilde{\rho}_j, \widetilde{h}_j, \widetilde{\bw}_j)$ satisfies \eqref{P30} and \eqref{P31}.
	%Moreover, by \eqref{P33} in Proposition \ref{DL3} and \eqref{fgh} in Proposition \ref{r6}, for $t \in [-2,2]$, we have

On the other hand, we recall that the initial data $(\widetilde{\bv}_{0j}, \widetilde{\rho}_{0j},\widetilde{h}_{0j}, \widetilde{\bw}_{0j})$ is a scaling of $(\bv_{0j}, \rho_{0j}, h_{0j}, \bw_{0j})$. As a scaling-invariant system \eqref{fc1}, then
\begin{equation*}
\begin{split}
  ({\bv}_{j}, {\rho}_{j},{h}_{j}, {\bw}_{j})= (\widetilde{\bv}_{j}, \widetilde{\rho}_{j},
   \widetilde{h}_{j}, \widetilde{\bw}_{j})((T_j)^{-1}t,(T_j)^{-1}x))
\end{split}
\end{equation*}
is also a solution of \eqref{fc1} on the time interval $[0,{T}_j]$(${T}_j$ is stated in \eqref{qu03}). Moreover, it matches the initial data
\begin{equation*}
  ({\bv}_{j}, {\rho}_{j}, {h}_{j}, {\bw}_{j})|_{t=0}=({\bv}_{0j}, {\rho}_{0j}, {h}_{0j}, {\bw}_{0j}).
\end{equation*}	
We still expect the behaviour of a linear wave equation endowed with $g_j=g(\bv_j,\rho_j,h_j)$. So let us consider the following linear wave equation
\begin{equation}\label{qu1}
	\begin{cases}
		\square_{{g}_j} f=0, \quad [0,T_j]\times \mathbb{R}^3,
		\\
		(f,f_t)|_{t=0}=(f_0,f_1)\in H_x^r \times H^{r-1}_x, \quad 1\leq r \leq \frac52+1.
	\end{cases}
\end{equation}
Following the proof of \eqref{po10} for System \eqref{po3}, then we can conclude that
\begin{equation}\label{qu9}
	\begin{split}
		\|\left< \partial \right>^{k-1} d{f}\|_{L^2_{[0,T_j]} L^\infty_x}
		\leq & CT_j^{r-1-k}(\|{f}_0\|_{\dot{H}_x^r}+ \| {f}_1 \|_{\dot{H}_x^{r-1}})
		\\
		\leq & C(\|{f}_0\|_{{H}_x^r}+ \| {f}_1 \|_{{H}_x^{r-1}}).
	\end{split}
\end{equation}
Above, we also use $k<r-1$ and $T_j=2^{-\frac{j}{3+\epsilon}}$. Following the proof of \eqref{po13} for System \eqref{po3}, then we obtain
\begin{equation}\label{bu1}
	\begin{split}
		\|{f}\|_{L^\infty_{[0,T_j]} H^{r}_x}+ \|\partial_t {f}\|_{L^\infty_{[0,T_j]} H^{r-1}_x} \leq  C(\| {f}_0\|_{H_x^r}+ \| {f}_1\|_{H_x^{r-1}}).
	\end{split}
\end{equation}
From \eqref{qu00}, \eqref{see0}, and \eqref{see1}, we shall prove
\begin{equation}\label{qu12}
	\begin{split}
 & \|d\widetilde{\bv}_j\|_{L^2_{[-2,2]}C_x^{a}}+\| d\widetilde{\rho}_j\|_{L^2_{[-2,2]}C_x^{a}}+\| d\widetilde{h}_j\|_{L^2_{[-2,2]}C_x^{a}}
 \\
 \leq  & C(\|d\bar{\bv}_j, d\bar{\rho}_j,d\bar{h}_j\|_{L^2_{[-2,2]}C_x^{a}}+ \|\widetilde{\bv}_{0j}, \widetilde{\rho}_{0j}, \widetilde{h}_{0j}\|_{L^\infty})
 \\
  \leq & C(1+ \mathbb{E}(0)), \qquad a\in[0,\frac12).
\end{split}
\end{equation}
and
\begin{equation}\label{qu13}
	\begin{split}
		& \|(\widetilde{\bv}_j, \widetilde{\rho}_j)\|_{L^\infty_{[-2,2]}H_x^{\frac52}}
		+\|\widetilde{h}_j\|_{L^\infty_{[-2,2]}H_x^{\frac52+}}+\|\widetilde{\bw}_j\|_{L^\infty_{[-2,2]}H_x^{\frac32+}}
		\\
		\leq & C(\|(\bar{\bv}_j, \bar{\rho}_j)\|_{L^\infty_{[-2,2]}H_x^{\frac52}}
		+\|\bar{h}_j\|_{L^\infty_{[-2,2]}H_x^{\frac52+}}+\|\bar{\bw}_j\|_{L^\infty_{[-2,2]}H_x^{\frac32+}} + \|\widetilde{\bv}_{0j}, \widetilde{\rho}_{0j}, \widetilde{h}_{0j}\|_{L^\infty} )
		\\
		\leq & C(1+ \mathbb{E}(0)).
	\end{split}
\end{equation}
By using \eqref{qu12} and by changing of coordinates $(t,\bx)\rightarrow ((T_j)^{-1}t, (T_j)^{-1}\bx)$, we have
\begin{equation}\label{qu15}
	\begin{split}
		\|d{\bv}_j, d{\rho}_j, dh_j\|_{L^1_{[0,T_{j}]}L^\infty_x}
		\leq & \|d\widetilde{\bv}_j, d\widetilde{\rho}_j, d\widetilde{h}_j\|_{L^1_{[0,1]}L^\infty_x}
		\\
		\leq & \|d\widetilde{\bv}_j, d\widetilde{\rho}_j, d\widetilde{h}_j\|_{L^2_{[0,1]}L^\infty_x}
		\\
		= & C(1+ \mathbb{E}(0)).
	\end{split}
\end{equation}
By using Theorem \ref{vve} and \eqref{qu15}, we obtain
\begin{equation}\label{qu16}
	\begin{split}
		 \mathbb{E}({T}_j ) = &	\|{\bv}_j, {\rho}_j\|_{L^\infty_{[0,T_{j}]}H^{\frac52}_x} +\|{h}_j, \|_{L^\infty_{[0,T_{j}]}H^{\frac52+\epsilon}_x}+\|{\bw}_j \|_{L^\infty_{[0,T_{j}]}H^{\frac32+\epsilon}_x}
		 \\ \leq & C(\mathbb{E}(0)+\mathbb{E}^2(0))\mathrm{e}^{C(1+ \mathbb{E}(0)) e^{C(1+ \mathbb{E}(0))}} .
	\end{split}
\end{equation}
The estimates \eqref{qu15} and \eqref{qu16} are about the solutions of \eqref{fc0}.
In \eqref{qu16} and \eqref{qu15}, ${T}_j$ depends on $j$.  So our next goal is to find a uniform time-interval and also the energy is uniform bounded.
\subsubsection{Strichartz estimates on a short time-interval}\label{keyd}
Considering ${T}_j=2^{-\frac{1}{3+\epsilon}j}$, then we discuss some good properties of the solutions $(\bv_j,\rho_j,h_j,\bw_j)$ in high frequency and low frequency.

\textit{Case 1: High frequency}. By space-time scaling $(t,\bx)\rightarrow ((T_j)^{-1}t,(T_j)^{-1}\bx)$, we get
\begin{equation}\label{qu17}
	\begin{split}
		& \| d \bv_j, d \rho_j, dh_j \|_{L^2_{[0,T_j]}C^a_x} \leq  \{ T_j \}^{-(\frac12+a)}  \|d \widetilde{\bv}_j, d \widetilde{\rho}_j, d \widetilde{h}_j \|_{L^2_{[0,1]}C^a_x}, \quad a \in [0,\frac12).
	\end{split}
\end{equation}
Using \eqref{qu12} again, \eqref{qu17} yields
\begin{equation}\label{qu18}
	\begin{split}
		& \| d \bv_j, d \rho_j, dh_j \|_{L^2_{[0,T_j]}C^a_x} \leq C \{ T_j \}^{-(\frac12+a)}  (1+\mathbb{E}(0)).
	\end{split}
\end{equation}
Using H\"older's inequality, it follows
\begin{equation}\label{qu19}
	\begin{split}
		& \|d \bv_j, d \rho_j, dh_j \|_{L^1_{[0,T_j]}C^a_x} \leq (T_j)^{\frac12}\| d \bv_j, d \rho_j, d h_j \|_{L^2_{[0,T_j]}C^a_x}
		\leq C \{ T_j \}^{-a} (1+\mathbb{E}(0)).
	\end{split}
\end{equation}
For $k \geq j$, by Bernstein inequality and \eqref{qu18}-\eqref{qu19}, we have
\begin{equation}\label{qu20}
	\begin{split}
		\| P_{k} d \bv_j, P_{k} d \rho_j, P_{k} d h_j \|_{L^2_{[0,T_j]}L^\infty_x}
		\leq & 2^{-ka}\| P_k d \bv_j, P_k d \rho_j \|_{L^2_{[0,T_j]}C^a_x}
		\\
		\leq & C2^{-ka} \| d \bv_j, d \rho_j,d h_j \|_{L^2_{[0,T_j]}C^a_x}
		\\
		\leq &  C2^{-ka}  \{ T_j \}^{-(\frac12+a)}  (1+\mathbb{E}(0)),
	\end{split}
\end{equation}
and
\begin{equation}\label{qu21}
	\begin{split}
		\| P_{k} d \bv_j, P_{k} d \rho_j, P_{k} d h_j  \|_{L^1_{[0,T_j]}L^\infty_x}
		\leq &  C2^{-ka}  \{ T_j \}^{-a}  (1+\mathbb{E}(0)).
	\end{split}
\end{equation}
Taking $a=\frac12-\frac{\epsilon}{40}$ in \eqref{qu20}-\eqref{qu21}, and using \eqref{qu03}, so we obtain
\begin{equation}\label{qu22}
	\begin{split}
		\| P_{k} d \bv_j, P_{k} d \rho_j, P_{k} d h_j \|_{L^2_{[0,T_j]}L^\infty_x}
		\leq & C2^{-(\frac12-\frac{\epsilon}{40})k}  [2^{\frac{1}{3+\epsilon}j}E(0)]^{1-\frac{\epsilon}{40}}(1+\mathbb{E}(0))
		\\
		\leq & C 2^{-\frac{\epsilon}{40}k} 2^{-(\frac12-\frac{\epsilon}{20})j} 2^{\frac{1-\frac{\epsilon}{40}}{3+\epsilon}j}  (1+\mathbb{E}^3(0))
		\\
		\leq & C  (1+\mathbb{E}^3(0)) \cdot 2^{-\frac{\epsilon}{40}k}  2^{-\frac{\epsilon}{10}j} 2^{-\frac{1}{2(3+\epsilon)}j},
	\end{split}
\end{equation}
and
\begin{equation}\label{qu23}
	\begin{split}
		\| P_{k} d \bv_j, P_{k} d \rho_j, P_{k} d h_j \|_{L^1_{[0,T_j]}L^\infty_x}
		\leq & C  (1+\mathbb{E}^2(0)) \cdot 2^{-\frac{\epsilon}{40}k}  2^{-\frac{\epsilon}{10}j} 2^{-\frac{1}{3+\epsilon}j}, \quad k \geq j.
	\end{split}
\end{equation}
Summing \eqref{qu23} over $k$ from $j$ to infinity, we get
\begin{equation}\label{qu24}
	\begin{split}
		\| P_{\geq j} d \bv_j, P_{\geq j} d \rho_j, P_{\geq j} d h_j \|_{L^1_{[0,T_j]}L^\infty_x}
		\leq &  C(1+\mathbb{E}^2(0)) \cdot (1-2^{-\frac{\epsilon}{40}})^{-1}  2^{-\frac{\epsilon}{10}j} 2^{-\frac{1}{3+\epsilon}j}.
	\end{split}
\end{equation}
\textit{Case 2: Low frequency}. For $k<j$, we decompose
\begin{equation}\label{qu25}
\begin{split}
P_k d{\bv}_{j}=& P_kd{\bv}_{k}+P_k(d{\bv}_{j}-d{\bv}_{k})
\\
=& P_kd{\bv}_{k}+\textstyle{\sum_{m=k}^{j-1}} P_k(d{\bv}_{m+1}-d{\bv}_{m}).
\end{split}
\end{equation}
Similarly, we have
\begin{equation}\label{qu26}
	\begin{split}
		P_k d{\rho}_j=& P_kd{\rho}_k+\textstyle{\sum_{m=k}^{j-1}} P_k( d{\rho}_{m+1} -d{\rho}_{m}),
		\\
			P_k  d{h}_j=& P_k d{h}_k+\textstyle{\sum_{m=k}^{j-1}} P_k ( d{h}_{m+1}-d{h}_m ).
	\end{split}
\end{equation}
In \eqref{qu25}-\eqref{qu26}, $P_k{\rho}_k, P_k{h}_k$ and $P_k{\bv}_{k}$ can be handled as \eqref{qu20} and \eqref{qu21}.

For any fixed integer $m$, the solutions $(\bv_{m+1}, {\rho}_{m+1}, {h}_{m+1}, \bw_{m+1})$
and $(\bv_{m}, {\rho}_{m}, {h}_{m}, \bw_{m})$ exist on the same region $[0, {T}_{m+1}] \times \mathbb{R}^3$. By using \eqref{qu1} and \eqref{qu9}, we can see
\begin{equation}\label{qu28}
	\begin{split}
		& \|d{\bv}_{m+1}-d{\bv}_{m}, d{\rho}_{m+1}-d{\rho}_{m},d{h}_{m+1}-d{h}_{m} \|_{L^2_{ [0, T_{m+1}] } L_x^\infty }
		\\
		\leq & C\|{\bv}_{m+1}-{\bv}_{m},{\rho}_{m+1}-{\rho}_{m},,{h}_{m+1}-{h}_{m}\|_{L^\infty_{ [0, T_{m+1}] } H_x^{2+\frac{1}{10}\epsilon}}
		\\
		& + C( \|{\bw}_{m+1}-{\bw}_{m}\|_{ L^2_{[0, T_{m+1}] } H_x^{\frac32+\frac{1}{10}\epsilon} }  + \|{\bw}_{m+1}-{\bw}_{m}\|_{ L^\infty_{[0, T_{m+1}] } H_x^{\frac12+\frac{1}{10}\epsilon} } )
		\\
		& +C\| {h}_{m+1}-{h}_{m}\|_{ L^2_{ [0, T_{m+1}] } H_x^{\frac52+\frac{1}{10}\epsilon} } .
	\end{split}
\end{equation}
Now let us estimate the right hand side of \eqref{qu28}. Note $(\bv_{m+1}, \rho_{m+1}, h_{m+1})$ and  $(\bv_{m}, \rho_{m}, h_{m})$ is the solution of \eqref{sq}. Let $\bU_{m}=(\bv_m, p(\rho_m), h_m)^\mathrm{T} $. Then $\bU_{m+1}-\bU_{m}$ satisfies
\begin{equation}\label{qu30}
	\begin{cases}
		& A^0(\bU_{m+1}) \partial_t ( \bU_{m+1}- \bU_{m}) + A^i(\bU_{m+1}) \partial_i ( \bU_{m+1}- \bU_{m})=\Pi_m,
		\\
		& ( \bU_{m+1}- \bU_{m} )|_{t=0}= \bU_{0(m+1)}- \bU_{0m},
	\end{cases}
\end{equation}
where
\begin{equation*}
	\Pi_m=-[A^0(\bU_{m+1})-A^0(\bU_m) ]\partial_t  \bU_m- [A^i(\bU_{m+1})-A^i(\bU_{m}) ]\partial_i  \bU_{m}.
\end{equation*}
Multiplying $\bU_{m+1}- \bU_{m} $ on \eqref{qu30} and integrating it on $\mathbb{R}^3$, we have
\begin{equation}\label{qu31}
	\begin{split}
		\frac{d}{dt} \| \bU_{m+1}- \bU_{m} \|^2_{L^2_x} \leq &  C\| (d \bU_{m+1}, d \bU_{m}) \|_{L^\infty_x}\| \bU_{m+1}-\bU_{m} \|_{L^2_x}.
	\end{split}
\end{equation}
Then Strichartz estimate \eqref{qu15} and Gronwall's inequality tell us
\begin{equation}\label{qu32}
	\begin{split}
		\| \bU_{m+1}- \bU_{m} \|^2_{L^\infty_{ [0, {T}_{m+1}] } L^2_x} \leq & C\| \bU_{0(m+1)}- \bU_{0m} \|^2_{L^2_x}
		\\
		\leq & C\|\bv_{0(m+1)}- \bv_{0m}, \rho_{0(m+1)}- \rho_{0m} \|^2_{L^2_x}
		\\
		\leq & C\{ 2^{-\frac52m} \mathbb{E}(0) \}^2.
	\end{split}
\end{equation}
By Lemma \ref{jh0} and \eqref{qu32}, then we get
\begin{equation}\label{qu33}
	\begin{split}
		& \|\bv_{m+1}-\bv_{m}, \rho_{m+1}-\rho_{m}, h_{m+1}-h_{m}\|_{L^\infty_{ [0, {T}_{m+1}] } L^2_x} \leq C\mathbb{E}(0) 2^{-\frac52m}.
	\end{split}
\end{equation}
Similarly, $\bw_{m+1} - \bw_m$ satisfies
 \begin{equation}\label{qu34}
   \begin{split}
   &\partial_t (\bw_{m+1}- \bw_{m}) + (\bv_{m+1} \cdot \nabla) (\bw_{m+1}- \bw_{m})
   \\
   =& -(\bv_{m+1}-\bv_{m}) \cdot \nabla \bw_{m}+ (\bw_{m+1}- \bw_{m})\cdot \nabla \bv_{m+1}
   \\
   & + \bw_{m} \cdot \nabla (\bv_{m+1}-\bv_{m})+ \bB_{m+1}-\bB_{m}.
 \end{split}
 \end{equation}
 Above, we set $\bB_{m}=(B^1_{m},B^2_{m},B^3_{m})$ and $B^i_{m}=\bar{\rho}^{\gamma-1} \mathrm{e}^{h_m+(\gamma-2)\rho_m}\epsilon^{iab}\partial_a \rho_m \partial_b h_m$. So we have
 \begin{equation}\label{qu35}
   \| \bB_{m+1}-\bB_{m} \|_{L^2_x} \lesssim \| \partial( \rho_{m+1}-\rho_{m})\|_{L^2_x} \| \partial h_m \|_{L^\infty_x}
                                          +\| \partial( h_{m+1}-h_{m})\|_{L^2_x} \| \partial \rho_m \|_{L^\infty_x}.
 \end{equation}
Multiplying with $\bw_{m+1}- \bw_{m}$ on \eqref{qu34}, then integrating it on $[0,T_{m+1}]\times \mathbb{R}^3$, and using \eqref{qu35}, we can obtain
 \begin{equation}\label{qu36}
 \begin{split}
   & \| \bw_{m+1}- \bw_{m} \|^2_{L^\infty_{[0,T_{m+1}]}L^2_x}
   \\ \leq  & C\| \bw_{0(m+1)}- \bw_{0m} \|^2_{L^2_x}
    +C\int^{T_{m+1}}_0 \| \bB_{m+1}-\bB_{m} \|_{L^2_x} \| \bw_{m+1}- \bw_{m} \|_{L^2_x} d\tau
   \\
   & + C\int^{T_{m+1}}_0  \| \partial \bv_{m+1}, \partial \rho_{m+1}\|_{L^\infty_x}\| \bw_{m+1}- \bw_{m} \|^2_{L^2_x}d\tau
   \\
   & + C\int^{T_{m+1}}_0 \| \bv_{m+1}-\bv_{m} \|_{L^2_x}\| \bw_{m+1}- \bw_{m} \|_{L^2_x}\|\partial \bw_{m} \|_{L^\infty_x}d\tau
   \\
   &+ C \int^{T_{m+1}}_0 \| \rho_{m+1}-\rho_{m} \|_{L^2_x}\| \bw_{m+1}- \bw_{m} \|_{L^2_x}\|\partial \bw_{m} \|_{L^\infty_x}d\tau
   \\
   &+ C \int^{T_{m+1}}_0 \| \rho_{m+1}-\rho_{m} \|_{L^2_x}\| \partial ( \rho_{m+1}- \rho_{m}) \|_{L^2_x}\|\bw_{m} \|^2_{L^\infty_x}d\tau
   \\
   &+ C \int^{T_{m+1}}_0  \| \partial (\rho_{m+1}- \rho_{m})\|_{L^2_x}\| \partial h_{m+1}, \partial h_{m}\|_{L^\infty_x} \| \bw_{m+1}- \bw_{m} \|_{L^2_x}d\tau
   \\
   &+ C \int^{T_{m+1}}_0  \| \partial (h_{m+1}- h_{m}) \|_{L^2_x}\| \partial \rho_{m+1}, \partial \rho_{m}\|_{L^\infty_x} \| \bw_{m+1}- \bw_{m} \|_{L^2_x}d\tau.
 \end{split}
 \end{equation}
Using \eqref{qu15}, \eqref{qu33}, \eqref{qu35}, \eqref{qu36} and $T_{m+1}=2^{-\frac{1}{3+\epsilon}(m+1)}$, and $\|\partial \bw_{m} \|_{L^\infty_{ [0, {T}_{m+1}] } L^\infty_x} \leq 2^m \mathbb{E}(0)$, we shall prove
 \begin{equation}\label{qu37}
 \begin{split}
   \|\bw_{m+1}-\bw_{m}\|_{L^\infty_{ [0, {T}_{m+1}] } L^2_x} \leq & C(1+\mathbb{E}^2(0)) 2^{-(\frac32+\epsilon)m}.
 \end{split}
 \end{equation}
On the other hand, we can derive
\begin{equation}\label{qu39}
	\partial_t (h_{m+1}-h_m)+ (\bv_{m+1} \cdot \nabla) ( h_{m+1}-h_m )=-( \bv_{m+1}-\bv_m )\cdot \nabla h_m.
\end{equation}
By \eqref{qu39} and it's energy estimates, we obtain
\begin{equation}\label{qu40}
	\begin{split}
		\| h_{m+1}- h_{m} \|^2_{L^\infty_{ [0, {T}_{m+1}] } L^2_x} \leq  & \| h_{0(m+1)}- h_{0m} \|^2_{L^2_x}
		+ C\int^{{T}_{m+1}}_0  \| \partial \bv_{m+1}\|_{L^\infty_x}\| h_{m+1}- h_{m} \|^2_{L^2_x}d\tau
		\\
		& + C\int^{{T}_{m+1}}_0 \| ( \bv_{m+1}-\bv_m )\|_{L^2_x} \|\partial h_m \|_{L^\infty_x}\| h_{m+1}- h_{m} \|_{L^2_x} d\tau
	\end{split}
\end{equation}
Using \eqref{qu15}, \eqref{qu33}, and $T_{m+1}=2^{-\frac{1}{3+\epsilon}(m+1)}$, we shall obtain
\begin{equation}\label{qu41}
	\begin{split}
	\| h_{m+1}- h_{m} \|^2_{L^\infty_{ [0, {T}_{m+1}] } L^2_x} \leq C(1+\mathbb{E}^2(0)) 2^{-(\frac52+\epsilon)m},
	\end{split}
\end{equation}
which is a stronger one than in \eqref{qu33} for $ h_{m+1}- h_{m} $. Combing \eqref{qu33}, \eqref{qu37}, and \eqref{qu41}, we have proved
\begin{equation}\label{qu42}
	\begin{split}
		\|\bv_{m+1}-\bv_{m}, \rho_{m+1}-\rho_{m}\|_{L^\infty_{ [0, {T}_{m+1}] } L^2_x} & \leq C(1+\mathbb{E}^2(0)) 2^{-\frac52m},
		\\
		\|h_{m+1}-h_{m}\|_{L^\infty_{ [0, {T}_{m+1}] } L^2_x} & \leq C(1+\mathbb{E}^2(0)) 2^{-(\frac52+\epsilon) m}.
		\\
		\|\bw_{m+1}-\bw_{m}\|_{L^\infty_{ [0, {T}_{m+1}] } L^2_x} & \leq C(1+\mathbb{E}^2(0)) 2^{-(\frac32+\epsilon) m}.
	\end{split}
\end{equation}
Therefore, by \eqref{qu42},  \eqref{qu28} yields
\begin{equation}\label{qu43}
	\begin{split}
		& \|d{\bv}_{m+1}-d{\bv}_{m}, d{\rho}_{m+1}-d{\rho}_{m},d{h}_{m+1}-d{h}_{m} \|_{L^2_{ [0, T_{m+1}] } L_x^\infty }
		\\
		\leq & C(1+\mathbb{E}^2(0)) 2^{-\frac{1}{2(3+\epsilon)}m} 2^{-\frac{8\epsilon}{10} m}.
	\end{split}
\end{equation}
Therefore, for $k < j$ and $k \leq m$, we get
\begin{equation}\label{qu44}
	\begin{split}
		& \|P_k( d{\bv}_{m+1}-d{\bv}_{m} ), P_k( d{\rho}_{m+1}-d{\rho}_{m}) ,P_k(d{h}_{m+1}-d{h}_{m}) \|_{L^2_{ [0, T_{m+1}] } L_x^\infty }
		\\
		\leq & C(1+\mathbb{E}^2(0)) 2^{-\frac{1}{2(3+\epsilon)}m}2^{-\frac{\epsilon}{10} k} 2^{-\frac{7\epsilon}{10} m}.
	\end{split}
\end{equation}
Then it follows
\begin{equation}\label{qu45}
	\begin{split}
		& \|P_k( d{\bv}_{m+1}-d{\bv}_{m} ), P_k( d{\rho}_{m+1}-d{\rho}_{m}) ,P_k(d{h}_{m+1}-d{h}_{m}) \|_{L^1_{ [0, T_{m+1}] } L_x^\infty }
		\\
		\leq & C(1+\mathbb{E}^2(0)) 2^{-\frac{1}{3+\epsilon}m}2^{-\frac{\epsilon}{10} k} 2^{-\frac{7\epsilon}{10} m}.
	\end{split}
\end{equation}
\subsubsection{Extending the solution from short time-interval $[0,T_{j}]$ to $[0,T_{N_1}]$}\label{keye} Our goal is to extend the time interval $I_1=[0,T_j]$ to $[0,T_{N_1}]$ based on \eqref{qu22}, \eqref{qu23}, \eqref{qu44}, \eqref{qu45} and Theorem \ref{vve}. Here $T_{N_1}=[\mathbb{E}(0)]^{-1}2^{- \frac{N_1}{3+\epsilon} }$. To be simple, we first set
\begin{equation}\label{qu46}
	I_1=[0,T_j]=[t_0,t_1], \quad \quad |I_1|=[\mathbb{E}(0)]^{-1}2^{- \frac{1}{3+\epsilon} j}.
\end{equation}
Recall \eqref{qu23} and \eqref{qu45}, we can see
\begin{equation}\label{qu47}
	\| P_{k} d \bv_j, P_{k} d \rho_j, P_{k} d h_j \|_{L^1_{[0,T_j]}L^\infty_x}
	\leq  C(1+\mathbb{E}^2(0)) 2^{-\frac{1}{3+\epsilon}j}2^{-\frac{\epsilon}{40} k} 2^{-\frac{\epsilon}{10} j}, \quad k \geq j,
\end{equation}
\begin{equation}\label{qu48}
	\begin{split}
		& \|P_k d({\rho}_{m+1}-{\rho}_{m}), P_k d({\bv}_{m+1}-{\bv}_{m}), P_k d(h_{m+1}-h_{m}) \|_{L^1_{[0,T_{m+1}]} L^\infty_x}
		\\
		\leq    & C(1+\mathbb{E}^2(0)) 2^{-\frac{1}{3+\epsilon}m}2^{-\frac{\epsilon}{40} k} 2^{-\frac{\epsilon}{10} m}, \quad k<j.
	\end{split}
\end{equation}
By frequency decomposition, we get
\begin{equation}\label{qu49}
	\begin{split}
		d \bv_j= & \textstyle{\sum}^{\infty}_{k=j} d \bv_j+ \textstyle{\sum}^{j-1}_{k=1}P_k d \bv_j
		\\
		=& P_{\geq j}d \bv_j+ \textstyle{\sum}^{j-1}_{k=1} \textstyle{\sum}_{m=k}^{j-1} P_k  (d\bv_{m+1}-d\bv_m)+ \textstyle{\sum}^{j-1}_{k=1}P_k d \bv_k .
	\end{split}
\end{equation}
Similarly, we also have
\begin{equation}\label{qu50}
	\begin{split}
		d \rho_j= & P_{\geq j}d \rho_j+ \textstyle{\sum}^{j-1}_{k=1} \textstyle{\sum}_{m=k}^{j-1} P_k  (d \rho_{m+1}-d \rho_m)+ \textstyle{\sum}^{j-1}_{k=1}P_k d \rho_k.
	\end{split}
\end{equation}
and
\begin{equation}\label{qu51}
	\begin{split}
		d h_j= & P_{\geq j}d h_j+ \textstyle{\sum}^{j-1}_{k=1} \textstyle{\sum}_{m=k}^{j-1} P_k  (d h_{m+1}-d h_m)+ \textstyle{\sum}^{j-1}_{k=1}P_k d h_k.
	\end{split}
\end{equation}
Considering \eqref{qu49}-\eqref{qu51} and \eqref{qu47}-\eqref{qu48}, we can see that each term in \eqref{qu49}-\eqref{qu51} is much different. For example, the component $P_1 d\rho_1$ in $\textstyle{\sum}^{j-1}_{k=1}P_k d\rho_k$ has time-interval $T_1=[\mathbb{E}(0)]^{-1}2^{-\frac{1}{3+\epsilon}}$, so we don't need to extend it any more. But the term $d \rho_j$ only exists on the time interval $[0, T_j]$ at present. So we need to treat \eqref{qu49}-\eqref{qu51}  in a precise way when we extend the time interval of $\rho_j$ from $[0,T_j]$ to $[0,T_{N_1}]$.

\textbf{Step 1: Extending $[0,T_j]$ to $[0,T_{j-1}] \ (j \geq N_1+1)$.} To start, referring \eqref{DTJ}, so we need to calculate $\mathbb{E}(T_j)$ for obtaining the length of a time-interval. Then we shall calculate $\|d \bv_j, d \rho_j, dh_j\|_{L^1_{[0,T_j]L^\infty_x}}$. Using \eqref{qu47} and \eqref{qu48}, we derive that
\begin{equation*}
	\begin{split}
		& \|d \bv_j, d \rho_j, dh_j\|_{L^1_{[0,T_j]} L^\infty_x }
		\\
		\leq & 	\|P_{\geq j}d\bv_j, P_{\geq j}d \rho_j, P_{\geq j}d h_j\|_{L^1_{[0,T_j]} L^\infty_x }  + \textstyle{\sum}^{j-1}_{k=1} \|P_k d\bv_k, P_k d\rho_k, , P_k dh_k\|_{L^1_{[0,T_j]}L^\infty_x}\\
		& + \textstyle{\sum}^{j-1}_{k=1} \textstyle{\sum}_{m=k}^{j-1}  \|P_k  (d\bv_{m+1}-d\bv_m), P_k  (d\rho_{m+1}-d\rho_m), P_k  (dh_{m+1}-dh_m)\|_{L^1_{[0,T_j]}L^\infty_x}
		\\
		\leq & 	\|P_{\geq j}d\bv_j, P_{\geq j}d \rho_j, P_{\geq j}d h_j\|_{L^1_{[0,T_j]}L^\infty_x} + \textstyle{\sum}^{j-1}_{k=1} \|P_k d\bv_k, P_k d\rho_k, P_k dh_k\|_{L^1_{[0,T_k]}L^\infty_x}\\
		& + \textstyle{\sum}^{j-1}_{k=1} \textstyle{\sum}_{m=k}^{j-1}\| P_k  (d\bv_{m+1}-d\bv_m), P_k  (d\rho_{m+1}-d\rho_m), P_k  (dh_{m+1}-dh_m)\|_{L^1_{[0,T_{m+1}]}L^\infty_x}
		\\
		\leq & C(1+\mathbb{E}^2(0))  \textstyle{\sum}_{k=j}^{\infty} 2^{-\frac{1}{3+\epsilon}j}2^{-\frac{\epsilon}{40} k} 2^{-\frac{\epsilon}{10} j}
		\\
		& + C(1+\mathbb{E}^2(0)) \textstyle{\sum}^{j-1}_{k=1} 2^{-\frac{1}{3+\epsilon}k}2^{-\frac{\epsilon}{40} k} 2^{-\frac{\epsilon}{10} k}
		\\
		& +  C(1+\mathbb{E}^2(0)) \textstyle{\sum}^{j-1}_{k=1} \textstyle{\sum}_{m=k}^{j-1} 2^{-\frac{1}{3+\epsilon}m}2^{-\frac{\epsilon}{40} k} 2^{-\frac{\epsilon}{10} m}
		\\
		\leq & C(1+\mathbb{E}^2(0))[\frac13(1-2^{-\frac{\epsilon}{40}})]^{-2}.
	\end{split}
\end{equation*}
So we get
\begin{equation}
	\begin{split}\label{qu52}
		\|d \bv_j, d \rho_j, dh_j\|_{L^1_{[0,T_j]}L^\infty_x}
		\leq  C(1+\mathbb{E}^2(0))[\frac13(1-2^{-\frac{\epsilon}{40}})]^{-2}.
	\end{split}
\end{equation}
By \eqref{qu52} and Theorem \ref{vve}, we have
\begin{equation}\label{qu53}
	\begin{split}
		\mathbb{E}(T_{j}) \leq C(\mathbb{E}(0)+\mathbb{E}^2(0) ) \exp\{C(1+\mathbb{E}^2(0))[\frac13(1-2^{-\frac{\epsilon}{40}})]^{-2} \}=\mathbb{E}_*.
	\end{split}
\end{equation}
Above, $\mathbb{E}_*$ is stated in \eqref{qu0q}. Starting at the time $T_j$, seeing \eqref{qu03} and \eqref{qu53}, we shall get an extending time-interval of $(\bv,\rho_j,h_j,\bw_j)$ with a length of $\mathbb{E}_*^{-1}2^{-\frac{1}{3+\epsilon} j}$. But, if $T^*_j + \mathbb{E}_*^{-1}2^{-\frac{1}{3+\epsilon} j} \geq T_{j-1}$, so we have finished this step. Else, we need to extend it again.

$\mathit{Case 1: T_j + \mathbb{E}_*^{-1}2^{-\frac{1}{3+\epsilon} j} \geq T_{j-1} }$, then we get a new interval
\begin{equation}\label{qu54}
	I_2=[T_j, T_{j-1}], \quad |I_2|= (2^{\frac{1}{3+\epsilon}}-1) \mathbb{E}(0)^{-1} 2^{-\frac{1}{3+\epsilon} j}.
\end{equation}
Moreover, referring \eqref{qu47} and \eqref{qu48}, we deduce that
\begin{equation}\label{qu55}
	\| P_{k} d \bv_j, P_{k} d \rho_j, P_{k} d h_j \|_{L^1_{[T_j, T_{j-1}]}L^\infty_x}
	\leq  C  (1+\mathbb{E}_*^2) 2^{-\frac{1}{3+\epsilon}j}2^{-\frac{\epsilon}{40} k} 2^{-\frac{\epsilon}{10} j}, \quad k \geq j,
\end{equation}
\begin{equation}\label{qu56}
	\begin{split}
		& \|P_k (d{\rho}_{j}-d{\rho}_{j-1}), P_k (d{\bv}_{j}-d{\bv}_{j-1}, P_k (d{h}_{j}-d{h}_{j-1}) \|_{L^1_{[T_j, T_{j-1}]} L^\infty_x}
		\\
		\leq    & C  (1+\mathbb{E}_*^2) 2^{-\frac{1}{3+\epsilon}(j-1)}2^{-\frac{\epsilon}{40} k} 2^{-\frac{\epsilon}{10} (j-1)}, \quad k\leq j-1.
	\end{split}
\end{equation}
Using \eqref{qu04} and $j \geq N_0+1$, \eqref{qu55} and \eqref{qu56} yields
\begin{equation}\label{qu57}
	\| P_{k} d \bv_j, P_{k} d \rho_j, P_{k} d h_j \|_{L^1_{[T_j, T_{j-1}]}L^\infty_x}
	\leq  C  (1+\mathbb{E}^2(0)) 2^{-\frac{1}{3+\epsilon}j}2^{-\frac{\epsilon}{40} k} 2^{-\frac{3\epsilon}{40} j}, \quad k \geq j,
\end{equation}
\begin{equation}\label{qu58}
\begin{split}
	& \|P_k (d{\rho}_{j}-d{\rho}_{j-1}), P_k (d{\bv}_{j}-d{\bv}_{j-1}, P_k (d{h}_{j}-d{h}_{j-1}) \|_{L^1_{[T_j, T_{j-1}]} L^\infty_x}
	\\
	\leq    & C   (1+\mathbb{E}^2(0)) 2^{-\frac{1}{3+\epsilon}(j-1)}2^{-\frac{\epsilon}{40} k} 2^{-\frac{3\epsilon}{40} (j-1)}, \quad k\leq j-1.
\end{split}	
\end{equation}
Therefore, we could obtain
\begin{equation}\label{qu59}
	\begin{split}
		& \|d \bv_j, d\rho_j, dh_j\|_{L^1_{[0,T_{j-1}]} L^\infty_x}
		\\
		\leq & 	\|P_{\geq j}d\bv_j, P_{\geq j}d \rho_j, P_{\geq j}d h_j\|_{L^1_{[0,T_{j-1}]} L^\infty_x}  + \textstyle{\sum}^{j-1}_{k=1} \|P_k d\bv_k, P_k d\rho_k, P_k dh_k\|_{L^1_{[0,T_{j-1}]} L^\infty_x}
		\\
		+ & \textstyle{\sum}^{j-1}_{k=1} \|P_k  (d\bv_{j}-d\bv_{j-1}), P_k  (d\rho_{j}-d\rho_{j-1}), P_k  (dh_{j}-dh_{j-1})\|_{L^1_{[0,T_{j-1}]} L^\infty_x}.
		\\
		& + \textstyle{\sum}^{j-2}_{k=1} \textstyle{\sum}_{m=k}^{j-2}  \|P_k  (d\bv_{m+1}-d\bv_m), P_k  (d\rho_{m+1}-d\rho_m), P_k  (dh_{m+1}-dh_m)\|_{L^1_{[0,T_{j-1}]} L^\infty_x}.
	\end{split}
\end{equation}
Due to \eqref{qu47} and \eqref{qu57}, it yields
\begin{equation}\label{qu60}
	\begin{split}
		\|P_{\geq j}d\bv_j, P_{\geq j}d\rho_j, P_{\geq j}dh_j\|_{L^1_{[0,T_{j-1}]} L^\infty_x}
		\leq & C   (1+\mathbb{E}^2(0)) \textstyle{\sum}^{\infty}_{k=j}2^{-\frac{1}{3+\epsilon}j}2^{-\frac{\epsilon}{40} k} ( 2^{-\frac{\epsilon}{10} j}+2^{-\frac{3\epsilon}{40} j})
		\\
		\leq & C   (1+\mathbb{E}^2(0)) \textstyle{\sum}^{\infty}_{k=j}2^{-\frac{1}{3+\epsilon}j}2^{-\frac{\epsilon}{40} k} 2^{-\frac{3\epsilon}{40} j} \times 2.
	\end{split}
\end{equation}
Due to \eqref{qu48} and \eqref{qu58}, it yields
\begin{equation}\label{qu61}
	\begin{split}
		& \textstyle{\sum}^{j-1}_{k=1} \|P_k  (d\bv_{j}-d\bv_{j-1}), P_k  (d\rho_{j}-d\rho_{j-1}), P_k  (dh_{j}-dh_{j-1})\|_{L^1_{[0,T^*_{j-1}]} L^\infty_x}
		\\
		\leq & C   (1+\mathbb{E}^2(0))  \textstyle{\sum}^{j-1}_{k=1} 2^{-\frac{1}{3+\epsilon}(j-1)}2^{-\frac{\epsilon}{40} k} ( 2^{-\frac{\epsilon}{10} (j-1)}+2^{-\frac{3\epsilon}{40} (j-1)})
		\\
		\leq & C   (1+\mathbb{E}^2(0)) \textstyle{\sum}^{j-1}_{k=1} 2^{-\frac{1}{3+\epsilon}(j-1)}2^{-\frac{\epsilon}{40} k} 2^{-\frac{3\epsilon}{40} (j-1)} \times 2.
	\end{split}
\end{equation}
Inserting \eqref{qu60}-\eqref{qu61} into \eqref{qu59}, we derive that
\begin{equation}\label{qu62}
	\begin{split}
		\|d \bv_j, d \rho_j, d h_j\|_{L^1_{[0,T_{j-1}]} L^\infty_x}
		\leq & 	C(1+\mathbb{E}^2(0))  \textstyle{\sum}_{k=j}^{\infty} 2^{-\frac{1}{3+\epsilon}j}2^{-\frac{\epsilon}{40} k} 2^{-\frac{3\epsilon}{40} j} \times 2
		\\
		& + C   (1+\mathbb{E}^2(0)) \textstyle{\sum}^{j-1}_{k=1} 2^{-\frac{1}{3+\epsilon}(j-1)}2^{-\frac{\epsilon}{40} k} 2^{-\frac{3\epsilon}{40} (j-1)} \times 2
		\\
		& + C(1+\mathbb{E}^2(0)) \textstyle{\sum}^{j-1}_{k=1} 2^{-\frac{1}{3+\epsilon}k}2^{-\frac{\epsilon}{40} k} 2^{-\frac{\epsilon}{10} k}
		\\
		& +  C(1+\mathbb{E}^2(0)) \textstyle{\sum}^{j-2}_{k=1} \textstyle{\sum}_{m=k}^{j-2} 2^{-\frac{1}{3+\epsilon}m}2^{-\frac{\epsilon}{40} k} 2^{-\frac{\epsilon}{10} m}
		\\
		\leq & 	C(1+\mathbb{E}^2(0)) [\frac13(1-2^{-\frac{\epsilon}{40}})]^{-2}.
	\end{split}
\end{equation}
By \eqref{qu62} and Theorem \ref{vve}, we also prove
\begin{equation}\label{qu63}
	\begin{split}
		\mathbb{E}(T_{j-1}) \leq \mathbb{E}_*.
	\end{split}
\end{equation}
$\mathit{Case 2: T_j + \mathbb{E}_*^{-1}2^{-\frac{1}{3+\epsilon} j} < T_{j-1} }$. In this situation, we will record
\begin{equation*}
	I_2=[T_j, t_2], \quad |I_2| = \mathbb{E}_*^{-1}2^{-\frac{1}{3+\epsilon} j}.
\end{equation*}
Referring \eqref{qu55} and \eqref{qu56}, we also deduce that
\begin{equation}\label{qu630}
	\| P_{k} d \bv_j, P_{k} d \rho_j, P_{k} d h_j \|_{L^1_{[T_j, T_{j-1}]}L^\infty_x}
	\leq  C  (1+\mathbb{E}_*^2) 2^{-\frac{1}{3+\epsilon}j}2^{-\frac{\epsilon}{40} k} 2^{-\frac{\epsilon}{10} j}, \quad k \geq j,
\end{equation}
and
\begin{equation}\label{qu631}
	\begin{split}
		& \|P_k (d{\rho}_{j}-d{\rho}_{j-1}), P_k (d{\bv}_{j}-d{\bv}_{j-1}, P_k (d{h}_{j}-d{h}_{j-1}) \|_{L^1_{[T_j, T_{j-1}]} L^\infty_x}
		\\
		\leq    & C  (1+\mathbb{E}_*^2) 2^{-\frac{1}{3+\epsilon}(j-1)}2^{-\frac{\epsilon}{40} k} 2^{-\frac{\epsilon}{10} (j-1)}, \quad k\leq j-1.
	\end{split}
\end{equation}
Using \eqref{qu04} and $j \geq N_1+1$,  \eqref{qu630} and \eqref{qu631} yields
\begin{equation}\label{qu64}
	\| P_{k} d \bv_j, P_{k} d \rho_j, dh_j \|_{L^1_{I_2}L^\infty_x}
	\leq  C  (1+\mathbb{E}^2(0)) 2^{-\frac{1}{3+\epsilon}j}2^{-\frac{\epsilon}{40} k} 2^{-\frac{3\epsilon}{40} j}, \quad k \geq j,
\end{equation}
and
\begin{equation}\label{qu65}
	\begin{split}
		& \|P_k (d{\rho}_{j}-d{\rho}_{j-1}), P_k  (d{\bv}_{j}-d{\bv}_{j-1}) \|_{L^1_{I_2} L^\infty_x}
		\\
	\leq    & C  (1+\mathbb{E}^2(0)) 2^{-\frac{1}{3+\epsilon}(j-1)}2^{-\frac{\epsilon}{40} k} 2^{-\frac{3\epsilon}{40} (j-1)}, \quad k<j.
	\end{split}
\end{equation}
Similarly, we can also get
\begin{equation*}
	\begin{split}
		& \|d\bv_j, d\rho_j,dh_j\|_{L^1_{ I_1 \cup I_2 } L^\infty_x }
		\\
		\leq & 	\|P_{\geq j}d\bv_j, P_{\geq j}d \rho_j, P_{\geq j}d h_j\|_{L^1_{I_1 \cup I_2} L^\infty_x }   + \textstyle{\sum}^{j-1}_{k=1} \|P_k d\bv_k, P_k d\rho_k, P_k dh_k\|_{L^1_{[0,T^*_k]} L^\infty_x }
		\\
		+ & \textstyle{\sum}^{j-1}_{k=1} \|P_k  (d\bv_{j}-d\bv_{j-1}), P_k  (d\rho_{j}-d\rho_{j-1}), P_k  (dh_{j}-dh_{j-1})\|_{L^1_{I_1 \cup I_2 } L^\infty_x } .
		\\
		& + \textstyle{\sum}^{j-2}_{k=1} \textstyle{\sum}_{m=k}^{j-2}  \|P_k  (d\bv_{m+1}-d\bv_m), P_k  (d\rho_{m+1}-d\rho_m), P_k  (dh_{m+1}-dh_m)\|_{L^1_{I_1 \cup I_2} L^\infty_x } .
	\end{split}
\end{equation*}
Noting $ I_1 \cup I_2 \subseteq T_k$ when $k \leq j-1$, then it follows that
\begin{equation*}
	\begin{split}
		& \|d \bv_j, d \rho_j, dh_j\|_{L^1_{ I_1 \cup I_2 } L^\infty_x }
		\\
		\leq & 	\|P_{\geq j}d\bv_j, P_{\geq j}d \rho_j, P_{\geq j}d h_j\|_{L^1_{I_1 \cup I_2} L^\infty_x }   + \textstyle{\sum}^{j-1}_{k=1} \|P_k d\bv_k, P_k d\rho_k, P_k dh_k\|_{L^1_{I_1 \cup I_2 } L^\infty_x }
		\\
		+ & \textstyle{\sum}^{j-1}_{k=1} \|P_k  (d \bv_{j}-d \bv_{j-1}), P_k  (d\rho_{j}-d\rho_{j-1}), P_k  (dh_{j}-dh_{j-1})\|_{L^1_{I_1 \cup I_2 } L^\infty_x } 
		\\
		& + \textstyle{\sum}^{j-2}_{k=1} \textstyle{\sum}_{m=k}^{j-2}  \|P_k  (d\bv_{m+1}-d\bv_m), P_k  (d\rho_{m+1}-d\rho_m), P_k  (dh_{m+1}-dh_m)\|_{L^1_{[0,T_{m+1}]} L^\infty_x }.
	\end{split}
\end{equation*}
Inserting \eqref{qu47}, \eqref{qu48} and \eqref{qu64} and \eqref{qu65}, we have
\begin{equation}\label{qu66}
	\begin{split}		
		& \|d \bv_j, d \rho_j, d h_j\|_{L^1_{ I_1 \cup I_2 } L^\infty_x }
		\\
		\leq & 	C(1+\mathbb{E}^2(0))  \textstyle{\sum}_{k=j}^{\infty} 2^{-\frac{1}{3+\epsilon}j}2^{-\frac{\epsilon}{40} k} 2^{-\frac{3\epsilon}{40} j} \times 2
		\\
		& + C   (1+\mathbb{E}^2(0)) \textstyle{\sum}^{j-1}_{k=1} 2^{-\frac{1}{3+\epsilon}(j-1)}2^{-\frac{\epsilon}{40} k} 2^{-\frac{3\epsilon}{40} (j-1)} \times 2
		\\
		& + C(1+\mathbb{E}^2(0)) \textstyle{\sum}^{j-1}_{k=1} 2^{-\frac{1}{3+\epsilon}k}2^{-\frac{\epsilon}{40} k} 2^{-\frac{\epsilon}{10} k}
		\\
		& +  C(1+\mathbb{E}^2(0)) \textstyle{\sum}^{j-2}_{k=1} \textstyle{\sum}_{m=k}^{j-2} 2^{-\frac{1}{3+\epsilon}(m+1)}2^{-\frac{\epsilon}{40} k} 2^{-\frac{\epsilon}{10} (m+1)}
		\\
		\leq & 	C(1+\mathbb{E}^2(0)) [\frac13(1-2^{-\frac{\epsilon}{40}})]^{-2}.
	\end{split}
\end{equation}
By \eqref{qu66} and Theorem \ref{vve}, we also prove
\begin{equation}\label{qu67}
	\begin{split}
		E(t_2) \leq \mathbb{E}_*.
	\end{split}
\end{equation}
Therefore, we can repeat the process with a length with $\mathbb{E}_*^{-1}2^{-\frac{1}{3+\epsilon} j}$ till extending it to $T_{j-1}$. Moreover, on every new time-interval with $\mathbb{E}_*^{-1}2^{-\frac{1}{3+\epsilon} j}$ \eqref{qu47} and \eqref{qu48} hold. Set
\begin{equation}\label{Times1}
	Y_1= \frac{T_{j-1}-T_j}{\mathbb{E}_*^{-1}2^{-\frac{1}{3+\epsilon} j}}= (2^{ \frac{1}{3+\epsilon} }-1) \mathbb{E}^{-1}(0) \mathbb{E}_*.
\end{equation}
So we need a maximum of $X_1$-times to reach the time $T_{j-1}$ both in case 2(it's also adapt to case 1 for calculating the times). As a result, we shall calculate
\begin{equation*}
	\begin{split}
		& \|d \bv_j, d \rho_j, d h_j\|_{L^1_{[0,T_{j-1}]} L^\infty_x}
		\\
		\leq & 	\|P_{\geq j}d\bv_j, P_{\geq j}d \rho_j, P_{\geq j}d h_j\|_{L^1_{[0,T_{j-1}]} L^\infty_x}
		+ \textstyle{\sum}^{j-1}_{k=1} \|P_k d\bv_k, P_k d\rho_k, P_k dh_k\|_{L^1_{[0,T_{k}]} L^\infty_x}
		\\
		+ & \textstyle{\sum}^{j-1}_{k=1} \|P_k  (d\bv_{j}-d\bv_{j-1}), P_k  (d\rho_{j}-d\rho_{j-1}), P_k  (dh_{j}-dh_{j-1})\|_{L^1_{[0,T_{j-1}]} L^\infty_x}
		\\
		& + \textstyle{\sum}^{j-2}_{k=1} \textstyle{\sum}_{m=k}^{j-2}  \|P_k  (d\bv_{m+1}-d\bv_m), P_k  (d\rho_{m+1}-d\rho_m), P_k  (dh_{m+1}-dh_m)\|_{L^1_{[0,T_{m+1}]} L^\infty_x}.
	\end{split}
\end{equation*}	
Due to \eqref{qu47}, \eqref{qu48} and \eqref{qu64} and \eqref{qu65}, and \eqref{qu04}, we obtain
\begin{equation}\label{qu68}
	\begin{split}
		& \|d \bv_j, d \rho_j, dh_j\|_{L^1_{[0,T_{j-1}]} L^\infty_x}
		\\
		\leq & C(1+\mathbb{E}^2(0))  \textstyle{\sum}_{k=j}^{\infty} 2^{-\frac{1}{3+\epsilon}j}2^{-\frac{\epsilon}{40} k} 2^{-\frac{3\epsilon}{40} j} \times \{ (2^{ \frac{1}{3+\epsilon} }-1) \mathbb{E}^{-1}(0) \mathbb{E}_* +1 \}
		\\
		& + C  (1+\mathbb{E}^2(0)) \textstyle{\sum}^{j-1}_{k=1} 2^{-\frac{1}{3+\epsilon}(j-1)}2^{-\frac{\epsilon}{40} k} 2^{-\frac{3\epsilon}{40} (j-1)} \times \{ (2^{ \frac{1}{3+\epsilon} }-1) \mathbb{E}^{-1}(0) \mathbb{E}_* +1 \}
		\\
		& + C (1+\mathbb{E}^2(0)) \textstyle{\sum}_{j-1}^{k=1} 2^{-\frac{1}{3+\epsilon}k}2^{-\frac{\epsilon}{40} k} 2^{-\frac{\epsilon}{10} k}
		\\
		& +  C(1+\mathbb{E}^2(0)) \textstyle{\sum}^{j-2}_{k=1} \textstyle{\sum}_{m=k}^{j-2} 2^{-\frac{1}{3+\epsilon}m}2^{-\frac{\epsilon}{40} k} 2^{-\frac{\epsilon}{10} m}
		\\
		\leq & 	C(1+\mathbb{E}^2(0)) [\frac13(1-2^{-\frac{\epsilon}{40}})]^{-2}.
	\end{split}
\end{equation}
Therefore, by \eqref{qu68} and Theorem \ref{vve}, we can see that
\begin{equation}\label{qu69}
	\begin{split}
		\mathbb{E}(T_{j-1}) \leq \mathbb{E}_*.
	\end{split}
\end{equation}
At this moment, both in case 1 or case 2, seeing from \eqref{qu68}, \eqref{qu69}, \eqref{qu62}, \eqref{qu63}, through a maximum of $X_1=(2^{ \frac{1}{3+\epsilon} }-1) \mathbb{E}^{-1}(0) \mathbb{E}_*$ times with each length $\mathbb{E}_*^{-1}2^{-\frac{1}{3+\epsilon} j}$ or $(2^{\frac{1}{3+\epsilon}}-1)\mathbb{E}(0)^{-1}2^{-\frac{1}{3+\epsilon}j}$, we shall extend the solutions $(\bv_j,\rho_j, h_j,\bw_j)$ from $[0,T_j]$ to $[0,,T_{j-1}]$, and
\begin{equation}\label{qu70}
	\begin{split}
	 \mathbb{E}(T_{j-1}) \leq & \mathbb{E}_*,
		\\
	 \|d \bv_j, d \rho_j, dh_j\|_{L^1_{[0,T_{j-1}]} L^\infty_x}
	\leq  	& C(1+\mathbb{E}^2(0)) [\frac13(1-2^{-\frac{\epsilon}{40}})]^{-2}.
	\end{split}
\end{equation}
As a result, we also have extended the solutions $(\bv_m,\rho_m,h_m,\bw_m)$ ($m\in[N_0, j-1]$) from $[0,T_m]$ to $[0,T_{m-1}]$. Moreover, referring \eqref{qu47} and \eqref{qu48}, \eqref{qu55}, \eqref{qu56}, and \eqref{Times1}, we get
\begin{equation}\label{qu71}
	\begin{split}
		& \| P_{k} d \bv_m, P_{k} d \rho_m , P_{k} d h_m\|_{L^1_{[0,T_{m-1}]}L^\infty_x}
		\\
		\leq  & \| P_{k} d\bv_m, P_{k} d \rho_m, P_{k} d h_m \|_{L^1_{[0,T_{m}]}L^\infty_x}
		+\| P_{k} d \bv_m, P_{k} d \rho_m, P_{k} d h_m \|_{L^1_{[T_{m},T_{m-1}]}L^\infty_x}
		\\
		\leq  & C(1+\mathbb{E}^2(0)) 2^{-\frac{1}{3+\epsilon}m}2^{-\frac{\epsilon}{40} k} 2^{-\frac{\epsilon}{10} m}
		\\
		& \ \ + C(1+\mathbb{E}^2_*) 2^{-\frac{1}{3+\epsilon}j}2^{-\frac{\epsilon}{40} k} 2^{-\frac{\epsilon}{10} j} \times (2^{ \frac{1}{3+\epsilon} }-1) \mathbb{E}^{-1}(0) \mathbb{E}_*.
	\end{split}	
\end{equation}
Using \eqref{qu04}, for $k \geq m$ and $m\geq N_0+1$,  \eqref{qu71} yields
\begin{equation}\label{qu72}
	\begin{split}
		\| P_{k} d \bv_m, P_{k} d \rho_m, P_{k} d h_m \|_{L^1_{[0,T_{m-1}]}L^\infty_x}
		\leq & C(1+\mathbb{E}^2(0)) 2^{-\frac{1}{3+\epsilon}j}2^{-\frac{\epsilon}{40} k} 2^{-\frac{3\epsilon}{40} j} \times 2^{ \frac{1}{3+\epsilon} }.
	\end{split}	
\end{equation}
Similarly, if $m\geq N_0+1$, the following estimate
\begin{equation}\label{qu73}
	\begin{split}
		& \|P_k (d{\rho}_{m}-d{\rho}_{m-1}), P_k (d{\bv}_{m}-d{\bv}_{m-1}), P_k (d{h}_{m}-d{h}_{m-1}) \|_{L^1_{[0,T_{m-1}]} L^\infty_x}
		\\
		\leq    &  \|P_k (d{\rho}_{m}-d{\rho}_{m-1}), P_k  (d{\bv}_{m}-d{\bv}_{m-1}), P_k  (d{h}_{m}-d{h}_{m-1}) \|_{L^1_{[0,T_{m}]} L^\infty_x}
		\\
		& +\|P_k (d {\rho}_{m}-d {\rho}_{m-1}), P_k (d {\bv}_{m}-d {\bv}_{m-1}), P_k (d {h}_{m}-d {h}_{m-1}) \|_{L^1_{[T_m,T_{m-1}]} L^\infty_x}
		\\
		\leq & C(1+\mathbb{E}^2(0)) 2^{-\frac{1}{3+\epsilon}(m-1)}2^{-\frac{\epsilon}{40} k} 2^{-\frac{\epsilon}{10} (m-1)}
		\\
		& \ \ + C(1+\mathbb{E}^2_*) 2^{-\frac{1}{3+\epsilon}(m-1)}2^{-\frac{\epsilon}{40} k} 2^{-\frac{\epsilon}{10} (m-1)} \times (2^{ \frac{1}{3+\epsilon} }-1) \mathbb{E}^{-1}(0) \mathbb{E}_*,
		\\
		\leq & C(1+\mathbb{E}^2(0)) 2^{-\frac{1}{3+\epsilon}(m-1)}2^{-\frac{\epsilon}{40} k} 2^{-\frac{3\epsilon}{40} (m-1)}\times 2^{\frac{1}{3+\epsilon}},\quad k<j,
	\end{split}	
\end{equation}
hold. Above, we also use the condition \eqref{qu04}.

\textbf{Step 2: Extending time interval $[0,T_j]$ to $[0,T_{N_1}]$}. Based the above analysis in Step 1, we can give a induction to achieve our goal. We assume the solutions $(\bv_j, \rho_j, h_j, \bw_j)$ can be extended from $[0,T_j]$ to $[0,T_{j-l}]$ through a maximal $Y_l$ times and
\begin{equation}\label{qu74}
	\begin{split}
		Y_l=& \frac{T_{j-l}- T_{j}}{\mathbb{E}^{-1}_* 2^{-\frac{1}{3+\epsilon} j}}
		\\
		=& \frac{\mathbb{E}(0)^{-1}( 2^{-\frac{1}{3+\epsilon}(j-l)} - 2^{-\frac{1}{3+\epsilon}j} )}{\mathbb{E}^{-1}_* 2^{-\frac{1}{3+\epsilon} j}}
		\\
		=& \frac{\mathbb{E}_*}{\mathbb{E}(0)} (2^{\frac{1}{3+\epsilon} l}-1).
	\end{split}	
\end{equation}
Moreover, the following bounds
\begin{equation}\label{qu75}
	\begin{split}
		\| P_{k} d \bv_j, P_{k} d \rho_j, P_{k} d h_j \|_{L^1_{[0,T_{j-l}]}L^\infty_x}
		\leq & C(1+\mathbb{E}^2(0)) 2^{-\frac{1}{3+\epsilon}j}2^{-\frac{\epsilon}{40} k} 2^{-\frac{3\epsilon}{40}j}\times 2^{\frac{1}{3+\epsilon}l}, \quad k \geq m,
	\end{split}	
\end{equation}
and
\begin{equation}\label{qu76}
	\begin{split}
		& \|P_k (d{\rho}_{m+1}-d{\rho}_{m}), P_k (d{\bv}_{m+1}-d{\bv}_{m}) , P_k (d{h}_{m+1}-d{h}_{m})\|_{L^1_{[0,T_{m-l}]} L^\infty_x}
		\\
		\leq  & C(1+\mathbb{E}^2(0)) 2^{-\frac{1}{3+\epsilon}m}2^{-\frac{\epsilon}{40} k} 2^{-\frac{3\epsilon}{40}m}\times 2^{\frac{1}{3+\epsilon}l}, \quad k<j,
	\end{split}	
\end{equation}
and
\begin{equation}\label{qu77}
	\|d \bv_j, d \rho_j, dh_j\|_{L^1_{[0,T_{j-l}]} L^\infty_x}
	\leq  	C(1+\mathbb{E}^2(0)) [\frac13(1-2^{-\frac{\delta_{\epsilon}}{40}})]^{-2},
\end{equation}
and
\begin{equation}\label{qu78}
	\mathbb{E}(T_{j-l}) \leq \mathbb{E}_*.
\end{equation}
In our step 1, seeing from  \eqref{Times1}, \eqref{qu70}, \eqref{qu72}, and \eqref{qu73}, then \eqref{qu74}-\eqref{qu78} hold when $l=1$. Next, we will prove \eqref{qu74}-\eqref{qu78} if we replace $l$ to $l+1$. If we want to extend the solution from $[0,T_{j-l}]$ to $[0,T_{j-(l+1)}]$. As a result, we have
\begin{equation}\label{qu79}
	j-l\geq N_1+1.
\end{equation}
Else, the solution $(\bv_j,\rho_j,h_j,\bw_j)$ has existed on $[0,T_{N_1}]$. so we don't need to extend it if $j-l\leq N_1$. In the following, we are extending the solution starting at the time $T_{j-l}$ based on assumptions \eqref{qu74}-\eqref{qu78}.

Seeing \eqref{qu03} and \eqref{qu78}, we shall get an extending time-interval of $(\bv,\rho_j,h_j,\bw_j)$ with a length of $(\mathbb{E}_*)^{-1}2^{-\frac{1}{3+\epsilon} j}$. So we can go to the case 2 in step 1, and the length every new time-interval is $(\mathbb{E}_*)^{-1}2^{-\frac{1}{3+\epsilon} j}$. Therefore, we need to extend it with $Y$-times to the time-interval $[0,T_{j-(l+1)}]$, where
\begin{equation}\label{qu80}
	Y= \frac{T_{j-(l+1)}-T_{j-l}}{(\mathbb{E}_*)^{-1}2^{-\frac{1}{3+\epsilon} j}}= 2^{\frac{1}{3+\epsilon} l}(2^{\frac{1}{3+\epsilon}}-1)\frac{\mathbb{E}_*}{\mathbb{E}(0)} .
\end{equation}
As a result, we extend the solutions with $Y_l+Y$-times from $[0,T_j]$ to $[0,T_{j-(l+1)}]$. Calculate
\begin{equation*}
	Y_l+Y= (2^{\frac{1}{3+\epsilon}(l+1)}-1)\frac{\mathbb{E}_*}{\mathbb{E}(0)}.
\end{equation*}
Then we shall set
\begin{equation}\label{qu81}
	Y_{l+1}=(2^{\frac{1}{3+\epsilon}(l+1)}-1)\frac{\mathbb{E}_*}{\mathbb{E}(0)}.
\end{equation}
Moreover, for $k\geq j$, we have
\begin{equation}\label{qu82}
	\begin{split}
		\| P_{k} d \bv_j, P_{k} d\rho_j, P_{k} dh_j \|_{L^1_{[0, T_{j-(l+1)}]}L^\infty_x}\leq &\| P_{k} d \bv_j, P_{k} d\rho_j, P_{k} dh_j \|_{L^1_{[0,T_{j-l}]}L^\infty_x}
		\\
		&+	\| P_{k} d \bv_j, P_{k} d\rho_j, P_{k} dh_j \|_{L^1_{[T_{j-l}, T_{j-(l+1)}]}L^\infty_x}.
	\end{split}
\end{equation}
Using \eqref{qu630} and \eqref{qu631}
\begin{equation}\label{qu83}
	\begin{split}
		& \| P_{k} d \bv_j, P_{k} d\rho_j , P_{k} dh_j \|_{L^1_{[T^*_{j-l}, T^*_{j-(l+1)}]}L^\infty_x}
		\\
		\leq   & C(1+\mathbb{E}_*^2) 2^{-\frac{1}{3+\epsilon}m}2^{-\frac{\epsilon}{40} k} 2^{-\frac{\epsilon}{10}m} \times 2^{\frac{1}{3+\epsilon} l}(2^{\frac{1}{3+\epsilon}}-1)\frac{\mathbb{E}_*}{\mathbb{E}(0)}, \quad k\geq j.
	\end{split}
\end{equation}
Due to \eqref{qu04}, we can see
\begin{equation*}
	\begin{split}
		(1+\mathbb{E}^2(0)) (1+\mathbb{E}_*^2) \frac{\mathbb{E}_*}{\mathbb{E}(0)} 2^{-\frac{\epsilon}{40} N_0} \leq 1.
	\end{split}
\end{equation*}
Hence \eqref{qu83} yields
\begin{equation}\label{qu84}
	\begin{split}
		& \| P_{k} d \bv_j, P_{k} d\rho_j, P_{k} dh_j \|_{L^1_{[T_{j-l}, T_{j-(l+1)}]}L^\infty_x}
		\\
		\leq   & C(1+\mathbb{E}^2(0)) 2^{-\frac{1}{3+\epsilon}m}2^{-\frac{\epsilon}{40} k} 2^{-\frac{3\epsilon}{40}m} \times 2^{\frac{1}{3+\epsilon} l}(2^{\frac{1}{3+\epsilon}}-1), \quad k\geq j.
	\end{split}
\end{equation}
Using \eqref{qu75} and \eqref{qu75}, for $k \geq j$, we get
\begin{equation}\label{qu85}
	\begin{split}
		\| P_{k} d \bv_j, P_{k} d \rho_j, P_{k} d h_j \|_{L^1_{[0,T_{j-(l+1)}]}L^\infty_x}
		\leq & C(1+\mathbb{E}^2(0))  2^{-\frac{1}{3+\epsilon}m} 2^{-\frac{\epsilon}{40} k} 2^{-\frac{3\epsilon}{40}m} \times 2^{\frac{1}{3+\epsilon}(l+1)}.
	\end{split}	
\end{equation}
On the other hand, if $k<j$, then we have
\begin{equation}\label{qu86}
	\begin{split}
		& \|P_k (d{\rho}_{m+1}-d{\rho}_{m}), P_k (d{\bv}_{m+1}-d{\bv}_{m}), P_k (d{h}_{m+1}-d{h}_{m}) \|_{L^1_{[0,T_{m-(l+1)}]} L^\infty_x}
		\\
		\leq  & \|P_k (d{\rho}_{m+1}-d{\rho}_{m}), P_k (d{\bv}_{m+1}-d{\bv}_{m}), P_k (d{h}_{m+1}-d{h}_{m}) \|_{L^1_{[0,T_{m-l}]} L^\infty_x}
		\\
		& + \|P_k (d{\rho}_{m+1}-d{\rho}_{m}), P_k (d{\bv}_{m+1}-d{\bv}_{m}), P_k (d{h}_{m+1}-d{h}_{m}) \|_{L^1_{[T_{m-l},T_{m-(l+1)}]} L^\infty_x}.
	\end{split}	
\end{equation}
When we extend the solutions $(\bv_j, \rho_j, h_j, \bw_j)$ from $[0,T_{j-l}]$ to $[0,T_{j-(l+1)}]$, then the solutions $(\bv_m, \rho_m, h_m, \bw_m)$ is also extended from $[0,T_{m-l}]$ to $[0,T_{m-(l+1)}]$. Seeing \eqref{qu631}, \eqref{qu73} and \eqref{qu78}, we can obtain
\begin{equation}\label{qu87}
	\begin{split}
		& \|P_k (d{\rho}_{m+1}-d{\rho}_{m}), P_k (d{\bv}_{m+1}-d{\bv}_{m}), P_k (d{h}_{m+1}-d{h}_{m}) \|_{L^1_{[T_{m-l},T_{m-(l+1)}]} L^\infty_x}
		\\
		\leq  & C(1+\mathbb{E}_*^2) 2^{-\frac{1}{3+\epsilon}m}2^{-\frac{\epsilon}{40} k} 2^{-\frac{\epsilon}{10}m} \times 2^{\frac{1}{3+\epsilon} l}(2^{\frac{1}{3+\epsilon}}-1)\frac{\mathbb{E}_*}{\mathbb{E}(0)}.
	\end{split}	
\end{equation}
Using \eqref{qu04} again, then \eqref{qu87} becomes to
\begin{equation}\label{qu88}
	\begin{split}
		& \|P_k (d{\rho}_{m+1}-d{\rho}_{m}), P_k (d{\bv}_{m+1}-d{\bv}_{m}), P_k (d{h}_{m+1}-d{h}_{m}) \|_{L^1_{[T_{m-l},T_{m-(l+1)}]} L^\infty_x}
		\\
		\leq  & C(1+\mathbb{E}^2(0)) 2^{-\frac{1}{3+\epsilon}m}2^{-\frac{\epsilon}{40} k} 2^{-\frac{3\epsilon}{40}m} \times 2^{\frac{1}{3+\epsilon} l}(2^{\frac{1}{3+\epsilon}}-1).
	\end{split}	
\end{equation}
Combing \eqref{qu76} and \eqref{qu88}, it follows
\begin{equation}\label{qu89}
	\begin{split}
		& \|P_k (d{\rho}_{m+1}-d{\rho}_{m}), P_k (d{\bv}_{m+1}-d{\bv}_{m}), P_k (d{h}_{m+1}-d{h}_{m}) \|_{L^1_{[0,T_{m-(l+1)}]} L^\infty_x}
		\\
		\leq  & C(1+\mathbb{E}^2(0)) 2^{-\frac{1}{3+\epsilon}m}2^{-\frac{\epsilon}{40} k} 2^{-\frac{3\epsilon}{40}m} \times 2^{\frac{1}{3+\epsilon} (l+1)}.
	\end{split}	
\end{equation}
We are now ready to bound the following Strichartz estimate
\begin{equation}\label{qu90}
	\begin{split}
		& \|d \bv_j, d \rho_j, d h_j\|_{L^1_{[0,T_{j-(l+1)}]} L^\infty_x}
		\\
		\leq & \|P_{\geq j}\bv\rho_j, P_{\geq j}\partial \rho_j, P_{\geq j}\partial h_j\|_{L^1_{[0,T_{j-(l+1)}]} L^\infty_x}
		\\
		& + \textstyle{\sum}^{j-1}_{k=j-l}\textstyle{\sum}^{j-1}_{m=k} \|P_k  (d\bv_{m+1}-d\bv_{m}), P_k  (d\rho_{m+1}-d\rho_{m}), P_k  (dh_{m+1}-dh_{m})\|_{L^1_{[0,T_{j-(l+1)}]} L^\infty_x}
		\\
		& + \textstyle{\sum}^{j-1}_{k=1} \|P_k d \bv_k, P_k d \rho_k, P_k d h_k\|_{L^1_{[0,T_{j-(l+1)}]} L^\infty_x}
		\\
		=& 	\|P_{\geq j}d\bv_j, P_{\geq j}d \rho_j, P_{\geq j}d h_j\|_{L^1_{[0,T^*_{j-(l+1)}]} L^\infty_x}
		\\
		& + \textstyle{\sum}^{j-1}_{k=j-l} \|P_k d\bv_k, P_k d\rho_k, P_k dh_k\|_{L^1_{[0,T_{j-(l+1)}]} L^\infty_x}
		+ \textstyle{\sum}^{j-(l+1)}_{k=1} \|P_k d\bv_k, P_k d\rho_k, P_k dh_k\|_{L^1_{[0,T_{j-(l+1)}]} L^\infty_x}
		\\
		& + \textstyle{\sum}^{j-1}_{k=1} \textstyle{\sum}^{j-1}_{m=j-(l+1)} \|P_k  (d\bv_{m+1}-d\bv_{m}), P_k  (d\rho_{m+1}-d\rho_{m}), P_k  (dh_{m+1}-dh_{m})\|_{L^1_{[0,T_{j-(l+1)}]} L^\infty_x}.
		\\
		& + \textstyle{\sum}^{j-(l+2)}_{k=1} \textstyle{\sum}_{m=k}^{j-(l+2)}  \|P_k  (d\bv_{m+1}-d\bv_m), P_k  (d\rho_{m+1}-d\rho_m), P_k  (dh_{m+1}-dh_m)\|_{L^1_{[0,T_{j-(l+1)}]} L^\infty_x}
		\\
		=& \Omega_1+  \Omega_2+ \Omega_3+ \Omega_4+ \Omega_5,
	\end{split}
\end{equation}
where
\begin{equation}\label{qu91}
	\begin{split}
		\Omega_1= &  \|P_{\geq j}d\bv_j, P_{\geq j}d \rho_j, P_{\geq j}d h_j\|_{L^1_{[0,T_{j-(l+1)}]} L^\infty_x} ,
		\\
		\Omega_2= & \textstyle{\sum}^{j-(l+2)}_{k=1}\textstyle{\sum}^{j-(l+2)}_{m=k} \|P_k  (d\bv_{m+1}-d\bv_{m}), P_k  (d\rho_{m+1}-d\rho_{m}), P_k  (dh_{m+1}-dh_{m})\|_{L^1_{[0,T_{j-(l+1)}]} L^\infty_x},
		\\
		\Omega_3= & \textstyle{\sum}^{j-1}_{k=1}\textstyle{\sum}^{j-1}_{m=j-(l+1)} \|P_k  (d\bv_{m+1}-d\bv_{m}), P_k  (d\rho_{m+1}-d\rho_{m}), P_k  (dh_{m+1}-dh_{m})\|_{L^1_{[0,T_{j-(l+1)}]} L^\infty_x},
		\\
		\Omega_4 =& \textstyle{\sum}^{j-(l+1)}_{k=1}	\|P_{k}d\bv_k, P_{k}d\rho_k, P_{k}dh_k\|_{L^1_{[0,T_{j-(l+1)}]} L^\infty_x},
		\\
		\Omega_5 =& \textstyle{\sum}^{j-1}_{k=j-l}	\|P_{k}d\bv_k, P_{k}d\rho_k, P_{k}dh_k\|_{L^1_{[0,T_{j-(l+1)}]} L^\infty_x}.
	\end{split}
\end{equation}
To get the estimate on time-interval $[0,T_{j-(l+1)}]$ for $\Omega_1, \Omega_2, \Omega_3, \Omega_4$ and $\Omega_5$, we should note that
there is no growth for $\Omega_2$ and $\Omega_4$ in this extending process. For example, considering $\Omega_2$, the existing time-interval of $P_k  (d\bv_{m+1}-d\bv_{m})$ is actually $[0,T_{m+1}]$, and $[0,T_{j-(l+1)}] \subseteq [0,T_{m+1}]$ if $m \geq j-(l+2)$. So we can use the initial bounds \eqref{qu47} and \eqref{qu48} to handle $\Omega_2$ and $\Omega_4$. While, considering $\Omega_1, \Omega_3$, and $\Omega_5$, we need to calculate the growth in the Strichartz estimate. Based on this idea, let us give a precise analysis on \eqref{qu91}.

According to \eqref{qu85}, we can estimate $\Omega_1$ as
\begin{equation}\label{qu92}
	\begin{split}
		\Omega_1 \leq & C(1+\mathbb{E}^2(0))  \textstyle{\sum}^{\infty}_{k=j} 2^{-\frac{1}{3+\epsilon}j} 2^{-\frac{\epsilon}{40} k} 2^{-\frac{3\epsilon}{40}j} \times 2^{\frac{1}{3+\epsilon}(l+1)}.
	\end{split}
\end{equation}
Due to \eqref{qu48}, we have
\begin{equation}\label{qu93}
	\begin{split}
		\Omega_2\leq &	\textstyle{\sum}^{j-(l+2)}_{k=1}\textstyle{\sum}^{j-(l+2)}_{m=k} \|P_k  (d\bv_{m+1}-d\bv_{m}), P_k  (d\rho_{m+1}-d\rho_{m}), P_k  (dh_{m+1}-dh_{m})\|_{L^1_{[0,T_{m+1}]} L^\infty_x},
		\\
		\leq &  C(1+\mathbb{E}^2(0)) \textstyle{\sum}^{j-(l+2)}_{k=1}\textstyle{\sum}^{j-(l+2)}_{m=k}   2^{-\frac{1}{3+\epsilon}m} 2^{-\frac{\epsilon}{40} k} 2^{-\frac{\epsilon}{10}m}.
	\end{split}	
\end{equation}
Using \eqref{qu89}, it follows
\begin{equation}\label{qu94}
	\begin{split}
		\Omega_3\leq & \textstyle{\sum}^{j-1}_{k=1}\textstyle{\sum}^{j-1}_{m=j-(l+1)} \|P_k  (d\bv_{m+1}-d\bv_{m}), P_k  (d\rho_{m+1}-d\rho_{m}), P_k  (dh_{m+1}-dh_{m})\|_{L^1_{[0,T_{m-(l+1)}]} L^\infty_x}
		\\
		\leq & C (1+\mathbb{E}^2(0)) \textstyle{\sum}^{j-1}_{k=1}\textstyle{\sum}^{j-1}_{m=j-(l+1)}   2^{-\frac{1}{3+\epsilon}m}2^{-\frac{\epsilon}{40} k} 2^{-\frac{3\epsilon}{40}m} \times 2^{\frac{1}{3+\epsilon} (l+1)}.
	\end{split}
\end{equation}
Above, we use $k \leq j-1< j$ and $m \leq j-1 < j$. Due to \eqref{qu47}, we can see
\begin{equation}\label{qu95}
	\begin{split}
		\Omega_4 \leq & \textstyle{\sum}^{j-(l+1)}_{k=1}	\|P_{k}d\bv_k, P_{k}d\rho_k\|_{L^1_{[0,T_{k}]} L^\infty_x}
		\\
		\leq & C(1+\mathbb{E}^2(0))  \textstyle{\sum}^{j-(l+1)}_{k=1}  2^{-\frac{1}{3+\epsilon}k} 2^{-\frac{\epsilon}{40} k} 2^{-\frac{\epsilon}{10}k}.
	\end{split}
\end{equation}
Noting $k\leq j-1 <j$ and using \eqref{kz51} again, we can estimate
\begin{equation}\label{qu96}
	\begin{split}
		\Omega_5 \leq & \textstyle{\sum}^{j-1}_{k=j-l}	\|P_{k}d\bv_k, P_{k}d\rho_k\|_{L^1_{[0,T_{j-(l+1)}]} L^\infty_x}
		\\
		\leq	& C(1+\mathbb{E}^2(0))  \textstyle{\sum}^{j-1}_{k=j-l}  2^{-\frac{1}{3+\epsilon}k} 2^{-\frac{\epsilon}{40} k} 2^{-\frac{3\epsilon}{40}k} \times 2^{\frac{1}{3+\epsilon}[(k+l+1)-j]}.
	\end{split}
\end{equation}
Inserting \eqref{qu92}-\eqref{qu96} to \eqref{qu90}, it follows
\begin{equation}\label{qu97}
	\begin{split}
		& \|d \bv_j, d \rho_j, dh_j\|_{L^1_{[0,T_{j-(l+1)}]} L^\infty_x}
		\\
		\leq & C(1+\mathbb{E}^2(0))  \textstyle{\sum}^{\infty}_{k=j} 2^{-\frac{1}{3+\epsilon}j} 2^{-\frac{\epsilon}{40} k} 2^{-\frac{3\epsilon}{40}j} \times 2^{\frac{1}{3+\epsilon}(l+1)}
		\\
		& + C(1+\mathbb{E}^2(0)) \textstyle{\sum}^{j-(l+2)}_{k=1}\textstyle{\sum}^{j-(l+2)}_{m=k}   2^{-\frac{1}{3+\epsilon}m} 2^{-\frac{\epsilon}{40} k} 2^{-\frac{\epsilon}{10}m}
		\\
		&+C (1+\mathbb{E}^2(0)) \textstyle{\sum}^{j-1}_{k=1}\textstyle{\sum}^{j-1}_{m=j-(l+1)}   2^{-\frac{1}{3+\epsilon}m}2^{-\frac{\epsilon}{40} k} 2^{-\frac{3\epsilon}{40}m} \times 2^{\frac{1}{3+\epsilon} (l+1)}
		\\
		& + C(1+\mathbb{E}^2(0))  \textstyle{\sum}^{j-(l+1)}_{k=1}  2^{-\frac{1}{3+\epsilon}k} 2^{-\frac{\epsilon}{40} k} 2^{-\frac{\epsilon}{10}k}
		\\
		& C(1+\mathbb{E}^2(0))  \textstyle{\sum}^{j-1}_{k=j-l}  2^{-\frac{1}{3+\epsilon}k} 2^{-\frac{\epsilon}{40} k} 2^{-\frac{3\epsilon}{40}k} \times 2^{\frac{1}{3+\epsilon}[(k+l+1)-j]}
	\end{split}
\end{equation}
In the case of $j-(l+1) \geq N_1$ and $j\geq N_1+1$,  \eqref{kz62} yields
\begin{equation}\label{qu98}
	\begin{split}
		\|d \bv_j, d \rho_j, dh_j\|_{L^1_{[0,T_{j-(l+1)}]} L^\infty_x} \leq & C(1+\mathbb{E}^2(0)) (1-2^{-\frac{\epsilon}{40}})^{-2} \big\{
		2^{-\frac{1}{3+\epsilon}(j-(l+1))} + 2^{-\frac{1}{3+\epsilon}}
		\\
		& \quad +2^{-\frac{1}{3+\epsilon} j} + 2^{-\frac{1}{3+\epsilon} }+2^{-\frac{1}{3+\epsilon}[j-(l+1)]}
		\big\}
		\\
		\leq & C(1+\mathbb{E}^2(0)) [\frac13(1-2^{-\frac{\epsilon}{40}})]^{-2}.
	\end{split}
\end{equation}
Due to Theorem \ref{vve} and \eqref{qu98}, we can achieve
\begin{equation}\label{qu99}
	\begin{split}
		\mathbb{E}(T_{j-(l+1)}) \leq \mathbb{E}_*.
	\end{split}
\end{equation}
Gathering \eqref{qu81}, \eqref{qu85}, \eqref{qu89}, \eqref{qu98} and \eqref{qu99}, we know that \eqref{qu74}-\eqref{qu78} holding for $l+1$. So our induction hold \eqref{qu74}-\eqref{qu78} for $l=1$ to $l=j-N_1$. Therefore, we can extend the solutions $(\bv_j,\rho_j,\bw_j)$ from $[0,T_j]$ to $[0,T_{N_1}]$ when $j \geq N_1$. Moreover, setting $l=j-N_1$ in \eqref{qu74}-\eqref{qu78}, we have
\begin{equation}\label{qu100}
	\begin{split}
		& \bar{T}=T_{N_1}=\mathbb{E}^{-1}(0)2^{-\frac{N_1}{3+\epsilon} }, \qquad \mathbb{E}(T) \leq  \mathbb{E}_*,
		\\
		& \|d \bv_j, d \rho_j, dh_j\|_{L^1_{[0,\bar{T}]} L^\infty_x}
		\leq  C(1+\mathbb{E}^2(0)) [\frac13(1-2^{-\frac{\epsilon}{40}})]^{-2},
	\end{split}
\end{equation}
where $N_1$ and $\mathbb{E}_*$ only depend on $\epsilon$ and $\bar{M}_0$, which is stated in \eqref{qu04} and \eqref{qu0q}.

By using \eqref{qu100} and Newton-Leibniz formula, we can see
\begin{equation*}
		 \|\bv_j, \rho_j, h_j\|_{L^\infty_{[0,\bar{T}]\times \mathbb{R}^3}} \leq C_0+C(1+\mathbb{E}^2(0)) [\frac13(1-2^{-\frac{\epsilon}{40}})]^{-2}.
\end{equation*}
Also, using a similar way of calculating \eqref{qu100}, we shall conclude
\begin{equation}\label{qu101}
	\begin{split}
		\|d \bv_j, d \rho_j, d \rho_j\|_{L^2_{[0,\bar{T}]} L^\infty_x}
		\leq & C(1+\mathbb{E}^2(0)) [\frac13(1-2^{-\frac{\epsilon}{40}})]^{-2}.
	\end{split}
\end{equation}
\end{proof}
\subsubsection{Strichartz estimates of linear wave equation on time-interval $[0,T_{N_1}]$}\label{keyf}
We still expect the behaviour of a linear wave equation endowed with $g_j=g(\bv_j,\rho_j,h_j)$. So we claim a theorem as follows
\begin{proposition}\label{zut}
	For $\frac{7}{6} \leq r< \frac72$ there is a solution $f_j$ on $[0,\bar{T}]\times \mathbb{R}^3$ satisfying the following linear wave equation
	\begin{equation}\label{zu01}
		\begin{cases}
			\square_{{g}_j} f_j=0,
			\\
			(f_j,\partial_t f_j)|_{t=0}=(f_{0j},f_{1j}),
		\end{cases}
	\end{equation}
	where $(f_{0j},f_{1j})=(P_{\leq j}f_0,P_{\leq j}f_1)$, and $(f_0,f_1)\in H_x^r \times H^{r-1}_x$, and $\bar{T}$ is stated in \eqref{qu100}. Moreover, for $a\leq r-\frac{7}{6}$, we have
	\begin{equation}\label{zu02}
		\begin{split}
			&\|\left< \partial \right>^{a-1} d{f}_j\|_{L^2_{[0,\bar{T}]} L^\infty_x}
			\leq  {\bar{M}_3} (\|{f}_0\|_{{H}_x^r}+ \|{f}_1 \|_{{H}_x^{r-1}}),
			\\
			&\|{f}_j\|_{L^\infty_{[0,\bar{T}]} H^{r}_x}+ \|\partial_t {f}_j\|_{L^\infty_{[0,\bar{T}]} H^{r-1}_x} \leq  {\bar{M}_3}(\| {f}_0\|_{H_x^r}+ \| {f}_1\|_{H_x^{r-1}}).
		\end{split}
	\end{equation}
\end{proposition}
\begin{proof}[Proof of Proposition \ref{zut}.] Our proof also relies a Strichartz estimates on a short time-interval. Then a loss of derivatives are then obtained by summing up the short time estimates with respect to these time intervals.

Recall \eqref{qu1}-\eqref{qu9}. For $1 \leq r_1 \leq \frac52+1$, consider the homogeneous linear wave equation on short time interval
	\begin{equation}\label{zu03}
		\begin{cases}
			\square_{{g}_j} F=0, \quad [0,T_j]\times \mathbb{R}^3,
			\\
			(F,F_t)|_{t=0}=(F_0,F_1)\in H_x^{r_1} \times H^{r_1-1}_x.
		\end{cases}
	\end{equation}
For $a_1<r_1-1$, then we can conclude that
	\begin{equation}\label{zu04}
		\begin{split}
			\|\left< \partial \right>^{a_1-1} dF\|_{L^2_{[0,T_j]} L^\infty_x}
			\leq & C(\|{F}_0\|_{{H}_x^{r_1}}+ \| {F}_1 \|_{{H}_x^{r_1-1}}),
		\end{split}
	\end{equation}
	and
	\begin{equation}\label{zu05}
		\begin{split}
			\|{F}\|_{L^\infty_{[0,T_j]} H^{r_1}_x}+ \|\partial_t {F}\|_{L^\infty_{[0,T_j]} H^{r_1-1}_x} \leq  C(\| {F}_0\|_{H_x^{r_1}}+ \| {F}_1\|_{H_x^{r_1-1}}).
		\end{split}
	\end{equation}
	Using \eqref{zu01}, \eqref{zu03} and \eqref{zu04}, we have
	\begin{equation}\label{zu06}
		\begin{split}
			\| \left< \partial \right>^{a-1+\frac{1}{2(3+\epsilon)}+\frac{\epsilon}{20}} df_j\|_{L^2_{[0,T_j]} L^\infty_x} \leq & C( \| f_{0j} \|_{H^r_x} + \| f_{1j} \|_{H^{r-1}_x}  )
			\\
			\leq &C ( \| f_{0} \|_{H^r_x} + \| f_{1} \|_{H^{r-1}_x}  ),
		\end{split}
	\end{equation}
where we use $a\leq r-\frac76$ and
\begin{equation}\label{zu060}
a+\frac{1}{2(3+\epsilon)}+\frac{\epsilon}{20} \leq r-\frac{7}{6}+ \frac{1}{2(3+\epsilon)}+\frac{\epsilon}{20}=r-1-\frac{7\epsilon}{60(3+\epsilon)}<r-1.		
\end{equation}

Due to Bernstein inequality, for $k\geq j$, we shall obtain
\begin{equation}\label{zu07}
	\begin{split}
		\| \left< \partial \right>^{a-1} P_k df_j\|_{L^2_{[0,T_j]} L^\infty_x}
		=& C2^{-(\frac{1}{2(3+\epsilon)}+\frac{\epsilon}{20}) k}\|\left< \partial \right>^{a-1+\frac{1}{2(3+\epsilon)}+\frac{\epsilon}{20}} P_k df_j \|_{L^2_{[0,T_j]} L^\infty_x}
		\\
		\leq & C2^{-\frac{1}{2(3+\epsilon)}j}2^{-\frac{\epsilon}{80} k}2^{-\frac{3\epsilon}{80} j} \| \left< \partial \right>^{a-1+\frac{1}{2(3+\epsilon)}+\frac{\epsilon}{20}} df_j\|_{L^2_{[0,T_j]} L^\infty_x}
	\end{split}
\end{equation}
Combining \eqref{zu06} with \eqref{zu07}, we get
\begin{equation}\label{zu070}
	\begin{split}
		\| \left< \partial \right>^{a-1} P_k df_j\|_{L^2_{[0,T_j]} L^\infty_x}
		\leq & C2^{-\frac{1}{2(3+\epsilon)}j}2^{-\frac{\epsilon}{80} k}2^{-\frac{3\epsilon}{80} j} ( \| f_{0} \|_{H^r_x} + \| f_{1} \|_{H^{r-1}_x}  ).
	\end{split}
\end{equation}
On the other hand, for any integer $m\geq 1$, we also have
\begin{equation*}
	\begin{cases}
		\square_{{g}_m} f_m=0, \quad [0,T_{m}]\times \mathbb{R}^3,
		\\
		(f_m,\partial_t f_m)|_{t=0}=(f_{0m},f_{1m}),
	\end{cases}
\end{equation*}
and
\begin{equation*}
	\begin{cases}
		\square_{{g}_{m+1}} f_{m+1}=0, \quad [0,T_{m+1}]\times \mathbb{R}^3,
		\\
		(f_{m+1},\partial_t f_{m+1})|_{t=0}=(f_{0(m+1)},f_{1(m+1)}).
	\end{cases}
\end{equation*}
So the difference term $f_{m+1}-f_m$ satisfies
\begin{equation}\label{zu08}
	\begin{cases}
		\square_{{g}_{m+1}} (f_{m+1}-f_m)=({g}^{\alpha i}_{m+1}-{g}^{\alpha i}_{m} )\partial_{\alpha i} f_m, \quad [0,T^*_{m+1}]\times \mathbb{R}^3,
		\\
		(f_{m+1}-f_m,\partial_t (f_{m+1}-f_m))|_{t=0}=(f_{0(m+1)}-f_{0m},f_{1(m+1)}-f_{1m}).
	\end{cases}
\end{equation}
Using \eqref{zu04}, and Duhamel's principle, and \eqref{zu060}, we get
\begin{equation}\label{zu09}
	\begin{split}
		& \| \left< \partial \right>^{a-1} P_k (d f_{m+1}-df_m)\|_{L^2_{[0,T_{m+1}]} L^\infty_x}
		\\
		\leq & C\| f_{0(m+1)}-f_{0m}\|_{H_x^{r_1}} + C\|f_{1(m+1)}-f_{1m}\|_{H_x^{r_1-1}}
		 \ + C\|({g}_{m+1}-{g}_{m}) \cdot\nabla d f_m\|_{H_x^{r_1}}
		\\
		\leq & C2^{ -\frac{1}{2(3+\epsilon)}m }2^{-\frac{\epsilon}{20}m}  \| (f_{0},f_1) \|_{H^r_x \times H^{r-1}_x}
		 + C \| {g}_{m+1}-{g}_{m}\|_{L^1_{[0,T_{m+1}]} L^\infty_x} \| \partial d f_m\|_{L^\infty_{[0,T_{m+1}]} H_x^{r_1}} .
	\end{split}
\end{equation}
where we set $r_1=r-(\frac{1}{2(3+\epsilon)}+\frac{\epsilon}{20})$. By using the energy estimates, we obtain
\begin{equation}\label{zu10}
	\begin{split}
		\| \partial d f_m\|_{L^\infty_{[0,T_{m+1}]} H_x^{r_1-1}} \leq & \|d f_m\|_{L^\infty_{[0,T_{m+1}]} H_x^{r}}
	\\
	\leq & C(\|f_{0m}\|_{H^{r+1}_x}+\|f_{1m}\|_{H^{r}_x})
		\\
		\leq & C2^m (\|f_{0m}\|_{H^r_x}+\|f_{1m}\|_{H^{r-1}_x}) .
	\end{split}
\end{equation}
By using Strichartz estimates \eqref{zu04} and Lemma \ref{LD}, and \eqref{qu24}, we shall prove that
\begin{equation}\label{zu11}
	\begin{split}
		& 2^m \| {g}_{m+1}-{g}_{m}\|_{L^1_{[0,T_{m+1}]} L^\infty_x}
		\\
		\leq & 2^m T^{\frac12}_{m+1}\| {\bv}_{m+1}-{\bv}_{m},  {\rho}_{m+1}-{\rho}_{m}\|_{L^2_{[0,T_{m+1}]} L^\infty_x}
		\\
		\leq & C2^m 2^{-\frac{1}{2(3+\epsilon)}(m+1)}( \| {\bv}_{m+1}-{\bv}_{m},  {\rho}_{m+1}-{\rho}_{m},  {h}_{m+1}-{h}_{m}\|_{L^\infty_{[0,T_{m+1}]} H^{1+\frac{\epsilon}{2}}_x}
		\\
		& \quad + \| {\bw}_{m+1}-{\bw}_{m}\|_{L^2_{[0,T_{m+1}]} H^{\frac12+\frac{\epsilon}{2}}_x}+ \| {h}_{m+1}-{h}_{m}\|_{L^2_{[0,T_{m+1}]} H^{\frac32+\frac{\epsilon}{2}}_x})
		\\
		\leq & C2^{-\frac{1}{2(3+\epsilon)}m}2^{-\frac{\epsilon}{80} m}2^{-\frac{39\epsilon}{80} m} (1+\mathbb{E}^2(0)).
	\end{split}
\end{equation}
Due to \eqref{zu10} and \eqref{zu11}, for $k<j$, and $k\leq m$, we get
\begin{equation}\label{zu12}
	\begin{split}
		& \| \left< \partial \right>^{a-1} P_k (df_{m+1}-df_m)\|_{L^2_{[0,T_{m+1}]} L^\infty_x}
		\\
		\leq & C2^{-\frac{1}{2(3+\epsilon)}m}2^{-\frac{\epsilon}{80} k}2^{-\frac{3\epsilon}{80} m} (\|f_{0}\|_{H^r_x}+\|f_{1}\|_{H^{r-1}_x})(1+\mathbb{E}^2(0)).
	\end{split}
\end{equation}
In the following part, let us consider the energy estimates. Taking the operator $\left< \partial \right>^{r-1}$ on \eqref{zu01}, then we get
\begin{equation}\label{zu15}
	\begin{cases}
		\square_{{g}_j} \left< \partial \right>^{r-1}f_j=[ \square_{{g}_j} ,\left< \partial \right>^{r-1}] f_j,
		\\
		(\left< \partial \right>^{r-1}f_j,\partial_t \left< \partial \right>^{r-1}f_j)|_{t=0}=(\left< \partial \right>^{r-1}f_{0j},\left< \partial \right>^{r-1}f_{1j}).
	\end{cases}
\end{equation}
To estimate $[ \square_{{g}_j} ,\left< \partial \right>^{r-1}] f_j$, we will divide it into two cases $\frac{7}{6}\leq r< \frac52$ and $\frac52<r\leq \frac72$.

\textit{Case 1: $\frac{7}{6}\leq r< \frac52$.} For $\frac{7}{6}\leq r\leq \frac52$, note that
\begin{equation}\label{zu14}
	\begin{split}
		[ \square_{{g}_j} ,\left< \partial \right>^{r-1}] f_j
		=& [{g}^{\alpha i}_j-\mathbf{m}^{\alpha i}, \left< \partial \right>^{r-1} \partial_i ] \partial_{\alpha}f_j + \left< \partial \right>^{r-1}( \partial_i g_j \partial_{\alpha}f_j )
		\\
		= &[{g}_j-\mathbf{m}, \left< \partial \right>^{r-1} \nabla] df_j + \left< \partial \right>^{r-1}( \nabla g_j df).
	\end{split}
\end{equation}
By \eqref{ru14} and Kato-Ponce estimates, we have\footnote{If $r=\frac52$, then $L^{\frac{3}{s-r}}_x=L^\infty_x$.}
\begin{equation}\label{zu16}
	\| [ \square_{{g}_j} ,\left< \partial \right>^{r-1}] f_j \|_{L^2_x} \leq C ( \|dg_j\|_{L^\infty_x} \|d f_j \|_{H^{r-1}_x} + \|\left< \partial \right>^{r} (g_j-\mathbf{m}) \|_{L^{\frac{3}{r-1}}_x} \| df_j \|_{L^{\frac{3}{\frac52-r}}_x} )
\end{equation}
By Sobolev's inequality, it follows
\begin{equation}\label{zu17}
	\|\left< \partial \right>^{r} (g_j-\mathbf{m}) \|_{L^{\frac{3}{r-1}}_x} \leq C\|g_j-\mathbf{m}\|_{H^{\frac52}_x}.
\end{equation}
By H\"older's inequality, we can get
\begin{equation}\label{zu18}
	\begin{split}
		\| df_j \|_{L^{\frac{3}{\frac52-r}}_x} \leq & C \|df_j \|_{H^{r-1}_x}.
	\end{split}
\end{equation}
\textit{Case 2: $\frac52\leq r < \frac72$.} For $\frac52<r< \frac72$, using Kato-Ponce estimates, we have
\begin{equation}\label{zu150}
	\begin{split}
		\| [ \square_{{g}_j} ,\left< \partial \right>^{r-1}] f_j \|_{L^2_x}
		= & \| {g}^{\alpha i}_j-\mathbf{m}^{\alpha i},\left< \partial \right>^{r-1}] \partial_{\alpha i}f_j \|_{L^2_x}
		\\
		\leq & C ( \|dg_j\|_{L^\infty_x} \|d f_j \|_{H^{r-1}_x} + \|\left< \partial \right>^{r-1} (g_j-\mathbf{m}) \|_{L^{\frac{3}{r-2}}_x} \|\nabla df_j \|_{L^{\frac{3}{\frac72-r}}_x} ) .
	\end{split}
\end{equation}
By Sobolev's inequality, it follows
\begin{equation}\label{zu151}
	\|\left< \partial \right>^{r-1} (g_j-\mathbf{m}) \|_{L^{\frac{3}{r-2}}_x} \leq C\|g_j-\mathbf{m}\|_{H^{\frac52}_x}.
\end{equation}
By H\"older's inequality, we can get
\begin{equation}\label{zu152}
	\begin{split}
		\| \nabla df_j \|_{L^{\frac{3}{\frac72-r}}_x} \leq & C \|df_j \|_{H^{r-1}_x}.
	\end{split}
\end{equation}
For $\frac{7}{6}\leq r <\frac72$, gathering \eqref{zu16}-\eqref{zu18} with \eqref{zu150}- \eqref{zu152}, so we have
\begin{equation}\label{zu19}
	\begin{split}
		\| [ \square_{{g}_j} ,\left< \partial \right>^{r-1}] f_j \|_{L^2_x}\|d f_j \|_{H^{r-1}_x} \leq & C  \|dg_j\|_{L^\infty_x} \|d f_j \|^2_{H^{r-1}_x}
		+ C  \| g_j-\mathbf{m} \|_{H^{\frac52}_x}\|d f_j \|^2_{H^{r-1}_x}
	\end{split}	
\end{equation}
By \eqref{zu15} and \eqref{zu19}, we get
\begin{equation}\label{au20}
	\begin{split}
		\frac{d}{dt}\|d f_j \|^2_{H^{r-1}_x}
		\leq & C  \|dg_j\|_{L^\infty_x} \|d f_j \|^2_{H^{r-1}_x}
		+ C  \| g_j-\mathbf{m} \|_{H^{\frac52}_x}\|d f_j \|^2_{H^{r-1}_x}
	\end{split}	
\end{equation}
Using Gronwall's inequality, we get
\begin{equation}\label{zu21}
	\begin{split}
		\|d f_j(t) \|^2_{H^{r-1}_x}
		\leq & C  \|d f_j(0) \|^2_{H^{r-1}_x} \cdot\exp\{\int^t_0  \|dg_j\|_{L^\infty_x}+ \|g_j-\mathbf{m}\|_{H^{\frac52}_x} d\tau \} .
	\end{split}	
\end{equation}
Note
\begin{equation}\label{zu22}
	\begin{split}
		\|f_j(t) \|_{L^{2}_x}
		\leq & \| f_j(0) \|_{L^2_x} + \int^t_0 \| \partial_t f_j \|_{L^2_x} d\tau.
	\end{split}	
\end{equation}
By \eqref{zu21}, \eqref{zu22} and \eqref{kz65}, if $t\in[0,\bar{T}]$, then it follows
\begin{equation}\label{zu23}
	\begin{split}
		\| f_j\|_{L^\infty_{[0,\bar{T}]} H^{r}_x} + \| d f_j \|_{L^\infty_{[0,\bar{T}]} H^{r-1}_x}
		\leq & C  (\|f_0 \|_{H^{r}_x}+\|f_1 \|_{H^{r-1}_x})\mathrm{e}^{\mathbb{E}_*},
	\end{split}	
\end{equation}
where $\mathbb{E}_*$ is stated in \eqref{qu04}.

Based on \eqref{zu070}, \eqref{zu12},  and \eqref{zu23}, following the extending method in subsection \ref{esest}, for $a\leq r-\frac{7}{6}$ we shall obtain for $k \geq j$
\begin{equation}\label{zu234}
	\begin{split}
		\| \left< \partial \right>^{a-1} P_k df_j\|_{L^2_{[0,T_{N_1}]} L^\infty_x}
		\leq & \| \left< \partial \right>^{a-1} P_k df_j\|_{L^2_{[0,T_j]} L^\infty_x} \times  2^{\frac{1}{2(3+\epsilon)}(j-N_1)}
		\\
		\leq & C2^{-\frac{1}{2(3+\epsilon)} N_1} 2^{-\frac{\epsilon}{80} k}2^{-\frac{3\epsilon}{80} j} ( \| f_{0} \|_{H^r_x} + \| f_{1} \|_{H^{r-1}_x}  ).
	\end{split}
\end{equation}
and for $k<j$
\begin{equation}\label{zu25}
	\begin{split}
		& \| \left< \partial \right>^{a-1} P_k (df_{m+1}-df_m)\|_{L^2_{[0,T_{N_1}]} L^\infty_x}
		\\
		\leq & C 2^{-\frac{\epsilon}{80} k}2^{-\frac{3\epsilon}{80} m}2^{-\frac{1}{2(3+\epsilon)}m}( \| f_{0} \|_{H^r_x} + \| f_{1} \|_{H^{r-1}_x}  )(1+ \mathbb{E}^2(0)) \times 2^{\frac{1}{2(3+\epsilon)}(m-N_1)}
		\\
		\leq & C2^{-\frac{1}{2(3+\epsilon)}N_1} 2^{-\frac{\epsilon}{80} k}2^{-\frac{3\epsilon}{80} m} ( \| f_{0} \|_{H^r_x} + \| f_{1} \|_{H^{r-1}_x}  ) (1+ \mathbb{E}^2(0)) .
	\end{split}
\end{equation}
By phase decomposition, we have
\begin{equation*}
	df_j= P_{\geq j} df_j+ \textstyle{\sum}^{j-1}_{k=1} \textstyle{\sum}^{j-1}_{m=k} P_k (df_{m+1}-df_m)+ \textstyle{\sum}^{j-1}_{k=1} P_k df_k.
\end{equation*}
By using \eqref{zu234} and \eqref{zu25}, we get
\begin{equation}\label{zu230}
	\begin{split}
		\| \left< \partial \right>^{a-1} df_{j} \|_{L^2_{[0,T_{N_1}]} L^\infty_x} \leq & C( \| f_{0} \|_{H^r_x} + \| f_{1} \|_{H^{r-1}_x}  ) (1+ \mathbb{E}^2(0))[\frac{1}{3}(1-2^{-\frac{\epsilon}{80}})]^{-2}.
	\end{split}
\end{equation}
Therefore, seeing \eqref{zu230} and \eqref{zu23}, then \eqref{zu02} hold.
\end{proof}
%\subsection{The proof of Proposition \ref{DDL3}}
%By Proposition \ref{p1}, taking $s_0=2+\delta_1$ and $h=0$, we could get the conclusion in Proposition \ref{DDL3}.

\section*{Acknowledgments}  The work of L. Andersson is partially supported by the National Natural Science Foundation of China, under Grant Number W2431012. H.L. Zhang would like to thank the Max-Planck Institute for Gravitational Physics (Albert Einstein Institute) for hospitality during her visits in 2018-2019, where a substantial part of the work was done. Part of this work also was done while the authors were in residence at Institute Mittag-Leffler in Djursholm, Sweden during the fall of 2019, supported by the Swedish Research Council under grant no. 2016-06596. H.L. Zhang is also supported by the Natural Science Foundation of Hunan Province (Grant No. 2021JJ40561) and the Fundamental Research Funds for the Central Universities.

\end{document}